\documentclass[aos,preprint]{imsart}

\RequirePackage[OT1]{fontenc}
\RequirePackage{amsthm,amsmath}
\RequirePackage{amssymb}
\RequirePackage[numbers]{natbib}
\RequirePackage[colorlinks,citecolor=blue,urlcolor=blue]{hyperref}
\RequirePackage{subcaption}
\RequirePackage{graphicx}
\RequirePackage[]{algorithm2e}
 
\usepackage{xr}

\startlocaldefs
\numberwithin{equation}{section}
\theoremstyle{plain}
\newtheorem{theorem}{Theorem}[section]
\newtheorem{condition}{Condition}[section]
\newtheorem{corollary}{Corollary}[section]
\newtheorem{criterion}{Criterion}
\newtheorem{example}{Example}
\newtheorem{lemma}[theorem]{Lemma}
\newtheorem{proposition}{Proposition}[section]
\newtheorem{remark}{Remark}[section]

\newcommand{\RNum}[1]{\uppercase\expandafter{\romannumeral #1\relax}}

\newtheorem{innercustomthm}{Condition}
\newenvironment{customthm}[1]
{\renewcommand\theinnercustomthm{#1}\innercustomthm}
  {\endinnercustomthm}

\endlocaldefs

\begin{document}

\begin{frontmatter}
\title{Asymptotically  Independent   U-Statistics in High-Dimensional    Testing\thanksref{T1}}
\runtitle{ASYMPTOTICALLY  INDEPENDENT U-STATISTICS}
\thankstext{T1}{
This research is supported by NSF grants DMS-1711226, DMS-1712717, SES-1659328, CAREER SES-1846747,    and NIH grants R01GM113250, R01GM126002, R01HL105397 and R01HL116720}

\begin{aug}
\author{\fnms{Yinqiu} \snm{He}\thanksref{t1}\ead[label=e1]{yqhe@umich.edu }},
\author{\fnms{Gongjun} \snm{Xu}\thanksref{t1}\ead[label=e2]{gongjun@umich.edu }}
\author{\fnms{Chong} \snm{Wu}\thanksref{t2}\ead[label=e3]{cwu3@fsu.edu}}
\and
\author{\fnms{Wei} \snm{Pan}\thanksref{t3}
\ead[label=e4]{panxx014@umn.edu}
\ead[label=u1,url]{http://www.foo.com}}

\runauthor{He et al.}

\affiliation{University of Michigan\thanksmark{t1},  Florida State University \thanksmark{t2}, and  University of Minnesota\thanksmark{t3}}

\address{Yinqiu He and Gongjun Xu\\
Department of Statistics\\
University of Michigan\\
\printead{e1}\\
\phantom{E-mail:\ }\printead*{e2}}

\address{Chong Wu\\
Department of Statistics\\
Florida State University\\
\printead{e3}\\
}

\address{Wei Pan\\
Division of Biostatistics\\
School of Public Health\\
University of Minnesota\\
\printead{e4}\\
}
\end{aug}

\begin{abstract}
Many high-dimensional hypothesis tests aim to globally examine   marginal or low-dimensional features of a high-dimensional joint distribution, such as testing of mean vectors, covariance matrices and regression coefficients. 
This paper constructs a family of U-statistics as unbiased estimators of the $\ell_p$-norms of those features. 
 We show that under the null hypothesis, the U-statistics of different finite orders are asymptotically independent and normally distributed. Moreover, they are also asymptotically independent with the maximum-type test statistic, whose limiting distribution is an extreme value distribution. Based on the asymptotic independence property, we   propose an adaptive testing procedure which combines $p$-values computed from the U-statistics of different orders. 
We further establish   power analysis results and  
   show that the proposed adaptive procedure maintains high power against various alternatives. 
\end{abstract}


\end{frontmatter}

\section{Introduction} \label{sec:intromain}
 \paragraph{Motivation}
  
Analysis of high-dimensional data, whose dimension $p$ could be much larger than the sample size $n$, has emerged as an important and active research area \cite[e.g.,][]{fan2014challenges,Storey9440,friston2009modalities,fan2011sparse}.  In many large-scale inference problems, one is often interested in  globally testing some overall patterns of low-dimensional features of the high-dimensional random observations. One example is genome-wide association studies (GWAS), whose primary goal is to identify single nucleotide polymorphisms (SNPs) associated with certain complex diseases of interest. 
 A popular approach in GWAS is to perform  univariate tests which examine each SNP one by one. This however may lead to low statistical power due to the weak effect size of each SNP \cite{manolio2009finding} and the small  statistical significance threshold ($\sim  10^{-8}$)  chosen to control the multiple-comparison type \RNum{1} error   \citep{kim2016powerful}.
Researchers therefore have proposed to globally test a genetic marker set with   many SNPs \cite{wang2011gene,kim2016powerful} in order to achieve higher statistical power and to better understand  the underlying  genetic mechanisms.  


In this paper, we focus on a family of global testing problems in the high-dimensional setting, including testing of mean vectors,   covariance matrices and regression coefficients in generalized linear models. These problems can be formulated as $H_0: \mathcal{E}=\mathbf{0}$, where $\mathbf{0}$ is an all zero vector, $\mathcal{E}=\{e_l: l \in{\mathcal{L}} \}$ is a parameter vector with  $\mathcal{L}$ being the index set,  and $e_l$'s being the corresponding parameters of  interest, e.g., elements in mean vectors, covariance matrices or coefficients in generalized linear models. 
 For the global testing problem $H_0: \mathcal{E}=\mathbf{0}$ versus $H_A: \mathcal{E} \neq \mathbf{0}$, two different types of methods are often used in the literature.   {One is sum-of-squares-type statistics. They are usually powerful against ``dense'' alternatives,  where $\mathcal{E}$ has a high proportion of nonzero elements  with a large  $\|\mathcal{E} \|_2=\sum_{l\in\mathcal{L}}e_l^2$  or its weighted variants. See examples  in mean testing \cite[e.g.,][]{bai1996effect,goeman2006testing,srivastava2008test,chen2010mean,chen2014two,gregory2015two,srivastava2016raptt} and   covariance testing \cite[e.g.,][]{bai2009,ledoit2002,chen2011,lijun2012}. The other is maximum-type statistics. They are  usually   powerful against ``sparse"  alternatives, where $\mathcal{E}$ has few nonzero elements with  a large $\|\mathcal{E} \|_{\infty}$   \citep[e.g.,][]{jiang2004,liu2008,hall2010,cai2011,cai2013two,cai2014two,shao2014}.   }
 More recently,     \cite{fan2015power,weightstatsyang2017}  also proposed to combine these two kinds of test statistics. However, for  denser or only moderately dense alternatives, neither of these two types of statistics may be powerful, as will be further illustrated in this paper both theoretically and numerically. 
  Importantly, in real applications,   the underlying  truth is usually unknown, which  could be either sparse,  dense, or in-between. 
 As global testing could be highly underpowered if an inappropriate testing method is used \citep[e.g.,][]{colantuoni2011temporal}, it is desired in practice to have a  testing procedure with high statistical power against a variety of alternatives.

 \paragraph{A Family of Asymptotically Independent U-Statistics} 
  
   To address these issues, we propose a  U-statistics framework   and introduce its applications  to adaptive  high-dimensional  testing. The  U-statistics  framework constructs {\it unbiased} and {\it asymptotically independent} estimators of   $\| \mathcal{E} \|_{a}^a:=\sum_{l\in \mathcal{L}}e_l^a$ for  different (positive) integers $a$, where $a=2$ corresponds to a sum-of-squares-type statistic, and  an even integer $a\rightarrow \infty$ yields a maximum-type statistic. The adaptive testing  then combines the information from different $\| \mathcal{E} \|_{a}^a$'s, and our power analysis shows that it is powerful against a wide range of alternatives, from highly sparse, moderately sparse to dense, to highly dense.

 To illustrate our idea, suppose $\mathbf{z}_1, \ldots, \mathbf{z}_n$ are $n$ independent and identically distributed (i.i.d.) copies of a random vector $\mathbf{z}$. We consider the setting where each parameter $e_l$ has an   unbiased kernel function estimator $K_l(\mathbf{z}_{i_1},\ldots,\mathbf{z}_{i_{\gamma_l}})$, and $\gamma_l$ is the smallest integer such that for any $1 \leq i_1\neq \ldots \neq i_{\gamma_l} \leq n$,  $
	\mathrm{E} [ K_l(\mathbf{z}_{i_1},\ldots,\mathbf{z}_{i_{\gamma_l}}) ]=e_l$. This includes many testing problems on moments of low orders, such as entries in mean vectors,  covariance matrices and score vectors of generalized linear models, which shall be discussed in details.
The family of U-statistics can be constructed generally as follows. For integers $a \geq 1$, and $1\leq i_1 \neq \ldots \neq i_{\gamma_l} \neq \ldots \neq i_{(a-1)\times\gamma_l+1} \ldots  \neq i_{a\times \gamma_l } \leq n $, since the $\mathbf{z}$'s are i.i.d., we have
$\mathrm{E}[ K_l(\mathbf{z}_{i_1},\ldots, \mathbf{z}_{i_{\gamma_l}}) \cdots K_l(\mathbf{z}_{i_{(a-1)\times \gamma_l+1}},\ldots, \mathbf{z}_{i_{a\times \gamma_l }})  ]=e_l^a.$
 Therefore, we can construct an unbiased estimator of the parameters of augmented powers  $e_l^a$ with different $a$.  
Then $\|\mathcal{E}\|_a^a$ has an unbiased estimator
 \begin{align}
\mathcal{U}(a)=& \sum_{l \in \mathcal{L} } ( P^n_{a\times \gamma_l })^{-1} \sum_{1\leq i_1\neq \ldots \neq i_{a\times  \gamma_l } \leq n} \prod_{k=1}^a K_l(\mathbf{z}_{i_{(k-1)\times \gamma_l+1}},\ldots, \mathbf{z}_{i_{k\times \gamma_l}}), \label{eq:generalustat}
\end{align} where $P^n_{k}=n!/(n-k)!$ denotes the number of $k$-permutations of $n$. We call $a$ the order of the U-statistic  $\mathcal{U}(a)$.
  If $a>b$, we say $\mathcal{U}(a)$ is of higher order than $\mathcal{U}(b)$ and vice versa.
 
This   construction procedure can be applied to many testing problems. We give three common  examples below for illustration and more detailed case-studies will be discussed in Sections \ref{sec:mainexamsec} and  \ref{sec:extension}. 
\begin{example}\label{eg:1}
 Consider one-sample mean testing of $H_0: \boldsymbol{\mu}=\mathbf{0}$, where  $\mathcal{E}=\boldsymbol{\mu}$  is the mean vector of a $p$-dimensional random vector $\mathbf{x}$. 
 Suppose $\mathbf{x}_1, \ldots, \mathbf{x}_n$ are $n$ i.i.d. copies of $\mathbf{x}$.  For each  $i=1,\ldots ,n$, $j = 1,\ldots, p $,  $x_{i,j}$ is a simple unbiased estimator of $\mu_j$, then we can take the kernel function  $K_j(\mathbf{x}_i)=x_{i,j}$. Following \eqref{eq:generalustat}, we know the U-statistic   
\begin{eqnarray*}
 \mathcal{U}(a)=(P^n_a)^{-1} \sum_{j=1}^p \sum_{1\leq i_1 \neq  \ldots \neq i_a \leq n} \prod_{k=1}^a x_{i_k,j }
\end{eqnarray*}
 is an unbiased estimator of $\| \mathcal{E} \|_a^a=\| \boldsymbol{\mu} \|_a^a = \sum_{j=1}^p \mu_j^a$. 
Please see Section  \ref{maintest} for the two-sample mean testing example and related theoretical properties.
\end{example}

\begin{example} \label{eg:2}
Suppose $\mathbf{x}_1, \ldots, \mathbf{x}_n$ are $n$ i.i.d. copies of a random vector $\mathbf{x}$ with mean vector $\boldsymbol{\mu}=\mathbf{0}$ and covariance matrix $\boldsymbol{\Sigma}=\{\sigma_{j_1,j_2}\}_{p\times p}$. For covariance testing $H_0:\sigma_{j_1,j_2}=0$ for any $1\leq j_1 \neq j_2 \leq p$, we have $\mathcal{E}=\{ \sigma_l: l \in \mathcal{L}\}$ with $\mathcal{L}=\{(j_1,j_2): 1\leq j_1 \neq j_2 \leq p \}$. Since $x_{i,j_1}x_{i,j_2}$ is a simple unbiased estimator of $\sigma_{j_1,j_2}$, then for each pair $l=(j_1,j_2)\in \mathcal{L}$, we can take the kernel function $K_l(\mathbf{x}_i)= x_{i,j_1}x_{i,j_2}$. Following \eqref{eq:generalustat},   the U-statistic 
\begin{eqnarray*}
	 \mathcal{U}(a)=(P^n_a)^{-1} \sum_{1\leq j_1 \neq j_2 \leq p} \sum_{1\leq i_1 \neq \ldots \neq i_a \leq n} \prod_{k=1}^a (x_{i_k,j_1 } x_{i_k,j_2}) 
\end{eqnarray*} 
is an unbiased estimator of $\| \mathcal{E} \|_a^a=\sum_{1\leq j_1 \neq j_2 \leq p}\sigma_{j_1, j_2}^a$.
Please see Section \ref{sec:mainexamsec} for the general case with unknown $\boldsymbol{\mu}$.
\end{example}  

\begin{example} \label{eg:3}
Consider a response variable $y$ and its covariates $\mathbf{x}\in {\mathbb  R}^p$ following a generalized linear model: $\mathrm{E}(y|\mathbf{x})=g^{-1}(\mathbf{x}^{\intercal}\boldsymbol{\beta})$, where $g$ is the canonical link function and $\boldsymbol{\beta}\in {\mathbb  R}^p$ are the regression coefficients. Suppose that  $(\mathbf{x}_i ,y_i) $,  $i=1,\ldots, n$, are i.i.d. copies of $(\mathbf{x} ,y)$. For testing $H_0: \boldsymbol{\beta}=\boldsymbol{\beta}_0$, the  score vectors  $(S_{i,j} =(y_i-\mu_{0,i})x_{i,j}: j=1,\ldots,p)^{\intercal}$ are often used in the literature, where  $\mu_{0,i}=g^{-1}(\mathbf{x}_i^{\intercal}\boldsymbol{\beta}_0)$. Note that $\mathrm{E}(S_{i,j})=0$ under $H_0$.  Thus to test $H_0$, we can take $\mathcal{E}=\{\mathrm{E}(S_{i,j}): j=1,\ldots,p\}$ and use  the U-statistic 
	\begin{eqnarray*}
	\mathcal{U}(a)=(P^n_a)^{-1}  \sum_{j=1}^p \sum_{1\leq i_1 \neq \ldots \neq i_a \leq n} \prod_{k=1}^a S_{i_k,j},
\end{eqnarray*} which  is an unbiased estimator of $\|\mathcal{E}\|_a^a=\sum_{j=1}^p \{ \mathrm{E}(S_{i,j})\}^a$. Please see Section  \ref{sec:newglm}. 
\end{example}

\paragraph{Related Literature} 
For high-dimensional testing, some other adaptive testing procedures have  recently been proposed in \cite[][]{pan2014powerful,xu2016adaptive,wu2019adaptive}. 
These works combine the $p$-values of a family of sum-of-powered statistics that are powerful against different $\|\mathcal{E} \|_{a}^a$'s.
However in these existing works, to evaluate the $p$-value of the adaptive test statistic, the joint asymptotic distribution of the statistics is  difficult to obtain or calculate. Accordingly computationally expensive resampling methods are often used in practice \cite{pan2014powerful,kim2016powerful,xu2017adaptive}. 
 For some special cases such as  testing means and the coefficients of generalized linear models,  \cite{xu2016adaptive} and \cite{wu2019adaptive} derived the limiting distributions of  the test statistics under the framework of a family of von Mises V-statistics. 
 However,  the constructed V-statistics are usually \textit{correlated}  and  \textit{biased} estimators of the target $\|\mathcal{E}\|_a^a$. It follows that in \cite{xu2016adaptive} and \cite{wu2019adaptive},  numerical approximations are still needed to calculate the tail probabilities of the adaptive test statistics;  see Remark \ref{rm:comparewithxu} and Section \ref{sec:newglm}.    In addition, these existing adaptive testing works mainly focus on the first-order moments, and their results  do not directly  apply to testing second-order moments,  such as covariance matrices.
 
To overcome these issues, this paper considers the proposed family of unbiased U-statistics.
There are some other recent works providing important results on high-dimensional U-statistics \cite[e.g.,][]{chen2018gaussian,leung2018testing,zhong2011tests}. For instance, \cite{zhong2011tests} considered testing the regression coefficients in  linear models using the fourth-order U-statistic; \cite{leung2018testing} studied the limiting distributions of rank-based U-statistics; and \cite{chen2018gaussian} studied bootstrap approximation of the second-order U-statistics. 
 However, these results do not directly apply to the   high-order  U-statistics considered in this paper. 
 

\paragraph{Our Contributions} We   establish the   theoretical properties of the U-statistics in various  high dimensional testing problems, including  testing mean vectors, regression coefficients of generalized linear models,  and  covariance matrices. Our contributions are summarized as follows.

Under the null hypothesis, we show that the normalized U-statistics  of different finite orders are  jointly normally distributed.  The result applies generally for any asymptotic regime with $n\to\infty$ and $p\to\infty$.  
In addition, we prove that all the finite-order U-statistics are asymptotically independent with each other under the null hypothesis.  
Moreover, we prove that U-statistics of finite orders are also asymptotically independent of the maximum-type test statistic with a limiting extreme value distribution.

Under the alternative hypothesis, we further analyze  the  asymptotic power  for U-statistics of different  orders. We show that   when $\mathcal{E}$ has denser  nonzero entries, $\mathcal{U}(a)$'s of lower orders  tend to be more powerful; and when $\mathcal{E}$ has  sparser nonzero entries, $\mathcal{U}(a)$'s of higher orders tend to be more powerful.
More interestingly, we show that in the  boundary case of ``moderate" sparsity levels,  $\mathcal{U}(a)$ with a finite $a>2$ gives the highest power among the family of U-statistics, clearly indicating the inadequacy of both the sum-of-squares- and the maximum-type statistics.

An important application of the independence property among $\mathcal{U}(a)$'s is to construct adaptive testing procedures by combining the information of different $\mathcal{U}(a)$'s, whose univariate distributions or $p$-values can be easily combined to form a joint distribution to calculate the $p$-value of an adaptive test statistic. Compared with other existing works \cite[e.g.,][]{xu2016adaptive,wu2019adaptive}, numerical approximations of tail probabilities are no longer needed.  
As shown in the   power analysis, an adaptive integration of information across   different tests leads to a powerful testing procedure. 


The rest of the paper is organized as follows. In Sections \ref{sec:mainexamsec} and \ref{sec:twosimdata}, we illustrate the  framework by a covariance testing problem. Particularly, in Section \ref{sec:aymindptusec}, we study the U-statistics under  null hypothesis; in Section \ref{sec:powerana}, we analyze the power of the U-statistics; in Section \ref{sec:computtest}, we develop an adaptive testing procedure. In Sections \ref{sec:simulation} and \ref{sec:realdata}, we report  simulations and a real dataset analysis. In  Section \ref{sec:extension}, we  study other high-dimensional testing  problems, including testing means, regression coefficients and two-sample covariances. In Section \ref{sec:discuss}, we discuss several   extensions   of the proposed framework. We give proofs and other stimulations in \ref{suppA}. 

\section{Motivating Example: One-Sample Covariance Testing} \label{sec:mainexamsec}
The constructed family of U-statistics and adaptive testing procedure can be applied to   various high-dimensional testing problems. In this section, we illustrate the framework with a motivating example of one-sample covariance testing.  
Analogous results for other high-dimensional testing problems in Section \ref{sec:extension} can be obtained following similar analyses.
 We showcase the study of one-sample covariance testing  problem since this is more challenging than mean testing due to the two-way dependency structure and the one-sample problem can be used as the building block for more general cases.

 Specifically, we focus on testing
\begin{eqnarray}
	{H}_0: \sigma_{j_1,j_2}=0 \quad  \forall \ 1 \leq j_1 \neq j_2 \leq p,  \label{eq:nullhypindepen}
\end{eqnarray} where $\boldsymbol{\Sigma}=\{\sigma_{j_1,j_2}: 1 \leq j_1, j_2\leq p\}$ is the covariance matrix of a $p$-dimensional real-valued random vector $\mathbf{x}=(x_1,\ldots, x_p)^{\intercal}$ with $\mathrm{E}(\mathbf{x})=\boldsymbol{\mu}=(\mu_{1}, \ldots, \mu_{p})^{\intercal}$. The observed data include   $n$ i.i.d. copies of $\mathbf{x}$, denoted by $\mathbf{x}_1, \ldots, \mathbf{x}_n$ with $\mathbf{x}_i=(x_{i,1},\ldots, x_{i,p})^{\intercal}$. In factor analysis, testing $H_0$ in \eqref{eq:nullhypindepen} can be used to examine  whether $\boldsymbol{\Sigma}$ has any significant factor or not \cite{anderson1958introduction}.

Global testing of covariance structure plays an important role in many statistical analysis and applications; see a review in  \cite{caicovreview2017}. 
Conventional tests include  the likelihood ratio test, John's test, and Nagao's test, etc.  \citep[][]{anderson1958introduction, muirhead2009aspects}. These methods, however, often fail in  the high-dimensional setting when both $n,p \rightarrow \infty$. To address this issue,   new procedures have been recently proposed \cite[e.g.,][]{bai2009,jiang2013,johnstone2001, soshnikov2002note,schott2007test, peche2009universality,ledoit2002, chen2011,jiang2004,liu2008,cai2011,lijun2012,shao2014,lan2015testing}.
However these  methods might suffer from loss of power when the sparsity level of the alternative covariance matrix varies. In the following subsections, we introduce the general U-statistics framework, study their asymptotic properties,  and develop a powerful adaptive testing procedure.

 We  introduce some notation.  For two series of numbers $u_{n,p}$, $v_{n,p}$ that depend on $n, p$: 
  $u_{n,p}=o(v_{n,p})$ denotes $\limsup_{n,p \rightarrow \infty} |u_{n,p}/v_{n,p}| =0$; 
  $u_{n,p}=O(v_{n,p})$  denotes $\limsup_{n,p \rightarrow \infty} |u_{n,p}/v_{n,p}| <\infty$; 
 $u_{n,p}=\Theta( v_{n,p})$ denotes $0< \liminf_{n,p \rightarrow \infty} |u_{n,p}/v_{n,p}| \leq \limsup_{n,p \rightarrow \infty} |u_{n,p}/v_{n,p}| <\infty$; $u_{n,p} \simeq v_{n,p}$ denotes $\lim_{n,p\rightarrow \infty} u_{n,p} /v_{n,p}=1$. Moreover, $\xrightarrow{P}$ and $\xrightarrow{D}$ represent the convergence in probability and distribution respectively. For $p$-dimensional random vector $\mathbf{x}$ with mean $\boldsymbol{\mu}$ and $\forall j_1, \ldots, j_t \in \{1,\ldots, p\}$, we write the central moment as
 \begin{eqnarray}
	\Pi_{j_1,\ldots ,j_t} =\mathrm{E}[(x_{j_1}-\mu_{j_1})\ldots (x_{j_t}-\mu_{j_t})]. \label{eq:highmomentdef}
\end{eqnarray}


\subsection{Asymptotically Independent U-Statistics} \label{sec:aymindptusec}  

For  testing   \eqref{eq:nullhypindepen}, the set of parameters that we are interested in is $\mathcal{E}=\{ \sigma_{j_1,j_2}:  1\leq j_1 \neq j_2 \leq p \}$. Following the previous analysis of \eqref{eq:generalustat}, since $\sigma_{j_1,j_2}$ has a simple unbiased estimator $x_{i_1,j_1}x_{i_1,j_2}-x_{i_1,j_1}x_{i_2,j_2}$ with $1\leq i_1\neq i_2 \leq n$, then for integers $a \geq 1$, an unbiased U-statistic of $\|\mathcal{E} \|_a^a=\sum_{1\leq j_1 \neq j_2 \leq p} \sigma_{j_1,j_2}^a$ is
\begin{align*}
	\mathcal{U}(a)=(P^n_{2a})^{-1}& \sum_{1\leq j_1 \neq j_2 \leq p} \sum_{1\leq i_1 \neq \ldots \neq i_{2a} \leq n} \prod_{k=1}^a (x_{i_{2k-1},j_1}x_{i_{2k-1},j_2}-x_{i_{2k-1},j_1}x_{i_{2k},j_2}). 
\end{align*}
 This is equivalent to
 \begin{eqnarray}
	\mathcal{U}(a)&=&\sum_{1 \leq j_1 \neq j_2 \leq p} \sum_{c=0}^a (-1)^c \binom{a}{c} \frac{1}{P^n_{a+c}}  \sum_{1\leq i_1\neq \ldots \neq i_{a+c} \leq n} \label{eq:originaluinvariance} \\
	&&\quad \prod_{k=1}^{a-c} (x_{i_k,j_1} x_{i_k,j_2})  \prod_{s=a-c+1}^{a} x_{i_{s},j_1}  \prod_{t=a+1}^{a+c}x_{i_{t},j_2}.  \notag
\end{eqnarray}
 
\begin{remark} \label{rm:anotherpesusstat}
The U-statistics can be constructed by another  method equivalently. Given  $1 \leq j_1 \neq j_2 \leq p$, define $\varphi_{j_1,j_2}=\sigma_{j_1,j_2}+\mu_{j_1}\mu_{j_2}$. Then 
 \begin{align}
 \sum_{1 \leq j_1\neq j_2\leq p}\sigma_{j_1, j_2}^a =\sum_{1 \leq j_1\neq j_2\leq p} \sum_{c=0}^a \binom{a}{c}  \varphi_{j_1,j_2}^{a-c} \times (-\mu_{j_1}\mu_{j_2})^{c}, \label{eq:unbaisedestgoal}
\end{align} which is a polynomial function of the moments $\mu_j$ and $\varphi_{j_1, j_2}$. Since $\mu_j$ and $\varphi_{j_1, j_2}$ have unbiased estimators $x_{i,j}$ and $x_{i,j_1}x_{i,j_2}$ respectively, then for $1\leq i_1 \neq \ldots \neq i_{a+c}\leq n$,
$  \mathrm{E} ( \prod_{k=1}^{a-c}x_{i_{k}, j_1}x_{i_{k}, j_2}  \prod_{s=a-c+1}^{a } x_{i_s, j_1}  \prod_{t=a+1}^{a+c } x_{i_t, j_2} )  =  \varphi_{j_1,j_2}^{a-c} \mu_{j_1} ^c \mu_{j_2}^c.
$ Given this and \eqref{eq:unbaisedestgoal}, the  U-statistics \eqref{eq:originaluinvariance} can be obtained.
\end{remark}

\begin{remark} \label{rm:leadingteststat}
The summed term with $c=0$ in \eqref{eq:originaluinvariance} is
\begin{align}
	\tilde{\mathcal{U}}(a):=\left(P^n_{a}\right)^{-1}\sum_{1\leq i_1 \neq \cdots \neq i_{a}\leq n}\sum_{1 \leq j_1\neq j_2\leq p}  \prod_{k=1}^a ( x_{i_k,j_1} x_{i_k,j_2}), \label{eq:originleadingterm} 
\end{align} which has the same form as the simplified U-statistic for mean zero observations in Example \ref{eg:2}, and is shown to be the leading term of \eqref{eq:originaluinvariance} in proof. 
\end{remark}

We next introduce  some nice properties of the U-statistics \eqref{eq:originaluinvariance}. The first one is the following  location invariant property.
 \begin{proposition}\label{prop:locinvariance}
$\mathcal{U}(a)$ constructed as in \eqref{eq:originaluinvariance} is location invariant; that is,  for any  vector $\mathbf{\Delta}\in {\mathbb  R}^p$,  the U-statistic  constructed based on the transformed data   $\{\mathbf{x}_i+\mathbf{\Delta}: i=1,\ldots,n\}$ is still $\mathcal{U}(a)$. 
 \end{proposition}

 The  following  proposition   verifies that the constructed U-statistics are unbiased estimators of $\|\mathcal{E}\|_a^a=\sum_{1 \leq j_1\neq j_2\leq p}\sigma_{j_1,j_2}^a$. 
\begin{proposition} \label{prop:unbaisedestimatorprop}
For any   integer $a$, 
$
	\mathrm{E} [\mathcal{U}(a)]=\sum_{1 \leq j_1\neq j_2\leq p}\sigma_{j_1, j_2}^a .
$ Under $H_0$ in \eqref{eq:nullhypindepen},   
$
	\mathrm{E} [\mathcal{U}(a)]=0.
$
\end{proposition}

We next study the limiting properties of the constructed U-statistics under $H_0$ given the following assumptions on the random vector $\mathbf{x}=(x_1,\ldots, x_p)^{\intercal}$. 
 \begin{condition} [Moment assumption]  \label{cond:finitemomt}$\lim_{p\rightarrow \infty} \max_{1\leq j\leq p} \mathrm{E}(x_{j}-\mu_j)^8< \infty$ and  $\lim_{p\rightarrow \infty} \min_{1\leq j\leq p} \mathrm{E}(x_{j}-\mu_j)^2>0$.
 \end{condition}

 \begin{condition} [Dependence  assumption] \label{cond:alphamixing}
For a sequence of random variables $\mathbf{z}=\{z_{j}: j \geq 1 \}$ and integers $a<b$, let $\mathcal{Z}_a^b$ be the $\sigma$-algebra generated by $\{z_{j}: j \in \{a,\ldots,b\} \}$. For each $s \geq 1$, define the $\alpha$-mixing coefficient $\alpha_{\mathbf{z}}(s)=\sup_{t \geq 1}\{ |P(A \cap B )- P(A)P(B)|: A \in  \mathcal{Z}_1^t, B\in \mathcal{Z}_{t+s}^{\infty} \}$. We assume that under $H_0$, $\mathbf{x}$ is $\alpha$-mixing with $\alpha_{\mathbf{x}}(s) \leq M\delta^s$, where $\delta \in (0,1)$ and $M>0$ are some constants. 
\end{condition}

\begin{customthm}{\ref{cond:alphamixing}${}^*$} [Alternative dependence  assumption to Condition \ref{cond:alphamixing}] \label{cond:highordmominter} \label{cond:ellpmoment}
Following the notation in  \eqref{eq:highmomentdef}, we assume that under $H_0$,  for any $j_1,j_2, j_3 \in \{1,\ldots, p\}$, $\Pi_{j_1,j_2, j_3}=0$;  for any	$j_1, j_2, j_3, j_4 \in \{1,\ldots, p\}$, $\Pi_{j_1,j_2,j_3,j_4}= \kappa_1 (\sigma_{j_1,j_2}\sigma_{j_3,j_4}+\sigma_{j_1,j_3}\sigma_{j_2,j_4}+\sigma_{j_1,j_4}\sigma_{j_2,j_3})$ for some constant $\kappa_1<\infty$;  and for $t=6,8$, and any $j_1, \cdots, j_t \in \{1,\ldots, p\}$,  $\Pi_{j_1,\cdots, j_t}=0$ when at least one of these indexes appears odd times in $\{j_1, \cdots, j_t\}$. 
\end{customthm}

Condition \ref{cond:finitemomt} assumes that the eighth marginal moments of  $\mathbf{x}$  are uniformly bounded from above and the second moments are uniformly bounded from below, which are true for most light-tailed distributions. 
Condition \ref{cond:alphamixing} assumes weak dependence among different $x_j$'s under $H_0$, since the uncorrelatedness  of $x_j$'s  under $H_0$ may not imply the independence of them, especially   when $x_j$'s are non-Gaussian.  
 Under $H_0$, Condition \ref{cond:alphamixing} automatically holds when $\mathbf{x}$ is Gaussian or    $m$-dependent.   The mixing-type weak dependence is  similarly considered in previous works such as \cite[][]{bickel2008regularized,chen2014two,xu2016adaptive} and also commonly assumed in time series and spatial statistics \cite{gaetan2010spatial,pham1985some}.   Moreover,  the variables in our motivating genome-wide association studies have a local dependence structure, with their associations often decreasing to zero as the corresponding physical distances on a chromosome increase.  We note that  it suffices to have  Condition \ref{cond:alphamixing} hold up to a permutation of the variables.
 
Alternatively, we can substitute Condition \ref{cond:alphamixing} with Condition \ref{cond:highordmominter}. Condition \ref{cond:highordmominter} specifies some higher order  moments of $\mathbf{x}$ and is satisfied when   $\mathbf{x}$ follows an elliptical  distribution with finite eighth   moments and  covariance  $\boldsymbol{\Sigma}$  \citep[see][]{anderson1958introduction,frahm2004generalized,muirhead2009aspects,paindaveine2014inference}.    
Conditions \ref{cond:highordmominter} and \ref{cond:alphamixing}   become   equivalent  when $\mathbf{x}$ follows a multivariate Gaussian distribution.  
The fourth moment condition is also assumed in other high-dimensional research  \citep{cai2013two}. In this work, the eighth moment condition is needed to establish the asymptotic joint distribution of different  U-statistics.

The following theorem  specifies the asymptotic variances   of the  finite order U-statistics and their joint limiting distribution.    Since the   U-statistics are    degenerate under $H_0$, an analysis different  from the asymptotic  theory on non-degenerate U-statistics  \cite[e.g.,][]{hoeffding1992class}   is needed in the proof.

\begin{theorem}\label{thm:jointnormal}
Under $H_0$ in \eqref{eq:nullhypindepen} and Conditions \ref{cond:finitemomt} and \ref{cond:alphamixing} (or \ref{cond:highordmominter}),
for $\mathcal{U}(a)$'s   defined in   \eqref{eq:originaluinvariance} and any distinct  finite (and positive) integers  $\{a_1, \ldots, a_m\}$, as  $n,p \rightarrow \infty$, 
\begin{eqnarray}
	\Big[ \frac{\mathcal{U}(a_1)}{\sigma(a_1)}, \ldots, \frac{\mathcal{U}(a_m)}{\sigma(a_m)}  \Big]^{\intercal} \xrightarrow{D} \mathcal{N}(0,I_m),  \label{eq:jointnormalformmodif1}
\end{eqnarray}	where 
\begin{eqnarray} 
	\sigma^2(a) := \mathrm{var} [\mathcal{U}(a) ] \simeq  \frac{a!}{P^n_a}  \sum_{ \substack{1\leq j_1\neq j_2 \leq p ;\, 1\leq j_3 \neq j_4 \leq p}} (\Pi_{j_1,j_2,j_3,j_4})^a,   \label{eq:varianceformmodif1}
\end{eqnarray} with $\Pi_{j_1, j_2,j_3,j_4}$ defined  in \eqref{eq:highmomentdef}. Note that   $\sigma^2(a)=\Theta(p^2n^{-a})$.
\end{theorem}

Theorem \ref{thm:jointnormal} shows that after normalization, the finite-order U-statistics have a joint normal limiting distribution with an identity covariance matrix,  which implies that they are asymptotically independent as $n,p \rightarrow \infty$. 
The nice independence property makes it easy to combine these U-statistics and apply our proposed adaptive testing later.  Moreover, the conclusion holds  on general asymptotic regime for $n,p \to \infty$, without any constraint on the relationship between $n$ and $p$.  We will also see in Section \ref{sec:extension} that similar results  hold generally for some other testing problems.
\begin{remark}
Theorem \ref{thm:jointnormal} discusses the U-statistics of finite orders, i.e., the $a$ values do not grow with $n,p$. When $\{x_1,\ldots, x_p\}$ are  independent, Theorem \ref{thm:jointnormal} can be extended when $a=O(1) \min\{\log^{\epsilon} n,  \log^{\epsilon} p\}$ for some  $\epsilon>0$. On the other hand, we will show in  Section \ref{sec:powerana} that it is usually enough to include  $\mathcal{U}(a)$'s of finite $a$.  Therefore, we do not pursue the general case when $a$  grows with $n,p$ in this work.  
\end{remark}

 In the following, we  further discuss the maximum-type test statistic $\mathcal{U}(\infty)$, which corresponds to the $\ell_{\infty}$-norm of the parameter vector $\mathcal{E}=\{e_l: l\in \mathcal{L} \}$, that is, $\| \mathcal{E}\|_{\infty}=\max_{l \in \mathcal{L}} |e_{l}|$. In the existing literature, there is already some corresponding  established work \citep{jiang2004,cai2011} on the test statistic:
\begin{eqnarray}
	\quad	\quad \quad M_n^*:=\max_{1\leq j_1\neq j_2\leq p}| { \hat{\sigma}_{j_1,j_2}  }/{\sqrt{ \hat{\sigma}_{j_1,j_1} \hat{\sigma}_{j_2,j_2}  }  } |,\label{eq:inftyteststat}
	\end{eqnarray} 	
	 where $(\hat{\sigma}_{j_1,j_2})_{p\times p}=\sum_{i=1}^n (\mathbf{x}_i-\bar{\mathbf{x}})(\mathbf{x}_i-\bar{\mathbf{x}})^{\intercal}/n$ and $\bar{\mathbf{x}}=\sum_{i=1}^n\mathbf{x}_i/n$.  We will take $\mathcal{U}(\infty)=M_n^*$ below. The limiting distribution of $\mathcal{U}(\infty)$ was first studied in \cite{jiang2004} and extended by \cite{cai2011,liu2008,shao2014}. Next we restate the result in \cite{cai2011}, which gives the limiting distribution of \eqref{eq:inftyteststat} under the following condition.	
 	 
\begin{condition} \label{cond:maxiidcolumn}
Consider the random vector $\mathbf{x}=(x_1,\ldots,x_p)^{\intercal}$ with mean vector $\boldsymbol{\mu}=(\mu_1,\ldots, \mu_p)^{\intercal}$ and covariance matrix $\boldsymbol{\Sigma}=\mathrm{diag}( \sigma_{1,1},\ldots, \sigma_{p,p} )$. $(x_j-\mu_j)/\sqrt{\sigma_{j,j}}$ are i.i.d. for  $j=1,\ldots,p$.  Furthermore, $\mathrm{E}e^{t_0( |x_1-\mu_1|/\sqrt{\sigma_{1,1}})^{\varsigma}} < \infty$ for some $0 <\varsigma \leq 2$ and $t_0>0$.	
\end{condition} 

\begin{theorem}[{\citet[Theorem 2]{cai2011}}] \label{thm:extlimit}
Assume  Condition    \ref{cond:maxiidcolumn}  and $\log p= o(n^{\beta})$, where $\beta=\varsigma/(4+\varsigma)$. Then  
$
		P( n \times \mathcal{U}(\infty)^2+\varpi_p \leq u) \rightarrow G(u)=e^{-(1/\sqrt{8\pi})e^{-u/2}},
	$ where $\varpi_p=-4\log p+\log \log p$ and $G(u)$ is an extreme value distribution of type I.
\end{theorem}

Theorems \ref{thm:jointnormal} and  \ref{thm:extlimit} give the limiting distributions of $\mathcal{U}(a)$ of finite orders and $\mathcal{U}(\infty)$ respectively;  it is of interest to examine their joint distribution. The following theorem shows that  although $\mathcal{U}(\infty)$ has limiting distribution different from $\mathcal{U}(a)$, $a<\infty$, they are still asymptotically independent. 
   
\begin{theorem} \label{thm:asymindpt}  
Assume that Condition \ref{cond:finitemomt}  is  satisfied,  Condition \ref{cond:maxiidcolumn} holds for  $ \varsigma=2$, and $\log p=o(n^{1/7})$.   For finite integers  $\{a_1, \ldots, a_m\}$, under $H_0$,  $\mathcal{U}(a_1),\ldots,\mathcal{U}(a_m)$ and $\mathcal{U}(\infty)$ are mutually  asymptotically independent. In specific, for any $z_1,\ldots, z_m,y \in \mathbb{R}$,  as $n,p \rightarrow \infty$,
\begin{align*}
	 &\Big| P \Big( n\mathcal{U}(\infty)^2+\varpi_p \geq y, \frac{\mathcal{U}(a_1)}{\sigma(a_1)} \leq z_1,\ldots, \frac{\mathcal{U}(a_m)}{\sigma(a_m)} \leq z_m \Big) \notag \\
	 &\ -P\Big(n\mathcal{U}(\infty)^2+\varpi_p \geq y \Big) \times  \prod_{r=1}^m  P\Big( \frac{\mathcal{U}(a_r)}{\sigma(a_r)} \leq z_r \Big)   \Big|\to 0.
\end{align*}
\end{theorem}


Theorem \ref{thm:jointnormal} suggests that all the finite-order U-statistics are asymptotically independent with each other. Given this, Theorem \ref{thm:asymindpt} further shows that the maximum-type test statistic $\mathcal{U}(\infty)$ is also asymptotically mutually independent with those finite-order U-statistics.
 The conclusion shares  similarity  with some classical results on the asymptotic independence between the sum-of-squares-type and maximum-type statistics. Specifically, for random variables $w_1,\ldots,w_n$, \cite{hsing1995,ho1996asymptotic} proved the asymptotic independence between $\sum_{i=1}^n w_i^2$ and $\max_{i=1,\ldots,n} |w_i|$ for weakly dependent observations.  The similar independence properties were extensively studied  in literature \cite[e.g.][]{mccormick2000asymptotic,ho1999asymind,peng2003joint,james2007limit,xu2016adaptive,2015arXiv151208819L}.  However, there are several differences between existing literature and the results in this paper. First,  we discuss a family of U-statistics $\mathcal{U}(a)$'s, which  takes  different $a$ values, and $\mathcal{U}(2)$ here corresponding to the sum-of-squares-type statistic is only a special case of general $\mathcal{U}(a)$. Furthermore,  we have shown not only the asymptotic independence between $\mathcal{U}(a)$ and $\mathcal{U}(\infty)$, but also the asymptotic independence among $\mathcal{U}(a)$'s of finite $a$ values.  Second, the constructed $\mathcal{U}(a)$'s are unbiased estimators, which are different from the sum-of-squares statistics  usually examined in the literature. Moreover, the $x$'s are allowed to be  dependent and the theoretical development in the covariance testing involves a two-way dependence structure, which requires  different proof techniques from the existing studies. 

\begin{remark}\label{rm:standizedmax}
An alternative way to construct $\mathcal{U}(\infty)$ is to standardize  $\hat{\sigma}_{j_1,j_2}$ by its variance $\widehat{\mathrm{var}}(\hat{\sigma}_{j_1,j_2})$.  Specifically, following  \citet{cai2013two}, we take
$
	\widehat{\mathrm{var}}(\hat{\sigma}_{j_1,j_2}) = n^{-1}\sum_{i=1}^n\{(x_{i,j_1}-\bar{x}_{j_1})(x_{i,j_2}-\bar{x}_{j_2})-\hat{\sigma}_{j_1,j_2}\}^2.
$ Define $M_n^{\dag}=\max_{1\leq j_1\neq j_2\leq p}|\hat{\sigma}_{j_1,j_2}|/\{\widehat{\mathrm{var}}(\hat{\sigma}_{j_1,j_2})\}^{1/2}$ and we take $\mathcal{U}(\infty)=M_n^{\dag}$. Theoretically, we prove  that Theorem \ref{thm:asymindpt} still holds with $\mathcal{U}(\infty)=M_n^{\dag}$ in \ref{suppA} Section \ref{sec:pfstandmax}.  Numerically, we provide the simulations in \ref{suppA} Section \ref{sec:suppsimu}, which shows that $M_n^*$ in \eqref{eq:inftyteststat} generally has higher power than $M_n^{\dag}$. 
\end{remark} 


To apply hypothesis testing using the asymptotic results in Theorems \ref{thm:jointnormal} and \ref{thm:asymindpt}, we need to estimate $\mathrm{var}\{\mathcal{U}(a)\}$. In particular, we propose the following moment estimator of   \eqref{eq:varianceformmodif1}:
\begin{align}
	\mathbb{V}_u(a)=\frac{2a!}{(P^n_a)^2} \sum_{\substack{1\leq j_1\neq j_2\leq p} } \sum_{1\leq i_1\neq \ldots \neq i_a \leq n} \prod_{t=1}^a (x_{i_t,j_1}-\bar{x}_{j_1})^2(x_{i_t,j_2}-\bar{x}_{j_2})^2. \label{eq:varestnew}
\end{align}
The next result establishes the statistical  consistency of   $\mathbb{V}_u(a)$. 
\begin{condition} \label{cond:higherordermomentvarest}
For integer $a$, $\lim_{p\rightarrow \infty} \max_{1\leq j\leq p} \mathrm{E}(x_{j}-\mu_j)^{8a}< \infty$.  
\end{condition}
\begin{theorem}\label{thm:computation}
Under $H_0$ in \eqref{eq:nullhypindepen}, assume  Conditions \ref{cond:finitemomt}, \ref{cond:alphamixing} and \ref{cond:higherordermomentvarest} hold.
Then $\mathbb{V}_u(a)/\mathrm{var}\{ \mathcal{U}(a)\}\xrightarrow{P} 1$.     	
\end{theorem}
Theorem \ref{thm:computation} implies that the asymptotic results in Theorems \ref{thm:jointnormal} and  \ref{thm:asymindpt} still hold by replacing $\mathrm{var}\{\mathcal{U}(a)\}$ with its estimator $\mathbb{V}_u(a)$. Specifically, under $H_0$,   $[ {\mathcal{U}(a_1)}/{\sqrt{\mathbb{V}_u(a_1)}}, \ldots, \allowbreak {\mathcal{U}(a_m)}/{\sqrt{\mathbb{V}_u(a_m)}}  ]^{\intercal} \xrightarrow{D} \mathcal{N}(0,I_m)$ under Conditions \ref{cond:finitemomt},     \ref{cond:alphamixing} and \ref{cond:higherordermomentvarest}. Moreover,  Theorem \ref{thm:asymindpt} implies that $\{\mathcal{U}(a)/\sqrt{\mathbb{V}_u(a)}\}$'s are  asymptotically independent with $\mathcal{U}(\infty)$. 


\subsection{Power Analysis} \label{sec:powerana}
In this section, we analyze the asymptotic power of the U-statistics.  The power of $\mathcal{U}(2)$ has been  studied in the literature. 
In particular, \cite{caima2013}  studied the hypothesis testing of a high-dimensional covariance matrix with   $H_0: \boldsymbol{\Sigma}={I}_p$.
  The authors characterized the boundary that distinguishes the testable region from the non-testable region in terms of the Frobenius norm $\| \boldsymbol{\Sigma}- {I}_p \|_F$, and showed that the   test statistic proposed by \cite{chen2011,caima2013}, which corresponds to $\mathcal{U}(2)$ in this paper, is rate optimal over their considered regime.   
However in practice, $\mathcal{U}(2)$ may be  not powerful if the alternative covariance matrix is sparse with a small $\|\boldsymbol{\Sigma}-{I}_p\|_F$. When  the   alternative covariance    has different sparsity levels, it is of interest to further examine  which $\mathcal{U}(a)$ achieves   the best power performance among the constructed family of U-statistics. 


To study the test power, we establish the limiting distributions of $\mathcal{U}(a)$'s under 
the alternative hypothesis $H_A: \boldsymbol{\Sigma}=\boldsymbol{\Sigma}_A$, where the alternative covariance matrix $\boldsymbol{\Sigma}_A=(\sigma_{j_1,j_2})_{p\times p}$ is specified in the following Condition  \ref{cond:rhoijconst}. Define $J_A=\{(j_1,j_2): \sigma_{j_1,j_2}\neq 0, 1\leq j_1 \neq j_2\leq p \},$ which indicates  the nonzero off-diagonal entries in  $\boldsymbol{\Sigma}_A$.   The cardinality  of $J_A$, denoted by $|J_A|,$ then represents the sparsity level of   $\boldsymbol{\Sigma}_A$.  

\begin{condition}\label{cond:rhoijconst}
Assume $|J_A|=o(p^2)$ and for $(j_1,j_2)\in J_A$, $|\sigma_{j_1,j_2}|=\Theta(\rho)$, where $\rho=\sum_{(j_1,j_2)\in J_A}|\sigma_{j_1,j_2}|/|J_A|$. 
\end{condition}
\noindent 
Here $\rho$ represents the average signal strength of $\boldsymbol{\Sigma}_A$. 
In our following power comparison of two U-statistics $\mathcal{U}(a)$ and $\mathcal{U}(b)$, we say $\mathcal{U}(a)$ is   ``better''  than $\mathcal{U}(b)$, if, under the same test power,  $\mathcal{U}(a)$ can detect a smaller average signal strength $\rho$  (please see the specific definition in Criterion \ref{cricomp} on Page \pageref{cricomp}). 
\noindent Condition \ref{cond:rhoijconst} specifies  a general family of ``local'' alternatives, which include  banded covariance matrices, block covariance matrices, and  sparse   covariance matrices whose nonzero entries are randomly located.
\begin{theorem} \label{thm:cltalternative}
Suppose  Conditions \ref{cond:finitemomt},    \ref{cond:rhoijconst}, and \ref{cond:altonecovellip} (an analogous condition to Condition \ref{cond:ellpmoment}  under $H_A$) in the Supplementary Material hold.  For $\mathcal{U}(a)$ in \eqref{eq:originaluinvariance} and  finite integers $\{a_1,\ldots, a_m\}$, if $\rho=O(|J_A|^{-1/a_t}p^{1/a_t}n^{-1/2})$ for $t=1,\ldots,m,$
then as  $n,p \rightarrow \infty$,
\begin{eqnarray*}
	\Big[ \frac{\mathcal{U}(a_1)-\mathrm{E}[\mathcal{U}(a_1)]}{\sigma(a_1)}, \ldots, \frac{\mathcal{U}(a_m)-\mathrm{E}[\mathcal{U}(a_m)]}{\sigma(a_m)}  \Big]^{\intercal} \xrightarrow{D} \mathcal{N}(0,I_m),  
\end{eqnarray*}	where for $a\in\{a_1,\ldots, a_m\},$ $\mathrm{E}[\mathcal{U}(a)]=\sum_{(j_1,j_2)\in J_A} \sigma_{j_1,j_2}^a$ and  $\sigma^2(a)=\mathrm{var}[\mathcal{U}(a)]\simeq 2a!\kappa_1^a n^{-a} \sum_{1\leq j_1\neq j_2\leq p}  \sigma_{j_1,j_1}^a \sigma_{j_2,j_2}^a,$ which is of order $\Theta(p^2n^{-a})$.
\end{theorem}
Theorem \ref{thm:cltalternative} shows that for a single U-statistic $\mathcal{U}(a)$ of finite order $a$,
\begin{align}
	P\biggr( \frac{\mathcal{U}(a)}{\sqrt{\mathrm{var}[\mathcal{U}(a)]}} > z_{1-\alpha} \biggr) \rightarrow  1- \Phi \biggr( z_{1-\alpha} - \frac{\mathrm{E}[\mathcal{U}(a)]}{\sqrt{\mathrm{var}[\mathcal{U}(a)]}} \biggr), \label{eq:powerbound} 
\end{align} where $z_{1-\alpha}$ is the upper $\alpha$ quantile of $\mathcal{N}(0,1)$ and $\Phi(\cdot)$ is the cumulative distribution function of $\mathcal{N}(0,1)$. By Theorem \ref{thm:cltalternative}, the asymptotic power of $\mathcal{U}(a)$ of the one-sided test   depends on
\begin{align}
	\frac{\mathrm{E}[\mathcal{U}(a)]}{\sqrt{\mathrm{var}[\mathcal{U}(a)]}} \simeq \frac{\sum_{(j_1,j_2)\in J_A} \sigma_{j_1,j_2}^a }{ \{2a!\kappa_1^an^{-a}\sum_{1\leq j_1\neq j_2\leq p}  (\sigma_{j_1,j_1}\sigma_{j_2,j_2})^a\}^{1/2}  }. \label{eq:ratiopower}
\end{align} 
By Theorem \ref{thm:cltalternative}, $\eqref{eq:ratiopower}=\Theta(|J_A|\rho^a p^{-1}n^{a/2}).$ It follows that  when $\mathrm{E}[\mathcal{U}(a)]$ is of the same order of $\sqrt{\mathrm{var}[\mathcal{U}(a)]}$, i.e., $\mathrm{E}[\mathcal{U}(a)]=O(1)\sqrt{\mathrm{var}[\mathcal{U}(a)]}$, the constraint of $\rho$ in  Theorem \ref{thm:cltalternative}  is satisfied.

In the following power analysis, we will first compare  $\mathcal{U}(a)$'s of finite $a$ and then compare them with $\mathcal{U}(\infty)$. As we focus on studying the relationship between the sparsity level and power, we consider an ideal case where $\sigma_{j_1,j_2}=\rho>0$ for $(j_1,j_2)\in J_A$ and $\sigma_{j,j}=\nu^2 >0$ for $j=1,\ldots, p.$ Then 
\begin{align}
	\eqref{eq:ratiopower}  \simeq |J_A|\rho^a/(\sqrt{2a!\kappa_1^a}\nu^{2a} pn^{-a/2}).\label{eq:powerasymideal1}
\end{align}
We next show  how the order of the ``best" U-statistics  changes when the sparsity level $|J_A|$  varies. 
 To be specific of the meaning of ``best", we compare the $\rho$ values needed by different U-statistics   to achieve the same asymptotic power.  Particularly,   we fix $ {\mathrm{E}[\mathcal{U}(a)]}/{\sqrt{\mathrm{var}[\mathcal{U}(a)]}} $, i.e., \eqref{eq:powerasymideal1} to be  some constant $ {M}/{\sqrt{2}}$ for different $a$'s and the asymptotic power of each $\mathcal{U}(a)$   is 
$
	\eqref{eq:powerbound}=1-\Phi( z_{1-\alpha}- {M}/{\sqrt{2}} ).
$
Then by \eqref{eq:powerasymideal1}, the $\rho$ value such that $\mathcal{U}(a)$ attains the power above is 
\begin{align}
	\rho_a=\sqrt{\kappa_1}(a!)^{\frac{1}{2a}}\nu^2 ({Mp}/{|J_A|})^{\frac{1}{a}} n^{-\frac{1}{2}}. \label{eq:rhoform}
\end{align}   
By the definition in \eqref{eq:rhoform}, we compare the power of two U-statistics $\mathcal{U}(a)$ and $\mathcal{U}(b)$ with $a\neq b$ following the Criterion \ref{cricomp} below. 
\begin{criterion}\label{cricomp}
We say $\mathcal{U}(a)$ is ``better'' than $\mathcal{U}(b)$ if $\rho_a < \rho_b$.
\end{criterion}
\noindent Given values of $n, p, |J_A|$ and $M$,   \eqref{eq:rhoform} is a function of $a$.
Therefore, to find the ``best" $\mathcal{U}(a)$, it suffices to find the order, denoted by  $a_0$, that gives the smallest $\rho_a$ value in \eqref{eq:rhoform}.    We then have the following proposition  discussing the optimality among the U-statistics of finite orders in  \eqref{eq:originaluinvariance}.  
\begin{proposition}\label{prop:ordercompare}  Given $n, p, |J_A|$ and  any constant $ M \in (0,+\infty)$, we consider $\rho_a$ in  \eqref{eq:rhoform} as a function of integer $a$, then
\begin{enumerate}
\item[(i)]  when $|J_A| \geq {Mp}$, the minimum of  $\rho_a$ is achieved at $a_0=1$; 
\item[(ii)]  when $|J_A|<{Mp}$, the minimum of $\rho_a$ is achieved at some $a_0$, which increases as $Mp/|J_A|$ increases. 
\end{enumerate}	
\end{proposition}

By Proposition \ref{prop:ordercompare}, the order $a_0$ that attains the smallest value of $\rho_a$ depends on the value of  $Mp/|J_A|$ and does not have a closed form solution.  We   use numerical plots to demonstrate the relationship between $a_0$ and the sparsity level. Particularly, let  $|J_A|=p^{2(1-\beta)}$, where $\beta\in (0,1)$ denotes the sparsity level. 
 To have a better visualization,  we use $g(a)=\log (\rho_a n^{1/2}\kappa_1^{-1/2}\nu^{-2})=(1/2a) \log a!+a^{-1} \log(Mp^{2\beta-1})$  instead of $\rho_a$. We plot $g(a)$ curves  in Figure \ref{fig:gavalue} for each $\beta \in \{0.1,\ldots,0.9\}$ with $M=4$ and $p\in \{100,10000\}$. Other values of $M$ and $p$ are also taken, which give similar  patterns to Figure \ref{fig:gavalue} and are not presented. 
 
 \begin{figure}[!htbp]
  \centering
\includegraphics[width=1\linewidth]{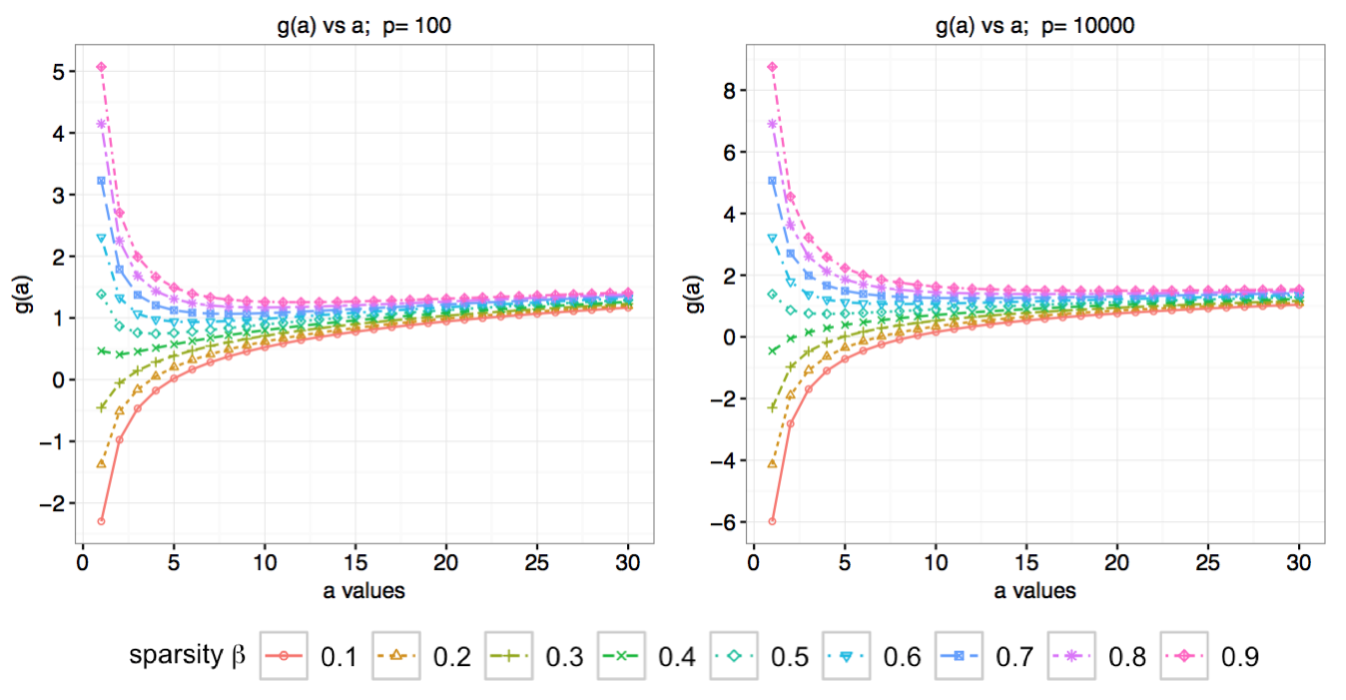}
 \caption{$g(a)$ versus $a$ with different sparsity level $\beta$ for $p=100, 10000$}
\label{fig:gavalue}
\end{figure}  
 
Figure \ref{fig:gavalue} shows that the $a_0$ such that $g(a)$ attains the smallest value increases when the sparsity level $\beta$ increases. In particular, when the sparsity level $\beta\leq 0.3$, that is, when $|J_A|$ is ``very" large and then $\boldsymbol{\Sigma}_A$ is ``very" dense, $g(a)$ has the smallest value at $a_0=1$. This is consistent with the conclusion in Proposition  \ref{prop:ordercompare} (i). When the  sparsity level $\beta$ is   between $0.4$ and $0.5$, we note that $a_0=2$ achieves the minimum of $g(a)$. This shows that when $|J_A|$ is ``moderately" large and   $\boldsymbol{\Sigma}_A$ is ``moderately" dense, $\mathcal{U}(2)$ is more powerful than  $\mathcal{U}(1)$. When the sparsity level $\beta>0.5$, we find that $a_0>2.$  This implies that when $|J_A|$ becomes smaller and   $\boldsymbol{\Sigma}_A$ becomes sparser, U-statistics of higher orders are more powerful. Additionally, we note that $a_0$   increases slowly as $\beta$ increases, which verifies   Proposition \ref{prop:ordercompare} (ii). 
Moreover, the curves converge  as $a$ increases and the differences of $g(a)$ for   large   $a$ values ($a\geq 6$)  are small.
This implies that  when selecting the range of considered orders of U-statistics, it suffices  to select an upper bound with  $a=6$ or $8$, which gives better or similar $\rho_a$ values to those larger $a$'s. 

 
In summary, when $|J_A|$ is large, i.e., $\boldsymbol{\Sigma}_A$ is dense, a small $a$ tends to obtain a smaller lower bound in terms of $\rho$. But when $|J_A|$ decreases, i.e., $\boldsymbol{\Sigma}_A$ becomes sparse, a  U-statistic of large finite order (or the maximum-type U-statistic as shown next) tends to obtain a smaller lower bound in $\rho$. This observation is consistent with the existing literature \cite{chen2011,cai2011,caima2013,caicovreview2017}.

Next, we proceed to examine the power of the maximum-type test statistic $\mathcal{U}(\infty)$, and compare it  with the U-statistics $\mathcal{U}(a)$ of finite $a$ defined in \eqref{eq:originaluinvariance}.  
By \cite{cai2011}, the rejection region for $\mathcal{U}(\infty)$ with significance level $\alpha$ is
\begin{eqnarray*}
	|\mathcal{U}(\infty)| \geq t_p:=n^{-1/2} \sqrt{4\log p - \log \log p -\log (8\pi)-2\log \log (1-\alpha)^{-1}}.
\end{eqnarray*}   
Note  $t_p \simeq 2\sqrt{\log p/n}$  and under alternative, the power for  $\mathcal{U}(\infty)$ is
\begin{eqnarray}
	P( |\mathcal{U}(\infty)| \geq t_p ). \label{eq:largestpowergoal}
\end{eqnarray}    
As discussed, we consider the alternatives satisfying Conditions \ref{cond:ellpmoment} and \ref{cond:rhoijconst},  $\sigma_{j_1,j_2}=\rho>0$ for $(j_1,j_2)\in J_A$, and $\sigma_{j,j}=\nu^2$ for $j=1,\ldots, p$. 
For simplicity, we assume $\mathrm{E}(\mathbf{x})=\boldsymbol{\mu}$ and  $\nu^2$ are given, and  focus on the simplified
\begin{align}
	\mathcal{U}(\infty)= \max_{1\leq j_1<j_2\leq p}  \Big| \nu^{-2}{n}^{-1}{{\sum}_{i=1}^n (x_{i,j_1}-\mu_{j_1})(x_{i,j_2}-\mu_{j_2}) } \Big| \label{eq:uinftyimplified}.
\end{align}  
 We show in the following proposition when  the power of $\mathcal{U}(\infty)$ asymptotically  converges to 1 or is strictly smaller than  1 under alternative.

\begin{proposition}\label{prop:exttestpowerana}
Under the considered alternative $\boldsymbol{\Sigma}_A$ above, suppose $\max_{j=1,\ldots,p}\allowbreak\mathrm{E}e^{t_0 |x_j-\mu_j|^{\varsigma}} < \infty$ for some $0 <\varsigma \leq 2$ and $t_0>0$,	    and $\log p= o( n^{\beta})$ with  $\beta=\varsigma/(4+\varsigma)$. Then for  \eqref{eq:uinftyimplified}, when $n,p\to \infty$,
\begin{enumerate}
	\item[(i)] there exists a constant $c_1 >2$ such that  if $\rho \geq  c_1 \sqrt{{\log p}/{n}}$, $\eqref{eq:largestpowergoal} \rightarrow 1$;
	\item[(ii)] there exists another constant  $0<c_2<2$ such that when $\rho\leq   c_2\sqrt{{\log p}/{n}}$,  Condition \ref{cond:ellpmoment} holds for $\kappa_1\leq 1$ and $|J_A|=o(1)p^{  \frac{2(1- {c_2}/{2} )^2}{\kappa_1+m} } (\log p)^{\frac{1}{2}-\frac{1}{2(\kappa_1+m)}} $ for some $m>0$,  we have $\eqref{eq:largestpowergoal} \leq \log (1-\alpha)^{-1}$.
 \end{enumerate}	
\end{proposition} 
  

 
Recall that Proposition \ref{prop:ordercompare} shows that there exists a finite integer $a_0$, such that $\rho_{a_0}$ is the minimum of \eqref{eq:rhoform}, and  $\rho_{a_0}$ is a lower bound of $\rho$ value for the finite-order U-statistics to achieve the given asymptotic power.  
With  Propositions   \ref{prop:ordercompare} and  \ref{prop:exttestpowerana}, we next  compare the finite-order U-statistics defined in \eqref{eq:originaluinvariance} with the maximum-type test statistic $\mathcal{U}(\infty)$.

\begin{proposition}{\label{col:maxordercompare}}
 Under the conditions of Theorem \ref{thm:cltalternative} and Proposition \ref{prop:exttestpowerana}, for any finite integer $a$, there exist constants $c_1$ and $c_2$ such that when $p$ is sufficiently large,
 \begin{enumerate}
 \item[(i)] For any $M$, when  
$
|J_A|< c_1^{-{a}} {(a!)}^{\frac{1}{2}} \kappa_1^{\frac{a}{2}} (\log p)^{-\frac{a}{2} } { {Mp} } ,	
$  $\mathcal{U}(\infty)$ has higher asymptotic power than $\mathcal{U}(a)$. 
\item[(ii)] When $M$ is big enough and $
|J_A|> c_2^{-{a}} {(a!)}^{\frac{1}{2}} \kappa_1^{\frac{a}{2}} (\log p)^{-\frac{a}{2} } { {Mp} }, 	
$  $\mathcal{U}(a)$ has higher asymptotic power than $\mathcal{U}(\infty)$.
 \end{enumerate}
\end{proposition}

From Proposition \ref{prop:ordercompare}, we know when $Mp/|J_A|=O(1)$, there exists a finite $a_0$ such that $\mathcal{U}(a_0)$ is the ``best'' among all the finite-order U-statistics; in this case, Proposition \ref{col:maxordercompare} (ii) 
further indicates that 
$\mathcal{U}(a_0)$  has higher asymptotic power than $\mathcal{U}(\infty)$. Specifically, if $Mp/|J_A|<1$, $a_0=1$, then $\mathcal{U}(1)$ is the ``best'' and its lowest detectable order of $\rho$ is $\Theta(p|J_A|^{-1}n^{-1/2})$.  More interestingly, when $\boldsymbol{\Sigma}_A$ is moderately dense or moderately sparse with $Mp/|J_A|>1$ and bounded, some U-statistic of finite order $a_0> 1$ would become the ``best''. By Figure \ref{fig:gavalue},  the value of $a_0$  
   increases as $\boldsymbol{\Sigma}_A$ becomes denser. 
   On the other hand, when $\boldsymbol{\Sigma}_A$ is ``very'' sparse with $
|J_A|< c_1^{-{a_0}} {(a_0!)}^{\frac{1}{2}} \kappa_1^{\frac{a_0}{2}}  (\log p)^{-\frac{a_0}{2} }   { {Mp} }  
$, 
 $\mathcal{U}(\infty)$ is the ``best" and its lowest detectable order of $\rho$ is  $\Theta(\sqrt{\log p/n})$.

\begin{remark}\label{rm:disthreshold}
The above power comparison results are under the constructed family of U-statistics. We note that additional formulation may further enhance the test power. For instance, \cite{chen2014two,zhong2013} showed that an adaptive thresholding in certain $\ell_{p}$-type test statistics can achieve  high power under the alternatives with sparse and faint signals. It is of interest to  incorporate the adaptive thresholding into the constructed family of U-statistics, which is left for future study.       	
\end{remark}

\begin{remark}\label{rm:notallpositive}
The analysis above focuses on the ideal case where the nonzero off-diagonal entries of $\boldsymbol{\Sigma}_A$ are the same for illustration. When these entries of $\boldsymbol{\Sigma}_A$  are different, similar analysis still applies by Theorem 	\ref{thm:cltalternative} for general covariance matrices. In specific, the asymptotic power of $\mathcal{U}(a)$  depends on the mean variance ratio \eqref{eq:ratiopower} and $\rho_a=\sqrt{\kappa_1}n^{-1/2}(a!)^{1/2a}\times (M \sum_{j=1}^p \sigma_{j,j}^a/\sum_{1\leq j_1,j_2\leq p}\sigma_{j_1,j_2}^a)^{1/a}$. We can then obtain conclusions similar to  Propositions \ref{prop:ordercompare}--\ref{col:maxordercompare}. One interesting case is  when $\boldsymbol{\Sigma}_A$ contains both positive and negative entries; the same analysis applies for even-order U-statistics, since $\sigma_{j_1,j_2}^a$'s are all non-negative for even $a$.  On the other hand, the odd-order U-statistics would have low power, since $\sum_{1\leq j_1\neq j_2\leq p} \sigma_{j_1,j_2}^a$ could be small due to the cancellation of positive and negative $\sigma_{j_1,j_2}^a$'s. We have conducted simulations when the  nonzero $\sigma_{j_1,j_2}$'s are different in Section \ref{sec:simulation}, and the results exhibit consistent  patterns as expected. 
\end{remark}

\subsection{Application to Adaptive Testing \&  Computation} \label{sec:computtest}

\paragraph{Adaptive Testing}

Power analysis in Section \ref{sec:powerana} shows that when the sparsity level of the alternative changes, the test statistic that achieves the highest power could vary. However,  since the truth is often unknown in practice,  it is unclear which test statistic should be chosen. Therefore, we develop an adaptive testing procedure by   combining the information from U-statistics of different orders, which would yield high power against various alternatives. 

In particular, we propose to combine the U-statistics through their $p$-values, which is widely used in  literature \cite{fishermethod,pan2014powerful,yu2009pathway}. One popular method is the  minimum combination, whose idea is to take the minimum $p$-value to approximate the maximum power \cite{pan2014powerful,yu2009pathway,xu2016adaptive}. Specifically, let $\Gamma$ be a candidate set of the orders of U-statistics, which   contains both finite values and $\infty$. We compute $p$-values $p_a$'s of the U-statistics $\mathcal{U}(a)$'s satisfying $a \in \Gamma$. The minimum combination takes the statistic $T_{\mathrm{adpUmin}}=\min\{p_a: a\in \Gamma \}$ and has the asymptotic $p$-value $p_{\mathrm{adpUmin}}=1-(1-T_{\mathrm{adpUmin}})^{|\Gamma|}$, where $|\Gamma|$ denotes the size of the candidate set $\Gamma$. We reject $H_0$ if $p_{\mathrm{adpUmin}}<\alpha$. Under $H_0$, $p_a$'s are asymptotically independent and uniformly distributed by the theoretical results in Section \ref{sec:aymindptusec}.  The type \RNum{1} error is asymptotically  controlled as  $P(p_{\mathrm{adpUmin}}< \alpha)=P(\min_{a\in \Gamma} p_a < p_{\alpha}^*)\to \alpha,$ where $p_{\alpha}^* = 1- (1-\alpha)^{1/|\Gamma|}$. Since  $P(\min_{a\in \Gamma} p_a < p_{\alpha}^* ) \geq P( p_a < p_{\alpha}^*)$, the  power of the  adaptive test goes to 1 if there exists $a\in \Gamma$ such that the power of $\mathcal{U}(a)$ goes to 1. We note that the power of the adaptive test is not necessarily higher than that of all the  U-statistics.  This is because the power of $\mathcal{U}(a)$ is $P(p_a<\alpha)$, and is different from $P(p_a<p_{\alpha}^*)$ since $p_{\alpha}^*<\alpha$ when $|\Gamma|>1$.  Based on our extensive simulations, we find that the adaptive test is  usually close to or even higher than the maximum power of the U-statistics.

\begin{remark}
Fisher's method \cite{fishermethod} is another popular method for combining independent $p$-values. It has the test statistic $T_{\mathrm{adpUf}}=-2\sum_{k=1}^{|\Gamma|}\log p_k$, which converges to $\chi^2_{2|\Gamma|}$ under $H_0$.  By our simulations,  the minimum combination and Fisher's method are generally comparable, while Fisher's method has higher power under several cases. Moreover, we can also use  other  methods to combine the  $p$-values, such as higher criticism \cite[][]{donoho2004,donoho2015higher}. We leave the study of how to efficiently combine the $p$-values for future research. 
\end{remark}

We select the candidate set $\Gamma$ by the power analysis in Section \ref{sec:powerana}. 
We would recommend including   $\{1,2,\ldots,6,\infty\}$, which can be powerful against a wide spectrum of alternatives. In particular, by Propositions \ref{prop:ordercompare} and \ref{col:maxordercompare}, we include $a=1,2$ that are powerful against dense signals;  $a=\infty$ that is powerful against sparse signals; and also $a = \{3,\ldots,6\}$ for the moderately dense and moderately sparse signals.
By Figure \ref{fig:gavalue}, it   generally suffices to choose finite $a$ up to 6--8, which often give similar/better performance to/than larger $a$ values. The simulations in Section \ref{sec:simulation} confirm the good performance of this choice of  $\Gamma$; and the proposed adaptive test appears to well approximate the ``best" performance even when $\Gamma$ may not always contain the unknown ``optimal"  U-statistics. 

We would like to mention that the adaptive procedure can be generalized to other testing problems, as long as similar theoretical properties are given, such as the examples in Section \ref{sec:extension}.


\paragraph{Computation}

Next we discuss the computation in the adaptive testing. 
 A direct calculation following the form of $\mathcal{U}(a)$ in \eqref{eq:originaluinvariance} and $\mathbb{V}(a)$ in \eqref{eq:varestnew} would be computationally expensive for large  $a$ with a cost of $O(p^2n^{2a})$. To address this issue, we introduce a method that can reduce the cost.


We first consider a simplified setting when   $\mathrm{E}(x_{i,j})=0$ to illustrate the idea.   As discussed in Remark \ref{rm:leadingteststat},  we examine $\tilde{\mathcal{U}}(a)$ defined in \eqref{eq:originleadingterm}. Let $\mathcal{L}=\{(j_1,j_2): 1\leq j_1 \neq j_2 \leq p \}$ denote the set of index tuples, and for each index tuple $l=(j_1,j_2)\in \mathcal{L}$, define $s_{i,l}=x_{i,j_1}x_{i,j_2}$. Note that $\tilde{\mathcal{U}}(a)=(P^n_a)^{-1} \sum_{l\in \mathcal{L}}\mathcal{U}_l (a)$, where $\mathcal{U}_l (a) = \sum_{1\leq i_1 \neq \cdots \neq i_{a}\leq n} \prod_{k=1}^a s_{i_k, l}.$
Calculating $\mathcal{U}_l(a)$ directly is of order $O(n^a)$. We then focus on reducing the computational cost of  $\mathcal{U}_l (a)$.
  For $l\in \mathcal{L}$ and  finite integers $t_1,\ldots, t_k$,  define 
  \begin{align}
	V_l^{(t_1,\ldots, t_k)}  =\prod_{r=1}^ k \Big(\sum_{i=1}^n s_{i,l}^{t_r}\Big),\quad  
 U_l^{(t_1,\cdots,t_k)}   =  \sum_{1\leq i_1\neq \ldots \neq i_k \leq n} \prod_{r=1}^k s_{i_1,l}^{t_r}. \label{eq:uvorderstat}
\end{align} 
 We can see that $\mathcal{U}_l (a) =U_l^{\mathbf{1}_{a}}$ with $\mathbf{1}_a$ being an $a$-dimensional vector of all ones,  and $U_l^{(a)}=V_l^{(a)}$ for any finite  integer $a$. 
 To reduce the computational cost of $\mathcal{U}_l(a)$, the main idea is to obtain $U_l^{\mathbf{1}_{a}}$ from $V_l^{(t_1,\ldots, t_k)}$, whose   computational cost is $O(n)$. In particular,  $\mathcal{U}_l(a)$ can be attained iteratively from $V_l^{(t_1,\ldots, t_k)}$ based on the following equation 
 \begin{align} \label{eq:2}
	U_l^{(k,\mathbf{1}_{r-k})} = V_l^{(k)} \times U_l^{\mathbf{1}_{r-k}} - (r-k) \times U_l^{(k+1, \mathbf{1}_{r-k-1})},
\end{align} 
which follows from the definitions. Algorithm \ref{Algo2} below summarizes the  steps.

\begin{algorithm}[!htbp]
 \KwData{$s_{i,l}$ \ $(1\leq i \leq n$, $l \in \mathcal{L})$. }
 \KwResult{$\tilde{\mathcal{U}}(a)$.}
 \For{$l \in \mathcal{L}$}{
  Compute and store $V_l^{(k)}=U_{l}^{(k)}= \sum_{i=1}^n s_{i,l}^k$, $(k=1,\cdots,a)$ during the algorithm\;
  $U_l^{\mathbf{1}_{1}}=V_l^{(1)}$ , $U_l^{\mathbf{1}_2}=U_l^{\mathbf{1}_{1}} V_l^{(1)}-U_l^{(2)}$\;
 \While{$3 \leq r \leq a$}{
  $T_l=U_l^{(r)}$\\
  \For{$k \leftarrow r-1$ \KwTo $1$}{
  	 $T_l =  V_l^{(k)} \times U_l^{\mathbf{1}_{r-k}} - (r-k) \times T_l $
  }
  $U_l^{\mathbf{1}_r}=T_l$
 } } $\tilde{\mathcal{U}}(a)=(P^n_a)^{-1} \sum_{l\in \mathcal{L}}U_l^{\mathbf{1}_a}$
\vspace{5pt}
 \caption{Iterative Computation Implementation}
 \label{Algo2}
\end{algorithm}

We illustrate the idea of the algorithm by some  examples. By definition,   $U_l^{(1)}=V_l^{(1)}$, which can be computed with cost    $O(n)$. Next consider in  \eqref{eq:2}, if $r=2$ and $k=1$, then
$
	U_l^{(1,1)} = V_l^{(1)} \times U_l^{{(1)}} - (2-1) \times U_l^{(2)}=V_l^{(1)}  \times V_l^{(1)} -V_l^{(2)} ,
$ which yields ${U}_l^{\mathbf{1}_2}$  with cost   $O(n)$. For $U_l^{\mathbf{1}_3}$, we first take $r=3$ and $k=2$ in \eqref{eq:2}, then with cost $O(n)$, we have
$
	U_l^{(2,{1})}= V_l^{(2)} \times U_l^{{(1)}} -  U_l^{(3)}=V_l^{(2)} \times V_l^{(1)} -  V_l^{(3)},
$ as $V_l^{(k)}=U_l^{(k)}$ by the definition. Given  $U_l^{\mathbf{1}_{2}}$ and  $U_l^{(2,{1})}$, we  obtain   
$
	 U_l^{(1,\mathbf{1}_{2})} = V_l^{(1)} \times U_l^{\mathbf{1}_{2}} - 2 \times U_l^{(2, \mathbf{1}_{1})}.
$ Thus $U_l^{\mathbf{1}_3}$ is also computed with cost   $O(n)$.  
Iteratively, for any finite integer $a$,  we can obtain $U_l^{\mathbf{1}_a}$  from $V_l^{(t_1,\ldots, t_k)}$ whose computational cost is $O(n)$. More closed form formulae representing $U_l^{\mathbf{1}_a}$ by $V_l^{(t_1,\ldots, t_k)}$ are given in Section \ref{sec:formulaevu} of  \ref{suppA}. 


Algorithm \ref{Algo2} reduces the computational cost of $\tilde{\mathcal{U}}(a)$  from $O(p^2n^a)$ to $O(p^2n)$. Its idea is general and can  be extended to compute other different U-statistics by changing the input $s_{i,l}$. In particular, the variance estimator $\mathbb{V}(a)$ can be computed with cost $O(p^2n)$ by specifying $s_{i,l}=(x_{i,j_1}-\bar{x}_{j_1})^2(x_{i,j_2}-\bar{x}_{j_2})^2$, for each $l\in \mathcal{L}=\{(j_1,j_2):1\leq j_1\neq j_2\leq p\}$. Then $\mathbb{V}(a)=2a!(P^n_a)^{-2}\sum_{l\in \mathcal{L}}\sum_{1\leq i_1\neq \ldots \neq i_a\leq n}\prod_{k=1}^a s_{i_k,l}$ and the Algorithm \ref{Algo2} can be applied.  
Moreover, when $\mathrm{E}(x_{i,j})$ is unknown, $\mathcal{U}(a)$ can still be computed with cost $O(p^2n)$ using the iterative method similar to Algorithm \ref{Algo2}. The   details are provided in Section \ref{sec:u2compute} of \ref{suppA}.

\section{Simulations and Real Data Analysis}\label{sec:twosimdata}

\subsection{Simulations} \label{sec:simulation}
 We conduct simulation studies to evaluate the performance of the proposed adaptive testing procedures, and  investigate the relationship between the power and sparsity levels.  For one-sample covariance testing discussed in Section \ref{sec:mainexamsec}, we
 generate $n$ i.i.d. $p$-dimensional $\mathbf{x}_i$ for $i=1,\ldots, n$, and consider the following five simulation settings. 


\textit{Setting 1:} $\mathbf{x}_i$ has $p$  i.i.d. entries of $\mathcal{N}(0,1)$ and $\mathrm{Gamma}(2,0.5)$ respectively. Under each case, we take $n=100$ and $p\in\{50, 100, 200, 400, 600, 800, 1000\}$  to verify the theoretical results under $H_0$ and the validity of the adaptive test across different $n$ and $p$ combinations.

For the following settings 2--5, we generate $\mathbf{x}_i$ from multivariate Gaussian distributions with mean zero and different covariance matrices $\boldsymbol{\Sigma}_A$'s. 

\textit{Setting 2:} $\boldsymbol{\Sigma}_A=(1-\rho)I_p+\rho \mathbf{1}_{p,k_0}\mathbf{1}_{p,k_0}^{\intercal},$ where $\mathbf{1}_{p,k_0} $ is a $p$-dimensional vector with the first $k_0$ elements one and the rest zero.   We take $(n,p)\in \{(100,300), (100, 600), (100, 1000)\}$, and study the power  with respect to different signal sizes $\rho$ and sparsity levels $k_0$. 

\textit{Setting 3:} The diagonal elements of $\boldsymbol{\Sigma}_A$ are all one and $|J_A|$ number of  off-diagonal elements are $\rho$ with random positions. We take $(n,p)\in \{(100, 600), (100, 1000)\}$ and let the signal size $\rho$ and sparsity level $|J_A|$ vary to examine how the power changes accordingly. 

\textit{Setting 4:} The diagonal elements of $\boldsymbol{\Sigma}_A$ are all one and $|J_A|$ number of  off-diagonal elements are uniformly generated from $(0,2\rho)$ with random positions. We take $(n,p)=(100, 1000)$ and similarly let the signal size $\rho$ and sparsity level $|J_A|$ vary to examine how the power changes accordingly. 
 
\textit{Setting 5:}  We consider the multivariate models in \cite{chen2011}. Specifically, for each $i=1,\ldots, n$, $\mathbf{x}_i= \Xi \mathbf{z}_i + \boldsymbol{\mu}$,  
where $\Xi$ is a matrix of dimension $p\times m$, and $\mathbf{z}_i$'s are i.i.d. Gaussian or Gamma random vectors. Under null hypothesis, $m=p$, $\Xi=I_p$ $\boldsymbol{\mu}=2 \mathbf{1}_p$; under alternative hypothesis, $m=p+1$, $\Xi=(\sqrt{1-\rho} I_p, \sqrt{2\rho} \mathbf{1}_p)$,  $\boldsymbol{\mu}=2(\sqrt{1-\rho}+\sqrt{2\rho})\mathbf{1}_p$. We also take the $n$ and $p$ combination in \cite{chen2011} with $(n,p)\in \{(40,159), (40,331), (80,159), (80,331), (80,642)\}$.

We compare several methods in the literature, including both maximum-type and sum-of-squares-type tests. In particular, the maximum-type test statistic in \citet{jiang2004} is taken as $\mathcal{U}(\infty)$ in this framework. Since the convergence in  \cite{jiang2004} is known to be slow, we use permutation to approximate the distribution in the simulations. In addition, we  consider some  sum-of-squares-type methods. Specifically, we examine  the identity and sphericity tests in \citet{chen2011}, which are denoted as ``Equal" and ``Spher", respectively.  We also compare the methods in  \citet{ledoit2002} and \citet{schott2007test}, which are referred to as ``LW" and ``Schott", respectively.

To illustrate, 
  Figure \ref{fig:alternsparsfigure}   summarizes the numerical results for the setting 3 when $n=100$ and $p=1000$. All the results are based on  1000 simulations at the $5\%$ nominal significance level.   In  Figure \ref{fig:alternsparsfigure}, we present the power of single U-statistics with orders in $\{1,\ldots, 6,\infty\}$.  ``adpUmin" and  ``adpUf" represent the  results of the adaptive testing procedure using the minimum combination and Fisher's method in Section  \ref{sec:powerana}  respectively.  
The simulation results show that the type \RNum{1} error rates of the U-statistics and adaptive test are well controlled under $H_0$.
In addition, Figure \ref{fig:alternsparsfigure}   exhibits several patterns that are consistent with the power analysis in Section \ref{sec:powerana}.  First, it shows that among the U-statistics, when $|J_A|$ is very small,  $\mathcal{U}(\infty)$ performs best; and when $|J_A|$ increases, the performances of some U-statistics of finite orders catch up. For instance, when $|J_A|=100$, $\mathcal{U}(6)$ and $\mathcal{U}(\infty)$ are similar and are better than the other U-statistics; when $|J_A|=400$, $\mathcal{U}(4)$ and $\mathcal{U}(5)$ are similar and better than the other U-statistics.    When $\boldsymbol{\Sigma}_A$ is relatively dense, $\mathcal{U}(2)$ and $\mathcal{U}(1)$ become more powerful. Particularly, when $|J_A|=1600$, $\mathcal{U}(2)$ is powerful; when $|J_A|$ becomes larger, such as when $|J_A|=3200$,  $\mathcal{U}(1)$ is overall the most powerful. Second, Figure \ref{fig:alternsparsfigure}  shows that ``LW", ``Schott", ``Equal", ``Spher" and $\mathcal{U}(2)$  perform similarly under various cases. In particular, these methods are not powerful when the alternative is  sparse but becomes more powerful when the alternative gets denser. This is because  they are all sum-of-squares-type statistics that target at  dense alternatives.  Third and importantly, the two adaptive tests ``adpUmin" and ``adpUf" maintain high power across different settings. Specifically, they perform better than most single U-statistics: their powers are usually close to or even higher than the best single U-statistic. Moreover, ``adpUmin" and ``adpUf"  generally have higher power than the compared existing  methods. 
 We also note that ``adpUf" overall performs better than ``adpUmin" in this simulation setting. 
In summary, Figure \ref{fig:alternsparsfigure}  demonstrates the relationship between the sparsity levels of alternatives and the power of the tests, confirming the theoretical conclusions in Section \ref{sec:powerana}. Notably, the proposed adaptive testing procedure is powerful against a wide range of alternatives, and thus advantageous in practice when the true  alternative is unknown.


\begin{figure}[h!]
    \centering
    \includegraphics[width=0.48\textwidth,height=0.28\textheight]{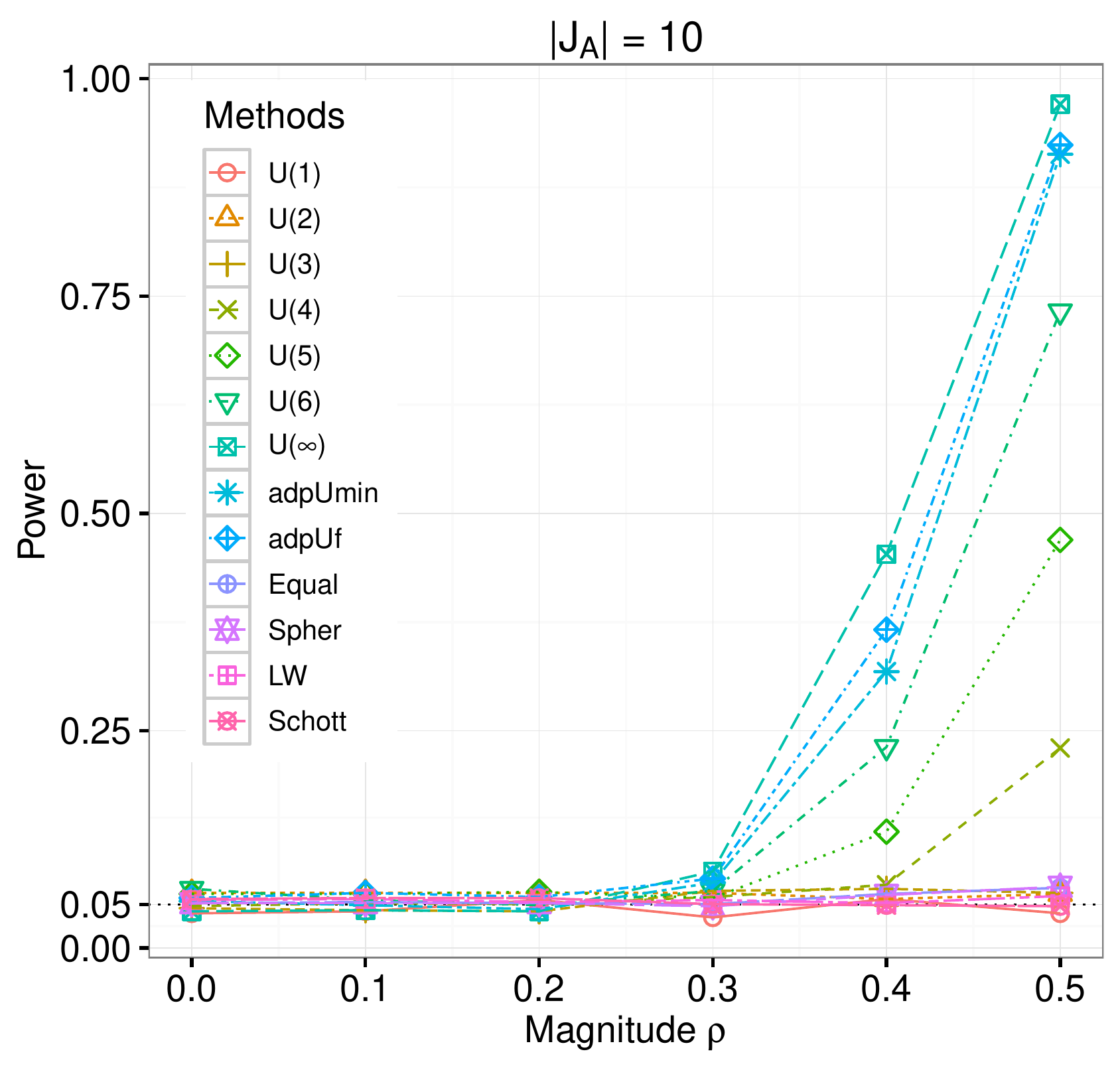} \quad
       \includegraphics[width=0.48\textwidth,height=0.28\textheight]{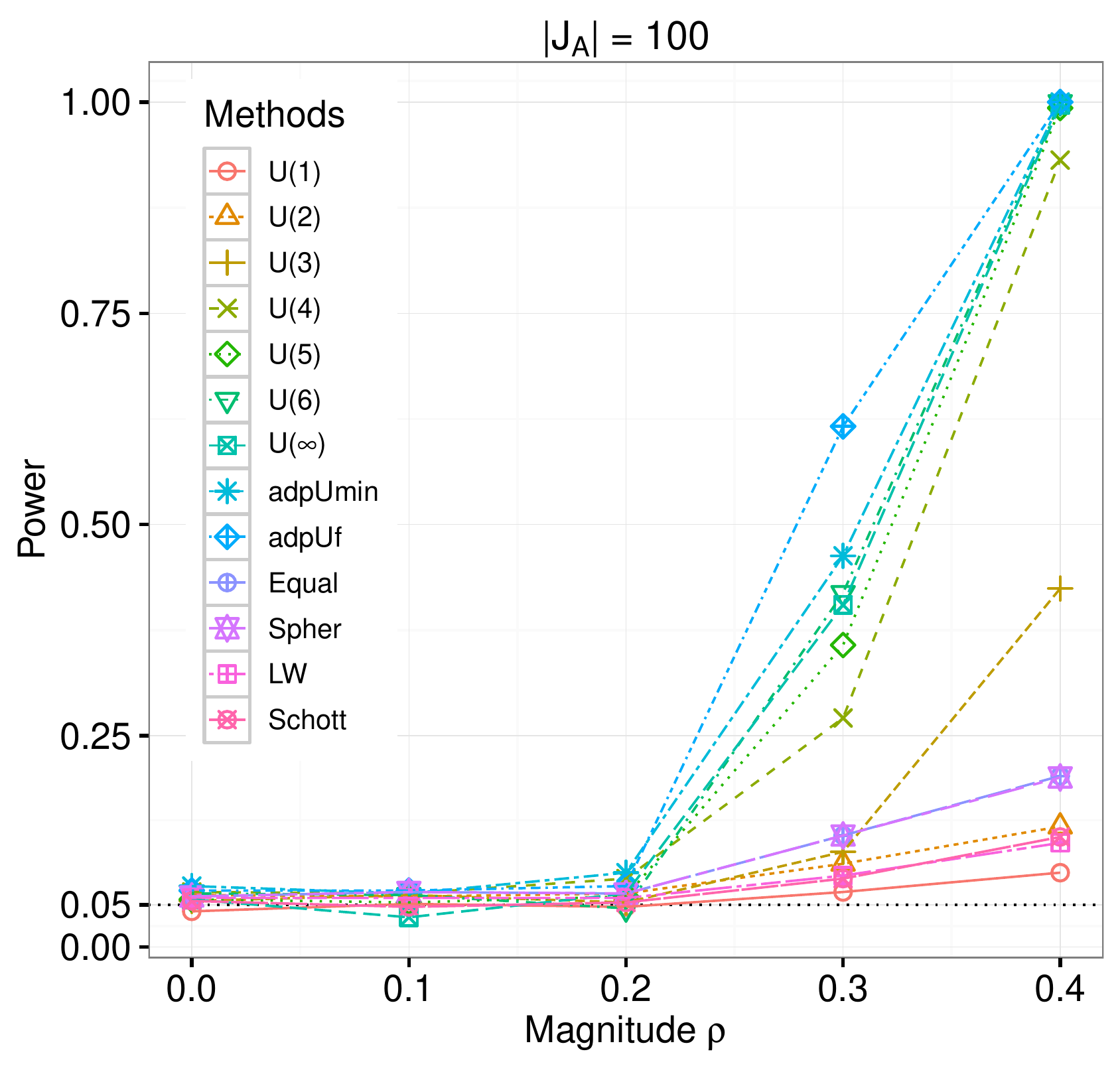} \\ 
           \includegraphics[width=0.48\textwidth,height=0.28\textheight]{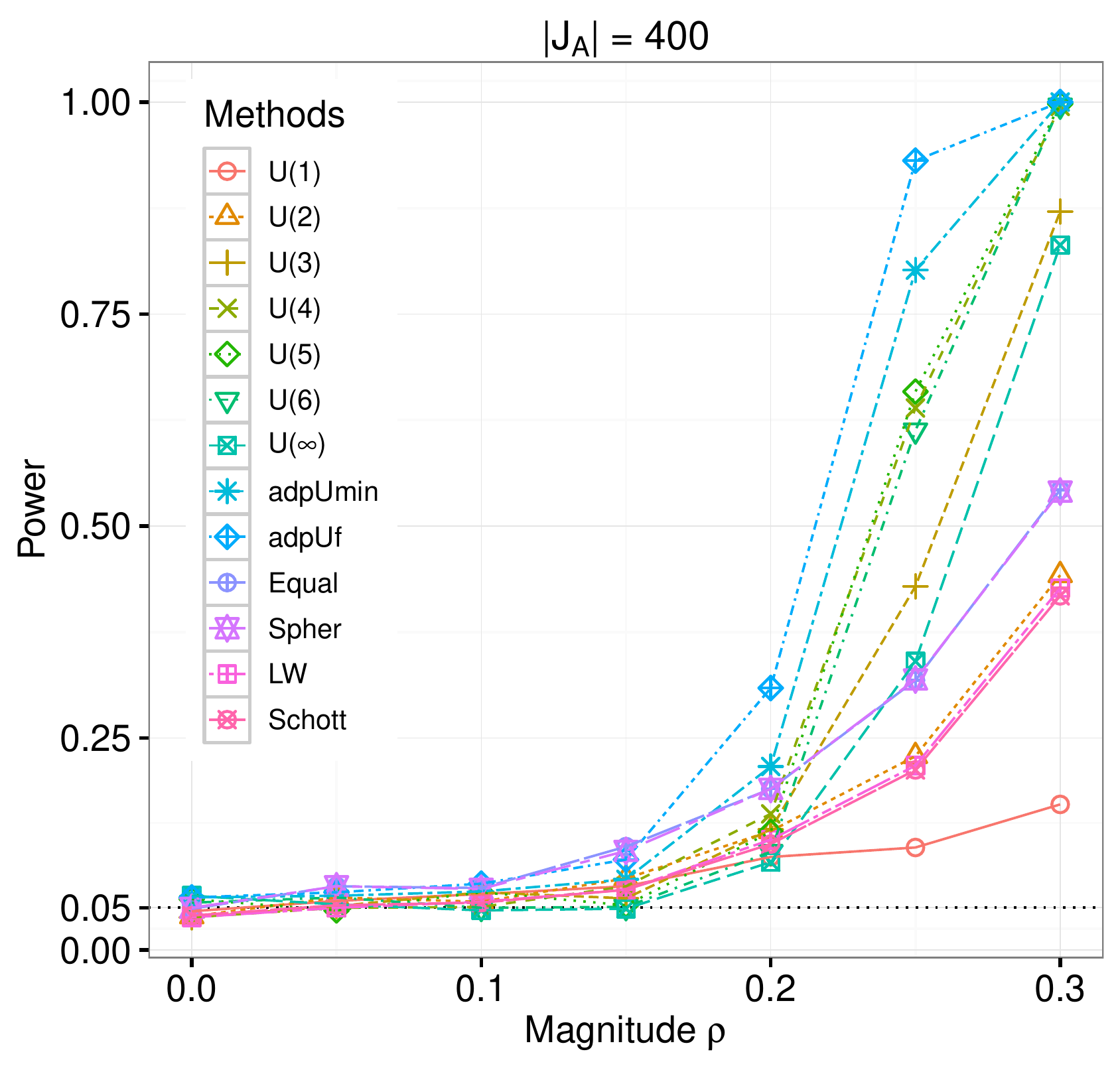} \quad
           \includegraphics[width=0.48\textwidth,height=0.28\textheight]{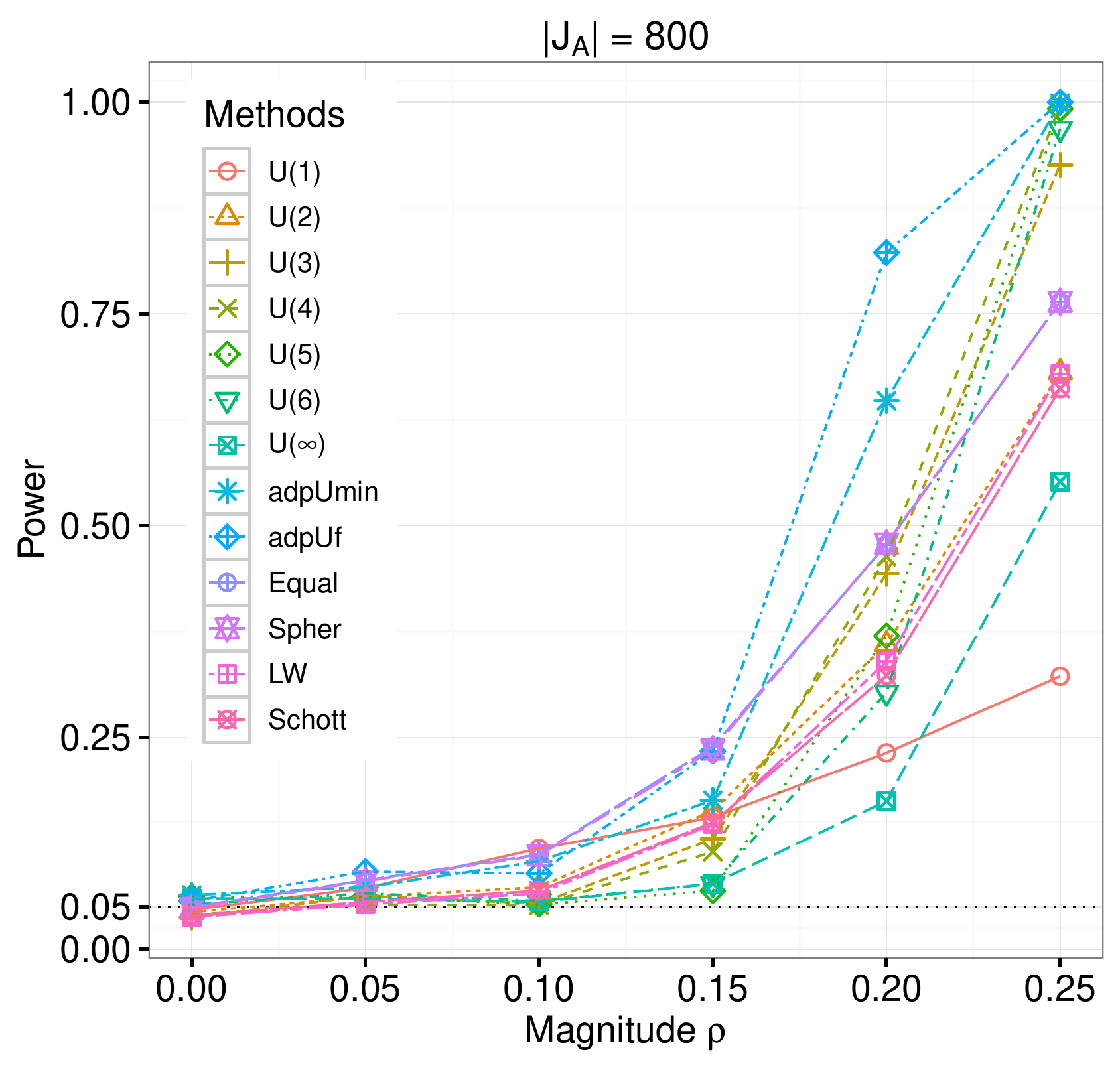} \\
           \includegraphics[width=0.48\textwidth,height=0.28\textheight]{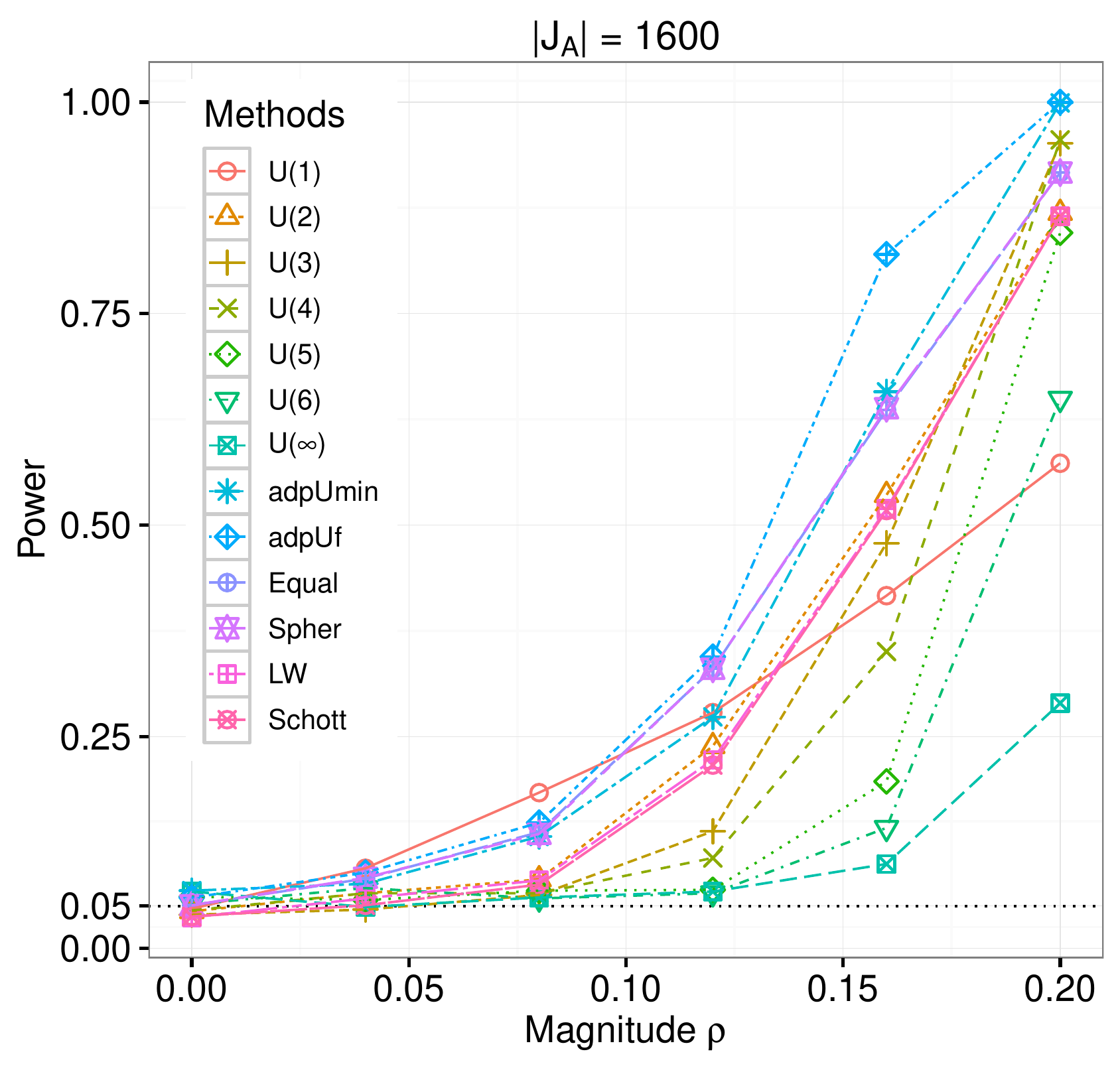} \quad
       \includegraphics[width=0.48\textwidth,height=0.28\textheight]{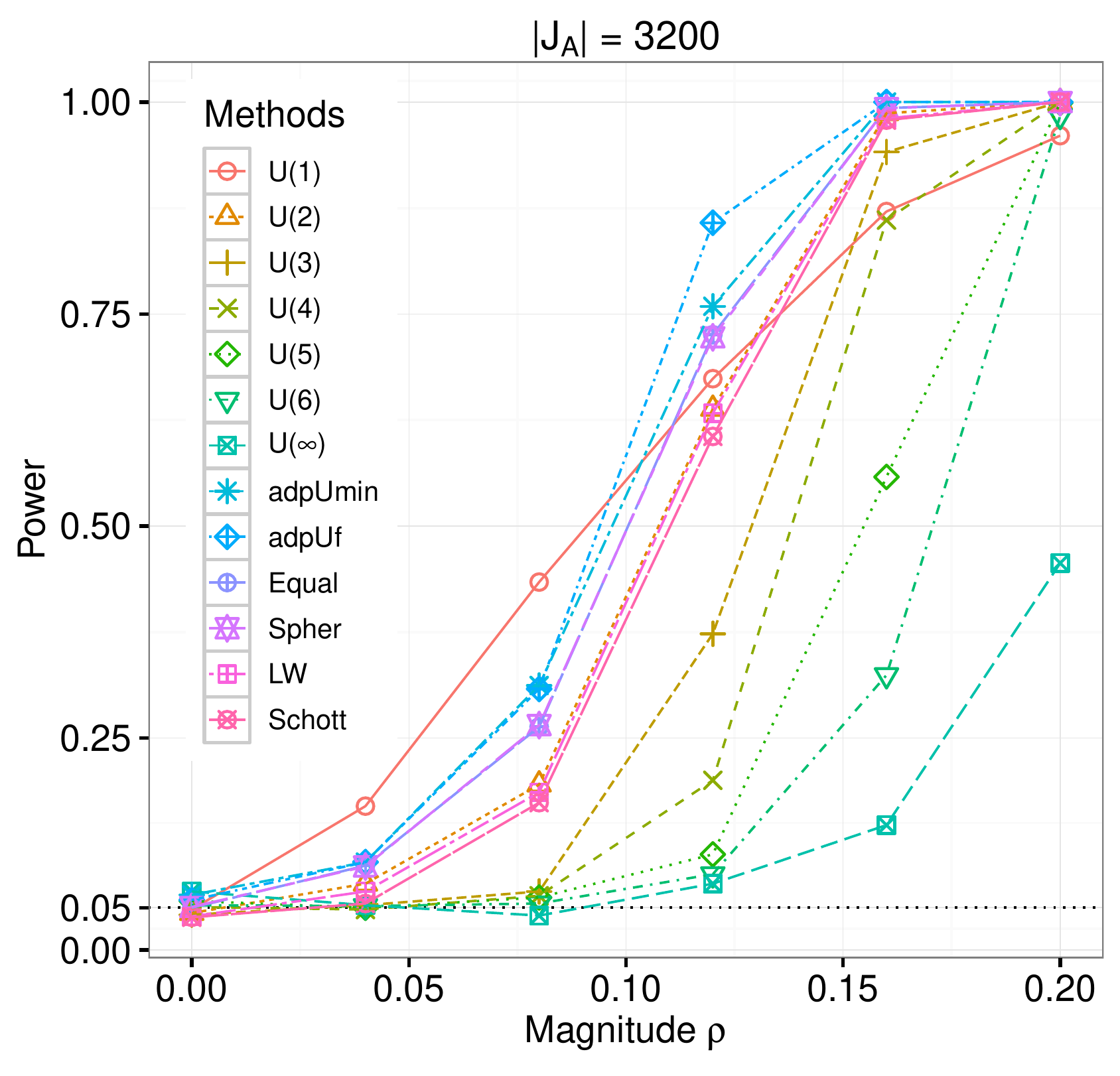} 
    \caption{Power comparison.
    }
    \label{fig:alternsparsfigure}
\end{figure}

Due to the space limitation, we provide other extensive numerical studies in \ref{suppA} Section \ref{sec:suppsimu}. The conclusions are similar to those of Figure \ref{fig:alternsparsfigure},  and consistent with the theoretical results in Section \ref{sec:powerana}. In particular, the results show that the empirical sizes of the tests are close to the nominal level, suggesting the good finite-sample performance of the asymptotic approximations. Moreover, under highly dense alternatives with only non-negative entries in the covariance matrix,  $\mathcal{U}(1)$ is the most powerful one among the $\mathcal{U}(a)$'s and the other tests in \cite{ledoit2002,schott2007test,chen2011}, in agreement with the results in Propositions \ref{prop:ordercompare} and \ref{col:maxordercompare}. 
 Furthermore, the proposed adaptive testing procedures often have higher power than most single U-statistics.  


 \subsection{Real Data Analysis} \label{sec:realdata}
 
Alzheimer's disease (AD) is the most prevalent neurodegenerative disease \citep{prince2013global} and is ranked as the sixth leading cause of death in the US  \citep{xu2018deaths}. Every 65 seconds, someone in the US develops AD \citep{alzheimer20182018}. To advance our understanding of AD, the Alzheimer's Disease Neuroimaging Initiative (ADNI) was started in 2004, collecting extensive genetic data for both healthy individuals and AD patients. To gain insight into the genetic mechanisms of AD, one can test a single SNP a time. However, due to a relatively small sample size of the ADNI data,  scanning across all SNPs failed to identify any genome-wide significant SNP (with $p$-value $<5\times 10^{-8}$)\citep{kim2016powerful}. To date, the largest meta-analysis of more than 600,000 individuals identified 29 significant risk loci \citep{jansen2019genome} and can only explain a small proportion of AD variance. On the other hand, a group of functionally related genes as annotated in a biological pathway are often involved in the same disease susceptibility and progression \citep{heinig2010trans}. Thus, pathway-based analyses, which jointly analyze a group of SNPs in a biological pathway, have become increasingly popular. We retrieve a total of 214 pathways from the KEGG database \citep{kanehisa2010kegg} for the subsequent analysis. 
 


Although pathway-based analyses with KEGG pathways are common in real studies, formally testing the correlations of the genes in a KEGG pathway has been largely untouched. Here, we apply our method and other competing methods in \cite{chen2011} to test if all the genes in a pathway have correlated gene expression levels. Perhaps as expected, all methods reject the null hypothesis for all pathways with highly significant $p$-values, since the KEGG pathways are constructed to include only the genes with similar function into the same pathway \citep{kanehisa2010kegg}, while similar function often implies co-expression (and vice versa). To compare the performance of the different tests, for each pathway we randomly select 50 subjects and restrict our analysis to pathways of at least 50 genes, leading to 103 pathways for the following analysis. Then we perturb the data by shuffling the gene expression levels of randomly selected $  100(1- \alpha)\%$ genes in a pathway before applying each test. Figure \ref{fig:realdataana} shows the performance of the tests with two significance cutoffs, where ``$\mathcal{U}(2)$" represents the single $\mathcal{U}(2)$ statistic, ``adpU" represents our proposed adaptive testing procedure using the minimum combination with candidate U-statistics of orders in $\{1,\ldots, 6, \infty\}$, and ``Equal" and ``Spher" represent the identity and sphericity tests in \cite{chen2011} respectively. Because all pathways are highly significant with all samples, we can treat all pathways as the true positives. Due to the adaptiveness of our proposed testing procedure, ``adpU" identifies more significant pathways than the competing methods across all the levels of data perturbation (mimicking the varying sparsity levels of the alternatives). 


\begin{figure}[!htbp]
    \centering
    \includegraphics[width=0.4\textwidth,height=0.25\textheight]{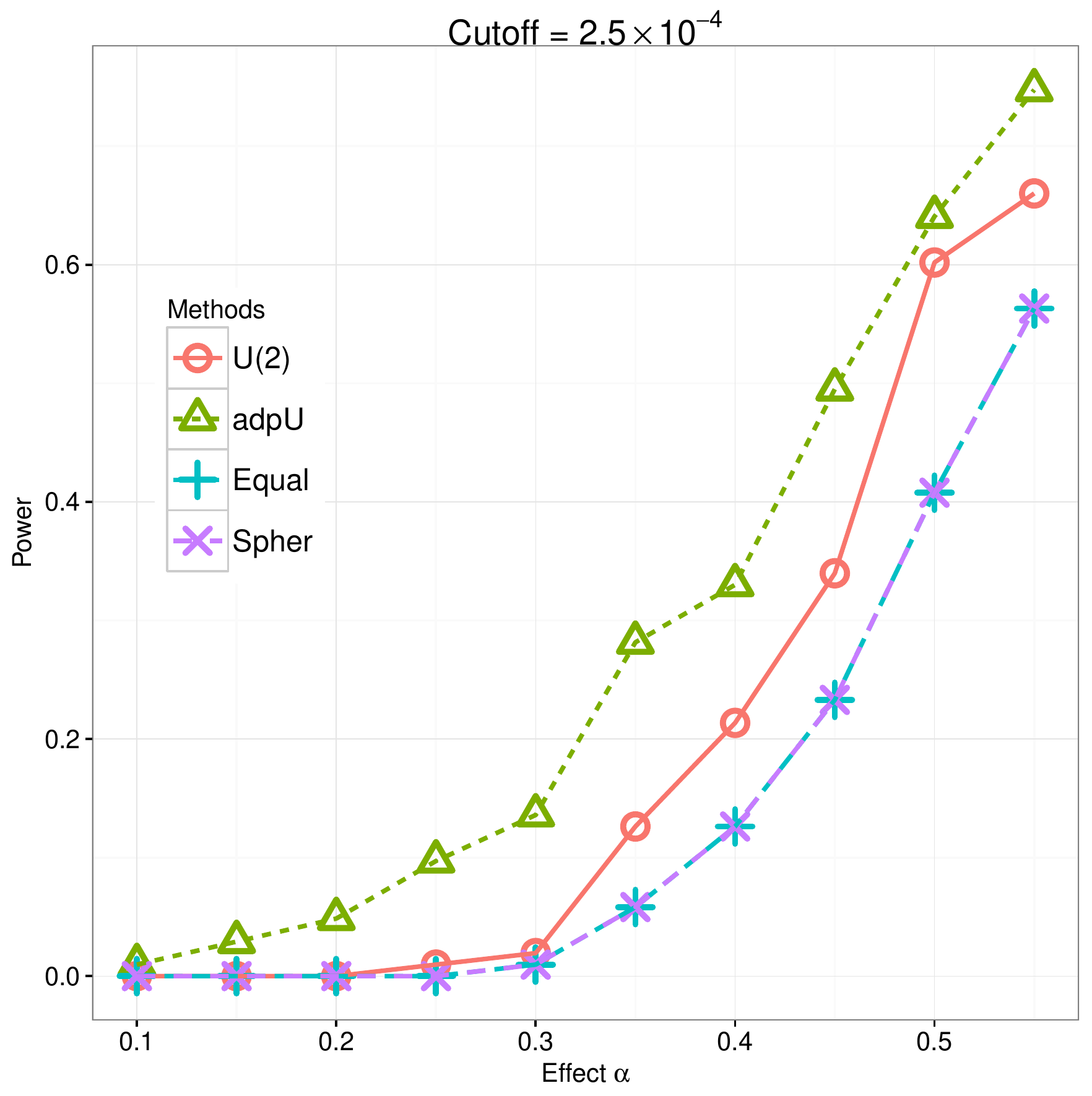} \qquad
       \includegraphics[width=0.4\textwidth,height=0.25\textheight]{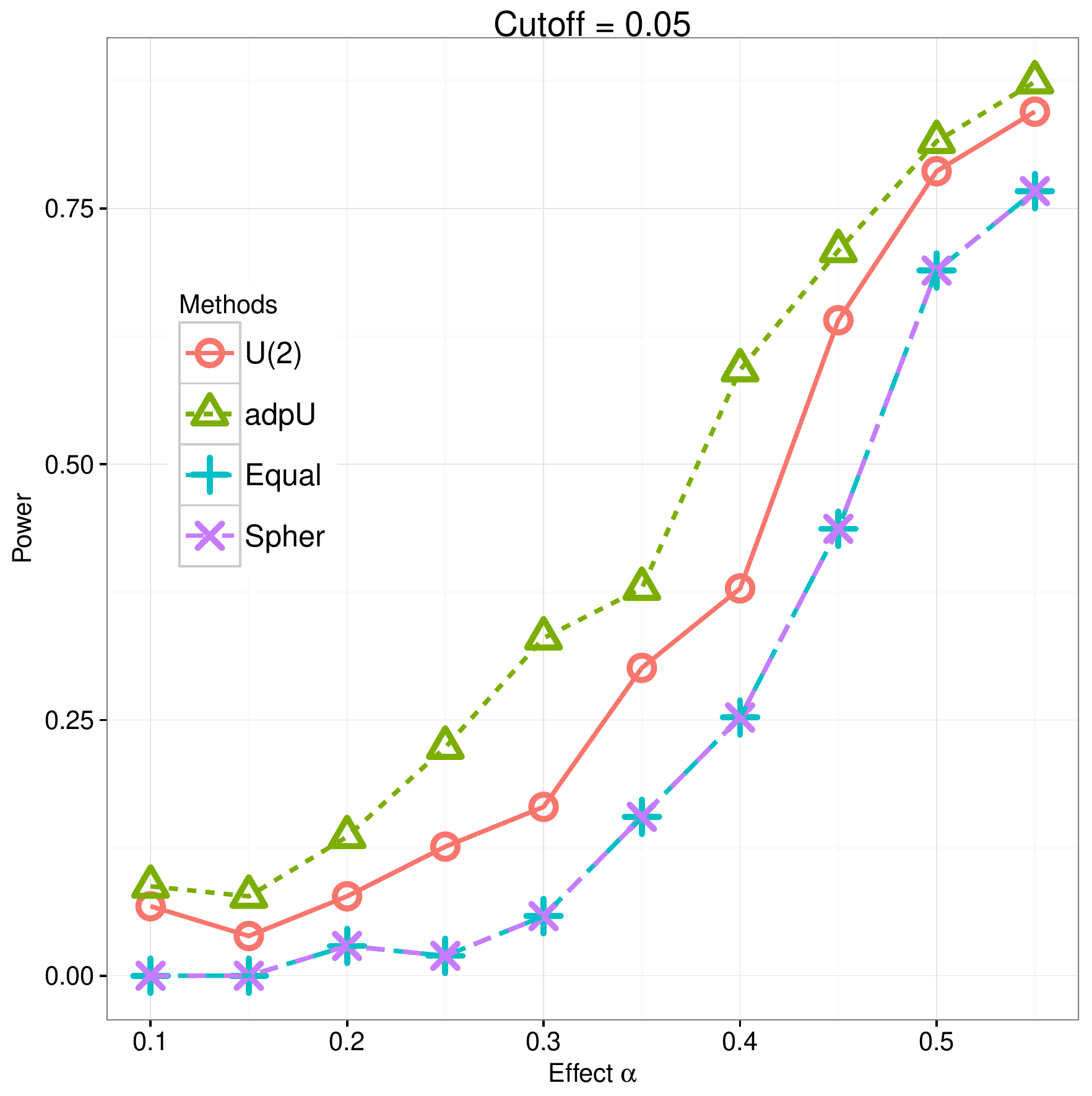}\\
    \caption{Power comparison of different methods with ADNI  data.}
    \label{fig:realdataana}
\end{figure}

\section{Other High-Dimensional Examples} \label{sec:extension}

 In this section, we apply  the proposed U-statistics framework to   other high-dimensional testing problems. Similar theoretical results  to Section \ref{sec:mainexamsec} are developed, with detailed proofs and related simulation studies   provided in   \ref{suppA}.


\subsection{Mean Testing}\label{maintest}

Testing mean vectors is widely used in many statistical analysis and applications \cite{anderson1958introduction,muirhead2009aspects}. Under high-dimensional scenarios, e.g., in genome-wide studies, dimension of the data is often much larger than the sample size, so traditional multivariate tests such as Hotelling's $T^2$-test either cannot be directly applied or have low power \cite{fan1996test}. To address this issue,  several new procedures for testing high-dimensional mean vectors have been proposed \cite[][]{bai1996effect,donoho2004,goeman2006testing,srivastava2008test,chen2010mean,hall2010,cai2014two,chen2014two,gregory2015two,donoho2015higher,srivastava2016raptt,xu2016adaptive}.   
However, many of the statistics only target at either sparse or dense alternatives, and suffer from loss of power for other types of alternatives. We next apply the U-statistics framework   to one-sample and two-sample mean testing problems.
 
\paragraph{One-sample mean testing}
We first discuss the one-sample mean vector testing. Assume that  $\mathbf{x}_1, \ldots, \mathbf{x}_n$ are $n$ i.i.d. copies of a $p$-dimensional real-valued random vector $\mathbf{x}=(x_1,\ldots, x_p)^{\intercal}$  with mean vector $\boldsymbol{\mu}=(\mu_{1}, \ldots, \mu_{p})^{\intercal}$,  covariance matrix $\boldsymbol{\Sigma}=\{\sigma_{j_1,j_2}: 1 \leq j_1, j_2\leq p\}$. We want to conduct the global test on $H_0: \boldsymbol{\mu}=\boldsymbol{\mu}_0$ where $\boldsymbol{\mu}_0=(\mu_{1,0},\ldots,\mu_{p,0})^{\intercal}$ is given.   

Similar to previous discussion, the parameter set that we are interested in is $\mathcal{E}=\{\mu_{1}-\mu_{1,0},\ldots, \mu_{p}-\mu_{p,0} \}$. For each $j=1,\ldots ,p$, $\mathrm{E}(x_{i,j})=\mu_j$, so $K_j(\mathbf{x}_i)=x_{i,j}-\mu_{j,0}$ is a kernel function, which is a simple unbiased estimator of the target. Following our construction, the U-statistic for finite $a$ is 
\begin{align}
	\mathcal{U}(a) =& \sum_{j=1}^p \frac{1}{P^n_{a}}\sum_{1 \leq i_1 \neq \cdots \neq i_{a} \leq n}  \prod_{k=1}^a (x_{i_k,j} -\mu_{j,0}) , \label{eq:ustatonesamplemean} 
\end{align}	which targets at $\|\mathcal{E}\|_a^a=\sum_{j=1}^p(\mu_j-\mu_{j,0})^a$, and the  U-statistic corresponding to $\|\mathcal{E}\|_{\infty}$ is
$	
	\mathcal{U}(\infty) =  \max_{1\leq j\leq p}  \sigma_{j,j}^{-1}(\bar{x}_j-\mu_{0,j})^2
$ with $\bar{x}_j=\sum_{i=1}^n x_{i,j}/n$. 

Given the statistics,  we have the theoretical results similar to Theorems \ref{thm:jointnormal}--\ref{thm:asymindpt}. The following Theorems \ref{thm:onesamplemean}--\ref{thm:onesamplemean2} are established under similar conditions to that of Theorems \ref{thm:jointnormal}--\ref{thm:asymindpt}. Due to the limited space, we provide  the conditions and corresponding discussions in \ref{suppA}.

\begin{theorem} \label{thm:onesamplemean}
Under $H_0$: $\boldsymbol{\mu}=\boldsymbol{\mu}_0$, assume Condition \ref{cond:onesamplemean} in \ref{suppA}. Then for any  finite integers $\{a_1,\ldots, a_m\}$, as $n,p \rightarrow \infty$, $ [ {\mathcal{U}(a_1)}/{\sigma(a_1)}, \ldots, \allowbreak {\mathcal{U}(a_m)}/{\sigma(a_m)}  ]^{\intercal} \xrightarrow{D} \mathcal{N}(0,I_m),$
where  $\sigma^2(a)=\mathrm{var}[\mathcal{U}(a)]=\sum_{i=1}^p \sum_{j=1}^p a!\sigma_{i,j}^{a}/P^n_a$ with the order of $\Theta (a!pn^{-a}).$ 
\end{theorem}

\begin{theorem}\label{thm:onesamplemean2}  Under $H_0$: $\boldsymbol{\mu}=\boldsymbol{\mu}_0$, assume Condition \ref{cond:onesamplemean2} in \ref{suppA}. Then $\forall u\in \mathbb{R}$,
$
	P ( {n} \mathcal{U}(\infty)-\tau_p \leq u  ) \rightarrow \exp  \{ - \pi^{-1/2} \exp(-u/2)  \}, 
$	as $n,p \rightarrow \infty$, where $\tau_p=2\log p - \log \log p$.  In addition, for any finite integer $a$, $\{\mathcal{U}(a)/\sigma(a)\}$ and $\{ n\mathcal{U}(\infty)-\tau_p\}$ are asymptotically independent.  
\end{theorem}

By Theorems \ref{thm:onesamplemean} and \ref{thm:onesamplemean2}, we obtain the asymptotic independence among the U-statistics and the corresponding limiting distributions of the U-statistics under $H_0$. Under the alternative hypothesis, since the power analysis of the one-sample mean testing is similar to that of the two-sample case,  we delay the power analysis after presenting the asymptotic independence property of the proposed  U-statistics in the two-sample mean testing problem.

\paragraph{Two-sample mean testing}
Next we discuss the two-sample mean testing problem. Suppose we have two groups of $p$-dimensional  observations  $\{ \mathbf{x}_{i}\}_{i=1}^{n_x}$ and $\{\mathbf{y}_{i}\}_{i=1}^{n_y}$, which are i.i.d. copies of two independent random vectors $\mathbf{x}=(x_1,\ldots, x_p)^{\intercal}$ and $\mathbf{y}=(y_1,\ldots, y_p)^{\intercal}$ respectively.  Suppose $\mathrm{E}(\mathbf{x})=\boldsymbol{\mu}=(\mu_{1},\ldots, \mu_{p})^{\intercal}$, $\mathrm{E}(\mathbf{y})=\boldsymbol{\nu}=(\nu_{1},\ldots, \nu_{p})^{\intercal}$, $\mathrm{cov}(\mathbf{x})=\boldsymbol{\Sigma}_x$ and $\mathrm{cov}(\mathbf{y})=\boldsymbol{\Sigma}_y$. We write $n = n_x+n_y$ and assume  $n_x=\Theta(n_y)$.  For easy illustration, we first consider $\boldsymbol{\Sigma}_x=\boldsymbol{\Sigma}_y=\boldsymbol{\Sigma}=\{\sigma_{j_1,j_2}: 1\leq j_1, j_2 \leq p \}$. We will then discuss the case when $\boldsymbol{\Sigma}_x\neq \boldsymbol{\Sigma}_y$,  where similar analysis applies. 
 
The two-sample mean testing examines  $H_0:$ $\boldsymbol{\mu}=\boldsymbol{\nu}$ versus $H_A:$ $\boldsymbol{\mu} \neq \boldsymbol{\nu}$, then $\mathcal{E}=(\mu_{1}-\nu_{1},\ldots, \mu_{p}-\nu_{p})^{\intercal}$.  For $1\leq j\leq p$, $1 \leq k\leq n_x$, $1\leq s \leq n_y$, $K_j(\mathbf{x}_{k}, \mathbf{y}_{s})=x_{k,j}-y_{s,j}$ is a simple unbiased estimator of $\mu_{j}-\nu_{j}$,  and thus we construct   $ \mathcal{U}(a)= \sum_{j=1}^p (P^{n_x}_{a} P^{n_y}_{a})^{-1}\sum_{1 \leq k_1 \neq \ldots \neq k_a \leq n_x\atop 1 \leq s_1 \neq \ldots \neq s_a \leq n_y}   \prod_{t=1}^a (x_{k_t,j}-y_{s_t,j}) ,$ which is also equivalent to
\begin{align}
	\mathcal{U}(a)=  \sum_{j=1}^p \sum_{c=0}^{a}  \binom{a}{c} \frac{(-1)^{a-c} }{P^{n_x}_{c} P^{n_y}_{a-c}} \sum_{\substack{1\leq k_1 \neq \cdots \neq k_c\leq n_x \\ 1 \leq s_1 \neq \cdots \neq s_{a-c} \leq n_y} } \prod_{t=1}^c x_{k_{t},j} \prod_{m=1}^{a-c}  y_{s_{m},j}. 
	 \label{eq:teststattwosample} 
\end{align}
We can check that \eqref{eq:teststattwosample} satisfies $\mathrm{E}\{\mathcal{U}(a)\}=\sum_{j=1}^p (\mu_{j}-\nu_{j})^a$, so $\mathcal{U}(a)$ is an unbiased estimator of $\|\mathcal{E}\|_a^a=\sum_{j=1}^p (\mu_{j}-\nu_{j})^a$. On the other hand, for $\|\mathcal{E} \|_{\infty}$, following the maximum-type test statistic in \citet{cai2014two}, we have
\begin{eqnarray}
	\mathcal{U}(\infty)= \max_{1\leq j \leq p} \sigma_{j,j}^{-1}(\bar{x}_{j}- \bar{y}_{j})^2, \label{eq:uinftytwosammean}
\end{eqnarray} where $\bar{x}_j=\sum_{i=1}^{n_x}x_{i,j}/n_x$, $\bar{y}_j=\sum_{i=1}^{n_y}y_{i,j}/n_y$. We then obtain results similar to Theorems \ref{thm:jointnormal}, \ref{thm:asymindpt} and \ref{thm:cltalternative}. As the conditions are similar to those in Section \ref{sec:mainexamsec}, we only keep the key conclusions, and the details of conditions and discussions are given in \ref{suppA} Section \ref{sec:extseccond}.

\begin{theorem} \label{thm:twosamplemean}
Under Condition \ref{cond:twosamplemeancond} in \ref{suppA}, $\boldsymbol{\Sigma}_x=\boldsymbol{\Sigma}_y$ and  $H_0:$ $\boldsymbol{\mu}=\boldsymbol{\nu}$,  for any   finite  integers  $(a_1,\ldots, a_m)$, as  $ n, p \rightarrow \infty $, $ [ {\mathcal{U}(a_1)}/{\sigma(a_1)},\allowbreak \ldots,  {\mathcal{U}(a_m)}/{\sigma(a_m)}  ]^{\intercal} \xrightarrow{D} \mathcal{N}(0,I_m),$
 	where   
 	$	\sigma^2(a)  \simeq a! \sum_{j_1,j_2=1}^p ( n_x+ n_y )^a \sigma_{j_1,j_2}^a /(n_xn_y)^a$ is of the order  $\Theta(a!pn^{-a})$. 
\end{theorem}

\begin{theorem} \label{THM:TWOSAMPLEMANINF}
Under Condition \ref{cond:twosamplemeancond} in \ref{suppA}, $\boldsymbol{\Sigma}_x=\boldsymbol{\Sigma}_y$ and  $H_0:$ $\boldsymbol{\mu}=\boldsymbol{\nu}$, $\forall u\in \mathbb{R}$,
$
	P ( \frac{n_xn_y}{n_x+n_y} \mathcal{U}(\infty)-\tau_p \leq u ) \rightarrow \exp  \{ - \pi^{-1/2} \exp(-u/2) \}, 
$	as $n,p \rightarrow \infty$, where $\tau_p=2\log p - \log \log p$. Moreover, $\{\mathcal{U}(a)/\sigma(a)\}$ of finite integer $a$ and $\{n_xn_y \mathcal{U}(\infty)/(n_x+n_y)-\tau_p\}$ are asymptotically independent.  
\end{theorem}

Theorems \ref{thm:twosamplemean} and \ref{THM:TWOSAMPLEMANINF} provide the asymptotic  properties of finite-order U-statistics and $\mathcal{U}(\infty)$ under $H_0$. To analyze the power of $\mathcal{U}(a)$'s, we derive the asymptotic results of $\mathcal{U}(a)$'s under the alternative hypotheses.  
We focus on the  two-sample mean testing problem, while one-sample mean testing  can be obtained similarly.   Specifically, we consider the alternative $\mathcal{E}_A=\{\mu_{j}-\nu_{j} =\rho > 0\mbox{ for }  j=1,\ldots, k_0; \mu_{j}-\nu_{j} =0 \mbox{ for } j=k_0+1,\cdots, p\}$.  We then obtain  similar conclusions to Theorem \ref{thm:cltalternative}.
\begin{theorem} \label{thm:altcltmeantest}
Assume Condition \ref{cond:twosamplemeancond} in \ref{suppA} and   $k_0=o(p)$. For  any finite integers $\{a_1,\ldots, a_m\}$, if $\rho$ in $\mathcal{E}_A$ satisfies $\rho=O(k_0^{-1/a_t}p^{1/(2a_t)}n^{-1/2})$ for $t=1,\ldots,m$, then  $
	[ \mathcal{U}(a_1)-\mathrm{E}\{\mathcal{U}(a_1)\}]/{\sigma(a_1)}, \ldots,\allowbreak [\mathcal{U}(a_m)-\mathrm{E}\{\mathcal{U}(a_m)\}]/{\sigma(a_m)}  ]^{\intercal} \xrightarrow{D} \mathcal{N}(0,I_m),  
$ as  $n,p \rightarrow \infty$. 
Here $\mathrm{E}[\mathcal{U}(a)]=\|\mathcal{E}_A\|_{a}^{a}=k_0\rho^{a}$  and $\sigma^2(a)=
	\mathrm{var}\{\mathcal{U}(a)\}\simeq V_{a}$, with $V_{a}=a!\sum_{j_1,j_2=k_0+1}^p  \allowbreak (n_x+n_y)^a \sigma_{j_1,j_2}^a/ (n_x n_y)^a$ of the order $\Theta(a!pn^{-a})$.
 \end{theorem}

Next we compare the power   of different U-statistics under alternatives with different sparsity levels. 
Theorem \ref{thm:altcltmeantest} shows that under the  local alternatives, the asymptotic power of $\mathcal{U}(a)$ mainly depends on $\mathrm{E}\{\mathcal{U}(a) \}/ \sqrt{\mathrm{var}\{\mathcal{U}(a)\}}$. 
Therefore by Theorem \ref{thm:altcltmeantest}, given constant $M>0$, for each $\mathcal{U}(a)$, if $\rho=M^{{1}/{a}} k_0^{-{1}/{a}}V_a^{{1}/{(2a)}}$, then
  $\mathrm{E}\{\mathcal{U}(a) \}/ \sqrt{\mathrm{var}\{\mathcal{U}(a)\}} \simeq M$; that is, different $\mathcal{U}(a)$'s have the same	power asymptotically. For easy illustration, we consider $\sigma_{j_1,j_2}=1$ when $j_1=j_2 \in \{k_0+1,\ldots,p\}$, and $\sigma_{j_1,j_2}=0$ when $j_1\neq j_2 \in \{k_0+1,\ldots,p\}$, then 
$M^{{1}/{a}} k_0^{-{1}/{a}}V_a^{{1}/{(2a)}} \simeq \rho_a$ with  
\begin{eqnarray}
	\rho_a := a!^{\frac{1}{2a}} ({M\sqrt{p}}/{k_0})^{\frac{1}{a}}\{ ({n_x+n_y})/({n_xn_y} )\}^{\frac{1}{2}}.
	 \label{eq:deltaavalue}
\end{eqnarray}    Therefore, similarly to the analysis in Section \ref{sec:powerana}, to find the ``best" $\mathcal{U}(a)$, it suffices to find the order, denoted by  $a_0$, that gives the minimum $\rho_a$  in \eqref{eq:deltaavalue}.    We have the following result similar to Proposition \ref{prop:ordercompare}.  
\begin{proposition} \label{cor:deltaminimval}
 Given any constant $ M \in (0,+\infty)$ and $n, p, k_0$, we consider $\rho_a$ in  \eqref{eq:deltaavalue} as a function of positive integers $a$, then
 \begin{enumerate}
\item[(i)] when $k_0 \geq M\sqrt{p}$, the minimum of $\rho_a$ is achieved at $a_0=1$; 
\item[(ii)] when $k_0<M\sqrt{p}$, the minimum of $\rho_a$ is achieved at some $a_0$, which increases as $M\sqrt{p}/|J_D|$ increases.
\end{enumerate}
\end{proposition}


Proposition \ref{cor:deltaminimval} shows that when the sparsity level $k_0$ is large, i.e., $\mathcal{E}_a$ is dense, a small $a$ tends to obtain a smaller lower bound in $\rho$, and vice versa. As \eqref{eq:deltaavalue} and  \eqref{eq:rhoform} are similar,  we have similar patterns to that in Figure \ref{fig:gavalue} when examining the corresponding  numerical plots of $\rho_a$. 
In addition, 
\cite{cai2014two} shows that when $\rho =\rho_{\infty}:=C_1\sqrt{\log p/n}$ for a large $C_1$, the power of $\mathcal{U}(\infty)$ converges to 1, and $\sqrt{\log p/n}$ is minimax rate optimal for sparse alternatives; see also \cite{donoho2015higher}. 
Thus, if $\rho_{\infty}<\rho_{a_0}$, i.e., 
$
	k_0<MC_1^{-a_0}\sqrt{pa_0!}/\log^{{a_0}/{2}}p,
$ $\mathcal{U}(\infty)$ is the ``best" and its lowest detectable order of $\rho$ is $\Theta(\sqrt{\log p/n})$. 
On the other hand, Proposition \ref{cor:deltaminimval} shows that when $\mathcal{E}_A$ is dense with $k_0>\sqrt{Mp}$, $\mathcal{U}(1)$ is the  ``best" and its lowest detectable  order of $\rho$ is $\Theta(\sqrt{p}k_0^{-1}n^{-1/2})$. Moreover, for some large $M$ and $C_2$, when $\mathcal{E}_A$ is ``moderately dense" or ``moderately sparse" with $C_2\sqrt{pa_0!}/\log^{{a_0}/{2}}p<k_0<\sqrt{Mp}$,    $\mathcal{U}(a_0)$ is the ``best" and  its lowest detectable order of $\rho$ is $\Theta\{(\sqrt{p}/k_0)^{\frac{1}{a_0}} n^{-1/2}\}$, which is of a smaller order than the optimal detection boundary of the sparse case $\Theta(\sqrt{\log p/n})$.

More generally, when $\boldsymbol{\Sigma}_x\neq \boldsymbol{\Sigma}_y$, similar results to Theorems \ref{thm:twosamplemean} and \ref{thm:altcltmeantest} can be obtained. In particular, we have the following corollary. 
\begin{corollary}\label{prop:generalrestwomean}
When $\boldsymbol{\Sigma}_x\neq \boldsymbol{\Sigma}_y$, under Condition \ref{cond:twosamplemeancond} in \ref{suppA}, 
Theorem \ref{thm:twosamplemean} holds with $
	\sigma^2(a)  \simeq a!\sum_{j_1,j_2=1}^p  (\sigma_{x,j_1,j_2}/n_x+ \sigma_{y,j_1,j_2}/n_y)^a$ and Theorem \ref{thm:altcltmeantest} holds with $V_{a}=a!\sum_{j_1,j_2=k_0+1}^p  \allowbreak (\sigma_{x,j_1,j_2}/n_x+\sigma_{y,j_1,j_2}/n_y)^a$. 
\end{corollary}
 Corollary \ref{prop:generalrestwomean} shows that  the asymptotic power of finite-order U-statistics   depends on $\mathrm{E}\{\mathcal{U}(a)\}/\sqrt{\mathrm{var}\{\mathcal{U}(a)\} }$. By the construction of finite-order U-statistics and the proof, we obtain that  $\mathrm{E}\{\mathcal{U}(a)\}=k_0\rho^a$ and $\mathrm{var}\{\mathcal{U}(a)\}=\Theta(a!pn^{-a})$. We then know that for finite-order U-statistics, similar results to Proposition \ref{cor:deltaminimval}   still hold by examining  $\mathrm{E}\{\mathcal{U}(a)\}/\sqrt{\mathrm{var}\{\mathcal{U}(a)\} }$.


The above power analysis shows that the optimal U-statistic varies when the alternative hypothesis changes.  To achieve high power across various alternatives, we can develop an adaptive test  similar to that in Section \ref{sec:computtest}. 
 Specifically, we calculate the $p$-values of the U-statistics \eqref{eq:ustatonesamplemean} and  \eqref{eq:teststattwosample} following the theoretical results above and the algorithm in Section \ref{sec:computtest}. By combining the $p$-values as discussed in  Section \ref{sec:computtest}, the asymptotic power of the adaptive test goes to 1 if there exists one $\mathcal{U}(a)$ whose power goes to 1. 

\begin{remark} \label{rm:comparewithxu}
\citet{xu2016adaptive} has also discussed the adaptive testing of two-sample mean that is powerful against various $\ell_p$-norm-like sums of $\boldsymbol{\mu}-\boldsymbol{\nu}$. But \cite{xu2016adaptive} is under the framework of a family of von Mises V-statistics where $\mathcal{V}(a)  =  \sum_{j=1}^p(\bar{x}_{j}- \bar{y}_{j})^a.$ We note that $\mathcal{V}(a)$ is equivalent  to
\begin{eqnarray*}
\mathcal{V}(a)  =\sum_{j=1}^p \sum_{c=0}^{a} (-1)^{a-c}  \binom{a}{c} ({n_x}^{c} {n_y}^{a-c})^{-1}  \sum_{\substack{1\leq k_1, \cdots, k_c\leq n_x \\1 \leq s_1, \cdots, s_{a-c} \leq n_y} } \prod_{t=1}^c x_{k_{t},j} \prod_{m=1}^{a-c}  y_{s_{m},j},	
\end{eqnarray*}   
which   allows the indexes $k$'s and $s$'s to be the same and thus is different from the U-statistics in \eqref{eq:teststattwosample}.
 \cite{xu2016adaptive} shows that the constructed V-statistics are biased estimators of $\| \boldsymbol{\mu}-\boldsymbol{\nu} \|_a^a$, and $\mathcal{V}(a)$ and $\mathcal{V}(b)$ are asymptotically independent if $a+b$ is odd, but are asymptotically correlated if $a+b$ is even.
The constructed U-statistics in this work extend the properties of those V-statistics such that $\mathcal{U}(a)$  in \eqref{eq:teststattwosample} is an  {unbiased} estimator of $\|\boldsymbol{\mu}-\boldsymbol{\nu}\|_a^a$, and  all $\mathcal{U}(a)$'s are {asymptotically independent} with each other. Given these nice statistical properties, it becomes easier to obtain the joint asymptotic distribution of the U-statistics, and then apply the adaptive test.
\end{remark}


\subsection{Two-Sample Covariance Testing} \label{sec:twosamcov}

The U-statistics framework can be applied similarly  to testing the equality of two  covariance matrices.  Suppose  $\{ \mathbf{x}_{i}\}_{i=1}^{n_x}$ and $\{\mathbf{y}_{i}\}_{i=1}^{n_y}$ are  i.i.d. copies of two independent random vectors $\mathbf{x}=(x_1,\ldots, x_p)^{\intercal}$ and $\mathbf{y}=(y_1,\ldots, y_p)^{\intercal}$ respectively.  Denote $\mathrm{E}(\mathbf{x})=\boldsymbol{\mu}=(\mu_{1},\ldots, \mu_{p})^{\intercal}$, $\mathrm{E}(\mathbf{y})=\boldsymbol{\nu}=(\nu_{1},\ldots, \nu_{p})^{\intercal}$; $\mathrm{cov}(\mathbf{x})=\boldsymbol{\Sigma}_x=\{\sigma_{x, j_1,j_2}: 1\leq j_1, j_2 \leq p\}$ and $\mathrm{cov}(\mathbf{y})=\boldsymbol{\Sigma}_y=\{\sigma_{y,j_1,j_2}: 1\leq j_1, j_2 \leq p\}$. Consider  $H_0: \boldsymbol{\Sigma}_x=\boldsymbol{\Sigma}_y=\boldsymbol{\Sigma}=(\sigma_{j_1,j_2})_{p\times p}$. Given  $1 \leq j_1,j_2 \leq p$, $1 \leq k_1 \neq k_2 \leq n_x$, and $1\leq s_1 \neq s_2 \leq n_y$,
$
	 K_{j_1,j_2}(\mathbf{x}_{k_1},\mathbf{x}_{k_2}, \mathbf{y}_{s_1},\mathbf{y}_{s_2}) 
	=  (x_{k_1,j_1}x_{k_1,j_2}-x_{k_1,j_1}x_{k_2,j_2})- (y_{s_1,j_1}y_{s_1,j_2}-y_{s_1,j_1}y_{s_2,j_2})
$ is a simple unbiased estimator of $\sigma_{x,j_1,j_2}-\sigma_{y,j_1,j_2}$.  Therefore, for  a finite positive integer  $a$,    we have the U-statistic
 \begin{align}\label{eq:ts-cov}
 \mathcal{U}(a) =  \sum_{1\leq j_1,j_2 \leq p} &\frac{1}{P^{n_x}_{2a}P^{n_y}_{2a}} \sum_{\substack{ 1\leq k_{1,1}\neq k_{1,2} \neq \ldots \\ \neq k_{a,1}\neq k_{a,2} \leq n_x}} \ \sum_{\substack{ 1\leq s_{1,1}\neq s_{1,2} \neq \ldots \\ \neq s_{a,1}\neq s_{a,2} \leq n_y}} \\
&\prod_{t=1}^a  K_{j_1,j_2}(\mathbf{x}_{k_{t,1}},\mathbf{x}_{k_{t,2}}, \mathbf{y}_{s_{t,1}},\mathbf{y}_{s_{t,2}}). \notag 
\end{align} 
As in Remark \ref{rm:anotherpesusstat}, another formulation of $ \mathcal{U}(a)$ equivalent to \eqref{eq:ts-cov} is 
\begin{align}
	\mathcal{U}(a)=&\sum_{c=0}^a \sum_{b_1=0}^c \sum_{b_2=0}^{a-c}  (-1)^{c-b_1+b_2} \sum_{1\leq j_1, j_2 \leq p}  \sum_{ \substack{ 1 \leq i_1 \neq \ldots \neq \\ i_{2c-b_1}  \leq n_x}}\ \sum_{\substack{ 1 \leq {w}_1 \neq \ldots  \neq \\ {w}_{2(a-c)-b_2}  \leq n_y}} \label{eq:twosamcovueqform2} \\
	& C_{n_x,n_y,a,c,b_1,b_2}\times \prod_{k=1}^{b_1}  (x_{i_k, j_1}x_{i_k,  j_2})  \prod_{s=b_1+1}^{c}  x_{i_{s}, j_1} \prod_{t=c+1}^{2c-b_1} x_{i_{t}, j_2}  \notag \\
	& \times \prod_{m=1}^{b_2}  (y_{{w}_m, j_1}y_{{w}_m,  j_2})\prod_{l=b_2+1}^{a-c} y_{{w}_{l}, j_1} \prod_{q=a-c+1}^{2(a-c)-b_2}y_{{w}_{q}, j_2},  \notag
\end{align} where $C_{n_x,n_y,c,b_1,b_2}=(P^{n_x}_{2c-b_1}P^{n_y}_{2(a-c)-b_2})^{-1}a!/\{b_1!(c-b_1)!b_2!(a-c-b_2)!\}$, and  \eqref{eq:twosamcovueqform2}  shall be used in the theoretical developments.

We next present the  asymptotic results of the constructed U-statistics under the null hypothesis. 
Here we assume the  regularity Condition \ref{cond:twosamplecov1} or \ref{cond:twosamplecov2}, whose details and discussions are provided in Section \ref{sec:twosamcondnull} of \ref{suppA} due to the space limitation. We mention that Condition \ref{cond:twosamplecov1} is a  mixing-type dependence assumption similar to Condition \ref{cond:alphamixing}, and Condition \ref{cond:twosamplecov2} is a moment-type  dependence assumption similar to Condition \ref{cond:ellpmoment}. Particularly, Condition \ref{cond:twosamplecov2} extends the moment assumption for second-order U-statistics in  \citet{lijun2012} to  U-statistics of general orders; please see the detailed discussions in Section \ref{sec:twosamcondnull}.

\begin{theorem} \label{thm:twosamnull}
Under $H_0$ and Condition \ref{cond:twosamplecov1} or \ref{cond:twosamplecov2} in \ref{suppA}, 
for finite integers $\{a_1,\ldots, a_m\}$,  $ [ {\mathcal{U}(a_1)}/{\sigma(a_1)}, \ldots, \allowbreak {\mathcal{U}(a_m)}/{\sigma(a_m)}  ]^{\intercal} \xrightarrow{D} \mathcal{N}(0,I_m),$ where for $a\in\{a_1,\ldots,a_m\}$, 
\begin{align*}
	&\sigma^2(a)  =   \mathrm{var}\{\mathcal{U}(a)\}\\
	 \simeq & \sum_{1\leq j_1,j_2,j_3,j_4\leq p}a! \Big\{ \frac{1}{n_x}(\Pi_{j_1,j_2,j_3,j_4}^x-\sigma_{j_1,j_2}\sigma_{j_3,j_4}) +\frac{1}{n_y}(\Pi_{j_1,j_2,j_3,j_4}^y-\sigma_{j_1,j_2}\sigma_{j_3,j_4})\Big\}^a
\end{align*}
  with  $\Pi_{j_1,j_2,j_3,j_4}^x=\mathrm{E}\{\prod_{t=1}^4(x_{1,j_t}-\mu_{j_t})\}$ and $\Pi_{j_1,j_2,j_3,j_4}^y=\mathrm{E}\{\prod_{t=1}^4(y_{1,j_t}-\nu_{j_t})\}$. 
\end{theorem}
  
Theorem \ref{thm:twosamnull} provides the asymptotic independence and joint normality of the finite-order U-statistics, which are similar to Theorems \ref{thm:jointnormal}, \ref{thm:onesamplemean} and \ref{thm:twosamplemean}. 
To further study the power of these finite-order U-statistics, we next consider the alternative hypotheses where $\boldsymbol{\Sigma}_x\neq \boldsymbol{\Sigma}_y$.  Let $\mathbb{J}_0$ be the largest subset of $\{1,\ldots,p\}$ such that $\sigma_{x,j_1,j_2}=\sigma_{y,j_1,j_2}=\sigma_{j_1,j_2}$ for any $j_1,j_2\in \mathbb{J}_0$. We then obtain the following theorem under the regularity conditions given in Section \ref{sec:pfthm49alttwo} of \ref{suppA}. 
\begin{theorem} \label{thm:twosamaltclt}
Under Conditions \ref{cond:twosamaltdist} and \ref{cond:twosamaltmoment} in the \ref{suppA}, for finite integers $\{a_1,\ldots, a_m\}$, 
$
	[ \mathcal{U}(a_1)-\mathrm{E}\{\mathcal{U}(a_1)\}]/{\sigma(a_1)}, \ldots, \allowbreak [\mathcal{U}(a_m)-\mathrm{E}\{\mathcal{U}(a_m)\}]/{\sigma(a_m)}  ]^{\intercal} \xrightarrow{D} \mathcal{N}(0,I_m),  
$ where $$\sigma^2(a)=\mathrm{var}\{\mathcal{U}(a)\}\simeq  a! C_{\kappa,a}\sum_{ j_1,j_2,j_3,j_4\in \mathbb{J}_0} \sigma_{j_1,j_2}^a\sigma_{j_3,j_4}^a,$$ and $ C_{\kappa,a}=\{(\kappa_x-1)/n_x+(\kappa_y-1)/n_y\}^a+ 2(\kappa_x/n_x+\kappa_y/n_y)^a$ with $\kappa_x$ and $\kappa_y$ given in Condition  \ref{cond:twosamaltdist}.
\end{theorem} 
 
Given the asymptotic results under the alternatives,
we next analyze the power of the finite-order U-statistics. By Theorem  \ref{thm:twosamaltclt},  the asymptotic power of $\mathcal{U}(a)$ depends on $\mathrm{E}\{\mathcal{U}(a)\}/\sqrt{ \mathrm{var}\{\mathcal{U}(a)\} }$. Let $J_D=\{(j_1,j_2): \sigma_{x,j_1,j_2}\neq \sigma_{y,j_1,j_2}, 1\leq j_1,j_2\leq p \}$, then $\mathrm{E}\{\mathcal{U}(a)\}=\sum_{(j_1,j_2)\in J_D}(\sigma_{x,j_1,j_2}-\sigma_{y,j_1,j_2})^a$. Similarly to Section \ref{sec:powerana}, to study the relationship between the sparsity level of $\boldsymbol{\Sigma}_x-\boldsymbol{\Sigma}_y$ and the   power of U-statistics, we consider the case where the non-zero differences between  $\boldsymbol{\Sigma}_x$ and   $\boldsymbol{\Sigma}_y$ are the same. Specifically, let $\sigma_{x,j_1,j_2}-\sigma_{y,j_1,j_2}=\rho$ for $(j_1,j_2)\in J_D$, and then $\mathrm{E}\{\mathcal{U}(a)\}= |J_D|\rho^a$.  Following the analysis in Section \ref{sec:powerana}, we compare the $\rho$ values needed by different $\mathcal{U}(a)$'s to achieve $\mathrm{E}\{\mathcal{U}(a) \}/\sqrt{\mathrm{var}\{ \mathcal{U}(a) \}}\simeq M$ for a given constant $M$. In particular, for given integer $a$, suppose $\mathrm{E}\{\mathcal{U}(a) \}/\sqrt{\mathrm{var}\{ \mathcal{U}(a) \}}\simeq M$ is achieved when $\rho=\rho_a$.  
For any $a\neq b$, we compare  $\mathcal{U}(a)$ and $\mathcal{U}(b)$ following Criterion \ref{cricomp}.


We use the following example as an illustration, where $\boldsymbol{\Sigma}_x$ and $\boldsymbol{\Sigma}_y$  satisfy the conditions of Theorem  \ref{thm:twosamaltclt}. Specifically, we assume that $\boldsymbol{\Sigma}_x=(\sigma_{x,j_1,j_2})_{p\times p}$ has the diagonal elements $\sigma_{x,j,j}=\nu^2$; and the off-diagonal elements $\sigma_{x,j_1,j_2}=h_{|j_1-j_2|}\in(0,\nu^2)$ with $h_{|j_1-j_2|}=\Theta(\nu^2)$ when $|j_1-j_2|\leq s$, while $\sigma_{x,j_1,j_2}=0$ when $|j_1-j_2|> s$. This  covers the moving average covariance structure of order $s$, and $\boldsymbol{\Sigma}_x$ is a banded matrix with bandwidth $s$. In addition,
 we assume the bandwidth $s=o(p)$ and $p-|\mathbb{J}_0|=o(p)$. By the definition of $\mathbb{J}_0$, the assumption $p-|\mathbb{J}_0|=o(p)$ implies that a large square sub-matrix of $\boldsymbol{\Sigma}_x$ and $\boldsymbol{\Sigma}_y$ are the same. For simplicity, we let $n_x=n_y$ with $n=n_x+n_y$, and a similar analysis can be applied  when $n_x\neq n_y$. 
 By Theorem \ref{thm:twosamaltclt}, $\mathrm{var}\{\mathcal{U}(a)\}\simeq (n/2)^{-a}a! \{2\kappa_1^a+\kappa_2^a\}\{ p\nu^{2a} + 2\sum_{t=1}^s h_{t}^a (p-t)\}^2$, where $\kappa_1=\kappa_x+\kappa_y$ and $\kappa_2=\kappa_x+\kappa_y-2$.  Therefore we know for given finite integer $a$, $\mathrm{E}\{\mathcal{U}(a) \}/\sqrt{\mathrm{var}\{ \mathcal{U}(a) \}}\simeq M$ holds when $\rho=\rho_a$ defined as 
  \begin{align*}
\rho_a = & \frac{(a!)^{\frac{1}{2a}} \sqrt{\kappa_1} \nu  }{(n/2)^{1/2}}\Big(\frac{M{p}}{|J_D|}\Big)^{1/a} \Big\{2+ \Big(\frac{\kappa_2}{\kappa_1}\Big)^a\Big\}^{\frac{1}{2a}} \Big\{ 1+ { 2\sum_{t=1}^s \Big( \frac{h_t}{\nu^2}\Big)^a \Big(1-\frac{t}{p}\Big) } \Big\}^{\frac{1}{a}}. 
\end{align*}
We next compare the $\rho_a$'s and obtain the following proposition. 

\begin{proposition}\label{prop:ordercomptwosam}
There exists $\mathbb{D}_0$ that only depends on the given $\kappa_x,\kappa_y,\nu^2,s$, and $h_t, t=1,\ldots, s$, and satisfies $\mathbb{D}_0=\Theta(1/s^2)$ such that 
\begin{enumerate}
	\item[(i)] When $|J_D|\geq Mp/\sqrt{\mathbb{D}_0}$,  the minimum of $\rho_a$ is achieved at $a_0=1$.
\item[(ii)] When $|J_D|< Mp/\sqrt{\mathbb{D}_0}$, the minimum of $\rho_a$ is achieved at some $a_0$, which increases as $Mp/|J_D|$ increases.
\end{enumerate} 
\end{proposition}

Proposition \ref{prop:ordercomptwosam} is similar to Propositions \ref{prop:ordercompare} and \ref{cor:deltaminimval}.
Following the analysis in Section \ref{sec:powerana}, Proposition \ref{prop:ordercomptwosam}  shows that 
when  the difference $\boldsymbol{\Sigma}_x-\boldsymbol{\Sigma}_y$ is ``very" dense with  $|J_D|\geq Mp/\sqrt{\mathbb{D}_0}$, $\mathcal{U}(1)$ is the most powerful U-statistic; when $\boldsymbol{\Sigma}_x-\boldsymbol{\Sigma}_y$ becomes sparser as $Mp/|J_D|$ decreases, a higher order U-statistic is more powerful;  when the  $\boldsymbol{\Sigma}_x-\boldsymbol{\Sigma}_y$ is ``moderately" dense or sparse, a  U-statistic of finite order $a_0>1$ would be the most powerful one.

The power analysis above shows that the power of the U-statistics varies when the alternative changes. To maintain high power across different alternatives, we can develop an adaptive testing procedure similar to that in Section \ref{sec:computtest}. 
Given the asymptotic independence in Theorem \ref{thm:twosamnull}, an adaptive testing procedure using the constructed $\mathcal{U}(a)$'s is valid with the type \RNum{1} error asymptotically controlled. Also, the adaptive test achieves high power  by combining the U-statistics as discussed in Section \ref{sec:computtest}.   



 We provide simulation studies on two-sample covariance testing in \ref{suppA} Section \ref{sec:simulothertesting}. By the simulations, we first find that the type \RNum{1} errors of the U statistics and the adaptive test are well controlled under $H_0$. This verifies the theoretical results in Theorem \ref{thm:twosamaltclt}. 
 Second, similarly to the one-sample covariance testing, we find that generally when the difference $\boldsymbol{\Sigma}_x-\boldsymbol{\Sigma}_y$ is sparser, a U-statistic of higher order is more powerful, and vice versa. Moreover, under moderately sparse/dense alternatives, $\mathcal{U}(a_0)$ with $a_0>1$ could achieve the highest power. The results are consistent with Proposition \ref{prop:ordercomptwosam}.  Third, we compare the proposed adaptive test with existing methods in literature including  \cite{schott2007test,srivastava2010testing,lijun2012,cai2013two}, and find that the proposed adaptive testing procedure  maintains high power across various alternatives.       
  
\begin{remark}
Similarly to Section \ref{sec:mainexamsec}, we can let $\mathcal{U}(\infty)$ be the maximum-type test statistic in \cite{cai2013two}, and expect that the result similar to Theorem \ref{thm:asymindpt} holds under certain regularity conditions. However, as the dependence structure of two-sample covariance matrices is more complicated than the one-sample case, it is more challenging to establish the asymptotic joint distribution of $\mathcal{U}(\infty)$ and  finite-order U-statistics. We leave this interesting problem for future study, while find in simulations that the performance of $\mathcal{U}(\infty)$ is similar to  high-order U-statistics  $\mathcal{U}(a)$'s. 
\end{remark}


\subsection{Generalized Linear Model}\label{sec:newglm}
In this section, we consider the Example \ref{eg:3} of generalized linear models (on Page \pageref{eg:3}) to show that the proposed framework can be extended to other testing problems.  
Similarly to the results in Section \ref{maintest}, we show that the constructed U-statistics are asymptotically  independent and normally distributed, and also establish the power analysis results of the U-statistics. We  provide the details in Section \ref{sec:extglm} of  \ref{suppA}. Recently, \citet{wu2019adaptive}   also discussed the adaptive testing of generalized linear model. But similarly to \cite{xu2016adaptive}, \cite{wu2019adaptive}  is under the framework of a family of von Mises V-statistics, and thus is different from the current paper as discussed in Remark \ref{rm:comparewithxu}. Moreover, the current work provides the theoretical power analysis while \cite{wu2019adaptive} did not.

\section{Discussion} \label{sec:discuss}


This paper introduces a general  U-statistics framework for applications to high-dimensional  adaptive testing. Particularly, we focus on the examples including testing of means, covariances and regression coefficients in generalized linear models. Under the null hypothesis, we prove that the U-statistics of finite orders have asymptotic  joint normality,   and establish the asymptotic mutual independence among the finite-order U-statistics and $\mathcal{U}(\infty)$. Moreover, under alternative hypotheses, we analyze the power of different U-statistics and demonstrate   how  the most powerful U-statistic changes with  the sparsity level of the  alternative parameters. Based on the theoretical results, we propose an adaptive testing procedure, which is powerful against different alternatives. The superior performance of this adaptive testing is confirmed  in the simulations and real data analysis.

There are several possible extensions of the U-statistics framework in this paper.  First, by our current proof, the convergence rate in Theorem \ref{thm:asymindpt} is bounded by $O(\log^{-1/2} p)$, which is an upper bound and not sharp. From our extensive simulations, we find that the type \RNum{1} error rate  of the adaptive testing is well-controlled with a relatively small $p$, e.g., $p=50.$ We might obtain a shaper bound  of the convergence rate, but more refined concentration property of the high-dimensional and high-order U-statistics is needed.
Second, the proposed framework requires that the elements in the parameter set $\mathcal{E}$ have unbiased estimates. When we can not obtain unbiased estimates easily, e.g., for the precision matrix, the proposed construction may not follow directly.  Nevertheless   we may use ``nearly" unbiased estimators to construct   ``U-statistics"   for hypothesis testing, such as the ``nearly" unbiased estimator of the precision matrix proposed in \cite{xia2015testing}; the main challenge is then to control the accumulative bias over the parameters under high-dimensions. 
Third, this paper discusses the examples where the elements in $\mathcal{E}$ are comparable. When the parameters in $\mathcal{E}$ are not comparable,  such as $\mathcal{E}$   containing both means and covariances parameters,  the construction of U-statistics still follows but the theoretical derivation may require a careful case-by-case examination. Fourth, the   construction of the U-statistics treats the parameters in $\mathcal{E}$ with equal weight. More generally, we could assign different weights to different parameter estimators. For instance, standardizing the data is one example of assigning different weights. 
As inappropriate weight assignments could lead to power loss, when the truth is unknown, how to effectively assign weights   to maximize the test power is an interesting research question. We shall discuss these  extensions in the future as a significant amount of additional work is still needed.

 In addition to the examples  in this paper, 
 the proposed U-statistics framework can be applied to other high-dimensional hypothesis testing problems. 
 For example, 
  it can  be applied to testing the block-diagonality of a covariance matrix, whose theoretical analysis would be similar to the considered one sample and two sample covariance testing problems. 
  It can also be used to test  high-dimensional regression coefficients in   complex regression models other than the generalized linear models, following a similar construction based on the score functions. A key step is then to characterize the impact of nuisance parameters that are estimated under the null hypothesis, and challenges arise especially   when the  nuisance parameters are  high-dimensional.  Such interesting extensions will be further explored in our follow-up studies.

  \section*{Acknowledgements}
  The authors thank  Co-Editors Prof. Edward I. George and  Prof. Richard J. Samworth, an Associate Editor, and three anonymous referees for their constructive comments.  
The authors also thank  Prof. Ping-Shou Zhong for sharing the code of the paper \cite{chen2011} and  Prof. Xuming He and Prof. Peter Song for helpful discussions.  

 \begin{supplement}[id=suppA]
 \sname{Supplementary Material}
 \slink[doi]{XXX}
 \sdatatype{Supplementary.pdf}
\stitle{}
 \sdescription{This supplementary material contains the technical proofs of the main paper and additional simulations.}
 \end{supplement}

\appendix
\bibliographystyle{chicago}
\bibliography{Citation.bib}

\newpage
\quad
\smallskip

\begin{center}
\large{\textbf{SUPPLEMENT TO ''ASYMPTOTICALLY INDEPENDENT
U-STATISTICS IN HIGH DIMENSIONAL ADAPTIVE
TESTING''}}	
\end{center}

\normalsize
\medskip

We give proofs of the main results and additional simulations in this supplementary material. 
For simplicity, we use $C$ to  represent some generic positive constant,  which does not change with $(n,p)$   and  may represent different values from place to place. 

\appendix

\section{Proofs and Supplementary Results} \label{sec:a}

\subsection{Proof of Proposition \ref{prop:locinvariance}} \label{proof:prop21}

To prove $\mathcal{U}(a)$ in  \eqref{eq:originaluinvariance} is location invariant, we examine the equivalent form,
\begin{align*}
\mathcal{U}(a)=(P^n_{2a})^{-1} \sum_{1\leq j_1 \neq j_2 \leq p} \sum_{1\leq i_1 \neq \ldots \neq i_{2a} \leq n} \prod_{k=1}^a (x_{i_{2k-1},j_1}x_{i_{2k-1},j_2}-x_{i_{2k-1},j_1}x_{i_{2k},j_2}).
\end{align*} We consider $\boldsymbol{\Delta}=(\Delta_1,\ldots,\Delta_p)^{\intercal}\in \mathbb{R}^p$, and examine $a=1$ first. For each $(j_1,j_2)$, since
\begin{eqnarray*}
	&&(x_{i_1,j_1}+\Delta_{j_1})(x_{i_1,j_2}+\Delta_{j_2})-(x_{i_1,j_1}+\Delta_{j_1})(x_{i_2,j_2}+\Delta_{j_2}) \notag \\
	&=&(x_{i_1,j_1}x_{i_1,j_2}-x_{i_1,j_1}x_{i_2,j_2})+\Delta_{j_1}(x_{i_1,j_2}-x_{i_2,j_2}),  
\end{eqnarray*} then it follows that
\begin{eqnarray*}
&& \sum_{1\leq i_1 \neq i_2 \leq n} [(x_{i_1,j_1}+\Delta_{j_1})(x_{i_1,j_2}+\Delta_{j_2})-(x_{i_1,j_1}+\Delta_{j_1})(x_{i_2,j_2}+\Delta_{j_2})] \notag \\
&& \quad \quad - \sum_{1\leq i_1 \neq i_2 \leq n} (x_{i_1,j_1}x_{i_1,j_2}-x_{i_1,j_1}x_{i_2,j_2})  \notag \\
	&= & \sum_{1\leq i_1 \neq i_2 \leq n} \Delta_{j_1}(x_{i_1,j_2}-x_{i_2,j_2}) + \sum_{i=1}^n \Delta_{j_1}(x_{i,j_2}-x_{i,j_2})  \notag \\
	&= & ~ \Delta_{j_1}  \sum_{i_1=1}^n \sum_{i_2=1}^n (x_{i_1,j_2}-  x_{i_2,j_2})  \notag \\
	&=&~0.
\end{eqnarray*} That is,   $\mathcal{U}(1)$ is location invariant. For $a=2$, given $(j_1,j_2)$, following a similar analysis to $\mathcal{U}(1)$,  we have
\begin{align}\label{eq:u2invariantproof1}
 & \sum_{1\leq i_1 \neq \ldots \neq i_{4} \leq n} \Big \{ [(x_{i_1,j_1}+\Delta_{j_1}) (x_{i_1,j_2}+\Delta_{j_2})-(x_{i_1,j_1}+\Delta_{j_1})    (x_{i_2,j_2}+\Delta_{j_2})]\\
	& \quad \quad \times   [(x_{i_3,j_1}+\Delta_{j_1})(x_{i_3,j_2}+\Delta_{j_2})-(x_{i_3,j_1}+\Delta_{j_1})(x_{i_4,j_2}+\Delta_{j_2})]\Big\}  \notag \\
	&   - \sum_{1\leq i_1 \neq \ldots \neq i_{4} \leq n}\Big\{  (x_{i_1,j_1}x_{i_1,j_2}-x_{i_1,j_1}x_{i_2,j_2})\notag \\
	& \quad  \quad\times  [(x_{i_3,j_1}+\Delta_{j_1})(x_{i_3,j_2}+\Delta_{j_2}) -(x_{i_3,j_1}+\Delta_{j_1})(x_{i_4,j_2}+\Delta_{j_2})] \Big\} \notag \\
	=&~ 0.  \notag  
\end{align}  
Similarly, we also have 
\begin{align}\label{eq:u2invariantproof2}
 & \sum_{1\leq i_1 \neq \ldots \neq i_{4} \leq n} \Big\{  (x_{i_1,j_1}x_{i_1,j_2}-x_{i_1,j_1}x_{i_2,j_2}) \\
	& \quad  \quad \times [(x_{i_3,j_1}+\Delta_{j_1}) (x_{i_3,j_2}+\Delta_{j_2})  -(x_{i_3,j_1}+\Delta_{j_1})  (x_{i_4,j_2}+\Delta_{j_2})] \Big\}  \notag \\
	& -\sum_{1\leq i_1 \neq \ldots \neq i_{4} \leq n} [(x_{i_1,j_1}x_{i_1,j_2}-x_{i_1,j_1}x_{i_2,j_2})  (x_{i_3,j_1}x_{i_3,j_2}-x_{i_3,j_1}x_{i_4,j_2}) ] \notag \\
	= &~ 0.  \notag
\end{align} Combining  \eqref{eq:u2invariantproof1} and \eqref{eq:u2invariantproof2}, we know $\mathcal{U}(2)$ is location invariant. Following the  argument above similarly, by  induction, we obtain that  $\mathcal{U}(a)$ is location invariant for a general integer $a \geq 3$.

\subsection{Proof of Theorem \ref{thm:jointnormal}} \label{sec:detailofjointnormal}

For  the covariance testing example in Section \ref{sec:mainexamsec}, $\mathcal{U}(a)$ is location invariant by  Proposition \ref{prop:locinvariance},  and $\mathcal{U}(\infty)$ is also  location invariant straightforwardly by its expression in \eqref{eq:inftyteststat}. Then we  assume without loss of generality that $\mathrm{E}({\mathbf x})={\mathbf 0}$  in this section. To prove Theorem \ref{thm:jointnormal}, we first derive the variances and  the covariances of the U-statistics, and then prove the asymptotic joint normality of the U-statistics. 



In particular, the next Lemma \ref{lm:varianceorder} derives the asymptotic form of variance $\sigma^2(a)$ in  \eqref{eq:varianceformmodif1}.


\begin{lemma} \label{lm:varianceorder}
Under the conditions of Theorem \ref{thm:jointnormal}, 
	  for any finite integer $a$,  following the notation in \eqref{eq:highmomentdef},
\begin{eqnarray*}
	\sigma^2(a) = \frac{a!}{P^n_a}  \sum_{ \substack{1\leq j_1\neq j_2 \leq p;\\ 1\leq j_3 \neq j_4 \leq p}} (\Pi_{j_1,j_2,j_3,j_4})^a \{1+o(1)\},   \notag 
\end{eqnarray*} which is of order $\Theta(p^2n^{-a})$. In addition,  for  $\tilde{\mathcal{U}}(a)$ defined in \eqref{eq:originleadingterm} and  $\tilde{\mathcal{U}}^*(a):=\mathcal{U}(a)-\tilde{\mathcal{U}}(a)$, we have    $\mathrm{var}\{ \mathcal{U} (a)\} = \mathrm{var}\{\tilde{\mathcal{U} }(a)\}\{1+o(1)\}$, $\mathrm{var}\{\tilde{\mathcal{U}}^* (a)\} = o(1)\times \mathrm{var}\{\tilde{\mathcal{U} }(a)\}$, and $\tilde{\mathcal{U}}^*(a)/\sigma(a) \xrightarrow{P} 0$. 
\end{lemma} 
\begin{proof}
\textit{See Section  \ref{sec:proofvarianceorder} on Page \pageref{sec:proofvarianceorder}.}
\end{proof}


\noindent Moreover, the following Lemma \ref{lm:covariancezro} shows that the covariances between different $\mathcal{U}(a)$'s asymptotically converge to 0.
\begin{lemma} \label{lm:covariancezro} 
Under the  conditions  of Theorem \ref{thm:jointnormal},  for finite integers $a\neq b$,   $\mathrm{cov}\{\mathcal{U}(a) /\sigma(a),  \mathcal{U}(b)/ \sigma(b) \} \to 0$, as $n,p\rightarrow \infty$. 
\end{lemma}
\begin{proof}
\textit{See Section  \ref{sec:proofcovariancezro} on Page \pageref{sec:proofcovariancezro}.}
\end{proof}

Lemmas \ref{lm:varianceorder} and  \ref{lm:covariancezro} together establish that the  covariance matrix of $[ {\mathcal{U}(a_1)}/{\sigma(a_1)}, \ldots, {\mathcal{U}(a_m)}/{\sigma(a_m)}  ]^{\intercal}$ converges to  $I_m$ asymptotically. To finish the proof of Theorem \ref{thm:jointnormal}, it remains to show that the joint limiting distribution of the U-statistics is normal.  
 
For finite integers $a_1,\ldots, a_m$, to obtain the joint asymptotic normality  of $[\mathcal{U}(a_1)/\sigma(a_1),\allowbreak \ldots, \mathcal{U}(a_m)/\sigma(a_m)]^{\intercal}$, by the Cram\'{e}r-Wold theorem, 
it is equivalent to prove that any fixed linear combination of $[\mathcal{U}(a_1)/\sigma(a_1),\allowbreak \ldots, \mathcal{U}(a_m)/\sigma(a_m)]^{\intercal}$ converges to normal. 
Recall that Lemma \ref{lm:varianceorder} shows that $\tilde{\mathcal{U}}^*(a)/\sigma(a)\xrightarrow{P} 0$ for any finite integer $a$. 
Thus by the Slutsky's theorem,  it suffices to prove that any fixed linear combination of $[\tilde{\mathcal{U}}(a_1)/\sigma(a_1),\allowbreak \ldots, \allowbreak \tilde{\mathcal{U}}(a_m)/\sigma(a_m)]^{\intercal}$ converges to normal.  
To be specific, we show that for constants $t_1,\ldots,t_m$  satisfying $\sum_{r=1}^m t_r^2=1$, 
\begin{eqnarray}
Z_n:=\sum_{r=1}^m t_r {\tilde{\mathcal{U}}(a_r)  }/{\sigma(a_r)}\xrightarrow{D} \mathcal{N}(0,1).  \label{eq:znnormal}
\end{eqnarray} 


To prove \eqref{eq:znnormal}, we apply the martingale central limit theorem  in \citet{heyde1970}  (similar arguments can date back to    \citet{bai1996effect}). 
Let $\mathcal{F}_{0}= \{  \emptyset,\Omega\}$, $\mathcal{F}_{k}=\sigma\{\mathbf{x}_1,\cdots,\mathbf{x}_k \}$,  and $\mathrm{E}_k(\cdot)$ denote the conditional expectation given $\mathcal{F}_k$ for $k=1, \cdots, n$.   Define $D_{n,k}=(\mathrm{E}_{k}-\mathrm{E}_{k-1})Z_n$ and $\pi^2_{n,k}=\mathrm{E}_{k-1}(D_{n,k}^2)$. Note that $\mathrm{E}_0(\cdot)=\mathrm{E}(\cdot)$, and  $\mathrm{E}(Z_n)=0$ as $\mathrm{E}(\mathbf{x})=\mathbf{0}$. It follows that $Z_n= \sum_{k=1}^n D_{n,k}$. By martingale central limit theorem, to prove \eqref{eq:znnormal}, it is sufficient to show
\begin{align}\label{eq:cltgoalnull}
	{\sum_{k=1}^n\pi^2_{n,k}}/{\mathrm{var}(Z_n)} \xrightarrow{P} 1,\qquad {\sum_{k=1}^n\mathrm{E}(D_{n,k}^4 )}/{\mathrm{var}^2(Z_n)} \rightarrow 0.
\end{align} 
Here  $\mathrm{var}(Z_n)\to \sum_{r=1}^m t_r^2=1$ by Lemmas \ref{lm:varianceorder} and   \ref{lm:covariancezro}, and $\mathrm{E}(\sum_{k=1}^n \pi^2_{n,k} )=\mathrm{var} (Z_n)$ by the following Lemma \ref{lm:meaninvarsigma}. 
\begin{lemma} \label{lm:meaninvarsigma}
Under the  conditions  of Theorem \ref{thm:jointnormal}, 
$\mathrm{E}(\sum_{k=1}^n \pi^2_{n,k} )=\mathrm{var} (Z_n)$. 
\end{lemma}
\begin{proof}
\textit{See Section \ref{sec:proofmeaninvarsigma} on Page \pageref{sec:proofmeaninvarsigma}.} 
\end{proof}

\noindent Therefore to prove \eqref{eq:cltgoalnull},   it suffices to show 
\begin{eqnarray} \label{eq:cltgoal1}
	\mathrm{var}\Biggr( {\sum_{k=1}^n\pi^2_{n,k}}\Biggr) \to 0 \quad \text{and} \quad 
	{\sum_{k=1}^n\mathrm{E}(D_{n,k}^4 )}\rightarrow 0. 
\end{eqnarray}

Note that $D_{n,k}$ and $\pi_{n,k}^2$ in \eqref{eq:cltgoal1} can be written as $D_{n,k}=\sum_{r=1}^m t_rA_{n,k,a_r}$ and $\pi_{n,k}^2=\sum_{1\leq r_1,r_2\leq m}\mathrm{E}_{k-1}(A_{n,k,a_{r_1}}A_{n,k,a_{r_2}})$, where we define $A_{n,k,a}=(\mathrm{E}_k-\mathrm{E}_{k-1})\{\tilde{\mathcal{U}}(a)/\sigma(a)\}$ for each finite integer $a$. 
The following Lemma \ref{lm:cltabnkform} gives the explicit form of $A_{n,k,a}$.
\begin{lemma} \label{lm:cltabnkform}
For finite integer $a$, when $k<a$, $A_{n,k,a}=0$; when $k \geq a$,
	\begin{align*}
		A_{n,k,a} = \frac{a}{\sigma(a)P^n_{a}} \sum_{1\leq i_1 \neq \cdots \neq i_{a-1}\leq k-1}\, \sum_{1\leq j_1\neq j_2 \leq p} 
	(x_{k,j_1} x_{k,j_2})\times \prod_{t=1}^{a-1}(x_{i_t,j_1} x_{i_t,j_2}). \notag 
	 \end{align*}
\end{lemma}

\begin{proof}\textit{See Section \ref{sec:proofcltabnkform} on Page \pageref{sec:proofcltabnkform}.} \end{proof}

\noindent 
With the form of  $A_{n,k,a}$ in Lemma \ref{lm:cltabnkform}, the forms of $D_{n,k}$ and $\pi_{n,k}^2$ can be obtained, and we can prove the  next two Lemmas \ref{lm:targetorder} and   \ref{lm:secondgoaljointnormal}, which  suggest that \eqref{eq:cltgoal1} holds. 

\begin{lemma}\label{lm:targetorder}
Under the conditions of Theorem \ref{thm:jointnormal}, $ \mathrm{var}( \sum_{k=1}^n \pi^2_{n,k} ) \to 0$. In particular, under Condition \ref{cond:alphamixing}, $ \mathrm{var}( \sum_{k=1}^n \pi^2_{n,k} ) = O(p^{-1}\log^3 p)$; under Condition \ref{cond:ellpmoment}, $ \mathrm{var}( \sum_{k=1}^n \pi^2_{n,k} ) = O(n^{-1}+p^{-2})$. 
\end{lemma}
 
 \begin{proof}
\textit{See Section \ref{sec:prooftargetorder} on Page \pageref{sec:prooftargetorder}.}	
\end{proof}
 
\begin{lemma}\label{lm:secondgoaljointnormal} 
Under the conditions of Theorem \ref{thm:jointnormal}, $\sum_{k=1}^n\mathrm{E}(D_{n,k}^4)=O(1/n)$. 
\end{lemma}

\begin{proof}
\textit{See Section \ref{sec:proofsecondgoaljointnormal} on Page \pageref{sec:proofsecondgoaljointnormal}.}
\end{proof}

Finally, by \citet{heyde1970},  we have as $n,p \rightarrow \infty$, 
\begin{eqnarray}\label{eq:thm1berryessenbound}
	& & \sup_t \Big| P(Z_n \leq t )-\Phi(t) \Big|   \\
	 &\leq & C\Biggr\{ \mathrm{E}\Biggr[\frac{\sum_{k=1}^n\mathrm{E}_{k-1}(D_{n,k}^2)}{\mathrm{var}(Z_n)} -1 \Biggr]^2  +  \frac{\sum_{k=1}^n\mathrm{E}\left(D_{n,k}^4 \right)}{\mathrm{var}^2(Z_n)}  \Biggr\}^{1/5} \notag \\
	&\rightarrow & 0, \notag
\end{eqnarray} which proves \eqref{eq:znnormal}.  In summary, Theorem \ref{thm:jointnormal} is proved.

\subsection{Proof of Theorem \ref{thm:asymindpt}} \label{sec:asymindptproofsec}



In this section, we first introduce some notation, and then present the proof. 

\vspace{0.2em}
\noindent \textit{Notation.}\ \ For $\mathcal{U}(a)$ in \eqref{eq:originaluinvariance},  by the symmetricity of covariance matrix,  
we can replace $\sum_{1 \leq j_1 \neq j_2 \leq p}$ by $2\times \sum_{1 \leq j_1 < j_2 \leq p}.$ 
This implies that the summation over $\{(j_1,j_2): 1 \leq j_1 \neq j_2 \leq p\}$ is equivalent to the  summation over $\{(j_1,j_2):1 \leq j_1 < j_2 \leq p\}$ up to a constant.  Without loss of generality, we consider $j_1 < j_2$ below. We rewrite the index set  $\{(j_1,j_2): 1\leq j_1 < j_2 \leq p \}$ as 
\begin{eqnarray}
	L:=\Big\{ (j^1_l, j^2_l): 1\leq l \leq q =\binom{p}{2} \Big\}, \label{eq:wdef}
\end{eqnarray} where 
$
	j_l^1=\arg\min_{1\leq k\leq p-1} \{\sum_{t=1}^k(p-t)\geq l\}$ and $j_l^2=l+j_l^1-\sum_{t=1}^{j_l^1-1}(p-t)
$.
For each $(j^1_l, j^2_l)\in L$, define
\begin{align} 
	U_l^a = & \sum_{1 \leq i_1 \neq \ldots \neq i_a \leq n} \prod_{k=1}^a x_{i_k,j^1_l}x_{i_k,j^2_l}.\label{eq:vladefinition} 
\end{align} Then $\tilde{\mathcal{U}}(a)=2(P^n_a)^{-1}\sum_{l=1}^q U_l^a$ following the definition in  \eqref{eq:originleadingterm}.  Furthermore, we define 
\begin{align}
	\tilde{G}_l = &~\sum_{i=1}^n  \frac{x_{i,j^1_l}}{ \sqrt{\sigma_{j_l^1, j_l^1}} }\times \frac{x_{i,j^2_l}}{\sqrt{\sigma_{j_l^2, j_l^2}}}, \label{eq:indpalldefinewvar} \\
	 M_n = & ~\max_{1 \leq l \leq q}\ ( \tilde{G}_l )^2, \notag \\
	 \hat{G}_l=&~  \sum_{i=1}^n   \frac{x_{i,j^1_l}}{ \sqrt{\sigma_{j_l^1, j_l^1} }}\times  \frac{x_{i,j^2_l}}{\sqrt{\sigma_{j_l^2, j_l^2}}} \mathbf{1}\Big\{ \Big| \frac{x_{i,j^1_l}}{ \sqrt{\sigma_{j_l^1, j_l^1}} } \times \frac{x_{i,j^2_l}}{\sqrt{\sigma_{j_l^2, j_l^2}}} \Big| \leq \tau_n \Big\} \notag \\
	 &-\mathrm{E} \Biggr[ \sum_{i=1}^n   \frac{x_{i,j^1_l}}{ \sqrt{\sigma_{j_l^1, j_l^1}} }\times  \frac{x_{i,j^2_l}}{\sqrt{\sigma_{j_l^2, j_l^2}}} \mathbf{1}\Big\{ \Big| \frac{x_{i,j^1_l}}{ \sqrt{\sigma_{j_l^1, j_l^1}} } \times \frac{x_{i,j^2_l}}{\sqrt{\sigma_{j_l^2, j_l^2}}} \Big| \leq \tau_n \Big\}   \Biggr],  \notag \\
	 \hat{M}_n = &~ \max_{1 \leq l \leq q}\ ( \hat{G}_l )^2,  \notag
\end{align} where we define $\sigma_{j_l^1, j_l^1}={\mathrm{var}( x_{i,j^1_l})} $, $\sigma_{j_l^2, j_l^2}={\mathrm{var}( x_{i,j^2_l})} $, $\tau_n=\tau \log (p+n)$ with $\tau$ being a sufficiently large positive constant and $\mathbf{1}\{\cdot\}$ represents an indicator function. In addition,  we define $|\textbf{a}|_{\min}=\min_{1\leq i \leq p} |a_i|$ for $\mathbf{a}\in \mathbb{R}^p$, and 
\begin{eqnarray}
	y_p = 4\log p - \log \log p+y . \label{eq:ypdefition}
\end{eqnarray}

\noindent \textit{Proof.}\ \  
Similarly to Section \ref{sec:detailofjointnormal}, since   $\mathcal{U}(a)$  in \eqref{eq:originaluinvariance}  and  $\mathcal{U}(\infty)$  in \eqref{eq:inftyteststat} are location invariant, we assume without loss of generality that $\mathrm{E}({\mathbf x})={\mathbf 0}$.

To prove Theorem \ref{thm:asymindpt}, we first establish  the asymptotic independence between $\hat{M}_n/n$ and $\tilde{\mathcal{U}}(a)/\sigma(a_r)$ for $r=1,\ldots,m$, and then we show that $n\mathcal{U}^2(\infty)$  and $\mathcal{U}(a_r)$  are close to $\hat{M}_n/n$ and $\tilde{\mathcal{U}}(a_r)$,   respectively. 
Specifically, the following Lemma \ref{lm:firstsetpthm3proof} shows that $\hat{M}_n/n$ and $\tilde{\mathcal{U}}(a_r)/\sigma(a_r)$'s are asymptotically independent.
 \begin{lemma}  \label{lm:firstsetpthm3proof}
Under the  conditions of Theorem \ref{thm:asymindpt}, when $\tau>0$ in \eqref{eq:indpalldefinewvar} is a sufficiently large constant, 
\begin{align*}
  	&\Biggr| P\Big( \frac{\hat{M}_n}{n} > y_p, \frac{\tilde{\mathcal{U}}(a_1)  }{\sigma(a_1)} \leq z_1,\ldots, \frac{\tilde{\mathcal{U}}(a_m)  }{\sigma(a_m)}\leq z_m \Big) \notag \\ 
  	& \ - P\Big( \frac{\hat{M}_n}{n} > y_p \Big) \prod_{r=1}^m P\Big(\frac{\tilde{\mathcal{U}}(a_r)  }{\sigma(a_r)} \leq z_r \Big) \Biggr| \rightarrow 0. \notag
\end{align*}	
\end{lemma}
\begin{proof}
	\textit{See Section \ref{sec:lemma9pproof} on Page \pageref{sec:lemma9pproof}}. 
\end{proof}

To show that $\hat{M}_n/n$ and $n\mathcal{U}(\infty)^2$ are close,  we use $M_n/n$ defined in \eqref{eq:indpalldefinewvar} as an intermediate variable. 
We next prove that $M_n/n$ and $\hat{M}_n/n$ have small difference in the sense that the conclusion in Lemma \ref{lm:firstsetpthm3proof} still holds by replacing $\hat{M}_n$ with $M_n$.  This is formally stated in the following Lemma \ref{lm:boundedsmalldiff}.
\begin{lemma} \label{lm:boundedsmalldiff}
Under the  conditions of Theorem \ref{thm:asymindpt},
\begin{align*}
	&\Biggr| P\Big( \frac{{M}_n}{n} > y_p, \frac{\tilde{\mathcal{U}}(a_1)  }{\sigma(a_1)} \leq z_1,\ldots, \frac{\tilde{\mathcal{U}}(a_m)  }{\sigma(a_m)}\leq z_m \Big) \notag \\ 
  	& \ - P\Big( \frac{{M}_n}{n} > y_p \Big) \prod_{r=1}^m P\Big(\frac{\tilde{\mathcal{U}}(a_r)  }{\sigma(a_r)} \leq z_r \Big) \Biggr| \rightarrow 0. \notag
\end{align*}
\end{lemma} 
\begin{proof}
\textit{See Section  \ref{sec:lemma10proof} on Page \pageref{sec:lemma10proof}.}	
\end{proof}

Given Lemma \ref{lm:boundedsmalldiff}, we further prove that $M_n/n$ and $\tilde{\mathcal{U}}(a)/\sigma(a_r)$ are close to $n\mathcal{U}^2(\infty)$ and $\mathcal{U}(a_r)$, respectively. 
In particular, by the proof of Theorem 3 in \citet{cai2011}, we know 
$
	\{{n^2 \mathcal{U}^2(\infty)-M_n}\}/{n} \xrightarrow{P} 0.
$ In addition, 
Lemma \ref{lm:varianceorder} proves  that  $\{\mathcal{U}(a_r)-\tilde{\mathcal{U}}(a_r)\}/{\sigma(a_r)}\xrightarrow{P} 0$. 
Based on these results and  Lemma \ref{lm:boundedsmalldiff}, the following Lemma \ref{lm:addprop0termtrue} shows that  the conclusion in Lemma \ref{lm:boundedsmalldiff} still holds by replacing $M_n/n$ with $n\mathcal{U}^2(\infty)$ and replacing $\tilde{\mathcal{U}}(a_r)$ with  $\mathcal{U}(a_r)$. 
\begin{lemma}\label{lm:addprop0termtrue}
 Under the conditions of Theorem \ref{thm:asymindpt},  
 \begin{align}
	& \Big| P\Big( n  \mathcal{U}^2(\infty) > y_p, \frac{{\mathcal{U}}(a_1)  }{\sigma(a_1)} \leq z_1, \ldots,  \frac{{\mathcal{U}}(a_m)  }{\sigma(a_m)} \leq z_m\Big) \notag \\ 
	&\ -  P\Big( n  \mathcal{U}^2(\infty) > y_p \Big) \prod_{r=1}^m P \Big( \frac{{\mathcal{U}}(a_r)  }{\sigma(a_r)} \leq z_r \Big)\Big| \to 0. \notag
\end{align}
\end{lemma}
\begin{proof}
\textit{See Section \ref{sec:lemma11proof} on Page \pageref{sec:lemma11proof}.}	
\end{proof}
\noindent Lemma \ref{lm:addprop0termtrue} then proves  Theorem \ref{thm:asymindpt}.  

\subsection{Proof of Theorem \ref{thm:computation}} \label{sec:proofthm24var}
As both $\mathcal{U}(a)$ and $\mathbb{V}_u(a)$ are location invariant in the sense of Proposition \ref{prop:locinvariance},    we  assume $\mathrm{E}({\mathbf x})={\mathbf 0}$.
To prove Theorem \ref{thm:computation}, we decompose $\mathbb{V}_u(a)=\mathbb{V}_{u,1}(a)+\mathbb{V}_{u,2}(a)$, where we define  
\begin{eqnarray*}
	\mathbb{V}_{u,1}(a)=\frac{2a!}{(P^n_a)^2} \sum_{\substack{1\leq j_1\neq j_2\leq p} } \sum_{1\leq i_1\neq \ldots \neq i_a \leq n} \prod_{t=1}^a x_{i_t,j_1}^2x_{i_t,j_2}^2, 
\end{eqnarray*}  and $\mathbb{V}_{u,2}(a)=\mathbb{V}_{u}(a)-\mathbb{V}_{u,1}(a)$. 
The next Lemma \ref{lm:varestcong2} shows that $\mathbb{V}_{u,1}(a)$ is of a larger order than $\mathbb{V}_{u,2}(a)$, and thus it is the leading term in   $\mathbb{V}_{u}(a)$. 
\begin{lemma}\label{lm:varestcong2}
Under the conditions of Theorem \ref{thm:computation}, ${ \mathbb{V}_{u,1}(a) }/{ \mathrm{E}\{\mathbb{V}_{u,1}(a)\} } \xrightarrow{P} 1$ and ${ \mathbb{V}_{u,2}(a) }/{ \mathrm{E}\{\mathbb{V}_{u,1}(a)\} } \xrightarrow{P} 0.$	
\end{lemma}  
\begin{proof}
\textit{See Section \ref{lm:pfthm24lm} on Page \pageref{lm:pfthm24lm}.}
\end{proof}

Lemma \ref{lm:varestcong2} implies  that $ \mathbb{V}_{u}(a) /\mathrm{E}\{\mathbb{V}_{u,1}(a)\}  \xrightarrow{P} 1$. 
As $\mathbb{V}_u(a) > 0$ with probability 1,  $\mathrm{E}\{\mathbb{V}_{u,1}(a)\}/ \mathbb{V}_{u}(a) \xrightarrow{P} 1 $. 
In addition,  note that 
$\mathrm{E}\{\mathbb{V}_{u,1}(a)\}={2a!}{(P^n_a)^{-1}} \sum_{1\leq j_1\neq j_2\leq p} \{\mathrm{E}(x_{1,j_1}^2 x_{1,j_2}^2)\}^a.$ By \eqref{eq:varoduaoncov}  and \eqref{eq:varordp4top2} in Section \ref{par:complpflm1}, we have $\mathrm{var}\{\mathcal{U}_{}(a)\}/ \mathrm{E}\{\mathbb{V}_{u,1}(a)\} \to 1$. Therefore,  
\begin{align*}
	\frac{\mathbb{V}_u(a)}{\mathrm{var}\{\mathcal{U}(a)\}}=\frac{\mathbb{V}_u(a)}{\mathrm{E}\{\mathbb{V}_{u,1}(a)\}} \times \frac{\mathrm{E}\{\mathbb{V}_{u,1}(a)\} }{ \mathrm{var}\{\mathcal{U}(a)\}}\xrightarrow{P} 1.
\end{align*}

\subsection{Proof of Theorem \ref{thm:cltalternative}} \label{sec:updatepoweronesampf}
We first present Condition \ref{cond:altonecovellip} in Theorem \ref{thm:cltalternative},  which is a generalized version of  Condition \ref{cond:ellpmoment} under $H_A$. 
\begin{condition}\label{cond:altonecovellip}
Following the central moment notation in  \eqref{eq:highmomentdef}, for $t\leq 8$, we assume that there exists constant $\tilde{\kappa}_t$ such that $\Pi_{j_1,\ldots,j_t}=\tilde{\kappa}_t\mathrm{E}(\prod_{k=1}^t z_{j_k}),$ where $1\leq j_1,\ldots, j_t \leq p$ and $(z_1,\ldots, z_p)^{\intercal}\sim \mathcal{N}(\mathbf{0}, \boldsymbol{\Sigma}_A).$
\end{condition}
\noindent  Condition \ref{cond:altonecovellip} generalizes Condition \ref{cond:ellpmoment} to the alternative setting. Similarly to Condition \ref{cond:ellpmoment}, Condition \ref{cond:altonecovellip} is satisfied when $\mathbf{x}$ follows an elliptical  distribution with certain moment conditions \citep[see][]{frahm2004generalized,maruyama2003estimation}. To be consistent with the notation in Condition \ref{cond:ellpmoment}, we let $\kappa_1=\tilde{\kappa}_4$ below. 



\vspace{0.3em}
	
We next introduce some notation, and then provide the proof.  

\vspace{0.3em}
\noindent \textit{Notation.}\ \  
For each given $j_1\in \{1,\ldots, p\},$ we define
\begin{align*}
&J_{j_1}=\{(j_1,j_2): \sigma_{j_1,j_2}\neq 0, 1\leq j_1 \neq j_2\leq p \}, \quad \\
 &J_{j_1}^c=\{ (j_1,j_2): \sigma_{j_1,j_2}=0, 1\leq j_1\neq j_2\leq p \}.
\end{align*} Then $J_A=\cup_{j_1=1}^p J_{j_1}$, and we correspondingly  define $J_A^c= \cup_{j_1=1}^p J_{j_1}^c,$ which is the 
set difference of  $\{(j_1,j_2):1\leq j_1\neq j_2\leq p\}$ and $J_A$.  
Moreover, we define $F(a,c)=(-1)^c \binom{a}{c}/P^n_{a+c}$, and  
\begin{align*}
	K(c,j_1,j_2)= F(a,c)\sum_{1\leq i_1\neq \ldots \neq i_{a+c} \leq n}  \prod_{t=1}^{a-c} (x_{i_t,j_1}x_{i_t,j_2})  \prod_{t=a-c+1}^{a} x_{i_t,j_1} \prod_{t=a+1}^{a+c} x_{i_t,j_2}. 
\end{align*}
We decompose $\mathcal{U}(a)=T_{U,a,1,1}+T_{U,a,1,2}+T_{U,a,2}$, where 
\begin{eqnarray}
\quad \quad &&T_{U,a,1,1}=\sum_{(j_1,j_2)\in J_A^c}K(0,j_1,j_2),\ \   T_{U,a,1,2}=\sum_{(j_1,j_2)\in J_A^c}\sum_{c=1}^aK(c,j_1,j_2), \label{eq:uadcomalt} \\
&& T_{U,a,2}=\sum_{(j_1,j_2)\in J_A}\sum_{c=0}^a K(c,j_1,j_2). \notag
\end{eqnarray}

\noindent \textit{Proof.}\ \ 
Similarly to Section \ref{sec:detailofjointnormal}, we first derive the variances and the covariances of the U-statistics, and then prove the asymptotic joint normality of the U-statistics.
Particularly, the next Lemma \ref{lm:pfaltvarlm1} derives  the asymptotic form of  $\mathrm{var}\{\mathcal{U}(a)\}$, and additionally shows that among the three terms in \eqref{eq:uadcomalt},  $T_{U,a,1,1}$ is the leading one. 
\begin{lemma}\label{lm:pfaltvarlm1}
Under the conditions of Theorem \ref{thm:cltalternative},  $\sigma^2(a)=\mathrm{var}\{\mathcal{U}(a)\}\simeq \mathrm{var}(T_{U,a,1,1})$, where
\begin{align*}
\mathrm{var}(T_{U,a,1,1})\simeq 2a!\kappa_1^a n^{-a} \sum_{1\leq j_1\neq j_2\leq p}  \sigma_{j_1,j_1}^a \sigma_{j_2,j_2}^a,	
\end{align*} which is $\Theta(p^2n^{-a})$. Moreover, $\mathrm{var}(T_{U,a,1,2})=o(p^2n^{-a})$,  $\mathrm{var}(T_{U,a,2})=o(p^2n^{-a})$ and $\{\mathcal{U}(a)-T_{U,a,1,1}\}/\sigma(a)\xrightarrow{P}0.$ 
\end{lemma}
\begin{proof}
\textit{See Section \ref{sec:pfaltvarlm1} on Page \pageref{sec:pfaltvarlm1}.}
\end{proof}
\noindent 
The following Lemma \ref{lm:pfaltvarlm2} shows that the covariance between two different U-statistics asymptotically converges to 0. 
\begin{lemma}\label{lm:pfaltvarlm2}
Under the conditions of Theorem \ref{thm:cltalternative},   for two integers $a\neq b,$ $\mathrm{cov}\{\mathcal{U}(a)/\sigma(a),\mathcal{U}(b)/\sigma(b)\}\to 0.$ 
\end{lemma}
\begin{proof}
\textit{See Section \ref{sec:pfaltvarlm2} on Page \pageref{sec:pfaltvarlm2}.}
\end{proof}

To finish the proof, it remains to obtain the joint asymptotic normality  of $[\mathcal{U}(a_1)/\sigma(a_1),\allowbreak \ldots, \mathcal{U}(a_m)/\sigma(a_m)]^{\intercal}$. 
By the Cram\'{e}r-Wold theorem, it is equivalent to prove that any fixed linear combination of $[\mathcal{U}(a_1)/\sigma(a_1),\allowbreak \ldots, \mathcal{U}(a_m)/\sigma(a_m)]^{\intercal}$  converges to a normal distribution. By Lemma \ref{lm:pfaltvarlm1},  $\{\mathcal{U}(a)-T_{U,a,1,1}\}/\sigma(a)\xrightarrow{P}0$, thus by the Slutsky's theorem,  it suffices to prove that any fixed linear combination of $[T_{U,a_1,1,1}/\sigma(a_1),\allowbreak \ldots, T_{U,a_m,1,1}/\sigma(a_m)]^{\intercal}$ converges to a normal distribution.  
Similarly to  Section \ref{sec:detailofjointnormal}, we redefine $Z_n$  as below with $\sum_{r=1}^m t_r^2=1$, and prove that
\begin{eqnarray}
Z_n:=\sum_{r=1}^m t_r {T_{U,a_r,1,1} }/{\sigma(a_r)}\xrightarrow{D} \mathcal{N}(0,1).  \label{eq:znnormalalt}
\end{eqnarray} 

We next prove \eqref{eq:znnormalalt} by the martingale central limit theorem, similarly to 
  Section \ref{sec:detailofjointnormal}. In particular, we define $\mathrm{E}_k(\cdot)$ in the same way as in Section \ref{sec:detailofjointnormal}, and still define  $D_{n,k}=(\mathrm{E}_k-\mathrm{E}_{k-1})Z_n$ and $\pi_{n,k}^2=\mathrm{E}_{k-1}(D_{n,k}^2)$. 
 It follows that  $D_{n,k}=\sum_{r=1}^m t_rA_{n,k,a_r}$ and $\pi_{n,k}^2=\sum_{1\leq r_1,r_2\leq m}t_{r_1}t_{r_2}\mathrm{E}_{k-1}(A_{n,k,a_{r_1}}A_{n,k,a_{r_2}})$, where we redefine $A_{n,k,a_r}=(\mathrm{E}_k-\mathrm{E}_{k-1})\{T_{U,a_r,1,1}/\sigma(a_r)\}$. 
  Note that $\sigma_{j_1,j_2}=0$ when $(j_1,j_2)\in J_A^c$,  and $T_{U,a,1,1}$ is a summation over $(j_1,j_2) \in J_A^c$. Thus the proof of Lemma  \ref{lm:cltabnkform} in Section \ref{sec:proofcltabnkform} applies similarly, and we  obtain the explicit form of $A_{n,k,a}$.  
Specifically, for each finite integer $a$, when $k<a$, $A_{n,k,a}=0$; when $k \geq a$,
\begin{align*}
		A_{n,k,a} = \frac{a}{\sigma(a)P^n_{a}} \sum_{1\leq i_1 \neq \cdots \neq i_{a-1}\leq k-1}\, \sum_{(j_1,j_2)\in J_A^c} 
	(x_{k,j_1} x_{k,j_2}) \prod_{t=1}^{a-1}(x_{i_t,j_1} x_{i_t,j_2}). \notag 
	 \end{align*}
With the form of $A_{n,k,a}$, we can obtain the explicit forms of $D_{n,k}$ and $\pi_{n,k}^2$. Then we can prove the following two  Lemmas  \ref{lm:altcltmomentgoal1} and  \ref{lm:altcltmomentgoal2}, which suggests that \eqref{eq:znnormalalt} holds. 

\begin{lemma} \label{lm:altcltmomentgoal1}
Under the conditions of Theorem \ref{thm:cltalternative}, $
	\mathrm{var}( \sum_{k=1}^n \pi_{n,k}^2 )\to 0$.
\end{lemma}
\begin{proof}
\textit{See Section \ref{sec:pfaltcltmomentgoal1}  on Page \pageref{sec:pfaltcltmomentgoal1}.}
\end{proof}
\begin{lemma} \label{lm:altcltmomentgoal2}
Under the conditions of Theorem \ref{thm:cltalternative}	, $\sum_{k=1}^n \mathrm{E}(D_{n,k}^4)\to 0$.
\end{lemma}
\begin{proof}
\textit{See Section \ref{sec:pfaltcltmomentgoal2} on Page \pageref{sec:pfaltcltmomentgoal2}.}
\end{proof}
\noindent By Lemmas  \ref{lm:altcltmomentgoal1} and  \ref{lm:altcltmomentgoal2}, \eqref{eq:znnormalalt} holds and thus Theorem  \ref{thm:cltalternative} is proved. 

\subsection{Proof of Proposition \ref{prop:ordercompare}} \label{proof:prop23}
Consider the setting when $n, p$ and $ |J_A|$ are given and  the value of $M$ is fixed as $\Theta(1)$. We next examine $\rho_a$ in \eqref{eq:rhoform} as a function of integer $a$ in the following two cases.
	\paragraph*{(i) $|J_A|>{Mp}$}
	
	When $Mp/|J_A|<1$, both $(Mp/|J_A|)^{1/a}$ and $ (a!)^{{1}/{(2a)}}$ are increasing functions of integer $a$. 
	Thus $\rho_a$ 
	is an increasing function of $a$.  Since $a\in \mathbb{Z}^+$, $\rho_a$ reaches the minimum  value at $a=1$.
	 
	\paragraph*{(ii)  $|J_A| \leq {Mp}$}
Define $\tilde{M}=Mp/|J_A|$, and $f(a)=(a!)^{{1}/{(2a)}} (\tilde{M})^{{1}/{a}} $. 
Note that $\rho_a$ and $f(a)$ only differs by a constant. 
To find the minimum of $\rho_a$,  it suffices to examine the minimum of $f(a)$.

In the following, we  show  that  when $f(a)$ starts to not decrease at some value, it will strictly increase afterwards.  Specifically, we prove that $f(a+2)/f(a+1)>1$  if $f(a+1)/f(a)\geq 1$.  Note that
\begin{align*}
	\frac{f(a+1)}{f(a)}=&~  \frac{ \{(a+1)!\}^{\frac{1}{2(a+1)}} (\tilde{M})^{\frac{1}{a+1}} }{ (a!)^{\frac{1}{2a}} (\tilde{M})^{\frac{1}{a}} } \notag \\
	=&~ \Big[\frac{ \{(a+1)!\}^{a} \tilde{M}^{2a}}{ (a!)^{a+1} \tilde{M}^{2(a+1)}}   \Big]^{\frac{1}{2a(a+1)}} = \{ d(a) \times \tilde{M}^{-2}\}^{\frac{1}{2a(a+1)}}, 	
\end{align*} where $d(a)=(a+1)^a (a!)^{-1}.$ It follows that $f(a+1)/f(a)>1$ and $f(a+1)/f(a)=1$ are equivalent to $d(a)>\tilde{M}^2$ and $d(a)=\tilde{M}^2$, respectively. We next show that $d(a)$ is a strictly increasing function of $a$. In particular,
\begin{align*}
	\frac{d(a+1)}{d(a)}=\frac{(a+2)^{a+1} a!}{(a+1)^a (a+1)!}=\Big( \frac{a+2}{a+1} \Big)^{a+1}>1.
\end{align*}Therefore we have  $d(a+1)>\tilde{M}^2$ if $d(a)\geq \tilde{M}^2$; and equivalently this implies that  $f(a+2)/f(a+1)>1$ if $f(a+1)/f(a)\geq 1$.  

Suppose  $a_0$ is the first integer such that $d(a_0)\geq \tilde{M}^2$, i.e.,  for any integer $ 1\leq a<a_0$,  $d(a)< \tilde{M}^2$. By the analysis above, we know  $f(a)$ is decreasing  when $a<a_0$, and $f(a)$ is strictly increasing when $a>a_0$. Thus $a_0$ achieves the minimum of $f(a)$. Since  $d(a)$ is a strictly increasing function of $a$, we know $a_0<\infty$ for fixed $\tilde{M}$, and $a_0$ increases as $\tilde{M}$ increases. Therefore the second part of proposition \ref{prop:ordercompare}  is proved.

\subsection{Proof of Proposition \ref{prop:exttestpowerana}} \label{proof:prop24}
\begin{proof}
Consider the simplified test statistic given in \eqref{eq:uinftyimplified}.
We assume   $\mathrm{E}(x_{i,j})=0$ and $\mathrm{var}(x_{i,j}^2)=1$, $\forall j=1,\ldots ,p$ without loss of generality. It is then equivalent to  examine 
$
	\mathcal{U}(\infty) =\max_{1\leq j_1 < j_2 \leq p} | {\sum_{k=1}^n x_{k,j_1}x_{k,j_2}}/{  n} |.
$   
We next prove (i) and (ii) of Proposition \ref{prop:exttestpowerana} in the following Sections \ref{sec:lowerboundproof} and  \ref{sec:uppboundproof},  respectively.


\subsubsection{Proof of (i)} \label{sec:lowerboundproof} 
Under the alternative, 
we consider $n$ i.i.d. observations $(x_{k,1}, x_{k,2})$, satisfying $\mathrm{E}(x_{k,1}x_{k,2})=\rho$, for $k=1,\ldots,n$. Then by Condition \ref{cond:ellpmoment},
$
 \mathrm{var}(x_{k,1}x_{k,2})=\mathrm{E}(x_{k,1}^2 x_{k,2}^2)-[\mathrm{E}(x_{k,1}x_{k,2})]^2= \kappa_1(1+2\rho^2)-\rho^2.
$ 
The power of $\mathcal{U}(\infty)$ satisfies that
\begin{eqnarray}
	&& P ( |\, \mathcal{U}(\infty)\, | \geq t_p  ) \label{eq:lowerbound1powerinfty} \\
	&=&  P\Big( \max_{1\leq j_1 < j_2 \leq p} \Big| {\sum_{k=1}^n x_{k,j_1}x_{k,j_2}}/{n} \Big| \geq t_p \Big) \notag \\
	&\geq &  P \Big( \Big| { \sum_{k=1}^n x_{k,1}x_{k,2} }/{n}  \Big| \geq t_p \Big)\notag\\ 
	&\geq & P \Big({ \sum_{k=1}^n x_{k,1}x_{k,2} }/{n}\geq t_p \Big)  \notag \\
   &= & P \Biggr(\frac{ \sum_{k=1}^n (x_{k,1}x_{k,2} -\rho) }{\sqrt{n} \sqrt{\mathrm{var}(x_{k,1}x_{k,2})}}\geq \frac{\sqrt{n}( t_p -\rho) }{\sqrt{\mathrm{var}(x_{k,1}x_{k,2})}} \Biggr). \notag
 \end{eqnarray} 
We apply the central limit theorem on $x_{k,1}x_{k,2}$, $k=1,\ldots,n$, and obtain
\begin{eqnarray*}
	\frac{ \sum_{k=1}^n (x_{k,1}x_{k,2}  - \rho) }{\sqrt{n} \sqrt{\mathrm{var}(x_{k,1}x_{k,2})} } \xrightarrow{D} \mathcal{N}(0,1).
\end{eqnarray*} 
 Suppose $Z$ follows a standard Gaussian distribution.  As $\log p\rightarrow \infty$, $\log p/n=o(1)$, and by  Berry-Esseen Theorem,  we have
 \begin{eqnarray*}  
 \eqref{eq:lowerbound1powerinfty}  &\geq & P \Biggr( Z \geq \frac{\sqrt{n}( t_p -\rho) }{\sqrt{\mathrm{var}(x_{k1}x_{k2})}} \Biggr) -\frac{C \mathrm{E}|x_{k1}x_{k2}|^3 }{ [\mathrm{var}(x_{k1}x_{k2})]^{\frac{3}{2}} \sqrt{n}} \notag \\
	&\geq & P \Biggr( Z \geq \frac{\sqrt{n}[ n^{-1/2}\sqrt{4\log p} -\rho] }{\sqrt{\kappa_1(1+2\rho^2)-\rho^2}} \Biggr) -\frac{C \sqrt{\mathrm{E}|x_{k1}|^6 \mathrm{E}|x_{k2}|^6 } }{ [\mathrm{var}(x_{k1}x_{k2})]^{\frac{3}{2}} \sqrt{n}} \\
	&\geq & P ( Z \geq  C(2-c_1)\sqrt{\log p}) -\frac{C }{ \sqrt{n}}   \\
	&\rightarrow & 1 + o(1),
\end{eqnarray*} where the second inequality uses $t_p \leq n^{-1/2} \sqrt{4\log p}$ when $p$ is sufficiently large; the third inequality uses $\rho\geq c_1\sqrt{\log p/n}$; and   the last step of  convergence holds when $c_1>2$.

\subsubsection{Proof of (ii)}  \label{sec:uppboundproof}
Recall the notation $J_A$ and $J_A^c$ in Section \ref{sec:updatepoweronesampf}. 
Under the considered alternative, when $(j_1,j_2)\in J_A$, $\mathrm{E}(x_{k,j_1}x_{k,j_2})=\rho$; and when $(j_3,j_4)\in J_A^c$, $\mathrm{E}(x_{k,j_3}x_{k,j_4})=0$. We have
\begin{eqnarray}
	&& P ( |\, \mathcal{U}(\infty)\, | \geq t_p  )\label{eq:twoboundstogether} \\
	 &\leq & \sum_{1\leq j_1 < j_2 \leq p } P\Big( \Big| {\sum_{k=1}^n x_{k,j_1}x_{k,j_2}}/{n} \Big| \geq t_p \Big) \notag \\
	&\leq & \frac{1}{2} \sum_{(j_1,j_2)\in J_A} P\Big( \Big| {\sum_{k=1}^n x_{k,j_1}x_{k,j_2}}/{n} \Big| \geq t_p \Big)  \notag \\
	& & + \frac{1}{2}\sum_{(j_3,j_4)\in J_A^c} P\Big(\Big| {\sum_{k=1}^n x_{k,j_3}x_{k,j_4}}/{n} \Big| \geq t_p   \Big). \notag
\end{eqnarray} 
Next we show that under the conditions of Proposition \ref{prop:exttestpowerana},
\begin{eqnarray}
	\sum_{(j_1,j_2)\in J_A} P\Big( \Big| {\sum_{k=1}^n x_{k,j_1}x_{k,j_2} }/{n} \Big| \geq t_p \Big) \rightarrow 0, \label{eq:smallbound1}
\end{eqnarray} and
\begin{eqnarray}
	&& \frac{1}{2} \sum_{(j_3,j_4)\in J_A^c}  P\Big(\Big| {\sum_{k=1}^n x_{k,j_3}x_{k,j_4}}/{n} \Big| \geq t_p   \Big) \leq  \log (1-\alpha)^{-1}. \label{eq:smallbound2}
\end{eqnarray}


%
\paragraph{Proof of \eqref{eq:smallbound1}}\label{sec:pfsmallbound1}
To prove \eqref{eq:smallbound1}, we derive an upper bound of $P( | {\sum_{k=1}^n x_{k,j_1}x_{k,j_2} }/{n}| \geq t_p )$ for each $(j_1,j_2)\in J_A$ by  Lemma 6.8 in \citet{cai2011}.
In the following, we consider a fixed index pair $(j_1,j_2)$; and  for easy presentation, 
we write $m_0=\sqrt{\mathrm{var}(x_{k,j_1}x_{k,j_2})}$ and $\xi_k=({x_{k,j_1}x_{k,j_2}-\rho})/m_0.$    
When $(j_1,j_2)\in J_A$, we have $\mathrm{E}(\xi_k)=0$, $\mathrm{var}(\xi_k)=1$, and by Condition \ref{cond:ellpmoment}, $m_0^2=\kappa_1(1+2\rho^2)-\rho^2$. It follows that
\begin{eqnarray*}
 P\Big( {\sum_{k=1}^n x_{k,j_1}x_{k,j_2} }/{n} \geq t_p \Big)
 =  P\Big( \frac{\sum_{k=1}^n \xi_k}{\sqrt{n \log p}} \geq y_n \Big),
\end{eqnarray*} where $y_n=\sqrt{{n}/{\log p}} m_0^{-1}(t_p -\rho)$. 
We next show that $y_n$ and $\xi_k, k=1,\ldots,n$ satisfy the conditions of  Lemma 6.8 in \cite{cai2011}. First  note that $y_n\rightarrow {y}=(2-c_2)m_0^{-1}$, and ${y}>0$ as $c_2<2$. 
 We then show that $\mathrm{E} \{\exp(\tilde{t}_0 |\xi_k|^{\vartheta})\}<\infty$ for some $\tilde{t}_0>0$ and $0<\vartheta\leq 1$.  In particular, given $\varsigma$ and $t_0$ in Proposition \ref{prop:exttestpowerana}, we take $\vartheta =\varsigma/2\in (0,1]$ and $\tilde{t}_0=t_0(2m_0)^{\vartheta}/2>0$.  By Lemma \ref{lm:inequaluvtheta},
\begin{eqnarray*}
	&& |x_{k,j_1} x_{k,j_2} -\rho|^{\vartheta} \leq (|x_{k,j_1} x_{k,j_2}| + |\rho|)^{\vartheta} \leq |x_{k,j_1} x_{k,j_2}|^{\vartheta}+  |\rho|^{\vartheta}\notag \\
	&&\leq \Big( \frac{x_{k,j_1}^2+ x_{k,j_2}^2}{2} \Big)^{\vartheta} + |\rho|^{\vartheta} \leq \frac{1}{2^{\vartheta}} (|x_{k,j_1}|^{2\vartheta}+|x_{k,j_2}|^{2\vartheta}) + |\rho|^{\vartheta}. 
\end{eqnarray*} It follows that 
\begin{eqnarray}
	\quad \quad && \mathrm{E} \exp(\tilde{t}_0 |\xi_k|^{\vartheta}) \label{eq:expmomentbdd}  \\
	 &\leq & \mathrm{E} \exp \Big[ \frac{\tilde{t}_0}{(2m_0)^{\vartheta}}(|x_{k,j_1}|^{2\vartheta}+|x_{k,j_2}|^{2\vartheta}) + \frac{\tilde{t}_0}{m_0^{\vartheta}} |\rho|^{\vartheta}  \Big] \notag \\
	&=& \mathrm{E}  [   \exp ( {2^{-1}{t}_0|x_{k,j_1}|^{\varsigma}} ) \times   \exp ({2^{-1}{t}_0|x_{k,j_2}|^{\varsigma} }) ] \times \exp ( t_0 2^{\vartheta-1} |\rho|^{\vartheta}  )  \notag \\
	&\leq &\sqrt{\mathrm{E}  [   \exp ( {{t}_0|x_{k,j_1}|^{\varsigma}} )] \times \mathrm{E} [  \exp ({{t}_0|x_{k,j_2}|^{\varsigma} }) ]  }\times \exp ( t_0 2^{\vartheta-1} |\rho|^{\vartheta}  ),  \notag
\end{eqnarray}where the last inequality follows from the   H\"older's inequality. By the conditions in Proposition \ref{prop:exttestpowerana}, we know  $\max_{(j_1,j_2)\in J_A}\mathrm{E}(t_0|x_{k,j_1}|^{\varsigma})\times \mathrm{E}(t_0|x_{k,j_2}|^{\varsigma})<\infty$ and $\rho\leq c_2\sqrt{\log p/n}=o(1)$. Therefore, $\eqref{eq:expmomentbdd}<\infty$. In summary, $y_n$ and $\xi_k, k=1,\ldots,n$ satisfy the conditions of  Lemma 6.8 in \cite{cai2011}.

By Lemma 6.8 in \cite{cai2011}, as $\log p=o(n^{\beta})$ and $\beta=\vartheta/(2+\vartheta)=\varsigma/(4+\varsigma)$,  
\begin{eqnarray}
	P\Big( \frac{\sum_{k=1}^n \xi_k}{\sqrt{n \log p}} \geq y_n \Big) \simeq \frac{p^{-y_n^2/2} (\log p)^{-1/2}}{\sqrt{2\pi} y}.  \label{eq:snlogpapproxprob}
\end{eqnarray}	
Let $z_0=-\log (8\pi)-2\log \log (1-\alpha)^{-1}$, then we can write $t_p=n^{-1/2}\{4\log p-\log\log p+z_0\}^{1/2}$ 
and
\begin{eqnarray*}
	y_n^2 &=& \frac{n}{\log p} (t_p-\rho)^2 \times \frac{1}{\mathrm{var}(x_{k,j_1}x_{k,j_2})} \notag \\
	&=& \frac{n}{\log p} (t_p^2 -2\rho t_p + \rho^2) \times \frac{1}{\mathrm{var}(x_{k,j_1}x_{k,j_2})} \notag \\
	 &\geq & \frac{1}{\mathrm{var}(x_{k,j_1}x_{k,j_2})} \times \Big\{ \frac{1}{\log p} \Big( 4\log p - \log \log p +z_0 \Big)   \notag \\
	&&- \frac{2c_2\sqrt{\log p}\sqrt{4\log p-\log \log p+ z_0} }{\log p } + \frac{c_2^2 \log p}{\log p }\Big\}, 
\end{eqnarray*} where the last inequality holds when $\rho\leq c_2\sqrt{\log p/n}$ and $c_2<2$.
Then
\begin{align*}
	&~p^{-y_n^2/2}\notag \\
=&~ \exp ( - (\log p) y_n^2/2 ) \notag \\
\leq &~\exp \Biggr\{ -  \frac{1}{\mathrm{var}(x_{k,j_1}x_{k,j_2})} \Big[ \frac{1}{2}\Big( 4\log p - \log \log p -\log (8\pi)-2\log \log (1-\alpha)^{-1} \Big) \notag  \\
&~ \quad \quad \quad \quad \quad \quad \quad \quad \quad \quad \quad - c_2 \sqrt{\log p}\sqrt{4\log p-\log \log p+ z_0}  +\frac{c_2^2\log p}{2}  \Big] \Biggr\}  \notag   \\
=&~ \Big\{ p^{-2} \sqrt{\log p} \times \sqrt{8\pi} \log (1-\alpha)^{-1} \times p^{-\frac{c_2^2}{2}}  \notag \\
&~ \quad \times \exp \Big( c_2 \sqrt{\log p}\sqrt{4\log p-\log \log p+ z_0}  \Big) \Big\}^{1/\{\mathrm{var}(x_{k,j_1}x_{k,j_2})\}}.  
\end{align*}	 
By Condition \ref{cond:ellpmoment}, $\mathrm{var}(x_{k,j_1}x_{k,j_2})=\kappa_1+(2\kappa_1-1)\rho^2$; and as $\rho=o(1)$, there exists a  constant $m>0$ such that $\mathrm{var}(x_{k,j_1}x_{k,j_2}) \leq \kappa_1 + m$. 
 Thus 
\begin{eqnarray*}
	&& p^{-y_n^2/2} (\log p)^{-1/2} \notag \\
	&\leq &(\log p)^{-1/2} \Big\{ p^{-2} \sqrt{\log p}\times \sqrt{8\pi} (\log (1-\alpha)^{-1}) p^{-\frac{c_2^2}{2}} \notag \\
	&& \times  \exp \Big( c_2 \sqrt{\log p}\sqrt{4\log p-\log \log p+ z_0}  \Big) \Big\}^{1/\{\mathrm{var}(x_{k1}x_{k2})\}} \notag  \\
	&\leq & (\log p)^{-1/2} \Big[  \sqrt{8\pi} \log (1-\alpha)^{-1} \sqrt{\log p} \times p^{-2-\frac{c_2^2}{2}+2c_2} \Big]^{1/(\kappa_1+m)}.
\end{eqnarray*}
Recall that $y=(2-c_2)[\mathrm{var}(x_{k,j_1}x_{k,j_2})]^{-1/2}$. Then by \eqref{eq:snlogpapproxprob},
\begin{align}
 \quad &~ \frac{1}{2}\sum_{(j_1,j_2)\in J_A} P\Big(  {\sum_{k=1}^n x_{k,j_1}x_{k,j_2}}/{n} \geq t_p \Big) \label{eq:smallpartupperbd0} \\
	\quad =&~\frac{1}{2}\sum_{(j_1,j_2)\in J_A} P\Big( \frac{\sum_{k=1}^n \xi_k}{\sqrt{n \log p}} \geq y_n \Big) \notag \\ 
	\leq &~ \frac{|J_A|}{2} \frac{(\log p)^{-1/2} }{y\sqrt{2\pi}} \Big(\sqrt{8\pi} \log (1-\alpha)^{-1} \sqrt{\log p} \times p^{-2-\frac{c_2^2}{2}+2c_2} \Big)^{\frac{1}{\kappa_1+m}} \notag \\
	=& ~C_{\alpha}\exp\Biggr( 2 \log \Big[ p^{ -\frac{(1-c_2+{c_2^2}/{4})}{ (\kappa_1+m)} }\Big\{\sqrt{|J_A|} (\log p)^{\frac{1}{4(\kappa_1+m)}-\frac{1}{4}}\Big\}\Big]  \Biggr), \notag
\end{align} where $C_{\alpha}=\frac{1}{2y\sqrt{2\pi}}[\sqrt{8\pi}\log (1-\alpha)^{-1}]^{1/(\kappa_1+m)}$.
Thus, $\eqref{eq:smallpartupperbd0}\rightarrow 0$ when
\begin{eqnarray*}
	{p^{- \frac{(1- {c_2}/{2} )^2}{\kappa_1+m} }}{\sqrt{|J_A|} (\log p)^{\frac{1}{4(\kappa_1+m)}-\frac{1}{4}}}\to 0.
\end{eqnarray*} 

Similarly, we have 
\begin{eqnarray} 
	 && \sum_{(j_1,j_2)\in J_A}P\Big( \frac{\sum_{k=1}^n x_{k,j_1}x_{k,j_2} }{n} \leq -t_p \Big) \label{eq:smallpartupperbd2} \\
 &=& \sum_{(j_1,j_2)\in J_A}  P\Big( \frac{\sum_{k=1}^n (-x_{k,j_1}x_{k,j_2} +\rho)}{n \sqrt{ \mathrm{var}(x_{k,j_1}x_{k,j_2}) }} \geq \frac{t_p +\rho}{\sqrt{\mathrm{var}(x_{k,j_1}x_{k,j_2})}} \Big), \notag
\end{eqnarray} and $\eqref{eq:smallpartupperbd2}\rightarrow 0$ following the similar arguments as above. In summary, \eqref{eq:smallbound1} holds when $J_A=o(1)p^{  \frac{2(1- {c_2}/{2} )^2}{\kappa_1+m} } (\log p)^{\frac{1}{2}-\frac{1}{2(\kappa_1+m)}} $ for some $m>0$.

\smallskip


%
\paragraph{Proof of  \eqref{eq:smallbound2}}
Similarly to Section \ref{sec:pfsmallbound1}, we derive an upper bound of $P( {\sum_{k=1}^n x_{k,j_3}x_{k,j_4}}/{n}\geq t_p )$ for each $(j_3,j_4)\in J_A^c$ by Lemma 6.8 in \cite{cai2011}. In the following, we consider a fixed index pair $(j_3,j_4)$; and for easy presentation, we write $\tilde{\xi}_k=x_{k,j_3}x_{k,j_4}/\sqrt{\kappa_1}$, $k=1,\ldots,n$. When $(j_3,j_4)\in J_A^c$, $\mathrm{E}(x_{k,j_3}x_{k,j_4})=0$ and $\mathrm{var}(x_{k,j_3}x_{k,j_4})=\mathrm{E}\{(x_{k,j_3}x_{k,j_4})^2\}=\kappa_1$, then we have $\mathrm{E}(\tilde{\xi}_k)=0$ and $\mathrm{var}(\tilde{\xi}_k)=1$. 
To prove  \eqref{eq:smallbound2}, we write
\begin{align*}
 P\Big( {\sum_{k=1}^n x_{k,j_3}x_{k,j_4}}/{n}\geq t_p   \Big) = P\Big( \frac{\sum_{k=1}^n \tilde{\xi}_k}{\sqrt{n \log p}} \geq \tilde{y}_n \Big),
\end{align*} where $\tilde{y}_n=\sqrt{{n}/{\log p}}\times t_p / \sqrt{\kappa_1} \rightarrow \tilde{y}=2/\sqrt{\kappa_1}.$
Similarly to Section \ref{sec:pfsmallbound1},
we know $\tilde{y}_n$ and $\tilde{\xi}_k$, $k=1,\ldots,n$ also satisfy the conditions of Lemma 6.8 in \cite{cai2011}.   Thus by Lemma 6.8 in \cite{cai2011}, for $z_0=-\log (8\pi)-2\log \log (1-\alpha)^{-1}$ and $t_p=n^{-1/2}\sqrt{4\log p-\log\log p+z_0}$,
\begin{eqnarray*}
	 &&~P\Big( \frac{\sum_{k=1}^n \tilde{\xi}_k}{\sqrt{n \log p}} \geq \tilde{y}_n \Big) \notag \\
	& \simeq &~ \frac{p^{-\tilde{y}_n^2/2} (\log p)^{-1/2}}{\sqrt{2\pi} \tilde{y}} \\
	 &=&~ p^{-2/\kappa_1} (\log p)^{1/(2\kappa_1)-1/2} \frac{\exp(-z_0/(2\kappa_1))}{\sqrt{2\pi}\tilde{y}}  \\
	 &\leq &~(8\pi)^{1/(2\kappa_1)}\frac{\sqrt{\kappa_1}}{2\sqrt{2\pi}} p^{-2/\kappa_1}(\log p)^{1/(2\kappa_1)-1/2}  \{\log (1-\alpha)^{-1}\}^{1/\kappa_1}.
\end{eqnarray*} Then for $\kappa_1\leq 1$ and a small $\alpha>0$, \begin{eqnarray}
	&&\frac{1}{2}\sum_{(j_1,j_2)\in J_A^c} P\Big( {\sum_{k=1}^n x_{k,j_3}x_{k,j_4}}/{n}  \geq t_p   \Big) \label{eq:sumupperboundorderk1} \\
	&\leq & \frac{1}{2} \frac{p(p-1)-|J_A|}{p^{2/\kappa_1} (\log p)^{-1/(2\kappa_1)+1/2} } (8\pi)^{1/(2\kappa_1)}\frac{\sqrt{\kappa_1}}{2\sqrt{2\pi}} \{\log (1-\alpha)^{-1}\}^{1/\kappa_1}, \notag
\end{eqnarray} which attains the maximum order at $\kappa_1=1$, when $\kappa_1\leq 1$ and $n, p \to \infty$. Therefore asymptotically,
$
	\eqref{eq:sumupperboundorderk1} \leq  2^{-1} \log (1-\alpha)^{-1}.
$ 
By similar arguments, we know when $n,p\rightarrow \infty$, 
\begin{eqnarray*}
	&&\frac{1}{2}\sum_{(j_3,j_4)\in J_A^c} P\Big( {\sum_{k=1}^n x_{k,j_3}x_{k,j_4}}/{n} \leq -t_p   \Big) \leq \frac{1}{2}  \log (1-\alpha)^{-1}.
\end{eqnarray*}   In summary,  we have \eqref{eq:smallbound2} holds.

Combining \eqref{eq:smallbound1} and \eqref{eq:smallbound2}, we obtain
$
 \eqref{eq:twoboundstogether} \leq \log (1-\alpha)^{-1}.
$
\end{proof}

\subsection{Conditions of Theorems \ref{thm:onesamplemean}--\ref{thm:altcltmeantest}} \label{sec:extseccond}


The conditions of Theorem \ref{thm:onesamplemean} are listed in the following Condition  \ref{cond:onesamplemean}. 
\begin{condition} \label{cond:onesamplemean} \quad

\begin{enumerate}
	\item[(1)] $\lim_{p\rightarrow \infty} \max_{1\leq j\leq p} \mathrm{E}(x_{j}-\mu_j)^4< \infty$;  $\lim_{p\rightarrow \infty} \min_{1\leq j\leq p} \mathrm{E}(x_{j}-\mu_j)^2>0$.
	\item[(2)] $\mathbf{x}$ is $\alpha$-mixing with $\alpha_{{x}}(s) \leq M\delta^s$, where $\delta \in (0,1)$ and $M>0$ are some constants. In addition, $\sum_{j_1,j_2=1}^p\sigma_{j_1,j_2}^a=\Theta(p)$.
\end{enumerate}	
\end{condition}

Condition \ref{cond:onesamplemean} is similar  to   Conditions \ref{cond:finitemomt} and \ref{cond:alphamixing} of Theorem \ref{thm:jointnormal}.  As  the mean is a lower order moment function than the covariance, Condition \ref{cond:onesamplemean} (1) is weaker than Condition \ref{cond:finitemomt} in  that  only the fourth  moments are needed to be uniformly bounded  instead of the eighth moments. Condition \ref{cond:onesamplemean} (2) is a regularization condition  of the structure of the covariance matrix.  
 

\bigskip

The conditions of Theorem \ref{thm:onesamplemean2} are list in the following Condition \ref{cond:onesamplemean2}.
\begin{condition} \label{cond:onesamplemean2} \quad
	
\begin{enumerate}
	\item[(1)]  There exists constant $B$ such that $B^{-1} \leq \lambda_{\min}(\boldsymbol{\Sigma}) \leq \lambda_{\max}(\boldsymbol{\Sigma}) \leq B$, where $\lambda_{\min}(\boldsymbol{\Sigma})$ and $\lambda_{\max}(\boldsymbol{\Sigma})$ denote the minimum and maximum eigenvalues of the covariance matrix $\boldsymbol{\Sigma}$; and  all correlations are bounded away from $-1$ and $1$, i.e., $\max_{1\leq j_1 \neq j_2\leq p}|\sigma_{j_1,j_2}|/(\sigma_{j_1,j_2} \sigma_{j_2,j_2})^{1/2}<1-\eta$ for some $\eta>0$.
	\item[(2)] $\log p=o(n^{1/4})$; $\max_{1\leq j \leq p}\mathrm{E}[\exp( h (x_{j} -\mu_{j} )^2 )] < \infty,$ for $h\in [-M_1,M_1]$, where $M_1>0$ is some constant. 
	\item[(3)] $\{ (x_{i,j}, i=1,\ldots,n): 1\leq j\leq p  \}$ is $\alpha$-mixing with $\alpha_{{{x}}}(s) \leq C\delta^s$, where $\delta \in (0,1)$ and $C>0$ is some constant, and  $\sum_{j_1,j_2=1}^p\sigma_{j_1,j_2}^a=\Theta(p)$.
\end{enumerate}	
\end{condition}

In Condition \ref{cond:onesamplemean2}, (1) and (2) are assumed to establish the extreme value distribution of $ \mathcal{U}(\infty)$, as in \citet{cai2014two} and \citet{xu2016adaptive}.
Furthermore, the mixing condition in Condition  (3) is used to establish the joint independence of finite order U-statistics and $\mathcal{U}(\infty)$, following the argument in \citet{hsing1995}.   

\bigskip

The conditions of Theorems \ref{thm:twosamplemean}--\ref{thm:altcltmeantest} are listed in the Condition \ref{cond:twosamplemeancond} below. 
\begin{condition} \label{cond:twosamplemeancond} \quad

\begin{enumerate}
		\item[(1)] There exists constant $B$ such that $B^{-1} \leq \lambda_{\min}(\boldsymbol{\Sigma}_x) \leq \lambda_{\max}(\boldsymbol{\Sigma}_x) \leq B$, where $\lambda_{\min}(\boldsymbol{\Sigma}_x)$ and $\lambda_{\max}(\boldsymbol{\Sigma}_x)$ denote the minimum and maximum eigenvalues of  $\boldsymbol{\Sigma}_x$; and  all correlations are bounded away from $-1$ and $1$, i.e., $\max_{1\leq j_1 \neq j_2\leq p}|\sigma_{x,j_1,j_2}|/(\sigma_{x,j_1,j_2} \sigma_{x,j_2,j_2})^{1/2}<1-\eta$ for some $\eta>0$. In addition, we assume the same assumptions  hold for $\boldsymbol{\Sigma}_y$. 
	\item[(2)] $n,p\rightarrow \infty$, $\log p=o(1)n^{1/4}$ and $n_x/n \rightarrow \gamma \in (0,1)$. In addition, $\max_{1\leq j \leq p}\mathrm{E}[\exp( h (x_{j}-\mu_{j} )^2 )] < \infty$ and $\max_{1\leq j \leq p}\mathrm{E}[\exp( h (y_{j}-\nu_{j} )^2 )] < \infty,$ for $h\in [-M,M]$, where $M$ is a positive constant.
	\item[(3)]  $\{ (x_{i,j}, i=1,\ldots,n) : 1\leq j\leq p  \}$ and $\{ (y_{i,j},i=1,\ldots,n) : 1\leq j\leq p  \}$ are $\alpha$-mixing with $\alpha_{{{x}}}(s) \leq C\delta_x^s$ and $\alpha_{{{y}}}(s) \leq C\delta_y^s$, where $\delta_x, \delta_y \in (0,1)$ and $C$ is some constant. We also assume $\sum_{j_1,j_2=1}^p\{\sigma_{x,j_1,j_2}/\gamma + \sigma_{y,j_1,j_2}/(1-\gamma)\}^a= \Theta(p)$. 
\end{enumerate}
\end{condition}

 Condition  \ref{cond:twosamplemeancond} is similar to Condition \ref{cond:onesamplemean2}. They are  assumed to establish both the limiting distributions  and asymptotic independence properties of $\mathcal{U}(a)$  and $\mathcal{U}(\infty)$ for testing two-sample mean. 

\subsection{Proof of Theorems \ref{thm:onesamplemean} and \ref{thm:onesamplemean2}} \label{sec:proof3132}
\begin{proof}
Under $H_0$, for $\mathcal{U}(a)$ in \eqref{eq:ustatonesamplemean}, we assume without loss of generality that $\boldsymbol{\mu}_0=\mathbf{0}$, and then write $\mathcal{U}(a) =\sum_{j=1}^p (P^n_a)^{-1} \sum_{1 \leq i_1 \neq \cdots \neq i_{a} \leq n}  \prod_{k=1}^a x_{i_k,j}.$ 

We start with the proof of Theorem \ref{thm:onesamplemean}.
Similarly to Section  \ref{sec:detailofjointnormal}, we first derive the variances and  the covariances of the U-statistics; and then prove the asymptotic joint normality of the U-statistics. In particular, for $\mathrm{var}\{ \mathcal{U}(a)\}$ in Theorem \ref{thm:onesamplemean}, as $\mathrm{E}\{\mathcal{U}(a) \}=0$ under $H_0$, 
\begin{align*}
	\mathrm{var}\{ \mathcal{U}(a)\}=\mathrm{E}\{\mathcal{U}^2(a)\}=& ({P^n_{a}})^{-2} \sum_{\substack{1\leq j_1 \leq p,\\ 1\leq j_2 \leq p}}\ \sum_{\substack{1 \leq i_1 \neq \cdots \neq i_{a} \leq n,\\ 1 \leq \tilde{i}_1 \neq \cdots \neq \tilde{i}_{a} \leq n}} \mathrm{E}\Big(\prod_{k=1}^a x_{i_k,j_1} x_{\tilde{i}_k,j_2}\Big).\end{align*}
Note that 	$
		\mathrm{E}(\prod_{k=1}^a x_{i_k,j_1} x_{\tilde{i}_k,j_2})= 0
$ when $\{i_1,\ldots, i_a\}\neq \{\tilde{i}_1,\ldots, \tilde{i}_a\}$; and $
		\mathrm{E}(\prod_{k=1}^a x_{i_k,j_1} x_{\tilde{i}_k,j_2})= \sigma_{j_1,j_2}^a
$ when $\{i_1,\ldots, i_a\}= \{\tilde{i}_1,\ldots, \tilde{i}_a\}$. Then
	\begin{eqnarray}
		\mathrm{var}\{ \mathcal{U}(a)\}= ({P^n_{a}})^{-1} \sum_{\substack{1\leq j_1,j_2 \leq p}} a! \sigma_{j_1,j_2}^a. \label{eq:meanvarpartsummedtermexp}
	\end{eqnarray} By Condition \ref{cond:onesamplemean},   $\sum_{1\leq j_1,j_2\leq p}\sigma_{j_1,j_2}^a=\Theta(p)$. Thus $\mathrm{var}\{ \mathcal{U}(a)\}=\Theta(pn^{-a})$. 

Second, we show that $\mathrm{cov}\{\mathcal{U}(a), \mathcal{U}(b)\}=0$. 
Note that $\mathrm{cov}\{\mathcal{U}(a), \mathcal{U}(b)\}=\mathrm{E}\{\mathcal{U}(a)\mathcal{U}(b)\}$ under $H_0$, and
\begin{align*}
	\mathrm{E}\{\mathcal{U}(a)\mathcal{U}(b)\}=(P^n_aP^n_b)^{-1} \sum_{\substack{1\leq j_1 \leq p,\\ 1\leq j_2 \leq p}}\ \sum_{\substack{1 \leq i_1 \neq \cdots \neq i_{a} \leq n,\\ 1 \leq \tilde{i}_1 \neq \cdots \neq \tilde{i}_{a} \leq n}} \mathrm{E}\Big(\prod_{k=1}^a x_{i_k,j_1} \prod_{t=1}^b  x_{\tilde{i}_t,j_2}\Big).
\end{align*} Since $a\neq b$, $\{i_1,\ldots, i_a\}\neq \{\tilde{i}_1,\ldots, \tilde{i}_b \}$. Suppose  there exists an index $i\in \{i_1,\ldots, i_a\}$ and $i\not \in \{\tilde{i}_1,\ldots, \tilde{i}_b \}$. Then under $H_0$,
\begin{align*}
	\mathrm{E}\Big( \prod_{k=1}^a x_{i_k,j_1} \prod_{t=1}^b  x_{\tilde{i}_t,j_2} \Big)=\mathrm{E}(x_{i,j})\mathrm{E}(\mathrm{all\ the\ remaining\ terms})= 0.
\end{align*}  Therefore, $\mathrm{E}\{\mathcal{U}(a)\mathcal{U}(b) \}=0$.


In summary, the  covariance matrix of $[\mathcal{U}(a_1)/\sigma(a_1),\allowbreak \ldots, \mathcal{U}(a_m)/\sigma(a_m)]^{\intercal}$ asymptotically converges to  $I_m$. To finish the proof of Theorem \ref{thm:onesamplemean}, it remains to show that the  joint limiting distribution of the U-statistics is normal. 
By the  Cram\'{e}r-Wold theorem, it is sufficient to prove that any fixed linear combination of these U-statistics  converges to a normal distribution. 
Similarly to  Section \ref{sec:detailofjointnormal},  we use the martingale central limit theorem \cite[p.476]{billingsley1995}. 
Specifically, we redefine $Z_n$  as below with $\sum_{r=1}^m t_r^2=1$, and prove that
\begin{eqnarray}
	Z_n:=\sum_{r=1}^mt_r \mathcal{U}(a_r)/\sigma(a_r)\xrightarrow{D} \mathcal{N}(0,1).\label{eq:znnormalaltonesam}
\end{eqnarray}
 With the redefined $Z_n$, we define $\mathrm{E}_k(\cdot)$ in the same way as in Section \ref{sec:detailofjointnormal}, and still define  $D_{n,k}=(\mathrm{E}_k-\mathrm{E}_{k-1})Z_n$ and $\pi_{n,k}^2=\mathrm{E}_{k-1}(D_{n,k}^2)$.  Similarly to Section \ref{sec:detailofjointnormal}, we have $D_{n,k}=(\mathrm{E}_{k}-\mathrm{E}_{k-1})Z_n=\sum_{r=1}^m t_rA_{n,k,a_r},$ where we redefine $A_{n,k,a_r}=(\mathrm{E}_k-\mathrm{E}_{k-1})\{\mathcal{U}(a_r)/\sigma(a_r)\}$. In addition, similarly to Lemma \ref{lm:cltabnkform}, 
 we obtain that when $k<a_r$,  $A_{n,k,a_r}=0$; and when $k\geq a_r$, 
\begin{eqnarray*}
	A_{n,k,a_r}=\frac{a_r}{\sigma(a_r)P^n_{a_r} } \sum_{j=1}^p \sum_{1 \leq i_1 \neq \cdots \neq i_{a_r-1} \leq k-1}   x_{k,j} \times \prod_{t=1}^{a_r-1}x_{i_t,j}. 
\end{eqnarray*}  
Given the form of $A_{n,k,a_r}$, we can obtain the forms of $D_{n,k}$ and $\pi_{n,k}^2$. To prove \eqref{eq:znnormalalt}, by the martingale central limit theorem, it suffices to prove the following Lemma \ref{lm:cltonesammeanlm}. 
\begin{lemma}\label{lm:cltonesammeanlm}
Under the conditions of Theorem \ref{thm:onesamplemean}, $\mathrm{var}(\sum_{k=1}^n\pi^2_{n,k})\to 0$ and $\sum_{k=1}^n\mathrm{E}(D_{n,k}^4)\to 0$. 	
\end{lemma} 
\begin{proof}
\textit{See Section \ref{sec:onesammeanlm} on Page \pageref{sec:onesammeanlm}.}	
\end{proof}
With Lemma \ref{lm:cltonesammeanlm}, the asymptotic joint normality in Theorem \ref{thm:onesamplemean} is obtained by the martingale central limit theorem. For Theorem \ref{thm:onesamplemean2}, the limiting distribution of $\mathcal{U}(\infty)$ follows from \citet{cai2014two}. In addition, the asymptotic independence between $\mathcal{U}(a)/\sigma(a)$ and $n\mathcal{U}(\infty)-\tau_p$ can be obtained similarly as the proof of Theorem \ref{THM:TWOSAMPLEMANINF}. We defer the details to Section  \ref{sec:proofasymindptwosam}.




\subsection{Proof of Theorem \ref{thm:twosamplemean}} \label{sec:proofthm33}


By the following Proposition \ref{prop:locinvatwosam}, we assume that under $H_0$, $\boldsymbol{\mu}=\boldsymbol{\nu}=\mathbf{0}$, without loss of generality. 
	
\begin{proposition}\label{prop:locinvatwosam}
 $\mathcal{U}(a)$  constructed in \eqref{eq:teststattwosample} and  \eqref{eq:uinftytwosammean} are location invariant; that is, for any vector $\mathbf{\Delta}\in {\mathbb  R}^p$, the U-statistic  constructed based on the transformed data $\{\mathbf{x}_i+\Delta: i=1,\ldots, n_x\}$ and $\{\mathbf{y}_i+\Delta: i=1,\ldots, n_y\}$ is still $\mathcal{U}(a)$.
\end{proposition}
\noindent Proposition \ref{prop:locinvatwosam} can be obtained straightforwardly from the definitions $ \mathcal{U}(a)= \sum_{j=1}^p (P^{n_x}_{a} P^{n_y}_{a})^{-1}\times \sum_{1 \leq k_1 \neq \ldots \neq k_a \leq n_x;\atop 1 \leq s_1 \neq \ldots \neq s_a \leq n_y}   \prod_{t=1}^a (x_{k_t,j}-y_{s_t,j})$ in \eqref{eq:teststattwosample}, and $\mathcal{U}(\infty)= \max_{1\leq j \leq p} \sigma_{j,j}^{-1}\times (\bar{x}_{j}- \bar{y}_{j})^2$ in \eqref{eq:uinftytwosammean}. The proof is thus skipped.  

\smallskip


The following proof proceeds by  deriving the variances,  covariances and asymptotic joint normality of the U-statistics. 
Particularly, the next Lemma \ref{lm:twosamvar}  derives the asymptotic form of $\sigma^2(a)$ in Theorem \ref{thm:twosamplemean}.
\begin{lemma} \label{lm:twosamvar} Under the conditions of Theorem \ref{thm:twosamplemean},
\begin{eqnarray*}
	\mathrm{var}\{\mathcal{U}(a)\}\simeq  \sum_{1\leq j_1,j_2\leq p} a! \Big(\frac{\sigma_{x,j_1,j_2}}{n_x} + \frac{\sigma_{y,j_1,j_2}}{n_y}\Big)^a = \Theta(pn^{-a}).
\end{eqnarray*} When $\sigma_{x,j_1,j_2}=\sigma_{y,j_1,j_2}=\sigma_{j_1,j_2},$ we have $\mathrm{var}[\mathcal{U}(a)]\simeq \sum_{j_1,j_2=1}^pa!(n_x+n_y)^a\sigma_{j_1,j_2}^a/(n_xn_y)^a$. 
\end{lemma}
\begin{proof}
\textit{See Section  \ref{sec:lmtwosamvar} on Page \pageref{sec:lmtwosamvar}.}	
\end{proof}
\noindent In addition, the following Lemma \ref{lm:twosamcov} shows that different $\mathcal{U}(a)$'s of finite $a$ are uncorrelated.
\begin{lemma} \label{lm:twosamcov}
Under the conditions of Theorem \ref{thm:twosamplemean}, 
 for finite integers $a\neq b$, $\mathrm{cov}\{ \mathcal{U}(a),\mathcal{U}(b) \}=0.$
\end{lemma} 
\begin{proof}
\textit{See Section  \ref{sec:prooftwosamcov} on Page \pageref{sec:prooftwosamcov}.}	
\end{proof}
\noindent We  then know $\mathrm{cov}\{\mathcal{U}(a_1)/\sigma(a_1),\ldots,\mathcal{U}(a_m)/\sigma(a_m)\}=I_m$ by Lemmas \ref{lm:twosamvar} and \ref{lm:twosamcov}. 
The next Lemma \ref{lm:twosamjointnormal} further proves the asymptotic joint normality of the U-statistics.
\begin{lemma}\label{lm:twosamjointnormal}
Under the conditions of Theorem \ref{thm:twosamplemean},
 for finite integers $a_1,\ldots,a_m$, $\{\mathcal{U}(a_1)/\sigma(a_1),\ldots,\mathcal{U}(a_m)/\sigma(a_m)\}\xrightarrow{D} \mathcal{N}(0,I_m). $
\end{lemma}
\begin{proof}
\textit{See Section \ref{sec:prooftwosamjointnormal} on Page \pageref{sec:prooftwosamjointnormal}.}	
\end{proof}
Combining Lemmas \ref{lm:twosamvar}--\ref{lm:twosamjointnormal}, we finish the proof of Theorem  \ref{thm:twosamplemean}.  


\subsection{Proof of Theorem \ref{THM:TWOSAMPLEMANINF}} \label{sec:proofasymindptwosam}
For $\mathcal{U}(\infty)$ in \eqref{eq:uinftytwosammean}, the limiting distribution of $\mathcal{U}(\infty)$ is established in  \citet{cai2014two} and \cite{xu2016adaptive}.  
We next prove the asymptotic independence between $\mathcal{U}(\infty)$ and $\mathcal{U}(a)$ by a similar argument to that in \citet{hsing1995}, see also \cite{xu2016adaptive}. In this proof, we reserve the notation $P$ for the probability measure on which $x_{i,j}$ and $y_{i,j}$ are defined, and the expectation with respect to $P$ is denoted as $\mathrm{E}$. Define $\tilde{\mathcal{U}}_c(a)/\sigma(a)$ on the conditional probability measure $\tilde{P}$, given the event $n_xn_y\mathcal{U}(\infty)/(n_x+n_y)-\tau_p\leq u$ such that
\begin{align*}
&~\tilde{P}\Big\{\tilde{\mathcal{U}}_c(a)/\sigma(a)\leq u'\Big\} \notag \\
=&~{P}\Big\{{\mathcal{U}}(a)/\sigma(a)\leq u' ~\Big| ~ \frac{n_xn_y}{n_x+n_y}\mathcal{U}(\infty)\leq \tau_p+u \Big\}.	
\end{align*}
The expectation with respect to $\tilde{P}$ is denoted by $\tilde{\mathrm{E}}$. 
To show the asymptotic independence, it is sufficient to prove the following Lemma \ref{lm:condindpmeas}.
\begin{lemma}\label{lm:condindpmeas}
Under the conditions of  Theorem \ref{THM:TWOSAMPLEMANINF}, 
$\tilde{\mathcal{U}}_c(a)/\sigma(a)\xrightarrow{D} \mathcal{N}(0,1)$ on the conditional measure $\tilde{P}$. 
\end{lemma}
\begin{proof}
\textit{See Section   \ref{sec:condindpmeas} on Page \pageref{sec:condindpmeas}.}
\end{proof}



\subsection{Proof of Theorem \ref{thm:altcltmeantest}} \label{sec:proofthem34}
 By Proposition \ref{prop:locinvatwosam}, we assume $\mathrm{E}(\mathbf{y})=\boldsymbol{\nu}=\mathbf{0}$, without loss of generality. Then under the considered alternative $\mathcal{E}_A$, $\mathrm{E}(\mathbf{x})=\boldsymbol{\mu}=\{\mu_j=\rho: j=1,\ldots,k_0; \mu_j=0: j=k_0+1,\ldots, p \}$. 
 Define $\varphi_{j_1,j_2}=\sigma_{j_1,j_2}+\mu_{j_1}\mu_{j_2}$. We have $\mathrm{E}(x_{i,j_1}x_{i,j_2})=\varphi_{j_1,j_2}$, and under $\boldsymbol{\nu}=\mathbf{0}$, $\mathrm{E}(y_{i,j_1}y_{i,j_2})=\sigma_{j_1,j_2}$.  

Similarly to the proof of Theorem \ref{thm:cltalternative} in Section \ref{sec:updatepoweronesampf},  we  decompose $\mathcal{U}(a)=T_{a,1}+T_{a,2}$, where 
\begin{align}
	T_{a,1}=& \sum_{j=1}^{k_0}\sum_{c=0}^{a}  \sum_{\substack{ 1\leq k_1 \neq \cdots \neq k_c\leq n_x, \\ 1 \leq s_1 \neq \cdots \neq s_{a-c} \leq n_y} } G(a,c)\prod_{t=1}^c x_{k_{t},j} \prod_{m=1}^{a-c}  y_{s_{m},j}, \label{eq:altmeandeft}\\
  T_{a,2}=& \sum_{j=k_0+1}^{p} \sum_{c=0}^{a}  \sum_{\substack{1\leq k_1 \neq \cdots \neq k_c\leq n_x, \\ 1 \leq s_1 \neq \cdots \neq s_{a-c} \leq n_y} } G(a,c)\prod_{t=1}^c x_{k_{t},j} \prod_{m=1}^{a-c}  y_{s_{m},j},
  \notag  
\end{align} with $G(a,c)=(-1)^{a-c}  \binom{a}{c} (P^{n_x}_{c} P^{n_y}_{a-c})^{-1}$. Then $\mathrm{E}(T_{a,1})=\sum_{j=1}^{k_0} (\mu_j-\nu_j)^a=k_0\rho^a$ and $\mathrm{E}(T_{a,2})=\sum_{j=k_0+1}^p (\mu_j-\nu_j)^a=0$.



To prove Theorem \ref{thm:altcltmeantest}, we derive the variances,  covariances, and asymptotic joint normality of the U-statistics. Particularly, 
the next Lemma \ref{lm:twomeanaltvar} gives the asymptotic form of $\sigma^2(a)=\mathrm{var}\{\mathcal{U}(a)\}$, and shows that $T_{a,2}$ is the leading component. 
\begin{lemma}\label{lm:twomeanaltvar}
Under the conditions of Theorem \ref{thm:altcltmeantest}, 
\begin{align}
	\mathrm{var}\{\mathcal{U}(a)\}\simeq \sum_{k_0+1\leq j_1,j_2\leq p} a! \Big(\frac{\sigma_{x,j_1,j_2}}{n_x} + \frac{\sigma_{y,j_1,j_2}}{n_y}\Big)^a.\label{eq:varalttwosammean}
\end{align}
$\mathrm{var}(T_{a,2})=\Theta(pn^{-a})$	and $\mathrm{var}(T_{a,1})=o(1)\mathrm{var}(T_{a,2})$. It follows that $\{T_{a,1}-\mathrm{E}(T_{a,1}) \}/\sigma(a) \xrightarrow{P} 0.$
\end{lemma}
\begin{proof}
\textit{See Section  \ref{sec:pftwomeanaltvar} on Page \pageref{sec:pftwomeanaltvar}.}
\end{proof}
\noindent In addition, the following Lemma \ref{lm:twomeanaltcov} shows that the covariance between two U-statistics asymptotically converges to 0.
\begin{lemma}\label{lm:twomeanaltcov}
Under the conditions of Theorem \ref{thm:altcltmeantest}, for two  finite integers $a\neq b$, $\{\sigma(a)\sigma(b)\}^{-1}\mathrm{cov}\{\mathcal{U}(a),\mathcal{U}(b)  \}\to 0$. 
\end{lemma}
\begin{proof}
\textit{See Section \ref{sec:pftwomeanaltcov} on Page \pageref{sec:pftwomeanaltcov}.}
\end{proof}

By the analysis above, we know that the  covariance matrix of $[\{\mathcal{U}(a_1)-\mathrm{E}[\mathcal{U}(a_1)]\}/\sigma(a_1),\allowbreak \ldots, \{\mathcal{U}(a_m)-\mathrm{E}[\mathcal{U}(a_m) ] \}/\sigma(a_m)]^{\intercal}$ asymptotically converges to  $I_m$. To prove Theorem \ref{thm:altcltmeantest}, it remains to show that the joint limiting distribution of the U-statistics is  normal. By the Cram\'{e}r-Wold theorem, it is equivalent to prove that any fixed linear combination of these U-statistics  converges to a normal distribution. By Lemma \ref{lm:twomeanaltvar} and the Slutsky's theorem, it suffices to show that any fixed linear combination of  $[T_{a_1,2}/\sqrt{\mathrm{var}(T_{a_1,2})},\ldots, T_{a_m,2}/\sqrt{\mathrm{var}(T_{a_m,2})}]^{\intercal}$ converges to  a normal distribution for any finite $m$. 
Since $\mu_j=\nu_j$ for $j\in\{k_0+1,\ldots, p \}$, and each $T_{a_t,2}$ is a summation over $j\in\{k_0+1,\ldots, p \}$, we know the analysis under $H_0$ in Section \ref{sec:proofthm33}  can be applied to $T_{a_t,2}$ similarly. 
 Given $k_0=o(p)$, we know $[T_{a_1,2}/\sqrt{\mathrm{var}(T_{a_1,2})},\ldots, T_{a_m,2}/\sqrt{\mathrm{var}(T_{a_m,2})}]^{\intercal}$ has the joint asymptotic normality. In summary, Theorem \ref{thm:altcltmeantest} is proved.

\subsection{Proof of Theorem \ref{thm:twosamnull}} \label{sec:firstpfthmtwoclt} 
We first provide the details of the conditions of  Theorem \ref{thm:twosamnull} in Section \ref{sec:twosamcondnull} and then prove Theorem \ref{thm:twosamnull} in Section \ref{secpftwosamplecov}.   

\subsubsection{Conditions of Theorem  \ref{thm:twosamnull}} \label{sec:twosamcondnull}


 Theorem \ref{thm:twosamnull} can be proved by the following Condition \ref{cond:twosamplecov1} or  Condition  \ref{cond:twosamplecov2}. 
 Note that Conditions  \ref{cond:twosamplecov1} and   \ref{cond:twosamplecov2} are assumed under $H_0$, where  $\boldsymbol{\Sigma}_x=\boldsymbol{\Sigma}_y=\boldsymbol{\Sigma}=(\sigma_{j_1,j_2})_{p\times p}$.

\begin{condition} \label{cond:twosamplecov1} \quad

\begin{enumerate}
	\item[(1)] $n,p\rightarrow \infty$, and $n_x/n \rightarrow \gamma \in (0,1)$.
	\item[(2)] $\lim_{p\rightarrow \infty} \max_{1\leq j\leq p} \mathrm{E}(x_{j}-\mu_{j})^8< \infty$;  $\lim_{p\rightarrow \infty} \min_{1\leq j\leq p} \mathrm{E}(x_{j}-\mu_{j})^2>0$;  $\lim_{p\rightarrow \infty} \max_{1\leq j\leq p} \mathrm{E}(y_{j}-\nu_{j})^8< \infty$; and  $\lim_{p\rightarrow \infty} \min_{1\leq j\leq p} \mathrm{E}(y_{j}-\nu_{j})^2>0$.
	\item[(3)] $\{ (x_{i,j}, i=1,\ldots,n) : 1\leq j\leq p  \}$ and $\{ (y_{i,j},i=1,\ldots,n) : 1\leq j\leq p  \}$ are $\alpha$-mixing with $\alpha_{{{x}}}(s) \leq C\delta_x^s$ and $\alpha_{{{y}}}(s) \leq C\delta_y^s$, where $\delta_x, \delta_y \in (0,1)$ and $C$ is some constant.
   \item[(4)]  For any finite integer $a$, $
	\sum_{ {1\leq j_1,j_2,j_3,j_4 \leq p} }  ( \sigma_{j_1,j_3}\sigma_{j_2,j_4})^{a} =\Theta(p^2).$ 
\end{enumerate}
\end{condition}


Condition \ref{cond:twosamplecov1} (2) is similar to Condition \ref{cond:finitemomt}. 
Condition \ref{cond:twosamplecov1} (3) assumes $\alpha$-mixing on the two samples, which is similar to Condition \ref{cond:alphamixing}. Condition  \ref{cond:twosamplecov1} (4) is a regularity condition on the covariance structure, and it is naturally satisfied for even $a$, given Condition  \ref{cond:twosamplecov1} (3).  

Alternatively, we introduce another set of conditions similar to Condition  \ref{cond:ellpmoment}. 
We define some notation. Suppose $(z_1,\ldots, z_p)^{\intercal}\sim \mathcal{N}(\mathbf{0}, \boldsymbol{\Sigma}).$ Given  indexes $1\leq j_1,\ldots, j_t \leq p$, define $\Pi^0_{j_1,\ldots,j_t}=\mathrm{E}(\prod_{k=1}^t z_{j_k}).$ Moreover, we  define $\Pi^x_{j_1,\ldots, j_t}=\mathrm{E}\{\prod_{k=1}^t (x_{j_k}-\mu_{j_k})\}$ and $\Pi^y_{j_1,\ldots, j_t}=\mathrm{E}\{\prod_{k=1}^t (y_{j_k}-\nu_{j_k})\}$. 
In addition, for given integers $a$ and $b$, let $\mathbb{G}_{a,b}$ be a collection of tuples $
	\mathcal{G}=(g_1,g_2,\ldots, g_{4(a+b)-1}, g_{4(a+b)})\in \{1,\ldots,8\}^{4(a+b)}
$, which satisfies that $g_{2t-1}\neq g_{2t}$ for $t=1,\ldots, 2(a+b)$, and the number of $g$'s equal to $m$ is $a$  for $m\in\{1,2,3,4\}$ and is $b$ for $m\in\{5,6,7,8\}$. For any $\mathcal{G}\in \mathbb{G}_{a,b}$, we define 
$
	\mathbb{V}_{a,b,\mathcal{G}}=\sum_{1\leq j_1,\ldots,j_8\leq p}\prod_{t=1}^{2(a+b)} \sigma_{j_{g_{2t-1}},\,  j_{g_{2t}} },
$  and  let $S_{\mathcal{G}}$ denote the  number of distinct sets among the $2(a+b)$ number of sets, $\{g_{2t-1},g_{2t}\},$ for $t=1,\ldots, 2(a+b),$  induced by $\mathcal{G}$. Note that generally $S_{\mathcal{G}}\geq 4$ and when $S_{\mathcal{G}}=4$, by the symmetricity of $j$ indexes, $\mathbb{V}_{a,b,\mathcal{G}}=\mathbb{V}_{a,b,0}$ where $\mathbb{V}_{a,b,0}:=\sum_{1\leq j_1,\ldots,j_8\leq  p}\sum_{1\leq j_1,\ldots,j_8\leq p} \sigma_{j_1,j_2}^a\sigma_{j_3,j_4}^a \sigma_{j_5,j_6}^b\sigma_{j_7,j_8}^{b}.$ 


\begin{condition} \label{cond:twosamplecov2} \quad

\begin{enumerate}
	\item[(1)] $n,p\rightarrow \infty$, and $n_x/n \rightarrow \gamma \in (0,1)$.
	\item[(2)] $\lim_{p\rightarrow \infty} \max_{1\leq j\leq p} \mathrm{E}(x_{j}-\mu_{j})^8< \infty$;  $\lim_{p\rightarrow \infty} \min_{1\leq j\leq p} \mathrm{E}(x_{j}-\mu_{j})^2>0$;  $\lim_{p\rightarrow \infty} \max_{1\leq j\leq p} \mathrm{E}(y_{j}-\nu_{j})^8< \infty$; and  $\lim_{p\rightarrow \infty} \min_{1\leq j\leq p} \mathrm{E}(y_{j}-\nu_{j})^2>0$.
	\item[(3)] For $t=3, 4, 6, 8$, there exist constants  $\kappa_{x,t},\kappa_{y,t}\geq 1$ such that  $\Pi^x_{j_1,\ldots, j_t}=\kappa_{x,t}\Pi^0_{j_1,\ldots,j_t}$ and $\Pi^y_{j_1,\ldots, j_t}=\kappa_{y,t}\Pi^0_{j_1,\ldots,j_t}$. 
\item[(4)] For $a, b\in \{a_1,\ldots ,a_m\}$, and any $\mathcal{G}\in \mathbb{G}_{a,b}$ define above,  if $S_{\mathcal{G}}>4$, we assume  $\mathbb{V}_{a,b,\mathcal{G}}=o(1)\mathbb{V}_{a,b,0}$.
\end{enumerate}
\end{condition}

We note that Condition \ref{cond:twosamplecov2} (3) and (4) are alternative dependence assumptions to Condition \ref{cond:twosamplecov1} (3) and (4).  Condition \ref{cond:twosamplecov2} (3) is an extension from Condition \ref{cond:ellpmoment}, and is also satisfied when the distributions of $\mathbf{x}$ and $\mathbf{y}$ follow elliptical distributions \cite{kan2008moments}. 
 Condition \ref{cond:twosamplecov2} (4) implies some weak dependence structure in covariance matrix $\boldsymbol{\Sigma}$. 
To better illustrate the condition, we consider the case when $a=b=2$ as an example. We note that 
 \begin{align*}
	\mathbb{V}_{a,b,0}=\sum_{1\leq j_1,\ldots, j_8\leq p} (\sigma_{j_1,j_2}\sigma_{j_3,j_4}\sigma_{j_5,j_6}\sigma_{j_7,j_8})^2= \{\mathrm{tr}(\Sigma^2)\}^4,
\end{align*} and $\mathbb{V}_{a,b,0}=\mathbb{V}_{a,b,\mathcal{G}}$ when $\mathcal{G}=(1,2,3,4,5,6,7,8,1,2,3,4,5,6,7,8)$ with $S_{\mathcal{G}}=4$. 
Moreover, if $\mathcal{G}=( 1,3,2,4,1,2,3,4,5,6,7,8,5,6,7,8)$ with $S_{\mathcal{G}}=6$, 
\begin{align*}
\mathbb{V}_{a,b,\mathcal{G}}=\sum_{1\leq j_1,\ldots, j_8\leq p} (\sigma_{j_1,j_3}\sigma_{j_2,j_4})(\sigma_{j_1,j_2}\sigma_{j_3,j_4})(\sigma_{j_5,j_6}\sigma_{j_7,j_8})^2=\mathrm{tr}(\Sigma^4)\{\mathrm{tr}(\Sigma^2)\}^2;	
\end{align*} if $\mathcal{G}=( 1,3,2,4, 1,2,3,4,5,7,6,8, 5,6,7,8)$ with $S_{\mathcal{G}}=8$,
 \begin{align*}
\mathbb{V}_{a,b,\mathcal{G}}=\sum_{1\leq j_1,\ldots, j_8\leq p} (\sigma_{j_1,j_3}\sigma_{j_2,j_4})(\sigma_{j_1,j_2}\sigma_{j_3,j_4})(\sigma_{j_5,j_7}\sigma_{j_6,j_8})(\sigma_{j_5,j_6}\sigma_{j_7,j_8})=\{\mathrm{tr}(\Sigma^4)\}^2;	
\end{align*} 
if $\mathcal{G}=(1,6,2,5,3,7,4,8,1,3,2,4,5,7,6,8)$ with $S_{\mathcal{G}}=8$, 
\begin{align*}
\mathbb{V}_{a,b,\mathcal{G}}=	 \sum_{1\leq j_1,\ldots, j_8\leq p}(\sigma_{j_1,j_6}\sigma_{j_2,j_5}) (\sigma_{j_3,j_7}\sigma_{j_4,j_8}) (\sigma_{j_1,j_3}\sigma_{j_2,j_4}) (\sigma_{j_5,j_7}\sigma_{j_6,j_8})=\mathrm{tr}(\Sigma^8).
\end{align*}
In this case, Condition \ref{cond:twosamplecov2} (4) is equivalent to  $\mathrm{tr}(\Sigma^4)=o[\{\mathrm{tr}(\Sigma^2)\}^2]$ and $\mathrm{tr}(\Sigma^8)=o[\{\mathrm{tr}(\Sigma^2)\}^4]$, which are  similarly assumed in \cite{lijun2012}.
 In addition, we consider another example where the $p\times p$ covariance matrix $\boldsymbol{\Sigma}$ is of banded structure with bandwidth $s$ and has the nonzero entries being positive constants. It follows that $\mathbb{V}_{a,b,0}=\Theta(p^4s^4)$ and $\mathbb{V}_{a,b,\mathcal{G}}=O(p^3s^5)$ when $S_{\mathcal{G}}>4$. Therefore, in this example, Condition \ref{cond:twosamplecov2} (4) is satisfied when $s=o(p)$. 

\subsubsection{Proof of Theorem  \ref{thm:twosamnull}} \label{secpftwosamplecov}
Since $\mathcal{U}(a)$ is location invariant, we assume $\mathrm{E}(\mathbf{x})=\mathbf{0}$ and $\mathrm{E}(\mathbf{y})=\mathbf{0}$, without loss of generality, in this section. 
We decompose  $\mathcal{U}(a)=\tilde{\mathcal{U}}(a)+\tilde{\mathcal{U}}^*(a)$, where we redefine
\begin{align*}
	\tilde{\mathcal{U}}(a)=\sum_{1\leq j_1, j_2 \leq p} \frac{1}{P^{n_x}_a P^{n_y}_a}\sum_{\substack{ 1\leq i_{1}\neq \ldots \neq i_{a} \leq n_x; \\  1\leq w_{1}\neq \ldots \neq w_{a} \leq n_y }}   \prod_{t=1}^a  (x_{i_{t},j_1}x_{i_{t},j_2} -y_{w_{t},j_1}y_{w_{t},j_2}),
\end{align*} and $\tilde{\mathcal{U}}^*(a)=\mathcal{U}(a)-\tilde{\mathcal{U}}(a).$ 
To prove Theorem \ref{thm:twosamnull}, we derive the variances,  covariances, and asymptotic joint normality of the U-statistics. Particularly, 
the following Lemma \ref{lm:twosampvartest} derives the asymptotic form of $\mathrm{var}\{\mathcal{U}(a)\}$, and shows that $\tilde{\mathcal{U}}(a)$ is the leading term. 
\begin{lemma}\label{lm:twosampvartest}
Under the conditions of Theorem \ref{thm:twosamnull}, 
$\mathrm{var}\{ \tilde{\mathcal{U}}^*(a) \}=o(1)\mathrm{var}\{\tilde{\mathcal{U}}(a)\}$, $\mathcal{U}^*(a)/\sigma(a)\xrightarrow{P}0,$ and  	
\begin{align*}
&\mathrm{var}\{\mathcal{U}(a)\}\notag \\
\simeq &\sum_{1\leq j_1,j_2,j_3,j_4\leq p}a! \Big\{ \frac{1}{n_x}(\Pi_{j_1,j_2,j_3,j_4}^x-\sigma_{j_1,j_2}\sigma_{j_3,j_4}) +\frac{1}{n_y}(\Pi_{j_1,j_2,j_3,j_4}^y-\sigma_{j_1,j_2}\sigma_{j_3,j_4})\Big\}^a.	
\end{align*}
In particular, under Condition \ref{cond:twosamplecov1}, $\mathrm{var}\{{\mathcal{U}}(a)\}=\Theta(p^2n^{-a});$ 
 under Condition \ref{cond:twosamplecov2}, $\mathrm{var}\{{\mathcal{U}}(a)\}=\Theta(n^{-a})\sum_{1\leq j_1,j_2,j_3,j_4\leq p}( \sigma_{j_1,j_3}\sigma_{j_2,j_4})^a.$  
\end{lemma}
\begin{proof}
\textit{See Section \ref{sec:pftwosampvartest} on Page \pageref{sec:pftwosampvartest}.}
\end{proof}

\noindent Given Lemma \ref{lm:twosampvartest}, the next Lemma \ref{lm:twosampcovtest} shows that the covariance between two U-statistics asymptotically converges to 0.
\begin{lemma} \label{lm:twosampcovtest}
Under the conditions of Theorem \ref{thm:twosamnull},  for finite integers $a\neq b$, $\mathrm{cov}\{\mathcal{U}(a)/\sigma(a),\mathcal{U}(b)/\sigma(b)\}\to 0$ as $n,p\to \infty$. 
\end{lemma}
\begin{proof}
\textit{See Section  \ref{lm:pftwosampcovtest} on Page \pageref{lm:pftwosampcovtest}.}
\end{proof}

To finish the proof, it remains to obtain the joint asymptotic normality  of $[\mathcal{U}(a_1)/\sigma(a_1),\allowbreak \ldots, \mathcal{U}(a_m)/\sigma(a_m)]^{\intercal}$ for different finite integers $a_1,\ldots, a_m$.
By Cram\'{e}r-Wold theorem, it is equivalent to prove that any of their fixed linear combination converges to normal. In addition, by Lemma \ref{lm:twosampvartest} and the Slutsky's theorem,  it suffices to prove that any fixed linear combination of $[\tilde{\mathcal{U}}(a_1)/\sigma(a_1),\allowbreak \ldots, \tilde{\mathcal{U}}(a_m)/\sigma(a_m)]^{\intercal}$ converges to normal. 
Specifically, similarly to  Section \ref{sec:detailofjointnormal}, we redefine $Z_n$  as below with $\sum_{r=1}^m t_r^2=1$, and prove that
\begin{eqnarray}
Z_n:=\sum_{r=1}^m t_r \tilde{\mathcal{U}}(a_r)/\sigma(a_r)\xrightarrow{D} \mathcal{N}(0,1).  \label{eq:znnormaltwosamcov}
\end{eqnarray} 
We next prove \eqref{eq:znnormaltwosamcov} following the proof of  Theorem \ref{thm:jointnormal} in Section \ref{sec:detailofjointnormal} and apply the martingale central limit theorem \cite[p.476]{billingsley1995}.
 To construct a martingale difference, we write $\mathbf{x}_{i}=(x_{i,1},\ldots, x_{i,p})^{\intercal}$ and $\mathbf{y}_i=(y_{i,1},\ldots, y_{i,p})^{\intercal}$; and define a new random vector
\begin{eqnarray*}
 R_{i}=  \mathbf{x}_{i} \quad \mbox{for}\ i=1,2,\ldots, n_x; && R_{n_x+j}=\mathbf{y}_j \quad \mbox{for}\ j=1,2,\ldots, n_y.
\end{eqnarray*}
 We then define $\mathcal{F}_0=\{\emptyset, \Omega \}$ and  $\mathcal{F}_k=\sigma \{R_1,\ldots, R_k \}$ for  $k=1,2,\ldots, n_x+n_y$; and let $\mathrm{E}_k(\cdot)$ denote the conditional expectation given $\mathcal{F}_k$ for $k=1, \cdots, n_x+n_y$.  Define  $D_{n,k}=(\mathrm{E}_k-\mathrm{E}_{k-1})Z_n$ and $\pi_{n,k}^2=\mathrm{E}_{k-1}(D_{n,k}^2)$.	 It follows that $Z_n=\sum_{k=1}^n D_{n,k}$ as $\mathrm{E}_0(Z_n)=\mathrm{E}(Z_n)=0.$ To prove \eqref{eq:znnormaltwosamcov}, by the martingale central limit theorem, it suffices to prove 
\begin{eqnarray}
\quad &&\sum_{k=1}^n \pi_{n,k}^2/ \mathrm{var}(Z_n) \xrightarrow{P} 1 \quad \text{and}\quad \sum_{k=1}^n \mathrm{E}(D_{n,k}^4)/ \mathrm{var}^2(Z_n)  \to 0. \label{eq:twosamcltg}
\end{eqnarray}
To prove \eqref{eq:twosamcltg}, we derive the explicit forms of $D_{n,k}$ and $\pi_{n,k}^2$ in Section \ref{lm:pftwosampiform}. 
Similarly to Section \ref{sec:detailofjointnormal}, the following Lemma   \ref{lm:twosamvario1} and Lemma \ref{lm:twosamvario2} suggest that \eqref{eq:twosamcltg} holds.
\begin{lemma} \label{lm:twosamvario1}
Under the conditions of Theorem  \ref{thm:twosamnull}, $\mathrm{var}(\sum_{k=1}^{n_x+n_y}\pi_{n,k}^2)\to 0$.
\end{lemma}
\begin{proof}
\textit{See Section \ref{sec:pftwosamvario1} on Page \pageref{sec:pftwosamvario1}.}
\end{proof}
\begin{lemma} \label{lm:twosamvario2}
Under the conditions of Theorem  \ref{thm:twosamnull}, $\sum_{k=1}^{n_x+n_y}\mathrm{E}(D_{n,k}^4)\to 0.$ 
\end{lemma}
\begin{proof}
\textit{See Section \ref{sec:pftwosamvario2} on Page \pageref{sec:pftwosamvario2}.}
\end{proof}
In summary, Theorem  \ref{thm:twosamnull} is proved.

\subsection{Proof of Theorem \ref{thm:twosamaltclt}} \label{sec:pfthm49alttwo}

In this section, we first provide the conditions of Theorem \ref{thm:twosamaltclt} in Section \ref{sec:twosampowerconds}
 and then prove Theorem \ref{thm:twosamaltclt} in Section \ref{sec:proofalttwosam}.
\subsubsection{Conditions} \label{sec:twosampowerconds}

Theorem \ref{thm:twosamaltclt} is established under the following Conditions \ref{cond:twosamaltdist} and \ref{cond:twosamaltmoment}, where Condition \ref{cond:twosamaltdist}  is the same as Condition \ref{cond:twosamplecov2} (1)--(3). 
\begin{condition}\label{cond:twosamaltdist} \quad

\begin{enumerate}
	\item[(1)] $n,p\rightarrow \infty$, and $n_x/n \rightarrow \gamma \in (0,1)$.
	\item[(2)] $\lim_{p\rightarrow \infty} \max_{1\leq j\leq p} \mathrm{E}(x_{j}-\mu_{j})^8< \infty$;  $\lim_{p\rightarrow \infty} \min_{1\leq j\leq p} \mathrm{E}(x_{j}-\mu_{j})^2>0$;  $\lim_{p\rightarrow \infty} \max_{1\leq j\leq p} \mathrm{E}(y_{j}-\nu_{j})^8< \infty$; and  $\lim_{p\rightarrow \infty} \min_{1\leq j\leq p} \mathrm{E}(y_{j}-\nu_{j})^2>0$.
	\item[(3)] For $t=3, 4, 6, 8$, there exist  $\kappa_{x,t},\kappa_{y,t}\geq 1$ such that  $\Pi^x_{j_1,\ldots, j_t}=\kappa_{x,t}\Pi^0_{j_1,\ldots,j_t}$ and $\Pi^y_{j_1,\ldots, j_t}=\kappa_{y,t}\Pi^0_{j_1,\ldots,j_t}$. 
\end{enumerate}	
\end{condition}

To provide Condition \ref{cond:twosamaltmoment}, we first define some notation.  
The difference between $\boldsymbol{\Sigma}_x$ and $\boldsymbol{\Sigma}_y$ is defined as $D_{x,y}=\boldsymbol{\Sigma}_x-\boldsymbol{\Sigma}_y=(D_{j_1,j_2})_{p\times p}$. 
Let $\mathbb{J}_0\subseteq \{1,\ldots, p\}$ be the largest set such that for any $j_1,j_2\in \mathbb{J}_0$, $\sigma_{x,j_1,j_2}=\sigma_{y,j_1,j_2}$. 
%
Define $J_{0,D}=\{(j_1,j_2): j_1\ \mathrm{or}\ j_2 \not\in \mathbb{J}_0  \}$. 
Given $\mathbb{J}_0$ and $a,b\in \{a_1,\ldots,a_m\}$, we define
$
	\mathbb{V}_{a,b,0,0}=\sum_{j_1,\ldots,j_8 \in \mathbb{J}_0}(\sigma_{x,j_1,j_2}\sigma_{x,j_3,j_4})^a(\sigma_{x,j_5,j_6}\sigma_{x,j_7,j_8})^b,
$ which also equals to $\sum_{j_1,\ldots,j_8 \in \mathbb{J}_0}(\sigma_{y,j_1,j_2}\sigma_{y,j_3,j_4})^a(\sigma_{y,j_5,j_6}\sigma_{y,j_7,j_8})^b$ by the definition of $\mathbb{J}_0$. In addition, for any tuple $\mathcal{G}=(g_1,g_2,\ldots,g_{4(a+b)-1},g_{4(a+b)})\in \mathbb{G}_{a,b}$ specified in Condition \ref{cond:twosamplecov2}, we define
$
	\mathbb{V}_{a,b,\mathcal{G},0}=\sum_{j_1,\ldots,j_8 \in \mathbb{J}_0}\prod_{t=1}^{2(a+b)} \sigma_{j_{g_{2t-1}},\,  j_{g_{2t}} }. 
$ Note that $\mathbb{V}_{a,b,0,0}$ and  $\mathbb{V}_{a,b,\mathcal{G},0}$ are defined similarly to $\mathbb{V}_{a,b,0}$ and  $\mathbb{V}_{a,b,\mathcal{G}}$ in Condition \ref{cond:twosamplecov2} by changing the range of $j$ indexes from $\{1,\ldots,p\}$ to $\mathbb{J}_0$. Moreover, let $\mathcal{H}=\{(h_1,h_2),(h_3,h_4)\}\in \mathbb{H}$, 
where $\mathbb{H}$ includes $\{(1,2),(3,4)\}$, $ \{(1,3), (2,4)\}$ and $\{(1,4), (2,3)\}$.
 For any $a\in \{a_1, \ldots, a_m\}$ and given $\mathcal{H}\in \mathbb{H}$, define
\begin{align}
&\mathbb{V}_{a, \mathcal{H} , x , 1}=\sum_{ \substack{ (j_1,j_2),(j_3,j_4) \in J_{0,D}} } 	|\sigma_{x,j_{h_{1}},j_{h_{2}}}\sigma_{x,j_{h_{3}},j_{h_{4}}}|^a\label{eq:condtwosamdef} \\
 & \mathbb{V}_{a,\mathcal{H},x,2}=\sum_{ \substack{ (j_1,j_2),(j_3,j_4) \in J_{0,D}} } 	|D_{j_{h_{1}},j_{h_{2}}} \sigma_{x,j_{h_3},j_{h_4}} |^a,	\notag \\
  & \mathbb{V}_{a,\mathcal{H},D,3}=\sum_{ \substack{ (j_1,j_2),(j_3,j_4) \in J_{0,D}} } 	|D_{j_{h_1},j_{h_2}}D_{j_{h_3},j_{h_4}}|^a.\notag
\end{align}
Similarly, we also define $\mathbb{V}_{a, \mathcal{H} , y , 1}$ and $\mathbb{V}_{a,\mathcal{H},y,2}$ by replacing $\sigma_{x}$'s with $\sigma_{y}$'s. We next present Condition \ref{cond:twosamaltmoment} of Theorem  \ref{thm:twosamaltclt}.
\begin{condition}\label{cond:twosamaltmoment}
For any $a,b\in \{a_1,\ldots,a_m\}$,    $\mathcal{G}\in \mathbb{G}_{a,b}$, and $\mathcal{H}\in \mathbb{H}$, we assume (A1) $\mathbb{V}_{a,b,\mathcal{G},0}=o(1)\mathbb{V}_{a,b,0,0}$;  (A2) $\mathbb{V}_{a,\mathcal{H},D,3}=O(n^{-a})\mathbb{V}_{a,a,0,0}^{1/2}$;  
and (A3) $ \mathbb{V}_{a,\mathcal{H},x,t}=o(1)\mathbb{V}_{a,a,0,0}^{1/2}$, for  $t=1,2$. 

\end{condition}

Equivalently we can also  replace (A3) in Condition \ref{cond:twosamaltmoment} by  (A3)$^*$ $ \mathbb{V}_{a,\mathcal{H},y,t}=o(1)\mathbb{V}_{a,a,0,0}^{1/2}$, for $t=1,2$.
This is because by $D_{j_1,j_2}=\sigma_{x,j_1,j_2}-\sigma_{y,j_1,j_2}$ and  the  H\"older's inequality, we know  \textit{(A2)} and \textit{(A3)} induce \textit{(A3)$^*$}; and \textit{(A2)} and \textit{(A3)$^*$} also induce \textit{(A3)}. 
Thus it is equivalent to assume \textit{(A3)} or \textit{(A3)$^*$} in Condition \ref{cond:twosamaltmoment}.


We next discuss Condition \ref{cond:twosamaltmoment}. Let $\boldsymbol{\Sigma}_C=\{\sigma_{x,j_1,j_2}:j_1,j_2\in \mathbb{J}_0\}=\{\sigma_{y,j_1,j_2}:j_1,j_2\in \mathbb{J}_0\}$, which is the common submatrix of $\boldsymbol{\Sigma}_x$ and $\boldsymbol{\Sigma}_y$ by the definition of $\mathbb{J}_0$. In Condition \ref{cond:twosamaltmoment}, \textit{(A1)} implies some weak dependence structure of $\boldsymbol{\Sigma}_C$ similar to Condition \ref{cond:twosamplecov2} (4). We consider an example where $\boldsymbol{\Sigma}_x$ has the banded structure with the bandwidth $s$ and the entries being positive constants.   Then \textit{(A1)} holds if   $s=o(p)$. Moreover, 
under the considered example, 
$\mathbb{V}_{a,a,0,0}^{1/2}=(\sum_{j_1,j_2\in \mathbb{J}_0}\sigma_{x,j_1,j_2}^a)^2\geq C|\mathbb{J}_0|^4$ and $\mathbb{V}_{a,\mathcal{H},x,1} \leq C|J_{0,D}|^2=C_2(p-|\mathbb{J}_0|)^4$. Then \textit{(A3)} for $t=1$  
holds when $p-|\mathbb{J}_0|=o(p)$, 
which implies that the number of entries that are different in $\boldsymbol{\Sigma}_x$ and $\boldsymbol{\Sigma}_y$ is $o(p^2)$. In addition, \textit{(A2)} and \textit{(A3)} for $t=2$ are  regularity conditions on the difference matrix $D_{x,y}$. 
For illustration, we consider an example where $D_{j_1,j_2}=\rho>0$ for any $(j_1,j_2)\in J_{0,D}$, and $\boldsymbol{\Sigma}_x=I_p$. Then $\mathbb{V}_{a,a,0,0}^{1/2}=|\mathbb{J}_0|^2$,   $\mathbb{V}_{a,\mathcal{H},x,2}\leq |J_{0,D}|\rho^ap$, and $ \mathbb{V}_{a,\mathcal{H},D,3}\leq |J_{0,D}|^2\rho^{2a}$. Under this example, \textit{(A2)} and \textit{(A3)} of $t=2$  hold if $|J_{0,D}|\rho^a=O(n^{-a/2}p)$ and $|\mathbb{J}_0|\simeq p$, which are similar to the assumption  in  Theorem \ref{thm:cltalternative}.

\subsubsection{Proof}\label{sec:proofalttwosam}
In this section, we prove Theorem \ref{thm:twosamaltclt} under  Conditions \ref{cond:twosamaltdist} and \ref{cond:twosamaltmoment}.      
Recall that we decompose $\mathcal{U}(a)=\tilde{\mathcal{U}}(a)+\tilde{\mathcal{U}}^*(a)$ in Section \ref{sec:firstpfthmtwoclt}. 
We further decompose $\tilde{\mathcal{U}}(a)=T_{D,a,1}+T_{D,a,2}$, where 
\begin{align*}
	T_{D,a,1}=&\sum_{j_1,j_2 \in \mathbb{J}_0} \frac{1}{P^{n_x}_a P^{n_y}_a}\sum_{\substack{ 1\leq i_{1}\neq \ldots \neq i_{a} \leq n_x; \\ 1\leq w_{1}\neq \ldots \neq w_{a} \leq n_y }} \   \prod_{t=1}^a  (x_{i_{t},j_1}x_{i_{t},j_2} -y_{w_{t},j_1}y_{w_{t},j_2}), \notag \\
	T_{D,a,2}=&\sum_{(j_1,j_2)\in J_{0,D}} \frac{1}{P^{n_x}_a P^{n_y}_a}\sum_{\substack{ 1\leq i_{1}\neq \ldots \neq i_{a} \leq n_x; \\ 1\leq w_{1}\neq \ldots \neq w_{a} \leq n_y }} \    \prod_{t=1}^a  (x_{i_{t},j_1}x_{i_{t},j_2} -y_{w_{t},j_1}y_{w_{t},j_2}). \notag
\end{align*} 
It follows that $\mathcal{U}(a)=T_{D,a,1}+T_{D,a,2}+\tilde{\mathcal{U}}^*(a).$ 
To prove Theorem \ref{thm:twosamaltclt}, we derive the variances, covariances and asymptotic joint normality of the U-statistics. In particular, 
the next Lemma \ref{lm:twosamaltsmall1} derives the asymptotic form of $\mathrm{var}\{\mathcal{U}(a)\}$, and shows that $T_{D,a,1}$ is the leading component. 
\begin{lemma}\label{lm:twosamaltsmall1}
Under the conditions of Theorem \ref{thm:twosamaltclt}, 
$$\mathrm{var}\{\mathcal{U}(a)\}\simeq \sum_{1\leq j_1,j_2,j_3,j_4\in \mathbb{J}_0} a! C_{\kappa,a}\sigma_{j_1,j_2}^a\sigma_{j_3,j_4}^a,$$ where $C_{\kappa,a}=\{(\kappa_x-1)/n_x+(\kappa_y-1)/n_y\}^a+ 2(\kappa_x/n_x+\kappa_y/n_y)^a$. In addition, 	
$\mathrm{var}(T_{D,a,2})=o(1)\mathrm{var}(T_{D,a,1})$ and $\mathrm{var}\{ \tilde{\mathcal{U}}^*(a)\}=o(1)\mathrm{var}\{ \tilde{\mathcal{U}}(a)\}$. It follows that  	$\{T_{D,a,2}-\mathrm{E}(T_{D,a,2})\}/\sigma(a)\xrightarrow{P}0$ and  $[\tilde{\mathcal{U}}^*(a)- \mathrm{E}\{\tilde{\mathcal{U}}^*(a)\}]/\sigma(a)\xrightarrow{P }0$.  
\end{lemma}

\begin{proof}
	\textit{See Section \ref{sec:pftwocovaltlm} on Page \pageref{sec:pftwocovaltlm}.}
\end{proof}

Lemma \ref{lm:twosamaltsmall1} gives that $\{T_{D,a,2}-\mathrm{E}(T_{D,a,2})\}/\sigma(a)\xrightarrow{P}0$ and  $[\tilde{\mathcal{U}}^*(a)- \mathrm{E}\{\tilde{\mathcal{U}}^*(a)\}]/\sigma(a)\xrightarrow{P}0$. Thus by the Slutsky's theorem, to prove Theorem \ref{thm:twosamaltclt}, it   suffices to prove 
\begin{align}
\Big[ \frac{T_{D,a_1,1}}{\sqrt{\mathrm{var}(T_{D,a_1,1})}},\ldots, \frac{T_{D,a_m,1}}{\sqrt{\mathrm{var}(T_{D,a_m,1})}}\Big]\xrightarrow{D} \mathcal{N}(\mathbf{0}, I_m). \label{eq:ta2twosamdist}
\end{align}
Note that $T_{D,a,1}$ is a summation over $j$ indexes in $\mathbb{J}_0$, and by the definition of $\mathbb{J}_0$, $\sigma_{x,j_1,j_2}=\sigma_{y,j_1,j_2}$ for any $j_1,j_2 \in \mathbb{J}_0$. Therefore the analysis under $H_0$ can be similarly applied to $T_{D,a,1}$. Given    Condition  \ref{cond:twosamaltdist} and Condition \ref{cond:twosamaltmoment} \textit{(A1)}, we can obtain  \eqref{eq:ta2twosamdist} similarly as in Section  \ref{secpftwosamplecov}. 
In summary, Theorem \ref{thm:twosamaltclt} is proved.

\subsection{Proof of Proposition \ref{prop:ordercomptwosam}}
In this section, we prove Proposition \ref{prop:ordercomptwosam}. Under the considered example, as $p-|\mathbb{J}_0|=o(p)$, we have 
$
\sum_{j_1,j_2,j_3,j_4\in \mathbb{J}_0} \sigma_{x,j_1,j_2}^a\sigma_{x,j_3,j_4}^a 
\simeq \{p\nu^{2a} + 2\sum_{t=1}^s h_{t}^a (p-t)\}^2. 
$ 
Then by Lemma \ref{lm:twosamaltsmall1},  when $n_x=n_y=n/2$, 
\begin{eqnarray}
&& \mathrm{var}\{\mathcal{U}(a)\}
\simeq	(n/2)^{-a}a! (2\kappa_1^a+\kappa_2^a)\Big\{ p\nu^{2a} + 2\sum_{t=1}^s h_{t}^a (p-t)\Big\}^2, 	\label{eq:varuatwocovpower}
\end{eqnarray}where $\kappa_1=\kappa_x+\kappa_y$ and $\kappa_2=\kappa_x+\kappa_y-2$.


Recall that $\rho_a$ is defined to be the value such that when $\rho=\rho_a$ under the  alternative, $\mathrm{E}\{\mathcal{U}(a) \}/\sqrt{\mathrm{var}\{ \mathcal{U}(a) \}}\simeq M$ for given $M$. By \eqref{eq:varuatwocovpower}, $\rho_a$ satisfies
\begin{align*}
|J_D|^2 \rho_a^{2a} = M^2 (n/2)^{-a}a! (2\kappa_1^a+\kappa_2^a)	\Big\{ p\nu^{2a} + 2 \sum_{t=1}^s h_t^a (p-t) \Big\}^2. 
\end{align*}
We next obtain
\begin{align*}
\rho_a =  \frac{(a!)^{\frac{1}{2a}}\sqrt{\kappa_1}\nu}{(n/2)^{1/2}} \Big(\frac{Mp}{|J_D|}\Big)^{1/a} \Big\{2+ \Big(\frac{\kappa_2}{\kappa_1}\Big)^a\Big\}^{\frac{1}{2a}} \Big\{ 1+ { 2\sum_{t=1}^s \Big( \frac{h_t}{\nu^2}\Big)^a \Big(1-\frac{t}{p}\Big) } \Big\}^{\frac{1}{a}}. 
\end{align*} 
Let $\tilde{M}=Mp/|J_D|$, $\tilde{h}_t=h_t/\nu^2$, $\tilde{\nu}=\sqrt{\kappa_1}\nu$, and $\tilde{\kappa}_r=\kappa_2/\kappa_1$. It follows that
\begin{align*}
\rho_a = \tilde{\nu}(a!)^{\frac{1}{2a}} (n/2)^{-1/2}(\tilde{M})^{\frac{1}{a}} (2+ \tilde{\kappa}_r^a)^{\frac{1}{2a}}\Big\{ 1+2\sum_{t=1}^s \tilde{h}_t^a \Big(1-\frac{t}{p}\Big) \Big\}^{\frac{1}{a}}.
\end{align*}

Similarly to Section \ref{proof:prop23}, we study $\rho_a$ as a function of integer $a$ and show that if $\rho_a$ starts to not decrease at some value, it will increase afterwards. Specifically, we show that when $\rho_{a+1}/ \rho_a\geq 1$, $\rho_{a+2}/\rho_{a+1}>1$. Note that
\begin{align*}
\frac{\rho_{a+1}}{\rho_a} =&~\Biggr[\frac{  (a+1)! \tilde{M}^2 (2+\tilde{\kappa}_r^{a+1})\Big\{1+2\sum_{t=1}^s \tilde{h}_t^{a+1} \Big(1-\frac{t}{p}\Big)  \Big\}^2  }{   (a!)^{ 1+ \frac{1}{a} } \tilde{M}^{ 2+\frac{2}{a} }  (2+\tilde{\kappa}_r^a)^{ 1+\frac{1}{a} } \Big\{ 1+ 2\sum_{t=1}^s \tilde{h}_t^a\Big(1-\frac{t}{p}\Big)\Big\}^{2(1+\frac{1}{a})}   }	\Biggr]^{\frac{1}{2(a+1)}} \notag \\
=&~ \{\mathbb{D}(a)\tilde{M}^{-2} \}^{\frac{1}{2a(a+1)}},
\end{align*} where $\mathbb{D}(a)= \mathbb{D}_1(a)\times \mathbb{D}_2(a)\times \mathbb{D}_3(a)$ with 
 $\mathbb{D}_1(a)=(a+1)^a/a!$, $\mathbb{D}_2(a)=(2+\tilde{\kappa}_r^{a+1})^a/ (2+\tilde{\kappa}_r^a)^{ {a+1}}$ and
\begin{align*}
\mathbb{D}_3(a)={  \Big\{1+2\sum_{t=1}^s \tilde{h}_t^{a+1} \Big(1-\frac{t}{p}\Big)  \Big\}^{2a}  }\Big/{ \Big\{ 1+ 2\sum_{t=1}^s \tilde{h}_t^a\Big(1-\frac{t}{p}\Big)\Big\}^{2(a+1)}  }.	
\end{align*} 
It follows that $\rho_{a+1}/\rho_a> 1$ and $\rho_{a+1}/\rho_a= 1$ are equivalent to $\mathbb{D}(a)> \tilde{M}^2$ and $\mathbb{D}(a)=\tilde{M}^2$, respectively. 
 

We next show that $\mathbb{D}(a)$ is a strictly increasing functions of $a$ as 
 $\mathbb{D}_1(a+1)/\mathbb{D}_1(a)>1$, $\mathbb{D}_2(a+1)/\mathbb{D}_2(a)\geq 1$ and $\mathbb{D}_3(a+1)/\mathbb{D}_3(a)\geq 1$. Particularly,
\begin{align*}
	\frac{\mathbb{D}_1(a+1)}{\mathbb{D}_1(a)}=\frac{(a+2)^{a+1}}{(a+1)!}\frac{a!}{(a+1)^a}=\Big(1+\frac{1}{a+1}\Big)^{a+1} >1;
\end{align*}
\begin{align*}
	\frac{\mathbb{D}_2(a+1)}{\mathbb{D}_2(a)}=\frac{(2+\tilde{\kappa}_r^{a+2})^{a+1}}{ (2+\tilde{\kappa}_r^{a+1})^{ {a+2}}}\times \frac{(2+\tilde{\kappa}_r^a)^{ {a+1}} }{ (2+\tilde{\kappa}_r^{a+1})^a}=\Big\{\frac{ (2+\tilde{\kappa}_r^{a+2})(2+\tilde{\kappa}_r^{a}) }{ (2+\tilde{\kappa}_r^{a+1})^{2} }\Big\}^{a+1} \geq 1,
\end{align*} where we use $2\tilde{\kappa}_r^{a+1}\leq \tilde{\kappa}_r^{a+2}+\tilde{\kappa}_r^{a}$ by the inequality of arithmetic and geometric means; and 
\begin{align*}
	\frac{\mathbb{D}_3(a+1)}{\mathbb{D}_3(a)}=&~\Biggr[ \frac{  \Big\{1+2\sum_{t=1}^s \tilde{h}_t^{a+2} (1-\frac{t}{p})  \Big\}  \Big\{1+2\sum_{t=1}^s \tilde{h}_t^{a} (1-\frac{t}{p})  \Big\} }{  \Big\{1+2\sum_{t=1}^s \tilde{h}_t^{a+1} (1-\frac{t}{p})  \Big\}^2  }\Biggr]^{4(a+1)}\geq 1,
\end{align*} where we use $\sum_{t=1}^s \tilde{h}_t^{a+2} (1-{t}/{p})+ \sum_{t=1}^s \tilde{h}_t^{a} (1-{t}/{p}) \geq  2\sum_{t=1}^s \tilde{h}_t^{a+1} (1-{t}/{p})$ by the inequality of arithmetic and geometric means and $\{\sum_{t=1}^s \tilde{h}_t^{a+2} (1-{t}/{p})\}\{ \sum_{t=1}^s \tilde{h}_t^{a} (1-{t}/{p})\} \geq  \{\sum_{t=1}^s \tilde{h}_t^{a+1} (1-{t}/{p})\}^2$ by the H\"older's inequality. 
In summary, $\mathbb{D}(a+1)/\mathbb{D}(a)>1$, and thus $\mathbb{D}(a)$ is a strictly increasing function of $a$.

Given the monotonicity of $\mathbb{D}(a)$, we know that if $\mathbb{D}(a)\geq \tilde{M}^2$, $\mathbb{D}(a+1)> \tilde{M}^2$; equivalently this implies that if $\rho_{a+1}\geq \rho_a$, $\rho_{a+2}> \rho_{a+1}$. Suppose $a_0$ is the first integer such that $\mathbb{D}(a_0)\geq \tilde{M}^2$, i.e.,  for any integer $ 1\leq a<a_0$,  $\mathbb{D}(a)< \tilde{M}^2$. By the analysis above, we know  $\rho_a$ is decreasing  when $a<a_0$, and $\rho_a$ is strictly increasing when $a>a_0$. Thus $a_0$ achieves the minimum of $\rho_a$. Since  $\mathbb{D}(a)$ is strictly increasing in $a$, we know $a_0<\infty$ given $ \tilde{M}$, and $a_0$ increases as $ \tilde{M}$ increases.


Moreover, as $s=o(p)$, there exists some constant $C$ such that
\begin{align*}
\mathbb{D}(1)=\frac{2+\tilde{\kappa}_r^2}{ (2+\tilde{\kappa}_r)^2 }\times \frac{ \{1+2\sum_{t=1}^s\tilde{h}_t^2(1-t/p)\}^2 }{ \{ 1+2\sum_{t=1}^s \tilde{h}_t(1-t/p) \}^4   }\geq \mathbb{D}_0, \notag
\end{align*} where 
\begin{align*}
\mathbb{D}_0=C\times \frac{2+\tilde{\kappa}_r^2}{ (2+\tilde{\kappa}_r)^2 }\times \frac{ \{1+2\sum_{t=1}^s\tilde{h}_t^2\}^2}{ \{ 1+2\sum_{t=1}^s \tilde{h}_t \}^4   },	
\end{align*}and we have $\mathbb{D}_0=\Theta(1/s^2)$. 
Therefore, when $ \mathbb{D}_0\geq \tilde{M}^2$, i.e., $|J_D|\geq Mp/\sqrt{\mathbb{D}_0}$, we know $\mathbb{D}(1)\geq \tilde{M}^2$ and the minimum of $\mathbb{D}(a)$ is achieved at $a_0=1$. This indicates that the minimum of $\rho_a$ is achieved at $a_0=1$.

\subsection{Results on the Generalized Linear Model in Section \ref{sec:newglm}} \label{sec:extglm}

\subsubsection{Limiting results and power analysis} \label{sec:mainresultglm}

We have shown that the U-statistics framework can be used to test means and covariance matrices. Here we give an example of  generalized linear models to show that the framework can be extended to other testing problems. 

Consider a response variable  $y$ and covariates $\mathbf{x} = (x_1,\cdots, x_p)^{\intercal}$ following a generalized linear model   
\begin{eqnarray}
	\mathrm{E}(y|\mathbf{x})=g^{-1}(\mathbf{x}^{\intercal}\boldsymbol{\beta}), \label{eq:glmmodel}
\end{eqnarray} where $g$ is the canonical link function and $\boldsymbol{\beta}$ is the regression coefficients of interest.  We are interested in testing:
$
	H_0: \boldsymbol{\beta}=\boldsymbol{\beta}_0$ {versus}  $H_A: \boldsymbol{\beta} \neq \boldsymbol{\beta}_0.
$
We define the score vector  $\mathbf{S}=(S_1,\ldots, S_p)^{\intercal}$ for $\boldsymbol{\beta}$ in \eqref{eq:glmmodel}, where $S_j=  (y-\mu_{0}) x_{j}$, $1\leq j \leq p$ with $\mu_{0}=g^{-1}(\mathbf{x}^{\intercal} \boldsymbol{\beta}_0)$. 
Given that $E(S_j)=0$ under $H_0$,  the  target parameters can be considered as  $\mathcal{E}=\{ \mathrm{E}(S_j): j=1,\ldots,p \}$. 

Suppose that   $(\mathbf{x}_{i}, y_i)$, $i=1,\ldots, n$, are $n$ i.i.d. observations. Many existing tests for generalized linear models \cite[see, e.g.,][]{guo2016tests,wu2019adaptive} are based on the score vectors ${\mathbf S}_i=(S_{i,1},\ldots, S_{i,p})^{\intercal}$, where $S_{i,j}=(y_i-\mu_{0,i})x_{i,j}$.
Note that  ${\mathbf S}_i$'s  are i.i.d. copies of $\mathbf{S}$ with mean $(\mathrm{E}(S_1),\ldots, \mathrm{E}(S_p))^{\intercal}$ and the covariance matrix denoted by $\boldsymbol{\Sigma}=\{\sigma_{j_1,j_2}: 1\leq j_1,j_2 \leq p \}$. Therefore, $K_j( \mathbf{x}_{i},y_i)=S_{i,j}=(y_i-\mu_{0,i})x_{i,j}$ provides a simple  kernel function. Following \eqref{eq:generalustat}, 
$\mathcal{U}(a)=\sum_{j=1}^p (P^n_a)^{-1} \sum_{1\leq i_1 \neq \ldots \neq i_a \leq n} \prod_{k=1}^a S_{i_k,j}$, 
which is an unbiased estimator of $\|\mathcal{E}\|_a^a=\sum_{j=1}^p \{\mathrm{E}(S_j)\}^a$ for finite integers $a$.  Moreover, we define $\mathcal{U}(\infty)=\max_{1\leq j\leq p}  \sigma_{j,j}^{-1} (\sum_{i=1}^n S_{i,j} /n)^2$, which corresponds to the  $\|\mathcal{E}\|_{\infty}$.

Asymptotic results of the U-statistics are stated below, where we assume the conditions  similar to that of Theorem \ref{thm:onesamplemean}.

\begin{condition} \label{cond:glmnormal} \quad
\begin{enumerate}
	\item[(1)] There exists constant $B$ such that $B^{-1} \leq \lambda_{\min}(\boldsymbol{\Sigma}) \leq \lambda_{\max}(\boldsymbol{\Sigma}) \leq B$, where $\lambda_{\min}(\boldsymbol{\Sigma})$ and $\lambda_{\max}(\boldsymbol{\Sigma})$ denote the minimum and maximum eigenvalues of the covariance matrix $\boldsymbol{\Sigma}$; and  all correlations are bounded away from $-1$ and $1$, i.e., $\max_{1\leq j_1 \neq j_2\leq p}|\sigma_{j_1,j_2}|/(\sigma_{j_1,j_2} \sigma_{j_2,j_2})^{1/2}<1-\eta$ for some $\eta>0$.
	\item[(2)] $\log p=o(1)n^{1/4}$ and $\max_{1\leq j \leq p}\mathrm{E}[\exp\{ h (S_{j} -\mathrm{E}(S_{j}))^2  \}] < \infty,$ for $h\in [-M,M]$, where $M$ is a positive constant.
	\item[(3)] Similarly to Condition \ref{cond:alphamixing},  $\{(S_{i,j},i=1\ldots,n): 1\leq j\leq p  \}$ is $\alpha$-mixing with $\alpha_S(s) \leq C\delta^s$, where $\delta \in (0,1)$ and $C$ is some constant. In addition, for finite integer $a$, $\sum_{j_1,j_2=1}^p\sigma_{j_1,j_2}^a=\Theta(p)$. 
\end{enumerate} 
\end{condition}

\begin{theorem} \label{thm:glmnormal}
Under Condition \ref{cond:glmnormal} and  $H_0$: $\boldsymbol{\beta}=\boldsymbol{\beta}_0$, 
 	for any finite  integers $(a_1,\ldots, a_m)$, as $n,p \rightarrow \infty$, $ [ {\mathcal{U}(a_1)}/{\sigma(a_1)},\allowbreak \ldots,  {\mathcal{U}(a_m)}/{\sigma(a_m)}  ]^{\intercal} \xrightarrow{D} \mathcal{N}(0,I_m),$
where    $\sigma^2(a) =\sum_{i=1}^p \sum_{j=1}^p \sigma_{i,j}^{a}/P^n_a$, which is of order $\Theta(pn^{-a})$. 
Besides, 
$
	P ( {n}\mathcal{U}(\infty)-\tau_p \leq u ) \rightarrow \exp \{ - \pi^{-1/2} \exp(-u/2)  \}, 
$ $\forall u\in \mathbb{R}$, 	where $\tau_p=2\log p - \log \log p$. 
In addition, for any finite integer  $a$, $\{\mathcal{U}(a)/\sigma(a)\}$ and $\{ n\mathcal{U}(\infty)-\tau_p\}$ are asymptotically independent.  
\end{theorem}

 Next we compare the power of $\mathcal{U}(a)$'s under alternatives with different sparsity levels. Similarly to the mean testing problems, we consider the  alternative $\mathcal{E}_A=\{\mathrm{E}(S_{j}) =\rho > 0\mbox{ for }  j=1,\ldots, k_0; \mathrm{E}(S_{j})=0 \mbox{ for } j=k_0+1,\cdots, p\}$, where $k_0$ denotes  the number of nonzero entries.  
   
\begin{theorem} \label{thm:altcltglmtest}
Assume Condition \ref{cond:glmnormal} and $k_0=o(p)$. For any finite integers $\{a_1,\ldots, a_m\}$, if $\rho$ in $\mathcal{E}_A$ satisfies $\rho=O(k_0^{-1/a_t} p^{1/(2a_t)}n^{-1/2})$ for $t=1,\ldots,m$, then $
	[ \mathcal{U}(a_1)-\mathrm{E}\{\mathcal{U}(a_1)\}]/{\sigma(a_1)}, \ldots,\allowbreak [\mathcal{U}(a_m)-\mathrm{E}\{\mathcal{U}(a_m)\}]/{\sigma(a_m)}  ]^{\intercal} \xrightarrow{D} \mathcal{N}(0,I_m),  
$ as  $n,p \rightarrow \infty$. 
In addition, $\mathrm{E}[\mathcal{U}(a)]=\|\mathcal{E}_A\|_a^a=k_0\rho^a$   and $$\sigma^2(a) \simeq \sum_{j_1=k_0+1}^p \sum_{j_2=k_0+1}^p a! \sigma_{j_1,j_2}^a /P^n_a,$$ which is $\Theta(a!pn^{-a})$.
\end{theorem}

Theorem \ref{thm:altcltglmtest}
 shows that under the considered local alternatives, the asymptotic  power of $\mathcal{U}(a)$ mainly depends on $\mathrm{E}\{\mathcal{U}(a) \}/ \sqrt{\mathrm{var}\{\mathcal{U}(a)\}}$. 
Therefore, for a given constant $M>0$, if $\rho=\rho_a$ defined as $\rho_a=M^{{1}/{a}} k_0^{-{1}/{a}}a!^{{1}/{(2a)}}\times (\sum_{j_1=k_0+1}^p \sum_{j_2=k_0+1}^p \sigma_{j_1,j_2}^a)^{{1}/{(2a)}}n^{-1/2}$,
  we know that different $\mathcal{U}(a)$'s asymptotically have the same	 power.  For illustration, we further assume that $\sigma_{j,j}=1$ when $j \in \{k_0+1,\ldots,p\}$, and $\sigma_{j_1,j_2}=0$ when $j_1\neq j_2 \in \{k_0+1,\ldots,p\}$, then 
\begin{eqnarray}
	\rho_a\simeq (M\sqrt{p}/k_0)^{\frac{1}{a}}a!^{\frac{1}{2a}}n^{-\frac{1}{2}}.\label{eq:deltaavalueglm}
\end{eqnarray}    Therefore, following the analysis in Section \ref{maintest}, to find the ``best" $\mathcal{U}(a)$, it suffices to find the order, denoted by  $a_0$, that gives the smallest $\rho_a$ value in \eqref{eq:deltaavalueglm}. 
Since \eqref{eq:deltaavalueglm} is only different from  \eqref{eq:deltaavalue} by a constant that does not depend on the order $a$,   Proposition \ref{cor:deltaminimval} still holds.
Consider  $a_0\geq 1$ as specified in  Proposition  \ref{cor:deltaminimval}; then, similar to results in the two-sample mean testing,  we know when $k_0\geq \sqrt{ Mp}$, $a_0=1$ and $\mathcal{U}(1)$ is   ``better" than $\mathcal{U}(\infty)$; when $ k_0 < C_1  { \sqrt{ p} }/{ \log^{{a_0}/{2} } p}$ for some $C_1$, $\mathcal{U}(\infty)$ is the ``best"; and when $C_2 { \sqrt{p} }/{ \log^{{a_0}/{2} } p}   < k_0 <  \sqrt{ Mp}$ for some $C_2$, 
$\mathcal{U}(a_0)$ is the ``best". 
In addition, given the similar  results obtained in Theorem \ref{thm:glmnormal} and  power analysis, we can also develop adaptive testing procedure similar to that in Section \ref{sec:computtest}.


\begin{remark}\label{rm:glmrm}
More generally, if the generalized linear model also has covariates $\mathbf{z}$ that we want to adjust for, the corresponding generalized linear model becomes  $\mathrm{E}(y|\mathbf{x})=g^{-1}(\mathbf{x}^{\intercal}\boldsymbol{\beta}+\mathbf{z}^{\intercal}\boldsymbol{\alpha})$, where $\boldsymbol{\alpha}$ denote  the regression coefficients for $\mathbf{z}$. To test 	$H_0: \boldsymbol{\beta}=\boldsymbol{\beta}_0   $ v.s. $H_A: \boldsymbol{\beta} \neq \boldsymbol{\beta}_0,$
we can replace $\mu_{0,j}$ by $\hat \mu_{0,j} = g^{-1}(\mathbf{x}_i^{\intercal} \boldsymbol{\beta}_0+\mathbf{z}_i^{\intercal} \hat{\boldsymbol{\alpha}})$
where $\hat{\boldsymbol{\alpha}}$ is an estimator of $\boldsymbol{\alpha}$. For instance, when $\mathbf{z}$ is low dimensional, we can take $\hat{\boldsymbol{\alpha}}$ as the maximum likelihood estimator under $H_0$. Then similar conclusion to Theorem \ref{thm:glmnormal} can be derived under certain regularity conditions.
We present simulation studies on generalized linear model   in \ref{suppA} Section \ref{sec:glmsim} to illustrate the good performance of the U-statistics and we leave the details of theoretical developments with nuisance parameters for future study. 
 \end{remark}

\subsubsection{Proof of Theorems \ref{thm:glmnormal} and \ref{thm:altcltglmtest} (on Page \pageref{thm:glmnormal})}
Theorem \ref{thm:glmnormal} is proved following the proof of Theorem \ref{thm:onesamplemean} in Section \ref{sec:proof3132}.  Specifically, the  arguments in Section \ref{sec:proof3132} can be  applied to proving Theorem \ref{thm:glmnormal} by replacing $x_{i,j}$'s  with $S_{i,j}$'s, and therefore the details are skipped.

The proof of Theorem \ref{thm:altcltglmtest}  is similar to the proof of Theorem \ref{thm:altcltmeantest} in Section \ref{sec:proofthem34}.  
In particular, we decompose $\mathcal{U}(a)=T_{a,1}+T_{a,2}$, where we redefine
\begin{eqnarray*}
	T_{a,1} = \sum_{j=1}^{k_0} \frac{1}{P^n_{a}}\sum_{1 \leq i_1 \neq \cdots \neq i_{a} \leq n}  \prod_{k=1}^a  S_{i_k,j}, \quad	T_{a,2} =\sum_{j=k_0+1}^{p} \frac{1}{P^n_{a}}\sum_{1 \leq i_1 \neq \cdots \neq i_{a} \leq n}  \prod_{k=1}^a  S_{i_k,j}.
\end{eqnarray*}
Note that $T_{a,2}$ is a summation over $j\in \{k_0+1,\ldots, p\}$ and $\mathrm{E}(S_j)=0$ for $j\in \{k_0+1,\ldots, p\}$. Thus the conclusions similar to that in  Theorem \ref{thm:glmnormal} hold for $T_{a,2}$. 
Specifically, we have $\mathrm{var}(T_{a,2})=\Theta\{(p-k_0)n^{-a}\}$ and  
\begin{align}
	\Big[T_{a_1,2}/\sqrt{\mathrm{var}(T_{a_1,2})},\ldots, T_{a_m,2}/\sqrt{\mathrm{var}(T_{a_m,2})}\Big]\xrightarrow{D}\mathcal{N}(0,I_m).\label{eq:glmnormpart1}
\end{align}
When $\mathrm{var}(T_{a,1})=o(1)\mathrm{var}(T_{a,2})$, which will be proved later, we have $\sigma^2(a)\simeq \mathrm{var}(T_{a,2})$ and $\{T_{a,1}-\mathrm{E}(T_{a,1})\}/\sigma(a)\xrightarrow{P}0$. By the Slutsky's theorem and \eqref{eq:glmnormpart1},  Theorem \ref{thm:altcltglmtest} is proved.
 


To finish the proof of Theorem \ref{thm:altcltglmtest}, it remains to prove $\mathrm{var}(T_{a,1})=o(1)\mathrm{var}(T_{a,2})$. 
The analysis above gives that $
	\mathrm{var}(T_{a,2})=\Theta\{(p-k_0)n^{-a}\}.   
$ 
As $k_0=o(p)$, to prove  $\mathrm{var}(T_{a,1})=o(1)\mathrm{var}(T_{a,2})$, it suffices to show $\mathrm{var}(T_{a,1})=o(pn^{-a})$.
Note that $\mathrm{var}(T_{a,1})=\mathrm{E}(T_{a,1}^2)-\{\mathrm{E}(T_{a,1})\}^2$,  $\mathrm{E}(T_{a,1})=k_0\rho^a$,  and 
\begin{align*}
	\mathrm{E}(T_{a,1}^2)=\frac{1}{(P^n_a)^2} \sum_{1\leq j_1,j_2\leq k_0}\sum_{\substack{ 1 \leq i_1 \neq \cdots \neq i_{a} \leq n; \\ 1 \leq \tilde{i}_1 \neq \cdots \neq \tilde{i}_{a} \leq n} }\mathrm{E}\Big\{ \prod_{k=1}^a  (S_{i_k,j_1}S_{\tilde{i}_k,j_2})\Big\}.
\end{align*} 
For $0\leq b\leq a$, define an event $B_{S,b}=\{ \{i_1,\ldots, i_a\}\cap \{\tilde{i}_1,\ldots, \tilde{i}_a \}\mathrm{\ is\ of\ size\ }b  \}$ and correspondingly 
\begin{align*}
G_{S,a,2,b}=&~(P^n_a)^{-2}\sum_{1\leq j_1,j_2\leq k_0}\sum_{\substack{ 1 \leq i_1 \neq \cdots \neq i_{a} \leq n; \\ 1 \leq \tilde{i}_1 \neq \cdots \neq \tilde{i}_{a} \leq n} }\mathrm{E}\Big\{ \prod_{k=1}^a  (S_{i_k,j_1}S_{\tilde{i}_k,j_2})\times \mathbf{1}_{B_{S,b} } \Big\}	. 
\end{align*}Then $\mathrm{E}(T_{a,1}^2)=\sum_{b=0}^aG_{S,a,2,b}$. To prove $\mathrm{E}(T_{a,1}^2)-\{\mathrm{E}(T_{a,1})\}^2=o(pn^{-a})$, we show $ G_{S,a,2,0}-\{\mathrm{E}(T_{a,1})\}^2=o(pn^{-a})$ and $\sum_{b=1}^a G_{S,a,2,b}=o(pn^{-a})$, respectively.  

When $b=0$, $\{i_1,\ldots, i_a\}\cap\{\tilde{i}_1,\ldots, \tilde{i}_a \}=\emptyset$, and it follows that $G_{S,a,2,0}=(P^n_a)^{-2}k_0^2 P^n_{2a} \rho^{2a}.$
By $\mathrm{E}(T_{a,1})=k_0\rho^a$ and $k_0^2\rho^{2a}=O(pn^{-a})$, we have $|G_{S,a,2,0}-\{ \mathrm{E}(T_{a,1})\}^2|=o(k_0^2\rho^{2a})=o(pn^{-a})$. When  $b\geq 1$, 
\begin{align*}
	G_{S,a,2,b}=&~C(P^n_a)^{-2}\sum_{1\leq j_1,j_2\leq k_0}P^n_{2a-b} (\sigma_{j_1,j_2}+\rho^2)^b \rho^{2(a-b)}.
\end{align*} The maximum order of $G_{S,a,2,b}$ is bounded by the following two quantities:
\begin{eqnarray}
	&&\sum_{1\leq j_1,j_2\leq k_0}\frac{P^n_{2a-b}}{(P^n_a)^{2}} \sigma_{j_1,j_2}^b \rho^{2(a-b)}, \label{eq:maxorderglmfirst} \\
	&& \sum_{1\leq j_1,j_2\leq k_0}\frac{P^n_{2a-b}}{(P^n_a)^{2}} \rho^{2a}. \label{eq:maxorderglmsecond}
\end{eqnarray} For \eqref{eq:maxorderglmfirst}, as $b\geq 1$, by  Condition \ref{cond:glmnormal} (3) and Lemma \ref{lm:mixingineq}, $\eqref{eq:maxorderglmfirst}=O\{k_0n^{-b}\rho^{2(a-b)}\}$. As $k_0=o(p)$ and $\rho=O(k_0^{-1/a}p^{1/(2a)}n^{-1/2})$, we know $\eqref{eq:maxorderglmfirst}=o(pn^{-a})$. For \eqref{eq:maxorderglmsecond}, when $b\geq 1$, $\eqref{eq:maxorderglmsecond} =O(k_0^2n^{-b}\rho^{2a})=o(k_0^2\rho^{2a})=o(pn^{-a})$. In summary, we have 
$
	|\mathrm{var}(T_{a,1})|
	\leq |\{\mathrm{E}(T_{a,1})\}^2-G_{S,a,2,0}|+\sum_{b=1}^a |G_{S,a,2,b}|=o(pn^{-a}).
$ Therefore, Theorem \ref{thm:altcltglmtest} is proved.

\section{Assisted Lemmas} \label{sec:alllemma}

In the following Sections \ref{sec:pfjointnormwhole}--\ref{sec:twocovalt}, we provide the proofs of all the assisted lemmas used in Section \ref{sec:a}. The proofs of Remark \ref{rm:standizedmax} and Corollary \ref{prop:generalrestwomean} are provided in Sections  \ref{sec:pfstandmax} and \ref{sec:pfcor41}, respectively. To facilitate the presentation of the proofs, we first introduce some notation and then provide four  technical Lemmas \ref{lm:mixingineq}--\ref{lm:inequaluvtheta}. 

\subsubsection*{Notation}\label{par:notationindpcond}
We define some notation to simplify the representation of  summations in the following proofs.  
For $a<n$, $\mathcal{P}(n,a)$ denotes the collection of $a$-tuples  $\mathbf{i}=(i_1,\ldots, i_a)$ satisfying $1\leq i_1\neq \ldots \neq  i_a\leq n.$ Given $\mathbf{i}\in \mathcal{P}(n,a)$, we define $\{\mathbf{i}\}$ as the corresponding set containing the elements of $\mathbf{i}$ without order, that is, $\{\mathbf{i}\}=\{i_1,\ldots, i_a\}$.   
 We apply usual set operations on the corresponding set of $\{\mathbf{i}\}$. For example, $|\{\mathbf{i}\}|$ denotes the size of the  set $\{i_1,\ldots, i_a\}$, which is $a$ in this case. In addition, for any two integers $a, b <n $, and two tuples $\mathbf{i}\in \mathcal{P}(n,a)$ and $\mathbf{m} \in \mathcal{P}(n,b)$, the operations $\{\mathbf{i}\}\cup \{\mathbf{m}\}$ and  $\{\mathbf{i}\}\cap \{\mathbf{m}\}$ give the sets that equal to the union $\{i_1,\ldots, i_a\}\cup \{m_1,\ldots, m_b\}$ and intersection $\{i_1,\ldots, i_a\}\cap \{m_1,\ldots, m_b\}$ respectively.  Moreover, we write $\{\mathbf{i}\}=\{\mathbf{m}\}$ and $\{\mathbf{i}\}\neq \{\mathbf{m}\}$ to indicate that the two sets $\{i_1,\ldots, i_a\}$ and $\{m_1,\ldots, m_b\}$ contain the same elements or not respectively. 
 
In addition, let $\mathcal{C}(n,a)$ denote the collection of $a$-tuples  $\mathbf{i}=(i_1,\ldots, i_a)$ satisfying $1\leq i_1, \ldots ,i_a\leq n$ without constraining the elements to be different. Similarly, we define $\{\mathbf{i}\}$  as the set containing the elements of $\mathbf{i}$ without order, and the set operations also apply similarly as above. Note that $|\{\mathbf{i}\}|$ may be smaller than $a$ under this case. 
 
We next list four technical lemmas which shall be used in the proofs later. 
\begin{lemma}\cite[Eq. (3.5)]{guyon1995random} \label{lm:mixingineq}
Under the mixing assumption in Condition \ref{cond:alphamixing},  
suppose $Z_1$ and $Z_2$ are $\mathcal{Z}_1^t$-measurable  and $\mathcal{Z}_{t+m}^{\infty}$-measurable random variables respectively. 
When  $\mathrm{E}(|Z_1|^{2+\epsilon})<\infty$ and $ \mathrm{E}(|Z_2|^{2+\epsilon}) < \infty$, for some constants $C$ and $\epsilon>0$, 
\begin{align*}
	|\mathrm{cov}(Z_1,Z_2)|\leq C \{\alpha(m) \}^{\frac{2}{2+\epsilon}} \{\mathrm{E}(|Z_1|^{2+\epsilon}) \}^{\frac{1}{2+\epsilon}}\{\mathrm{E}(|Z_2|^{2+\epsilon}) \}^{\frac{1}{2+\epsilon}}. 
\end{align*}
\end{lemma}\noindent The lemma above can also be obtained  from  Lemma 2.4 in \cite{kim1994momentbounds} by taking $p=q=2+\epsilon$. 

\begin{lemma}{\cite[Lemma 3.4.3]{durrett2019probability}} \label{lemma:prodcutcom}
	When $|a_i| \leq A$ and $|b_i| \leq A$, then
	\begin{eqnarray*}
		\left|\prod_{i=1}^q a_i - \prod_{i=1}^q b_i \right| \leq \sum_{i=1}^q \left| a_i -b_i \right| A^{q-1}.
	\end{eqnarray*} 
\end{lemma}

\begin{lemma}{\cite[Eq. (24)]{cai2013two}} \label{lm:maxdiffbound}
	for two series of numbers $A_j$ and $B_j$ for $j=1,\ldots,p$. 
\begin{align*}
	\Big|\max_{1\leq j \leq p} A_j^2 - \max_{1\leq j \leq p} B_j^2\Big| \leq 2 \max_{1\leq j \leq p}|B_j| \max_{1\leq j \leq p} |A_j-B_j| + \max_{1\leq j \leq p} |A_j-B_j|^2. 
\end{align*}
\end{lemma} 
\begin{lemma}\label{lm:inequaluvtheta}
When $u,v\geq 0$ and $0<\vartheta\leq 1$, $(u+v)^{\vartheta}\leq u^{\vartheta}+v^{\vartheta}$.
\end{lemma}
\begin{proof}
When $u\geq 0$ and $0<\vartheta \leq 1$,  $f(u)=u^{\vartheta}$ is concave function with $f(0)=0$. By the subadditivity property of concave function, we have $f(u+v)\leq f(u)+f(v)$. 
\end{proof}


\subsection{Lemmas for the proof of Theorem \ref{thm:jointnormal}}\label{sec:pfjointnormwhole}

In this section, we  prove  the  lemmas for the proof of Theorem \ref{thm:jointnormal} in Section  \ref{sec:detailofjointnormal}. We still assume without loss of generality that $\mathrm{E}(\mathbf{x})=\mathbf{0}$ as in Section \ref{sec:detailofjointnormal}.  

\subsubsection{Proof of Lemma \ref{lm:varianceorder} (on Page \pageref{lm:varianceorder}, Section \ref{sec:detailofjointnormal})} \label{sec:proofvarianceorder}

To illustrate the main idea of the proof of Lemma \ref{lm:varianceorder}, we first consider a setting where   $x_{i,j}$'s are all independent, and under this independence case we prove  Lemma \ref{lm:varianceorder} in Section \ref{par:pfvarindep}. 
 Next in Section \ref{par:complpflm1}, we prove Lemma \ref{lm:varianceorder} under the dependence case with Condition  \ref{cond:alphamixing}. Last in Section \ref{par:pfvarunderell}, we present the proof under Condition \ref{cond:ellpmoment}


\paragraph{Proof illustration} \label{par:pfvarindep} 
In this section, we present the proof of  Lemma \ref{lm:varianceorder} by only replacing Condition \ref{cond:alphamixing} with the assumption that  $x_{i,j}$'s are independent. 
Recall   $\tilde{\mathcal{U}}(a)$ defined in \eqref{eq:originleadingterm} and $\tilde{\mathcal{U}}^*(a)=\mathcal{U}(a)-\tilde{\mathcal{U}}(a)$.
Then $\mathrm{var}\{\mathcal{U}(a)\}\leq \mathrm{var}\{\tilde{\mathcal{U}}(a)\} + 2\sqrt{ \mathrm{var}\{\tilde{\mathcal{U}}(a)\} \mathrm{var}\{\tilde{\mathcal{U}}^*(a) \}}  +  \mathrm{var}\{\tilde{\mathcal{U}}^*(a)\}$. To prove Lemma \ref{lm:varianceorder}, we  derive $\mathrm{var}\{\tilde{\mathcal{U}}(a)\}$ and show $\mathrm{var}\{\tilde{\mathcal{U}}^*(a)\}=o(1)\mathrm{var}\{\tilde{\mathcal{U}}(a)\}$. 

  We derive $\mathrm{var}\{\tilde{\mathcal{U}}(a)\}$ first. Under $H_0$, $\mathrm{E}(x_{i,j_1}x_{i,j_2})=0$ when $j_1\neq j_2$. It follows that $\mathrm{E}\{\tilde{\mathcal{U}}(a)\} =0$ and $\mathrm{var}\{\tilde{\mathcal{U}}(a)\}=\mathrm{E}[\{\tilde{\mathcal{U}}(a)\}^2]$, and then 
\begin{align*}
	\mathrm{var}\{\tilde{\mathcal{U}}(a)\}=& \frac{1}{(P^n_a)^2}\sum_{\substack{1\leq j_1\neq j_2\leq p\\1\leq j_3\neq j_4\leq p}}~  \sum_{ \substack{\mathbf{i},\, \tilde{\mathbf{i}}\, \in  \mathcal{P}(n,a)} } \mathrm{E}\Big\{  \prod_{k=1}^a ( x_{i_k,j_1} x_{i_k,j_2}) ( x_{\tilde{i}_k,j_3} x_{\tilde{i}_k,j_4}) \Big\},
\end{align*} where following the notation defined at the beginning of Section \ref{sec:alllemma}, 
$\mathbf{i}$ and $\tilde{\mathbf{i}}$ represent some tuples $\mathbf{i}=(i_1,\ldots, i_a)$ satisfying $1\leq i_1\neq \ldots \neq  i_a\leq n;$ and $\tilde{\mathbf{i}}=(i_1,\ldots, i_a)$ satisfying $1\leq i_1\neq \ldots \neq  i_a\leq n.$
When the corresponding two sets $\{\mathbf{i}\}\neq \{\tilde{\mathbf{i}}\}$, for example, when index $i_1\in \{\mathbf{i}\}$ but $i_1 \not \in \{\tilde{\mathbf{i}}\}$,
\begin{align}
& \mathrm{E}\Big\{  \prod_{k=1}^a ( x_{i_k,j_1} x_{i_k,j_2}) ( x_{\tilde{i}_k,j_3} x_{\tilde{i}_k,j_4}) \Big\} \label{eq:summedtermprodvarind} \\
=& \mathrm{E}(x_{i_1,j_1}x_{i_1,j_2})\times \mathrm{E}(\text{all the remaining terms})=0. \notag 
\end{align}  
Therefore, $\eqref{eq:summedtermprodvarind}\neq 0$ only when $\{\mathbf{i}\}=\{\tilde{\mathbf{i}}\}$, i.e., $\{i_1,\ldots, i_a\}=\{\tilde{i}_1,\ldots, \tilde{i}_a\}$. In particular, when $\{\mathbf{i}\}=\{\tilde{\mathbf{i}}\}$,
\begin{align*}
\mathrm{E}\Big\{  \prod_{k=1}^a ( x_{i_k,j_1} x_{i_k,j_2}) ( x_{\tilde{i}_k,j_3} x_{\tilde{i}_k,j_4}) \Big\}=\{ \mathrm{E}( x_{1,j_1} x_{1,j_2}x_{1,j_3} x_{1,j_4})\}^a.	
\end{align*} It follows that
\begin{align*}
\mathrm{var}\{\tilde{\mathcal{U}}(a)\}=&~ \frac{a!}{(P^n_a)^2} \sum_{\mathbf{i}\in \mathcal{P}(n,a) }\sum_{\substack{1\leq j_1\neq j_2\leq p\\1\leq j_3\neq j_4\leq p}}	\{ \mathrm{E}( x_{1,j_1} x_{1,j_2}x_{1,j_3} x_{1,j_4})\}^a \notag \\
 =&~\frac{a!}{P^n_a}\sum_{1\leq j_1\neq j_2\leq p;\,  1\leq j_3\neq j_4\leq  p}	\{ \mathrm{E}( x_{1,j_1} x_{1,j_2}x_{1,j_3} x_{1,j_4})\}^a. 
\end{align*} 
 When $x_{i,j}$'s are independent, as $j_1\neq j_2$ and $j_3\neq j_4$, $ \mathrm{E}( x_{1,j_1} x_{1,j_2}x_{1,j_3} x_{1,j_4}) \neq 0$ only when $\{j_1,j_2\}=\{j_3,j_4\}$, which gives $ \mathrm{E}( x_{1,j_1} x_{1,j_2}x_{1,j_3} x_{1,j_4})=\mathrm{E}( x_{1,j_1}^2)\times  \mathrm{E}(x_{1,j_2}^2)$. Therefore, $\mathrm{var}\{\tilde{\mathcal{U}}(a)\}=2a!(P^n_a)^{-1}\sum_{1\leq j_1\neq j_2\leq p}\mathrm{E}( x_{1,j_1}^2) \mathrm{E}(x_{1,j_2}^2).$ By Condition \ref{cond:finitemomt}, we have $\mathrm{var}\{\tilde{\mathcal{U}}(a)\}=\Theta(p^2n^{-a})$.

We next show $\mathrm{var}\{\tilde{\mathcal{U}}^*(a)\}=o(1)\mathrm{var}\{\tilde{\mathcal{U}}(a)\}$. 
As $\mathrm{E}\{\tilde{\mathcal{U}}^*(a)\} =0$, $\mathrm{var}\{\tilde{\mathcal{U}}^*(a)\}=\mathrm{E}[\{\tilde{\mathcal{U}}^*(a)\}^2]$. Recall the definition of $\mathcal{U}^*(a)$, then we have
\begin{align*}
	\mathrm{var}\{\tilde{\mathcal{U}}^*(a)\}=& \sum_{\substack{1\leq j_1\neq j_2\leq p\\1\leq j_3\neq j_4\leq p}} \sum_{\substack{ 1\leq c_1, c_2\leq a}}\sum_{ \substack{\mathbf{i}\in \mathcal{P}(n,a+c_1)\\ \tilde{\mathbf{i}}\in \mathcal{P}(n,a+c_2) }}\frac{(-1)^{c_1+c_2} \binom{a}{c_1}\binom{a}{c_2} }{P^n_{a+c_1}P^n_{a+c_2}} Q(\mathbf{i},j_1,j_2,\tilde{\mathbf{i}},j_3,j_4), 
\end{align*} where we correspondingly define
\begin{align*}
 Q(\mathbf{i},j_1,j_2,\tilde{\mathbf{i}},j_3,j_4)
=&~\mathrm{E} \Biggr[\prod_{k=1}^{a-c_1}x_{i_k,j_1} x_{i_k,j_2}\prod_{k=a-c_1+1}^{a} x_{i_{k},j_1}  \prod_{k=a+1}^{a+c_1}x_{i_{k},j_2} \notag \\
&~\quad \times \prod_{\tilde{k}=1}^{a-c_2}  x_{\tilde{i}_{\tilde{k}},j_3} x_{\tilde{i}_{\tilde{k}},j_4} \prod_{\tilde{k}=a-c_2+1}^a x_{\tilde{i}_{\tilde{k}},j_3} \prod_{\tilde{k}=a+1}^{a+c_2}x_{\tilde{i}_{\tilde{k}},j_4}  \Biggr].
\end{align*}

To evaluate $\mathrm{var}\{\tilde{\mathcal{U}}^*(a)\}$, we examine the value of $Q(\mathbf{i},j_1,j_2,\tilde{\mathbf{i}},j_3,j_4)$. We first note that if $Q(\mathbf{i},j_1,j_2,\tilde{\mathbf{i}},j_3,j_4)\neq 0$, the following two claims hold:
\begin{itemize}
	\item[]  \textit{Claim 1:} $\{j_1,j_2\}=\{j_3,j_4\}$; \quad \textit{Claim 2:} $\{\mathbf{i}\}= \{\tilde{\mathbf{i}}\}$ and $c_1=c_2$. 
\end{itemize} 
To prove \textit{Claim 1}, we show that if $\{j_1,j_2\}\neq \{j_3,j_4\}$, $Q(\mathbf{i},j_1,j_2,\tilde{\mathbf{i}},j_3,j_4)=0$. We consider $j_1\not \in \{j_3,j_4\}$ as an example. When $j_1\not \in \{j_3,j_4\}$, as $j_1\neq j_2$, we further know $j_1\not \in \{j_2,j_3,j_4\}$   and we can write $$Q(\mathbf{i},j_1,j_2,\tilde{\mathbf{i}},j_3,j_4)= \mathrm{E}\Big( \prod_{k=1}^a x_{i_k,j_1}\Big)\times \mathrm{E}(\text{other terms  with subscripts } j_2,j_3,j_4)=0, $$
where we use $\mathrm{E}( \prod_{k=1}^a x_{i_k,j_1}) =\{\mathrm{E}(x_{1,j_1})\}^a=0$ as $\mathrm{E}(x_{1,j_1})={0}$. In addition, to prove \textit{Claim 2}, we show that if $\{\mathbf{i}\}\neq \{\tilde{\mathbf{i}}\}$, $Q(\mathbf{i},j_1,j_2,\tilde{\mathbf{i}},j_3,j_4)=0$.  If $\{\mathbf{i}\}\neq \{\tilde{\mathbf{i}}\}$, similarly to \eqref{eq:summedtermprodvarind}, suppose an index $i\in \{\mathbf{i}\}$ but $i\not \in \{\tilde{\mathbf{i}}\}$. Then we can write $Q(\mathbf{i},j_1,j_2,\tilde{\mathbf{i}},j_1,j_2)=\mathrm{E}(x_{i,j_1})\times \mathrm{E}(\text{other terms})=0$ or $Q(\mathbf{i},j_1,j_2,\tilde{\mathbf{i}},j_1,j_2)=\mathrm{E}(x_{i,j_1}x_{i,j_2})\times \mathrm{E}(\text{other terms})=0.$ 
As $\{\mathbf{i}\}$ and $\{\tilde{\mathbf{i}}\}$ are of sizes $a+c_1$ and $a+c_2$ respectively, $\{\mathbf{i}\}=\{\tilde{\mathbf{i}}\}$ induces $c_1=c_2$.

Given \textit{Claim 1} and \textit{Claim 2}, 
 we write $c_1=c_2=c$ and decompose  $\{\mathbf{i}\}$ and $\{\tilde{\mathbf{i}}\}$ into three disjoint subsets respectively as follows:
\begin{align*}
\{\mathbf{i}\}_{(1)}=\{i_1,\ldots,i_{a-c}\}, \ 	\{\mathbf{i}\}_{(2)}=\{i_{a-c+1},\ldots, i_a\},\ \{\mathbf{i}\}_{(3)}=\{i_{a+1},\ldots, i_{a+c}\}, \notag \\
\tilde{\{\mathbf{i}\}}_{(1)}=\{\tilde{i}_1,\ldots,\tilde{i}_{a-c}\}, \ 	\tilde{\{\mathbf{i}\}}_{(2)}=\{\tilde{i}_{a-c+1},\ldots, \tilde{i}_a\},\ \tilde{\{\mathbf{i}\}}_{(3)}=\{\tilde{i}_{a+1},\ldots, \tilde{i}_{a+c}\},
\end{align*}which satisfies that $\{\mathbf{i}\}=\cup_{l=1}^3 \{\mathbf{i}\}_{(l)}$ and $\tilde{\{\mathbf{i}\}}=\cup_{l=1}^3 \{\tilde{\mathbf{i}}\}_{(l)}$. 
We next prove the following \textit{Claim 3}: if  $Q(\mathbf{i},j_1,j_2,\tilde{\mathbf{i}},j_3,j_4)\neq 0$,  one of the following two cases hold:
 \begin{enumerate}
	\item $j_1=j_3$, $j_2=j_4$, $\{\mathbf{i}\}_{(1)}=\tilde{\{\mathbf{i}\}}_{(1)}$,  $\{\mathbf{i}\}_{(2)}=\tilde{\{\mathbf{i}\}}_{(2)}$,   $\{\mathbf{i}\}_{(3)}=\tilde{\{\mathbf{i}\}}_{(3)}$; 
	\item $j_1=j_4$, $j_2=j_3$,  $\{\mathbf{i}\}_{(1)}=\tilde{\{\mathbf{i}\}}_{(1)}$,  $\{\mathbf{i}\}_{(2)}=\tilde{\{\mathbf{i}\}}_{(3)}$,  $\{\mathbf{i}\}_{(3)}=\tilde{\{\mathbf{i}\}}_{(2)}$.  
\end{enumerate} 
To prove \textit{Claim 3}, we note that  \textit{Claim 1} suggests that if $Q(\mathbf{i},j_1,j_2,\tilde{\mathbf{i}},j_3,j_4)\neq 0$, either $\{j_1=j_3, j_2=j_4\}$ or  $\{j_1=j_4, j_2=j_3\}$ holds. We consider   $j_1=j_3$ and $j_2=j_4$ as an example.  Suppose that there exists an index $i\in \{\mathbf{i}\}_{(2)}$. Since $x_{i,j}$'s are independent with mean 0, if $i\in \{\tilde{\mathbf{i}}\}_{(1)}$,    $Q(\mathbf{i},j_1,j_2,\tilde{\mathbf{i}},j_1,j_2)=\mathrm{E}(x_{i,j_1}^2x_{i,j_2})\times \mathrm{E}(\text{other terms})=0$; or if $i\in \{\tilde{\mathbf{i}}\}_{(3)}$,  $Q(\mathbf{i},j_1,j_2,\tilde{\mathbf{i}},j_1,j_2)=\mathrm{E}(x_{i,j_1}x_{i,j_2})\times \mathrm{E}(\text{other terms})=0$. Symmetrically,  if $Q(\mathbf{i},j_1,j_2,\tilde{\mathbf{i}},j_1,j_2)\neq 0$, we know $\{\mathbf{i}\}_{(l)}=\tilde{\{\mathbf{i}\}}_{(l)}$ for $l=1,2,3$ under this case. The similar analysis  also applies to the second case in \textit{Claim 3}. 
Moreover, under the two cases in \textit{Claim 3}, we have
$
	Q(\mathbf{i},j_1,j_2,\tilde{\mathbf{i}},j_3,j_4) = \{\mathrm{E}(x_{1,j_1}^2x_{1,j_2}^2)\}^{a-c}\{\mathrm{E}(x_{1,j_1}^2) \}^c\{\mathrm{E}(x_{1,j_2}^2) \}^c. 
$

In summary,
\begin{align*}
	\mathrm{var}\{\tilde{\mathcal{U}}^*(a)\}=&~ \sum_{1 \leq j_1 \neq j_2 \leq p} \sum_{c=1}^a \sum_{\mathbf{i}\in \mathcal{P}(n,a+c) }\frac{2(a-c)!c!c!}{ (P^n_{a+c})^2}\{\mathrm{E}(x_{1,j_1}^2)\mathrm{E}(x_{1,j_2}^2) \}^a \notag \\
	\leq &~ C\sum_{1 \leq j_1 \neq j_2 \leq p} \sum_{c=1}^a n^{-(a+c)}\{\mathrm{E}(x_{1,j_1}^2) \mathrm{E}(x_{1,j_2}^2) \}^a,
\end{align*} which is of order $O(p^2n^{-(a+1)})$. Since we have obtained that $\mathrm{var}\{\tilde{\mathcal{U}}(a)\}=\Theta(p^2n^{-a})$, then  $\mathrm{var}\{\tilde{\mathcal{U}}^*(a)\}=o(1)\mathrm{var}\{\tilde{\mathcal{U}}(a)\}$ is proved.

\smallskip

\paragraph{Proof under Condition \ref{cond:alphamixing}} \label{par:complpflm1}
{Section \ref{par:pfvarindep}    considers
 the case where $x_{i,j}$'s are independent. 
 In this section, we further  prove Lemma \ref{lm:varianceorder} under Condition \ref{cond:alphamixing}. 
We first explain the proof idea intuitively. 
Under Condition \ref{cond:alphamixing}, 
$x_{i,j}$'s may be no longer  independent, but  the dependence between $x_{i,j_1}$ and $x_{i,j_2}$ degenerates exponentially with their distance $|j_1-j_2|$. 
We expect that when $|j_1-j_2|$ is large enough, $x_{i,j_1}$ and $x_{i,j_2}$ 
are ``asymptotically   independent''. 
Specifically, we will introduce a threshold $K_0$ to be defined in \eqref{eq:thresholddK} below. Then   we will show that the majority of $(x_{i,j_1}, x_{i,j_2})$ pairs satisfy $|j_1-j_2|>K_0$, and when $|j_1-j_2|>K_0$,  $x_{i,j_1}$ and $x_{i,j_2}$ are weakly dependent with similar properties to those under the independence case.}

We next present the detailed  proof under Condition \ref{cond:alphamixing}.  
 Under $H_0$, similarly to Section \ref{par:pfvarindep}, we have
 $\mathrm{E} \{\mathcal{U}(a)\}=0$ and $\mathrm{var}\{\mathcal{U}(a)\}=\mathrm{E}\{\mathcal{U}^2(a)\}$.
Then 
\begin{eqnarray}
&&\quad \quad \mathrm{E}\{\mathcal{U}^2(a)\}=	\sum_{\substack{1 \leq j_1 \neq j_2 \leq p;\\ 1 \leq j_3 \neq j_4 \leq p}} \ \sum_{ \substack{0\leq c_1,c_2\leq a;\\ \mathbf{i}\in \mathcal{P}(n,a+c_1);\\ \tilde{\mathbf{i}}\in  \mathcal{P}(n,a+c_2)} }F(c_1,c_2,a)\times Q(\mathbf{i},j_1, j_2,\tilde{\mathbf{i}},j_3,j_4),\label{eq:varesquaredef}
\end{eqnarray}
where we define $F(c_1,c_2,a)=(-1)^{c_1+c_2}\binom{a}{c_1}\binom{a}{c_2}(P^n_{a+c_1}P^n_{a+c_2})^{-1},$ and recall
\begin{eqnarray}\label{eq:summedadjustedutermvar}
&& Q(\mathbf{i},j_1, j_2,\tilde{\mathbf{i}},j_3,j_4)  \\
&=& \mathrm{E} \Big\{\prod_{k=1}^{a-c_1}x_{i_k,j_1} x_{i_k,j_2}\prod_{k=a-c_1+1}^{a} x_{i_{k},j_1}  \prod_{k=a+1}^{a+c_1}x_{i_{k},j_2}   \notag \\
&&\quad \quad \times  \prod_{\tilde{k}=1}^{a-c_2}  x_{\tilde{i}_{\tilde{k}},j_3} x_{\tilde{i}_{\tilde{k}},j_4} \prod_{\tilde{k}=a-c_2+1}^a x_{\tilde{i}_{\tilde{k}},j_3} \prod_{\tilde{k}=a+1}^{a+c_2}x_{\tilde{i}_{\tilde{k}},j_4}  \Big\}.\notag
\end{eqnarray}
Similarly to Section \ref{par:pfvarindep}, to evaluate $\mathrm{var}\{{\mathcal{U}}(a)\}$, we next examine the value of $Q(\mathbf{i},j_1,j_2,\tilde{\mathbf{i}},j_3,j_4)$ under  different cases. 


\vspace{0.5em}


When $\{\mathbf{i}\}\neq \{\tilde{\mathbf{i}}\}$, we show $\eqref{eq:summedadjustedutermvar}=0$, that is,  \textit{Claim 2} in Section \ref{par:pfvarindep} also holds here. To see this, we assume without loss of generality that an index $i\in \{\mathbf{i}\}$ and $i\not \in \{\tilde{\mathbf{i}}\}$. Then \eqref{eq:summedadjustedutermvar} takes one of the two following forms:
\begin{eqnarray*}
&& \eqref{eq:summedadjustedutermvar} = \mathrm{E}(x_{i,j_1})\times \mathrm{E}(\mathrm{all\ the \ remaining\ terms})  \quad (j_1=1,\ldots,p), \notag \\
&&\eqref{eq:summedadjustedutermvar} = \mathrm{E}(x_{i,j_1} x_{i,j_2})\times \mathrm{E}(\mathrm{all\ the \ remaining\ terms})  \quad (1\leq j_1 \neq j_2 \leq p). \notag 
\end{eqnarray*} Since  $\mathrm{E}(x_{i,j_1})=0$ and $\mathrm{E}(x_{i,j_1}x_{i,j_2})=0$ under $H_0$, we know $\eqref{eq:summedadjustedutermvar}=0$ when $\{\mathbf{i}\}\neq \{\tilde{\mathbf{i}}\}$.  It follows that
\begin{eqnarray}
	&&\quad \quad \sum_{\substack{1 \leq j_1 \neq j_2 \leq p;\\ 1 \leq j_3 \neq j_4 \leq p}} \ \sum_{ \substack{0\leq c_1,c_2\leq a;\\ \mathbf{i}\in \mathcal{P}(n,a+c_1);\\ \tilde{\mathbf{i}}\in  \mathcal{P}(n,a+c_2)} }F(c_1,c_2,a) Q(\mathbf{i},j_1, j_2,\tilde{\mathbf{i}},j_3,j_4)\mathbf{1}_{\{\{\mathbf{i}\}\neq \{\tilde{\mathbf{i}}\}\}} =0, \label{eq:varsum1noteqs1s2}
\end{eqnarray}
where $\mathbf{1}_{\{ \cdot \}}$ represents an  indicator function. 
 
When $\{\mathbf{i}\}= \{\tilde{\mathbf{i}}\}$, we know $c_1=c_2$ and we write $c_1=c_2=c$. 
 If $c=0$,
\begin{align*}
	Q(\mathbf{i},j_1, j_2,\tilde{\mathbf{i}},j_3,j_4)\mathbf{1}_{\{\{\mathbf{i}\}= \{\tilde{\mathbf{i}}\},c=0 \}}=\{\mathrm{E}(x_{i,j_1} x_{i,j_2} x_{i,j_3} x_{i,j_4} )\}^a.
\end{align*} Then we have
\begin{eqnarray}
	&&\sum_{\substack{1 \leq j_1 \neq j_2 \leq p;\\ 1 \leq j_3 \neq j_4 \leq p}} \ \sum_{ \substack{0\leq c\leq a;\\ \mathbf{i},\tilde{\mathbf{i}}\in \mathcal{P}(n,a+c)} } F(c,c,a)    Q(\mathbf{i},j_1, j_2,\tilde{\mathbf{i}},j_3,j_4)\mathbf{1}_{\{\{\mathbf{i}\}= \{\tilde{\mathbf{i}}\},c=0 \}} \label{eq:varsum1ceq0} \\
	&=& \frac{1}{(P^n_a)^2} \sum_{\mathbf{i} \in \mathcal{P}(n,a)} a! \sum_{ \substack{1\leq j_1\neq j_2 \leq p; \\1\leq j_3 \neq j_4 \leq p}} \{\mathrm{E}(x_{i,j_1} x_{i,j_2} x_{i,j_3} x_{i,j_4} )\}^a \notag \\
	&=&a!(P^n_a)^{-1}\sum_{ \substack{1\leq j_1\neq j_2 \leq p; \\1\leq j_3 \neq j_4 \leq p}} \{\mathrm{E}(x_{i,j_1} x_{i,j_2} x_{i,j_3} x_{i,j_4} )\}^a. \notag
\end{eqnarray}
If $c\geq 1$, for given $\mathbf{i},\tilde{\mathbf{i}}\in \mathcal{P}(n,a+c)$, we decompose the sets $\{\mathbf{i}\}$ and $\{\tilde{\mathbf{i}}\}$ into three disjoint sets respectively, defined as:
\begin{align*}
\{\mathbf{i}\}_{(1)}=\{i_1,\ldots,i_{a-c}\}, \ 	\{\mathbf{i}\}_{(2)}=\{i_{a-c+1},\ldots, i_a\},\ \{\mathbf{i}\}_{(3)}=\{i_{a+1},\ldots, i_{a+c}\}, \notag \\
\tilde{\{\mathbf{i}\}}_{(1)}=\{\tilde{i}_1,\ldots,\tilde{i}_{a-c}\}, \ 	\tilde{\{\mathbf{i}\}}_{(2)}=\{\tilde{i}_{a-c+1},\ldots, \tilde{i}_a\},\ \tilde{\{\mathbf{i}\}}_{(3)}=\{\tilde{i}_{a+1},\ldots, \tilde{i}_{a+c}\},
\end{align*} which satisfy that $\{\mathbf{i}\}=\cup_{l=1}^3 \{\mathbf{i}\}_{(l)}$ and $\tilde{\{\mathbf{i}\}}=\cup_{l=1}^3 \{\tilde{\mathbf{i}}\}_{(l)}$. The definitions are similarly used in Section \ref{par:pfvarindep}. We next examine the value of \eqref{eq:summedadjustedutermvar} by further discussing different cases.


\smallskip

\subparagraph*{Case 1} We consider the cases where $\{\mathbf{i}\}= \{\tilde{\mathbf{i}}\}$, $1\leq c\leq a-1$ and $\{\mathbf{i}\}_{(1)}= \{\tilde{\mathbf{i}}\}_{(1)}$. Then we have $\{\mathbf{i}\}_{(2)}\cup \{\mathbf{i}\}_{(3)}= \{\tilde{\mathbf{i}}\}_{(2)}\cup \{\tilde{\mathbf{i}}\}_{(3)}$. Note that here $\{\mathbf{i}\}_{(1)}= \{\tilde{\mathbf{i}}\}_{(1)}$ is assumed, and $ \{\mathbf{i}\}_{(2)}, \{\mathbf{i}\}_{(3)},  \{\tilde{\mathbf{i}}\}_{(2)}$ and  $\{\tilde{\mathbf{i}}\}_{(3)}$ are all nonempty as $ c\geq 1$. Similarly to \textit{Claim 3} in  Section \ref{par:pfvarindep}, we next prove that if  $\eqref{eq:summedadjustedutermvar}\neq 0$, one of the following two cases holds:
\begin{align}
&\{\mathbf{i}\}_{(3)}=\tilde{\{\mathbf{i}\}}_{(3)}, \{\mathbf{i}\}_{(2)}=\tilde{\{\mathbf{i}\}}_{(2)}, j_1=j_3, j_2=j_4; \label{eq:twocasevarc1} \\
	& \{\mathbf{i}\}_{(3)}=\tilde{\{\mathbf{i}\}}_{(2)}, \{\mathbf{i}\}_{(2)}=\tilde{\{\mathbf{i}\}}_{(3)}, j_1=j_4, j_2=j_3. \notag
\end{align}
We prove \eqref{eq:twocasevarc1} by contradiction. 



If $\{\mathbf{i}\}_{(2)} \cap \{\tilde{\mathbf{i}}\}_{(2)}\neq \emptyset$ and  $\{\mathbf{i}\}_{(2)}\cap \{\tilde{\mathbf{i}}\}_{(3)}\neq \emptyset$, 
it means that $\{\mathbf{i}\}_{(2)}$ intersects with both $\{\tilde{\mathbf{i}}\}_{(2)}$ and $\{\tilde{\mathbf{i}}\}_{(3)}$.  Suppose $i_1 \in \{\mathbf{i}\}_{(2)} \cap \{\tilde{\mathbf{i}}\}_{(2)}$ and $i_2 \in \{\mathbf{i}\}_{(2)}\cap \{\tilde{\mathbf{i}}\}_{(3)}$.
It follows that
\begin{align*}
	\eqref{eq:summedadjustedutermvar}= \mathrm{E}(x_{i_1,j_1}x_{i_1,j_3})\times \mathrm{E}(x_{i_2,j_1}x_{i_2,j_4})\times \mathrm{E}(\mathrm{all\ the \ remaining\ terms}).
\end{align*} As $j_3\neq j_4$, $\mathrm{E}(x_{i_1,j_1}x_{i_1,j_3})\times \mathrm{E}(x_{i_2,j_1}x_{i_2,j_4})=0$  under $H_0$. Therefore $\eqref{eq:summedadjustedutermvar}=0$. Similarly if $\{\mathbf{i}\}_{(3)} \cap \{\tilde{\mathbf{i}}\}_{(2)}\neq \emptyset$ and  $\{\mathbf{i}\}_{(3)}\cap \{\tilde{\mathbf{i}}\}_{(3)}\neq \emptyset$, we know $\eqref{eq:summedadjustedutermvar}=0$. The analysis shows  that when $\eqref{eq:summedadjustedutermvar}\neq 0$,  $\{\mathbf{i}\}_{(2)}$  only intersects with one of $\{\tilde{\mathbf{i}}\}_{(2)}$ and $\{\tilde{\mathbf{i}}\}_{(3)}$. Symmetrically, $\{\mathbf{i}\}_{(3)}$ only intersects with another one of $\{\tilde{\mathbf{i}}\}_{(2)}$ and $\{\tilde{\mathbf{i}}\}_{(3)}$. Since $|\{\mathbf{i}\}_{(2)}|=|\{\mathbf{i}\}_{(3)}|=|\{\tilde{\mathbf{i}}\}_{(2)}|=|\{\tilde{\mathbf{i}}\}_{(3)}|$, it remains to consider two cases
 $\{\{\mathbf{i}\}_{(2)}=\{\tilde{\mathbf{i}}\}_{(2)}\ \mathrm{and} \ \{\mathbf{i}\}_{(3)}=\{\tilde{\mathbf{i}}\}_{(3)} \}$ or $\{\{\mathbf{i}\}_{(2)}=\{\tilde{\mathbf{i}}\}_{(3)}\ \mathrm{and} \ \{\mathbf{i}\}_{(3)}=\{\tilde{\mathbf{i}}\}_{(2)}\}$. To obtain   \eqref{eq:twocasevarc1}, we next examine the two cases respectively.  
 
If $\{\mathbf{i}\}_{(2)}=\{\tilde{\mathbf{i}}\}_{(2)}\ \mathrm{and} \ \{\mathbf{i}\}_{(3)}=\{\tilde{\mathbf{i}}\}_{(3)}$, suppose $i_1\in \{\mathbf{i}\}_{(2)}$ and $i_2\in \{\mathbf{i}\}_{(3)}$. Then as  $\{\mathbf{i}\}_{(2)}\cap \{\mathbf{i}\}_{(3)}=\emptyset$, 
\begin{align*}
	\eqref{eq:summedadjustedutermvar}= \mathrm{E}(x_{i_1,j_1}x_{i_1,j_3})\times \mathrm{E}(x_{i_2,j_2}x_{i_2,j_4})\times \mathrm{E}(\mathrm{all\ the \ remaining\ terms}),
\end{align*} which is nonzero only when $j_1=j_3$ and $j_2=j_4$. 
Similarly,  if $\{\mathbf{i}\}_{(2)}=\{\tilde{\mathbf{i}}\}_{(3)}\ \mathrm{and} \ \{\mathbf{i}\}_{(3)}=\{\tilde{\mathbf{i}}\}_{(2)}$, $\eqref{eq:summedadjustedutermvar}\neq 0$ only when $j_1=j_4$ and $j_2=j_3$.  
In summary, if  $\eqref{eq:summedadjustedutermvar}\neq 0$, \eqref{eq:twocasevarc1} is obtained, and
\begin{align*}
	& ~Q(\mathbf{i},j_1, j_2,\tilde{\mathbf{i}},j_3,j_4)\times \mathbf{1}_{\{ \{\mathbf{i}\}=\{\tilde{\mathbf{i}}\}, \{\mathbf{i}\}_{(1)}=\{\tilde{\mathbf{i}}\}_{(1)},1\leq c\leq a-1 \} }\notag \\
	=&~Q(\mathbf{i},j_1, j_2,\tilde{\mathbf{i}},j_3,j_4) \mathbf{1}_{\{\{\mathbf{i}\}_{(1)}=\{\tilde{\mathbf{i}}\}_{(1)}, 1\leq c\leq a-1 \} } \notag \\
	&~\times  \Biggr(\mathbf{1}_{\Big\{\substack{ \{\mathbf{i}\}_{(2)}=\{\tilde{\mathbf{i}}\}_{(2)},\, j_1=j_3,\\\{\mathbf{i}\}_{(3)}=\{\tilde{\mathbf{i}}\}_{(3)},\, j_2=j_4} \Big\}} + \mathbf{1}_{\Big\{\substack{\{\mathbf{i}\}_{(2)}=\{\tilde{\mathbf{i}}\}_{(3)},\, j_1=j_4,  \\ \{\mathbf{i}\}_{(3)}=\{\tilde{\mathbf{i}}\}_{(2)},\, j_2=j_3}\Big\}}\Biggr). \notag 
\end{align*} 
In addition, under the two cases in  \eqref{eq:twocasevarc1}, we have $Q(\mathbf{i},j_1, j_2,\tilde{\mathbf{i}},j_3,j_4)=\{\mathrm{E}(x_{i,j_1}^2 x_{i,j_2}^2)\}^{a-c} \{\mathrm{E}(x_{i,j_1}^2) \mathrm{E}(x_{i,j_2}^2)\}^c$.  
Therefore,
\begin{eqnarray}\label{eq:varsum1vars1stareqs2star}
	&&\sum_{\substack{1 \leq j_1 \neq j_2 \leq p;\\ 1 \leq j_3 \neq j_4 \leq p}} \ \sum_{ \substack{0\leq c\leq a;\\ \mathbf{i},\tilde{\mathbf{i}}\in \mathcal{P}(n,a+c)} } F(c,c,a)    Q(\mathbf{i},j_1, j_2,\tilde{\mathbf{i}},j_3,j_4)\mathbf{1}_{\Big\{\substack{ \{\mathbf{i}\}=\{\tilde{\mathbf{i}}\},\\ \{\mathbf{i}\}_{(1)}=\{\tilde{\mathbf{i}}\}_{(1)},\\1\leq c\leq a-1} \Big\} }  \\
&=&\sum_{\substack{1\leq c\leq a-1;\\ \mathbf{i}\in \mathcal{P}(n,a+c);\\ 1\leq j_1 \neq j_2 \leq p}}   \frac{\binom{a}{c}^22(a-c)!c!c!}{(P^n_{a+c})^2}\{\mathrm{E}(x_{i,j_1}^2 x_{i,j_2}^2)\}^{a-c} \{\mathrm{E}(x_{i,j_1}^2) \mathrm{E}(x_{i,j_2}^2)\}^c.\notag \\
     &=&\sum_{c=1}^{a-1} O(p^2n^{-(a+c)}), \notag
\end{eqnarray}where the last equation uses Condition \ref{cond:finitemomt}.
\medskip

\subparagraph*{Case 2} We consider the cases when $\{\mathbf{i}\}=\{\tilde{\mathbf{i}}\}$, $1\leq c\leq a-1$, $\{\mathbf{i}\}_{(1)}\neq \{\tilde{\mathbf{i}}\}_{(1)}$ and $\{\mathbf{i}\}_{(1)}\cap \{\tilde{\mathbf{i}}\}_{(1)}\neq \emptyset$. Suppose that there exists an index $i_1 \in \{\mathbf{i}\}_{(1)}\cap \{\tilde{\mathbf{i}}\}_{(1)}$. Since $\{\mathbf{i}\}_{(1)}\neq \{\tilde{\mathbf{i}}\}_{(1)}$ and $|\{\mathbf{i}\}_{(1)}|=|\{\tilde{\mathbf{i}}\}_{(1)}|$, there exists another  index $i_2\in \{\mathbf{i}\}_{(1)}$ and $i_2\not \in \{\tilde{\mathbf{i}}\}_{(1)}$. As $\{\mathbf{i}\}=\{\tilde{\mathbf{i}}\}$, we know  $i_2\in \{\tilde{\mathbf{i}}\}_{(2)}\cup \{\tilde{\mathbf{i}}\}_{(3)}$. Without loss of generality, we assume $i_2\in \{\tilde{\mathbf{i}}\}_{(2)}$, then 
\begin{eqnarray}
\quad\ && \eqref{eq:summedadjustedutermvar}=\mathrm{E}(x_{i_1,j_1}x_{i_1,j_2}x_{i_1,j_3}x_{i_1,j_4}) \mathrm{E}(x_{i_2,j_1}x_{i_2,j_2}x_{i_2,j_3}) \mathrm{E}(\mathrm{other\ terms}).\label{eq:var1case2}
\end{eqnarray} 
As  $j_1\neq j_2$ and $j_3\neq j_4$ in summation, it suffices to discuss four sub-cases $\{j_1=j_3 \mathrm{\ and \ } j_2=j_4\}$, $\{j_1=j_4 \mathrm{\ and \ } j_2=j_3\}$, $\{j_1 \neq j_3 \mathrm{\ and \ } j_1 \neq j_4\}$ and $\{j_2 \neq j_3 \mathrm{\ and \ } j_2 \neq j_4\}$ under Case 2.  

 \smallskip
\textit{Case 2.1} If $j_1=j_3$ and $j_2=j_4$, \eqref{eq:var1case2} gives
\begin{eqnarray*}
	\eqref{eq:summedadjustedutermvar}=\mathrm{E}(x_{i_1,j_1}^2x_{i_1,j_2}^2)\times \mathrm{E}(x_{i_2,j_1}^2x_{i_2,j_2}) \times \mathrm{E}(\mathrm{all\ the \ remaining\ terms}). 
\end{eqnarray*} 
{When $x_{i,j}$'s are independent as in Section \ref{par:pfvarindep}, we know $\mathrm{E}(x_{i_2,j_1}^2x_{i_2,j_2})=\mathrm{E}(x_{i_2,j_1}^2)\mathrm{E}(x_{i_2,j_2})=0$ and thus $\eqref{eq:summedadjustedutermvar}=0$. Alternatively, under Condition \ref{cond:alphamixing}, \eqref{eq:summedadjustedutermvar} may no longer be 0 due to the dependence of $x_{i,j}$'s.  But as discussed at the beginning of Section \ref{par:complpflm1}, we expect that $x_{i,j_1}$ and $x_{i,j_2}$ are ``asymptotically independent'' as $|j_1-j_2|$ increases, and thus we expect that  \eqref{eq:summedadjustedutermvar} is close to 0 when $|j_1-j_2|$ is large. 
To quantitatively evaluate  \eqref{eq:summedadjustedutermvar} based on $|j_1-j_2|$,  we  introduce a threshold $K_0$ below, and discuss the value of  \eqref{eq:summedadjustedutermvar} when $|j_1-j_2|>K_0$ and $|j_1-j_2|\leq K_0$,  respectively.} 

Specifically, given $\delta$ in Condition \ref{cond:alphamixing} and  positive constants $\mu$ and $\epsilon$, we  define
\begin{eqnarray} \label{eq:thresholddK}
K_0 = {-(2+\epsilon)(4+\mu) (\log p)}/{(\epsilon \log \delta)}.
\end{eqnarray} When $|j_1-j_2|>K_0$,  by Conditions \ref{cond:finitemomt} and  \ref{cond:alphamixing},
 we have 
\begin{align*}
|\eqref{eq:summedadjustedutermvar}|\leq C\times |\mathrm{E}(x_{i_2,j_1}^2x_{i_2,j_2})|&= C\times |\mathrm{cov}(x_{i_2,j_1}^2, x_{i_2,j_2})|\notag \\
& \leq C\delta^{\frac{K_0\epsilon}{2+\epsilon}}=O(1)p^{-(4+\mu)},	
\end{align*} where $|\mathrm{cov}(x_{i_2,j_1}^2, x_{i_2,j_2})|\leq C\delta^{\frac{K_0\epsilon}{2+\epsilon}}$ holds by the $\alpha$-mixing inequality in Lemma \ref{lm:mixingineq}. When $|j_1-j_2|\leq K_0$, by the uniform boundedness of moments from Condition \ref{cond:finitemomt}, we have $\eqref{eq:summedadjustedutermvar}=O(1)$. To summarize, 
we define an event  $S_{nem}=\{\{\mathbf{i}\}=\{\tilde{\mathbf{i}}\}, 1\leq c\leq a-1, \{\mathbf{i}\}_{(1)}\neq \{\tilde{\mathbf{i}}\}_{(1)},\{\mathbf{i}\}_{(1)}\cap \{\tilde{\mathbf{i}}\}_{(1)} \neq \emptyset\}$. Then
\begin{align*}
&~ Q(\mathbf{i},j_1, j_2,\tilde{\mathbf{i}},j_3,j_4)\times \mathbf{1}_{\{S_{nem}, j_1=j_3, j_2=j_4 \}}\notag \\
=&~	Q(\mathbf{i},j_1, j_2,\tilde{\mathbf{i}},j_3,j_4) \times \Biggr(\mathbf{1}_{\Big\{\substack{S_{nem}, j_1=j_3, j_2=j_4,\\ |j_1-j_2|>K_0 }\Big\}}+ \mathbf{1}_{\Big\{\substack{S_{nem}, j_1=j_3, j_2=j_4,\\ |j_1-j_2|\leq K_0 }\Big\}} \Biggr).\notag 
\end{align*} 
The analysis above gives $Q(\mathbf{i},j_1, j_2,\tilde{\mathbf{i}},j_3,j_4)\mathbf{1}_{\{S_{nem}, j_1=j_3, j_2=j_4, |j_1-j_2|>K_0 \}}=O(1)p^{-(4+\mu)}$ and $Q(\mathbf{i},j_1, j_2,\tilde{\mathbf{i}},j_3,j_4)\mathbf{1}_{\{S_{nem}, j_1=j_3, j_2=j_4, |j_1-j_2|\leq K_0 \}}=O(1)$, respectively.  
Moreover, the total number of $(j_1,j_2)$ pairs satisfying $|j_1-j_2|\leq K_0$ and $|j_1-j_2|> K_0$ are $O(p^2)$ and $O(pK_0)$, respectively.  Therefore,
\begin{eqnarray}
	&& \quad  \Biggr|\sum_{\substack{1 \leq j_1 \neq j_2 \leq p;\\ 1 \leq j_3 \neq j_4 \leq p}} \ \sum_{ \substack{0\leq c\leq a;\\ \mathbf{i},\tilde{\mathbf{i}}\in \mathcal{P}(n,a+c)} } F(c,c,a)    Q(\mathbf{i},j_1, j_2,\tilde{\mathbf{i}},j_3,j_4)\mathbf{1}_{\{S_{nem}, j_1=j_3, j_2=j_4 \}} \Biggr|\label{eq:match22smallcase2}  \\
	&\leq &  \sum_{\substack{1 \leq j_1 \neq j_2 \leq p;\\ 1 \leq j_3 \neq j_4 \leq p}} \ \sum_{ \substack{0\leq c\leq a;\\ \mathbf{i},\tilde{\mathbf{i}}\in \mathcal{P}(n,a+c)} } \Big|F(c,c,a)\Big|\times \mathbf{1}_{\{S_{nem}, j_1=j_3, j_2=j_4\}} \notag \\
	&& \quad \quad \quad \quad \quad\quad\quad\quad  \times \Big\{O(p^{-(4+\mu)}) \mathbf{1}_{\{  |j_1-j_2|> K_0 \}} +C\times  \mathbf{1}_{\{  |j_1-j_2|\leq K_0 \}}\Big\} \notag\\
	&=& \sum_{c=1}^{a-1} n^{-(a+c)}\Big\{O(1) p^2 p^{-(4+\mu)}+O(1) pK_0 \Big\}= o(p^2n^{-a}). \notag
\end{eqnarray}

\textit{Case 2.2} If $j_1=j_4$ and $j_2=j_3$, similarly to \textit{Case 2.1}, we have
\begin{eqnarray}\label{eq:match22smallcase2sim}
\quad &\quad &\Biggr|\sum_{\substack{1 \leq j_1 \neq j_2 \leq p;\\ 1 \leq j_3 \neq j_4 \leq p}}  \sum_{ \substack{0\leq c\leq a;\\ \mathbf{i},\tilde{\mathbf{i}}\in \mathcal{P}(n,a+c)} } F(c,c,a)    Q(\mathbf{i},j_1, j_2,\tilde{\mathbf{i}},j_3,j_4)\mathbf{1}_{\{S_{nem}, j_1=j_4, j_2=j_3 \}}\Biggr|\\
 &=&	o(p^2n^{-a}).\notag 
\end{eqnarray}

\textit{Case 2.3}  We discuss the cases where $j_1\neq j_3$ and $j_1 \neq j_4$. If $x_{i,j}$'s are independent as in Section \ref{par:pfvarindep}, we know $\mathrm{E}(x_{i_1,j_1}x_{i_1,j_2}x_{i_1,j_3}x_{i_1,j_4})=\mathrm{E}(x_{i_1,j_1})\mathrm{E}(\text{other terms})=0$; thus by \eqref{eq:var1case2},  $\eqref{eq:summedadjustedutermvar}=0$ under this setting. Similarly to \textit{Case 2.1}, under Condition \ref{cond:alphamixing}, \eqref{eq:summedadjustedutermvar} may be no longer 0, and we will discuss the value of  \eqref{eq:summedadjustedutermvar} using the threshold $K_0$ in \eqref{eq:thresholddK}.    

To evaluate \eqref{eq:summedadjustedutermvar}, by \eqref{eq:var1case2}, we examine $\mathrm{E}(x_{i_1,j_1}x_{i_1,j_2}x_{i_1,j_3}x_{i_1,j_4})$. Let $(\tilde{j}_1,\tilde{j}_2,\tilde{j}_3,\tilde{j}_4)$  be the ordered version of $(j_1,j_2,j_3,j_4)$ satisfying $\tilde{j}_1\leq \tilde{j}_2\leq \tilde{j}_3 \leq \tilde{j}_4$, then $\mathrm{E}(x_{i_1,j_1}x_{i_1,j_2}x_{i_1,j_3}x_{i_1,j_4})=\mathrm{E}(x_{i_1,\tilde{j}_1}x_{i_1,\tilde{j}_2}x_{i_1,\tilde{j}_3}x_{i_1,\tilde{j}_4})$. 
Under the considered cases where  $j_1\neq j_3$ and $j_1 \neq j_4$,  at least  one of the two equations, $\mathrm{E}(x_{i_1,\tilde{j}_1}x_{i_1,\tilde{j}_2})=0$ and  $\mathrm{E}(x_{i_1,\tilde{j}_3}x_{i_1,\tilde{j}_4})=0$, holds. It follows that $\mathrm{E}(x_{i_1,\tilde{j}_1}x_{i_1,\tilde{j}_2}x_{i_1,\tilde{j}_3}x_{i_1,\tilde{j}_4})=\mathrm{cov}( x_{i_1,\tilde{j}_1}x_{i_1,\tilde{j}_2} \ , \  x_{i_1,\tilde{j}_3}x_{i_1,\tilde{j}_4})$. We thus can write 
 \begin{align}
	  \quad	 | \mathrm{E}(x_{i_1,j_1}x_{i_1,j_2}x_{i_1,j_3}x_{i_1,j_4})|=&~ |\mathrm{E}(x_{i_1,\tilde{j}_1}x_{i_1,\tilde{j}_2}x_{i_1,\tilde{j}_3}x_{i_1,\tilde{j}_4})| \label{eq:fourthintermomtbound} \\
	   	 =&~ |\mathrm{cov}( x_{i_1,\tilde{j}_1}x_{i_1,\tilde{j}_2} \ , \  x_{i_1,\tilde{j}_3}x_{i_1,\tilde{j}_4})| \notag \\
	   	 =&~ |\mathrm{cov}( x_{i_1,\tilde{j}_1} \ ,\ x_{i_1,\tilde{j}_2} x_{i_1,\tilde{j}_3}x_{i_1,\tilde{j}_4}) | \notag \\
	   	 =&~ |\mathrm{cov}( x_{i_1,\tilde{j}_1} x_{i_1,\tilde{j}_2} x_{i_1,\tilde{j}_3} \ , \ x_{i_1,\tilde{j}_4})|. \notag 
\end{align}

We next discuss the value of \eqref{eq:fourthintermomtbound} based on the
the maximum distance between the  indexes in $(\tilde{j}_1,\tilde{j}_2,\tilde{j}_3,\tilde{j}_4)$, which is defined as  
\begin{eqnarray}
	\kappa_m=\max\{|\tilde{j}_2-\tilde{j}_1|, |\tilde{j}_3-\tilde{j}_2|,|\tilde{j}_4-\tilde{j}_3|\}. \label{eq:kappamaxdef}
\end{eqnarray}We evaluate \eqref{eq:fourthintermomtbound} when $\kappa_m> K_0$ and $\kappa_m\leq K_0$, respectively. 
First, if $\kappa_m> K_0$, by $\mathrm{E}(\mathbf{x})=\mathbf{0}$, Conditions \ref{cond:finitemomt},   \ref{cond:alphamixing}, and Lemma \ref{lm:mixingineq}, we have $\eqref{eq:fourthintermomtbound}\leq C\delta^{\frac{K_0\epsilon}{2+\epsilon}}=O(p^{-(4+\mu)}).$ If $\kappa_m\leq K_0$, by Condition \ref{cond:finitemomt}, $\eqref{eq:fourthintermomtbound}=O(1).$ 
It follows that  $Q(\mathbf{i},j_1, j_2,\tilde{\mathbf{i}},j_3,j_4)\mathbf{1}_{\{S_{nem}, j_1\neq j_3, j_1\neq j_4, \kappa_m >K_0 \}}=O(p^{-(4+\mu)}),$ and $Q(\mathbf{i},j_1, j_2,\tilde{\mathbf{i}},j_3,j_4)\mathbf{1}_{\{S_{nem}, j_1\neq j_3, j_1\neq j_4, \kappa_m \leq K_0 \}}=O(1),$ where the event $S_{nem}$ is defined in \textit{Case 2.1}.  
Note that the total number of $(j_1,j_2,j_3,j_4)$ tuples satisfying $\kappa_m>K_0$ and $\kappa_m\leq K_0$ are $O(p^4)$ and $O(pK_0^3)$, respectively. Thus
\begin{eqnarray}
\quad\ && \Big|\sum_{\substack{1 \leq j_1 \neq j_2 \leq p;\\ 1 \leq j_3 \neq j_4 \leq p}}  \sum_{ \substack{0\leq c\leq a;\\ \mathbf{i},\tilde{\mathbf{i}}\in \mathcal{P}(n,a+c)} } F(c,c,a)  Q(\mathbf{i},j_1, j_2,\tilde{\mathbf{i}},j_3,j_4)\mathbf{1}_{\{S_{nem},j_1\neq j_3 ,j_1\neq j_4\}} \Big|\label{eq:match22smallcase3dif}  \\
	&\leq &  \sum_{\substack{1 \leq j_1 \neq j_2 \leq p;\\ 1 \leq j_3 \neq j_4 \leq p}}  \sum_{ \substack{0\leq c\leq a;\\ \mathbf{i},\tilde{\mathbf{i}}\in \mathcal{P}(n,a+c)} } |F(c,c,a)|\times \mathbf{1}_{\{ S_{nem}, j_1\neq j_3 ,j_1\neq j_4\}} \notag \\
	&&\quad \quad \quad \quad \quad \quad \quad \quad \quad  \times \Big[O(p^{-(4+\mu)})\mathbf{1}_{\{\kappa_m>K_0\}}+C\times  \mathbf{1}_{\{  \kappa_m \leq K_0  \}}\Big]\notag\\
	&=& \sum_{c=1}^{a-1} n^{-(a+c)}\{ p^2O( p^{-(4+\mu)})+O(1) pK_0^3 \}= o(p^2n^{-a}). \notag
\end{eqnarray} 

\textit{Case 2.4} If $j_2\neq j_3$ and $j_2 \neq j_4$,  similarly to \textit{Case 2.3}, we have 
\begin{eqnarray} \label{eq:match22smallcase3difsim}
	\quad\ &\quad &\Big|\sum_{\substack{1 \leq j_1 \neq j_2 \leq p;\\ 1 \leq j_3 \neq j_4 \leq p}}  \sum_{ \substack{0\leq c\leq a;\\ \mathbf{i},\tilde{\mathbf{i}}\in \mathcal{P}(n,a+c)} }  F(c,c,a)  Q(\mathbf{i},j_1, j_2,\tilde{\mathbf{i}},j_3,j_4)\mathbf{1}_{\{S_{nem},j_2\neq j_3 ,j_2\neq j_4\}} \Big| \\
	&=&o(p^2n^{-a}). \notag
\end{eqnarray}
By \eqref{eq:match22smallcase2}, \eqref{eq:match22smallcase2sim}, \eqref{eq:match22smallcase3dif},  \eqref{eq:match22smallcase3difsim}, and the definition of $S_{nem}$, we obtain
\begin{eqnarray}
	&&\sum_{\substack{1 \leq j_1 \neq j_2 \leq p;\\ 1 \leq j_3 \neq j_4 \leq p}}  \sum_{ \substack{0\leq c\leq a;\\ \mathbf{i},\tilde{\mathbf{i}}\in \mathcal{P}(n,a+c)} }  F(c,c,a)  Q(\mathbf{i},j_1, j_2,\tilde{\mathbf{i}},j_3,j_4)\label{eq:smsummasmallord}\\
	&&\quad \quad  \times \mathbf{1}_{\{\{\mathbf{i}\}=\{\tilde{\mathbf{i}}\}, 1\leq c\leq a-1, \{\mathbf{i}\}_{(1)}\neq \{\tilde{\mathbf{i}}\}_{(1)},\{\mathbf{i}\}_{(1)}\cap \{\tilde{\mathbf{i}}\}_{(1)} \neq \emptyset\}}=o(p^2n^{-a}). \notag
\end{eqnarray}

\smallskip
\subparagraph*{Case 3} We consider $\{\mathbf{i}\}=\{\tilde{\mathbf{i}}\}$, $1\leq c\leq a-1$, and $\{\mathbf{i}\}_{(1)}\cap \{\tilde{\mathbf{i}}\}_{(1)}=\emptyset$. Here $\{\mathbf{i}\}_{(1)}$ and $\{\tilde{\mathbf{i}}\}_{(1)}$ are not empty as $c\leq a-1$. Suppose there exist  $i_1\in \{\mathbf{i}\}_{(1)}$ and $i_2 \in \{\tilde{\mathbf{i}}\}_{(1)}$ with $i_1\neq i_2$. Since $\{\mathbf{i}\}=\{\tilde{\mathbf{i}}\}$ and $\{\mathbf{i}\}_{(1)}\cap \{\tilde{\mathbf{i}}\}_{(1)}=\emptyset$, we know $i_1\in  \{\tilde{\mathbf{i}}\}_{(2)}\cup  \{\tilde{\mathbf{i}}\}_{(3)}$ and $i_2 \in \{\mathbf{i}\}_{(2)}\cup \{\mathbf{i}\}_{(3)}$.  Without loss of generality, we assume $i_1\in \{\tilde{\mathbf{i}}\}_{(2)}$ and $i_2 \in \{\mathbf{i}\}_{(2)}$, then
\begin{align*}
	\eqref{eq:summedadjustedutermvar}=\mathrm{E}(x_{i_1,j_1}x_{i_1,j_2}x_{i_1,j_3}) \times \mathrm{E}(x_{i_2,j_3}x_{i_2,j_4}x_{i_2,j_1})  \times \mathrm{E}(\mathrm{other\ terms}).
\end{align*} 
To evaluate \eqref{eq:summedadjustedutermvar}, we examine $\mathrm{E}(x_{i_1,j_1}x_{i_1,j_2}x_{i_1,j_3})  \mathrm{E}(x_{i_2,j_3}x_{i_2,j_4}x_{i_2,j_1}).$ As  $\mathrm{E}(\mathbf{x})=\mathbf{0}$, we can write
\begin{align*}
	\mathrm{E}(x_{i_1,j_1}x_{i_1,j_2}x_{i_1,j_3})=\mathrm{cov}(x_{i_1,j_1}\, ,\, x_{i_1,j_2}x_{i_1,j_3} )=&~\mathrm{cov}(x_{i_1,j_2}\, ,\, x_{i_1,j_1}x_{i_1,j_3} )\notag \\
	=&~\mathrm{cov}(x_{i_1,j_3}\, ,\, x_{i_1,j_1}x_{i_1,j_2} ),
\end{align*} and similarly,
\begin{align*}
\mathrm{E}(x_{i_2,j_3}x_{i_2,j_4}x_{i_2,j_1})=\mathrm{cov}(x_{i_2,j_3}\, ,\, x_{i_2,j_4}x_{i_2,j_1} ) =&~\mathrm{cov}(x_{i_2,j_4}\, , \allowbreak \, x_{i_2,j_3}x_{i_2,j_1} )\notag \\
=&~\mathrm{cov}(x_{i_2,j_1}\, ,\, x_{i_2,j_3}x_{i_2,j_4} ).	
\end{align*}
Recall $\kappa_m$  in \eqref{eq:kappamaxdef} and $K_0$  in \eqref{eq:thresholddK}. If $\kappa_m> K_0$,  
by  Conditions \ref{cond:finitemomt} and  \ref{cond:alphamixing}, and Lemma \ref{lm:mixingineq}, we have
\begin{align} \label{eq:threeorderintervar}
	\Big|\mathrm{E}(x_{i_1,j_1}x_{i_1,j_2}x_{i_1,j_3}) \mathrm{E}(x_{i_2,j_3}x_{i_2,j_4}x_{i_2,j_1})\Big|\leq C\delta^{\frac{K_0\epsilon}{2+\epsilon}}=O(1)p^{-(4+\mu)}.
\end{align}If $\kappa_m\leq K_0$, by Condition \ref{cond:finitemomt}, $\mathrm{E}(x_{i_1,j_1}x_{i_1,j_2}x_{i_1,j_3}) \mathrm{E}(x_{i_2,j_3}x_{i_2,j_4}x_{i_2,j_1})=O(1)$. 
Note that the total number of $(j_1,j_2,j_3,j_4)$ tuples satisfying $\kappa_m>K_0$ and $\kappa_m\leq K_0$ are $O(p^4)$ and $O(pK_0^3)$, respectively. Therefore,
\begin{eqnarray}
\quad && \Biggr|\sum_{\substack{1 \leq j_1 \neq j_2 \leq p;\\ 1 \leq j_3 \neq j_4 \leq p}}  \sum_{ \substack{0\leq c\leq a;\\ \mathbf{i},\tilde{\mathbf{i}}\in \mathcal{P}(n,a+c)} }  F(c,c,a)  Q(\mathbf{i},j_1, j_2,\tilde{\mathbf{i}},j_3,j_4) \mathbf{1}_{ \Big\{\substack{\{\mathbf{i}\}=\{\tilde{\mathbf{i}}\};\\ 1 \leq c\leq a-1;\\ \{\mathbf{i}\}_{(1)} \cap \{\tilde{\mathbf{i}}\}_{(1)}=\emptyset} \Big\} }\Biggr| \label{eq:smallcasenointer}  \\
	&\leq &  \sum_{\substack{1 \leq j_1 \neq j_2 \leq p;\\ 1 \leq j_3 \neq j_4 \leq p}}  \sum_{ \substack{0\leq c\leq a;\\ \mathbf{i},\tilde{\mathbf{i}}\in \mathcal{P}(n,a+c)} } \Big|F(c,c,a)\Big|\times \mathbf{1}_{ \Big\{\substack{\{\mathbf{i}\}=\{\tilde{\mathbf{i}}\}; 1 \leq c\leq a-1;\\ \{\mathbf{i}\}_{(1)} \cap \{\tilde{\mathbf{i}}\}_{(1)}=\emptyset} \Big\} } \notag \\
	&& \quad \quad \quad \quad  \times  \Big[Cp^{-(4+\mu)}\mathbf{1}_{\{ \kappa_{m}> K_0 \}} +C \mathbf{1}_{\{ \kappa_{m}\leq K_0 \}}\Big] \notag\\
	&=& \sum_{c=1}^{a-1} n^{-(a+c)}\{O(1) p^4 p^{-(4+\mu)}+O(1) pK_0^3 \}\notag = o(p^2n^{-a}). \notag
\end{eqnarray}

\subparagraph*{Case 4} When $\{\mathbf{i}\}=\{\tilde{\mathbf{i}}\}$ and $c=a$, we know $\{\mathbf{i}\}_{(1)}=\{\tilde{\mathbf{i}}\}_{(1)}=\emptyset$ and $\{\mathbf{i}\}_{(2)}\cup \{\mathbf{i}\}_{(3)}=\{\tilde{\mathbf{i}}\}_{(2)}\cup \{\tilde{\mathbf{i}}\}_{(3)}$. Then similarly \textit{Case 1}, we have
\begin{eqnarray}\label{eq:covcase4ca}
\quad && \Big|\sum_{\substack{1 \leq j_1 \neq j_2 \leq p;\\ 1 \leq j_3 \neq j_4 \leq p}}  \sum_{ \substack{\mathbf{i},\tilde{\mathbf{i}}\in \mathcal{P}(n,a+c)} }  F(c,c,a)  Q(\mathbf{i},j_1, j_2,\tilde{\mathbf{i}},j_3,j_4)\mathbf{1}_{\{\{\mathbf{i}\}=\{\tilde{\mathbf{i}}\}, c=a \}}\Big| \\
	&=&o(p^2n^{-a}).\notag
\end{eqnarray}

In summary, by \eqref{eq:varesquaredef},  \eqref{eq:varsum1noteqs1s2}--\eqref{eq:varsum1vars1stareqs2star},  \eqref{eq:smsummasmallord}, \eqref{eq:smallcasenointer}, and \eqref{eq:covcase4ca},
	 \begin{align}
	 		\mathrm{var}\{\mathcal{U}(a)\}= \frac{a!}{P^n_a}  \sum_{ \substack{1\leq j_1\neq j_2 \leq p \\1\leq j_3 \neq j_4 \leq p}} \{\mathrm{E}(x_{i,j_1} x_{i,j_2} x_{i,j_3} x_{i,j_4} )\}^a + o(p^2n^{-a}). \label{eq:varoduaoncov}
	 \end{align} 
Note that we assume $\mathrm{E}(\mathbf{x})=\mathbf{0}$. For the general case with $\mathrm{E}(\mathbf{x})=\boldsymbol{\mu}$, by Proposition \ref{prop:locinvariance}, it is equivalent to replace $x_{i,j}$ by  $x_{i,j}-\mu_j$ in \eqref{eq:varoduaoncov}.

We next show that $\mathrm{var}\{\tilde{\mathcal{U}}(a)\}=\eqref{eq:varsum1ceq0}$ and $\mathrm{var}[\tilde{\mathcal{U}}^*(a)]=o(p^2n^{-a})$. First note that $\mathrm{E}\{\tilde{\mathcal{U}}(a)\}=\mathrm{E}\{\tilde{\mathcal{U}}^*(a)\}=0$ under $H_0$ as $\mathrm{E}(\mathbf{x})=\mathbf{0}$. Then it suffices to show $\mathrm{E}\{\{\tilde{\mathcal{U}}(a)\}^2\}=\eqref{eq:varsum1ceq0}$ and $\mathrm{E}\{\{\tilde{\mathcal{U}}^*(a)\}^2\}=o(p^2n^{-a})$. By the definition of  $\tilde{\mathcal{U}}(a)$  in \eqref{eq:originleadingterm}, we know
\begin{eqnarray}\label{eq:utildeholdsvar}
	&&\quad \quad \mathrm{E}\{\tilde{\mathcal{U}}^2(a) \} \\
	&=&\sum_{ \substack{1\leq j_1\neq j_2 \leq p \\1\leq j_3 \neq j_4 \leq p}}\sum_{ \substack{0\leq c_1,c_2\leq a;\\ \mathbf{i}\in \mathcal{P}(n,a+c_1);\\ \tilde{\mathbf{i}}\in  \mathcal{P}(n,a+c_2)} }F(c_1,c_2,a)  Q(\mathbf{i},j_1, j_2,\tilde{\mathbf{i}},j_3,j_4)\times \mathbf{1}_{\{c_1=c_2=0\}}. \notag
\end{eqnarray} Therefore, $\mathrm{E}\{\tilde{\mathcal{U}}^2(a)\}=\eqref{eq:varsum1ceq0}$ from previous discussion. Moreover, as $\tilde{\mathcal{U}}^*(a)=\mathcal{U}(a)-\tilde{\mathcal{U}}(a)$, we know
\begin{align}\label{eq:utildestarlemmab1}
	\tilde{\mathcal{U}}^*(a)
	=& \sum_{c=0}^a \mathbf{1}_{\{c\geq 1\}}\sum_{1 \leq j_1 \neq j_2 \leq p} (-1)^c \binom{a}{c} \frac{1}{P^n_{a+c}}  \sum_{\mathbf{i}\in \mathcal{P}(n,a+c)} \\
	&\quad\times \prod_{k=1}^{a-c} (x_{i_k,j_1} x_{i_k,j_2})  \prod_{k=a-c+1}^{a} x_{i_{k},j_1}  \prod_{k=a+1}^{a+c}x_{i_{k},j_2}.  \notag
\end{align} 
It follows that
\begin{eqnarray}
\quad &&\quad \mathrm{E}[\{\tilde{\mathcal{U}}^*(a)\}^2]\label{eq:utildstordervar} \\
&=&\sum_{ \substack{1\leq j_1\neq j_2 \leq p \\1\leq j_3 \neq j_4 \leq p}}\sum_{ \substack{0\leq c_1,c_2\leq a;\\ \mathbf{i}\in \mathcal{P}(n,a+c_1);\\ \tilde{\mathbf{i}}\in  \mathcal{P}(n,a+c_2)} }F(c_1,c_2,a)  Q(\mathbf{i},j_1, j_2,\tilde{\mathbf{i}},j_3,j_4)\times \mathbf{1}_{\left\{\substack{c_1\geq 1,  c_2\geq 1} \right\}}. \notag
\end{eqnarray} Also by  previous discussion, we know $\mathrm{E}[\{\tilde{\mathcal{U}}^*(a)\}^2]=o(p^2n^{-a})$.

To finish the proof of Lemma \ref{lm:varianceorder}, it remains to show $\mathrm{var}\{\tilde{\mathcal{U}}(a)\}=\eqref{eq:varsum1ceq0}=\Theta(p^2n^{-a})$, and it suffices to prove
	\begin{eqnarray}
		\sum_{ \substack{1\leq j_1\neq j_2 \leq p; \\1\leq j_3 \neq j_4 \leq p}} \{\mathrm{E}(x_{i,j_1} x_{i,j_2} x_{i,j_3} x_{i,j_4} )\}^a=\Theta(p^2). \label{eq:sum4jordervar}
	\end{eqnarray} 
To prove \eqref{eq:sum4jordervar}, we examine 
$\mathrm{E}(x_{i,j_1} x_{i,j_2} x_{i,j_3} x_{i,j_4} )$. 	Similarly to \textit{Case 2} above, as  $j_1\neq j_2$ and $j_3\neq j_4$ in summation, it suffices to discuss four cases $\{j_1=j_3 \mathrm{\ and \ } j_2=j_4\}$, $\{j_1=j_4 \mathrm{\ and \ } j_2=j_3\}$, $\{j_1 \neq j_3 \mathrm{\ and \ } j_1 \neq j_4\}$, and $\{j_2 \neq j_3 \mathrm{\ and \ } j_2 \neq j_4\}$. 	

If $j_1=j_3$, $j_2=j_4$, and $|j_1-j_2|> K_0$, then by 	Conditions \ref{cond:finitemomt},   \ref{cond:alphamixing}, and Lemma \ref{lm:mixingineq}, we have
\begin{align*}
	   	 |\mathrm{E}(x_{i,j_1}x_{i,j_2}x_{i,j_3}x_{i,j_4})|=& \mathrm{E}(x_{i,j_1}^2x_{i,j_2}^2) = \mathrm{cov} (x_{i,j_1}^2, x_{i,j_2}^2) + \mathrm{E}(x_{i,j_1}^2) \mathrm{E}(x_{i,j_2}^2) \notag \\
	   	 \geq & \Theta(1)- |\mathrm{cov} (x_{i,j_1}^2, x_{i,j_2}^2) | \geq \Theta(1)-C\delta^{\frac{K_0\epsilon}{2+\epsilon}}=\Theta(1).
	   \end{align*}	
If $j_1=j_3$, $j_2=j_4$, and $|j_1-j_2|\leq K_0$, by Condition \ref{cond:finitemomt}, 	$\mathrm{E}(x_{i,j_1}x_{i,j_2}x_{i,j_3}x_{i,j_4})=O(1)$.    
Note that $(j_1,j_2)$ pairs satisfying $|j_1-j_2|>K_0$ and $|j_1-j_2|\leq K_0$ are $O(p^2)$ and $O(pK_0)$, respectively. Thus,
\begin{align} \label{eq:varorder4thone}
	&~\sum_{ \substack{1\leq j_1\neq j_2 \leq p; \\1\leq j_3 \neq j_4 \leq p}} [\mathrm{E}(x_{i,j_1} x_{i,j_2} x_{i,j_3} x_{i,j_4} )]^a \mathbf{1}_{\{ j_1=j_3,j_2=j_4\}}  \\
	=&~\sum_{ \substack{1\leq j_1\neq j_2 \leq p; \\1\leq j_3 \neq j_4 \leq p}} \Big[\mathrm{E}\Big(\prod_{t=1}^4x_{i,j_t} \Big)\Big]^a \mathbf{1}_{\left\{ \substack{ j_1=j_3,\\j_2=j_4}\right\}} [\mathbf{1}_{\{  |j_1-j_2|> K_0\}}+\mathbf{1}_{\{ |j_1-j_2|\leq K_0\}}] \notag \\
	=&~ \Theta(p^2)+O(pK_0)=\Theta(p^2). \notag
\end{align}	If $j_1=j_4$ and $j_2=j_3$, similarly to \eqref{eq:varorder4thone}, we have
\begin{align}\label{eq:varorder4thtwo}
	\sum_{ \substack{1\leq j_1\neq j_2 \leq p \\1\leq j_3 \neq j_4 \leq p}} [\mathrm{E}(x_{i,j_1} x_{i,j_2} x_{i,j_3} x_{i,j_4} )]^a \mathbf{1}_{\{ j_1=j_4,j_2=j_3\}}=\Theta(p^2).
\end{align}
 If $j_1\neq j_3$ and $j_1 \neq j_4$, we know \eqref{eq:fourthintermomtbound} holds. Recall $K_0$  in \eqref{eq:thresholddK} and $\kappa_m$  in \eqref{eq:kappamaxdef}.  
 Similarly to the analysis of  \eqref{eq:match22smallcase3dif}, we have
 \begin{align} \label{eq:varorder4ththree}
	&\sum_{ \substack{1\leq j_1\neq j_2 \leq p \\1\leq j_3 \neq j_4 \leq p}} [\mathrm{E}(x_{i,j_1} x_{i,j_2} x_{i,j_3} x_{i,j_4} )]^a \mathbf{1}_{\{ j_1\neq j_3,j_1 \neq j_4\}} \\
	=&~\sum_{ \substack{1\leq j_1\neq j_2 \leq p; \\1\leq j_3 \neq j_4 \leq p}} \Big[\mathrm{E}\Big(\prod_{t=1}^4x_{i,j_t} \Big)\Big]^a \mathbf{1}_{\left\{ \substack{j_1\neq j_3, j_1 \neq j_4}\right\}} \Big[\mathbf{1}_{\{\kappa_m>K_0\}}+\mathbf{1}_{\{\kappa_m\leq K_0\}}\Big] \notag \\
	=&~ o(p^2). \notag
\end{align}If $j_2\neq j_3$ and $j_2\neq j_4$, similarly to \eqref{eq:varorder4ththree}, we have
\begin{eqnarray}\label{eq:varorder4thfour}
&&\sum_{ \substack{1\leq j_1\neq j_2 \leq p \\1\leq j_3 \neq j_4 \leq p}} [\mathrm{E}(x_{i,j_1} x_{i,j_2} x_{i,j_3} x_{i,j_4} )]^a \mathbf{1}_{\{ j_2\neq j_3,j_2 \neq j_4\}}=o(p^2).	
\end{eqnarray}
In summary, combining \eqref{eq:varorder4thone}--\eqref{eq:varorder4thfour}, we have
\begin{eqnarray}
&& \sum_{ \substack{1\leq j_1\neq j_2 \leq p; \\1\leq j_3 \neq j_4 \leq p}}\Big[\mathrm{E}\Big(\prod_{t=1}^4x_{i,j_t} \Big)\Big]^a\simeq 	2\sum_{1\leq j_1\neq j_2 \leq p}\{\mathrm{E}(x_{i,j_1}^2x_{i,j_2}^2)\}^a. \label{eq:varordp4top2}
\end{eqnarray}
Combining  \eqref{eq:varoduaoncov},  \eqref{eq:utildeholdsvar} and \eqref{eq:varordp4top2}, Lemma \ref{lm:varianceorder} is proved. 


\paragraph{Proof under Condition \ref{cond:highordmominter}} \label{par:pfvarunderell}

In this section, we prove Lemma \ref{lm:varianceorder} by substituting Condition \ref{cond:alphamixing} with Condition  \ref{cond:highordmominter}. 
Following the notation in  Section  \ref{par:complpflm1}, we have
\begin{align*}
	\mathrm{var}\{\mathcal{U}(a)\}
	=\sum_{\substack{1 \leq j_1 \neq j_2 \leq p;\\ 1 \leq j_3 \neq j_4 \leq p}} \ \sum_{ \substack{0\leq c_1,c_2\leq a;\\ \mathbf{i}\in \mathcal{P}(n,a+c_1);\\ \tilde{\mathbf{i}}\in  \mathcal{P}(n,a+c_2)} }F(c_1,c_2,a)\times Q(\mathbf{i},j_1, j_2,\tilde{\mathbf{i}},j_3,j_4).
\end{align*}

When $\{\mathbf{i}\}\neq \{\tilde{\mathbf{i}}\}$, under $H_0$, we know  $\eqref{eq:summedadjustedutermvar}=0$ and  \eqref{eq:varsum1noteqs1s2} holds similarly. As $\{\mathbf{i}\}$ and $\{\tilde{\mathbf{i}}\}$ are of sizes $a+c_1$ and $a+c_2$ respectively, in the following we consider $\{\mathbf{i}\}=\{\tilde{\mathbf{i}}\}$, which induces $c_1=c_2$ and we write $c_1=c_2=c$. 

When $\{\mathbf{i}\}= \{\tilde{\mathbf{i}}\}$ and $c=0$, we know \eqref{eq:varsum1ceq0} also holds similarly, and   $\mathrm{var}\{\tilde{\mathcal{U}}(a)\}=\eqref{eq:varsum1ceq0}$ by \eqref{eq:utildeholdsvar}.  By Condition \ref{cond:highordmominter},
\begin{align}
 &~\mathrm{E}(x_{i,j_1}x_{i,j_2}x_{i,j_3}x_{i,j_4}) \label{eq:lmb1fourthorder} \\
 	       	 =&~ \kappa_1\Big\{\mathrm{E}(x_{i,j_1} x_{i,j_2}) \mathrm{E}(x_{i,j_3} x_{i,j_4})+ \mathrm{E}(x_{i,j_1} x_{i,j_3})\mathrm{E}(x_{i,j_2} x_{i,j_4}) \notag \\
	        &\quad \ +\mathrm{E}(x_{i,j_1} x_{i,j_4})\mathrm{E}(x_{i,j_2} x_{i,j_3})\Big\}. \notag
\end{align}Since $j_1\neq j_2$ and $j_3\neq j_4$, we know under $H_0$, $\eqref{eq:lmb1fourthorder}\neq 0$ only when $\{j_1=j_3, j_2=j_4\}$ or $\{j_1=j_4, j_2=j_3\}$; and then $
	\eqref{eq:lmb1fourthorder}=\kappa_1\mathrm{E}(x_{i,j_1}^2)\mathrm{E}(x_{i,j_2}^2)$. Thus 
	\begin{align*}
		\eqref{eq:varsum1ceq0}=2a!(P^n_a)^{-1}\sum_{1\leq j_1\neq j_2\leq p}\{\kappa_1\mathrm{E}(x_{i,j_1}^2)\mathrm{E}(x_{i,j_2}^2)\}^a=\Theta(p^2n^{-a}),
	\end{align*}  where the second equation follows from  Condition \ref{cond:finitemomt}. 

 When $\{\mathbf{i}\}= \{\tilde{\mathbf{i}}\}$ and $c\geq 1$,  $|\{\mathbf{i}\}_{(2)}|=|\{\mathbf{i}\}_{(3)}|=| \{\tilde{\mathbf{i}}\}_{(2)}|=| \{\tilde{\mathbf{i}}\}_{(3)}|>0$. Without loss of generality, we first consider  an index $i\in \{\mathbf{i}\}_{(2)}$, and discuss four cases.  

\textit{Case 1.1} If $i\not \in \{\tilde{\mathbf{i}}\}$,   since  $\mathrm{E}(\mathbf{x})=\mathbf{0}$, we know
\begin{align*}
\eqref{eq:summedadjustedutermvar} = \mathrm{E}(x_{i,j_1})\times \mathrm{E}(\mathrm{all\ the \ remaining\ terms})=0.
\end{align*} 

\textit{Case 1.2} If $i\in \{\tilde{\mathbf{i}}\}_{(2)}$,
\begin{align*}
\eqref{eq:summedadjustedutermvar} = \mathrm{E}(x_{i,j_1}x_{i,j_3})\times \mathrm{E}(\mathrm{all\ the \ remaining\ terms}),
\end{align*} which is nonzero when $j_1=j_3$. 

\textit{Case 1.3} If $i\in\{\tilde{\mathbf{i}}\}_{(3)}$, 
\begin{align*}
\eqref{eq:summedadjustedutermvar} = \mathrm{E}(x_{i,j_1}x_{i,j_4})\times \mathrm{E}(\mathrm{all\ the \ remaining\ terms})=0,
\end{align*} which is nonzero when $j_1=j_4$. 

\textit{Case 1.4} If $i\in \{\tilde{\mathbf{i}}\}_{(1)}$, 
this suggests $\{{\mathbf{i}}\}_{(1)}\neq \emptyset$ and thus $c\leq a-1$.  
By Condition \ref{cond:highordmominter},
\begin{align}\label{eq:threelmb1term}
\eqref{eq:summedadjustedutermvar} = \mathrm{E}(x_{i,j_1}x_{i,j_3}x_{i,j_4})\times \mathrm{E}[\mathrm{all\ the \ remaining\ terms}]=0.
\end{align}

 When $\{\mathbf{i}\}= \{\tilde{\mathbf{i}}\}$ and $c\leq a-1$,  we have $\{{\mathbf{i}}\}_{(1)}\neq \emptyset.$ 
We assume without loss of generality that an index $i\in \{{\mathbf{i}}\}_{(1)}$, and then discuss two cases. 

\textit{Case 2.1} If $i\in \{\tilde{\mathbf{i}}\}_{(2)}\cup \{\tilde{\mathbf{i}}\}_{(3)}$, symmetrically, \eqref{eq:summedadjustedutermvar} takes a form similarly to that in \eqref{eq:threelmb1term}, which is 0 under $H_0$ by Condition \ref{cond:highordmominter}. 

\textit{Case 2.2} If $i\not \in \{\tilde{\mathbf{i}}\}$,  by $j_1\neq j_2$, we know under $H_0$,
\begin{align*}
	\eqref{eq:summedadjustedutermvar}=\mathrm{E}(x_{i,j_1}x_{i,j_2})\times \mathrm{E}(\mathrm{all\ the \ remaining\ terms})=0. 
\end{align*}
In summary, $\eqref{eq:summedadjustedutermvar}\neq 0$  only when 
one of the following two cases holds:
\begin{enumerate}
	\item[1.] $j_1=j_3$, $j_2=j_4$, $\{{\mathbf{i}}\}_{(1)}=\{\tilde{\mathbf{i}}\}_{(1)}$, $\{{\mathbf{i}}\}_{(2)}=\{\tilde{\mathbf{i}}\}_{(2)}$, $\{{\mathbf{i}}\}_{(3)}=\{\tilde{\mathbf{i}}\}_{(3)}$;
	\item[2.] $j_1=j_4$, $j_2=j_3$, $\{{\mathbf{i}}\}_{(1)}=\{\tilde{\mathbf{i}}\}_{(1)}$, $\{{\mathbf{i}}\}_{(2)}=\{\tilde{\mathbf{i}}\}_{(3)}$, $\{{\mathbf{i}}\}_{(3)}=\{\tilde{\mathbf{i}}\}_{(2)}$. 
\end{enumerate}Under these two cases, $\eqref{eq:summedadjustedutermvar} = \{\kappa_1\mathrm{E}(x_{i,j_1}^2x_{i,j_2}^2)\}^{a-c}\{\mathrm{E}(x_{i,j_1}^2)\}^{c}\{\mathrm{E}(x_{i,j_2}^2)\}^{c}.$
It follows that when $\{\mathbf{i}\}=\{\tilde{\mathbf{i}}\}$ and $c\geq 1$, 
\begin{eqnarray}
	&&\quad \sum_{\substack{1 \leq j_1 \neq j_2 \leq p;\\ 1 \leq j_3 \neq j_4 \leq p}}  \sum_{ \substack{0\leq c\leq a;\\ \mathbf{i},\tilde{\mathbf{i}}\in \mathcal{P}(n,a+c)} } F(c,c,a)    Q(\mathbf{i},j_1, j_2,\tilde{\mathbf{i}},j_3,j_4)\mathbf{1}_{\{\{\mathbf{i}\}=\{\tilde{\mathbf{i}}\},c\geq 1 \}}  \label{eq:cond2proofvarcgeq1} \\
	&=&\sum_{\substack{1\leq c \leq a; \\ 1 \leq j_1 \neq j_2 \leq p}}  \binom{a}{c}^2\frac{2}{P^n_{a+c}}  \{\kappa_1\mathrm{E}(x_{i,j_1}^2x_{i,j_2}^2)\}^{a-c}\{\mathrm{E}(x_{i,j_1}^2)\}^{c}\{\mathrm{E}(x_{i,j_2}^2)\}^{c} \notag \\
	&=& \sum_{c=1}^a O(p^2n^{-(a+c)})=o(pn^{-a}), \notag
\end{eqnarray} where the last two equations use  Condition \ref{cond:finitemomt}. Similarly to Section \ref{sec:proofvarianceorder}, by  \eqref{eq:varsum1noteqs1s2} and \eqref{eq:utildstordervar}, we know $\mathrm{var}\{\tilde{\mathcal{U}}^*(a)\}=\eqref{eq:cond2proofvarcgeq1}=o(pn^{-a})=o(1)\mathrm{var}\{\tilde{\mathcal{U}}(a)\}$. 

\begin{remark}
$\kappa_1$ is assumed to be a constant in Condition \ref{cond:ellpmoment}. But the similar arguments apply in the proof if $\kappa_1$ changes with $n,p$ but converges to a constant. 
\end{remark}

\subsubsection{Proof of Lemma \ref{lm:covariancezro} (on Page \pageref{lm:covariancezro}, Section \ref{sec:detailofjointnormal})} \label{sec:proofcovariancezro}

Note that for two integers $a\neq b$, $\mathrm{cov}\{ \mathcal{U}(a)/\sigma(a), \mathcal{U}(b)/\sigma(b)  \}=\mathrm{E}[\mathcal{U}(a)\mathcal{U}(b)/\{\sigma(a)\sigma(b) \} ]$, and by Lemma \ref{lm:varianceorder}, $\mathrm{var}\{\tilde{\mathcal{U}}^*(a)\}=o(1)\mathrm{var}\{\tilde{\mathcal{U}}(a)\}$. 
Recall $\sigma^2(a)=\mathrm{var}\{ \mathcal{U}(a)\}$ from definition. Then by  Cauchy-Schwarz inequality,  we have  
\begin{align*}
 \mathrm{cov}\{ {\mathcal{U}(a) }/{\sigma(a) }, \, {\mathcal{U}(b) }/{\sigma(b) }  \}
=\mathrm{E}\{\tilde{\mathcal{U}}(a)\tilde{\mathcal{U}}(b) \}/\{\sigma(a) \sigma(b) \}+ o(1). 
\end{align*} In addition,
\begin{align*}
	 \mathrm{E}\{\tilde{\mathcal{U}}(a)\tilde{\mathcal{U}}(b) \} =&  \sum_{ \substack{1 \leq j_1 \neq j_2 \leq p,\\ 1 \leq j_3 \neq j_4 \leq p}}\sum_{ \substack{{\mathbf{i} \in \mathcal{P}(n,a) }, \\  \tilde{\mathbf{i} } \in \mathcal{P}(n,b)} }\mathrm{E}\Big\{  \prod_{k=1}^a ( x_{i_k,j_1} x_{i_k,j_2}) \prod_{\tilde{k}=1}^b( x_{\tilde{i}_{\tilde{k}},j_3} x_{\tilde{i}_{\tilde{k}},j_4}) \Big\}. 
\end{align*}
Since $a\neq b$, we know the two sets $\{i_1,\ldots, i_a\}$ and $ \{\tilde{i}_1, \ldots, \tilde{i}_b\}$ can not be the same. Following similar analysis to that of  \eqref{eq:summedtermprodvarind}, as $\mathrm{E}(x_{i,j_1}x_{i,j_2})=0$ under $H_0$, we have $\mathrm{E}\{\tilde{\mathcal{U}}(a)\tilde{\mathcal{U}}(b) \}=0$, and thus $\mathrm{cov}\{ \mathcal{U}(a)/\sigma(a), \mathcal{U}(b)/\sigma(b)  \}=o(1)$.

In particular, we note that given Lemma
 \ref{lm:varianceorder}, the argument does not depend on whether Condition \ref{cond:alphamixing} or \ref{cond:ellpmoment} is specified.

\subsubsection{Proof of Lemma \ref{lm:meaninvarsigma} (on Page \pageref{lm:meaninvarsigma}, Section \ref{sec:detailofjointnormal})} \label{sec:proofmeaninvarsigma}

We first show for $1\leq k_1 \neq k_2 \leq n$, $\mathrm{E}(D_{n,k_1}D_{n,k_2})=0$. Without loss of generality, we consider $k_1<k_2$.  Then $\mathrm{E}_{k_1}Z_n \in \mathcal{F}_{k_2}$, and
\begin{align*}
	& \mathrm{E}(D_{n,k_1}D_{n,k_2})\\
	=&\mathrm{E}\left[ (\mathrm{E}_{k_1}Z_n-\mathrm{E}_{k_1-1}Z_n)(\mathrm{E}_{k_2}Z_n-\mathrm{E}_{k_2-1}Z_n) \right] \\
	=& \mathrm{E} [ \mathrm{E}_{k_1}Z_n\times \mathrm{E}_{k_2}Z_n- \mathrm{E}_{k_1-1}  Z_n\times\mathrm{E}_{k_2}Z_n -\mathrm{E}_{k_1}Z_n\times \mathrm{E}_{k_2-1}Z_n \notag \\
	& \quad +\mathrm{E}_{k_1-1}Z_n\times  \mathrm{E}_{k_2-1}Z_n    ]\\
	=& \mathrm{E}[(\mathrm{E}_{k_1}Z_n ) Z_n] - \mathrm{E}[(\mathrm{E}_{k_1-1}Z_n)Z_n]-\mathrm{E}[(\mathrm{E}_{k_1}Z_n)Z_n]+\mathrm{E}[(\mathrm{E}_{k_1-1}Z_n)Z_n] \\
	=& 0.
\end{align*} 	
It follows that 
\begin{align*}
	\mathrm{E}\Biggr(\sum_{k=1}^n \pi^2_{n,k} \Biggr)=\sum_{k=1}^n \mathrm{E}\left(D^2_{n,k} \right) = \mathrm{E} \Biggr( \sum_{k=1}^n D_{n,k}\Biggr)^2 = \mathrm{var} (Z_n),
\end{align*} where the last equation uses the fact that $\mathrm{E}(D_{n,k})=0$ and $Z_n=\sum_{k=1}^nD_{n,k}$ from construction. 

In particular, we  note that the argument does not depend on whether Condition \ref{cond:alphamixing} or \ref{cond:ellpmoment} is specified.  

\subsubsection{Proof of Lemma \ref{lm:cltabnkform} (on Page \pageref{lm:cltabnkform}, Section \ref{sec:detailofjointnormal})} \label{sec:proofcltabnkform}

For given finite integer $a$, we derive the expression of $( \mathrm{E}_{k} -\mathrm{E}_{k-1}) [{\tilde{\mathcal{U}}(a) }/{\sigma(a)}]$. The form of $A_{n,k,a_r}$ for a general finite integer $a_r$ in Lemma \ref{lm:cltabnkform}  follows similarly. 

By the definition in \eqref{eq:originleadingterm}, we know
\begin{eqnarray}
	&&\quad \quad (\mathrm{E}_k-\mathrm{E}_{k-1})\tilde{\mathcal{U}}(a)=(P^n_{a})^{-1}\sum_{\substack{ 1\leq j_1 \neq j_2 \leq p; \\ \mathbf{i}\in \mathcal{P}(n,a) }} 
	 (\mathrm{E}_k-\mathrm{E}_{k-1})\Big[\prod_{t=1}^a x_{i_t,j_1} x_{i_t,j_2}\Big].\label{eq:diffexpank}
\end{eqnarray}
To derive \eqref{eq:diffexpank},  we next examine the value of
\begin{align}
	( \mathrm{E}_{k} -\mathrm{E}_{k-1} ) \Big[\prod_{t=1}^a x_{i_t,j_1} x_{i_t,j_2}\Big]. \label{eq:differenceuaexpect}
\end{align}  

We claim $\eqref{eq:differenceuaexpect} \neq 0$  only when $k \in \{i_1 , \ldots, i_{a} \}$. If $k \not \in \{i_1 , \ldots, i_{a} \}$, we assume without loss of generality that $i_1,\ldots, i_m<k$ and $i_{m+1},\ldots, i_a>k$. Then
\begin{align*}
	&~ \left( \mathrm{E}_{k} -\mathrm{E}_{k-1}\right) \Big[\prod_{t=1}^a x_{i_t,j_1} x_{i_t,j_2}\Big] \\
	=&~ \Big(\prod_{t=1}^m x_{{i}_t,j_1}x_{{i}_t,j_2}\Big) \Big[\mathrm{E}_k \Big( \prod_{t=m+1}^ax_{{i}_{t},j_1}x_{{i}_{t},j_2} \Big)  -   \mathrm{E}_{k-1} \Big( \prod_{t=m+1}^ax_{{i}_{t},j_1}x_{{i}_{t},j_2} \Big)\Big] \\
	 =&~ 0.
\end{align*}
Thus  if $\eqref{eq:differenceuaexpect} \neq 0$, we know  $k\in \{i_1,\ldots, i_a\}$. In addition, we next show $\eqref{eq:differenceuaexpect} \neq 0$ only when $i_1,\ldots, i_a\leq k$. Suppose that if there exist some indexes  in $\{i_1,\ldots,i_a\}$ that are greater than $k$, we assume without loss of generality that $i_m=k$, $i_1,\ldots, i_{m-1}<k$,   and $i_{m+1},\ldots, i_a>k$. Then
\begin{eqnarray*}
 \mathrm{E}_{k}\Big(\prod_{t=1}^a x_{i_t,j_1} x_{i_t,j_2}\Big) 
	 &=& \Big(\prod_{t=1}^m x_{{i}_t,j_1}x_{{i}_t,j_2}\Big) \mathrm{E}_{k} \Big( \prod_{t=m+1}^ax_{{i}_{t},j_1}x_{{i}_{t},j_2} \Big) \\
	 &=& \Big(\prod_{t=1}^m x_{{i}_t,j_1}x_{{i}_t,j_2}\Big) \mathrm{E} \Big( \prod_{t=m+1}^ax_{{i}_{t},j_1}x_{{i}_{t},j_2} \Big) =0,
\end{eqnarray*} and    
\begin{align*}
 \mathrm{E}_{k-1}\Big(\prod_{t=1}^a x_{i_t,j_1} x_{i_t,j_2}\Big) 
	 =&~ \Big(\prod_{t=1}^{m-1} x_{{i}_t,j_1}x_{{i}_t,j_2}\Big)  \mathrm{E}_{k-1} \Big(x_{k,j_1}x_{k,j_2} \prod_{t=m+1}^ax_{{i}_{t},j_1}x_{{i}_{t},j_2}\Big) \\
	 =&~ \Big(\prod_{t=1}^{m-1} x_{{i}_t,j_1}x_{{i}_t,j_2}\Big) \mathrm{E} (x_{k,j_1}x_{k,j_2})\prod_{t=m+1}^a \mathrm{E} (x_{{i}_t,j_1}x_{{i}_t,j_2}) =0.
\end{align*} Therefore, we know $\eqref{eq:differenceuaexpect}\neq 0$ when $k\in\{i_1,\ldots,i_a\}$ and $i_1,\ldots,i_a\leq k$. 

When $k<a$, there exist some indexes  in  $\{i_1,\ldots,i_a\} > k$. Thus $\eqref{eq:differenceuaexpect}=0$, and  $\eqref{eq:diffexpank}=0$. 
When $k\geq a$, assume without loss of generality that  $i_a=k$ and $i_1,\cdots,i_{a-1} \leq k-1$, then
\begin{eqnarray*}
\mathrm{E}_{k-1}\Big[\Big(\prod_{t=1}^{a-1} x_{i_t,j_1} x_{i_t,j_2}\Big) x_{k,j_1}x_{k,j_2}\Big] =\Big(\prod_{t=1}^{a-1} x_{i_t,j_1} x_{i_t,j_2}\Big) \mathrm{E}(x_{k,j_1} x_{k,j_2})=0,\notag 
\end{eqnarray*} and
\begin{eqnarray*}
	&& \mathrm{E}_{k}\Big[\Big(\prod_{t=1}^{a-1} x_{i_t,j_1} x_{i_t,j_2}\Big) x_{k,j_1}x_{k,j_2}\Big]  =\Big(\prod_{t=1}^{a-1} x_{i_t,j_1} x_{i_t,j_2}\Big) x_{k,j_1}x_{k,j_2}.
\end{eqnarray*}
In summary,  for $k \geq a$,
\begin{align*}
	&~ \left( \mathrm{E}_{k} -\mathrm{E}_{k-1}\right) \frac{\tilde{\mathcal{U}}(a) }{\sigma(a)} \\
	=&~ \frac{1}{\sigma(a)P^n_{a}}    \sum_{ \substack{ 1\leq i_1 \neq \cdots \neq i_{a-1}\leq k-1;\\ 1\leq j_1\neq j_2\leq p}}\binom{a}{1}\times \left( \mathrm{E}_{k} -\mathrm{E}_{k-1}\right)\Big[\Big(\prod_{t=1}^{a-1} x_{i_t,j_1} x_{i_t,j_2}\Big) x_{k,j_1}x_{k,j_2}\Big] \\
	 =&~\frac{a}{\sigma(a)P^n_a} \sum_{1\leq i_1 \neq \cdots \neq i_{a-1}\leq k-1}\sum_{1\leq j_1\neq j_2\leq p} (x_{k,j_1} x_{k,j_2})\times \prod_{t=1}^{a-1}(x_{i_t,j_1} x_{i_t,j_2}).
\end{align*} 
In particular, we  note that the argument does not depend on whether Condition \ref{cond:alphamixing} or \ref{cond:ellpmoment} is specified.

\subsubsection{Proof of Lemma \ref{lm:targetorder} (on Page \pageref{lm:targetorder}, Section  \ref{sec:detailofjointnormal})} \label{sec:prooftargetorder}
By Lemma \ref{lm:cltabnkform}, we know the explicit form of $D_{n,k}=\sum_{r=1}^m t_{r}A_{n,k,a_r}$, and it follows that  
$\pi_{n,k}^2=  \sum_{1\leq r_1,r_2\leq m} t_{r_1}t_{r_2}\mathrm{E}_{k-1}( A_{n,k,a_{r_1}} A_{n,k,a_{r_2}}).$
Note that by Cauchy-Schwarz inequality, for some constant $C$, $$\mathrm{var}\Big(\sum_{k=1}^n \pi_{n,k}^2\Big)\leq Cn^2 \max_{1\leq k\leq n;\, 1\leq r_1,r_2\leq m} \mathrm{var}(\mathbb{T}_{k,a_{r_1},a_{r_2}}), $$
where we define  $c(n,a_r)=[a_r\times \{ \sigma(a_r)P^n_{a_r}\}^{-1}]^{2}$ and 
\begin{eqnarray*}
\mathbb{T}_{k,a_{r_1},a_{r_2}}&=&\mathrm{E}_{k-1}( A_{n,k,a_{r_1}} A_{n,k,a_{r_2}}) \notag \\
&=& \sum_{ \substack{ \mathbf{i} \in \mathcal{P}(k-1,a_{r_1}-1), \\ \tilde{\mathbf{i}}\in  \mathcal{P}(k-1,a_{r_2}-1) } } \sum_{\substack{1\leq j_1 \neq j_2\leq p,\\ 1\leq j_3 \neq j_4\leq p}} \{c(n,a_{r_1})c(n,a_{r_2})\}^{1/2} \notag \\
&& \times \, \mathrm{E}\Big( \prod_{t=1}^4x_{k,j_t}\Big) \times \Big(\prod_{t=1}^{a_{r_1}-1} x_{i_t,j_1} x_{i_t,j_2}\Big)\times\Big(\prod_{{t}=1}^{a_{r_2}-1}   x_{\tilde{i}_{{t}},j_3} x_{\tilde{i}_{{t}},j_4}\Big).
\end{eqnarray*} Therefore to prove Lemma \ref{lm:targetorder}, it suffices to prove $\mathrm{var}(\mathbb{T}_{k,a_{r_1},a_{r_2}})=o(n^{-2})$ for every  $1\leq k\leq n$ and $1\leq r_1,r_2\leq m$.  

Without loss of generality, we prove $\mathrm{var}(\mathbb{T}_{k,a_{1},a_{2}})=o(n^{-2})$ for any fixed constants $a_1$ and $a_2$ and $1\leq k\leq n$. Similarly to Section \ref{sec:proofvarianceorder}, for illustration,  we first consider  a simple setting where $x_{i,j}$'s are independent in Section \ref{par:proftaaindpcond}. 
Next in Section \ref{sec:pfvartaamigixing},  we prove that under Condition \ref{cond:alphamixing}, $\mathrm{var}(\mathbb{T}_{k,a_{1},a_{2}})=O(n^{-2}p^{-1}\log^3 p)=o(n^{-2})$. Last in Section \ref{par:pfmomentvartaa}, we prove that under Condition \ref{cond:ellpmoment}, $\mathrm{var}(\mathbb{T}_{k,a_{1},a_{2}})=O(n^{-2}p^{-2}+n^{-3})=o(n^{-2})$. Then Lemma \ref{lm:targetorder} is proved.

\paragraph{Proof illustration}\label{par:proftaaindpcond}
In this section, we assume $x_{i,j}$'s are independent and prove $\mathbb{T}_{k,a_{1},a_{2}}=o(n^{-2})$.


When $x_{i,j}$'s are independent, since $j_1\neq j_2$ and $j_3\neq j_4$, we know that $\mathrm{E}( x_{k,j_1}x_{k,j_2}x_{k,j_3}x_{k,j_4}) \neq 0$ only when $\{j_1, j_2\}=\{j_3, j_4\}$; and it follows that $\mathrm{E}( x_{k,j_1}x_{k,j_2}x_{k,j_3}x_{k,j_4})=\mathrm{E}(x_{1,j_1}^2)\mathrm{E}(x_{1,j_2}^2)$. Thus $\mathbb{T}_{k,a_1,a_2}=2c(n,a)\times T_{k,a_1,a_2}$, where we define   
\begin{align*}
T_{k,a_1,a_2}=\sum_{\substack{ \mathbf{i} \in \mathcal{P}(k-1,a_1-1), \\ \tilde{\mathbf{i}} \in \mathcal{P}(k-1,a_2-1) }} \sum_{1\leq j_1 \neq j_2\leq p} \prod_{t=1}^2\mathrm{E}(x_{1,j_t}^2)\Big(\prod_{t=1}^{a_1-1} x_{i_t,j_1} x_{i_t,j_2}\Big)\Big(\prod_{{t}=1}^{a_2-1}   x_{\tilde{i}_{{t}},j_1} x_{\tilde{i}_{{t}},j_2}\Big).
\end{align*}
We note that $c(n,a)$ is of order $\Theta(p^{-2}n^{-a})$ by Lemma \ref{lm:varianceorder}. To prove $\mathrm{var}(\mathbb{T}_{k,a,a})=o(n^{-2})$, it suffices to show that $\mathrm{var}(T_{k,a_1,a_2})=o(n^{a_1+a_2-2}p^4)$. If $a_1=a_2=1$, $T_{k,a_1,a_2}$ is not random and thus $\mathrm{var}(T_{k,a_1,a_2})=0$. It remains to consider $a_1\geq 1$ or $a_2\geq 1$ below. To examine $\mathrm{var}(T_{k,a_1,a_2})$, we will first consider $\mathrm{E}(T_{k,a_1,a_2})$ and $\mathrm{E}(T_{k,a_1,a_2}^2)$, then  $\mathrm{var}(T_{k,a_1,a_2})=\mathrm{E}(T_{k,a_1,a_2}^2)-\{ \mathrm{E}(T_{k,a_1,a_2})\}^2 $.


For $\mathrm{E}(T_{k,a_1,a_2})$,  note that $\mathrm{E}\{  (\prod_{t=1}^{a_1-1} x_{i_t,j_1} x_{i_t,j_2})(\prod_{{t}=1}^{a_2-1}   x_{\tilde{i}_{{t}},j_1} x_{\tilde{i}_{{t}},j_2}) \} \neq 0$ only when $\{\mathbf{i}\}=\{\tilde{\mathbf{i}}\}$ for given $\mathbf{i} \in \mathcal{P}(k-1,a_1-1)$ and $ \tilde{\mathbf{i}} \in \mathcal{P}(k-1,a_2-1)$. 
Therefore, if $a_1\neq a_2$, $\mathrm{E}(T_{k,a_1,a_2})=0$. 
If $a_1=a_2=a$ for some  $a$, we have
\begin{align}
	\mathrm{E}(T_{k,a_1,a_2})=\sum_{\substack{ \mathbf{i} \in \mathcal{P}(k-1,a-1), \\ \tilde{\mathbf{i}} \in \mathcal{P}(k-1,a-1) }} \mathbf{1}_{\{ \{\mathbf{i}\}=\{\tilde{\mathbf{i}} \} \} } \sum_{1\leq j_1\neq j_2\leq p} \{ \mathrm{E}(x_{1,j_1}^2)\mathrm{E}(x_{1,j_2}^2) \}^{a}, \label{eq:etaaksq}
\end{align} 
where $\mathbf{1}_{\{ \{\mathbf{i}\}=\{\tilde{\mathbf{i}} \} \} } $ represents an indicator such that the two sets $\{\mathbf{i}\}=\{\tilde{\mathbf{i}} \}$; and we write 
\begin{align*}
	\{\mathrm{E}(T_{k,a_1,a_2})\}^2=&\sum_{\substack{ \mathbf{i},\, \mathbf{m} \in \mathcal{P}(k-1,a-1), \\ \tilde{\mathbf{i}},\, \tilde{\mathbf{m}} \in \mathcal{P}(k-1,a-1) }} \,  \sum_{\substack{1\leq j_1\neq j_2\leq p,\\ 1\leq j_3\neq j_4\leq p}} \mathbf{1}_{\{\substack{\{\mathbf{i}\}=\{\tilde{\mathbf{i}}\},\, \{\mathbf{m}\}=\{\tilde{\mathbf{m}} \}  \}}}  \prod_{t=1}^4\{\mathrm{E}(x_{1,j_t}^2)\}^a.
\end{align*}  where $\mathbf{1}_{\{\substack{\{\mathbf{i}\}=\{\tilde{\mathbf{i}}\},\, \{\mathbf{m}\}=\{\tilde{\mathbf{m}} \}  \}}}$ represents an indicator such that $\{\mathbf{i}\}=\{\tilde{\mathbf{i}}\}$ and $\{\mathbf{m}\}=\{\tilde{\mathbf{m}}\}$ hold at the same time. 

For $\mathrm{E}(T_{k,a_1,a_2}^2)$, we have
\begin{align}
	\mathrm{E}(T_{k,a_1,a_2}^2)=&\sum_{\substack{ \mathbf{i},\, \mathbf{m} \in \mathcal{P}(k-1,a_1-1), \\ \tilde{\mathbf{i}},\, \tilde{\mathbf{m}} \in \mathcal{P}(k-1,a_2-1) }} \,   \sum_{\substack{1\leq j_1\neq j_2\leq p,\\ 1\leq j_3\neq j_4\leq p}} \tilde{Q}(\mathbf{i},\tilde{\mathbf{i}},\mathbf{m},\tilde{\mathbf{m}},\mathbf{j}),\label{eq:taakindpsqexp}  
\end{align} where for the simplicity of notation, we define 
\begin{align*}
	\tilde{Q}(\mathbf{i},\tilde{\mathbf{i}},\mathbf{m},\tilde{\mathbf{m}},\mathbf{j})=\mathrm{E}\Big(  \prod_{t=1}^{a-1} x_{i_t,j_1} x_{i_t,j_2}  x_{\tilde{i}_t,j_1} x_{\tilde{i}_t,j_2}x_{m_t,j_3} x_{m_t,j_4}  x_{\tilde{m}_t,j_3} x_{\tilde{m}_t,j_4} \Big)\prod_{t=1}^4\mathrm{E}(x_{1,j_t}^2).
\end{align*}
We decompose $\mathrm{E}(T_{k,a_1,a_2}^2)=\mathrm{E}(T_{k,a_1,a_2}^2)_{(1)}+\mathrm{E}(T_{k,a_1,a_2}^2)_{(2)}$, where
\begin{align*}
\mathrm{E}(T_{k,a_1,a_2}^2)_{(1)}=&\sum_{\substack{ \mathbf{i},\, \mathbf{m} \in \mathcal{P}(k-1,a_1-1), \\ \tilde{\mathbf{i}},\, \tilde{\mathbf{m}} \in \mathcal{P}(k-1,a_2-1) }} \,  \sum_{\substack{1\leq j_1\neq j_2\leq p,\\ 1\leq j_3\neq j_4\leq p}}  \mathbf{1}_{\Big\{\substack{\{\mathbf{i}\}=\{\tilde{\mathbf{i}}\},\\ \, \{\mathbf{m}\}=\{\tilde{\mathbf{m}} \}}  \Big\}} \tilde{Q}(\mathbf{i},\tilde{\mathbf{i}},\mathbf{m},\tilde{\mathbf{m}},\mathbf{j}),\\ 
\mathrm{E}(T_{k,a_1,a_2}^2)_{(2)}=&\sum_{\substack{ \mathbf{i},\, \mathbf{m} \in \mathcal{P}(k-1,a_1-1), \\ \tilde{\mathbf{i}},\, \tilde{\mathbf{m}} \in \mathcal{P}(k-1,a_2-1) }} \,  \sum_{\substack{1\leq j_1\neq j_2\leq p,\\ 1\leq j_3\neq j_4\leq p}}  \mathbf{1}_{\Big\{ \substack{\{\mathbf{i}\}\neq \{\tilde{\mathbf{i}}\} \text{ or }\\ \{\mathbf{m}\}\neq \{\tilde{\mathbf{m}}\} }\Big\}} \tilde{Q}(\mathbf{i},\tilde{\mathbf{i}},\mathbf{m},\tilde{\mathbf{m}},\mathbf{j}),  \notag
\end{align*} where the two indicators 
$\mathbf{1}_{\{\substack{\{\mathbf{i}\}=\{\tilde{\mathbf{i}}\}, \, \{\mathbf{m}\}=\{\tilde{\mathbf{m}} \}}  \}}$ and $\mathbf{1}_{\{ \substack{\{\mathbf{i}\}\neq \{\tilde{\mathbf{i}}\} \text{ or } \{\mathbf{m}\}\neq \{\tilde{\mathbf{m}}\} }\}}$ represent that $\{\mathbf{i}\}=\{\tilde{\mathbf{i}}\}$ and $\{\mathbf{m}\}=\{\tilde{\mathbf{m}}\}$ hold at the same time or not, respectively. 
 To prove $\mathrm{var}(T_{k,a_1,a_2})=o(n^{a_1+a_2-2}p^4)$, since $|\mathrm{var}(T_{k,a_1,a_2})|\leq |\mathrm{E}(T_{k,a_1,a_2}^2)_{(1)}-\{\mathrm{E}(T_{k,a_1,a_2})\}^2|+|\mathrm{E}(T_{k,a_1,a_2})_{(2)}|$, we show $|\mathrm{E}(T_{k,a_1,a_2}^2)_{(1)}-\{\mathrm{E}(T_{k,a_1,a_2})\}^2|=o(n^{2(a-1)}p^4)$ and $\mathrm{E}(T_{k,a_1,a_2})_{(2)}=o(n^{a_1+a_2-2}p^4)$,  respectively below.  
 
\subparagraph*{Part I: $|\mathrm{E}(T_{k,a_1,a_2}^2)_{(1)}-\{\mathrm{E}(T_{k,a_1,a_2})\}^2|=o(n^{a_1+a_2-2}p^4)$}
By the analysis above, $\mathrm{E}(T_{k,a_1,a_2})=0$  if $a_1\neq a_2$. Also we know $\mathrm{E}(T_{k,a_1,a_2}^2)_{(1)}=0$ if $a_1\neq a_2$, since $\{\mathbf{i}\}= \{\tilde{\mathbf{i}}\}$ and $\{\mathbf{m}\}= \{\tilde{\mathbf{m}}\}$ will not happen.
Thus it remains to consider  $a_1=a_2=a$ for some $a$ below. 
By the forms of $\mathrm{E}(T_{k,a_1,a_2}^2)_{(1)}$ and $\{\mathrm{E}(T_{k,a_1,a_2})\}^2$, we consider $\{\mathbf{i}\}=\{\tilde{\mathbf{i}}\}$ and  $\{\mathbf{m}\}=\{\tilde{\mathbf{m}}\}$. 
If $\{\mathbf{i}\}\cap \{\mathbf{m}\}=\emptyset$, 
\begin{align}
	\tilde{Q}(\mathbf{i},\tilde{\mathbf{i}},\mathbf{m},\tilde{\mathbf{m}},\mathbf{j})=\prod_{t=1}^4\{\mathrm{E}(x_{1,j_t}^2)\}^a,\label{eq:sumqtermindecond}
\end{align}
where we use the independence between $x_{i,j}$'s and $j_1\neq j_2$ and $j_3\neq j_4$.   
 If $\{j_1,j_2\}\cap \{j_3,j_4\}=\emptyset$,   \eqref{eq:sumqtermindecond} also holds similarly by the independence between $x_{i,j}$'s.
In summary, when $\{\mathbf{i}\}=\{\tilde{\mathbf{i}}\}$ and $\{\mathbf{m}\}=\{\tilde{\mathbf{m}}\}$, we know that  $| \mathrm{E}\{\tilde{Q}(\mathbf{i},\tilde{\mathbf{i}},\mathbf{m},\tilde{\mathbf{m}},\mathbf{j})\} - \prod_{t=1}^4\{\mathrm{E}(x_{1,j_t}^2)\}^a |= 0$, if $\{\mathbf{i}\}\cap \{\mathbf{m}\}=\emptyset$ or $\{j_1,j_2\}\cap \{j_3,j_4\}=\emptyset$. It follows that
\begin{align}
	&~|\mathrm{E}(T_{k,a_1,a_2}^2)_{(1)}-\{\mathrm{E}(T_{k,a_1,a_2})\}^2| \label{eq:bdddifftaaindpcond} \\
\leq &~  \sum_{\substack{ \mathbf{i},\, \mathbf{m} \in \mathcal{P}(k-1,a_1-1), \\ \tilde{\mathbf{i}},\, \tilde{\mathbf{m}} \in \mathcal{P}(k-1,a_2-1) }} \,  \sum_{\substack{1\leq j_1\neq j_2\leq p,\\ 1\leq j_3\neq j_4\leq p}} \mathbf{1}_{\Big\{ \substack{\{\mathbf{i}\}=\{\tilde{\mathbf{i}}\}, \, \{\mathbf{m}\}=\{\tilde{\mathbf{m}}\}, \, \{\mathbf{i}\}\cap \{\mathbf{m}\}\neq \emptyset,\\  \{j_1,j_2\}\cap \{j_3,j_4\}\neq \emptyset }\Big\}} \notag \\
&~\quad \times   \Big| \tilde{Q}(\mathbf{i},\tilde{\mathbf{i}},\mathbf{m},\tilde{\mathbf{m}},\mathbf{j}) - \prod_{t=1}^4\{\mathrm{E}(x_{1,j_t}^2)\}^a \Big|
   \notag \\
	\leq & ~ C n^{a_1+a_2-3}p^{4-1}=o(n^{a_1+a_2-2}p^4), \notag
\end{align} where we use the boundedness of moments in Condition \ref{cond:finitemomt} and the facts:
\begin{align*}
\sum_{\substack{ \mathbf{i},\, \mathbf{m} \in \mathcal{P}(k-1,a_1-1); \, \tilde{\mathbf{i}},\, \tilde{\mathbf{m}} \in \mathcal{P}(k-1,a_2-1) }} \,&\mathbf{1}_{\{\{\mathbf{i}\}=\{\tilde{\mathbf{i}}\}, \, \{\mathbf{m}\}=\{\tilde{\mathbf{m}}\}, \, \{\mathbf{i}\}\cap \{\mathbf{m}\}\neq \emptyset\}} \leq Cn^{a_1+a_2-3}, \notag \\
 	\sum_{\substack{1\leq j_1\neq j_2\leq p;\, 1\leq j_3\neq j_4\leq p}} &\mathbf{1}_{\{ \{j_1,j_2\}\cap \{j_3,j_4\}\neq \emptyset\} } \leq Cp^{4-1}.
\end{align*} 

\subparagraph*{Part II: $\mathrm{E}(T_{k,a_1,a_2})_{(2)}=o(n^{a_1+a_2-2}p^4)$} 
We claim that $\tilde{Q}(\mathbf{i},\tilde{\mathbf{i}},\mathbf{m},\tilde{\mathbf{m}},\mathbf{j}) =0$ when $|\{\mathbf{i}\}\cup \{\tilde{\mathbf{i}}\}\cup \{\mathbf{m}\} \cup \{\tilde{\mathbf{m}}\}|>a_1+a_2-2$, i.e., one of the index only appears once in the four index sets. To see this, we assume, without loss of generality, $i_1\in \{\mathbf{i}\}$ but $i_1 \not \in \{\tilde{\mathbf{i}}\} \cup \{\mathbf{m}\}\cup \{\tilde{\mathbf{m}}\}$, then 
\begin{align}
	\tilde{Q}(\mathbf{i},\tilde{\mathbf{i}},\mathbf{m},\tilde{\mathbf{m}},\mathbf{j})=\mathrm{E}(x_{i_1,j_1}x_{i_1,j_2})\times \mathrm{E} (\text{the remaining terms})=0.\label{eq:qtildezero}
\end{align}Thus when $\tilde{Q}(\mathbf{i},\tilde{\mathbf{i}},\mathbf{m},\tilde{\mathbf{m}},\mathbf{j})\neq 0$, the  union of the four sets satisfies
\begin{align}
	|\{\mathbf{i}\}\cup \{\tilde{\mathbf{i}}\}\cup \{\mathbf{m}\} \cup \{\tilde{\mathbf{m}}\}|\leq a_1+a_2-2. \label{eq:sizecontroltaaindpcond}
\end{align}
In addition,  note that we need to consider $\{\mathbf{i}\}\neq \{\tilde{\mathbf{i}}\}$ or  $\{\mathbf{m}\}\neq \{ \tilde{\mathbf{m}}\}$ when analyzing $\mathrm{E}(T_{k,a_1,a_2}^2)_{(2)}$. Assume, without loss of generality, that there exists an index $i_1\in \{\mathbf{i}\}$ but $i_1\not \in \{\tilde{\mathbf{i}}\}$. Similarly to  \eqref{eq:qtildezero}, we have $\tilde{Q}(\mathbf{i},\tilde{\mathbf{i}},\mathbf{m},\tilde{\mathbf{m}},\mathbf{j})\neq 0$ only when $i_1\in \{\mathbf{m}\}\cup \{\tilde{\mathbf{m}}\}$.  If $i_1\in \{\mathbf{m}\}$ and $i_1 \in \{\tilde{\mathbf{m}}\}$,
\begin{align*}
	\tilde{Q}(\mathbf{i},\tilde{\mathbf{i}},\mathbf{m},\tilde{\mathbf{m}},\mathbf{j})=\mathrm{E}(x_{1,j_1}x_{1,j_3}x_{1,j_4})\times \mathrm{E}(\text{all the remaining terms})=0,
\end{align*} as $j_3\neq j_4$ and $x_{i,j}$'s are independent;
if $i_1$ is only in one of $\{\mathbf{m}\}$ and $\{\tilde{\mathbf{m}}\}$, for example, $i_1\in \{\mathbf{m}\}$ but $i_1\not \in \{\tilde{\mathbf{m}}\}$, then 
\begin{align*}
	\tilde{Q}(\mathbf{i},\tilde{\mathbf{i}},\mathbf{m},\tilde{\mathbf{m}},\mathbf{j})=\mathrm{E}(x_{1,j_1}x_{1,j_3})\times \mathrm{E}(\text{all the remaining terms}), 
\end{align*} which is nonzero only when $j_1=j_3$. 
 By analyzing  the indexes in $\{\tilde{\mathbf{i}}\}$ symmetrically, we further know $\tilde{Q}(\mathbf{i},\tilde{\mathbf{i}},\mathbf{m},\tilde{\mathbf{m}},\mathbf{j})\neq 0$ only when $\{j_1,j_2\}=\{j_3,j_4\}$.  Therefore,  
\begin{align}
	|\{j_1,j_2,j_3,j_4\}|= 2. \label{eq:jindxorderindep}
\end{align}
Combining \eqref{eq:sizecontroltaaindpcond} and \eqref{eq:jindxorderindep}, and by the boundedness of moments in Condition \ref{cond:finitemomt}, we have 
\begin{align}
	|\mathrm{E}(T_{k,a_1,a_2}^2)_{(2)}|= O(n^{a_1+a_2-2}p^2). \label{eq:taasecondindpcond}
\end{align}

In summary, combining \eqref{eq:bdddifftaaindpcond} and \eqref{eq:taasecondindpcond}, we have 
\begin{align*}
	|\mathrm{var}(T_{k,a_1,a_2})|=&~ |\mathrm{E}(T_{k,a_1,a_2}^2)-\{ \mathrm{E}(T_{k,a_1,a_2})\}^2|\notag \\
	\leq & ~|\mathrm{E}(T_{k,a_1,a_2}^2)_{(1)}-\{\mathrm{E}(T_{k,a_1,a_2})\}^2|+|\mathrm{E}(T_{k,a_1,a_2}^2)_{(2)}|\notag \\
	=&~ O(n^{a_1+a_2-3}p^3)+O(n^{a_1+a_2-2}p^2). \notag 
\end{align*} which is $o(n^{a_1+a_2-2}p^4)$. 

\paragraph{Proof under Condition \ref{cond:alphamixing}} \label{sec:pfvartaamigixing} \quad


\smallskip
\subparagraph*{Proof idea}
Section \ref{par:proftaaindpcond} assumes that $x_{i,j}$'s are independent. In this section, we further prove Lemma \ref{lm:targetorder} under Condition \ref{cond:alphamixing}. Similarly to Section \ref{par:complpflm1}, we know that under Condition \ref{cond:alphamixing}, $x_{i,j}$'s may be  no longer independent, but the dependence between $x_{i,j_1}$ and $x_{i,j_2}$ degenerates exponentially with their distance $|j_1-j_2|$. To quantitatively examine $|j_1-j_2|$,  we will introduce a threshold of distance $D_0$ to be defined in  \eqref{eq:thresholdd} below, which is similar to $K_0$ in  \eqref{eq:thresholddK}.  Intuitively, when $|j_1-j_2|>D_0$, $x_{i,j_1}$ and $x_{i,j_2}$ are ``asymptotically independent'' with similar properties to those under the independence case in Section \ref{par:proftaaindpcond}.  The following proof  will provide comprehensive discussions based on  $D_0$. 

Recall that as argued at the beginning of Section \ref{sec:prooftargetorder},  to prove Lemma \ref{lm:targetorder}, it suffices to show $\mathrm{var}(\mathbb{T}_{k,a_{1},a_{2}})=O(n^{-2}p^{-1}\log^3 p)=o(n^{-2})$ for any fixed integers $a_1$ and $a_2$. To facilitate the discussion, we define some notation to be used in the proof.


\smallskip
\subparagraph*{Notation} \label{sec:vartaanotations}
For given tuples $\mathbf{i}^{(l)}=(i_1,\ldots,i_{a_l-1}) \in \mathcal{P}(k-1,a_l-1)$ with $l=1,2$, we define $(\mathbf{i}^{(1)},{\mathbf{i}}^{(2)} )=(i_1^{(1)},\ldots,i_{a_1-1}^{(1)}, {i}_1^{(2)},\ldots,{i}_{a_2-1}^{(2)})$, and let $S(\mathbf{i}^{(1)},{\mathbf{i}}^{(2)})$ be a collection of tuples $(\mathbf{i}^{(1)},{\mathbf{i}}^{(2)} )$ where $\mathbf{i}^{(l)} \in \mathcal{P}(k-1,a_l-1)$ for $l=1,2$.   
Moreover, we  define $\mathcal{J}=\{(j_1,j_2): 1\leq j_1\neq j_2 \leq p)$. Then 
\begin{align*}
	\mathbb{T}_{k,a_{1},a_{2}}=\sum_{\substack{ S(\mathbf{i}^{(1)},\mathbf{i}^{(2)}); \\  (j_1,j_2), (j_3,j_4)\in \mathcal{J} }  }\Big\{\prod_{l=1}^2c(n,a_l)\Big\}^{1/2}\times  \mathbb{X}( k,\mathbf{i}^{(l)}, j_{2l-1},j_{2l}: l=1,2), 
\end{align*} where we recall that $c(n,a)=[a\times \{ \sigma(a)P^n_a\}^{-1}]^{2}$ and we define
\begin{eqnarray*}
\mathbb{X}( k,\mathbf{i}^{(l)}, j_{2l-1},j_{2l}: l=1,2)=\mathrm{E}\Big( \prod_{t=1}^4 x_{k,j_t}  \Big)\prod_{l=1}^2 \prod_{t=1}^{a_l-1} (x_{i_t^{(l)},\, j_{2l-1}} x_{i_t^{(l)},\, j_{2l}} ). \notag 
\end{eqnarray*}

In addition, for easy representation, we define $a_3=a_1$ and $a_4=a_2$. Then for given tuples 
$\mathbf{i}^{(l)} \in \mathcal{P}(k-1,a_l-1)$ with $l=1,2,3,4$, 
we define the tuple 
$$(\mathbf{i}^{(1)},{\mathbf{i}}^{(2)},\mathbf{i}^{(3)},{\mathbf{i}}^{(4)} )=(i_1^{(1)},\ldots,i_{a_1-1}^{(1)}, {i}_1^{(2)},\ldots,{i}_{a_2-1}^{(2)}, {i}_1^{(3)},\ldots,{i}_{a_1-1}^{(3)}, {i}_1^{(4)},\ldots,{i}_{a_2-1}^{(4)}),$$ and let  $S(\mathbf{i}^{(1)},{\mathbf{i}}^{(2)}, \mathbf{i}^{(3)},{\mathbf{i}}^{(4)})$ be a collection of $(\mathbf{i}^{(1)},{\mathbf{i}}^{(2)},\mathbf{i}^{(3)},{\mathbf{i}}^{(4)} )$ where $\mathbf{i}^{(l)} \in \mathcal{P}(k-1,a_l-1)$ with $l=1,2,3,4$. Then we can write 
\begin{align*}
	\mathbb{T}_{k,a_{1},a_{2}}^2 =\sum_{ \substack{S(\mathbf{i}^{(1)},{\mathbf{i}}^{(2)}, \mathbf{i}^{(3)}, \mathbf{i}^{(4)});\\ (j_1,j_2), (j_3,j_4), (j_5,j_6), (j_7,j_8) \in \mathcal{J} } } \prod_{l=1}^2c(n,a_l)\, \mathbb{X}( k,\mathbf{i}^{(l)}, j_{2l-1},j_{2l}: l=1,2,3,4),  
\end{align*} where we define
\begin{align*}
&~\mathbb{X}( k,\mathbf{i}^{(l)}, j_{2l-1},j_{2l}: l=1,2,3,4)\notag \\
=&~	\mathrm{E}\Big(\prod_{t=1}^4 x_{k,j_t}\Big) \mathrm{E}\Big(\prod_{t=5}^8 x_{k,j_t} \Big) \prod_{l=1}^4\prod_{t=1}^{a_l-1} x_{i_t^{(l)}, \, j_{2l-1}}x_{i_t^{(l)},\, j_{2l}}. 
\end{align*}

 Recall the definitions at the beginning of Section \ref{par:notationindpcond}. $\{\mathbf{i}^{(1)}\} = \{\mathbf{i}^{(2)}\}$ represents that the two tuples have the same elements without order. 
 We next decompose $S(\mathbf{i}^{(1)},{\mathbf{i}}^{(2)})$ into two parts: the collection $S(\mathbf{i}^{(1)},{\mathbf{i}}^{(2)},1)$ contains the tuples $(\mathbf{i}^{(1)},{\mathbf{i}}^{(2)} )$ satisfying $\{\mathbf{i}^{(1)}\} \neq \{\mathbf{i}^{(2)}\}$, and the collection $S(\mathbf{i}^{(1)},{\mathbf{i}}^{(2)},2)$ contains the tuples $(\mathbf{i}^{(1)},{\mathbf{i}}^{(2)} )$ satisfying $\{\mathbf{i}^{(1)}\} = \{\mathbf{i}^{(2)}\}$. Then we can write $\mathbb{T}_{k,a_{1},a_{2}}=\sum_{v=1}^2\mathbb{T}_{k,a_{1},a_{2},v}$, where 
\begin{align*}
	\mathbb{T}_{k,a_{1},a_{2},v}=\sum_{\substack{ S(\mathbf{i}^{(1)},\mathbf{i}^{(2)},v); \\  (j_1,j_2), (j_3,j_4)\in \mathcal{J} }  }\Big\{\prod_{l=1}^2c(n,a_l)\Big\}^{1/2}\times  \mathbb{X}( k,\mathbf{i}^{(l)}, j_{2l-1},j_{2l}: l=1,2). 
\end{align*}
In addition, for $v=1,2$, we let  the collection $S(\mathbf{i}^{(1)},{\mathbf{i}}^{(2)},\mathbf{i}^{(3)},{\mathbf{i}}^{(4)},v,v)$ contain the tuples $(\mathbf{i}^{(1)},{\mathbf{i}}^{(2)},\mathbf{i}^{(3)},{\mathbf{i}}^{(4)})$ such that $ (\mathbf{i}^{(1)},{\mathbf{i}}^{(2)})\in S(\mathbf{i}^{(1)},{\mathbf{i}}^{(2)},v)$ and $ (\mathbf{i}^{(3)},{\mathbf{i}}^{(4)})\in S(\mathbf{i}^{(3)},{\mathbf{i}}^{(4)},v)$. It follows that for $v=1,2$, we can write
\begin{align}
	\mathbb{T}_{k,a_{1},a_{2},v}^2 = \sum_{ \substack{ S(\mathbf{i}^{(1)},{\mathbf{i}}^{(2)},\mathbf{i}^{(3)},{\mathbf{i}}^{(4)},v,v);\\(j_1,j_2), (j_3,j_4), (j_5,j_6), (j_7,j_8)\in \mathcal{J}  }} \prod_{l=1}^2c(n,a_l)\label{eq:tbbavsq} \\
	\times \mathbb{X}( k,\mathbf{i}^{(l)}, j_{2l-1},j_{2l}: l=1,2,3,4). \notag
\end{align}

We next define some notation on  the $j$ indexes. Given a tuple $(j_{t_1},j_{t_2},j_{t_3},j_{t_4})$, we write its corresponding ordered version as
\begin{eqnarray}\label{eq:orderversion}
	(\tilde{j}_{t_1}, \tilde{j}_{t_2}, \tilde{j}_{t_3}, \tilde{j}_{t_4})\quad \text{satisfying} \quad \tilde{j}_{t_1}\leq \tilde{j}_{t_2}\leq \tilde{j}_{t_3}\leq \tilde{j}_{t_4}. 
\end{eqnarray}
 Given the ordered indexes, we define the maximum distance between indexes in the given tuple as $\mathbb{D}_{M}(j_{t_1},j_{t_2},j_{t_3},j_{t_4})= \max \{ \tilde{j}_{t_2}-\tilde{j}_{t_1}, \tilde{j}_{t_3}-\tilde{j}_{t_2}, \tilde{j}_{t_4}-\tilde{j}_{t_3}\}$. For the simplicity of  presentation later, for tuples $(j_1,j_2), (j_3,j_4), \allowbreak (j_5,j_6), (j_7,j_8)\in \mathcal{J}$,  we further define
\begin{eqnarray}\label{eq:distancedef}
	\kappa_1=\mathbb{D}_{M}(j_1,j_2,j_3,j_4), &&\kappa_2=\mathbb{D}_{M}(j_5,j_6,j_7,j_8),  \\ 
	\kappa_3=\mathbb{D}_{M}(j_1,j_2,j_5,j_6) &&\kappa_4=\mathbb{D}_{M}(j_1,j_2,j_7,j_8).  \notag
\end{eqnarray}
In the following discussion, to quantitatively evaluate the distances in \eqref{eq:distancedef}, we introduce a threshold $D_0$ below. In particular, given small positive constants $\mu$ and $ \epsilon$, and $\delta$ in Condition \ref{cond:alphamixing}, we define
\begin{eqnarray} \label{eq:thresholdd}
	D_0 = \frac{-(2+\epsilon)(8+\mu) \log p}{\epsilon \log \delta}, 
\end{eqnarray} which will be used as discussed at the beginning of this section on Page \pageref{sec:pfvartaamigixing}. 


\medskip
\subparagraph*{Proof}  
We present the proof of $\mathrm{var}(\mathbb{T}_{k,a_{1},a_{2}})=O(n^{-2}p^{-1}\log^3p)$ based on the notation above. 
 Note that  we can write $\mathbb{T}_{k,a_{1},a_{2}}=\sum_{v=1}^2\mathbb{T}_{k,a_{1},a_{2},v}$. By the Cauchy-Schwarz inequality,  we know it suffices to show $\mathrm{var}(\mathbb{T}_{k,a_{1},a_{2},v})=O(n^{-2}p^{-1}\log^3p)$  for $v=1,2$ respectively.  

\smallskip

\subparagraph*{Step I: $\mathrm{var}(\mathbb{T}_{k,a_{1},a_{2},1})=O(n^{-2}p^{-1}\log^3p)$}

By the definition of $\mathbb{T}_{k,a_{1},a_{2},1}$, we have $\{\mathbf{i}^{(1)}\} \neq \{\mathbf{i}^{(2)}\}$ for $(\mathbf{i}^{(1)},\mathbf{i}^{(2)})\in S(\mathbf{i}^{(1)},{\mathbf{i}}^{(2)},1)$. Suppose, without loss of generality, that index $i\in \{\mathbf{i}^{(1)}\}$  but $i \not \in \{\mathbf{i}^{(2)}\}$. Then under $H_0$, 
\begin{eqnarray}
&& \mathrm{E}\{\mathbb{X}( k,\mathbf{i}^{(l)}, j_{2l-1},j_{2l}: l=1,2,3,4)\}\label{eq:exptermzeroclt1} \\
&=&\mathrm{E}(x_{i,j_1}x_{i,j_2}) \times \mathrm{E}(\text{other terms})=0.	 \notag
\end{eqnarray}
Therefore $\mathrm{E}(\mathbb{T}_{k,a_{1},a_{2},1})=0$ and $\mathrm{var}(\mathbb{T}_{k,a_{1},a_{2},1})=\mathrm{E}(\mathbb{T}_{k,a_{1},a_{2},1}^2)$.

By \eqref{eq:tbbavsq}, we have
\begin{align*}
	\mathbb{T}_{k,a_{1},a_{2},1}^2=\sum_{\substack{ S(\mathbf{i}^{(1)},\mathbf{i}^{(2)},1,1); \\  (j_1,j_2), (j_3,j_4), (j_5,j_6), (j_7,j_8)\in \mathcal{J} }  }\prod_{l=1}^2c(n,a_l) \notag \\ \times  \mathbb{X}( k,\mathbf{i}^{(l)}, j_{2l-1},j_{2l}: l=1,2,3,4).
\end{align*}
To prove $\mathrm{var}(\mathbb{T}_{k,a_{1},a_{2},1})=O(n^{-2}p^{-1}\log^3p)$, we will next show that for given $(j_1,j_2), (j_3,j_4), (j_5,j_6), (j_7,j_8)\in \mathcal{J}$, 
\begin{align}
\mathrm{E}\Big\{ \sum_{S(\mathbf{i}^{(1)},\mathbf{i}^{(2)},1,1)} \mathbb{X}( k,\mathbf{i}^{(l)}, j_{2l-1},j_{2l}: l=1,2,3,4)\Big\}=O(n^{a_1+a_2-2}); \label{eq:expcltp1nor}
\end{align} and for given $(\mathbf{i}^{(1)},\mathbf{i}^{(2)},\mathbf{i}^{(3)},\mathbf{i}^{(4)})\in S(\mathbf{i}^{(1)},\mathbf{i}^{(2)},1,1)$,
\begin{eqnarray}
\quad\quad \ &&\mathrm{E}\Big\{  \sum_{\substack{(j_1,j_2), (j_3,j_4),\\ (j_5,j_6), (j_7,j_8)\in \mathcal{J}}} 	 \mathbb{X}( k,\mathbf{i}^{(l)}, j_{2l-1},j_{2l}: l=1,2,3,4)\Big\} =O(p^{3}\log^3p).\label{eq:expcltppor2}
\end{eqnarray}Given \eqref{eq:expcltp1nor} and \eqref{eq:expcltppor2}, since $c(n,a_l)=\Theta(p^{-2}n^{-a_l})$, we can  obtain $\mathrm{E}(\mathbb{T}_{k,a_{1},a_{2},1}^2)=O(n^{-2}p^{-1}\log^3p)$. Thus to finish the proof, it remains to prove \eqref{eq:expcltp1nor} and \eqref{eq:expcltppor2}.

To prove \eqref{eq:expcltp1nor}, we claim that $\mathrm{E}\{  \mathbb{X}( k,\mathbf{i}^{(l)}, j_{2l-1},j_{2l}: l=1,2,3,4)\}= 0$ when  $|\cup_{l=1}^4 \{\mathbf{i}^{(l)}\}| > a_1+a_2-2$,  i.e., there exists one index only appears once in the four index sets $\{\mathbf{i}^{(l)}\}$, $l=1,\ldots,4$. Too see this, suppose an index $i\in \{\mathbf{i}^{(1)}\}$ but $i \not \in \{\mathbf{i}^{(2)}\}$, $i \not \in \{\mathbf{i}^{(3)}\}$ and $i \not \in \{\mathbf{i}^{(4)}\}$, then \eqref{eq:exptermzeroclt1} holds.  Therefore, $\mathrm{E}\{  \mathbb{X}( k,\mathbf{i}^{(l)}, j_{2l-1},j_{2l}: l=1,2,3,4)\}\neq  0$ only when
\begin{align}
	\Big|\cup_{l=1}^4 \{\mathbf{i}^{(l)} \}\Big|\leq a_1+a_2-2. \label{eq:totalsizeindesibound}
\end{align}
 By the boundedness of moments from Condition \ref{cond:finitemomt}, we know \eqref{eq:expcltp1nor} holds.

We next prove \eqref{eq:expcltppor2}. For given $(\mathbf{i}^{(1)},\mathbf{i}^{(2)},\mathbf{i}^{(3)},\mathbf{i}^{(4)})\in S(\mathbf{i}^{(1)},\mathbf{i}^{(2)},1,1)$, we know $\{\mathbf{i}^{(1)}\}\neq \{\mathbf{i}^{(2)}\}$ and $\{\mathbf{i}^{(3)}\}\neq \{\mathbf{i}^{(4)}\}$. Suppose, without loss of generality, there exists an index  $i\in \{\mathbf{i}^{(3)}\}$ and $i\not \in \{\mathbf{i}^{(4)}\}$. If $i\not \in \{\mathbf{i}^{(1)}\}$ and $i \not \in \{\mathbf{i}^{(2)}\}$,  similarly, \eqref{eq:exptermzeroclt1} holds. Then we consider $i \in \{\mathbf{i}^{(1)}\}$ or $i  \in \{\mathbf{i}^{(2)}\}$ in the following three cases.

\smallskip


\textbf{Case 1:} When $i \in \{ \mathbf{i}^{(1)}\}$ and $i  \not \in \{\mathbf{i}^{(2)}\}$,  we know
\begin{eqnarray}
	\quad \ & & \mathrm{E}\Big\{ \mathbb{X}( k,\mathbf{i}^{(l)}, j_{2l-1},j_{2l}: l=1,2,3,4) \Big\} \label{eq:step2case1onlyinone} \\
	\quad \ &=&\mathrm{E}\Big(\prod_{t=1}^4 x_{k,j_t}\Big) \times \mathrm{E}\Big(\prod_{t=5}^8 x_{k,j_t}\Big)\times \mathrm{E}\Big(\prod_{t=1,2,5,6}x_{i,j_t}\Big) \times \mathrm{E}(\mathrm{other\ terms}). \notag
\end{eqnarray} 
If $x_{i,j}$'s are independent as in Section \ref{par:proftaaindpcond}, 
we know $\eqref{eq:step2case1onlyinone}\neq 0$ only when $\{j_1,j_2\}=\{j_3,j_4\}=\{j_5,j_6\}=\{j_7,j_8\},$ which induces $|\{j_1,\ldots,j_8\}|=2$ and  $  \sum_{\substack{(j_1,j_2), (j_3,j_4), (j_5,j_6), (j_7,j_8)\in \mathcal{J}}}\allowbreak \mathrm{E}\{ \mathbb{X}( k,\mathbf{i}^{(l)}, j_{2l-1},j_{2l}: l=1,2,3,4)\}=O(p^2),$  i.e.,   \eqref{eq:expcltppor2} is obtained.   
 Under Condition  \ref{cond:alphamixing},  $x_{i,j}$'s may be no longer  independent, but as discussed at the beginning of Section \ref{sec:pfvartaamigixing},  we can still prove \eqref{eq:expcltppor2} similarly to the independence case.  
In particular, based on $D_0$ in \eqref{eq:thresholdd}, we evaluate \eqref{eq:step2case1onlyinone} by discussing the following three sub-cases (a)--(c).  
	\begin{enumerate}
	   \item[(a)] When both $(j_1,j_2,j_3,j_4)$ and  $(j_5,j_6,j_7,j_8)$ contain only two distinct indexes within   each tuple,  i.e., $|\{j_1,j_2,j_3,j_4\}|=|\{j_5,j_6,j_7,j_8\}|=2$, we consider without loss of generality that $j_1=j_3$, $j_2=j_4$, $j_5=j_7$, and $j_6=j_8$. Then
	   \begin{align*}
	   	\eqref{eq:step2case1onlyinone}=\mathrm{E}(x_{k,j_1}^2x_{k,j_2}^2 )\mathrm{E}(x_{k,j_5}^2x_{k,j_6}^2 )\mathrm{E}(x_{k,j_1}x_{k,j_2}x_{k,j_5} x_{k,j_6} )   \mathrm{E}(\mathrm{other\ terms}).
	   \end{align*}  
	      (a.1) If $(j_1,j_2,j_5,j_6)$    contains  two distinct indexes, i.e., $|\{j_1,j_2,j_5,j_6\}|=2$, we assume   without loss of generality that $j_1=j_5$ and $j_2=j_6 $. Then   $|\{j_1,\ldots,j_8\}|=2$ and in this case, the total number of distinct $j$ indexes is $O(p^2)$.  
	      \vspace{0.5em}
	      
	      (a.2) If $(j_1,j_2,j_5,j_6)$ contains at least three distinct indexes, that is,  $|\{j_1,j_2,j_5,j_6\}|\geq 3,$ we have $|\{\tilde{j}_1,\tilde{j}_2,\tilde{j}_5,\tilde{j}_6\}|\geq 3$, where  
	      $(\tilde{j}_1,\tilde{j}_2,\tilde{j}_5,\tilde{j}_6)$ denotes the ordered version of $(j_1,j_2,j_5,j_6)$ following the notation in \eqref{eq:orderversion}. Then we have $\mathrm{E}(x_{k,\tilde{j}_1}x_{k,\tilde{j}_2})\mathrm{E}(x_{k,\tilde{j}_5}x_{k,\tilde{j}_6} )=0$. 
	      Together with $\mathrm{E}(\mathbf{x})=\mathbf{0}$, we can write 
	      \begin{align}
	             	|\mathrm{E}(x_{1,j_1}x_{1,j_2}x_{1,j_5}x_{1,j_6}) | =&~ | \mathrm{cov} (x_{k,\tilde{j}_1}x_{k,\tilde{j}_2}\ , \ x_{k,\tilde{j}_5}x_{k,\tilde{j}_6} ) | \label{eq:e1256equivalentform} \\
	             	=&~| \mathrm{cov}(x_{k,\tilde{j}_1} \ , \  x_{k,\tilde{j}_2}x_{k,\tilde{j}_5}x_{k,\tilde{j}_6}) | \notag \\
	             	=&~| \mathrm{cov}(x_{k,\tilde{j}_1} x_{k,\tilde{j}_2}x_{k,\tilde{j}_5} \ , \  x_{k,\tilde{j}_6}) |.  \notag
	             \end{align}
	             Recall that $\kappa_3$ in  \eqref{eq:distancedef} represents the maximum distance between $(j_1,j_2,j_5,j_6).$
	             If $\kappa_3>D_0$, by Conditions \ref{cond:finitemomt} and    \ref{cond:alphamixing}, and the $\alpha$-mixing inequality in Lemma \ref{lm:mixingineq}, we know 
	      \begin{eqnarray*}
	          	  |\eqref{eq:step2case1onlyinone}|\leq C\times \eqref{eq:e1256equivalentform} \leq C \delta^{\frac{D_0 \epsilon}{2+\epsilon}} =O(p^{-(8+\mu)}).
	          \end{eqnarray*} If $\kappa_3\leq D_0$, the total number of distinct $j$ indexes is $O(pD_0^3)$. 
	      
	       \vspace{0.5em}

	    \item[(b)]   When both   $(j_1,j_2,j_3,j_4)$ and $(j_5,j_6,j_7,j_8)$ have at least 3 distinct elements, i.e., $|\{j_1,j_2,j_3,j_4\}|\geq 3$ and $|\{j_5,j_6,j_7,j_8\}|\geq 3$, following the notation in \eqref{eq:orderversion}, similarly to \eqref{eq:e1256equivalentform}, we can write
	    \begin{align}
	    	|\mathrm{E}(x_{k,j_1} x_{k,j_2} x_{k,j_3} x_{k,j_4})| =&~ | \mathrm{cov} (x_{k,\tilde{j}_1}x_{k,\tilde{j}_2}\ , \ x_{k,\tilde{j}_3}x_{k,\tilde{j}_4} ) | \label{eq:e1234equivalentform} \\
	             	=&~ | \mathrm{cov}(x_{k,\tilde{j}_1} \ , \  x_{k,\tilde{j}_2}x_{k,\tilde{j}_3}x_{k,\tilde{j}_4}) | \notag \\
	             	=&~ | \mathrm{cov}(x_{k,\tilde{j}_1} x_{k,\tilde{j}_2}x_{k,\tilde{j}_3} \ , \  x_{k,\tilde{j}_4}) |,\notag
	    \end{align} and
	    \begin{align}
	    	|\mathrm{E}(x_{k,j_5} x_{k,j_6} x_{k,j_7} x_{k,j_8})| =&~ | \mathrm{cov} (x_{k,\tilde{j}_5}x_{k,\tilde{j}_6}\ , \ x_{k,\tilde{j}_7}x_{k,\tilde{j}_8} ) | \label{eq:e5678equivalentform} \\
	             	=&~ | \mathrm{cov}(x_{k,\tilde{j}_5} \ , \  x_{k,\tilde{j}_6}x_{k,\tilde{j}_7}x_{k,\tilde{j}_8}) | \notag \\
	             	=&~ | \mathrm{cov}(x_{k,\tilde{j}_5} x_{k,\tilde{j}_6}x_{k,\tilde{j}_7} \ , \  x_{k,\tilde{j}_8}) |.\notag
	    \end{align}
	    When $\max\{\kappa_1,\kappa_2\}>D_0$  in this case, by Conditions \ref{cond:finitemomt} and    \ref{cond:alphamixing}, and the $\alpha$-mixing inequality,
	    \begin{eqnarray}
	    	&&\quad \quad |\eqref{eq:step2case1onlyinone}|\leq C\times \eqref{eq:e1234equivalentform}\times  \eqref{eq:e5678equivalentform} \leq C \delta^{\frac{D_0 \epsilon}{2+\epsilon}} =O(p^{-(8+\mu)}). \label{eq:vtaacase1b}
	    \end{eqnarray}
	    When $\max\{\kappa_1,\kappa_2\} \leq D_0$, by the definitions in \eqref{eq:distancedef}, we know   under this case, the indexes in $(j_1,j_2,j_3,j_4)$ are close to each other within the distance $D_0$, and  the indexes in $(j_5,j_6,j_7,j_8)$ are also close to each other  within the distance $D_0$. Then the total number of distinct indexes is $O(pD_0^3\times p D_0^3)=O(p^2D_0^6)$.
    \item[(c)] If only one of  $(j_1,j_2,j_3,j_4)$ and $(j_5,j_6,j_7,j_8)$ contains at least 3  distinct indexes, without loss of generality, we assume $|\{j_1,j_2,j_3,j_4\}|\geq 3$ and $|\{j_5,j_6,j_7,j_8\}|=2$. 
      When $\kappa_1\leq D_0$, the indexes in $(j_1,j_2,j_3,j_4)$  are close within distance $D_0$. As $(j_5,j_6,j_7,j_8)$ only contains 2 distinct indexes, the total number of distinct $j$ indexes is $O(p^3D_0^3)$. When $\kappa_1>D_0$, by Conditions \ref{cond:finitemomt} and    \ref{cond:alphamixing}, and the $\alpha$-mixing inequality, we know
    \begin{align}
	    	 |\eqref{eq:step2case1onlyinone}|\leq C\times \eqref{eq:e1234equivalentform} \leq C \delta^{\frac{D_0 \epsilon}{2+\epsilon}} =O(p^{-(8+\mu)}).\label{eq:vtaacase1c}
	    \end{align} 
	\end{enumerate}
	

\textbf{Case 2:} When $i \not \in \{ \mathbf{i}^{(1)}\}$ and $i \in \{\mathbf{i}^{(2)}\}$, we know similar conclusion holds by symmetricity.


\textbf{Case 3:} When $i  \in \{ \mathbf{i}^{(1)}\}$ and $i \in \{\mathbf{i}^{(2)}\}$, we have 
\begin{eqnarray}
	& &\mathrm{E}\Big\{ \mathbb{X}( k,\mathbf{i}^{(l)}, j_{2l-1},j_{2l}: l=1,2,3,4) \Big\} \label{eq:step2case2bothin} \\
	&=&\mathrm{E}\Big(\prod_{t=1}^4 x_{k,j_t} \Big)\times \mathrm{E}\Big( \prod_{t=5}^8 x_{k,j_t}\Big)\times \mathrm{E}\Big(\prod_{t=1}^6 x_{k,j_t} \Big)\times \mathrm{E}(\text{other terms}) \notag 
\end{eqnarray} 
Similarly to \textbf{Case 1} above, to evaluate \eqref              {eq:step2case2bothin}, we next discuss two sub-cases with $D_0$ in \eqref{eq:thresholdd}.

	\begin{enumerate}
	   \item[(a)] When both  $(j_1,j_2,j_3,j_4)$ and $(j_5,j_6,j_7,j_8)$ only contain 2 distinct indexes within each tuple, i.e., $|\{j_1,j_2,j_3,j_4\}|=|\{j_5,j_6,j_7,j_8\}|=2$, we assume  $j_1=j_3$, $j_2=j_4$, $j_5=j_7$ and $j_6=j_8$ without loss of generality. Then
	    \begin{align*}
	   	\eqref{eq:step2case2bothin}=\mathrm{E}(x_{k,j_1}^2x_{k,j_2}^2)\mathrm{E}(x_{k,j_5}^2x_{k,j_6}^2 )\mathrm{E}(x_{i,j_1}^2x_{i,j_2}^2x_{i,j_5}x_{i,j_6})  \mathrm{E}(\mathrm{other\ terms}).
	   \end{align*} 
	     Following the notation in \eqref{eq:orderversion}, when $\tilde{k}_3^*:=\min\{\tilde{j}_2-\tilde{j}_1, \tilde{j}_5-\tilde{j}_2, \tilde{j}_6-\tilde{j}_5 \}<D_0$,  the total number of distinct $j$   indexes is $O(p^3D_0)$.
	     When $\tilde{k}_3^*>D_0$, by Conditions \ref{cond:finitemomt},   \ref{cond:alphamixing}, and the $\alpha$-mixing inequality,	     
	         \begin{align*}
	       	& ~|\mathrm{E}(x_{1,\tilde{j}_1}^2 x_{1,\tilde{j}_2}^2 x_{1,\tilde{j}_5}x_{1,\tilde{j}_6})| \\
	       	=& ~|\mathrm{cov}(x_{1,\tilde{j}_1}^2\ ,\  x_{1,\tilde{j}_2}^2 x_{1,\tilde{j}_5}x_{1,\tilde{j}_6}) + \mathrm{E}(x_{1,\tilde{j}_1}^2) \mathrm{cov} (x_{1,\tilde{j}_2}^2 \ ,\ x_{1,\tilde{j}_5}x_{1,\tilde{j}_6}) \notag \\
	       	& +[\mathrm{E}(x_{1,\tilde{j}_1}^2)]^2  \mathrm{cov}(x_{1,\tilde{j}_5}\ , \ x_{1,\tilde{j}_6})|\\
	       	\leq & ~ C \delta^{\frac{D_0 \epsilon}{2+\epsilon}}=O(p^{-(8+\mu)}).
				\end{align*} 
				
            \item[(b)] If at least one of $(j_1,j_2,j_3,j_4)$ and $(j_5,j_6,j_7,j_8)$ has at least 3 distinct indexes within the tuple,  it means that $|\{j_1,j_2,j_3,j_4\}|\geq 3$ or $|\{j_5,j_6,j_7,j_8\}|\geq 3$.  Similarly to  \eqref{eq:vtaacase1b} and \eqref{eq:vtaacase1c}, we know that when $\max\{\kappa_1,\kappa_2\}>D_0$, $| 	\eqref{eq:step2case2bothin}|=O(p^{-(8+\mu)})$; when $\max\{\kappa_1,\kappa_2\}\leq D_0$, the total number of distinct $j$ indexes is $O(p^3D_0^3)$.
           \end{enumerate}

Combining Cases 1--3 discussed above, we obtain
\begin{align*}
	&~\mathrm{E}\Big\{  \sum_{\substack{(j_1,j_2), (j_3,j_4),\\ (j_5,j_6), (j_7,j_8)\in \mathcal{J}}} 	 \mathbb{X}( k,\mathbf{i}^{(l)}, j_{2l-1},j_{2l}: l=1,2,3,4)\Big\} \notag \\
	=&~O(p^3 D_0^3)+\sum_{\substack{(j_1,j_2), (j_3,j_4), (j_5,j_6), (j_7,j_8)\in \mathcal{J}}}O(p^{-(8+\mu)})\notag \\
	=&~O(p^3 \log ^3 p)+ p^8 O(p^{-(8+\mu)})= O(p^3 \log ^3 p),
\end{align*} where we use $\mu>0$ and $D_0=O(\log p)$ by \eqref{eq:thresholdd}. Thus \eqref{eq:expcltppor2} is proved. 




\medskip

\subparagraph*{Step II: $\mathrm{var}(\mathbb{T}_{k,a_{1},a_{2},2})=O(n^{-2}p^{-1}\log^3p)$}
Recall that $\mathbb{T}_{k,a_{1},a_{2},2}$ is constructed from $(\mathbf{i}^{(1)},\mathbf{i}^{(2)})\in S(\mathbf{i}^{(1)},{\mathbf{i}}^{(2)},2)$, where 
 $\{\mathbf{i}^{(1)}\} = \{\mathbf{i}^{(2)}\}$. 
As $\{\mathbf{i}^{(1)}\} = \{\mathbf{i}^{(2)}\}$ happens only when $a_1=a_2$, so it remains to consider $a_1=a_2=a$ for some integer $a$ below. It follows that 
$\mathrm{E}\{\mathbb{X}( k,\mathbf{i}^{(l)}, j_{2l-1},j_{2l}: l=1,2) \}=\{\mathrm{E}( \prod_{t=1}^4 x_{1,j_t}) \}^a,
$ then
\begin{align*}
\mathrm{E}(\mathbb{T}_{k,a_{1},a_{2},2})=\sum_{\substack{ S(\mathbf{i}^{(1)},\mathbf{i}^{(2)},2); \\  (j_1,j_2), (j_3,j_4)\in \mathcal{J} }  }\Big\{\prod_{l=1}^2c(n,a_l)\Big\}^{1/2}\times   \Big\{\mathrm{E}\Big( \prod_{t=1}^4 x_{1,j_t}\Big) \Big\}^a, 
\end{align*}and 
\begin{align*}
	\{\mathrm{E}(\mathbb{T}_{k,a_{1},a_{2},2})\}^2=\sum_{\substack{ S(\mathbf{i}^{(1)},\mathbf{i}^{(2)},\mathbf{i}^{(3)},\mathbf{i}^{(4)},2,2); \\  (j_1,j_2), (j_3,j_4),(j_5,j_6), (j_7,j_8)\in \mathcal{J} }  }\prod_{l=1}^2c(n,a_l)   \Big\{\mathrm{E}\Big( \prod_{t=1}^4 x_{1,j_t}\Big)\mathrm{E}\Big( \prod_{t=5}^8 x_{1,j_t}\Big) \Big\}^a.
\end{align*}
Moreover, by \eqref{eq:tbbavsq}, we know  $\mathbb{T}_{k,a_{1},a_{2},2}^2$ is a summation over $(\mathbf{i}^{(1)},\mathbf{i}^{(2)},\mathbf{i}^{(3)},\mathbf{i}^{4)}) \in S(\mathbf{i}^{(1)},\mathbf{i}^{(2)},\mathbf{i}^{(3)},\mathbf{i}^{(4)},2,2)$, where $\{\mathbf{i}^{(1)}\}=\{\mathbf{i}^{(2)}\}$ and $\{\mathbf{i}^{(3)}\}=\{\mathbf{i}^{(4)}\}$ by the construction. 
We further define $S(\mathbf{i}^{(1)},\mathbf{i}^{(2)},\mathbf{i}^{(3)},\mathbf{i}^{(4)},2,2,q)$ to be the collection of tuples $(\mathbf{i}^{(1)},\mathbf{i}^{(2)},\mathbf{i}^{(3)},\mathbf{i}^{4)})$ such that $|\{\mathbf{i}^{(1)}\} \cap \{\mathbf{i}^{(3)}\}|=q$, where $0\leq q \leq a-1$. Then we write  $\mathbb{T}_{k,a_{1},a_{2},2}^2=\sum_{q=0}^{a-1}\mathbb{T}_{k,a_{1},a_{2},2, (q)}^2$, where we define
\begin{align*}
	\mathbb{T}_{k,a_{1},a_{2},2, (q)}^2= \sum_{\substack{ S(\mathbf{i}^{(1)},\mathbf{i}^{(2)},\mathbf{i}^{(3)},\mathbf{i}^{(4)},2,2,q); \\  (j_1,j_2), (j_3,j_4),(j_5,j_6), (j_7,j_8)\in \mathcal{J} }  }\prod_{l=1}^2c(n,a_l)\notag \\
	\times  \mathbb{X}( k,\mathbf{i}^{(l)}, j_{2l-1},j_{2l}: l=1,2,3,4).
\end{align*} 
In particular, when $|\{\mathbf{i}^{(1)}\} \cap \{\mathbf{i}^{(3)}\}|=q$, 
\begin{align*}
	\mathrm{E}\Big\{  \mathbb{X}( k,\mathbf{i}^{(l)}, j_{2l-1},j_{2l}: l=1,2,3,4) \Big\}=\Big\{\mathrm{E}\Big( \prod_{t=1}^4 x_{1,j_t}\Big)\mathrm{E}\Big( \prod_{t=5}^8 x_{1,j_t}\Big) \Big\}^{a-q} \Big\{ \prod_{t=1}^8 x_{1,j_t}\Big\}^q.
\end{align*}
Therefore, for $a_1=a_2=a$, 
\begin{eqnarray*}
\mathrm{var}(\mathbb{T}_{k,a_{1},a_{2},2})&=&\mathrm{E}(\mathbb{T}_{k,a_{1},a_{2},2}^2)-\{\mathrm{E}(\mathbb{T}_{k,a_{1},a_{2},2})\}^2 \notag \\
&=& \sum_{q=0}^{a-1} \mathrm{E}(\mathbb{T}_{k,a_{1},a_{2},2, (q)}^2)-\{\mathrm{E}(\mathbb{T}_{k,a_{1},a_{2},2})\}^2 \notag \\
&=&  \sum_{q=1}^{a-1}\sum_{\substack{ S(\mathbf{i}^{(1)},\mathbf{i}^{(2)},\mathbf{i}^{(3)},\mathbf{i}^{(4)},2,2,q)}  } \prod_{l=1}^2c(n,a_l)\times \mathbb{D}_{k,a,a,2,q}, \notag 
\end{eqnarray*}
where we define
\begin{align*}
	\mathbb{D}_{k,a,a,2,q}=&\sum_{ \substack{ (j_1,j_2), (j_3,j_4),\\(j_5,j_6), (j_7,j_8)\in \mathcal{J}} }\Big\{\mathrm{E}\Big( \prod_{t=1}^4 x_{1,j_t}\Big)\mathrm{E}\Big( \prod_{t=5}^8 x_{1,j_t}\Big) \Big\}^{a-q} \notag \\
	 & \times \Biggr[  \Big\{ \mathrm{E}\Big(\prod_{t=1}^8 x_{1,j_t}\Big)\Big\}^q - \Big\{\mathrm{E}\Big( \prod_{t=1}^4 x_{1,j_t}\Big)\mathrm{E}\Big( \prod_{t=5}^8 x_{1,j_t}\Big) \Big\}^q \Biggr],
\end{align*} and use $\mathbb{D}_{k,a,a,2,q}=0$ when $q=0$. 
By the construction,  we know 
the total number of  tuples in the collection $S(\mathbf{i}^{(1)},\mathbf{i}^{(2)},\mathbf{i}^{(3)},\mathbf{i}^{(4)},2,2,q)$ is bounded by $Cn^{2(a-1)-q}$, that is, 
for some constant $C$, 
\begin{align}
	\sum_{S(\mathbf{i}^{(1)},\mathbf{i}^{(2)},\mathbf{i}^{(3)},\mathbf{i}^{(4)},2,2,q) } 1\leq Cn^{2(a-1)-q}.  \label{eq:snumberbd} 
\end{align} 
Since $c(n,a)=\Theta(p^{-2}n^{-a})$, to prove $\mathrm{var}(\mathbb{T}_{k,a_{1},a_{2},2}^2)=O(n^{-2}p^{-1}\log^3p)$, it suffices to show for given tuple $(\mathbf{i}^{(1)},\mathbf{i}^{(2)},\mathbf{i}^{(3)},\mathbf{i}^{(4)})$, $\mathbb{D}_{k,a_1,a_2,2,q}=O(p^3\log^3 p)$ for $1\leq q\leq a-1$.

By Condition \ref{cond:finitemomt} and Lemma \ref{lemma:prodcutcom} (on Page  \pageref{lemma:prodcutcom}), for $1\leq q\leq a-1$, 

\begin{align*}
	|\mathbb{D}_{k,a,a,2,q}|\leq & ~C\sum_{\substack{ (j_1,j_2),(j_3,j_4),\\  (j_5,j_6),(j_7,j_8)\in \mathcal{J} } } \Biggr| \mathrm{E}\Big(\prod_{t=1}^4 x_{1,j_t}\Big)\Biggr|  \times \Biggr| \mathrm{E}\Big(\prod_{t=5}^8x_{1,j_t}\Big)\Biggr|\notag \\
	&~\times \Biggr| \mathrm{E}\Big( \prod_{t=1}^4x_{1,j_t}\Big)\times \mathrm{E}\Big( \prod_{t=5}^8x_{1,j_t}\Big)-  \mathrm{E}\Big(\prod_{t=1}^8 x_{1,j_t}\Big)\Biggr|.
\end{align*}
 To evaluate $\mathbb{D}_{k,a,a,2,q}$, we next  discuss several cases, based on the notation $\kappa_1,\ldots,\kappa_4$ in \eqref{eq:distancedef},
  and $D_0$ in \eqref{eq:thresholdd}.
	    \begin{enumerate}   
	     \item[(a)] When both tuples $(j_1,j_2,j_3,j_4)$ and $(j_5,j_6,j_7,j_8)$ contain only  two distinct indexes, i.e., $|\{j_1,j_2,j_3,j_4\}|=|\{j_5,j_6,j_7,j_8\}|=2$, we assume without loss of generality that   $j_1=j_3$, $j_2=j_4$, $j_5=j_7$ and $j_6=j_8$. Then $\mathrm{E}(\prod_{t=1}^4 x_{1,{j}_t} ) = \mathrm{E}(x_{1,{j}_1}^2x_{1,{j}_2}^2),$	 $ \mathrm{E}(\prod_{t=5}^8 x_{1,{j}_t})= \mathrm{E}(x_{1,{j}_5}^2 x_{1,{j}_6}^2)$ and $\mathrm{E}(\prod_{t=1}^8  x_{1,j_t})   =\mathrm{E}(x_{1,j_1}^2 x_{1,j_2}^2 x_{1,j_5}^2x_{1,j_6}^2).$
	     Following the notation in \eqref{eq:orderversion}, let  $(\tilde{j}_1 \leq \tilde{j}_2 \leq \tilde{j}_5 \leq \tilde{j}_6)$ be the ordered version of $(j_1,j_2,j_5,j_6)$.  When $\min\{\tilde{j}_2-\tilde{j}_1,\tilde{j}_5-\tilde{j}_2, \tilde{j}_6-\tilde{j}_5 \}\leq D_0$, the total number of distinct $j$ indexes is $O(p^3D_0)$. When $\min\{\tilde{j}_2-\tilde{j}_1,\tilde{j}_5-\tilde{j}_2, \tilde{j}_6-\tilde{j}_5 \}> D_0$, by  Conditions \ref{cond:finitemomt} and    \ref{cond:alphamixing}, and the $\alpha$-mixing inequality in Lemma  \ref{lm:mixingineq},  
	     \begin{align*}
	     		&~\Big|	\mathrm{E}\Big(\prod_{t=1}^4 x_{1,{j}_t} \Big) 	\mathrm{E}\Big(\prod_{t=5}^8 x_{1,{j}_t} \Big)-	\mathrm{E}\Big(\prod_{t=1}^8 x_{1,{j}_t} \Big)  \Big| \notag \\
	     		=&~\Big| \mathrm{E}(x_{1,{j}_1}^2x_{1,{j}_2}^2)\mathrm{E}(x_{1,{j}_5}^2 x_{1,{j}_6}^2)-\mathrm{E}(x_{1,j_1}^2 x_{1,j_2}^2 x_{1,j_5}^2x_{1,j_6}^2)\Big|\notag \\
	     		\leq &~ C\delta^{\frac{D_0 \epsilon}{2+\epsilon}}=O(p^{-(8+\mu)}).
	     	\end{align*}

    \item[(b)] When both $(j_1,j_2,j_3,j_4)$ and $(j_5,j_6,j_7,j_8)$ contain at least 3 distinct indexes, i.e., $|\{j_1,j_2,j_3,j_4\}|\geq 3$ and $|\{j_5,j_6,j_7,j_8\}|\geq 3$,  we know  similarly \eqref{eq:e1234equivalentform} and \eqref{eq:e5678equivalentform} hold. 
    	    When $\max\{\kappa_1,\kappa_2\}>D_0$,  by Conditions \ref{cond:finitemomt} and    \ref{cond:alphamixing}, and the $\alpha$-mixing inequality in Lemma \ref{lm:mixingineq}, we obtain
    	    \begin{align*}
    	    	|\mathbb{D}_{k,a_1,a_2,2,q}|\leq C \eqref{eq:e1234equivalentform} \times \eqref{eq:e5678equivalentform} \leq  C\delta^{\frac{D_0 \epsilon}{2+\epsilon}}=O\{p^{-(8+\mu)}\}.
    	    \end{align*}
    	    When $\max\{\kappa_1,\kappa_2\} \leq D_0$, by the definitions in \eqref{eq:distancedef}, we know   under this case the indexes in  $(j_1,j_2,j_3,j_4)$ are close to each other within the distance $D_0$, and the indexes in $(j_5,j_6,j_7,j_8)$ are also close  to each other within the distance $D_0$. Then the total number of distinct $j$ indexes is $O(pD_0^3\times p D_0^3)=O(p^2D_0^6)$.
    	    \item[(c)] When only one of  $(j_1,j_2,j_3,j_4)$ and $(j_5,j_6,j_7,j_8)$ contains at least 3  distinct indexes, without loss of generality, we assume $|\{j_1,j_2,j_3,j_4\}|\geq 3$ and $|\{j_5,j_6,j_7,j_8\}|=2$.
    	     Recall $\kappa_1$ defined in \eqref{eq:distancedef}. When $\kappa_1\leq D_0$, the indexes in $(j_1,j_2,j_3,j_4)$  are close within distance $D_0$. As $(j_5,j_6,j_7,j_8)$ only contains 2 distinct indexes, the total number of distinct $j$ indexes is $O(p^3D_0^3)$. When $\kappa_1>D_0$, by Conditions \ref{cond:finitemomt} and    \ref{cond:alphamixing}, and the $\alpha$-mixing inequality in Lemma \ref{lm:mixingineq}, we know similarly   \eqref{eq:e1234equivalentform} holds, and
    	     \begin{align*}
    	    	|\mathbb{D}_{k,a_1,a_2,2,q}|\leq C \eqref{eq:e1234equivalentform} \leq  C\delta^{\frac{D_0 \epsilon}{2+\epsilon}}=O(p^{-(8+\mu)}).
    	    \end{align*}
    \end{enumerate} 	    
In summary,
\begin{align}
	|\mathbb{D}_{k,a_1,a_2,2,q}|=p^8\times O(p^{-(8+\mu)})+O(p^3D_0^3)=O(p^3\log^3p). \label{eq:orderpdttaa}
\end{align} 
Thus we obtain that for given $(\mathbf{i}^{(1)},\mathbf{i}^{(2)},\mathbf{i}^{(3)},\mathbf{i}^{(4)})$, $\mathbb{D}_{k,a_1,a_2,2,q}=O(p^3\log^3 p)$. Combined with \eqref{eq:snumberbd},    $\mathrm{var}(\mathbb{T}_{k,a_{1},a_{2},2}^2)=O(n^{-2}p^{-1}\log^3p)$ follows. 

Combining the results in \textit{Step I}  and \textit{Step II}  above, we obtain $\mathrm{var}(\mathbb{T}_{k,a_{1},a_{2}})=O(n^{-2}p^{-1}\log^3p)$, and thus Lemma \ref{lm:targetorder} is proved under Condition \ref{cond:alphamixing}. 



\smallskip

\paragraph{Proof under Condition \ref{cond:highordmominter}}\label{par:pfmomentvartaa}
In this section, we prove Lemma \ref{lm:targetorder}  by substituting Condition \ref{cond:alphamixing} with Condition  \ref{cond:highordmominter}. 
 Note that although the independence between $x_{i,j}$'s is assumed in Section \ref{par:proftaaindpcond}, it is only used to specify certain joint moments of $x_{i,j}$'s. Alternatively, Condition \ref{cond:ellpmoment} is assumed to obtain similar properties on the joint  moments, 
 and the proof follows similarly to that in  Section   \ref{par:proftaaindpcond}.

In particular,  we will prove that  $\mathrm{var}(\mathbb{T}_{k,a_{1},a_{2}})= O(n^{-3}+n^{-2}p^{-2})$ for two given finite integers $a_1$ and $a_2$ below.   
 Under $H_0$ and given  Condition \ref{cond:highordmominter}, as $j_1\neq j_2$ and $j_3\neq j_4$, we have $\mathrm{E}(x_{1,j_1} x_{1,j_2} x_{1,j_3} x_{1,j_4}) \neq 0$ only when $\{j_1, j_2\}=\{j_3, j_4\}$, and then  $\mathrm{E}(x_{1,j_1} x_{1,j_2} x_{1,j_3} x_{1,j_4})=\kappa_1\mathrm{E}(x_{1,j_1}^2)\mathrm{E}(x_{1,j_2}^2)$.  It follows that $\mathbb{T}_{k,a_1,a_2}=2c(n,a)\times \tilde{T}_{k,a_1,a_2}$, where   
$\tilde{T}_{k,a,a}=\kappa_1T_{k,a,a}$ with $T_{k,a,a}$ defined  in Section  \ref{par:proftaaindpcond}. To prove $\mathrm{var}(\mathbb{T}_{k,a,a})=o(n^{-2})$, it suffices to show that $\mathrm{var}(\tilde{T}_{k,a_1,a_2})=n^{a_1+a_2-2}p^4O(n^{-1}+p^{-2})$ as argued in Section  \ref{par:proftaaindpcond}.



Similarly to Section \ref{par:proftaaindpcond}, to show $\mathrm{var}(\tilde{T}_{k,a_1,a_2})=n^{a_1+a_2-2}p^4O(n^{-1}+p^{-2})$, we examine $\{ \mathrm{E}(\tilde{T}_{k,a_1,a_2})\}^2 $ and $\mathrm{E}(\tilde{T}_{k,a_1,a_2}^2)$ respectively.  
For $\mathrm{E}(\tilde{T}_{k,a_1,a_2})$, under Condition \ref{cond:highordmominter}, 
 similarly to \eqref{eq:etaaksq}, we know  
$\mathrm{E}\{  (\prod_{t=1}^{a_1-1} x_{i_t,j_1} x_{i_t,j_2})\times (\prod_{{t}=1}^{a_2-1}   x_{\tilde{i}_{{t}},j_1} x_{\tilde{i}_{{t}},j_2}) \} \neq 0$ only when  $\{\mathbf{i}\}=\{\tilde{\mathbf{i}}\}$. When $\{\mathbf{i}\}=\{\tilde{\mathbf{i}}\}$, we write $a_1=a_2=a$ for some $a$ and then $\mathrm{E}\{  (\prod_{t=1}^{a_1-1} x_{i_t,j_1} x_{i_t,j_2})\times (\prod_{{t}=1}^{a_2-1}   x_{\tilde{i}_{{t}},j_1} x_{\tilde{i}_{{t}},j_2}) \}=\{ \kappa_1\mathrm{E}(x_{1,j_1}^2)\mathrm{E}(x_{1,j_2}^2) \}^{a-1}.$ 
We thus have $\{\mathrm{E}(\tilde{T}_{k,a_1,a_2})\}^2=\{\kappa_1^{a}\mathrm{E}({T}_{k,a_1,a_2})\}^2$ with $T_{k,a_1,a_2}$ defined  in Section  \ref{par:proftaaindpcond}. Moreover, following \eqref{eq:taakindpsqexp} in Section \ref{par:proftaaindpcond}, we have 
\begin{align*}
	\mathrm{E}(\tilde{T}_{k,a_1,a_2}^2)=&\sum_{\substack{ \mathbf{i},\, \mathbf{m} \in \mathcal{P}(k-1,a_1-1); \\ \tilde{\mathbf{i}},\, \tilde{\mathbf{m}} \in \mathcal{P}(k-1,a_2-1) }} \, \sum_{\substack{1\leq j_1\neq j_2\leq p;\\ 1\leq j_3\neq j_4\leq p}} \tilde{Q}(\mathbf{i},\tilde{\mathbf{i}},\mathbf{m},\tilde{\mathbf{m}},\mathbf{j}). 
\end{align*} 
We further decompose $\mathrm{E}(\tilde{T}_{k,a_1,a_2}^2)=\mathrm{E}(\tilde{T}_{k,a_1,a_2}^2)_{(1)}+\mathrm{E}(\tilde{T}_{k,a_1,a_2}^2)_{(2)}$, where $\mathrm{E}(\tilde{T}_{k,a_1,a_2}^2)_{(1)}$ and $\mathrm{E}(\tilde{T}_{k,a_1,a_2}^2)_{(2)}$ are defined with  the same forms as $\mathrm{E}(T_{k,a_1,a_2}^2)_{(1)}$ and $\mathrm{E}(T_{k,a_1,a_2}^2)_{(2)}$ in Section \ref{par:proftaaindpcond}, respectively.  To prove $\mathrm{var}(\tilde{T}_{k,a_1,a_2})=n^{a_1+a_2-2}p^4O(n^{-1}+p^{-2})$, similarly to Section \ref{par:proftaaindpcond}, we derive $|\mathrm{E}(\tilde{T}_{k,a_1,a_2}^2)_{(1)}-\{\mathrm{E}(\tilde{T}_{k,a_1,a_2})\}^2|$ and $\mathrm{E}(T_{k,a_1,a_2}^2)_{(2)}$ respectively. 

\vspace{0.5em}
\subparagraph*{Step I: $|\mathrm{E}(\tilde{T}_{k,a_1,a_2}^2)_{(1)}-\{\mathrm{E}(\tilde{T}_{k,a_1,a_2})\}^2|$} 
By the forms of  $\mathrm{E}(\tilde{T}_{k,a_1,a_2}^2)_{(1)}$ and $\mathrm{E}(\tilde{T}_{k,a_1,a_2})$, we consider $ \{\mathbf{i}\}=\{\tilde{\mathbf{i}}\}$ and  $ \{\mathbf{m}\}=\{\tilde{\mathbf{m}}\}$ below. 
If $\{\mathbf{i}\}\cap \{\mathbf{m}\}=\emptyset$, $| \mathrm{E}\{\tilde{Q}(\mathbf{i},\tilde{\mathbf{i}},\mathbf{m},\tilde{\mathbf{m}},\mathbf{j})\} - \kappa_1^{2a}\prod_{t=1}^4\{\mathrm{E}(x_{1,j_t}^2)\}^a |= 0$ by Condition \ref{cond:highordmominter}; if $\{\mathbf{i}\}\cap \{\mathbf{m}\}\neq \emptyset$,  $|\{\mathbf{i}\}\cup \{\mathbf{m}\} |\leq a_1+a_2-2-1$, thus $|\mathrm{E}(\tilde{T}_{k,a_1,a_2}^2)_{(1)}-\{\mathrm{E}(\tilde{T}_{k,a_1,a_2})\}^2|=O(n^{a_1+a_2-3}p^4)$ by Condition \ref{cond:finitemomt}.  

\vspace{0.5em}
\subparagraph*{Step II: $\mathrm{E}(T_{k,a_1,a_2}^2)_{(2)}$}
We note that for $j_1\neq j_2$, $\mathrm{E}(x_{1,j_1}x_{1,j_2})=0$; and for any additional index $j_3$, we have $\mathrm{E}(x_{1,j_1}x_{1,j_2}x_{1,j_3})=0$ under Condition \ref{cond:highordmominter}. 
Thus \eqref{eq:jindxorderindep} and \eqref{eq:taasecondindpcond} still hold here, and we obtain $\mathrm{E}(T_{k,a_1,a_2}^2)_{(2)}=O(n^{a_1+a_2-2}p^2)$. 

\vspace{0.5em}

In summary, 
\begin{align*}
|\mathrm{var}(T_{k,a_1,a_2})|\leq &~|\mathrm{E}(\tilde{T}_{k,a_1,a_2}^2)_{(1)}-\{\mathrm{E}(\tilde{T}_{k,a_1,a_2})\}^2|+ |\mathrm{E}(T_{k,a_1,a_2}^2)_{(2)}| \notag \\
=&~n^{a_1+a_2-2}p^4O(n^{-1}+p^{-2}).	
\end{align*}
It follows that $\mathrm{var}(\sum_{k=1}^n \pi_{n,k}^2)=O(n^{-1}+p^{-2})$ by the argument at the beginning of Section \ref{sec:prooftargetorder}.  Therefore Lemma \ref{lm:targetorder} is proved.

\subsubsection{Proof of Lemma \ref{lm:secondgoaljointnormal} (on Page \pageref{lm:secondgoaljointnormal}, Section  \ref{sec:detailofjointnormal})} \label{sec:proofsecondgoaljointnormal}

By Lemma \ref{lm:cltabnkform}, 
\begin{align}
 	\sum_{k=1}^n\mathrm{E}(D_{n,k}^4)=\sum_{k=1}^n \sum_{{1\leq r_1,r_2,r_3,r_4\leq m}}\prod_{l=1}^4 t_{r_l}\times \mathrm{E}\Big(\prod_{l=1}^4 A_{n,k,a_{r_l}}\Big). \label{eq:dnk4all}
 \end{align} To prove  Lemma \ref{lm:secondgoaljointnormal}, it suffices to show that for given $1\leq k \leq n$ and $1\leq r_1,r_2,r_3,r_4\leq m$, we have $\mathrm{E}(\prod_{l=1}^4 A_{n,k,a_{r_l}})=O(n^{-2})$.

Similarly to Sections \ref{sec:proofvarianceorder} and \ref{sec:prooftargetorder} above, we first illustrate the proof of Lemma  \ref{lm:secondgoaljointnormal},  
when $x_{i,j}$'s are independent. 
Then in Section  \ref{pf:dnk4cond21pfcompl}, we prove Lemma  \ref{lm:secondgoaljointnormal} under Condition \ref{cond:alphamixing}. Last in Section \ref{par:ednkpfell},  we prove Lemma  \ref{lm:secondgoaljointnormal} under Condition \ref{cond:ellpmoment}. 

\paragraph{Proof illustration} \label{sec:pfindpdnk4}


In this section, we assume that $x_{i,j}$'s are independent and prove $\mathrm{E}(\prod_{l=1}^4 A_{n,k,a_{l}})=O(n^{-2})$ for  given integers $a_l$,   $l=1,\ldots,4$. 
 By Lemma  \ref{lm:cltabnkform}, when $k<a_l$,  $A_{n,k,a_l}=0$. We next focus on $\max_{1\leq l\leq 4}a_l\leq k\leq n$. By Lemma  \ref{lm:cltabnkform}, we have 
\begin{align}
\mathrm{E}\Big(\prod_{l=1}^4 A_{n,k,a_{l}}\Big)=&~ \Big\{\prod_{l=1}^4c(n,a_l)\Big\}^{1/2}	\sum_{ \substack{ \mathbf{i}^{(l)} \in \mathcal{P}(k-1,a_l-1),\, l=1,\ldots,4; \\ (j_1,j_2),(j_3,j_4), (j_5,j_6),(j_7,j_8)\in \mathcal{J} } }\label{eq:expank4indpill} \\
&~ \ Q^*( \mathbf{i}^{(1)}, \mathbf{i}^{(2)},\mathbf{i}^{(3)},\mathbf{i}^{(4)},\mathbf{j}_8),\notag 
\end{align}
where $\mathbf{i}^{(l)}=(i_1^{(l)},\ldots, i_{a_l-1}^{(l)})$, $l=1,\ldots,4$  represent the tuples satisfying $1\leq i_1^{(l)}\neq \ldots \neq i_{a_l-1}^{(l)}\leq n$; $\mathcal{J}=\{(j_1,j_2): 1\leq j_1\neq j_2 \leq p)$;    $\mathbf{j}_8$  represents the tuple $(j_1,j_2,j_3,j_4,j_5,j_6,j_7,j_8)$; and we define
\begin{align*}
Q^*( \mathbf{i}^{(1)}, \mathbf{i}^{(2)},\mathbf{i}^{(3)},\mathbf{i}^{(4)},\mathbf{j}_8)=\mathrm{E}\Big( \prod_{r=1}^8 x_{k,j_r}\Big)\,\mathrm{E}\Big( \prod_{t=1}^{a_l-1}\prod_{l=1}^4 x_{i_t^{(l)},j_{2l-1}} x_{i_t^{(l)},j_{2l}}\Big).
\end{align*}

We claim that $\mathrm{E}( \prod_{r=1}^8 x_{k,j_r})\neq 0$ only when 
\begin{align}
	|\{j_t: t=1,\ldots,8 \}|\leq 4.  \label{eq:eightjindpcond}
\end{align} If $|\{j_t: t=1,\ldots,8\}|\geq 5$, it implies that one of the $j$ index in  $\{j_t: t=1,\ldots,8\}$ only appears once. We assume without loss of generality that $j_1$ only appears once, i.e., $j_1\not \in \{j_t: t=2,\ldots,8\}$. Since $x_{k,j}$'s are independent, 
$\mathrm{E}( \prod_{r=1}^8 x_{k,j_r})=\mathrm{E}( x_{k,j_1})\mathrm{E}(\text{all the remaining terms}) =0. $ Thus \eqref{eq:eightjindpcond} is proved. 
Similarly to \eqref{eq:qtildezero} and \eqref{eq:sizecontroltaaindpcond}, we further know
$Q^*( \mathbf{i}^{(1)}, \mathbf{i}^{(2)},\mathbf{i}^{(3)},\mathbf{i}^{(4)},\mathbf{j}_8)\neq 0$ 
only when 
\begin{align}
	\Big|\bigcup_{l=1}^4 \{\mathbf{i}^{(l)}\} \Big|\leq \sum_{l=1}^4(a_l-1)/2 \label{eq:expank4indpcond}. 
\end{align}

In summary, combining \eqref{eq:eightjindpcond} and \eqref{eq:expank4indpcond}, we have
\begin{align*}
	\mathrm{E}\Big(\prod_{l=1}^4 A_{n,k,a_{l}}\Big)=O(p^{-4}n^{-\frac{1}{2}\sum_{l=1}^4a_l}n^{\frac{1}{2}\sum_{l=1}^4(a_l-1)}p^4)=O(n^{-2}). 
\end{align*}

\paragraph{Proof under Condition \ref{cond:alphamixing}}  \label{pf:dnk4cond21pfcompl}
Section \ref{sec:pfindpdnk4} proves Lemma \ref{lm:secondgoaljointnormal} when $x_{i,j}$'s are independent. In this section, we further prove Lemma \ref{lm:secondgoaljointnormal} under Condition \ref{cond:alphamixing}. 
We first illustrate the proof idea intuitively, which is similar to Sections \ref{par:complpflm1} and \ref{sec:pfvartaamigixing}. 
		Under Condition \ref{cond:alphamixing}, $x_{i,j}$'s may be  no longer independent, but the dependence between $x_{i,j_1}$ and $x_{i,j_2}$ degenerates exponentially with their distance $|j_1-j_2|$. To quantitatively examine $|j_1-j_2|$, we use the threshold of distance $D_0$ defined in \eqref{eq:thresholdd}. Intuitively, when $|j_1-j_2|>D_0$, $x_{i,j_1}$ and $x_{i,j_2}$ are ``asymptotically independent'' with similar properties to those under the independence case in Section \ref{sec:pfindpdnk4}.  The following proof  will provide comprehensive discussions based on  $D_0$.

We next present the detailed proof of Lemma \ref{lm:secondgoaljointnormal}. 
Note that to prove Lemma \ref{lm:secondgoaljointnormal}, by the analysis at the beginning of Section  \ref{sec:proofsecondgoaljointnormal}, 
it suffices to show $\mathrm{E}(\prod_{l=1}^4 A_{n,k,a_{l}})=O(n^{-2})$. Recall that  we can write   \eqref{eq:expank4indpill} and we have  
 $\prod_{l=1}^4c^{1/2}(n,a_l)=\Theta(p^{-4}n^{-\frac{1}{2}\sum_{l=1}^4a_l})$. It remains to show
\begin{eqnarray}
	&&   \sum_{ \substack{ \mathbf{i}^{(l)} \in \mathcal{P}(k-1,a_l-1),\, l=1,\ldots,4; \\ (j_1,j_2),(j_3,j_4), (j_5,j_6),(j_7,j_8)\in \mathcal{J} } } Q^*( \mathbf{i}^{(1)}, \mathbf{i}^{(2)},\mathbf{i}^{(3)},\mathbf{i}^{(4)},\mathbf{j}_8)=O(p^4n^{\frac{1}{2}\sum_{l=1}^4(a_l-1)}). \label{eq:secondgoalakn4}
\end{eqnarray}
To prove \eqref{eq:secondgoalakn4}, we show the order of \eqref{eq:secondgoalakn4} in $n$ and $p$ respectively in the following two steps.  



\smallskip

\subparagraph*{Step I: order of $n$}

We show for any fixed $\mathbf{j}_8=(j_1,j_2,j_3,j_4,j_5,j_6,j_7,j_8)$,
\begin{eqnarray}
	&&\quad \quad \Biggr|\sum_{ \substack{ \mathbf{i}^{(l)} \in \mathcal{P}(k-1,a_l-1),\, l=1,\ldots,4 } } Q^*( \mathbf{i}^{(1)}, \mathbf{i}^{(2)},\mathbf{i}^{(3)},\mathbf{i}^{(4)},\mathbf{j}_8)\Biggr|=O(n^{\frac{1}{2}\sum_{l=1}^4(a_l-1)}). \label{eq:n2amin1order}
\end{eqnarray} 
We note that $Q^*( \mathbf{i}^{(1)}, \mathbf{i}^{(2)},\mathbf{i}^{(3)},\mathbf{i}^{(4)},\mathbf{j}_8) \neq 0$ only if \eqref{eq:expank4indpcond} holds. Too see this, suppose one index $i_1$ only appears once in the four sets $\{ \mathbf{i}^{(1)}\},\{ \mathbf{i}^{(2)} \},\{ \mathbf{i}^{(3)}\},\{ \mathbf{i}^{(4)}\} $. For example $i_1\in \{ \mathbf{i}^{(1)}\}$, but $i_1\not \in \cup_{l=2}^4\{ \mathbf{i}^{(l)}\}$. 
Then 
\begin{eqnarray}
Q^*( \mathbf{i}^{(1)}, \mathbf{i}^{(2)},\mathbf{i}^{(3)},\mathbf{i}^{(4)},\mathbf{j}_8)  =\mathrm{E}(x_{i_1,j_1}x_{i_1,j_2})\times \mathrm{E}(\text{the  remaining terms}) =0, \label{eq:ankfourexpzero} \notag
\end{eqnarray} 
Therefore by \eqref{eq:expank4indpcond}  and Condition \ref{cond:finitemomt}, 
\begin{align}
\eqref{eq:n2amin1order}=O(n^{\frac{1}{2}\sum_{l=1}^4(a_l-1)}). \label{eq:ordernmixingd4}
\end{align} 

\smallskip

\subparagraph*{Step II: order of $p$}
To prove \eqref{eq:secondgoalakn4}, it remains to show that for  given $(\mathbf{i}^{(1)}, \mathbf{i}^{(2)},\mathbf{i}^{(3)},\mathbf{i}^{(4)})$, 
\begin{align}
\sum_{\substack{ (j_1,j_2),(j_3,j_4),  (j_5,j_6),(j_7,j_8)\in \mathcal{J} } } Q^*( \mathbf{i}^{(1)}, \mathbf{i}^{(2)},\mathbf{i}^{(3)},\mathbf{i}^{(4)},\mathbf{j}_8)=O(p^4).\label{eq:secondgoalp4}
\end{align}
Let $\mu$ be a positive constant same as in \eqref{eq:thresholdd}. Define an event
$B_J^c=\{ Q^*( \mathbf{i}^{(1)}, \mathbf{i}^{(2)},\mathbf{i}^{(3)},\mathbf{i}^{(4)},\mathbf{j}_8) =O(p^{-(8+\mu)})   \}$ and let $B_J$ represent the complement set of $B_J^c$ correspondingly. 
Note that
\begin{align*}
\sum_{\substack{ (j_1,j_2),(j_3,j_4),  (j_5,j_6),(j_7,j_8)\in \mathcal{J} } }Q^*( \mathbf{i}^{(1)}, \mathbf{i}^{(2)},\mathbf{i}^{(3)},\mathbf{i}^{(4)},\mathbf{j}_8)\times \mathbf{1}_{B_J^c}=O(p^8 p^{-(8+\mu)})=o(1).
\end{align*} Moreover by Condition \ref{cond:finitemomt},   $Q^*( \mathbf{i}^{(1)}, \mathbf{i}^{(2)},\mathbf{i}^{(3)},\mathbf{i}^{(4)},\mathbf{j}_8)=O(1)$ always holds. Thus to prove \eqref{eq:secondgoalp4}, it remains to show
\begin{eqnarray}
\sum_{\substack{ (j_1,j_2),(j_3,j_4),  (j_5,j_6),(j_7,j_8)\in \mathcal{J} } } \mathbf{1}_{B_J}=O(p^4). \label{eq:secondgoalp4suff}
\end{eqnarray} We write the ordered version of  $\mathbf{j}_8=(j_1,j_2,j_3,j_4,j_5,j_6,j_7,j_8)$ as  $\tilde{\mathbf{j}}_8=(\tilde{j}_1, \tilde{j}_2,\allowbreak \tilde{j}_3, \tilde{j}_4,  \tilde{j}_5,\tilde{j}_6, \tilde{j}_7,\tilde{j}_8)$, which     satisfies $\tilde{j}_1\leq \tilde{j}_2 \leq \tilde{j}_3 \leq \tilde{j}_4 \leq \tilde{j}_5 \leq \tilde{j}_6 \leq \tilde{j}_7 \leq \tilde{j}_8$.  To facilitate the proof, we first introduce three claims below, which will be proved later. In particular, for given $\mathbf{j}_8$, if $\mathbf{1}_{B_J}=1$, the corresponding ordered tuple $\tilde{\mathbf{j}}_8$ of $\mathbf{j}_8$ satisfies the following three claims with $D_0$ defined in \eqref{eq:thresholdd}. 


\begin{enumerate}
\setlength\itemindent{1.3cm}
	\item [\textbf{Claim 1}]: For any index $\tilde{j}_k\in \tilde{\mathbf{j}}_8$,  if it has two neighbors $\tilde{j}_{k-1}$ and $\tilde{j}_{k+1}$, its  distances with the two neighbors $\tilde{j}_{k-1}$ and $\tilde{j}_{k+1}$ can not be bigger than $D_0$ together.  That is, at least one of $|\tilde{j}_{k-1}-\tilde{j}_{k}| \leq D_0$ and $|\tilde{j}_{k}-\tilde{j}_{k+1}|\leq D_0$ is true. 
	For $\tilde{j}_1$ and $\tilde{j}_8$ with only one neighbor,  they satisfy $|\tilde{j}_1-\tilde{j}_2|\leq D_0$ and $|\tilde{j}_7-\tilde{j}_8|\leq D_0$.
	\item [\textbf{Claim 2}]: For a pair of indexes $(\tilde{j}_{k-1}, \tilde{j}_{k})$ in $\tilde{\mathbf{j}}_8$, when $\tilde{j}_{k-1} \neq \tilde{j}_{k}$, if it has  two neighbors $\tilde{j}_{k-2}$ and $\tilde{j}_{k+1}$, the distances of the pair with the two neighbors can not be bigger than $D_0$ together. That is, at least one of $|\tilde{j}_{k-2}-\tilde{j}_{k-1}| \leq D_0$ and $|\tilde{j}_{k}-\tilde{j}_{k+1}|\leq D_0$  holds. For the pairs $(\tilde{j}_{1}, \tilde{j}_{2})$ and $(\tilde{j}_{7}, \tilde{j}_{8})$ with only one neighbor, when $\tilde{j}_{1}\neq  \tilde{j}_{2}$ and $\tilde{j}_{7}\neq  \tilde{j}_{8}$, they satisfy $|\tilde{j}_{2}-\tilde{j}_{3}|\leq D_0$ and $|\tilde{j}_{6}-\tilde{j}_{7}|\leq D_0$. 
	\item [\textbf{Claim 3}]: \begin{enumerate}
	\item For given $\{\tilde{j}_4,\tilde{j}_5 ,\tilde{j}_6,\tilde{j}_7 ,\tilde{j}_8\}$, \begin{align*}
   \sum_{\tilde{j}_1, \tilde{j}_2, \tilde{j}_3}\mathbf{1}_{B_J\cap \{ \tilde{j}_1=\tilde{j}_2 \}}=O(p^2), \quad \sum_{\tilde{j}_1, \tilde{j}_2, \tilde{j}_3}\mathbf{1}_{B_J\cap \{ \tilde{j}_1\neq \tilde{j}_2 \}}=O(pD_0^2).	
	\end{align*} 
	\item For given $\{\tilde{j}_1,\tilde{j}_2 ,\tilde{j}_3,\tilde{j}_4 ,\tilde{j}_5\}$,
	\begin{align*}
		  \sum_{\tilde{j}_6, \tilde{j}_7, \tilde{j}_8}\mathbf{1}_{B_J\cap \{ \tilde{j}_7=\tilde{j}_8 \}}=O(p^2), \quad \sum_{\tilde{j}_6, \tilde{j}_7, \tilde{j}_8}\mathbf{1}_{B_J\cap \{ \tilde{j}_7\neq \tilde{j}_8 \}}=O(pD_0^2).
	\end{align*}
	\end{enumerate}
\end{enumerate}


Given three claims above, we show  \eqref{eq:secondgoalp4suff} by  discussing different cases. 
\begin{enumerate}
	\item When both  $\tilde{j}_1 \neq \tilde{j}_2$ and $\tilde{j}_7 \neq \tilde{j}_8$, by Claim 3, we know the summation over indexes $(\tilde{j}_1, \tilde{j}_2, \tilde{j}_3)$ is of order $pD_0^2$ and the summation over indexes $(\tilde{j}_6, \tilde{j}_7, \tilde{j}_8)$ is also of order $pD_0^2$. Then we consider $(\tilde{j}_4, \tilde{j}_5)$.  
	   When $|\tilde{j}_4- \tilde{j}_5| \leq D_0$, the summation is of order $(pD_0^2)\times pD_0 \times pD_0^2=p^3 D_0^5=p^4$. When $|\tilde{j}_4- \tilde{j}_5|> D_0$, applying Claim 1 on $\tilde{j}_4$ and $\tilde{j}_5$ respectively, we know $|\tilde{j}_3- \tilde{j}_4| \leq D_0$ and $|\tilde{j}_5- \tilde{j}_6| \leq D_0 $ hold. Therefore, the summation is of order $pD_0^2 \times D_0 \times p \times D_0 \times pD_0^2 =p^3D_0^6=p^4$. In summary,
	   \begin{align*}
	   \sum_{\substack{ (j_1,j_2),(j_3,j_4),  (j_5,j_6),(j_7,j_8)\in \mathcal{J} } }\mathbf{1}_{B_J\cap \{ \tilde{j}_1 \neq \tilde{j}_2, \tilde{j}_7 \neq \tilde{j}_8 \}}=O(p^4).
	   \end{align*}
    \item When only one of $\tilde{j}_1 \neq \tilde{j}_2$ and $\tilde{j}_7 \neq \tilde{j}_8$ holds, without loss of generality, we consider $\tilde{j}_1 = \tilde{j}_2$ and $\tilde{j}_7 \neq \tilde{j}_8$. 
      \begin{enumerate}
       \item When $|\tilde{j}_2 - \tilde{j}_3|>D_0$, applying Claim 1 on $\tilde{j}_3$, we know $|\tilde{j}_3-\tilde{j}_4| \leq D_0$.  Then consider the pair $(\tilde{j}_3, \tilde{j}_4)$. If $\tilde{j}_3 =  \tilde{j}_4$,  by Claim 1, $|\tilde{j}_5-\tilde{j}_4|\leq D_0$ or $|\tilde{j}_5-\tilde{j}_6|\leq D_0$ holds. As $\tilde{j}_7 \neq \tilde{j}_8$,  by Claim 3,   the summation  over $(\tilde{j}_6, \tilde{j}_7, \tilde{j}_8)$ is of order $pD_0^2$.   Therefore, the total summation order is $O(p\times p \times D_0 \times p D_0^2)=O(p^4)$. If $\tilde{j}_3 \neq  \tilde{j}_4$,   applying Claim 2 on the pair $(\tilde{j}_3, \tilde{j}_4)$,  we know $|\tilde{j}_4 - \tilde{j}_5|\leq D_0 $ as we discuss $|\tilde{j}_2 - \tilde{j}_3|>D_0$. Also, as $\tilde{j}_7 \neq \tilde{j}_8$, by Claim 3, the summation order over $(\tilde{j}_6, \tilde{j}_7, \tilde{j}_8)$ is $O(pD_0^2)$.  Thus the total order of summation is $O(pD_0 pD_0^2 p D_0^2)=O(p^4)$. In summary,
       \begin{align*}
	   	\sum_{\substack{ (j_1,j_2),(j_3,j_4),\\  (j_5,j_6),(j_7,j_8)\in \mathcal{J} } }\mathbf{1}_{\left\{B_J\cap \{ \mathrm{one\, of\, } \tilde{j}_1 \neq \tilde{j}_2 \mathrm{\, or \, } \tilde{j}_7 \neq \tilde{j}_8,\, |\tilde{j}_2 - \tilde{j}_3|>D_0\}\right\}}=O(p^4).
	   \end{align*}

       
       \item  When $|\tilde{j}_2 - \tilde{j}_3|\leq D_0$, the summation over $\tilde{j}_1,\tilde{j}_2, \tilde{j}_3$ is of order  $pD_0$. Then we consider $\tilde{j}_4, \tilde{j}_5$.  If  $|\tilde{j}_4-\tilde{j}_5|\leq D_0$, the summation over $\tilde{j}_1, \tilde{j}_2, \tilde{j}_3, \tilde{j}_4, \tilde{j}_5$ is of order $p D_0 pD_0=p^2 D_0^2$. As  $\tilde{j}_7 \neq \tilde{j}_8$, by Claim 3, we know the summation order of $\tilde{j}_6, \tilde{j}_7, \tilde{j}_8$ is $pD_0^2$. Then the total summation order of this case is $O(1)p^2 D_0^2 pD_0^2=O(p^4)$. If $|\tilde{j}_4-\tilde{j}_5|> D_0$, applying Claim 1 on $\tilde{j}_4$ and $\tilde{j}_5$ respectively, we have $|\tilde{j}_3 -\tilde{j}_4 | \leq D_0$ and $|\tilde{j}_5- \tilde{j}_6| \leq D_0$. Also, as  $\tilde{j}_7 \neq \tilde{j}_8$, by Claim 3, we know the summation order of $\tilde{j}_6, \tilde{j}_7, \tilde{j}_8$ is $O(pD_0^2)$. Then the total summation order is $O(1)pD_0\times D_0 p D_0 \times pD_0^2=O(p^4)$. In summary,
       \begin{align*}
	   \sum_{\substack{ (j_1,j_2),(j_3,j_4),\\  (j_5,j_6),(j_7,j_8)\in \mathcal{J} } }\mathbf{1}_{\left\{ B_J \cap \{ \mathrm{one\, of\, } \tilde{j}_1 \neq \tilde{j}_2 \mathrm{\, or \, } \tilde{j}_7 \neq \tilde{j}_8, |\tilde{j}_2 - \tilde{j}_3|\leq D_0\}\right\}}=O(p^4).
	   \end{align*}
      \end{enumerate}
	\item When both  $\tilde{j}_1 = \tilde{j}_2$ and $\tilde{j}_7 = \tilde{j}_8$, then we consider $(\tilde{j}_3, \tilde{j}_4, \tilde{j}_5, \tilde{j}_6)$. 
	 \begin{enumerate}
	  \item If the number of distinct elements  in $\{\tilde{j}_3, \tilde{j}_4, \tilde{j}_5, \tilde{j}_6\}$   is smaller and equal to 2, the order of  summation over $\tilde{j}_3, \tilde{j}_4, \tilde{j}_5, \tilde{j}_6$ is $O(p^2)$. We use $|\{\tilde{j}_3, \tilde{j}_4, \tilde{j}_5, \tilde{j}_6\}|\leq 2$ to represent this case, then
	  \begin{align*}
	   \sum_{\substack{ (j_1,j_2),(j_3,j_4),\\  (j_5,j_6),(j_7,j_8)\in \mathcal{J} } } \mathbf{1}_{\left\{B_J\cap \{ \tilde{j}_1=\tilde{j}_2,\,  \tilde{j}_7 = \tilde{j}_8, \, |\{\tilde{j}_3, \tilde{j}_4, \tilde{j}_5, \tilde{j}_6\}|\leq 2\} \right\}}=O(p^4).
	   \end{align*}
	  	  
	  
	   \item If the number of distinct elements  in $\{\tilde{j}_3, \tilde{j}_4, \tilde{j}_5, \tilde{j}_6\}$ is  3,  we use $|\{\tilde{j}_3, \tilde{j}_4, \tilde{j}_5, \tilde{j}_6\}|= 3$ to represent this case. Then  two of $\tilde{j}_3\neq \tilde{j}_4$, $\tilde{j}_4\neq \tilde{j}_5$ and $\tilde{j}_5\neq \tilde{j}_6$ hold. We consider without loss of generality $\tilde{j}_3\neq \tilde{j}_4$, $\tilde{j}_4\neq \tilde{j}_5$ and  $\tilde{j}_5=\tilde{j}_6$. We apply  Claim 2 on the pair $(\tilde{j}_3,\tilde{j}_4)$ and Claim 1 on $\tilde{j}_3$. Then at least two of $|\tilde{j}_2-\tilde{j}_3 |\leq D_0$, $|\tilde{j}_3-\tilde{j}_4 |\leq D_0$  and $|\tilde{j}_4-\tilde{j}_5 |\leq D_0$ holds. Thus the summation order is $O(pD_0^2p^2)=O(p^4)$. In summary,
	    \begin{align*}
	   \sum_{\substack{ (j_1,j_2),(j_3,j_4),\\ (j_5,j_6),(j_7,j_8)\in \mathcal{J} } } \mathbf{1}_{\left\{B_J\cap \{ \tilde{j}_1=\tilde{j}_2,\,  \tilde{j}_7 = \tilde{j}_8,\, |\{\tilde{j}_3, \tilde{j}_4, \tilde{j}_5, \tilde{j}_6\}|= 3\}\right\}}=O(p^4).
	   \end{align*}
	  \item If the number of distinct elements  in $\{\tilde{j}_3, \tilde{j}_4, \tilde{j}_5, \tilde{j}_6\}$ is 4, we use $|\{\tilde{j}_3, \tilde{j}_4, \tilde{j}_5, \tilde{j}_6\}|=4$ to represent this case, and we know $\tilde{j}_3\neq \tilde{j}_4$, $\tilde{j}_4\neq \tilde{j}_5$ and $\tilde{j}_5\neq \tilde{j}_6$.  Applying Claim 2 on the pair $(\tilde{j}_3, \tilde{j}_4)$, and  applying Claim 1 on the two single indexes $\tilde{j}_3$ and  $\tilde{j}_4$ respectively, we know at least two of $|\tilde{j}_2-\tilde{j}_3 |\leq D_0$, $|\tilde{j}_3-\tilde{j}_4 |\leq D_0$ and $|\tilde{j}_4-\tilde{j}_5 |\leq D_0$ hold. Therefore  the summation over $(\tilde{j}_1,  \tilde{j}_2, \tilde{j}_3, \tilde{j}_4, \tilde{j}_5)$ is of order $O(p\times pD_0^2)=O(p^2 D_0^2)$. Then applying Claim 1 on  $\tilde{j}_6$,  we know at least one of $|\tilde{j}_5-\tilde{j}_6|\leq D_0$ and $|\tilde{j}_6-\tilde{j}_7|\leq D_0$ holds. Then the total order of summation for this part is $O(p^2D_0^2 \times pD_0)=O(p^4)$, that is,
	  \begin{align*}
	   \sum_{\substack{ (j_1,j_2),(j_3,j_4),\\  (j_5,j_6),(j_7,j_8)\in \mathcal{J} } } \mathbf{1}_{B_J\cap \{ \tilde{j}_1=\tilde{j}_2,\,  \tilde{j}_7 = \tilde{j}_8,\, |\{\tilde{j}_3, \tilde{j}_4, \tilde{j}_5, \tilde{j}_6\}|= 4\}}=O(p^4).
	   \end{align*}
	 \end{enumerate}
\end{enumerate}
Combining the results obtained, we know \eqref{eq:secondgoalp4suff} is proved. Thus to prove \eqref{eq:secondgoalp4}, it remains to  prove the three claims above. 

\vspace{0.5em}

By the definition of $Q^*( \mathbf{i}^{(1)}, \mathbf{i}^{(2)},\mathbf{i}^{(3)},\mathbf{i}^{(4)},\mathbf{j}_8)$  in Section \ref{sec:pfindpdnk4}, 
\begin{eqnarray*}
\Big|Q^*( \mathbf{i}^{(1)}, \mathbf{i}^{(2)},\mathbf{i}^{(3)},\mathbf{i}^{(4)},\mathbf{j}_8) \Big|\leq 	C\Big|\mathrm{E}\Big( \prod_{t=1}^8 x_{k, \tilde{j}_t} \Big)\Big|.  
\end{eqnarray*}
Then it is sufficient to show that for given $\mathbf{j}_8$, when the ordered version $\tilde{\mathbf{j}}_8$ of $\mathbf{j}_8$ does not follow the three claims,
\begin{eqnarray}
	\Big|\mathrm{E}\Big( \prod_{t=1}^8 x_{k, \tilde{j}_t} \Big)\Big|= O(p^{-(8+\mu)}). \label{eq:suffgoalp4order}
\end{eqnarray}

\subparagraph*{Proof of Claim 1}  \quad

\vspace{0.4em}

(1) When the index $\tilde{j}_k$  has  two neighbors, we give the proof by an example of $k=3$. All the other cases can be obtained following similar analysis without loss of generality.  Suppose $ \tilde{j}_3$'s distances between its neighbors  $ \tilde{j}_2$ and $ \tilde{j}_4$ are both bigger than $D_0$, i.e., $|\tilde{j}_2-\tilde{j}_3|>D_0 $ and $|\tilde{j}_3-\tilde{j}_4|>D_0 $. Then by Conditions \ref{cond:finitemomt}, \ref{cond:alphamixing}, and the $\alpha$-mixing inequality in Lemma \ref{lm:mixingineq},
\begin{eqnarray*}
&& \Big|\mathrm{E}\Big( \prod_{t=1}^8 x_{k, \tilde{j}_t} \Big)\Big| \notag \\
&=&\Big|\mathrm{cov}\Big( \prod_{t=1}^3 x_{k, \tilde{j}_t}\, , \, \prod_{t=4}^8 x_{k, \tilde{j}_t} \Big)+ \mathrm{E}\Big(  \prod_{t=1}^3 x_{k, \tilde{j}_t}\Big)\times  \mathrm{E}\Big(\prod_{t=4}^8 x_{k, \tilde{j}_t}  \Big) \Big| \notag \\
&\leq &	 C \delta^{\frac{D_0 \epsilon}{2+\epsilon}} + C \times | \mathrm{cov}( x_{k,\tilde{j}_1}x_{k,\tilde{j}_2}\ ,\ x_{k,\tilde{j}_3} ) +  \mathrm{E}( x_{k,\tilde{j}_1}x_{k,\tilde{j}_2})  \mathrm{E}(x_{k,\tilde{j}_3})  | \\
	&=& O(p^{-(8+\mu)})  + C \times | \mathrm{cov}( x_{k,\tilde{j}_1}x_{k,\tilde{j}_2}\ ,\  x_{k,\tilde{j}_3} )|  \notag \\
	&=& O( p^{-(8+\mu)}).  
\end{eqnarray*}Thus \eqref{eq:suffgoalp4order} holds.


(2) For $\tilde{j}_1$ and $\tilde{j}_8$ with only one neighbor,  we give the proof on $\tilde{j}_1$, while  $\tilde{j}_8$ can be proved similarly. By Conditions \ref{cond:finitemomt},  \ref{cond:alphamixing}, and Lemma \ref{lm:mixingineq},
\begin{eqnarray*}
	\Big|\mathrm{E}\Big( \prod_{t=1}^8 x_{k, \tilde{j}_t} \Big)\Big| 
	&=& \Big|\mathrm{cov}\Big(x_{k,\tilde{j}_1} \, ,\, \prod_{t=2}^8 x_{k,\tilde{j}_t} \Big) +\mathrm{E}(x_{k,\tilde{j}_1})\times\mathrm{E}\Big( \prod_{t=2}^8 x_{k,\tilde{j}_t} \Big) \Big|\notag \\
	&\leq & C \delta^{\frac{D_0 \epsilon}{2+\epsilon}} + 0 \quad \quad (\ \mathrm{E}(x_{k,\tilde{j}_1})=0\ ) \\
	&= & O(p^{-(8+\mu)}) . 
\end{eqnarray*}Thus \eqref{eq:suffgoalp4order} also holds.
\smallskip

\vspace{0.4em}

\subparagraph*{Proof of Claim 2:} \quad 

\vspace{0.4em}

(1) When the pair $(\tilde{j}_{k-1}, \tilde{j}_{k})$ has two neighbors, we give the proof by the example when $k=5$, i.e., we consider the pair $(\tilde{j}_4, \tilde{j}_5)$.    The other cases can be proved similarly without loss of generality. Suppose $\tilde{j}_4 \neq \tilde{j}_5 $ with $|\tilde{j}_3-\tilde{j}_4|>D_0$ and $|\tilde{j}_5-\tilde{j}_6|>D_0$. As $ \mathrm{E}(x_{k,\tilde{j}_4}x_{k,\tilde{j}_5})=0$ under $H_0$, by Conditions \ref{cond:finitemomt} and   \ref{cond:alphamixing}, and  Lemma \ref{lm:mixingineq}, we have
\begin{eqnarray*}
&& \Big|\mathrm{E}\Big( \prod_{t=1}^8 x_{k, \tilde{j}_t} \Big)\Big| \notag \\
&=&\Big|\mathrm{cov}\Big( \prod_{t=1}^3 x_{k, \tilde{j}_t}\, , \, \prod_{t=4}^8 x_{k, \tilde{j}_t} \Big)+ \mathrm{E}\Big(  \prod_{t=1}^3 x_{k, \tilde{j}_t}\Big)\times  \mathrm{E}\Big(\prod_{t=4}^8 x_{k, \tilde{j}_t}  \Big) \Big| \notag \\
&\leq &	C \delta^{\frac{D_0 \epsilon}{2+\epsilon}} + \Big|  \mathrm{E}\Big(\prod_{t=1}^3 x_{k,\tilde{j}_t}\Big) \times \Big\{ \mathrm{cov}\Big( \prod_{t=4}^5 x_{k,\tilde{j}_t} \ ,\  \prod_{t=6}^8 x_{k,\tilde{j}_t}\Big) + \mathrm{E}\Big(\prod_{t=4}^5 x_{\tilde{j}_t}\Big)\mathrm{E}\Big(\prod_{t=6}^8 x_{k,\tilde{j}_t}\Big) \Big\}\Big| \notag \\
&=& C \delta^{\frac{D_0 \epsilon}{2+\epsilon}} + \Big|\mathrm{E}(x_{k,\tilde{j}_1}x_{k,\tilde{j}_2}x_{\tilde{j}_3}) \times \{ \mathrm{cov}(x_{k,\tilde{j}_4}x_{k,\tilde{j}_5} \ ,\  x_{k,\tilde{j}_6}x_{k,\tilde{j}_7}x_{k,\tilde{j}_8})+0\}\Big| \notag \\
	&\leq & C \delta^{\frac{D_0 \epsilon}{2+\epsilon}}=O(  p^{-(8+\mu)} ). 
\end{eqnarray*}
Thus \eqref{eq:suffgoalp4order} holds.

(2) For the pairs $(\tilde{j}_{1}, \tilde{j}_{2})$ and $(\tilde{j}_{7}, \tilde{j}_{8})$ with only one neighbor, we give the proof on $(\tilde{j}_{1}, \tilde{j}_{2})$,  while the proof on $(\tilde{j}_{7}, \tilde{j}_{8})$ can be obtained similarly.   
If $\tilde{j}_1\neq \tilde{j}_2$ and $|\tilde{j}_2-\tilde{j}_3|>D_0 $, as  $ \mathrm{E}(x_{k,\tilde{j}_1}x_{k,\tilde{j}_2}) =0 $ under $H_0$,  by Conditions \ref{cond:finitemomt} and  \ref{cond:alphamixing}, and the $\alpha$-mixing inequality in Lemma \ref{lm:mixingineq}, we have
\begin{eqnarray*}
 \Big|\mathrm{E}\Big( \prod_{t=1}^8 x_{k, \tilde{j}_t} \Big)\Big| 
	&=&\Big| \mathrm{cov}\Big( \prod_{t=1}^2x_{k,\tilde{j}_t}\ , \ \prod_{t=3}^8 x_{ \tilde{j}_t}\Big) + \mathrm{E}\Big( \prod_{t=1}^2x_{k,\tilde{j}_t}\Big) \mathrm{E}\Big(\prod_{t=3}^8 x_{ \tilde{j}_t}\Big)\Big| \\
	&\leq & C\delta^{\frac{D_0 \epsilon}{2+\epsilon}}=O(  Cp^{-(8+\mu)}).
\end{eqnarray*} Thus \eqref{eq:suffgoalp4order} holds.


\smallskip

\subparagraph*{Proof of Claim 3:}

The Claim 3 (a) is obtained by applying Claim 1 on the  $\tilde{j}_1$ and Claim 2 on the pair $(\tilde{j}_1, \tilde{j}_2 )$ when $\tilde{j}_1 \neq \tilde{j}_2 $. The Claim 3 (b) is also obtained similarly.

\paragraph{Proof under Condition \ref{cond:highordmominter}} \label{par:ednkpfell}

In this section, we prove Lemma \ref{lm:secondgoaljointnormal} by substituting Condition \ref{cond:alphamixing} with Condition  \ref{cond:highordmominter}.
Similarly to Section \ref{par:pfmomentvartaa}, the proof under Condition  \ref{cond:highordmominter} follows similarly to the proof under the independence case  in Section \ref{sec:pfindpdnk4}. 
 In particular, we note that Condition  \ref{cond:highordmominter} implies that if one of the indexes in $\{j_1,\ldots ,j_8\}$ only appears once,  $\mathrm{E}( \prod_{r=1}^8 x_{k,j_r})= 0$. Therefore when $\mathrm{E}( \prod_{r=1}^8 x_{k,j_r})\neq 0$, 
\eqref{eq:eightjindpcond} holds. Also following similar analysis,  we know \eqref{eq:expank4indpcond} holds by  Condition  \ref{cond:highordmominter} and $\mathrm{E}(x_{1,j_1}x_{1,j_2})=0$ for $j_1\neq j_2$. Combining \eqref{eq:eightjindpcond} and \eqref{eq:expank4indpcond}, Lemma \ref{lm:secondgoaljointnormal} is proved.

\subsection{Lemmas for the proof of Theorem \ref{thm:asymindpt}}
\subsubsection{Proof of Lemma \ref{lm:firstsetpthm3proof} (on Page \pageref{lm:firstsetpthm3proof}, Section \ref{sec:asymindptproofsec})} \label{sec:lemma9pproof}

For easy illustration, we first prove Lemma \ref{lm:firstsetpthm3proof}   when $m=1$ in Section \ref{par:indtwopfm1}, and next  present the proof for $m>1$ in Section \ref{par:indtwopfmlarger1}. 

\paragraph{Proof for $m=1$} \label{par:indtwopfm1}
Specifically, in this section, we prove  
  \begin{eqnarray*}
  	\Big| P\Big( \frac{\hat{M}_n}{n} > y_p, \frac{\tilde{\mathcal{U}}(a)  }{\sigma(a)} \leq 2z \Big) - P\Big(\frac{\hat{M}_n}{n} > y_p \Big) P\Big(\frac{\tilde{\mathcal{U}}(a)  }{\sigma(a)} \leq 2z \Big) \Big| \rightarrow 0. 
\end{eqnarray*} 
Note that by definitions in \eqref{eq:vladefinition} and  \eqref{eq:indpalldefinewvar},
\begin{eqnarray}
	&& P\Big( \frac{ \hat{M}_n}{n} > y_p, \frac{\tilde{\mathcal{U}}(a)  }{\sigma(a)} \leq 2z \Big) \label{eq:asymindptfirstpartprob}  \\
	&=& P\Big( \max_{1 \leq l \leq q}\ ( \hat{G}_l )^2 > ny_p  ,  (\sigma(a) P_a^n)^{-1} \sum_{m=1}^q U_m^a  \leq z  \Big) \notag  \\
	&=& P\Big(\Big\{\cup_{l=1}^q \{( \hat{G}_l )^2 > ny_p\} \Big\} \cap \Big\{ (\sigma(a) P_a^n)^{-1} \sum_{m=1}^q U_m^a  \leq z \Big\} \Big).   \notag
\end{eqnarray} 
Define the  events
$
	E_l=\{ (\hat{G}_l )^2 > n y_p  \} \cap \{(\sigma(a) P_a^n)^{-1} \sum_{m=1}^q U_m^a  \leq z \}, 
$ then we have
\begin{eqnarray}
	\eqref{eq:asymindptfirstpartprob}=P(\cup_{l=1}^q E_{l} ) \label{eq:jointform}.	
\end{eqnarray}
We next examine the upper and lower bounds of  \eqref{eq:jointform}. Particularly, using the Bonferroni's inequality, for any even number $d<[q/2]$, we obtain
\begin{align}
	\sum_{s=1}^d (-1)^{s-1} \sum_{1\leq l_1<\ldots<l_s \leq q} P(\cap_{t=1}^s E_{l_t} ) \leq P(\ \cup_{l=1}^q E_{l} \ )  \label{eq:jointbound} \\ \leq ~\sum_{s=1}^{d-1} (-1)^{s-1} \sum_{1\leq l_1 < \ldots < l_s\leq q} P(\cap_{t=1}^s E_{l_t} ). \notag
\end{align} We consider $d=O(\log^{1/5} p)$ below. The following proof proceeds by examining the upper and lower bounds of $ P(\cap_{t=1}^s E_{l_t} )$ first and combining them based on \eqref{eq:jointbound}.

To facilitate the discussion, we define some notation.  Let
\begin{eqnarray*}
	H_d=\sum_{s=1}^d (-1)^{s-1}\sum_{1 \leq l_1 < \ldots <l_s \leq q} P(\cap_{t=1}^s \{ (\hat{G}_{l_t} )^2 > n y_p \} ).
\end{eqnarray*} By the  Bonferroni's inequality, we have 
\begin{eqnarray}
	H_d \leq P(\ \cup_{l=1}^q \{ ( \hat{G}_l )^2 > n y_p \}\  )    \leq H_{d-1}. \label{eq:hdiffbound}
\end{eqnarray} Given $l_1,\ldots,l_s$, we  define two index sets: $I_s=\{ (j^1_{l_t} ,j^2_{l_t}), 1\leq t\leq s \}$ and correspondingly
\begin{align}
	L_{I_s}=\{ (j_1,j_2):\  (j_1,j_2)\cap (u,t) \neq \emptyset, (u,t) \in I_s \ \mathrm{and} \ (j_1,j_2)\in L \}, \label{eq:defwidset}
\end{align} where  $L$ is defined in \eqref{eq:wdef}. \eqref{eq:defwidset} suggests that  $L_{I_s}$ contains all the index pairs that have overlap with the  index pairs in $I_s$. Note that the definitions of $I_s$ and $L_{I_s}$ depend on the given indexes $l_1,\ldots,l_s$; for the simplicity of notation, we write $I_s$ and $L_{I_s}$ in this proof without ambiguity.  It follows that
\begin{eqnarray}\label{eq:usumldef}
	\sum_{m=1}^q U_m^a = \sum_{(j^1_l, j^2_l) \in  L_{I_s} } U_l^a+ \sum_{(j^1_l, j^2_l) \in L \backslash L_{I_s} } U_l^a.
\end{eqnarray}
The cardinality of $L_{I_s}$ is no greater than $2ps$ by construction. Furthermore, $2ps\leq 2pd$ as $s\leq d$.  Note that the indexes in $I_s$ and $L \backslash L_{I_s}$ have no intersection. By this construction and the independence assumption in  Condition \ref{cond:maxiidcolumn}, for any finite integers $a_1, a_2 \geq 1$, we know
\begin{eqnarray*}
	\{  U_l^{a_1}, \  (j^1_l,j^2_l) \in I_s \} \quad \mathrm{and} \quad \{U_l^{a_2}, \  (j^1_l,j^2_l) \in L\backslash L_{I_s}  \}
\end{eqnarray*} are independent. 

We next examine the upper bound of  $P(\cap_{t=1}^s E_{l_t})$. 
By  the definition of $E_l$ and \eqref{eq:usumldef}, 
\begin{align}
	&~ P(\cap_{t=1}^s E_{l_t}) \label{eq:intersect}  \\
	=  &~ P\biggr(\bigcap_{t=1}^s \biggr\{  \Big\{(\sigma(a) P_a^n)^{-1} \sum_{m=1}^q U_{m}^a  \leq z \Big\} \bigcap \{ (\hat{G}_{l_t} )^2 > n y_p  \} \biggr\} \biggr)  \notag \\
	=&~  P\biggr(\bigcap_{t=1}^s \biggr\{ \Big\{(\sigma(a) P_a^n)^{-1} \Big[ \sum_{(j^1_l, j^2_l) \in  L_{I_s} } U_l^a+ \sum_{(j^1_l, j^2_l) \in L \backslash L_{I_s} } U_l^a   \Big] \leq z \Big\}   \notag \\
	&~ \quad \quad \quad \quad \quad  \bigcap \{ (\hat{G}_{l_t} )^2 > n y_p  \}\biggr\} \biggr). \notag 
\end{align}
Let $\Gamma_p$ represent a number of order $ \Theta \{(\log p)^{-1/2}\}$ and we have 
\begin{align*}
	&~ \Big \{(\sigma(a) P_a^n)^{-1} \Big( \sum_{(j^1_l, j^2_l) \in  L_{I_s} } U_l^a+ \sum_{(j^1_l, j^2_l) \in L \backslash L_{I_s} } U_l^a   \Big) \leq z \Big\} \\
	\subseteq &~  \Big \{ (\sigma(a) P_a^n)^{-1} \Big|  \sum_{(j^1_l, j^2_l) \in  L_{I_s} } U_l^a \Big| \geq \Gamma_p \Big\} \bigcup  \Big \{ (\sigma(a) P_a^n)^{-1} \sum_{(j^1_l, j^2_l) \in L \backslash L_{I_s} } U_l^a \leq \Gamma_p+z  \Big\}.\notag
\end{align*} Thus \eqref{eq:intersect} has the following upper bound,
\begin{align*} 
	\eqref{eq:intersect} \leq & ~ P\Big(\Big\{ \cap_{t=1}^s  \{ (\hat{G}_{l_t} )^2 > n y_p  \} \Big\} \bigcap \Big\{   (\sigma(a) P_a^n)^{-1} \Big|  \sum_{(j^1_l, j^2_l) \in  L_{I_s} } U_l^a \Big| \geq \Gamma_p  \Big\}  \Big) \notag\\
	&~+ P \Big( \Big\{ \cap_{t=1}^s  \{ ( \hat{G}_{l_t} )^2 > n y_p  \} \Big\} \bigcap  \Big \{ (\sigma(a) P_a^n)^{-1} \sum_{(j^1_l, j^2_l) \in L \backslash L_{I_s} } U_l^a \leq \Gamma_p+z  \Big\}   \Big). \notag \end{align*}
In addition, we note that $
	\{  \hat{G}_l, \  (j^1_l,j^2_l) \in I_s \}$ and $\{U_l^a, \  (j^1_l,j^2_l) \in L\backslash L_{I_s}  \}
$ are independent, because of $I_s \cap (L\backslash L_{I_s})=\emptyset$ by the construction and the independence assumption   in Condition \ref{cond:maxiidcolumn}. It follows that
\begin{align}\label{eq:diffbound1}
	\eqref{eq:intersect}  \leq     P_s + P_{ys} P_{+z}, 
\end{align}
 where for simplicity we define
\begin{eqnarray} \label{eq:multiplepsdefs}
    P_s &=&P\Big(\Big\{   (\sigma(a) P_a^n)^{-1} \Big|  \sum_{(j^1_l, j^2_l) \in  L_{I_s} } U_l^a \Big| \geq \Gamma_p  \Big\}  \Big), \\
P_{ys} &=& P\Big(  \bigcap_{t=1}^s  \{ (\hat{G}_{l_t} )^2 > n y_p  \} \Big),  \notag \\
	P_{+z}&=& P\Big( \Big \{ (\sigma(a) P_a^n)^{-1} \sum_{(j^1_l, j^2_l) \in L \backslash L_{I_s} } U_l^a \leq \Gamma_p+z  \Big\} \Big). \notag 
\end{eqnarray} 
 Note that although the notation  $P_{ys}$, $P_{+z}$ and $P_s$  in \eqref{eq:multiplepsdefs} suppress their dependence on the specific choice of $(l_1,\ldots, l_s)$, this will not influence the proof due to the  i.i.d. assumption   in Condition \ref{cond:maxiidcolumn}.    



Similarly we examine the lower bound of  $P(\cap_{t=1}^s E_{l_t})$. In particular, 
\begin{eqnarray*}
	&& \Big \{ (\sigma(a) P_a^n)^{-1} \sum_{(j^1_l, j^2_l) \in L \backslash L_{I_s} } U_l^a \leq z-\Gamma_p  \Big\} \\
	&\subseteq & \Big \{ (\sigma(a) P_a^n)^{-1} \Big|  \sum_{(j^1_l, j^2_l) \in  L_{I_s} } U_l^a \Big| \geq \Gamma_p \Big\} \bigcup \Big\{ (\sigma(a) P_a^n)^{-1} \sum_{m=1}^q U_m^a \leq z   \Big\}.
\end{eqnarray*} Then \eqref{eq:intersect} has the following lower bound,
\begin{align*}
	\eqref{eq:intersect} \geq &~- P\Big(\Big\{ \cap_{t=1}^s  \{ (\hat{G}_{l_t} )^2 > n y_p  \} \Big\} \bigcap \Big\{   (\sigma(a) P_a^n)^{-1} \Big|  \sum_{(j^1_l, j^2_l) \in  L_{I_s} } U_l^a \Big| \geq \Gamma_p  \Big\}  \Big) \notag\\
	&~+ P \Big( \Big\{ \cap_{t=1}^s  \{ ( \hat{G}_{l_t} )^2 > n y_p  \} \Big\} \bigcap  \Big \{ (\sigma(a) P_a^n)^{-1} \sum_{(j^1_l, j^2_l) \in L \backslash L_{I_s} } U_l^a \leq z-\Gamma_p  \Big\}   \Big). \notag
\end{align*}
Similarly to \eqref{eq:diffbound1}, by the independence between $\{  \hat{G}_l, \  (j^1_l,j^2_l) \in I_s \}$ and $\{U_l^a, \  (j^1_l,j^2_l) \in L\backslash L_{I_s}  \}$, we obtain
\begin{align}
	\eqref{eq:intersect} \geq P_{ys} \times P_{-z} - P_s, \label{eq:diffbound2}
\end{align} where $P_{ys}$ and $P_s$ are defined same as in \eqref{eq:multiplepsdefs}, and we define
\begin{eqnarray*}
	P_{-z}=P\Big(  (\sigma(a) P_a^n)^{-1} \sum_{(j^1_l, j^2_l) \in L \backslash L_{I_s} } U_l^a \leq z-\Gamma_p \Big).
\end{eqnarray*} 

We have obtained the upper and lower bounds of $P(\cap_{t=1}^s E_{l_t})$ in \eqref{eq:diffbound1} and  \eqref{eq:diffbound2}   respectively. 
 We next prove that $P_{+z}$  in \eqref{eq:diffbound1} and $P_{-z}$ in   \eqref{eq:diffbound2} are close in the sense that there exists some constant $C>0$,
\begin{eqnarray}
 | P_{+z}- P_z | \leq C\times \Gamma_p \quad \text{and} \quad | P_{-z}- P_z | \leq C\times \Gamma_p, \label{eq:zunifapproxbound}
\end{eqnarray}  where we define
$
	P_z=P( (\sigma(a) P_a^n)^{-1} \sum_{m=1}^q U_m^a \leq z  ).
$ To obtain \eqref{eq:zunifapproxbound}, we note that $\sum_{(j^1_l, j^2_l) \in L_{I_s} } U_l^a$ is a summation over index pairs in $L_{I_s}$, and $L_{I_s}$ is of size $2ps$, which is $o(p^2)$ as $s\leq d$ and $d=O(\log^5 p)$. Following similar analysis of $\tilde{\mathcal{U}}^*(a)/\sigma(a)\xrightarrow{P}0$ in  Lemma \ref{lm:varianceorder},
 we know $(\sigma(a) P_a^n)^{-1}\sum_{(j^1_l, j^2_l) \in L_{I_s} } U_l^a\xrightarrow{P} 0$. Moreover, by $\tilde{\mathcal{U}}(a)=2(P^n_a)^{-1}\sum_{l=1}^q U_l^a $  in \eqref{eq:vladefinition}, $\Gamma_p=\Theta(\log^{-1/2}p)$ and the convergence result in \eqref{eq:thm1berryessenbound}, we have for  given $z$, 
$$
 	|P_{+z}-\Phi(2z+2\Gamma_p)|\leq C\Gamma_p, \ |P_{-z}-\Phi(2z-2\Gamma_p)| \leq C\Gamma_p, \ |P_{z}-\Phi(2z)| \leq C\Gamma_p. 
$$ 
As $|\Phi(2z+ 2\Gamma_p)-\Phi(2z)|\leq C\Gamma_p$ for given $z$,  $| P_{+ z}- P_z | \leq | P_{+ z} - \Phi(2z+ 2\Gamma_p) |+|\Phi(2z+ 2\Gamma_p)-\Phi(2z)| +|P_z - \Phi(2z)| \leq C\Gamma_p.$ Similarly, as $|\Phi(2z- 2\Gamma_p)-\Phi(2z)|\leq C\Gamma_p$, $| P_{-z}- P_z | \leq C\Gamma_p$. Therefore \eqref{eq:zunifapproxbound} is obtained.

In summary, given  \eqref{eq:diffbound1},   \eqref{eq:diffbound2} and \eqref{eq:zunifapproxbound},   we have \begin{eqnarray*}
	 |P(\cap_{t=1}^s E_{l_t})- P_{ys} \times P_z | 
	 \leq  P_s +C\times \Gamma_p\times P_{ys}. \label{eq:diffbound3}
\end{eqnarray*} 
Given the above property of  $P(\cap_{t=1}^s E_{l_t})$,  we next derive an upper bound of \eqref{eq:jointform} based on  the relationship in \eqref{eq:jointbound}. Specifically,
\begin{align}
\quad &~P( \cup_{l=1}^q E_{l} ) \notag \\
	 \leq & ~  \sum_{s=1}^{d-1} (-1)^{s-1} \sum_{1\leq l_1 < \ldots < l_s\leq q} P(\cap_{t=1}^s E_{l_t} ) \notag \\
	\leq & ~ \sum_{s=1}^{d-1} (-1)^{s-1} \sum_{1\leq l_1 < \ldots < l_s\leq q} \{  P_{ys}  P_z    + (-1)^{s-1} \times[ C\Gamma_p \times P_{ys} + P_s ] \} \notag  \\
\leq & ~ H_{d-1} \times P_z + \sum_{s=1}^{d-1} \sum_{1\leq l_1 < \ldots < l_s\leq q}   ( C\times \Gamma_p \times P_{ys} +P_s), \label{eq:unionblbound1} 
\end{align}  where the last inequality  uses the notation  in \eqref{eq:hdiffbound}, i.e.,
\begin{eqnarray}
	H_{d-1}=\sum_{s=1}^{d-1} (-1)^{s-1} \sum_{1\leq l_1 < \ldots < l_s\leq q}  P_{ys}, \label{eq:hdminus1def}
\end{eqnarray} 
and the fact that $P_z$ does not depend on $l_1,\ldots, l_s$ in summation.  
From \eqref{eq:hdiffbound}, we know $H_{d-1}\leq P_y+|H_{d-1}-H_d|$,  where we define
\begin{eqnarray}
	P_y=P\Big(\ \bigcup_{l=1}^q \{ (\hat{G}_l )^2 > n y_p \}\, \Big). \label{eq:pydef} 
\end{eqnarray} 
As a result, we have
\begin{align*}
	\eqref{eq:unionblbound1} \leq P_y\times P_z + |H_{d-1}-H_d|\times P_z + \sum_{s=1}^{d-1} \sum_{1\leq l_1 < \ldots < l_s\leq q}  (C\Gamma_p P_{ys}+P_s).
\end{align*} 
Next we prove $|H_{d-1}-H_d|\times P_z \to 0,$ $  \sum_{s=1}^{d-1} \sum_{1\leq l_1 < \ldots < l_s\leq q}  \Gamma_p\times P_{ys}\to 0$ and $ \sum_{s=1}^{d-1} \sum_{1\leq l_1 < \ldots < l_s\leq q}P_s\to 0$ by  the following three  Lemmas \ref{lm:maxupperlowerbd}--\ref{lm:expdecylm}, respectively.

\begin{lemma} \label{lm:maxupperlowerbd} 
Under the conditions of Theorem \ref{thm:asymindpt}, when $s=O(\log^{1/5} p)$,
	\begin{eqnarray*}
		&& \sum_{1\leq l_1<\ldots<l_s \leq q } P\Big(\bigcap_{t=1}^s \{ ( \hat{G}_{l_t})^2/n \geq 4\log p - \log \log p +y \} \Big) \notag \\
		&=&  \frac{1}{s!} \Big( \frac{1}{2\sqrt{2\pi}} e^{-\frac{y}{2}} \Big)^s (1+o(1)) + o(1).
	\end{eqnarray*}
\end{lemma}
\begin{proof}
	\textit{See Section \ref{sec:lemma17proof} on Page \pageref{sec:lemma17proof}.}
\end{proof}
\begin{lemma} \label{lm:maxupperlowerbd1} 
Under the conditions of Theorem \ref{thm:asymindpt}, when $d=O(\log^{1/5} p)$,
\begin{eqnarray*}
	&& \sum_{s=1}^{d-1} \sum_{1\leq l_1<\ldots<l_s \leq q } P\Big(\bigcap_{t=1}^s \{ ( \hat{G}_{l_t})^2/n \geq 4\log p - \log \log p +y \} \Big) \\
	&= &  \sum_{s=1}^{d -1}  \frac{1}{s!} \Big( \frac{1}{2\sqrt{2\pi}} e^{-y/2} \Big)^s\{1+o(1)\}  + o(1).
\end{eqnarray*}
\end{lemma}
\begin{proof}
	\textit{See Section \ref{sec:lemma18proof} on Page \pageref{sec:lemma18proof}.}
\end{proof}
\begin{lemma} \label{lm:expdecylm}
Under the conditions of Theorem \ref{thm:asymindpt},
\begin{eqnarray*}
	\sum_{s=1}^{d-1}  \sum_{1\leq l_1 < \ldots < l_s\leq q} P\Big(\Big\{   (\sigma(a) P_a^n)^{-1} \Big|  \sum_{(j^1_l, j^2_l) \in  L_{I_s} } U_l^a \Big| \geq \Gamma_p  \Big\} \Big) \to 0, 
\end{eqnarray*}	where $L_{I_s}$ is defined in \eqref{eq:defwidset},  $d=O(\log^{1/5} p)$,  $q=\binom{p}{2}$ and $\Gamma_p=\Theta( \log^{-1/2} p)$.
\end{lemma}
\begin{proof}
	\textit{See Section \ref{sec:lemma19proof} on Page \pageref{sec:lemma19proof}.}
\end{proof}

First, we show $|H_{d-1}-H_d|\times P_z \to 0$. By Lemma \ref{lm:maxupperlowerbd}, when $d\to \infty$,
\begin{eqnarray*}
	 |H_{d-1} -H_d |&=& \sum_{1\leq l_1 < \ldots < l_d\leq q}   P \Big( \bigcap_{t=1}^d \{ (\hat{G}_{l_t} )^2 > n y_p \} \Big)\\
	 &\leq & C \frac{1}{d!} \Big(\frac{1}{2\sqrt{2\pi}} e^{-y/2} \Big)^d \leq   
	 C e \times\Big( \frac{ e^{1-y/2}}{2\sqrt{2\pi} d} \Big)^d \rightarrow 0, 
\end{eqnarray*} where the last inequality follows from $d!\geq e(d/e)^d$.  Second, we show that  $  \sum_{s=1}^{d-1} \sum_{1\leq l_1 < \ldots < l_s\leq q}  \Gamma_p P_{ys}\to 0$. By  the definition of $P_{ys}$ in \eqref{eq:multiplepsdefs}, and Lemma  \ref{lm:maxupperlowerbd}, $\sum_{s=1}^{d-1} \sum_{1\leq l_1 < \ldots < l_s\leq q} \Gamma_p P_{ys}=\Gamma_p \sum_{s=1}^{d-1} \frac{1}{s!} ( \frac{1}{2\sqrt{2\pi}} e^{-y/2} )^s + o(1)\to 0,$ where we use $\Gamma_p =\Theta(\log^{-1/2} p) \rightarrow 0$ and $\sum_{s=1}^{d-1} \frac{1}{s!} ( \frac{1}{2\sqrt{2\pi}} e^{-y/2} )^s < \infty$ from $s! \geq e(s/e)^s$. Third, we obtain $
	\sum_{s=1}^{d-1}  \sum_{1\leq l_1 < \ldots < l_s\leq q} P_s \to 0
$ directly from Lemma \ref{lm:expdecylm} following the notation $P_s$ in \eqref{eq:multiplepsdefs}.

 

In summary, the analysis above shows that $P( \cup_{l=1}^q E_{l} ) \leq  P_y \times P_z + o(1).$ On the other hand,  following similar arguments, we can obtain	$P( \cup_{l=1}^q E_{l} ) \geq  P_y \times P_z + o(1).$ Therefore, $|P( \cup_{l=1}^q E_{l} )-P_y \times P_z|\rightarrow 0$ is obtained, that is,
\begin{align*}
	\Big| P( \cup_{l=1}^q E_{l} ) - P(\ \cup_{l=1}^q \{ (\hat{G}_l )^2 > n y_p \}\  )  P\Big( \Big \{ (\sigma(a) P_a^n)^{-1} \sum_{m=1}^q U_m^a \leq z  \Big\} \Big) \Big| \rightarrow 0.
\end{align*}
Recall the notation in \eqref{eq:asymindptfirstpartprob} and \eqref{eq:jointform}. We then  know Lemma \ref{lm:firstsetpthm3proof} is proved for $m=1$. 

\paragraph{Proof for $m>1$} \label{par:indtwopfmlarger1}

We still use the notation defined in Section \ref{sec:asymindptproofsec}, where $U_l^{a_r}$ and $\tilde{\mathcal{U}}(a_r)$   for $r=1,\ldots,m$ follow the definitions in \eqref{eq:vladefinition} and \eqref{eq:originleadingterm} respectively.  
 To prove Lemma \ref{lm:firstsetpthm3proof} for $m>1$, we 
  note that similarly to \eqref{eq:jointform}, we can write
\begin{eqnarray}
	&& P\Big( \frac{ \hat{M}_n}{n} > y_p, \frac{\tilde{\mathcal{U}}(a_1)  }{\sigma(a_1)} \leq 2z_1,\ldots,  \frac{\tilde{\mathcal{U}}(a_m)  }{\sigma(a_m)} \leq 2z_m\Big) = P(\cup_{l=1}^q E_l),\label{eq:prodprobjoint}  
\end{eqnarray}	
where we redefine the events 
\begin{eqnarray*}
	E_l=\bigcap_{r=1}^m \Big\{ (\sigma(a_r) P_{a_r}^n)^{-1} \sum_{v=1}^q U_v^{a_r}  \leq z_r \Big\} \cap\{ (\hat{G}_l )^2 > n y_p  \}. 
\end{eqnarray*} 
It follows that  \eqref{eq:jointbound} and \eqref{eq:hdiffbound} still hold. For given $l_1,\ldots,l_s$, we define $I_s$ and $L_{I_s}$ same as in  \eqref{eq:defwidset}. Then for $r=1,\ldots, m$, we write
\begin{eqnarray*}
\sum_{v=1}^q U_v^{a_r} = \sum_{(j^1_l, j^2_l) \in  L_{I_s} } U_l^{a_r}+ \sum_{(j^1_l, j^2_l) \in L \backslash L_{I_s} } U_l^{a_r}.
\end{eqnarray*} 
By the construction of $L_{I_s}$ and the independence assumption in  Condition \ref{cond:maxiidcolumn}, we know
\begin{eqnarray*}
	\cup_{r=1}^m \{  U_l^{a_{r}}, \  (j^1_l,j^2_l) \in I_s \} \quad \mathrm{and} \quad  \cup_{r=1}^m \{U_l^{a_{r}}, \  (j^1_l,j^2_l) \in L\backslash L_{I_s}  \}
\end{eqnarray*}
 are independent. 

Similarly to \eqref{eq:intersect}, given $l_1,\ldots, l_s$,  
we have 
\begin{align}
& ~P(\cap_{t=1}^s E_{l_t}) \label{eq:newintersect}  \\
	=&~  P\biggr( \bigcap_{r=1}^m \Big\{(\sigma(a_r) P_{a_r}^n)^{-1} \Big[ \sum_{(j^1_l, j^2_l) \in  L_{I_s} } U_l^{a_r}+ \sum_{(j^1_l, j^2_l) \in L \backslash L_{I_s} } U_l^{a_r}   \Big] \leq z_r \Big\}   \notag \\
	&~ \quad \quad \quad \quad \quad  \cap \{ \cap_{t=1}^s  \{ (\hat{G}_{l_t} )^2 > n y_p  \} \} \biggr).\notag
\end{align}
We take $\Gamma_p$ same as in Section \ref{par:indtwopfm1} with $\Gamma_p= \Theta \{(\log p)^{-1/2}\}$. Then for each $r=1,\ldots,m$, we have
\begin{align*}
	&~ \Big \{(\sigma(a_r) P_{a_r}^n)^{-1} \Big[ \sum_{(j^1_l, j^2_l) \in  L_{I_s} } U_l^{a_r}+ \sum_{(j^1_l, j^2_l) \in L \backslash L_{I_s} } U_l^{a_r}   \Big] \leq z_r \Big\} \\
	\subseteq &~  \Big \{ (\sigma(a_r) P_{a_r}^n)^{-1} \Big|  \sum_{(j^1_l, j^2_l) \in  L_{I_s} } U_l^{a_r} \Big| \geq \Gamma_p \Big\} \bigcup  \Big \{ (\sigma(a_r) P_{a_r}^n)^{-1} \sum_{(j^1_l, j^2_l) \in L \backslash L_{I_s} } U_l^{a_r} \leq \Gamma_p+z_r  \Big\},\notag
\end{align*} and
\begin{align*}
	& ~\Big \{ (\sigma(a_r) P_{a_r}^n)^{-1} \sum_{(j^1_l, j^2_l) \in L \backslash L_{I_s} } U_l^{a_r} \leq z_r-\Gamma_p  \Big\} \\
	\subseteq &~ \Big \{ (\sigma(a_r) P_{a_r}^n)^{-1} \Big|  \sum_{(j^1_l, j^2_l) \in  L_{I_s} } U_l^{a_r} \Big| \geq \Gamma_p \Big\} \bigcup \Big\{ (\sigma(a_r) P_{a_r}^n)^{-1} \sum_{v=1}^q U_v^{a_r} \leq z_r   \Big\}.	
\end{align*}
Therefore similarly to \eqref{eq:diffbound1} and \eqref{eq:diffbound2}, we know 
\begin{align}
	\eqref{eq:newintersect}\leq  P_{ys}  P_{+z}+\sum_{r=1}^m P_{s_r}, \quad
	\eqref{eq:newintersect} \geq  P_{ys} P_{-z}-\Big( \sum_{r=1}^m P_{s_r} \Big), \label{eq:c43twobounds}
\end{align} where $P_{ys}$ is defined in  \eqref{eq:multiplepsdefs} and we further define
\begin{eqnarray*}
	P_{+z} &=& P\Big( \bigcap_{r=1}^m \Big \{ (\sigma(a_r) P_{a_r}^n)^{-1} \sum_{(j^1_l, j^2_l) \in L \backslash L_{I_s} } U_l^{a_r} \leq z_r+\Gamma_p  \Big\} \Big), \notag \\
	P_{s_r} &=& P\Big( (\sigma(a_r) P_{a_r}^n)^{-1} \Big|  \sum_{(j^1_l, j^2_l) \in  L_{I_s} } U_l^{a_r} \Big| \geq \Gamma_p  \Big), \notag \\
	P_{-z} &=& P\Big( \bigcap_{r=1}^m \Big \{ (\sigma(a_r) P_{a_r}^n)^{-1} \sum_{(j^1_l, j^2_l) \in L \backslash L_{I_s} } U_l^{a_r} \leq z_r - \Gamma_p \Big\} \Big).
\end{eqnarray*}

%

We note that the cardinality  of $L_{I_s}$ is no greater than $2ps$ which is $o(p^2)$. Similarly to  Section  \ref{par:indtwopfm1}, we  know $(\sigma(a_r) P_{a_r}^n)^{-1}\times   \sum_{(j^1_l, j^2_l) \in L \backslash L_{I_s} } U_l^{a_r}\xrightarrow{P} 0$ for $ r=1,\ldots,m$. Combined with Theorem \ref{thm:jointnormal}, we know  $\{(\sigma(a_r) P_{a_r}^n)^{-1}\times   \sum_{(j^1_l, j^2_l) \in L \backslash L_{I_s} } U_l^{a_r}: r=1,\ldots,m \}$ converges to $\mathcal{N}(0,I_m)$ and thus are asymptotically independent. We then have    
\begin{eqnarray}
	\Big|P_{+z}-\prod_{r=1}^m P_{+z_r}\Big|\to 0, \quad \Big|P_{-z}-\prod_{r=1}^m P_{-z_r}\Big|\to 0, \label{eq:diffplusminus}
\end{eqnarray} 
where we define
\begin{align*}
	P_{+z_r}=& P\Big(  (\sigma(a_r) P_{a_r}^n)^{-1} \sum_{(j^1_l, j^2_l) \in L \backslash L_{I_s} } U_l^{a_r} \leq z_r +\Gamma_p \Big), \notag \\
	P_{-z_r}=& P\Big(  (\sigma(a_r) P_{a_r}^n)^{-1} \sum_{(j^1_l, j^2_l) \in L \backslash L_{I_s} } U_l^{a_r} \leq z_r- \Gamma_p \Big). \notag
\end{align*}
Similarly to \eqref{eq:zunifapproxbound}, for each $r=1,\ldots,m$, we have
\begin{eqnarray}
	|P_{+z_r}-P_{z_r}|\leq C\Gamma_p \quad \text{and} \quad |P_{-z_r}-P_{z_r}|\leq C\Gamma_p, \label{eq:rplusminsdiff}
\end{eqnarray} where we define 
$
	P_{z_r}=P(  (\sigma(a_r) P_{a_r}^n)^{-1} \sum_{v=1}^q U_v^{a_r} \leq z_r     ).
$  Combining \eqref{eq:diffplusminus} and \eqref{eq:rplusminsdiff}, we have
\begin{eqnarray*}
\Big|P_{+z}-\prod_{r=1}^m P_{z_r}\Big|\to 0 \quad \text{and} \quad \Big|P_{-z}-\prod_{r=1}^m P_{z_r}\Big|\to 0.
\end{eqnarray*}
By \eqref{eq:c43twobounds} and \eqref{eq:rplusminsdiff}, 
\begin{eqnarray}
	\Big|\eqref{eq:newintersect}-P_{ys}\prod_{r=1}^m P_{z_r}\Big|\leq o(1)P_{ys}+\sum_{r=1}^m P_{s_r}. \label{eq:elsttwoboundsm}
\end{eqnarray}

Given \eqref{eq:elsttwoboundsm}, similarly to \eqref{eq:unionblbound1}, we have
\begin{align*}
	&~P( \cup_{l=1}^q E_{l} ) \notag \\
 \leq &~ \sum_{s=1}^{d-1} (-1)^{s-1} \sum_{1\leq l_1 < \ldots < l_s\leq q} \Big\{  P_{ys}  \prod_{r=1}^m P_{z_r} + (-1)^{s-1} \times \Big[ o(1)  P_{ys} + \sum_{r=1}^m P_{s_r} \Big] \Big\} \notag  \\
 \leq &~ H_{d-1} \prod_{r=1}^m P_{z_r} +  \sum_{s=1}^{d-1} \sum_{1\leq l_1<\ldots < l_s \leq q}\Big\{ o(1) P_{ys}+\sum_{r=1}^m P_{s_r}\Big\} \notag \\
\leq &~ P_y  \prod_{r=1}^m P_{z_r} + |H_{d-1}-H_d| \prod_{r=1}^m P_{z_r}+  \sum_{s=1}^{d-1} \sum_{1\leq l_1<\ldots < l_s \leq q}\Big\{ o(1) P_{ys}+\sum_{r=1}^m P_{s_r}\Big\}, \notag 
\end{align*} where $H_{d-1}$ follows the definition in \eqref{eq:hdminus1def} and we use \eqref{eq:hdiffbound} and the definition \eqref{eq:pydef} in the last inequality. By Lemma \ref{lm:maxupperlowerbd}, $|H_{d-1}-H_d|\to 0$; by Lemma \ref{lm:maxupperlowerbd1}, $o(1)\sum_{s=1}^{d-1}\sum_{1\leq l_1<\ldots < l_s \leq q} P_{ys}\to 0$; by Lemma  \ref{lm:expdecylm}, $ \sum_{r=1}^m \sum_{s=1}^{d-1} \sum_{1\leq l_1<\ldots < l_s \leq q} P_{s_r}=\to 0$.
 
 In summary, we have shown that $P( \cup_{l=1}^q E_{l} ) \leq  P_y \times \prod_{r=1}^m P_{z_r} + o(1).$ Moreover,  following similar arguments, we have $P( \cup_{l=1}^q E_{l} ) \geq  P_y \times \prod_{r=1}^m P_{z_r} + o(1).$ Therefore, $|P( \cup_{l=1}^q E_{l} )-P_y \times \prod_{r=1}^m P_{z_r} |\to 0$ is obtained, that is, 
\begin{eqnarray*}
\Big|P( \cup_{l=1}^q E_{l} )-P( \cup_{l=1}^q \{ (\hat{G}_l )^2 > n y_p \} ) \prod_{r=1}^m P(  (\sigma(a_r) P_{a_r}^n)^{-1} \sum_{v=1}^q U_v^{a_r} \leq z_r ) \Big|	\to 0.
\end{eqnarray*}
Since $\eqref{eq:prodprobjoint}=P(\cup_{l=1}^q E_l)$, $\{\hat{M}_n/n>y_p\}=\cup_{l=1}^q \{ (\hat{G}_l)^2 >ny_p \}$ and $\tilde{\mathcal{U}}(a_r)=2(P^n_{a_r})^{-1} \sum_{v=1}^q U_v^{a_r}$, we know Lemma \ref{lm:firstsetpthm3proof}  is proved for $m>1$.

\subsubsection{Proof of Lemma \ref{lm:maxupperlowerbd} (on Page \pageref{lm:maxupperlowerbd}, Section \ref{sec:lemma9pproof})} \label{sec:lemma17proof}

In this section, we prove Lemma \ref{lm:maxupperlowerbd}. The proof will use Lemmas \ref{sec:lemma16proof} and \ref{sec:covidentityelemdiff}, which will be presented and  proved in Sections \ref{sec:lemma16proof} and \ref{sec:covidentityelemdiff}, respectively. 
\begin{proof}
Following the definitions in \eqref{eq:indpalldefinewvar}, $\hat{G}_l$  will not change if $x_{i,j}$ is scaled by its standard deviation $\sigma_{j,j}$. Thus in the  discussion below, we assume  without loss of generality  that $\sigma_{j,j}=1$, $j=1,\ldots, p$ for the simplicity of representation.

Given $i$ and $ 1\leq l_1<\ldots<l_s \leq q$, we define $\check{\mathcal{X}}_{i,j_{l_t}^1,j_{l_t}^2}=x_{i,j_{l_1}^1}x_{i,j_{l_1}^2} \times \mathbf{1}\{|x_{i,j_{l_t}^1}x_{i,j_{l_t}^2}| \leq \tau_n \}$ for $t=1,\ldots, s$, $\mathbf{W}_i=( \check{\mathcal{X}}_{i,j_{l_1}^1,j_{l_1}^2},\ldots,  \check{\mathcal{X}}_{i,j_{l_s}^1,j_{l_s}^2})^{\intercal}$, and let $|\mathbf{W}_i|_{\min}$ denote the minimum absolute value of the entries in the vector $\mathbf{W}_i$. 
It follows that $P(\cap_{t=1}^s \{ (\hat{G}_{l_t})^2/n \geq 4\log p - \log \log p +y \} ) = P( | \sum_{i=1}^n \mathbf{W}_i|_{\min} \geq \sqrt{n}y_p^{1/2} ),$ where $y_p$ is defined in \eqref{eq:ypdefition}.

We prove Lemma \ref{lm:maxupperlowerbd} through examining $\mathbf{W}_i$, $i=1,\ldots,n$. 
Since $\mathbf{W}_i$'s are independent and identically distributed random vectors,  
$\mathrm{cov}( \sum_{i=1}^n \mathbf{W}_i)= n \times \mathrm{cov}(\mathbf{W}_1).$ 
 We apply Theorem 1.1 in  \cite{zaitsev1987gaussian}  and obtain
\begin{eqnarray}
	& & P\Big( \Big|  \sum_{i=1}^n \mathbf{W}_i \Big|_{\min} \geq \sqrt{n} y_p^{1/2} \Big) \label{eq:zaitsevrestail} \\
	& \leq & P \Big( |\mathbf{N}_s|_{\min} \geq \sqrt{n}y_p^{1/2} - \epsilon \sqrt{n}(\log p)^{-1/2} \Big)  + \notag \\
	&&  c_1 s^{5/2} \exp \Big( - \frac{n^{1/2} \epsilon}{c_2 s^{5/2} \tau_n (\log p)^{1/2}} \Big), \notag
\end{eqnarray} 
where $c_1$ and $ c_2$ are positive  constants; $\epsilon \rightarrow 0$, which will be specified later; and $\mathbf{N}_s:= (N_{l_1}, \ldots, N_{l_s})^{\intercal}$ follows multivariate normal distribution  
with $\mathrm{E}(\mathbf{N}_s)=0$ and $\mathrm{cov}(\mathbf{N}_s)=\mathrm{cov}( \sum_{i=1}^n \mathbf{W}_i)= n\times \mathrm{cov}(\mathbf{W}_1)$. 
Moreover, we apply Theorem 1.1 in  \cite{zaitsev1987gaussian} in terms of lower bound and obtain
\begin{align*}
	&~ P\Biggr( \Big|  \sum_{i=1}^n \mathbf{W}_i \Big|_{\min} \geq \sqrt{n} y_p^{1/2} \Biggr)\notag \\
	\geq &~ P \Big( |\mathbf{N}_s|_{\min} \geq \sqrt{n}y_p^{1/2} + \epsilon \sqrt{n}(\log p)^{-1/2} \Big) -  c_1 s^{5/2} \exp \Biggr( - \frac{n^{1/2} \epsilon}{c_2 s^{5/2} \tau_n (\log p)^{1/2}} \Biggr). \notag
\end{align*} 
As $s=O( \log^{1/5} p)$, $\log p=o(n^{1/7})$, and $\tau_n=\tau \log(p+n)$, when $\epsilon \rightarrow 0$  sufficiently slow, there exists a constant $M>0$ such that
\begin{eqnarray*}
	c_1 s^{5/2}\exp \Big( -\frac{\epsilon n^{1/2}}{c_2 s ^{5/2} \tau_n (\log p)^{1/2} }  \Big) = O(1) e^{- M n^{3/14} }. 
\end{eqnarray*}Therefore, for $s=O(\log^{1/5}p),$
\begin{eqnarray}
	&& \sum_{1\leq l_1 < \ldots < l_s \leq q}  c_1 s^{5/2}\exp \Biggr( -\frac{\epsilon n^{1/2}}{c_2 s ^{5/2} \tau_n (\log p)^{1/2} }  \Biggr) \label{eq:requirementdincrease} \\
	&=& O(1) q^s \times  e^{- M n^{3/14}} = O(1) e^{-M n^{3/14} + 2s\log p}= o(1).\notag
\end{eqnarray} 
 In summary, by  \eqref{eq:requirementdincrease} and Lemma \ref{lm:normalapproxierror}  in Section \ref{sec:lemma16proof} below, Lemma  \ref{lm:maxupperlowerbd} is proved.   
\end{proof}

\paragraph{Lemma \ref{lm:normalapproxierror} and its proof}\label{sec:lemma16proof}
\begin{lemma} \label{lm:normalapproxierror}For $s=O(\log^{1/5}p)$ and $\mathbf{N}_s$ in \eqref{eq:zaitsevrestail},
	\begin{align*}
		\sum_{1\leq l_1 < \ldots < l_s \leq q} P \Big[ |\mathbf{N}_s|_{\min} \geq \sqrt{n}\{ y_p^{1/2} \pm \epsilon (\log p)^{-1/2}\} \Big]\simeq \frac{1}{s!} \Big\{  \frac{1}{ 2\sqrt{2\pi} } \exp \Big(- \frac{y}{2} \Big) \Big\}^s.
	\end{align*}
\end{lemma} 

\begin{proof} 
We write $v_p=y_p^{1/2} \pm \epsilon (\log p)^{-1/2}$, which represents two numbers in this proof. Since the proof below will be the same for the two numbers respectively, we abuse the use of notation $v_p$ below.

We define $\mathbf{U}_s=\mathrm{cov}(\mathbf{W}_1)$, where $\mathbf{W}_1$ is defined in Section \ref{sec:lemma17proof}. 
 By the density of multivariate normal,
\begin{eqnarray}
&& P \Big( |\mathbf{N}_s|_{\min} \geq \sqrt{n}( y_p^{1/2} \pm \epsilon (\log p)^{-1/2}) \Big) \notag \\
		&=&P \Biggr(\frac{1}{\sqrt{n}} |\mathbf{N}_s|_{\min} \geq y_p^{1/2} \pm \epsilon (\log p)^{-1/2} \Biggr) \notag \\
		&=& \frac{1}{(2\pi)^{s/2} |\mathbf{U}_{s}|^{1/2}} \int_{|\mathbf{y}_{\min}| \geq v_p}  \exp \Big( -\frac{1}{2} \mathbf{y}^{\intercal} (\mathbf{U}_{s})^{-1} \mathbf{y} \Big) d \mathbf{y} \notag \\	
	&=& \frac{1}{(2\pi)^{s/2} } \int_{|\mathbf{U}_{s}^{1/2}\mathbf{z}_{\min}| \geq v_p}  \exp \Big( -\frac{1}{2} \mathbf{z}^{\intercal}  \mathbf{z} \Big) d \mathbf{z}. \label{eq:normalapproxgoal1}
\end{eqnarray}
We note that $\mathbb{Z}_{P,1}\leq \eqref{eq:normalapproxgoal1} \leq \mathbb{Z}_{P,1}+\mathbb{Z}_{P,2},$ 
where we define
\begin{align*}
\mathbb{Z}_{P,1}=&~	\frac{1}{(2\pi)^{s/2} } \int_{|\mathbf{U}_{s}^{1/2}\mathbf{z}|_{\min} \geq v_p, |\mathbf{z}|_{\max} \leq 4\sqrt{s\log p} }  \exp \Big( -\frac{1}{2} \mathbf{z}^{\intercal}  \mathbf{z} \Big) d \mathbf{z}, \notag \\
\mathbb{Z}_{P,2}=&~	\frac{1}{(2\pi)^{s/2} } \int_{ |\mathbf{z}|_{\max} >4\sqrt{s\log p} }  \exp \Big( -\frac{1}{2} \mathbf{z}^{\intercal}  \mathbf{z} \Big) d \mathbf{z}.  
\end{align*}
To prove Lemma  \ref{lm:normalapproxierror}, we show  $\mathbb{Z}_{P,2}=o(1)\{  \frac{1}{\sqrt{2\pi} p^2}e^{-y/2} \}^s$ and $\mathbb{Z}_{P,1}\simeq \{ \frac{1}{\sqrt{2\pi} p^2} e^{-y/2}\}^s$ respectively in the following. 


We first prove $\mathbb{Z}_{P,2}=o(1)\{  \frac{1}{\sqrt{2\pi} p^2}e^{-y/2} \}^s$.
Let  $z\sim \mathcal{N}(0,1)$. 
By the property of standard normal distribution, we have
\begin{eqnarray}
	P(z>t) \simeq (\sqrt{2 \pi} t)^{-1} e^{- t^2/2}\ as \ t\rightarrow +\infty. \label{eq:normalapproxcdf}
\end{eqnarray}
 It follows that
\begin{eqnarray}
	\quad \mathbb{Z}_{P,2} &=& s \times P(|z|>4\sqrt{s\log p}) \label{eq:zp2approxprob} \\
	&\simeq & s \times \frac{2}{\sqrt{2\pi}\times  4\sqrt{s\log p}} \exp(-8s\log p) \notag \\
	&=&  \frac{1}{2\sqrt{2\pi}} \sqrt{\frac{s}{\log p}} \times p^{-8s}=o(1)\Big\{ \frac{1}{\sqrt{2\pi}p^2} e^{-y/2} \Big\}^s. \notag
\end{eqnarray} 

Next we prove $\mathbb{Z}_{P,1}\simeq \{ \frac{1}{\sqrt{2\pi} p^2} e^{-y/2} \}^s$. Note that
\begin{eqnarray*}
 \mathbb{Z}_{P,1}&=& \frac{1}{(2\pi)^{s/2} }\int_{|\mathbf{z}+ (\mathbf{U}_{s}^{1/2}-{I}_s)\mathbf{z}|_{\min} \geq v_p, |\mathbf{z}|_{\max} \leq 4\sqrt{s\log p} }  \exp \Big( -\frac{\mathbf{z}^{\intercal}  \mathbf{z}}{2}\Big) d \mathbf{z} \notag \\
	 &\leq & \frac{1}{(2\pi)^{s/2} }\int_{ \substack{|\mathbf{z}|_{\min} \geq v_p-|(\mathbf{U}_{s}^{1/2}-{I}_s)\mathbf{z}|_{\max}; |\mathbf{z}|_{\max} \leq 4\sqrt{s\log p} }}  \exp \Big( -\frac{\mathbf{z}^{\intercal}  \mathbf{z}}{2}  \Big) d \mathbf{y}, \notag 	
\end{eqnarray*}
where $I_s$ represents an identity matrix of size $s\times s$. 
When $|\mathbf{z}|_{\max} \leq 4\sqrt{s\log p}$, we have $|(\mathbf{U}_{s}^{1/2}-{I}_s)\mathbf{z}|_{\max}  \leq 4Cs\sqrt{s\log p} (p+n)^{-c_0\tau}$ by Lemma \ref{lm:covidentityelemdiff} in Section \ref{sec:covidentityelemdiff} below. It follows that
\begin{align}
\mathbb{Z}_{P,1}\leq \frac{1}{(2\pi)^{s/2} }\int_{|\mathbf{z}|_{\min} \geq \tilde{v}_p }  \exp \Big( -\frac{\mathbf{z}^{\intercal}  \mathbf{z}}{2}  \Big) d \mathbf{y}, \label{eq:zp1upperbound}
\end{align} where we define $\tilde{v}_p=v_p-4Cs\sqrt{s\log p} (p+n)^{-c_0\tau}$. We set $\tau$ as a sufficiently large constant such that $s\sqrt{s\log p}=o\{(p+n)^{c_0\tau}\}$, then
$
	\tilde{v}_p =  2\sqrt{\log p} \{1+o(1)\}.
$ 
By \eqref{eq:normalapproxcdf} and \eqref{eq:zp1upperbound}, 
\begin{align*}
\mathbb{Z}_{P,1} \leq &~ \Big\{ \frac{2}{\sqrt{2\pi} \tilde{v}_p } \exp (-\tilde{v}^2_p/2 )\Big\}^s  \notag \\
	=& ~\Big\{ 2 \frac{1+o(1)}{\sqrt{2\pi} \sqrt{4\log p} } \exp \Big(-2\log p + (\log \log p)/2 - y/2 + o(1) \Big)\Big\}^s \notag \\
	=& ~\Big\{ \frac{1}{\sqrt{2\pi} p^2} e^{-y/2} \Big\}^s \{1+o(1)\}.  \notag 
\end{align*} 
Similarly, we have
\begin{align*}
	\mathbb{Z}_{P,1} \geq &~ \frac{1}{(2\pi)^{s/2} }\int_{|\mathbf{z}|_{\min} \geq v_p+|(\mathbf{U}_{s}^{1/2}-{I}_s)\mathbf{z}|_{\max}, |\mathbf{z}|_{\max} \leq 4\sqrt{s\log p} }  \exp \Big( -\frac{\mathbf{z}^{\intercal}  \mathbf{z}}{2}  \Big) d \mathbf{y} \notag \\
	\geq &~ \frac{1}{(2\pi)^{s/2} }\int_{|\mathbf{z}|_{\min} \geq v_p+4Cs\sqrt{s\log p} (p+n)^{-\tau/2}}  \exp \Big( -\frac{\mathbf{z}^{\intercal}  \mathbf{z}}{2}  \Big) d \mathbf{y}  -\mathbb{Z}_{P,2} \notag \\
  =&~ \Big\{ \frac{1}{\sqrt{2\pi} p^2} e^{-y/2} \Big\}^s \{1+o(1)\}. \notag 
\end{align*} 
We therefore obtain $\mathbb{Z}_{P,1}\simeq \{ \frac{1}{\sqrt{2\pi} p^2} e^{-y/2} \}^s.$ 




Since $\mathbb{Z}_{P,1}\leq \eqref{eq:normalapproxgoal1} \leq \mathbb{Z}_{P,1}+\mathbb{Z}_{P,2}$, $\mathbb{Z}_{P,1}\simeq \{ \frac{1}{\sqrt{2\pi} p^2} e^{-y/2}\}^s$  and $\mathbb{Z}_{P,2}=o(1)\{  \frac{1}{\sqrt{2\pi} p^2}e^{-y/2} \}^s$, we obtain $\eqref{eq:normalapproxgoal1}\simeq \{  \frac{1}{\sqrt{2\pi} p^2}e^{-y/2} \}^s$. 
It follows that as $p \rightarrow \infty$ and  $s=O(\log^{1/5} p)$,
	\begin{eqnarray*}
		&& \sum_{1\leq l_1 < \ldots < l_s \leq q} P \Big( |\mathbf{N}_s|_{\min} \geq \sqrt{n}( y_p^{1/2} \pm \epsilon (\log p)^{-1/2}) \Big) \\
		&=& \binom{q}{s} \Big\{  \frac{1}{ \sqrt{2\pi} p^2} \exp (- y/2 ) \Big\}^s\{1+o(1)\} \quad \ \Big(q=\frac{p(p-1)}{2}\Big) \\
		&=& \frac{1}{s!} \Big\{  \frac{1}{ 2\sqrt{2\pi} } \exp (- y/2 )  \Big\}^s \{1+o(1)\}.
	\end{eqnarray*}
\end{proof}


\paragraph{Lemma \ref{lm:covidentityelemdiff}  and its proof}\label{sec:covidentityelemdiff}
\begin{lemma}\label{lm:covidentityelemdiff} 
For $\mathbf{U}_s$ in Section \ref{sec:lemma16proof}, there exist some positive constants $C$ and $c_0$ such that 
$ 
	| \mathbf{U}_s^{1/2}-I_s |_{\max} \leq C (p+n)^{-c_0\tau}, 
$ where $|\cdot|_{\max}$ represents the element-wise maximum absolute value, and $\tau$ is the constant satisfying $\tau_n=\tau\log(p+n)$ from \eqref{eq:indpalldefinewvar}.
\end{lemma}
\begin{proof}
Recall that  $\mathbf{U}_s=\mathrm{cov}(\mathbf{W}_1)$ and $\mathbf{W}_1=(\check{\mathcal{X}}_{1,j_{l_1}^1,j_{l_1}^2},\ldots,\check{\mathcal{X}}_{1,j_{l_s}^1,j_{l_s}^2})$ for given $1\leq l_1<\ldots<l_s\leq q$, which is defined at the beginning of Section \ref{sec:lemma17proof}. 
To prove Lemma \ref{lm:covidentityelemdiff},  we  prove $| \mathbf{U}_s-I_s |_{\max} \leq C (p+n)^{-c_0\tau}$ first. Specifically,  we show the diagonal and off-diagonal elements of $\mathrm{cov} (\mathbf{W}_1)  - I_s$ are bounded by $C(p+n)^{-c_0\tau}$ respectively. 

First we show  
for given $(j_l^1,j_l^2)$, $|\mathrm{var}(\check{\mathcal{X}}_{1,j_{l}^1,j_{l}^2})-1|\leq C(p+n)^{-c_0\tau}$.  By the independence assumption in Condition  \ref{cond:maxiidcolumn} and $\sigma_{j,j}=1$ for  $j=1,\ldots,p$, we know $ \mathrm{var}(x_{1,j^1_l}x_{1,j^2_l})=1$; 
by $\mathrm{E}(x_{1,j^1_l}x_{1,j^2_l})=0,$ we have $\mathrm{var}(x_{1,j^1_l}x_{1,j^2_l})=\mathrm{E}\{(x_{1,j^1_l}x_{1,j^2_l})^2\}$.  It follows that 
\begin{align}
\quad \Big|\mathrm{var}(\check{\mathcal{X}}_{1,j_{l}^1,j_{l}^2})-1\Big| 
=&~\Big|\mathrm{var}(\check{\mathcal{X}}_{1,j_{l}^1,j_{l}^2})-\mathrm{var}(x_{1,j^1_l}x_{1,j^2_l}) \Big|\notag \\
=&~\Big| \mathrm{E}\Big\{(\check{\mathcal{X}}_{1,j_{l}^1,j_{l}^2})^2\Big\}-\Big\{\mathrm{E}(\check{\mathcal{X}}_{1,j_{l}^1,j_{l}^2}) \Big\}^2 - \mathrm{E}\Big\{(x_{1,j^1_l}x_{1,j^2_l})^2\Big\} \Big| \notag \\
\leq &~\Big|\mathrm{E}\Big\{(\check{\mathcal{X}}_{1,j_{l}^1,j_{l}^2})^2\Big\} -\mathrm{E}\Big\{(x_{1,j^1_l}x_{i,j^2_l})^2\Big\}\Big|+\Big|\mathrm{E}(\check{\mathcal{X}}_{1,j_{l}^1,j_{l}^2}) \Big|^2, \label{eq:varbdind} 
\end{align}where we use $ \mathrm{var}(x_{1,j^1_l}x_{1,j^2_l})=1$ in the first equation; and we use the definition of $\mathrm{var}(\check{\mathcal{X}}_{1,j_{l}^1,j_{l}^2})$ and $\mathrm{var}(x_{1,j^1_l}x_{1,j^2_l})=\mathrm{E}\{(x_{1,j^1_l}x_{1,j^2_l})^2\}$ in the second equation.  
Recall the definition $\check{\mathcal{X}}_{1,j_{l}^1,j_{l}^2}=x_{1,j_{l}^1}x_{1,j_{l}^2} \times \mathbf{1}\{|x_{1,j_{l}^1}x_{1,j_{l}^2}| \leq \tau_n \}$. We then have
\begin{eqnarray}
&&\Big|\mathrm{E}\Big\{(x_{1,j^1_l}x_{1,j^2_l})^2\Big\}-\mathrm{E}\Big\{(\check{\mathcal{X}}_{1,j_{l}^1,j_{l}^2})^2\Big\}\Big| \label{eq:diffexpindv1} \\
 &=&\Big|\mathrm{E}\Big[ (x_{1,j^1_l}x_{1,j^2_l})^2\mathbf{1}\{|x_{1,j_{l}^1}x_{1,j_{l}^2}|>\tau_n\}\Big]\Big|,\notag
\end{eqnarray} and 
 $|\mathrm{E} (\check{\mathcal{X}}_{1,j_{l}^1,j_{l}^2})|=|\mathrm{E} (x_{1,j_{l}^1}x_{1,j_{l}^2} \times \mathbf{1}\{|x_{1,j_{l}^1}x_{1,j_{l}^2}|> \tau_n \})|$ as $\mathrm{E}(x_{1,j^1_l}x_{1,j^2_l})=0.$ Since $\mathbf{1}\{|x_{1,j_{l}^1}x_{1,j_{l}^2}|> \tau_n \})\leq \mathbf{1}\{|x_{1,j_{l}^1}|> \sqrt{\tau_n}\}+\mathbf{1}\{|x_{1,j_{l}^2}|>\sqrt{\tau_n} \}$, and $x_{1,j_{l}^1}$ and $x_{1,j_{l}^2}$ are i.i.d. by Condition \ref{cond:maxiidcolumn}, by  H\"older's inequality,  we know
 \begin{align}
 	\eqref{eq:diffexpindv1} \leq  &~ C\times \mathrm{E}\Big(x_{1,j_{l}^1}^2\mathbf{1} \{|x_{1,j_{l}^1}|> \sqrt{\tau_n}\} \Big)\times \mathrm{E}(x_{1,j_{l}^2}^2) \label{eq:diffexpindv2}  \\
 	 \leq  &~ C\times\{\mathrm{E}(x_{1,j_{l}^1}^4)P(|x_{1,j_{l}^1}|> \sqrt{\tau_n}\} )\}^{1/2}\times \mathrm{E}(x_{1,j_{l}^2}^2), \notag 
 \end{align} and also
 \begin{align}
 \Big|\mathrm{E} (\check{\mathcal{X}}_{1,j_{l}^1,j_{l}^2})\Big|\leq 	C\times\{\mathrm{E}(x_{1,j_{l}^1}^2)P(|x_{1,j_{l}^1}|> \sqrt{\tau_n}\} )\}^{1/2}\times \mathrm{E}(|x_{1,j_{l}^2}|).\label{eq:diffexpindv3}
 \end{align}
By Markov's inequality, $P(|x_{1,j_{l}^1}|> \sqrt{\tau_n}\} )\leq \mathrm{E}\{\exp(t_0x_{1,j_{l}^1}^2)\} \exp(-t_0\tau_n)$, where $t_0$ is given in Condition \ref{cond:maxiidcolumn}.  Combining \eqref{eq:varbdind}--\eqref{eq:diffexpindv3}, we obtain  that there exists some positive constants $C$ and $c_0$ such that  
\begin{align*}
\eqref{eq:varbdind}\leq 	 C\times \{\mathrm{E}(\exp(t_0x_{1,j_{l}^1}^2)) \exp(-t_0\tau_n)\}^{1/2}\leq C(p+n)^{-c_0\tau},
\end{align*} 
where we use the assumption that  $x_{1,j_{l}^1}$ and $x_{1,j_{l}^2}$ are i.i.d. and $\mathrm{E}\{\exp(t_0x_{1,j_{l}^1}^2)\}<\infty$ as Condition \ref{cond:maxiidcolumn} holds for $\vartheta=2$.  

Second, we prove that for given $l_1\neq l_2$, there exist  some positive constants $C$ and $c_0$ such that $|\mathrm{cov} (\check{\mathcal{X}}_{1,j_{l_1}^1,j_{l_1}^2}, \check{\mathcal{X}}_{1,j_{l_2}^1,j_{l_2}^2})| \leq C(p+n)^{-c_0\tau}.$ We note that under $H_0$, $
	\mathrm{cov} (x_{1,j_{l_1}^1}x_{1,j_{l_1}^2}, x_{1,j_{l_2}^1}x_{1,j_{l_2}^2} )= \mathrm{E} ( x_{1,j_{l_1}^1}x_{1,j_{l_1}^2} x_{1,j_{l_2}^1}x_{1,j_{l_2}^2}  ) = 0$ as $j_{l_1}^1\neq j_{l_1}^2$ and $j_{l_2}^1\neq j_{l_2}^2$. It follows that
\begin{align*}
&~\Big|\mathrm{cov} (\check{\mathcal{X}}_{1,j_{l_1}^1,j_{l_1}^2}, \check{\mathcal{X}}_{1,j_{l_2}^1,j_{l_2}^2})\Big|\notag \\
=&~\Big|\mathrm{cov} (\check{\mathcal{X}}_{1,j_{l_1}^1,j_{l_1}^2}, \check{\mathcal{X}}_{1,j_{l_2}^1,j_{l_2}^2})-\mathrm{E} ( x_{1,j_{l_1}^1}x_{1,j_{l_1}^2} x_{1,j_{l_2}^1}x_{1,j_{l_2}^2}  )\Big| \notag \\
\leq &~ \Big|\mathrm{E} (\check{\mathcal{X}}_{1,j_{l_1}^1,j_{l_1}^2} \check{\mathcal{X}}_{1,j_{l_2}^1,j_{l_2}^2})-\mathrm{E} ( x_{1,j_{l_1}^1}x_{1,j_{l_1}^2} x_{1,j_{l_2}^1}x_{1,j_{l_2}^2}  )\Big|+\Big|\mathrm{E} (\check{\mathcal{X}}_{1,j_{l_1}^1,j_{l_1}^2})\times \mathrm{E}  (\check{\mathcal{X}}_{1,j_{l_2}^1,j_{l_2}^2})\Big|. 
\end{align*} By the definition of $\check{\mathcal{X}}_{1,j_{l_2}^1,j_{l_2}^2}$, 
\begin{align*}
	&~\Big|\mathrm{E} (\check{\mathcal{X}}_{1,j_{l_1}^1,j_{l_1}^2} \check{\mathcal{X}}_{1,j_{l_2}^1,j_{l_2}^2})-\mathrm{E} ( x_{1,j_{l_1}^1}x_{1,j_{l_1}^2} x_{1,j_{l_2}^1}x_{1,j_{l_2}^2}  )\Big| \notag \\
\leq &~\Big| \mathrm{E}\Big[|x_{1,j_{l_1}^1}x_{1,j_{l_1}^2} x_{1,j_{l_2}^1}x_{1,j_{l_2}^2}| \Big( \mathbf{1}\{ |x_{1,j^1_{l_1}}x_{1,j^2_{l_1}}  | > \tau_n \}  +\mathbf{1}\{  |x_{1,j^1_{l_2}}x_{1,j^2_{l_2}}|> \tau_n \} \Big) \Big] \Big|.
\end{align*} 
Similarly to \eqref{eq:diffexpindv2} and  \eqref{eq:diffexpindv3}, by H\"older's inequality, we know that there exist some positive constants $C$ and $c_0$ such that
\begin{align*}
	\Big|\mathrm{E} (\check{\mathcal{X}}_{1,j_{l_1}^1,j_{l_1}^2} \check{\mathcal{X}}_{1,j_{l_2}^1,j_{l_2}^2})-\mathrm{E} ( x_{1,j_{l_1}^1}x_{1,j_{l_1}^2} x_{1,j_{l_2}^1}x_{1,j_{l_2}^2}  )\Big| \leq &~ C(p+n)^{-c_0\tau}, \notag \\
	\Big|\mathrm{E} (\check{\mathcal{X}}_{1,j_{l_1}^1,j_{l_1}^2})\times \mathrm{E}  (\check{\mathcal{X}}_{1,j_{l_2}^1,j_{l_2}^2})\Big|\leq &~ C(p+n)^{-c_0\tau}.  
\end{align*}
It follows that $|\mathrm{cov} (\check{\mathcal{X}}_{1,j_{l_1}^1,j_{l_1}^2}, \check{\mathcal{X}}_{1,j_{l_2}^1,j_{l_2}^2})| \leq C(p+n)^{-c_0\tau}.$

In summary, $| \mathbf{U}_s-I_s |_{\max} \leq C (p+n)^{-c_0\tau}$ is obtained. By the matrix version taylor expansion of $\mathbf{U}_s^{1/2}$ at ${I}_s$   \cite[see, e.g.,][]{higham2008functions}, the element wise differences between $\mathbf{U}_s^{1/2}$ and $I_s$ are also bounded by $C (p+n)^{-c_0\tau}$.
\end{proof}

\subsubsection{Proof of Lemma \ref{lm:maxupperlowerbd1} (on Page \pageref{lm:maxupperlowerbd1}, Section \ref{sec:lemma9pproof})} \label{sec:lemma18proof}
By the proof of Lemma \ref{lm:maxupperlowerbd} in Section \ref{sec:lemma17proof}, we have
	\begin{eqnarray*}
	&&	\sum_{s=1}^{d-1} \sum_{1\leq l_1<\ldots<l_s \leq q } P\Big[\bigcap_{t=1}^s \Big\{ (\hat{G}_{l_t})^2/n \geq 4\log p - \log \log p +y \Big\} \Big] \\
	& = & \sum_{s=1}^{d-1} \Biggr[ \frac{1}{s!} \Big( \frac{1}{2\sqrt{2\pi}} e^{-y/2} \Big)^s \{1+ o(1)\} + O(1) e^{-M n^{3/14} + 2s\log p} \Biggr].
	\end{eqnarray*} Since $\log p = o(1) n^{1/7}$ and $d=O(\log ^{1/5} p)$, we know  $ Mn^{3/14} -2d\log p - \log d \rightarrow \infty$ and
$
	 \sum_{s=1}^{d-1} O(1)e^{-M n^{3/14} + 2s\log p} \leq O(1)e^{ - Mn^{3/14} + 2d\log p + \log d }=o(1).
$	It follows that
\begin{eqnarray*}
	&&	\sum_{s=1}^{d-1} \sum_{1\leq l_1<\ldots<l_s \leq q } P(\cap_{t=1}^s \{ (\hat{G}_{l_t})^2/n \geq 4\log p - \log \log p +y \} ) \\
    &=&  \sum_{s=1}^{d-1}  \frac{1}{s!} \Big( \frac{1}{2\sqrt{2\pi}} e^{-y/2} \Big)^s \{1+o(1)\}+ o(1).
\end{eqnarray*}

\subsubsection{Proof of Lemma \ref{lm:expdecylm} (on Page \pageref{lm:expdecylm}, Section \ref{sec:lemma9pproof})} \label{sec:lemma19proof}
\begin{proof}
Recall the definition of $U_l^a$ in \eqref{eq:vladefinition}, and we write 
$
	U_{(j_l^1, j_l^2)}^a=U_l^a. 
$
By Lemma \ref{lm:varianceorder}, we know $\sigma(a)P^n_a =\Theta(pn^{a/2})$. Then for given $l_1,\ldots,l_s$, 
\begin{align}
	&~ P\Big\{    (\sigma(a) P_a^n)^{-1} \Big| \sum_{(j^1_l, j^2_l) \in  L_{I_s} }  U_l^a  \Big| \geq C \Gamma_p  \Big\} \label{eq:firstfirstgeqboundv} \\
\leq & ~ P \Big\{  \Big|n^{-a/2}   \sum_{(j^1_l, j^2_l) \in  L_{I_s} } U_{(j_l^1, j_l^2)}^a\Big|  \geq  Cp\Gamma_p \Big\} \leq P_{U,+}+P_{U,-}, \notag
\end{align}  
where we define
\begin{align*}
P_{U,+}	=&~P \Big(  n^{-a/2}   \sum_{(j^1_l, j^2_l) \in  L_{I_s} } U_{(j_l^1, j_l^2)}^a  \geq  Cp\Gamma_p \Big), \notag \\
 P_{U,-}=&~P \Big(  n^{-a/2}   \sum_{(j^1_l, j^2_l) \in  L_{I_s} } U_{(j_l^1, j_l^2)}^a  \leq  -Cp\Gamma_p  \Big).
\end{align*}
By $q=\binom{p}{2}$,
\begin{eqnarray}
	&&\sum_{s=1}^{d-1}  \sum_{1\leq l_1 < \ldots < l_s\leq q} P\Big(\Big\{   (\sigma(a) P_a^n)^{-1} \Big|  \sum_{(j^1_l, j^2_l) \in  L_{I_s} } U_l^a \Big| \geq \Gamma_p  \Big\} \Big)\label{eq:eachsummedexptwofirst}  \\
	&\leq & dp^{2d} \max_{1\leq s \leq d-1;\, 1\leq l_1 < \ldots < l_s\leq q} (P_{U,+}+P_{U,-}).  \notag
\end{eqnarray}
To prove Lemma \ref{lm:expdecylm}, it suffices to prove  that $P_{U,+}$ and $P_{U,-}$ are $o(d^{-1}p^{-2d})$ for each given $s$ and $l_1,\ldots,l_s$. 

We show  $P_{U,+}=o(d^{-1}p^{-2d})$ in the following and  the same conclusion holds for $P_{U,-}$ by applying similar analysis.  
By the  construction of $L_{I_s}$ in \eqref{eq:defwidset} and the i.i.d. assumption in Condition \ref{cond:maxiidcolumn}, we know that there exists an integer $D\leq 2s$ such that
\begin{align}
P_{U,+}		\leq &~\sum_{k=1}^D P \Big(   \sum_{m=k+1}^p   n^{-a/2} U_{(k, m)}^a \geq {Cp\Gamma_n}/{D} \Big)\label{eq:firstgeqboundv}\\
\leq &~D\max_{1\leq k\leq D} \mathrm{E}\Big[ P_{k} \Big(\sum_{m=k+1}^p   n^{-a/2}  U_{(k,m)}^a \geq  {Cp\Gamma_p }/{D}  \Big)   \Big], \notag	
\end{align} 
where $P_k$ represents the  probability measure conditioning on $\{x_{1,k},\ldots, x_{n,k}\}$ with $k \in \{1,\ldots, p\}$. 
To prove $P_{U,+}=o(d^{-1}p^{-2d})$,  in the following we show that $\mathrm{E}[ P_{k} (\sum_{m=k+1}^p   n^{-a/2}  U_{(k,m)}^a \geq  {C\times p\Gamma_p }/{D})]=o(D^{-1}d^{-1}p^{-2d})$  for $k=1$; and the same conclusion holds for $k\geq 2$ by similar analysis given the i.i.d. assumption in  Condition \ref{cond:maxiidcolumn} and $k\leq D=O(\log^{1/5}p)$.  Specifically, we next prove that $ \mathrm{E}[ P_{1} ( \{  \sum_{m=2}^p    n^{-a/2} U_{(1,m)}^a  \geq  {Cp\Gamma_p }/{D} \} ) ]=o(D^{-1}d^{-1}p^{-2d}).$


Define 
\begin{eqnarray}
	\bar{U}_{x}=n^{-a} \sum_{1 \leq i_1 \neq \ldots \neq i_{a}\leq n}x_{i_1,1}^2\ldots x_{i_{a},1}^2,\label{eq:ubarxdef}
\end{eqnarray} then $\mathrm{E}(\bar{U}_{x})\leq \{\mathrm{E}(x_{11}^2)\}^{a}=\Theta(1)$. Given a constant $t>0$, we define an event 
$
	T_{t,1}=\{  |\bar{U}_{x}-\mathrm{E}(\bar{U}_{x}) |\leq t \},
$ and let $\mathbf{1}_{T_{t,1}}$ denote the indicator function of the event $T_{t,1}$. It follows that
\begin{eqnarray}
	&& \mathrm{E}\Big[ P_{1} \Big( \Big\{  \sum_{m=2}^p    n^{-a/2} U_{(1,m)}^a  \geq  {Cp\Gamma_p }/{D} \Big\} \Big)   \Big] \label{eq:eachsummedexptwo} \\
	&=& \mathrm{E}\Big[ P_{1} \Big( \Big\{  \sum_{m=2}^p    n^{-a/2} U_{(1,m)}^a  \geq  {Cp\Gamma_p }/{D} \Big\} \Big) \times( \mathbf{1}_{T_{t,1}} +\mathbf{1}_{T_{t,1}^c} )  \Big] \notag \\
  &\leq & \mathrm{E}(P_{T_{t,1}})+ P(T_{t,1}^c), \notag
\end{eqnarray} where $\mathbf{1}_{T_{t,1}^c}=1-\mathbf{1}_{T_{t,1}}$; $T_{t,1}^c$ denotes the complement set of the event $T_{t,1}$; and
$
	P_{T_{t,1}}= P_{1} \{  \sum_{m=2}^p    n^{-a/2} U_{(1,m)}^a  \geq  {Cp\Gamma_p }/{D} \}  \times \mathbf{1}_{T_{t,1}}. 
$ It remains to prove that $\mathrm{E}(P_{T_{t,1}})$ and $P(T_{t,1}^c)$ are $o(D^{-1}d^{-1}p^{-2d})$ respectively. 

\smallskip

\textbf{Part 1: $\boldsymbol{\mathrm{E}(P_{T_{t,1}})}$} Given an integer $a$, define $h_p=C(p/\log^{2} p)^{a/(a+1)}$. 
For easy presentation, we let $\mathbf{1}_{H}$ denote an indicator function of the event  $\{|n^{-a/2}U_{(1,m)}^a|\leq h_p\}$. We next  decompose $n^{-a/2}U_{(1,m)}=z_{m,1}+z_{m,2}$,  where  
\begin{align}
  \quad \quad z_{m,1}=&~n^{-a/2} \Big[U_{(1,m)}^a\mathbf{1}_{H}-\mathrm{E}_1\{U_{(1,m)}^a\mathbf{1}_{H}\}\Big], \label{eq:twozdefinitions} \\
    z_{m,2}=&~n^{-a/2} \Big[\mathrm{E}_1\{U_{(1,m)}^a\mathbf{1}_{H}\}+U_{(1,m)}^a(1-\mathbf{1}_{H})\Big]\notag \\
    =&~n^{-a/2} \Big[-\mathrm{E}_1\{U_{(1,m)}^a(1-\mathbf{1}_{H})\}+U_{(1,m)}^a(1-\mathbf{1}_{H})\Big];  \notag
\end{align} in \eqref{eq:twozdefinitions},  $\mathrm{E}_1$ denotes the expectation conditioning on  $\{x_{1,1},\ldots, x_{n,1}\}$, and we use $\mathrm{E}_1\{U_{(1,m)}^a\mathbf{1}_H\}=-\mathrm{E}_1\{U_{(1,m)}^a(1-\mathbf{1}_H)\}$ as $\mathrm{E}_1\{U_{(1,m)}^a\}=0$. Given $n^{-a/2}U_{(1,m)}^a=z_{m,1}+z_{m,2}$, we have $P_{T_{t,1}}\leq P_{z,1}+P_{z,2},$ where we define 
\begin{align*}
	P_{z,1}=P_{1} \Big(  \sum_{m=2}^p z_{m,1} \geq  {Cp\Gamma_p }/{D} \Big)\mathbf{1}_{T_{t,1}}, \   
	P_{z,2}= P_{1} \Big(  \sum_{m=2}^p z_{m,2} \geq  {Cp\Gamma_p }/{D}  \Big)\mathbf{1}_{T_{t,1}}.
\end{align*}To evaluate $\mathrm{E}(P_{T_1})$, we examine $\mathrm{E}(P_{z,1})$ and $ \mathrm{E}(P_{z,2})$ respectively below. 

\smallskip

\textbf{Part 1.1: $\boldsymbol{\mathrm{E}(P_{z,1})}$}
When conditioning on $\{x_{1,1},\ldots, x_{n,1}\}$, since $z_{m,1}$'s are independent and bounded random variables, by Bernstein inequality, 
\begin{eqnarray}
	\quad P_{z,1} \leq C \exp \Big( - \frac{Cp^2 \Gamma_p^2/D^2 }{ \sum_{m=2}^p \mathrm{E}_1 (z_{m,1}^2) +Ch_p p \Gamma_p/D }  \Big)\mathbf{1}_{T_1}. \label{eq:pz1probbound11}
\end{eqnarray} 
Note that $0\leq \mathrm{E}_1 (z_{m,1}^2)   
	 \leq  \mathrm{E}_1 [\{n^{-a/2}U_{(1,m)}^a\}^2]$ and
\begin{align*}
	\mathrm{E}_1 \Big[\{n^{-a/2}U_{(1,m)}^a\}^2\Big] =&~ n^{-a } \sum_{\substack{1\leq i_1\neq \ldots \neq i_a \leq n; \\ 1\leq \tilde{i}_1\neq \ldots \neq \tilde{i}_a \leq n}}  \Big(\prod_{r=1}^ax_{i_r,1} x_{\tilde{i}_r,1}\Big)  \times \mathrm{E}\Big(\prod_{r=1}^ax_{i_r,m}x_{\tilde{i}_r,m}\Big) \notag \\
	=&~ a! n^{-a} \sum_{1\leq i_1\neq \ldots \neq i_a \leq n}\Big(\prod_{r=1}^a x_{i_r,1}^2\Big) \times \{\mathrm{E}(x_{1,m}^2)\}^a \notag \\
	=&~ a! \bar{U}_x \times \{\mathrm{E}(x_{1,m}^2)\}^a , \notag 
\end{align*} 
where from the first equation to the second equation, we use the fact that $\mathrm{E}(\prod_{r=1}^ax_{i_r,m}x_{\tilde{i}_r,m})\neq 0$ 
only when $\{i_1,\ldots , i_a \}=\{\tilde{i}_1,\ldots , \tilde{i}_a \}$. It follows that $\mathrm{E}_1 (z_{m,1}^2) \leq C\times \bar{U}_x$. As $\mathbf{1}_{T_{t,1}}$ indicates the event $\{  |\bar{U}_{x}-\mathrm{E}(\bar{U}_{x}) |\leq t \}$ and $\mathrm{E}(\bar{U}_{x}) =\Theta(1)$, it suffices to consider $\mathrm{E}_1(z_{m,1}^2)=\Theta(1)$ in \eqref{eq:pz1probbound11} and  then 
\begin{eqnarray}
	\mathrm{E}(P_{z,1}) \leq \exp\{-Cp\Gamma_p/(Dh_p)\}. \label{eq:pz1probbound}
\end{eqnarray}


\textbf{Part 1.2: $\boldsymbol{\mathrm{E}(P_{z,2})}$}  By the definition of $z_{m,2}$ in \eqref{eq:twozdefinitions}, 
\begin{eqnarray}
	\quad \quad &&\mathrm{E}(P_{z,2} )\leq  P \Big( \max_{2 \leq  m \leq p} |n^{-a/2}U_{(1,m)}^a|> h_p \Big) 	\leq  p P(|n^{-a/2}U_{(1,2)}^a|>h_p), \label{eq:epz2updef}
\end{eqnarray} where the last inequality follows from the i.i.d. assumption in Condition \ref{cond:maxiidcolumn}. By the result in Section \ref{sec:formulaevu}, we know
$
U_{(1,2)}^a=\sum_{1 \leq i_1 \neq \ldots \neq i_a \leq n}\prod_{k=1}^a x_{i_k, 1} x_{i_k, 2}
$ can be written as  a linear combination of 
$
	\prod_{k=1}^{\iota}\{  \sum_{i=1}^n  (x_{i,1}x_{i,2})^{a_k} \},
$ where $a_1,\ldots, a_{\iota}$ are positive integers such that $a_1+\ldots+a_{\iota}=a$. It follows that for finite integer $a$,
\begin{align*}
	& P( |n^{-a/2}U_{(1,2)}^a | > h_p )  \notag \\
	\leq  & \sum_{a_1+\ldots+a_{\iota}=a}  P\Big(   n^{-a/2}  \prod_{k=1}^{\iota}\Big| \sum_{i=1}^n  (x_{i,1}x_{i,2})^{a_k}\Big| > C h_p  \Big).  \notag \\
	\leq & \sum_{a_1+\ldots+a_{\iota}=a} \sum_{k=1}^{\iota} P\Big( \sum_{i=1}^n |x_{i,1}x_{i,2}/\sqrt{n}|^{a_k} >Ch_p^{a_k/a} \Big). \notag 
\end{align*} 
\textit{Case 1:} If $a_k=1$, since  Condition \ref{cond:maxiidcolumn} holds for $\varsigma=2$ in Theorem \ref{thm:extlimit}, we know $x_{i,1}x_{i,2}$, $i=1,\ldots, n$, are i.i.d. sub-exponential random variables. By the  Bernstein-type inequality of sub-exponential random variables, we have
\begin{eqnarray}
	\quad \quad  && P\Big( \sum_{i=1}^n |x_{i,1}x_{i,2}| >C\sqrt{n}h_p^{1/a} \Big) \label{eq:sboundexp11} 	\leq  C\exp ( -C\min \{ Ch_p^{2/a}, C\sqrt{n}h_p^{1/a} \} ). 
\end{eqnarray} 
\textit{Case 2:} If $2 \leq a_k \leq a$, we let $B_p=Cn^{-1/6} h_p^{2/(3a)}$. We then decompose $|{x_{i,1}x_{i,2}}/{\sqrt{n} }|^{a_k}=s_i+t_i$, where we define
\begin{align*}
&s_i=|{x_{i,1}x_{i,2}}/{\sqrt{n} } |^{a_k} \mathbf{1}_{H_{B_p}}-\mu_i, \quad t_i=|{x_{i,1}x_{i,2}}/{\sqrt{n} } |^{a_k} (1-\mathbf{1}_{H_{B_p}})+\mu_i, \notag \\
&\mathbf{1}_{H_{B_p}}=\mathbf{1}_{\{ |   {x_{i,1}x_{i,2}}/{\sqrt{n} } | \leq B_p\} }, \quad \mu_i = \mathrm{E}\{|{x_{i,1}x_{i,2}}/{\sqrt{n} } |^{a_k} \mathbf{1}_{H_{B_p}}\}. 	\notag 
\end{align*}
It follows that  
\begin{align*}
 P\Big( \sum_{i=1}^n |x_{i,1}x_{i,2}/\sqrt{n}|^{a_k} >Ch_p^{a_k/a} \Big)
	\leq  P\Big( \sum_{i=1}^n s_i > Ch_p^{a_k/a} \Big)+P\Big( \sum_{i=1}^n t_i > Ch_p^{a_k/a} \Big).
\end{align*} Since $|s_i| \leq C\times B_p^{a_k}$ from construction, by Bernstein inequality,
\begin{eqnarray}
	\quad \quad \quad \quad P\Big( \sum_{i=1}^n s_i > Ch_p^{a_k/a} \Big)  \leq  C\exp\Big(-\frac{C h_p^{2a_k/a}}{ \sum_{i=1}^n \mathrm{E}(s_i^2)  +CB_p^{a_k}h_p^{a_k/a}}  \Big). \label{eq:sboundexp}
\end{eqnarray}
As $2 \leq a_k \leq a$, by Condition \ref{cond:maxiidcolumn}, we have
\begin{align*}
	\sum_{i=1}^n \mathrm{E}(s_i^2) \leq  \sum_{i=1}^n \mathrm{E} \Big\{\Big( \frac{x_{i,1}x_{i,2}}{\sqrt{n} } \Big)^{2a_k}  \Big\}\leq   \frac{\mathrm{E}\{(x_{1,1}x_{1,2})^{2a_k}\}}{n^{a_k-1}} \leq  \mathrm{E}[(x_{1,1}x_{1,2})^{2a_k}] < \infty .
\end{align*} Since $h_p^{1/a}/B_p \to \infty$,  from \eqref{eq:sboundexp}, we have
\begin{eqnarray}
	P\Big( \sum_{i=1}^n s_i > Ch_p^{a_k/a} \Big) \leq  \exp ( - C{  h_p^{2/a} }/B_p^{2}). \label{eq:sboundexp2}
\end{eqnarray}  
In addition, by the definition of $t_i$, 
\begin{align*}
P\Big( \sum_{i=1}^n t_i > Ch_p^{a_k/a} \Big)\leq 	P\Big\{ \sum_{i=1}^n |{x_{i,1}x_{i,2}}/{\sqrt{n} } |^{a_k} (1-\mathbf{1}_{H_{B_p}})> Ch_p^{a_k/a}-\Big|\sum_{i=1}^n\mu_i\Big| \Big\}.
\end{align*}We note that $\sum_{i=1}^n \mu_i\leq n^{-1} \times \sum_{i=1}^n [\mathrm{E}\{(x_{1,1}x_{1,2} )^{2a_k}\}]^{1/2}<\infty$ by H\"older's inequality and Condition \ref{cond:maxiidcolumn}. 
As $h_p\to \infty$,  $Ch_p^{a_k/a}-|\sum_{i=1}^n\mu_i|>0$ when $n$ and $p$ are sufficiently large. Since $1-\mathbf{1}_{H_{B_p}}$ indicates $|{x_{i,1}x_{i,2}}/{\sqrt{n} } |>B_p$, 
\begin{align}
P\Big( \sum_{i=1}^n t_i > Ch_p^{a_k/a} \Big) \leq & ~P\Big( \max_{1\leq i \leq n}  |{x_{i,1}x_{i,2}}/{\sqrt{n} } |^{a_k}  > B_p^{a_k}  \Big)\label{eq:sboundexp3}	\\
\leq &~ n \times P (  |{x_{i,1}x_{i,2}}/{\sqrt{n} } |>B_p) \notag \\
\leq &~ n \times \mathrm{E}\{ \exp (t_0 | x_{1,1}x_{1,2}| )\}/ \exp\{t_0 (\sqrt{n} B_p)\}  \notag \\
\leq &~  \exp (-C \sqrt{n}B_p  + \log n), \notag
\end{align}
where we use $\mathrm{E}\{ \exp (t_0 | x_{1,1}x_{1,2}|)\} \leq \mathrm{E}\{ \exp (t_0 (x_{1,1}^2+x_{1,2}^2)/2)\} <\infty $ as Condition \ref{cond:maxiidcolumn} holds for $\varsigma=2$. By  \eqref{eq:epz2updef},  \eqref{eq:sboundexp11}, \eqref{eq:sboundexp2} and \eqref{eq:sboundexp3}, 
\begin{align}
\quad	\mathrm{E}(P_{z,2}) \leq &~  Cp \times \Big[ \exp \Big( -C\min \{ Ch_p^{2/a}, C\sqrt{n}h_p^{1/a} \} \Big) \label{eq:pz2upperbound} \\
	&~+\exp(-C {  h_p^{2/a} }/B_p^{2})+ \exp (-C \sqrt{n}B_p  + \log n)\Big]. \notag 
\end{align}

\smallskip

\textbf{Part 2: $\boldsymbol{P(T_{t,1}^c)}$}
By the definition in  \eqref{eq:eachsummedexptwo},  $P(T_{t,1}^c)=P(|\bar{U}_{x}-\mathrm{E}(\bar{U}_{x})|>t)$.
Moreover, by the definition in \eqref{eq:ubarxdef}, $\mathrm{E}(\bar{U}_{x})=\Theta(1)$ and $\bar{U}_x \geq 0$. Therefore we know there exist large positive constants $C$ and $t$ such that
$
	\{ |\bar{U}_{x}-\mathrm{E}(\bar{U}_{x})|>t\} \subseteq \{ \bar{U}_x > Ct \}
$  and $P(T_{t,1}^c)\leq P(\bar{U}_x > Ct).$
Since $	\bar{U}_x \leq (\sum_{i=1}^n x_{i,1}^2 /n)^a$ and $x_{i,1}^2$ are i.i.d.  sub-exponential random variables, we have
\begin{align}
P(T_{t,1}^c)\leq &~P\Big\{ \Big(\sum_{i=1}^n x_{i1}^2/n \Big)^a  \geq Ct \Big\} =P\Big( \sum_{i=1}^n x_{i1}^2/n  \geq Ct^{1/a} \Big) \label{eq:ptcprobupbound} 
 \\
	\leq &~ C \exp( - C n ), \notag	
\end{align}
where the last inequality is obtained by the Bernstein-type inequality of sub-exponential random variables.

By the analysis above, $\eqref{eq:eachsummedexptwo} \leq \mathrm{E}(P_{z,1})+\mathrm{E}(P_{z,2})+P(T_1^c)$. 
Recall that $h_p=C(p/\log^2 p)^{a/(a+1)}$, $\log p=o(n^{1/7})$, $\Gamma_p=\Theta(\log^{-1/2}p)$, $D=O(\log^{1/5}p)$ and $B_p=Cn^{-1/6}h_p^{2/(3a)}$. 
Then combining  \eqref{eq:pz1probbound}, \eqref{eq:pz2upperbound} and  \eqref{eq:ptcprobupbound}, we have $\eqref{eq:eachsummedexptwo}=o(D^{-1}d^{-1}p^{-2d}).$ Therefore Lemma \ref{lm:expdecylm} is proved.\end{proof}

\subsubsection{Proof of Lemma \ref{lm:boundedsmalldiff} (on Page \pageref{lm:boundedsmalldiff}, Section  \ref{sec:asymindptproofsec})} \label{sec:lemma10proof}

Similarly to Section \ref{sec:lemma9pproof}, we first prove Lemma \ref{lm:boundedsmalldiff} for $m=1$ in Section \ref{par:m1lmma32} and then for $m>1$ in Section \ref{par:mlarger1lmma32}. 

\paragraph{Proof for $m=1$} \label{par:m1lmma32}
Specifically, in this section, we prove for finite integer $a$, 
\begin{equation}
	\begin{aligned}
		\Biggr| P\Big( \frac{{M}_n}{n} > y_p, \frac{\tilde{\mathcal{U}}(a)  }{\sigma(a)} \leq z \Big) - P\Big( \frac{{M}_n}{n} > y_p \Big) P\Big(\frac{\tilde{\mathcal{U}}(a)  }{\sigma(a)} \leq z \Big) \Biggr| \rightarrow 0.  \label{eq:asymindptgoal22}
	\end{aligned}
\end{equation}

To prove \eqref{eq:asymindptgoal22}, we start by proving the following two conclusions \eqref{eq:mnmnhatdif} and \eqref{eq:smalldifmmhat2}, which suggest that $M_n$ and $\hat{M}_n$ have small difference in probability.
   To be specific,  as $n,p\to \infty$, 
\begin{eqnarray}
	|P(M_n/n>y_p)-P(\hat{M}_n/n>y_p)|\to 0, \label{eq:mnmnhatdif}
\end{eqnarray} and 
\begin{align}
	& |P(M_n/n>y_p\, , \, {\tilde{\mathcal{U}}(a)  }/{\sigma(a)}  \leq z )\label{eq:smalldifmmhat2} \\
	&\ -P(\hat{M}_n/n>y_p \, , \, {\tilde{\mathcal{U}}(a)  }/{\sigma(a)}  \leq z )|\to 0. \notag
\end{align}
To prove \eqref{eq:mnmnhatdif} and \eqref{eq:smalldifmmhat2}, recall that in  \eqref{eq:indpalldefinewvar}, $M_n$ and $\hat{M}_n$ are defined using $\tilde{G}_l$ and $\hat{G}_l$ respectively. We next focus on the difference between  $\tilde{G}_l$ and $\hat{G}_l$.  
Since  $\tilde{G}_l$ and $\hat{G}_l$ will not change if the data $x_{i,j}$ is scaled by its standard deviation, then  we assume, without loss of generality,  $\sigma_{j,j}=1$, $j=1,\ldots, p$ in the following discussion. 

By the definitions in \eqref{eq:indpalldefinewvar}, we have 
\begin{align*}
P\Big( \max_{1\leq l \leq q} | \tilde{G}_l-\hat{G}_l | \geq (\log p)^{-1} \Big) 
	\leq  P \Big( \max_{1\leq l \leq q} \max_{1 \leq i \leq n} | x_{i,j^1_l}x_{i,j^2_l}| \geq \tau_n    \Big).  	
\end{align*} Note that $ | x_{i,j^1_l}x_{i,j^2_l}| \leq (  x_{i,j^1_l}^2+ x_{i,j^2_l}^2 )/2$. Then 
\begin{align}
	&~ P \Big( \max_{1\leq l \leq q} \max_{1 \leq i \leq n} | x_{i,j^1_l}x_{i,j^2_l}| \geq \tau_n    \Big) \notag \\
 \leq & ~ P \Big( \max_{1\leq l \leq q} \max_{1 \leq i \leq n}  (x_{i,j^1_l}^2+ x_{i,j^2_l}^2) \geq 2 \tau_n    \Big) \notag \\
  \leq &~ P \Big( \max_{1\leq l \leq q} \max_{1 \leq i \leq n}  x_{i,j^1_l}^2 \geq \tau_n    \Big) + P \Big( \max_{1\leq l \leq q} \max_{1 \leq i \leq n}  x_{i,j^2_l}^2 \geq \tau_n    \Big) \label{eq:maxbound1taun}\\
 \leq & ~ 2 P \Big( \max_{1\leq j \leq p} \max_{1 \leq i \leq n}  x_{i,j}^2 \geq \tau_n    \Big)  \label{eq:maxbound2taun} \\
 \leq & ~ 2 np \max_{1\leq j\leq p} P(|x_{1,j}^2|\geq \tau_n). \notag
\end{align}
From \eqref{eq:maxbound1taun} to \eqref{eq:maxbound2taun}, we use $\max_{1\leq l\leq q}x_{i,j_l^k}^2=\max_{1\leq j\leq p} x_{i,j}^2$ for each $i$ and $k=1,2$. To see this, recall the notation  defined in  Section \ref{sec:asymindptproofsec} (on Page \pageref{sec:asymindptproofsec}). In particular, subscript $l$ is defined to indicate a pair of indexes $(j_l^1,j_l^2)$ with $1\leq j_l^1<j_l^2\leq p$. Since $j_l^1$ and $j_l^1$  only take values from the range $\{1,\ldots,p\}$, we know $\{j_l^k: 1\leq l\leq q\}\subseteq \{1,\ldots,p\}$ for $k=1,2$, and  then $\max_{1\leq l\leq q}x_{i,j_l^1}^2=\max_{1\leq j\leq p} x_{i,j}^2$. 
Moreover, by Condition \ref{cond:maxiidcolumn} with $\varsigma=2$, 
 \begin{align*}
 	 np \max_{1\leq j\leq p} P(|x_{1,j}^2|\geq \tau_n)\leq C np  (n+p)^{-\tau} \mathrm{E} \exp(x_{1,1}^2) \rightarrow 0.
 \end{align*}
 It follows that $P( \max_{1\leq l \leq q} | \tilde{G}_l-\hat{G}_l | \geq (\log p)^{-1}) \to 0$.
Conditioning on $\max_{1\leq l \leq q} | \tilde{G}_l-\hat{G}_l |\leq (\log p)^{-1}$, by Lemma \ref{lm:maxdiffbound} and $|\hat{G}_l|\leq \tau_n$, 
\begin{eqnarray*}
	|M_n-\hat{M}_n|&=&\Big| \max_{1\leq l \leq q} ( \tilde{G}_l)^2 -  \max_{1\leq l \leq q} ( \hat{G}_l)^2  \Big| \notag \\
	&\leq & 2 \max_{1\leq l \leq q} | \hat{G}_l| \max_{1\leq l \leq q} |  \tilde{G}_l- \hat{G}_l |+ \max_{1\leq l \leq q}  |  \tilde{G}_l- \hat{G}_l |^2 \notag \\
	&\leq & 2\tau_n /\log p  + (\log p)^{-2}. 
\end{eqnarray*} 
Recall that  $\tau_n=O(\log(p+n))$, then $|M_n/n-\hat{M}_n/n|\xrightarrow{P} 0$. Therefore \eqref{eq:mnmnhatdif} and \eqref{eq:smalldifmmhat2}  are obtained.



 


Given \eqref{eq:mnmnhatdif} and \eqref{eq:smalldifmmhat2}, we next prove \eqref{eq:asymindptgoal22}. In particular, we write
\begin{align*}
&~  P\Big( \frac{{M}_n}{n} > y_p, \frac{\tilde{\mathcal{U}}(a)  }{\sigma(a)} \leq z \Big) -P\Big( \frac{{M}_n}{n} > y_p \Big) P\Big(\frac{\tilde{\mathcal{U}}(a)  }{\sigma(a)} \leq z \Big) =\Delta_{P,1}+ \Delta_{P,2}+\Delta_{P,3},	
\end{align*} where we define
\begin{align*}
\Delta_{P,1}=&~P \Big( {{M}_n}/{n} > y_p,\, {\tilde{\mathcal{U}}(a)  }/{\sigma(a)} \leq z \Big)   -  P\Big( {\hat{M}_n}/{n} > y_p,\,  {\tilde{\mathcal{U}}(a)  }/{\sigma(a)} \leq z \Big), \notag \\
\Delta_{P,2}=&~ P\Big( {\hat{M}_n}/{n} > y_p,\,  {\tilde{\mathcal{U}}(a)  }/{\sigma(a)} \leq z \Big) -P\Big( {\hat{M}_n}/{n} > y_p)\times P({\tilde{\mathcal{U}}(a)  }/{\sigma(a)} \leq z \Big), \notag \\
\Delta_{P,3}=&~P\Big( {\hat{M}_n}/{n} > y_p\Big)\times P\Big({\tilde{\mathcal{U}}(a)  }/{\sigma(a)} \leq z \Big) \notag \\
&~\ - P\Big( {{M}_n}/{n} > y_p\Big)\times P\Big({\tilde{\mathcal{U}}(a)  }/{\sigma(a)} \leq z \Big).  
\end{align*}
Note that the left hand side of $\eqref{eq:asymindptgoal22} \leq |\Delta_{p,1}|+ |\Delta_{p,2}|+|\Delta_{p,3}|$. 
By Lemma  \ref{lm:firstsetpthm3proof}, $|\Delta_{p,2}|\to 0$;  by \eqref{eq:smalldifmmhat2}, $|\Delta_{p,1}|\to 0$; by $|\Delta_{p,3}|\leq | P( {\hat{M}_n}/{n} > y_p ) - P( {{M}_n}/{n} > y_p)|$ and \eqref{eq:mnmnhatdif},  $|\Delta_{p,3}|\to 0$. In summary, \eqref{eq:asymindptgoal22} is proved.

\paragraph{Proof for $m>1$} \label{par:mlarger1lmma32}
Following the proof in Section  \ref{par:m1lmma32}, we know that \eqref{eq:mnmnhatdif} still holds and similarly to  \eqref{eq:smalldifmmhat2}, 
\begin{align}
	&|P(M_n/n>y_p,\, \tilde{\mathcal{U}}(a_1)/\sigma(a_1)\leq z_1, \, \ldots, \, \tilde{\mathcal{U}}(a_m)/\sigma(a_m)\leq z_m ) \notag \\ 
	&\, - P(\hat{M}_n/n>y_p,\,  \tilde{\mathcal{U}}(a_1)/\sigma(a_1)\leq z_1, \, \ldots, \, \tilde{\mathcal{U}}(a_m)/\sigma(a_m)\leq z_m) |\to 0. \notag
\end{align} 
Given these results and Lemma \ref{lm:firstsetpthm3proof}, we know that Lemma \ref{lm:boundedsmalldiff} holds for $m>1$,  following the arguments in Section \ref{par:m1lmma32} similarly. 

\smallskip

\subsubsection{Proof of Lemma \ref{lm:addprop0termtrue} (on Page \pageref{lm:addprop0termtrue}, Section \ref{sec:asymindptproofsec})}\label{sec:lemma11proof}

Similarly to Section \ref{sec:lemma10proof}, we first prove Lemma \ref{lm:addprop0termtrue}  for $m=1$ in Section \ref{par:pfm1lemma33}, and then discuss the case for $m>1$ in Section \ref{par:pfmlarger1lemma33}. 
\paragraph{Proof for $m=1$}  \label{par:pfm1lemma33}
Specifically, in this section, we prove for finite integer $a$ and given $z$, 
\begin{align}
\Big| &P\Big(  \frac{{\mathcal{U}}(a)  }{\sigma(a)} \leq z,\, n  \mathcal{U}^2(\infty) > y_p \Big) \label{eq:mnyptildem2}  \\
& -   P \Big( \frac{{\mathcal{U}}(a)  }{\sigma(a)} \leq z \Big)P\Big( n  \mathcal{U}^2(\infty) > y_p \Big)\Big| \to 0. \notag
\end{align}
To prove this,  we use $M_n/n$ as an intermediate variable and first show 
\begin{equation}
\begin{aligned}
	\Big|P \Big( \frac{\mathcal{U}(a)}{\sigma(a)} >z, \frac{M_n}{n} > y_p  \Big)-P \Big( \frac{\mathcal{U}(a)}{\sigma(a)} >z \Big) P\Big(\frac{M_n}{n} > y_p\Big) \Big| \rightarrow 0. \label{eq:indfirststep}
\end{aligned} 	
\end{equation} 
To facilitate the proof, we define some notation. Given small constant $\epsilon>0$, 
\begin{align*}
	& P_{uz} =P \Big( \frac{\mathcal{U}(a)}{\sigma(a)} >z \Big), \quad P_{zy}= P \Big( \frac{\mathcal{U}(a)}{\sigma(a)} >z, \frac{M_n}{n} > y_p  \Big), \notag \\
	& P_{uz+\epsilon}= P\Big( \frac{\tilde{\mathcal{U}}(a)}{\sigma(a)} > z +\epsilon \Big), \quad P_{z+\epsilon}= P\Big( \frac{\tilde{\mathcal{U}}(a)}{\sigma(a)} > z +\epsilon , \frac{M_n}{n} > y_p \Big), \notag \\
& P_{uz-\epsilon}= P\Big( \frac{\tilde{\mathcal{U}}(a)}{\sigma(a)} > z -\epsilon \Big), \quad
P_{z-\epsilon}= P\Big( \frac{\tilde{\mathcal{U}}(a)}{\sigma(a)} > z -\epsilon , \frac{M_n}{n} > y_p \Big), \notag \\
& P_{y_p}=P\Big(\frac{M_n}{n}  >y_p\Big),
\end{align*}  $\Phi(\cdot)$ is the cumulative distribution function of standard normal distribution, and $\Bar{\Phi}(\cdot)=1-\Phi(\cdot)$. Then
\begin{align*}
	\eqref{eq:indfirststep} =&~ |P_{zy}-P_{uz}\times P_{y_p}| \notag \\
	\leq &~ |P_{zy}-P_{z+\epsilon}|+|P_{z+\epsilon}-P_{uz+\epsilon}P_{y_p}|+ |P_{uz+\epsilon}P_{y_p}-P_{uz}P_{y_p}|. 
\end{align*} 
We next show $\eqref{eq:indfirststep}\to 0$ by proving the three parts above all converges to $0$ respectively. 

First we show $|P_{zy}-P_{z+\epsilon}|\to 0$. Note that $P_{z+\epsilon} \leq P_{zy} \leq P_{z-\epsilon}$, then $|P_{zy}-P_{z+\epsilon}|\leq |P_{z-\epsilon}-P_{z+\epsilon}|$. In addition, 
\begin{align*}
&~|P_{z-\epsilon}-P_{z+\epsilon}| \notag \\
\leq &~ |P_{z-\epsilon}-P_{uz-\epsilon}\times P_{y_p}|	+ |P_{uz-\epsilon}\times P_{y_p} -P_{uz+\epsilon} \times P_{y_p} | +|P_{uz+\epsilon}\times P_{y_p}-P_{z+\epsilon}| \notag \\
\leq &~ o(1)+|P_{uz+\epsilon} -P_{uz-\epsilon} |,
\end{align*} where we use  \eqref{eq:asymindptgoal22} in the last inequality.  
Moreover, by the proof of Theorem \ref{thm:jointnormal} in Section \ref{sec:detailofjointnormal}, we know $\tilde{\mathcal{U}}(a)/\sigma(a)\xrightarrow{D}\mathcal{N}(0,1)$. Thus when $n,p\rightarrow \infty$ and $\epsilon \rightarrow 0$,
\begin{align*}
	& ~|P_{uz+\epsilon} -P_{uz-\epsilon}| \notag \\
	\leq &~ |P_{uz+\epsilon}-\bar{\Phi}(z+\epsilon)|+ |\bar{\Phi}(z+\epsilon)-\bar{\Phi}(z-\epsilon)|+|P_{uz-\epsilon}-\bar{\Phi}(z-\epsilon)|+o(1) \notag \\
	\to & ~0.
\end{align*}
Second, we know $|P_{z+\epsilon}-P_{uz+\epsilon}P_{y_p}|\to 0$ by \eqref{eq:asymindptgoal22}. Last, we show $|P_{uz+\epsilon}P_{y_p}-P_{uz}P_{y_p}|\to 0.$  
By the proof of Theorem \ref{thm:jointnormal} in Section \ref{sec:detailofjointnormal}, we know $\tilde{\mathcal{U}}(a)/\sigma(a)\xrightarrow{D}\mathcal{N}(0,1)$, $ \{ \mathcal{U}(a)-\tilde{\mathcal{U}}(a)/\sigma(a)\}\xrightarrow{P} 0$, 
and ${\mathcal{U}}(a)/\sigma(a)\xrightarrow{D}\mathcal{N}(0,1)$. 
Thus when $n,p\rightarrow \infty$ and $\epsilon \rightarrow 0$,
\begin{eqnarray*}
	&&|P_{uz+\epsilon}P_{y_p}-P_{uz}P_{y_p} |\notag \\
	&\leq &|P_{uz+\epsilon}-P_{uz}| \notag \\
	&\leq &|P_{uz+\epsilon}-\bar{\Phi}(z+\epsilon)|+ |\bar{\Phi}(z+\epsilon)-\bar{\Phi}(z)|+|P_{uz}-\bar{\Phi}(z)|+o(1)\notag \\
& \to &0.
\end{eqnarray*}
In summary \eqref{eq:indfirststep} is proved.

We next prove \eqref{eq:mnyptildem2} similarly to the proof of  \eqref{eq:indfirststep}. 
 Specifically, we write
 \begin{align*}
 	& ~\Big| P\Big( n  \mathcal{U}^2(\infty) > y_p, \frac{{\mathcal{U}}(a)  }{\sigma(a)} \leq z \Big) -  P\Big( n  \mathcal{U}^2(\infty) > y_p \Big) P \Big( \frac{{\mathcal{U}}(a)  }{\sigma(a)} \leq z \Big)\Big|\notag \\
	=&~|P_{z0}-P_{y0}\times P_{uz}|, 
\end{align*}where we define $P_{z0}=P ( n  \mathcal{U}^2(\infty) > y_p ,\frac{\mathcal{U}(a)}{\sigma(a)} >z )$ and $P_{y0}=P(n  \mathcal{U}^2(\infty) > y_p ).$
 Note that
\begin{align*}
 |P_{z0}-P_{y0}P_{uz}| \leq  |P_{z0}-P_{zy-\epsilon}|+ |P_{zy-\epsilon}-P_{y-\epsilon}P_{uz}| + | P_{y-\epsilon} P_{uz}-P_{y0}P_{uz} |,
\end{align*} 
where 
\begin{align*}
	& P_{zy-\epsilon}=P \Big(\frac{M_n}{n} > y_p -\epsilon, \, \frac{\mathcal{U}(a)}{\sigma(a)} >z \Big), \quad P_{y-\epsilon}=P\Big( \frac{M_n}{n} > y_p -\epsilon   \Big), \notag \\
	& P_{zy+\epsilon}=P \Big(\frac{M_n}{n} > y_p +\epsilon , \, \frac{\mathcal{U}(a)}{\sigma(a)} >z\Big), \quad P_{y+\epsilon}=P\Big( \frac{M_n}{n} > y_p +\epsilon   \Big).	
\end{align*}
To prove \eqref{eq:mnyptildem2}, we will show  $|P_{z0}-P_{zy-\epsilon}|, |P_{zy-\epsilon}-P_{y-\epsilon}P_{uz}|,$ and $ | P_{y-\epsilon} P_{uz}-P_{y0}P_{uz} |$ all converge to 0 respectively.

First we show $|P_{z0}-P_{zy-\epsilon}|\to 0$. Note that   $W_n\xrightarrow{P}0$ where $W_n=(n^2\mathcal{U}^2(\infty)-M_n)/n$ by the proof of Theorem 3 in \cite{cai2011}. Then for any $ \epsilon >0$,  $P(|W_n| >\epsilon ) \rightarrow 0$. Since $P_{zy+\epsilon}-P( |W_n|>\epsilon ) \leq P_{z0} \leq P_{zy-\epsilon}+P( |W_n|>\epsilon )$, we have $|P_{z0}-P_{zy-\epsilon}| \leq  |P_{zy-\epsilon}-P_{zy+\epsilon}| + o(1)$. Furthermore,
\begin{align*}
&~|P_{zy-\epsilon}-P_{zy+\epsilon}| \notag \\
\leq &~|P_{zy-\epsilon}-P_{y-\epsilon}P_{uz}|+|P_{y-\epsilon}P_{uz}- P_{y+\epsilon}P_{uz}| +|P_{y+\epsilon}P_{uz}-P_{zy+\epsilon}| \to 0, 
\end{align*}where the last equation follows from \eqref{eq:indfirststep} and $|P_{y-\epsilon}-P_{y+\epsilon}|\to 0$ when $\epsilon \rightarrow 0$.
Second we know $|P_{zy-\epsilon}-P_{y-\epsilon}P_{uz}|\to 0$ by \eqref{eq:indfirststep}. 
Last we show $ | P_{y-\epsilon} P_{uz}-P_{y0}P_{uz} |\to 0$. In particular, as $P_{y+\epsilon}- P( |W_n|>\epsilon ) \leq P_{y0} \leq P_{y-\epsilon}+ P( |W_n|>\epsilon )$ and $P( |W_n|>\epsilon )\to 0$, we have
\begin{align*}
| P_{y-\epsilon} P_{uz}-P_{y0}P_{uz} | \leq |P_{y-\epsilon}-P_{y0}| \leq | P_{y-\epsilon}-P_{y+\epsilon} |+o(1)\to 0.
\end{align*}
In summary, Lemma \ref{lm:addprop0termtrue} is proved.

\paragraph{Proof for $m>1$}  \label{par:pfmlarger1lemma33}

Note that  $W_n=\{{n^2 \mathcal{U}^2(\infty)-M_n}\}/{n}\xrightarrow{P}0$ and $\tilde{\mathcal{U}}^*(a_r)=\mathcal{U}(a_r)-\tilde{\mathcal{U}}(a_r)\xrightarrow{P}0$ for each $r=1,\ldots,m$ as argued in Section \ref{sec:asymindptproofsec}. Therefore when $m$ is finite,  the arguments above can be applied to prove Lemma \ref{lm:addprop0termtrue} for $m>1$ similarly. 

\subsection{Lemmas for the proof of Theorem \ref{thm:computation}} \label{lm:pfthm24lm}

\subsubsection{Proof of Lemma \ref{lm:varestcong2} (on Page \pageref{lm:varestcong2}, Section \ref{sec:proofthm24var})} \label{sec:lma41}

We first prove ${ \mathbb{V}_{u,1}(a) }/{ \mathrm{E}\{\mathbb{V}_{u,1}(a)\} } \xrightarrow{P} 1,$ and it suffices to prove
$
	{ \mathrm{var}\{\mathbb{V}_{u,1}(a)\}  }/{ \mathrm{E}^2\{\mathbb{V}_{u,1}(a) \} } \to 0.
$ 
By the notation defined at the beginning of Section \ref{par:notationindpcond}, we have
\begin{align*}
	&~\mathrm{var}\{\mathbb{V}_{u,1}(a)\} \notag \\
	=&~\mathrm{E}\{\mathbb{V}_{u,1}^2(a)\}-\mathrm{E}^2\{\mathbb{V}_{u,1}(a)\} \notag \\
=&~\frac{(2a!)^2}{(P^n_a)^4} \sum_{\substack{ \mathbf{i},\,\tilde{\mathbf{i}}\in \mathcal{P}(n,a) ;\\1\leq j_1\neq j_2\leq p,\\ 1\leq j_3\neq j_4\leq p} }\Big[ \mathrm{E}\Big(\prod_{t=1}^a x_{i_t,j_1}^2x_{i_t,j_2}^2 x_{\tilde{i}_t,j_3}^2 x_{\tilde{i}_t,j_4}^2\Big)-\Big\{\mathrm{E}(x_{1,j_1}^2 x_{1,j_2}^2)\mathrm{E}(x_{1,j_3}^2 x_{1,j_4}^2)\Big\}^a \Big].
\end{align*} 
To evaluate $\mathrm{var}\{\mathbb{V}_{u,1}(a)\}$, we consider the summed term in $\mathrm{var}\{\mathbb{V}_{u,1}(a)\}$, that is, 
\begin{align}
\mathrm{E}\Big(\prod_{t=1}^a x_{i_t,j_1}^2x_{i_t,j_2}^2 x_{\tilde{i}_t,j_3}^2 x_{\tilde{i}_t,j_4}^2\Big) - 	\{\mathrm{E}(x_{1,j_1}^2 x_{1,j_2}^2)\}^a \{\mathrm{E}(x_{1,j_3}^2 x_{1,j_4}^2)\}^a.   \label{eq:varestlmv1}
\end{align} 
When $\{\mathbf{i}\}\cap\{ \tilde{\mathbf{i}}\} =\emptyset$, $\eqref{eq:varestlmv1}=0$. We then know that $ \eqref{eq:varestlmv1} \neq 0$ only when $|\{\mathbf{i}\}\cup \{ \tilde{\mathbf{i}}\}| \leq 2a-1$. Along with   Condition \ref{cond:finitemomt}, we have
 \begin{align*}
 	|\mathrm{var}\{\mathbb{V}_{u,1}(a)\}|\leq Cp^4n^{-4a} n^{2a-1}, 
 \end{align*} which induces  $\mathrm{var}\{\mathbb{V}_{u,1}(a)\} = O(p^4 n^{-2a-1}) $.  
By \eqref{eq:sum4jordervar} and \eqref{eq:varordp4top2}, we know $\mathrm{E}\{\mathbb{V}_{u,1}(a)\} =\Theta(p^2n^{-a})$. It follows that $ \mathrm{var}\{\mathbb{V}_{c,1}(a)\}/\mathrm{E}^2\{\mathbb{V}_{c,1}(a)\} \to 0$ as $n\to \infty$.


We next prove ${ \mathbb{V}_{u,2}(a) }/{ \mathrm{E}\{\mathbb{V}_{u,1}(a)\} } \xrightarrow{P} 0.$
By the Markov's inequality, it  suffices to prove $\mathrm{E}\{\mathbb{V}^2_{u,2}(a)\}=o(1)[\mathrm{E}\{\mathbb{V}_{u,1}(a)\}]^2$. As 
$\mathrm{E}\{\mathbb{V}_{u,1}(a)\} =\Theta(p^2n^{-a})$, it is sufficient to prove  $\mathrm{E}\{\mathbb{V}^2_{u,2}(a)\}=o(p^4n^{-2a})$ below. 

We first derive the form of  $\mathbb{V}_{u,2}(a)$. In particular, when $a=1$, 
\begin{eqnarray*}
	\mathbb{V}_{u,2}(1)&=&\mathbb{V}_{u}(1)-\mathbb{V}_{u,1}(1) \notag \\
	&=& \frac{1}{n^2}\sum_{1\leq j_1\neq j_2\leq p}\sum_{i \in \mathcal{P}(n,1)} \Big\{(x_{i,j_1}-\bar{x}_{j_1})^2(x_{i,j_2}-\bar{x}_{j_2})^2 - x_{i,j_1}^2x_{i,j_2}^2\Big\} \notag \\
	&=& \frac{1}{n^2}\sum_{1\leq j_1\neq j_2\leq p}\sum_{1\leq i\leq n} \sum_{\substack{s_1+r_1=1,\\s_2+r_2=1}}C_{s_1,r_1,s_2,r_2}\prod_{k=1}^2 \Big\{(-x_{i,j_k}\bar{x}_{j_k})^{s_k}(\bar{x}_{j_k}^2)^{r_k}\Big\},
	\end{eqnarray*} where $C_{s_1,r_1,s_2,r_2}$ is some constant and we use 
\begin{eqnarray*}
	&& (x_{i,j_1}-\bar{x}_{i,j_1})^2(x_{i,j_2}-\bar{x}_{i,j_2})^2 - x_{i,j_1}^2x_{i,j_2}^2  \notag \\
	&=& (x_{i,j_1}^2-2x_{i,j_1}\bar{x}_{j_1}+\bar{x}_{j_1}^2)( x_{i,j_2}^2-2x_{i,j_2}\bar{x}_{j_2}+\bar{x}_{j_2}^2)- x_{i,j_1}^2x_{i,j_2}^2\notag \\
	&=& \sum_{\substack{s_1+r_1=1,\,s_2+r_2=1}}\Big\{(-2x_{i,j_1}\bar{x}_{j_1})^{s_1}(\bar{x}_{j_1}^2)^{r_1}\Big\}\times \Big\{(-2x_{i,j_2}\bar{x}_{j_2})^{s_2}(\bar{x}_{j_2}^2)^{r_2}\Big\}.  \notag 
\end{eqnarray*}
Following this example, we similarly give the form of $\mathbb{V}_{u,2}(a)$ for general $a\geq 1$. 
Given tuple $\mathbf{i}\in \mathcal{P}(n,a)$, for $k=1,2$, let $\mathbf{i}_{(a-r_k)}^{(k)}$ represent a sub-tuple of $\mathbf{i}$ with length $a-r_k$, and define  $\mathcal{S}(\mathbf{i},a-r_k)$ to be the collection of sub-tuples of $\mathbf{i}$ with length $a-r_k$. Then for $a\geq 1$, we write $
	\mathbb{V}_{u,2}(a)=\sum_{1\leq s_1+{r}_1\leq a, 1\leq s_2+r_2\leq a} T_{s_1,r_1,s_2,r_2},
$ where 
\begin{eqnarray*}
	T_{s_1, r_1,s_2,r_2} &=& \frac{a!}{(P^n_a)^2} \sum_{\substack{1\leq j_1\neq j_2\leq p} }  \ \sum_{\substack{\mathbf{i}\in \mathcal{P}(n,a);\\ \mathbf{i}_{(a-r_k)}^{(k)}\in \mathcal{S}(\mathbf{i},a-r_k): \, k=1,2} }C_{s_1,r_1,s_2,r_2}  \notag \\
	&&\times \prod_{k=1}^2\Big\{ (-\bar{x}_{j_k})^{s_k+2r_k} \prod_{t_k=1}^{{s}_k}x_{i^{(k)}_{t_k},j_k} \prod_{t_k=s_k+1}^{a-r_k}(x_{i^{(k)}_{t_k},j_k} )^2  \Big\} \notag.
\end{eqnarray*} 


When $a$ is finite, it  suffices to prove $\mathrm{E}(T_{s_1, r_1,s_2,r_2}^2)=o(p^4n^{-2a})$. Note that
\begin{eqnarray*}
&&\mathrm{E}(T_{s_1, r_1,s_2,r_2}^2)\notag \\	
&=&\frac{(a!)^2}{(P^n_a)^4} \sum_{\substack{1\leq j_1\neq j_2\leq p   \\  1\leq \tilde{j}_1\neq \tilde{j}_2\leq p } }\ \sum_{ \substack{\mathbf{i}, \tilde{\mathbf{i}}\in \mathcal{P}(n,a);\\\mathbf{i}_{(a-r_k)}^{(k)}\in \mathcal{S}(\mathbf{i},a-r_k): \, k=1,2;\\  \tilde{\mathbf{i}}_{(a-r_k)}^{(k)}\in \mathcal{S}(\tilde{\mathbf{i}},a-r_k): \, k=1,2 }} C_{s_1,r_1,s_2,r_2}^2   \notag \\ 
&&\times \mathrm{E}\Big\{ \prod_{k=1}^2 (\bar{x}_{j_k}\bar{x}_{\tilde{j}_k})^{s_k+2r_k}  \prod_{t_k=1}^{{s}_k} (x_{i^{(k)}_{t_k},{j}_k}x_{\tilde{i}^{(k)}_{{t}_k},\tilde{j}_k}) \prod_{t_k=s_k+1}^{a-r_k}(x_{i^{(k)}_{t_k},{j}_k} x_{\tilde{i}^{(k)}_{{t}_k},\tilde{j}_k} )^2\Big\}. \notag 
\end{eqnarray*}Recall that $\bar{x}_j=\sum_{i=1}^nx_{i,j}/n$. We have 
\begin{eqnarray*}
	&&\mathrm{E}(T_{s_1, r_1,s_2,r_2}^2)\notag \\	
&=& \frac{(a!)^2}{(P^n_a)^4n^{\sum_{k=1}^2 (2s_k+4r_k)}} \sum_{\substack{1\leq j_1\neq j_2\leq p   \\  1\leq \tilde{j}_1\neq \tilde{j}_2\leq p } }\ \sum_{ \substack{\mathbf{i}, \tilde{\mathbf{i}}\in \mathcal{P}(n,a);\\\mathbf{i}_{(a-r_k)}^{(k)}\in \mathcal{S}(\mathbf{i},a-r_k): \, k=1,2;\\  \tilde{\mathbf{i}}_{(a-r_k)}^{(k)}\in \mathcal{S}(\tilde{\mathbf{i}},a-r_k): \, k=1,2 }} C_{s_1,r_1,s_2,r_2}^2 \notag \\
&& \times \sum_{\substack{\mathbf{m}^{(k)},\tilde{\mathbf{m}}^{(k)} \in \mathcal{C}(n,s_k+2r_k);\, k=1,2}} T\{ \mathbf{i}_{(a-r_k)}^{(k)}, \tilde{\mathbf{i}}_{(a-r_k)}^{(k)},\mathbf{m}^{(k)},\tilde{\mathbf{m}}^{(k)}; k=1,2\}, 
\end{eqnarray*} where $\mathcal{C}(n,s_k+2r_k)$ follows the notation at the beginning of Section \ref{par:notationindpcond} and 
\begin{eqnarray*}
	&&T\{ \mathbf{i}_{(a-r_k)}^{(k)}, \tilde{\mathbf{i}}_{(a-r_k)}^{(k)},\mathbf{m}^{(k)},\tilde{\mathbf{m}}^{(k)};k=1,2 \}\notag \\
	&=&\mathrm{E}\Big\{ \prod_{k=1}^2 \prod_{\tilde{t}_k=1}^{s_k+2r_k} ({x}_{m_{\tilde{t}_k}, j_k}{x}_{\tilde{m}_{\tilde{t}_k}, \tilde{j}_k})  \prod_{t_k=1}^{{s}_k} (x_{i^{(k)}_{t_k},j_k}x_{ \tilde{i}^{(k)}_{{t}_k},\tilde{j}_k}) \prod_{t_k=s_k+1}^{a-r_k}(x_{i^{(k)}_{t_k},j_k} x_{\tilde{i}^{(k)}_{{t}_k},\tilde{j}_k} )^2\Big\}.
\end{eqnarray*}
Since $\mathrm{E}(x_{i,j})=0$, $T\{ \mathbf{i}_{(a-r_k)}^{(k)}, \tilde{\mathbf{i}}_{(a-r_k)}^{(k)},\mathbf{m}^{(k)},\tilde{\mathbf{m}}^{(k)};k=1,2 \}\neq 0$ only when 
\begin{align*}
&~\Biggr|\bigcup_{k=1}^2\{\mathbf{m}^{(k)}\}\cup \tilde{\{\mathbf{m}}^{(k)}\}\cup \{\mathbf{i}^{(k)}_{(a-r_k)}\}\cup \tilde{\{\mathbf{i}}^{(k)}_{(a-r_k)}\} \Biggr|-\Biggr|\bigcup_{k=1}^2 \{\mathbf{i}^{(k)}_{(a-r_k)}\}\cup \tilde{\{\mathbf{i}}^{(k)}_{(a-r_k)}\}\Biggr|\notag \\
\leq &~\sum_{k=1}^2(s_k+2r_k). 
\end{align*}
Since $\mathbf{i}^{(k)}_{(a-r_k)}$ and $\tilde{\mathbf{i}}^{(k)}_{(a-r_k)}$ are sub-tuples of $\mathbf{i}$ and $\tilde{\mathbf{i}} \in \mathcal{P}(n,a)$,  $|\cup_{k=1}^2 \{\mathbf{i}^{(k)}_{(a-r_k)}\}\cup \tilde{\{\mathbf{i}}^{(k)}_{(a-r_k)}\}|\leq |\{\mathbf{i}\}\cup \{\tilde{\mathbf{i}}\}|\leq 2a$. Therefore,
\begin{eqnarray}
	&&\quad  \Biggr|\bigcup_{k=1}^2\{\mathbf{m}^{(k)}\}\cup \tilde{\{\mathbf{m}}^{(k)}\}\cup \{\mathbf{i}^{(k)}_{(a-r_k)}\}\cup \tilde{\{\mathbf{i}}^{(k)}_{(a-r_k)}\} \Biggr|\leq 2a+\sum_{k=1}^2(s_k+2r_k). \label{eq:thm24bddnsumnum}
\end{eqnarray}
By \eqref{eq:thm24bddnsumnum} and the boundedness of moments in Condition \ref{cond:higherordermomentvarest}, we have
\begin{eqnarray*}
\mathrm{E}(T_{s_1, r_1,s_2,r_2}^2)&=&  O\Big( p^4n^{-4a-\sum_{k=1}^2(2s_k+4r_k)+2a+\sum_{k=1}^2(s_k+2r_k) }  \Big)	\notag \\
&=&O(p^4n^{-2a-\sum_{k=1}^2(s_k+2r_k)})=o(p^4n^{-2a}),
\end{eqnarray*} where we use $\sum_{k=1}^2(s_k+2r_k)\geq 1.$





\subsection{Lemmas for the proof of Theorem \ref{thm:cltalternative}}
\subsubsection{Proof of Lemma \ref{lm:pfaltvarlm1} (on Page \pageref{lm:pfaltvarlm1},  Section \ref{sec:updatepoweronesampf})} \label{sec:pfaltvarlm1} 
To show $\mathrm{var}\{\mathcal{U}(a)\}\simeq \mathrm{var}(T_{U,a,1,1})$, it suffices to prove $\mathrm{var}(T_{U,a,1,1})=\Theta(p^2n^{-a})$, $\mathrm{var}(T_{U,a,1,2})=o(p^2n^{-a})$ and  $\mathrm{var}(T_{U,a,2})=o(p^2n^{-a}).$  The following three sections \ref{sec:pfaltvarlmvartuaa1}--   \ref{sec:pfaltvarlm4} prove the three results respectively.

\paragraph{$\mathrm{var}(T_{U,a,1,1})=\Theta(p^2n^{-a})$} \label{sec:pfaltvarlmvartuaa1}
As $\mathrm{E}(T_{U,a,1,1})=0$, $\mathrm{var}(T_{U,a,1,1})=\mathrm{E}(T_{U,a,1,1}^2)$, and we have
\begin{eqnarray*}
\mathrm{var}(T_{U,a,1,1})=	\sum_{(j_1,j_2),(j_3,j_4)\in J_A^c}(P^n_a)^{-2}\sum_{\mathbf{i},\tilde{\mathbf{i}}\in \mathcal{P}(n,a)}\mathrm{E}\Big( \prod_{k=1}^ax_{i_k,j_1}x_{i_k,j_2}x_{\tilde{i}_k,j_3}x_{\tilde{i}_k,j_4} \Big). 
\end{eqnarray*}
Similarly to Section \ref{sec:proofvarianceorder}, $\mathrm{E}( \prod_{k=1}^ax_{i_k,j_1}x_{i_k,j_2}x_{\tilde{i}_k,j_3}x_{\tilde{i}_k,j_4} )\neq 0$ only when $\{\mathbf{i}\}=\{\tilde{\mathbf{i}}\}$. Therefore,
\begin{eqnarray*}
	\mathrm{var}(T_{U,a,1,1})=	\sum_{(j_1,j_2),(j_3,j_4)\in J_A^c}(P^n_a)^{-1}a! \times \Big\{ \mathrm{E}\Big(\prod_{t=1}^4x_{1,j_t} \Big) \Big\}^a.
\end{eqnarray*} By Condition \ref{cond:altonecovellip}, as $(j_1,j_2),(j_3,j_4)\in J_A^c$,
\begin{align}
	\mathrm{E}(x_{1,j_1}x_{1,j_2}x_{1,j_3}x_{1,j_4})=\kappa_1( \sigma_{j_1,j_3}\sigma_{j_2,j_4}+\sigma_{j_1,j_4}\sigma_{j_2,j_3}).  \label{eq:4thmomentfirstorderalt}
\end{align}
We next evaluate  \eqref{eq:4thmomentfirstorderalt} by discussing three cases on $(j_1,j_2,j_3,j_4)$. 
First, if $|\{j_1,j_2\}\cap\{j_3,j_4\}|=2$, $\eqref{eq:4thmomentfirstorderalt}=\kappa_1 \sigma_{j_1,j_1}\sigma_{j_2,j_2}=\Theta(1)$ by Condition \ref{cond:finitemomt}.  
\begin{eqnarray*}
	\sum_{ \substack{(j_1,j_2),\\(j_3,j_4)\in J_A^c}} \Big\{ \mathrm{E}\Big(\prod_{t=1}^4x_{1,j_t} \Big) \Big\}^a \times \mathbf{1}_{ \{ |\{j_1,j_2\}\cap\{j_3,j_4\}|=2 \} } = 2 \sum_{(j_1,j_2)\in J_A^c} (\kappa_1\sigma_{j_1,j_1}\sigma_{j_2,j_2})^a.  
\end{eqnarray*}
Second, if $|\{j_1,j_2\}\cap\{j_3,j_4\}|=1$, we assume without loss of generality $j_1=j_3$ and $j_2\neq j_4$, $\eqref{eq:4thmomentfirstorderalt}=\kappa_1\sigma_{j_1,j_1}\sigma_{j_2,j_4},$ which is nonzero only when $(j_2,j_4)\in J_A,$ and  then $\eqref{eq:4thmomentfirstorderalt}=O(\rho^a)$. By the symmetricity of the indexes, 
\begin{eqnarray*}
	&& \sum_{(j_1,j_2),(j_3,j_4)\in J_A^c} \Big\{ \mathrm{E}\Big(\prod_{t=1}^4x_{1,j_t} \Big) \Big\}^a \times \mathbf{1}_{ \{ |\{j_1,j_2\}\cap\{j_3,j_4\}|=1 \} } \notag \\
	&\leq &C \sum_{\substack{1\leq j\leq p;\, (j_2,j_4)\in J_A}} \rho^a= O(1) p|J_A|\rho^a.   
\end{eqnarray*}
Third, if $|\{j_1,j_2\}\cap\{j_3,j_4\}|=0$, we know $j_1\neq j_2 \neq j_3 \neq j_4$,  and  $\eqref{eq:4thmomentfirstorderalt}\neq 0$ only if $(j_1,j_3), (j_2,j_4)\in J_A$ or $(j_1,j_4), (j_2,j_3)\in J_A.$ Then $\eqref{eq:4thmomentfirstorderalt}=O(\rho^{2a})$. By the symmetricity of the indexes, 
\begin{eqnarray*}
		&& \sum_{(j_1,j_2),(j_3,j_4)\in J_A^c} \Big\{ \mathrm{E}\Big(\prod_{t=1}^4x_{1,j_t} \Big) \Big\}^a \times \mathbf{1}_{ \{ |\{j_1,j_2\}\cap\{j_3,j_4\}|=0 \} } \notag \\
	&\leq & C\sum_{(j_1,j_3), (j_2,j_4)\in J_A} \rho^{2a} = O(1)|J_A|^2 \rho^{2a}.
\end{eqnarray*}

In summary, we know
\begin{eqnarray*}
\mathrm{var}(T_{U,a,1,1}) &=&2a!\kappa_1^a(P^n_a)^{-1} \sum_{(j_1,j_2)\in J_A^c} \sigma_{j_1,j_1}^a\sigma_{j_2,j_2}^a \notag \\
&&+\, O(1)p|J_A|\rho^a n^{-a} + O(1) |J_A|^2\rho^{2a} n^{-a}.
\end{eqnarray*} Since we assume $|J_A|\rho^a=O(pn^{-a/2})$, $|J_A|=o(p^2)$ and $|J_A^c|=\Theta(p^2),$  
\begin{align*}
	\mathrm{var}(T_{U,a,1,1})\simeq  2a!\kappa_1^a(P^n_a)^{-1}\sum_{1\leq j_1\neq j_2\leq p}  (\sigma_{j_1,j_1}\sigma_{j_2,j_2})^a,
\end{align*} which is of order $\Theta(p^2n^{-a}).$



\paragraph{$\mathrm{var}(T_{U,a,1,2})=o(p^2n^{-a})$} \label{sec:pfaltvarlm3}

 In this section, we prove $\mathrm{var}(T_{U,a,1,2})=o(p^2n^{-a})$. 
As $T_{U,a,1,2}=\sum_{(j_1,j_2)\in J_A^c}\sum_{c=1}^aK(c,j_1,j_2)$, by the Cauchy-Schwarz inequality, 
\begin{eqnarray*}
	\mathrm{var}(T_{U,a,1,2})\leq C\times \sum_{c=1}^a \mathrm{var}\Big\{\sum_{(j_1,j_2)\in J_A^c}K(c,j_1,j_2) \Big\},
\end{eqnarray*}
where $C$ is some constant. As $a$ is finite, to prove $\mathrm{var}(T_{U,a,1,2})=o(p^2n^{-a})$, it suffices to prove $\mathrm{var}\{\sum_{(j_1,j_2)\in J_A^c}K(c,j_1,j_2)\}=o(p^2n^{-a})$, for each $1\leq c\leq a$. Note that $\mathrm{E}\{K(c,j_1,j_2)\}=0$ and then 
\begin{align*}
\mathrm{var}\Big\{\sum_{(j_1,j_2)\in J_A^c}K(c,j_1,j_2)\Biggr\} 
=& ~\mathrm{E}\Big[\Big\{\sum_{(j_1,j_2)\in J_A^c}K(c,j_1,j_2)\Big\}^2\Big] \notag \\
=& ~ F^2(a,c)\sum_{\substack{  \mathbf{i},\tilde{\mathbf{i}}\in \mathcal{P}(n,a+c); \\ (j_1,j_2),(j_3,j_4)\in J_A^c }}Q_c(\mathbf{i}, j_1,j_2, \tilde{\mathbf{i}},j_3,j_4),
\end{align*} where we define 
\begin{align*}
Q_c(\mathbf{i}, j_1,j_2, \tilde{\mathbf{i}},j_3,j_4)=&~ \mathrm{E} \Big[\prod_{t=1}^{a-c} x_{i_t,j_1} x_{i_t,j_2}\prod_{t=a-c+1}^{a} x_{i_{t},j_1}  \prod_{t=a+1}^{a+c}x_{i_{t},j_2}   \notag \\
&~\quad \ \times  \prod_{\tilde{t}=1}^{a-c}  x_{\tilde{i}_{\tilde{t}},j_3} x_{\tilde{i}_{\tilde{t}},j_4} \prod_{\tilde{t}=a-c+1}^a x_{\tilde{i}_{\tilde{t}},j_3} \prod_{\tilde{t}=a+1}^{a+c_2}x_{\tilde{i}_{\tilde{t}},j_4}  \Big].\notag 
\end{align*}
As  $F^2(a,c)=O(n^{-2(a+c)})$, to finish the proof, it remains to prove 
\begin{align}
	\sum_{\substack{  \mathbf{i},\tilde{\mathbf{i}}\in \mathcal{P}(n,a+c);\\  (j_1,j_2),(j_3,j_4)\in J_A^c }}Q_c(\mathbf{i}, j_1,j_2, \tilde{\mathbf{i}},j_3,j_4)=o(n^{2(a+c)-a}p^2).\label{eq:sumqgoal}
\end{align}

We note that $\mathrm{E}(x_{1,j})=0$ and $\mathrm{E}(x_{1,j_1}x_{1,j_2})=\mathrm{E}(x_{1,j_3}x_{1,j_4})=0$ for $(j_1,j_2),(j_3,j_4)\in J_A^c$.  Similarly to Section \ref{sec:pfaltvarlmvartuaa1},    $Q_c(\mathbf{i}, j_1,j_2, \tilde{\mathbf{i}},j_3,j_4)=0$ if  $\{\mathbf{i}\}\neq \{\mathbf{\tilde{i}}\}$, and 
\begin{eqnarray}
	&&\ \sum_{ \mathbf{i},\tilde{\mathbf{i}}\in \mathcal{P}(n,a+c)} \mathbf{1}_{\{Q_c(\mathbf{i}, j_1,j_2, \tilde{\mathbf{i}},j_3,j_4)\neq 0\} }=\sum_{ \mathbf{i},\tilde{\mathbf{i}}\in \mathcal{P}(n,a+c)}  \mathbf{1}_{\{\{\mathbf{i}\}=\{\mathbf{\tilde{i}}\} \}}=O (n^{a+c}). \label{eq:nordervaralt}
\end{eqnarray}
To prove \eqref{eq:sumqgoal}, it remains to prove for given $\mathbf{i}, \tilde{\mathbf{i}}\in \mathcal{P}(n,a+c)$,
\begin{eqnarray}
	&&\Big|\sum_{(j_1,j_2),(j_3,j_4)\in J_A^c} Q_c(\mathbf{i}, j_1,j_2, \tilde{\mathbf{i}},j_3,j_4)\Big|=O(p^2).\label{eq:pordervaralt}
\end{eqnarray}

We next prove \eqref{eq:pordervaralt} by discussing the value of $Q_c(\mathbf{i}, j_1,j_2, \tilde{\mathbf{i}},j_3,j_4)$. 
To facilitate the discussion,  for given $\mathbf{i},\tilde{\mathbf{i}}\in \mathcal{P}(n,a+c)$, we decompose the sets $\{\mathbf{i}\}$ and $\{\tilde{\mathbf{i}}\}$ into three disjoint sets  respectively, defined as
\begin{align*}
\{\mathbf{i}\}_{(1)}=\{i_1,\ldots,i_{a-c}\}, \ 	\{\mathbf{i}\}_{(2)}=\{i_{a-c+1},\ldots, i_a\},\ \{\mathbf{i}\}_{(3)}=\{i_{a+1},\ldots, i_{a+c}\}, \notag \\
\tilde{\{\mathbf{i}\}}_{(1)}=\{\tilde{i}_1,\ldots,\tilde{i}_{a-c}\}, \ 	\tilde{\{\mathbf{i}\}}_{(2)}=\{\tilde{i}_{a-c+1},\ldots, \tilde{i}_a\},\ \tilde{\{\mathbf{i}\}}_{(3)}=\{\tilde{i}_{a+1},\ldots, \tilde{i}_{a+c}\},
\end{align*} which satisfy that $\{\mathbf{i}\}=\cup_{l=1}^3 \{\mathbf{i}\}_{(l)}$ and $\tilde{\{\mathbf{i}\}}=\cup_{l=1}^3 \{\tilde{\mathbf{i}}\}_{(l)}$. 


When $c\leq a-1$, $\{\mathbf{i}\}_{(1)}\neq \emptyset$. We consider  an index $i\in \{\mathbf{i}\}_{(1)}$,  and  discuss four different cases.
First, if $i \not \in \tilde{\{\mathbf{i}\}}$,
\begin{eqnarray*}
	&& Q_c(\mathbf{i}, j_1,j_2, \tilde{\mathbf{i}},j_3,j_4)=\mathrm{E}(x_{i,j_1}x_{i,j_2})\mathrm{E}(\text{other terms})=0,
\end{eqnarray*} 
where the last equation follows from $\mathrm{E}(x_{i,j_1}x_{i,j_2})=0$ when $(j_1,j_2)\in J_A$.  Second, if   $i\in \tilde{\{\mathbf{i}\}}_{(2)}$,  
\begin{eqnarray*}
	&& Q_c(\mathbf{i}, j_1,j_2, \tilde{\mathbf{i}},j_3,j_4)=\mathrm{E}(x_{i,j_1}x_{i,j_2}x_{i,j_3}) \mathrm{E}(\text{other terms})=0 
\end{eqnarray*}	where the last equation is obtained by  Condition \ref{cond:altonecovellip}. Third, if $i\in \tilde{\{\mathbf{i}\}}_{(3)}$, similarly by  Condition \ref{cond:altonecovellip}, we also know
\begin{eqnarray}\label{eq:altcov12threeintertwo}
	&& Q_c(\mathbf{i}, j_1,j_2, \tilde{\mathbf{i}},j_3,j_4)=\mathrm{E}(x_{i,j_1}x_{i,j_2}x_{i,j_4}) \mathrm{E}(\text{other terms})=0.
\end{eqnarray} Fourth, if $i\in \tilde{\{\mathbf{i}\}}_{(1)}$,
\begin{eqnarray}
	&&Q_c(\mathbf{i}, j_1,j_2, \tilde{\mathbf{i}},j_3,j_4)=\mathrm{E}(x_{i,j_1}x_{i,j_2}x_{i,j_3}x_{i,j_4}) \mathrm{E}(\text{other terms}).\label{eq:altcov12fourinter} 
\end{eqnarray} Under Condition \ref{cond:altonecovellip}, as $\mathrm{E}(x_{i,j_1}x_{i,j_2})=\mathrm{E}(x_{i,j_3}x_{i,j_4})=0$ when $(j_1,j_2)$ and $(j_3,j_4) \in J_A^c$,
\begin{align*}
 \mathrm{E}\Big(\prod_{t=1}^4x_{1,j_t}\Big)=\kappa_1\Big\{\mathrm{E}(x_{i,j_1} x_{i,j_3})\mathrm{E}(x_{i,j_2} x_{i,j_4}) +\mathrm{E}(x_{i,j_1} x_{i,j_4})\mathrm{E}(x_{i,j_2} x_{i,j_3})\Big\}. 
\end{align*}
In addition, when  $c=a$, $\{\mathbf{i}\}_{(1)}=\emptyset$ but $\{\mathbf{i}\}_{(1)}$ and $\{\mathbf{i}\}_{(3)}\neq \emptyset$.  We next consider an index $i\in \{\mathbf{i}\}_{(1)}$ without loss of generality. Following similar analysis, we know $Q_c(\mathbf{i}, j_1,j_2, \tilde{\mathbf{i}},j_3,j_4)=0$ when $i \not \in \{\tilde{\mathbf{i}}\}$. 

By symmetrically analyzing the indexes in $\mathbf{i}$ and $\tilde{\mathbf{i}}$ similarly as above, we know that $Q_c(\mathbf{i}, j_1,j_2, \tilde{\mathbf{i}},j_3,j_4)\neq 0$ only when $\{\mathbf{i}\}_{(1)}=\{\tilde{\mathbf{i}}\}_{(1)}$ and $\{\mathbf{i}\}_{(2)}\cup \{\mathbf{i}\}_{(3)}=\{\tilde{\mathbf{i}}\}_{(2)}\cup \{\tilde{\mathbf{i}}\}_{(3)}$. 
When $Q_c(\mathbf{i}, j_1,j_2, \tilde{\mathbf{i}},j_3,j_4)\neq 0$, suppose $r=|\{\mathbf{i}\}_{(2)}\cap \{\tilde{\mathbf{i}}\}_{(2)} |$ then $|\{\mathbf{i}\}_{(2)}\cap \{\tilde{\mathbf{i}}\}_{(3)}|=c-r$, $|\{\mathbf{i}\}_{(3)}\cap \{\tilde{\mathbf{i}}\}_{(2)}|=c-r$, and $|\{\mathbf{i}\}_{(3)}\cap \{\tilde{\mathbf{i}}\}_{(3)}|=r$. 
It follows that
\begin{eqnarray}
	&&  Q_c(\mathbf{i}, j_1,j_2, \tilde{\mathbf{i}},j_3,j_4)\label{eq:qvalueijdef} \\
	&=&\Big\{\mathrm{E}\Big(\prod_{t=1}^4x_{1,j_t}\Big) \Big\}^{a-c} \{\mathrm{E}(x_{1,j_1}x_{1,j_3})\mathrm{E}(x_{1,j_2}x_{1,j_4}) \}^{r} \notag \\
	&& \times \{\mathrm{E}(x_{1,j_1}x_{1,j_4})\mathrm{E}(x_{1,j_2}x_{1,j_3}) \}^{c-r} \notag \\
	&=& \Big\{\mathrm{E}\Big(\prod_{t=1}^4x_{1,j_t}\Big) \Big\}^{a-c} (\sigma_{j_1,j_3}\sigma_{j_2,j_4})^r(\sigma_{j_1,j_4}\sigma_{j_2,j_3})^{c-r}. \notag
\end{eqnarray}
To prove \eqref{eq:pordervaralt}, 
we next examine the value of  \eqref{eq:qvalueijdef} with respect to three different cases of $(j_1,j_2,j_3,j_4).$


\medskip

\noindent \textbf{Case (1)}\ If $|\{j_1,j_2\}|\cap |\{j_3,j_4\}|=2,$ it means that $\{j_1,j_2\}=\{j_3,j_4\}$.  
Assume, without loss of generality, that $j_1=j_3$ and $j_2=j_4$. Then
$
\eqref{eq:qvalueijdef}=O(1)(\sigma_{j_1,j_1}\sigma_{j_2,j_2})^{a-c+r}(\sigma_{j_1,j_2}^2)^{c-r},  	
$ which is nonzero only when $r=c$ as $\sigma_{j_1,j_2}=0$. 
By the symmetricity of $j$ indexes and the boundedness of moments in Condition \ref{cond:finitemomt}, 
\begin{eqnarray*}
	&& \Biggr|\sum_{\substack{   (j_1,j_2),(j_3,j_4)\in J_A^c }} Q_c(\mathbf{i}, j_1,j_2, \tilde{\mathbf{i}},j_3,j_4) \times \mathbf{1}_{\{ |\{j_1,j_2\}|\cap |\{j_3,j_4\}|=2\}}\Biggr| \leq Cp^2.
\end{eqnarray*}

\smallskip
\noindent \textbf{Case (2)} If $|\{j_1,j_2\}|\cap |\{j_3,j_4\}|=1,$ we assume without loss of generality that $j_1=j_3$ but  $j_2\neq j_4.$ Then
$
	\eqref{eq:qvalueijdef}=O(1) (\sigma_{j_1,j_1}\sigma_{j_2,j_4})^{a-c+r} (\sigma_{j_1,j_4}\sigma_{j_1,j_2})^{c-r},
$ which is also nonzero only when $r=c$. By the symmetricity of $j$ indexes and  Condition \ref{cond:finitemomt}, we have
\begin{eqnarray*}
	&& \Biggr|\sum_{\substack{   (j_1,j_2),(j_3,j_4)\in J_A^c }} Q_c(\mathbf{i}, j_1,j_2, \tilde{\mathbf{i}},j_3,j_4) \times \mathbf{1}_{\{ |\{j_1,j_2\}|\cap |\{j_3,j_4\}|=1\}} \Biggr|\notag \\
	&\leq & C\Big|\sum_{\substack{   (j_1,j_2),(j_3,j_4)\in J_A^c }}(\sigma_{j_1,j_1}\sigma_{j_2,j_4})^{a}  \Big|\leq \sum_{\substack{1\leq j\leq p,\, (j_2,j_4)\in J_A}} O(\rho^a)=O(p|J_A|\rho^a),
\end{eqnarray*} where we use Condition \ref{cond:rhoijconst} that $\sigma_{j_2,j_4}=\rho$ when $(j_2,j_4)\in J_A$ and $\sigma_{j_2,j_4}= 0$ when $(j_2,j_4)\not \in J_A$. 

\smallskip


\smallskip

\noindent \textbf{Case (3)} If $|\{j_1,j_2\}|\cap |\{j_3,j_4\}|=0,$ it means that $j_1\neq j_2\neq j_3 \neq j_4$. Then
\begin{align*}
\eqref{eq:qvalueijdef}=O(1)(\sigma_{j_1,j_3}\sigma_{j_2,j_4}+ \sigma_{j_1,j_4}\sigma_{j_2,j_3})^{a-c}	(\sigma_{j_1,j_3}\sigma_{j_2,j_4})^r(\sigma_{j_1,j_4}\sigma_{j_2,j_3})^{c-r}, 
\end{align*} which nonzero only when $(j_1,j_3),(j_2,j_4)\in J_A^c$ or $(j_1,j_4),(j_2,j_3)\in J_A^c$. By the symmetricity of $j$ indexes,  Condition \ref{cond:finitemomt} and Condition \ref{cond:rhoijconst},    
\begin{eqnarray*}
	&& \Biggr|\sum_{\substack{   (j_1,j_2),(j_3,j_4)\in J_A^c }} Q_c(\mathbf{i}, j_1,j_2, \tilde{\mathbf{i}},j_3,j_4) \times \mathbf{1}_{\{ |\{j_1,j_2\}|\cap |\{j_3,j_4\}|=0\}} \Biggr|\notag \\
	&\leq & C\sum_{(j_1,j_3),(j_2,j_4)\in J_A^c} \rho^{2a}=O(|J_A|^2\rho^{2a}). 
\end{eqnarray*}
In summary,
\begin{align*}
	\Big|\sum_{\substack{   (j_1,j_2),(j_3,j_4)\in J_A^c }} Q_c(\mathbf{i}, j_1,j_2, \tilde{\mathbf{i}},j_3,j_4) \Big|=O(p^2+p|J_A|\rho^a+|J_A|^2\rho^{2a})=o(p^2), 
\end{align*} as we assume $|J_A|\rho^a=O(pn^{-a/2}).$

\smallskip

\paragraph{$\mathrm{var}(T_{U,a,2})=o(p^2n^{-a})$}\label{sec:pfaltvarlm4}

Similarly to Section \ref{sec:pfaltvarlm3}, by the Cauchy-Schwarz inequality, 
\begin{eqnarray}
	\mathrm{var}(T_{U,a,2})\leq C\sum_{c=0}^a \mathrm{var}(T_{U,a,2,c}),  \label{eq:cauchyboundtua2}
\end{eqnarray} where $T_{U,a,2,c}=\sum_{(j_1,j_2)\in J_A}K(c,j_1,j_2).$ 
To prove $\mathrm{var}(T_{U,a,2})=o(p^2n^{-a})$, it suffices to prove 
 $\mathrm{var}(T_{U,a,2,c})=o(p^2n^{-a})$ for  $0\leq c\leq a$.  
Following the notation in Section \ref{sec:pfaltvarlm3},   we have
\begin{eqnarray*}
\mathrm{E}(T_{U,a,2,c}^2)=F^2(a,c)\sum_{\substack{  \mathbf{i},\tilde{\mathbf{i}}\in \mathcal{P}(n,a+c); \\ (j_1,j_2),(j_3,j_4)\in J_A }}Q_c(\mathbf{i}, j_1,j_2, \tilde{\mathbf{i}},j_3,j_4). 	
\end{eqnarray*} When $1\leq c\leq a$, $\mathrm{E}(T_{U,a,2,c})= 0$;  when $c=0$, $\mathrm{E}(T_{U,a,2,0})= \sum_{(j_1,j_2)\in J_A}\sigma_{j_1,j_2}^a$. Then 
\begin{eqnarray}
&& \mathrm{var}(T_{U,a,2,c})=	F^2(a,c)\sum_{\substack{  \mathbf{i},\tilde{\mathbf{i}}\in \mathcal{P}(n,a+c); \\ (j_1,j_2),(j_3,j_4)\in J_A }}\tilde{Q}_c(\mathbf{i}, j_1,j_2, \tilde{\mathbf{i}},j_3,j_4), \label{eq:vartua2ceach}
\end{eqnarray} where we define $\tilde{Q}_c(\mathbf{i}, j_1,j_2, \tilde{\mathbf{i}},j_3,j_4)=Q_c(\mathbf{i}, j_1,j_2, \tilde{\mathbf{i}},j_3,j_4)$ when $1\leq c\leq a$; and $\tilde{Q}_c(\mathbf{i}, j_1,j_2, \tilde{\mathbf{i}},j_3,j_4)=Q_c(\mathbf{i}, j_1,j_2, \tilde{\mathbf{i}},j_3,j_4)-(\sigma_{j_1,j_2}\sigma_{j_3,j_4})^{a}$ when $c=0$. 

To prove $\mathrm{var}(T_{U,a,2,c})=o(p^2n^{-a})$ for  $1\leq c\leq a$, we next examine the value of $Q_c(\mathbf{i}, j_1,j_2, \tilde{\mathbf{i}},j_3,j_4)$. For given $\mathbf{i},\tilde{\mathbf{i}}\in \mathcal{P}(n,a+c)$, we define  $\{\mathbf{i}\}_{(l)}$ and $\{\tilde{\mathbf{i}}\}_{(l)}$ for $l=1,2,3$ same as in Section \ref{sec:pfaltvarlm3}.   
Consider an index $i\in \{\mathbf{i}\}_{(2)}$. If $i\not \in \{\tilde{\mathbf{i}}\}$, $$Q_c(\mathbf{i}, j_1,j_2, \tilde{\mathbf{i}},j_3,j_4)=\mathrm{E}(x_{i,j_1})\mathrm{E}(\text{other terms})=0.$$ If $i\in \{\tilde{\mathbf{i}}\}_{(1)}$, by Condition \ref{cond:altonecovellip}, $$Q_c(\mathbf{i}, j_1,j_2, \tilde{\mathbf{i}},j_3,j_4)=\mathrm{E}(x_{i,j_1}x_{i,j_3}x_{i,j_4})\mathrm{E}(\text{other terms})=0.$$   Similarly, for an index $i\in \{\mathbf{i}\}_{(3)}$, we have $Q_c(\mathbf{i}, j_1,j_2, \tilde{\mathbf{i}},j_3,j_4)=0$ if $i\not \in  \{\tilde{\mathbf{i}}\}$ or $i\in \{\tilde{\mathbf{i}}\}_{(1)}$. Analyzing the indexes in $\{\tilde{\mathbf{i}}\}$ symmetrically, we know that $Q_c(\mathbf{i}, j_1,j_2, \tilde{\mathbf{i}},j_3,j_4)\neq 0$ only when $\{\mathbf{i}\}_{(2)}\cup\{\mathbf{i}\}_{(3)}=\{\tilde{\mathbf{i}}\}_{(2)}\cup \{\tilde{\mathbf{i}}\}_{(3)}$. Suppose $|\{\mathbf{i}\}_{(2)}\cap  \{\tilde{\mathbf{i}}\}_{(2)}|=r$, then $|\{\mathbf{i}\}_{(2)}\cap \{\tilde{\mathbf{i}}\}_{(3)}|=c-r$, $|\{\mathbf{i}\}_{(3)}\cap \{\tilde{\mathbf{i}}\}_{(2)}|=c-r$, and $|\{\mathbf{i}\}_{(3)}\cap \{\tilde{\mathbf{i}}\}_{(3)}|=r$. Moreover, we let $|\{\mathbf{i}\}_{(1)}\cap \{\tilde{\mathbf{i}}\}_{(1)}|=t_c$ then $0\leq t_c\leq a-c$. It follows that 
\begin{eqnarray}
&&Q_c(\mathbf{i}, j_1,j_2, \tilde{\mathbf{i}},j_3,j_4) \label{eq:qdefijpart2} \label{eq:vart12qdef} \\
&=&\Big\{\mathrm{E}\Big( \prod_{t=1}^4x_{i,j_t}\Big)\Big\}^{t_c}\Big\{ \mathrm{E}(x_{i,j_1}x_{i,j_2})\mathrm{E}(x_{i,j_3}x_{i,j_4})\Big\}^{a-c-t_c} \notag \\
&& \times \{\mathrm{E}(x_{i,j_1}x_{i,j_3}) \mathrm{E}(x_{i,j_2}x_{i,j_4})\}^{r} \{\mathrm{E}(x_{i,j_1}x_{i,j_4}) \mathrm{E}(x_{i,j_2}x_{i,j_3})\}^{c-r}. \notag  
\end{eqnarray}
To examine \eqref{eq:sumqgoal}, we next analyze \eqref{eq:qdefijpart2}  with respect to different $c$ and $t_c$ values, where $0\leq c\leq a$, $0\leq r\leq c$ and $0\leq t_c\leq a-c$. 

When $c=0$ and $t_c=t_0=0$, it means that $\{\mathbf{i}\}=\{\mathbf{i}\}_{(1)}$, $ \{\tilde{\mathbf{i}}\}=\{\tilde{\mathbf{i}}\}_{(1)}$, $\{\mathbf{i}\}\cap \{\tilde{\mathbf{i}}\}=\emptyset$ and $Q_0(\mathbf{i}, j_1,j_2, \tilde{\mathbf{i}},j_3,j_4)=(\sigma_{j_1,j_2}\sigma_{j_3,j_4})^a$. Then 
\begin{align}
	&~\sum_{\substack{  \mathbf{i},\tilde{\mathbf{i}}\in \mathcal{P}(n,a);\\  (j_1,j_2),(j_3,j_4)\in J_A }}\tilde{Q}_0(\mathbf{i}, j_1,j_2, \tilde{\mathbf{i}},j_3,j_4) \times \mathbf{1}_{\{t_0=0\}} \label{eq:exptaasq100}\\
	=&~\sum_{\substack{  \mathbf{i},\tilde{\mathbf{i}}\in \mathcal{P}(n,a);\\  (j_1,j_2),(j_3,j_4)\in J_A }}\Big\{{Q}_0(\mathbf{i}, j_1,j_2, \tilde{\mathbf{i}},j_3,j_4)-(\sigma_{j_1,j_2}\sigma_{j_3,j_4})^a\Big\}  \mathbf{1}_{\{t_0=0\}} =0.\notag 
\end{align}

In the following, it remains to consider the cases when $c\geq 1$ or $t_c \geq 1$ in \eqref{eq:qdefijpart2}, which are examined by discussing three cases $(j_1,j_2,j_3,j_4)$ below.

\smallskip
\noindent \textbf{Case (1)}\  If $|\{j_1,j_2\}\cap \{j_3,j_4\}|=2,$ we assume without loss of generality that $j_1=j_3$ and $j_2=j_4$. Then by Condition \ref{cond:altonecovellip}, $\mathrm{E}(x_{1,j_1}x_{1,j_2}x_{1,j_3}x_{1,j_4} )=\kappa_1 (2\sigma_{j_1,j_2}^2+ \sigma_{j_1,j_1}\sigma_{j_2,j_2}),$
and
\begin{align*}
\eqref{eq:qdefijpart2}= \{\kappa_1( 2\sigma_{j_1,j_2}^2 + \sigma_{j_1,j_1}\sigma_{j_2,j_2})\}^{t_c}\sigma_{j_1,j_2}^{2(a-c- t_c)}  (\sigma_{j_1,j_1}\sigma_{j_2,j_2})^r  (\sigma_{j_1,j_2})^{2(c-r)}.	
\end{align*}
\smallskip
\textbf{Case (1.1)}\  For $c=0$ and $1\leq t_c =t_0\leq a$, we have
$
	|\{\mathbf{i}\}\cup \{\tilde{\mathbf{i}}\}|\leq 2a-t_0,
$ and 
\begin{eqnarray} \label{eq:diff2varfirstpart}
&&\quad \Biggr|\sum_{\substack{  \mathbf{i},\tilde{\mathbf{i}}\in \mathcal{P}(n,a);\\  (j_1,j_2),(j_3,j_4)\in J_A }}\tilde{Q}_0(\mathbf{i}, j_1,j_2, \tilde{\mathbf{i}},j_3,j_4)  \mathbf{1}_{\{c=0,1\leq t_0\leq a,|\{j_1,j_2\}\cap \{j_3,j_4\}|=2\}}\Biggr| \\
	&\leq &C\sum_{t_0=1}^a n^{2a-t_0}\sum_{(j_1,j_2)\in J_A} |\sigma_{j_1,j_2}|^{2(a- t_0)}| 2\sigma_{j_1,j_2}^2 + \sigma_{j_1,j_1}\sigma_{j_2,j_2}|^{t_0} +|\sigma_{j_1,j_2}|^{2a}  \notag \\
	&= & \sum_{t_0=1}^a  O(1)n^{2a-t_0} |J_A|\times (\rho^{2a}+\rho^{2(a-t_0)}), \notag
\end{eqnarray}where we use Condition \ref{cond:rhoijconst}.  

\smallskip
\textbf{Case (1.2)}\  For $1\leq c \leq a$ and $0\leq t_c \leq a-c$, we have $
	|\{\mathbf{i}\}\cup \{\tilde{\mathbf{i}}\}|\leq 2a-t_c,
$ and for each $c$ given,
\begin{eqnarray}\label{eq:diff2varsecondpart}
	&&\quad \Biggr|\sum_{\substack{  \mathbf{i},\tilde{\mathbf{i}}\in \mathcal{P}(n,a);\\  (j_1,j_2),(j_3,j_4)\in J_A }}\tilde{Q}_c(\mathbf{i}, j_1,j_2, \tilde{\mathbf{i}},j_3,j_4)  \mathbf{1}_{\{1\leq t_c\leq a-c,|\{j_1,j_2\}\cap \{j_3,j_4\}|=2\}}\Biggr| \\
	&\leq & C\sum_{ \substack{  0\leq r\leq c;\\ 0 \leq t_c \leq a-c}} n^{2a-t_c} \sum_{(j_1,j_2)\in J_A}  |\sigma_{j_1,j_2}|^{2(a-c- t_c)}\notag \\
	&&\quad \ \times | 2\sigma_{j_1,j_2}^2 + \sigma_{j_1,j_1}\sigma_{j_2,j_2}|^{t_c}  |\sigma_{j_1,j_1}\sigma_{j_2,j_2}|^r |\sigma_{j_1,j_2}|^{2(c-r)}  \notag \\
	&=&\sum_{\substack{0\leq r\leq c;\\ 0 \leq t_c \leq a-c}}   O(1) n^{2a-t_c}|J_A|\{\rho^{2(a-r)}+ \rho^{2(a-t_c-r)} \}. \notag
\end{eqnarray}

\smallskip
\noindent \textbf{Case (2)}\ If $|\{j_1,j_2\}\cap \{j_3,j_4\}|=1,$ we assume without loss of generality that $j_1=j_3$ and $j_2\neq j_4$. Then by Condition \ref{cond:altonecovellip}, 
$
	\mathrm{E}(x_{1,j_1}x_{1,j_2}x_{1,j_3}x_{1,j_4} ) \notag 
	= \kappa_1(2\sigma_{j_1,j_2}\sigma_{j_1,j_4} + \sigma_{j_1,j_1}\sigma_{j_2,j_4} ).$
We then know
\begin{eqnarray*}
	\eqref{eq:vart12qdef}&=& \{\kappa_1(2\sigma_{j_1,j_2}\sigma_{j_1,j_4} + \sigma_{j_1,j_1}\sigma_{j_2,j_4} )\}^{t_c} (\sigma_{j_1,j_2}\sigma_{j_1,j_4})^{a-c-t_c} \notag \\
	&& \times (\sigma_{j_1,j_1}\sigma_{j_2,j_4})^r (\sigma_{j_1,j_4}\sigma_{j_1,j_2} )^{c-r}.
\end{eqnarray*} 

\textbf{Case (2.1)}\, For $c=0$ and  $1\leq t_c=t_0 \leq a$, we have $
	|\{\mathbf{i}\}\cup \{\tilde{\mathbf{i}}\}|\leq 2a-t_0,
$ and $\eqref{eq:vart12qdef}\neq 0$ at least when  $(j_1,j_2),(j_1,j_4)\in J_A$. Then
\begin{eqnarray}\label{eq:diff3firstpart}
	&&\quad \Biggr|\sum_{\substack{  \mathbf{i},\tilde{\mathbf{i}}\in \mathcal{P}(n,a);\\  (j_1,j_2),(j_3,j_4)\in J_A}}\tilde{Q}_0(\mathbf{i}, j_1,j_2, \tilde{\mathbf{i}},j_3,j_4)  \mathbf{1}_{\{c=0,1\leq t_0\leq a,|\{j_1,j_2\}\cap \{j_3,j_4\}|=1\}}\Biggr| \\
	&\leq & C\sum_{t_0=1}^a n^{2a-t_0}\sum_{(j_1,j_2), (j_1,j_4)\in J_A} \Big(|\sigma_{j_1,j_2}\sigma_{j_1,j_4}|^{a}+|\sigma_{j_2,j_4}|^{t_0}|\sigma_{j_1,j_2}\sigma_{j_1,j_4}|^{a-t_0} \Big) \notag \\
	&=&  \sum_{t_0=1}^a O(1)n^{2a-t_0}\max_{1\leq j_1\leq p}|J_{j_1}|\times |J_A|(\rho^{2a}+\rho^{2a-t_0}). \notag 
\end{eqnarray}

\textbf{Case (2.2)}\,  For $c \geq 1$ and  $0\leq t_c \leq a-c$, we have $
	|\{\mathbf{i}\}\cup \{\tilde{\mathbf{i}}\}|\leq 2a-t_c.
$ $\eqref{eq:vart12qdef}\neq 0$ when $(j_1,j_2), (j_1,j_4)\in J_A$ or $(j_2,j_4)\in J_A$. For given $c$, the range of $\eqref{eq:vart12qdef}$ is between $O(\rho^{2a-t_c-r})$ and $O(\rho^{2a-r})$. 
\begin{eqnarray}\label{eq:diff3secondpart}
	&&\quad \Biggr|\sum_{\substack{  \mathbf{i},\tilde{\mathbf{i}}\in \mathcal{P}(n,a);\\  (j_1,j_2),(j_3,j_4)\in J_A}}\tilde{Q}_c(\mathbf{i}, j_1,j_2, \tilde{\mathbf{i}},j_3,j_4)  \mathbf{1}_{\{0\leq t_c\leq a-c,|\{j_1,j_2\}\cap \{j_3,j_4\}|=1\}}\Biggr| \\
	&=&\sum_{\substack{0\leq r\leq c;\\ 0 \leq t_c \leq a-c}} O(1)n^{2a-t_c}\max_{1\leq j_1\leq p}|J_{j_1}|\times |J_A| (\rho^{2a-t_c-r}+\rho^{2a-r})  \notag. 
\end{eqnarray}


\smallskip

\noindent \textbf{Case (3)}\  If $|\{j_1,j_2\}\cap \{j_3,j_4\}|=0,$ we know $j_1\neq j_2\neq j_3\neq j_4.$ Then by Condition \ref{cond:altonecovellip} and \ref{cond:rhoijconst},
$
	\mathrm{E}(x_{1,j_1}x_{1,j_2}x_{1,j_3}x_{1,j_4} ) \notag 
	= \kappa_1(\sigma_{j_1,j_2}\sigma_{j_3,j_4} + \sigma_{j_1,j_3}\sigma_{j_2,j_4}+\sigma_{j_1,j_4}\sigma_{j_2,j_3} )=O(\rho^2).  	
$
Therefore, 
$\eqref{eq:vart12qdef}=O( \rho^{2a}).$ 
\smallskip

\textbf{Case (3.1)}\ For $c=0$ and $1\leq t_c=t_0 \leq a$, we have $
	|\{\mathbf{i}\}\cup \{\tilde{\mathbf{i}}\}|\leq 2a-t_0.
$
\begin{eqnarray}\label{eq:varorder4highorder} 
	&&\quad \Biggr|\sum_{\substack{  \mathbf{i},\tilde{\mathbf{i}}\in \mathcal{P}(n,a);\\  (j_1,j_2),(j_3,j_4)\in J_A}}\tilde{Q}_0(\mathbf{i}, j_1,j_2, \tilde{\mathbf{i}},j_3,j_4)  \mathbf{1}_{\{c=0,1\leq t_0\leq a,|\{j_1,j_2\}\cap \{j_3,j_4\}|=0\}}\Biggr| \\
	&\leq & C \sum_{t_0=1}^a  \sum_{(j_1,j_2),(j_3,j_4)\in J_A}|\sigma_{j_1,j_2}\sigma_{j_3,j_4} |^a = \sum_{t_0=1}^a n^{2a-t_0}|J_A|^2O(\rho^{2a}).  \notag 
\end{eqnarray}


\smallskip
\textbf{Case (3.2)}\ For $1\leq c \leq a$ and $0\leq t_c \leq a$, we have  $
	|\{\mathbf{i}\}\cup \{\tilde{\mathbf{i}}\}|\leq 2a-t_c.
$ Then for given $c\geq 1$,
\begin{eqnarray}\label{eq:smallordervaradd}
&&\quad \Biggr|\sum_{\substack{  \mathbf{i},\tilde{\mathbf{i}}\in \mathcal{P}(n,a);\\  (j_1,j_2),(j_3,j_4)\in J_A}}\tilde{Q}_c(\mathbf{i}, j_1,j_2, \tilde{\mathbf{i}},j_3,j_4)  \mathbf{1}_{\{0\leq t_c\leq a-c,|\{j_1,j_2\}\cap \{j_3,j_4\}|=0\}}\Biggr| \\
&\leq &C \sum_{\substack{0\leq r\leq c;\\ 0 \leq t_c \leq a-c}} n^{2a-t_c}\sum_{(j_1,j_2),(j_3,j_4)\in J_A}|\sigma_{j_1,j_2}\sigma_{j_3,j_4} |^a \notag \\
& =& \sum_{t_c=0}^{a-c} n^{2a-t_c}|J_A|^2O(\rho^{2a}), \notag	
\end{eqnarray} where we use the symmetricity of indexes. 



Combining  \eqref{eq:diff2varfirstpart}--\eqref{eq:smallordervaradd} 
 above, and by \eqref{eq:cauchyboundtua2} and \eqref{eq:vartua2ceach} and $F(a,c)=O(n^{-(a+c)})$, we know
\begin{eqnarray}
	&& \mathrm{var}(T_{1,a,2}) \label{eq:varsecondremain}\\
	&=&\sum_{t_0=1}^a  O(1)\frac{1}{n^{t_0}} |J_A|\times \{\rho^{2a}+\rho^{2(a-t_0)}\} \notag \\
	&& +\sum_{c=1}^a \sum_{t_c=0}^{a-c} \sum_{r=0}^c O(1)|J_A|\frac{1}{n^{2c+t_c}} \{\rho^{2(a-r)}+ \rho^{2(a-t_c-r)} \}  \notag \\
	&&+\sum_{t_0=1}^a O(1)\frac{1}{n^{t_0}}  \max_{1\leq j_1\leq p}|J_{j_1}|\times |J_A| (\rho^{2a}+\rho^{2a-t_0}) \notag \\
	&& +\sum_{c=1}^a  \sum_{t_c=0}^{a-c}\sum_{r=0}^c O(1) \frac{1}{n^{2c+t_c}}\max_{1\leq j_1\leq p}|J_{j_1}||J_A|(\rho^{2a-t_c-r}+\rho^{2a-r})  \notag \\
	&&+ \sum_{t_0=1}^a O(1)\frac{1}{n^{t_0}} |J_A|^2\rho^{2a}+ \sum_{c=1}^a \sum_{t_c=0}^{a-c} O(1) \frac{1}{n^{2c+t_c}}|J_A|^2\rho^{2a }. \notag
\end{eqnarray}

We then examine the six summed terms in the right hand side of \eqref{eq:varsecondremain} and show that they are $o(p^2n^{-a})$ respectively. 

\smallskip
\textbf{(1)} For the first term in  \eqref{eq:varsecondremain}, as $|J_A|\rho^a=O(pn^{-a/2}),$ 
\begin{align*}
	n^{-t_0}|J_A|\rho^{2a}=n^{-t_0}|J_A|^{-1}|J_A|^2\rho^{2a}=o( p^2n^{-a} ),
\end{align*} and
\begin{align*}
n^{-t_0}|J_A|\rho^{2(a-t_0)}=&n^{-t_0}|J_A|^{1-2(a-t_0)/a} (|J_A|\rho^a)^{ 2(a-t_0)/a }\notag \\
=& O(1)n^{-t_0} |J_A|^{-1+2t_0/a} ( pn^{-a/2})^{ 2(a-t_0)/a }\notag \\
=& O(1)p^2n^{-a}|J_A|^{-1+t_0/a} (|J_A|/p^2)^{t_0/a} = o(p^2n^{-a}),
\end{align*}where we use $1\leq t_0\leq a$ and $|J_A|=o(p^2)$ in the last equation.

\smallskip
\textbf{(2)} For the second term in  \eqref{eq:varsecondremain}, as $ r\leq c\leq a$ and $|J_A|=o(p^2)$,
\begin{align*}
n^{-(2c+t_c)}|J_A|\rho^{2(a-r)}=& n^{-(2c+t_c)} |J_A|^{1-2(a-r)/a} ( |J_A|\rho^a )^{2(a-r)/a}\notag \\
=& O(1) p^2n^{-a+r-2c-t_c} |J_A|^{-1+r/a} (|J_A|/p^2)^{r/a}\notag \\
 =& o(p^2n^{-a}),
\end{align*} 
and similarly as $ r\leq c\leq a$, $t_c+r\leq a$ and $c\geq 1,$
\begin{align*}
	& n^{-(2c+t_c)}|J_A|\rho^{2(a-t_c-r)}\notag \\
=& O(1) p^2 n^{-a+t_c+r-2c-t_c}|J_A|^{-1+(t_c+r)/a}(|J_A|/p^2)^{(t_c+r)/a}  \notag \\
=& o(p^2n^{-a}).\notag 
\end{align*} 

\vspace{5px}
\textbf{(3)} For the third term in  \eqref{eq:varsecondremain}, as $1\leq t_0\leq a,$ and $|J_A|\rho^a=O(pn^{-a/2})$,
\begin{align*}
	n^{-t_0}\max_{1\leq j_1\leq p}|J_{j_1}|\times |J_A| \rho^{2a}=\frac{\max_{1\leq j_1\leq p}|J_{j_1}|}{n^{t_0}|J_A|}|J_A|^2 \rho^{2a} = o(p^2n^{-a}),
\end{align*} and 
\begin{align*}
&~n^{-t_0}\max_{1\leq j_1\leq p}|J_{j_1}|\times |J_A| \rho^{2a-t_0} \notag \\
=&~	n^{-t_0}\max_{1\leq j_1\leq p}|J_{j_1}|\times |J_A|^{1-(2a-t_0)/a}(|J_A|\rho^a)^{(2a-t_0)/a}\notag \\
=&~O(1)p^2n^{-a-t_0/2} \max_{1\leq j_1\leq p}|J_{j_1}|\times |J_A|^{-1+t_0/(a)} p^{-t_0/a}\notag \\
=&~O(1)\frac{p^2}{n^{a+t_0/2}} \frac{\max_{1\leq j_1\leq p}|J_{j_1}|}{ |J_A|}\Big(\frac{|J_A|}{\max_{1\leq j_1\leq p}|J_{j_1}|}\frac{\max_{1\leq j_1\leq p}|J_{j_1}|}{p}\Big)^{t_0/a} \notag \\
=&~O(1)\frac{p^2}{n^{a+t_0/2}} \Big(\frac{\max_{1\leq j_1\leq p}|J_{j_1}|}{ |J_A|}\Big)^{1-t_0/a} \Big( \frac{\max_{1\leq j_1\leq p}|J_{j_1}|}{p}  \Big)^{t_0/a}=o(p^2n^{-a}),
\end{align*} where in the last equation, we use $1\leq t_0\leq a$, $\max_{1\leq j_1\leq p}|J_{j_1}|\leq |J_A|$ and $\max_{1\leq j_1\leq p}|J_{j_1}|\leq p.$

\smallskip

\textbf{(4)} For the fourth term in \eqref{eq:varsecondremain}, \begin{align*}
&~{n^{-(2c+t_c)}}\max_{1\leq j_1\leq p}|J_{j_1}||J_A|\rho^{2a-t_c-r}\notag \\
=&~{n^{-(2c+t_c)}}\max_{1\leq j_1\leq p}|J_{j_1}||J_A|^{1-(2a-t_c-r)/a} (|J_A|\rho^a)^{ (2a-t_c-r)/a} \notag \\
=&~ O(1)\frac{p^2}{n^a}\frac{1}{n^{2c+t_c/2-r/2}} \frac{\max_{1\leq j_1\leq p}|J_{j_1}|}{|J_A|}\Big( \frac{|J_A|}{p}  \Big)^{(t_c+r)/a} \notag \\
=&~ O(1)\frac{p^2}{n^a}\frac{1}{n^{2c+t_c/2-r/2}}\Big(\frac{\max_{1\leq j_1\leq p}|J_{j_1}|}{|J_A|}\Big)^{1-(t_c+r)/a}\Big( \frac{\max_{1\leq j_1\leq p}|J_{j_1}|}{p}  \Big)^{(t_c+r)/a}\notag \\
=&~ o(p^2n^{-a}),
\end{align*} where we obtain the last equation by noting that $t_c+r\leq a,$ $r\leq c$ and $c\geq 1$. 
Similarly, we have
\begin{align*}
	&{n^{-(2c+t_c)}}\max_{1\leq j_1\leq p}|J_{j_1}||J_A|\rho^{2a-r}\notag \\
=& O(1)\frac{p^2}{n^a}\frac{1}{n^{2c+t_c-r/2}}\Big(\frac{\max_{1\leq j_1\leq p}|J_{j_1}|}{|J_A|}  \Big)^{1-r/a}\Big( \frac{\max_{1\leq j_1\leq p}|J_{j_1}|}{p}  \Big)^{r/a}\notag \\
=& o(p^2n^{-a}).
\end{align*}

\smallskip
\textbf{(5)} For the fifth and sixth terms in \eqref{eq:varsecondremain}, as $|J_A|\rho^a=O(pn^{-a/2}),$  $t_0\geq 1$ and $c\geq 1$, we know
\begin{align*}
\frac{1}{n^{t_0}} |J_A|^2\rho^{2a}= o(p^2n^{-a}),\, \mbox{ and }\, 	\frac{1}{n^{2c+t_c}}|J_A|^2\rho^{2a }=o(p^2n^{-a}).
\end{align*}
\vspace{1px}

\subsubsection{Proof of Lemma \ref{lm:pfaltvarlm2} (on Page \pageref{lm:pfaltvarlm2}, Section \ref{sec:updatepoweronesampf})}\label{sec:pfaltvarlm2}

The proof is similar to Section \ref{sec:proofcovariancezro}. In particular, 
Lemma \ref{lm:pfaltvarlm2} shows that $\mathrm{var}\{\mathcal{U}(a)\}\simeq \mathrm{var}(T_{U,a,1,1})$.   By the Cauchy-schwarz inequality, 
\begin{eqnarray*}
	\mathrm{cov}\{ {\mathcal{U}(a)}/{\sigma(a)}, {\mathcal{U}(b)}/{\sigma(b)}\}= \mathrm{E}\{T_{U,a,1,1} T_{U,b,1,1}\}/\{\sigma(a)\sigma(b)\}+o(1), 
\end{eqnarray*} where we use $\mathrm{E}(T_{U,a,1,1})=\mathrm{E}(T_{U,b,1,1})=0$. 
For two integers $a\neq b,$ we next prove $\mathrm{E}(T_{U,a,1,1}T_{U,b,1,1})$=0. Specifically,  
\begin{eqnarray*}
	&&\mathrm{E}(T_{U,a,1,1}T_{U,b,1,1}) \notag \\
	&=&(P^n_a P^n_b)^{-1} \sum_{\substack{\mathbf{i}\in \mathcal{P}(n,a),  \tilde{\mathbf{i}} \in \mathcal{P}(n,b); \\ (j_1,j_2), (j_3,j_4) \in J_A^c}} \mathrm{E}\Big( \prod_{k=1}^a x_{i_k,j_1}x_{i_k,j_2} \prod_{\tilde{k}=1}^b x_{\tilde{i}_{\tilde{k}},j_3}x_{\tilde{i}_{\tilde{k}},j_4}\Big).
\end{eqnarray*} 
Since $a\neq b$, $\{\mathbf{i}\}\neq \{\tilde{\mathbf{i}}\}$. Assume without loss of generality that $a<b$ and index $i\in  \{\tilde{\mathbf{i}}\}$ but $i\not \in \{\mathbf{i}\}$. Then 
\begin{eqnarray*}
\mathrm{E}\Big( \prod_{k=1}^a x_{i_k,j_1}x_{i_k,j_2} \prod_{\tilde{k}=1}^b x_{\tilde{i}_{\tilde{k}},j_3}x_{\tilde{i}_{\tilde{k}},j_4}\Big)=\mathrm{E}(x_{1,j_3}x_{1,j_4})\times \mathrm{E}(\text{other terms})=0,	
\end{eqnarray*} where we use the $\sigma_{j_1,j_2}=\sigma_{j_3,j_4}=0$ for $(j_1,j_2), (j_3,j_4)\in J_A^c.$ Therefore $\mathrm{cov}(T_{U,a,1,1},T_{U,b,1,1})=0$ and the lemma is proved.

\subsubsection{Proof of Lemma \ref{lm:altcltmomentgoal1} (on Page \pageref{lm:altcltmomentgoal1}, Section \ref{sec:updatepoweronesampf})} \label{sec:pfaltcltmomentgoal1}

We prove Lemma \ref{lm:altcltmomentgoal1} similarly as in Section \ref{sec:prooftargetorder}.
 By the Cauchy-Schwarz inequality, for some constant $C$, $$\mathrm{var}\Big(\sum_{k=1}^n \pi_{n,k}^2\Big)\leq Cn^2 \max_{1\leq k\leq n;\, 1\leq r_1,r_2\leq m} \mathrm{var}(\mathbb{T}_{k,a_{r_1},a_{r_2}}), $$
where $c(n,a)=[a\times \{ \sigma(a)P^n_{a}\}^{-1}]^{2}$ and for two finite integers $a_1$ and $a_2$,  $\mathbb{T}_{k,a_{1},a_{2}}=\mathrm{E}_{k-1}( A_{n,k,a_{1}} A_{n,k,a_{2}}).$ 
In particular, 
 when $k< \max\{a_1,a_2\}$,  $\mathbb{T}_{k,a_{1},a_{2}}=0$; when $k\geq \max\{a_1,a_2\}$,
\begin{align*}
\mathbb{T}_{k,a_{1},a_{2}}=&~\mathrm{E}_{k-1}( A_{n,k,a_{1}} A_{n,k,a_{2}}) \notag \\
=&~ \sum_{ \substack{ \mathbf{i}^{(l)} \in \mathcal{P}(k-1,a_{l}-1),\, l=1,2; \\(j_1,j_2),(j_3,j_4)\in J_A^c } }  \Big\{\prod_{l=1}^2c(n,a_l)\Big\}^{1/2}  \mathbb{X}( k,\mathbf{i}^{(l)}, j_{2l-1},j_{2l}: l=1,2) \notag 
\end{align*} with
\begin{eqnarray*}
\mathbb{X}( k,\mathbf{i}^{(l)}, j_{2l-1},j_{2l}: l=1,2)&=&\mathrm{E}\Big( \prod_{t=1}^4 x_{k,j_t}  \Big)\prod_{l=1}^2 \prod_{t=1}^{a_l-1} (x_{i_t^{(l)},\, j_{2l-1}} x_{i_t^{(l)},\, j_{2l}} ). \notag 
\end{eqnarray*} 
To prove $\mathrm{var}(\sum_{k=1}^n \pi_{n,k}^2)\to 0$, it suffices to prove $\mathrm{var}(\mathbb{T}_{k,a_{r_1},a_{r_2}})=o(n^{-2})$ for any $1\leq r_1,r_2\leq m$. Without loss of generality, we consider two finite integers $a_1$ and $a_2$, and prove $\mathrm{var}(\mathbb{T}_{k,a_{1},a_{2}})=o(n^{-2})$ when $\max\{a_1,a_2\}\leq k\leq n$. 

To prove  $\mathrm{var}(\mathbb{T}_{k,a_{1},a_{2}})=o(n^{-2})$, we decompose $\mathbb{T}_{k,a_{1},a_{2}}=\sum_{M=2}^4\mathbb{T}_{k,a_{1},a_{2},(M)}$, where
\begin{align*}
\mathbb{T}_{k,a_{1},a_{2},(M)}=&~ \sum_{ \substack{ \mathbf{i}^{(l)} \in \mathcal{P}(k-1,a_{l}-1),\, l=1,2; \\(j_1,j_2),(j_3,j_4)\in J_A^c } } \mathbf{1}_{\{ |\{j_1,j_2\}\cup \{j_3,j_4\} |=M\}}    \notag 	 \\
&~\times \Big\{\prod_{l=1}^2c(n,a_l)\Big\}^{1/2}\mathbb{X}( k,\mathbf{i}^{(l)}, j_{2l-1},j_{2l}: l=1,2).
\end{align*} 
Here $2\leq M\leq 4$ because $2\leq |\{j_1,j_2\}\cup \{j_3,j_4\} |\leq 4$ when $(j_1,j_2),  (j_3,j_4)\in J_A^c$.  By the Cauchy-Schwarz inequality, to prove  $\mathrm{var}(\mathbb{T}_{k,a_{1},a_{2}})=o(n^{-2})$, it suffices to prove $ \mathrm{var}(\mathbb{T}_{k,a_{1},a_{2},(M)})=o(n^{-2})$ for $M=2,3,4$. For easy presentation, we let $a_3=a_1$ and $a_4=a_2$, and then
\begin{eqnarray*}\label{eq:sqmomentalt}
 \mathbb{T}_{k,a_{1},a_{2},(M)}^2  &=& \sum_{ \substack{ \mathbf{i}^{(l)} \in \mathcal{P}(k-1,a_{l}-1),\, l=1,2,3,4; \\(j_{1},j_2),(j_3,j_4),(j_{5},j_6),(j_7,j_8)\in J_A^c } } 	\mathbf{1}_{\Big\{ \substack{ |\{j_1,j_2\}\cup \{j_3,j_4\}|=M, \\ |\{j_5,j_6\}\cup \{j_7,j_8\}|=M} \Big\}} \notag \\
 &&\times \Big\{\prod_{l=1}^2c(n,a_l)\Big\}\times \mathbb{X}( k,\mathbf{i}^{(l)}, j_{2l-1},j_{2l}: l=1,2,3,4),
\end{eqnarray*} where
\begin{align*}
&~\mathbb{X}( k,\mathbf{i}^{(l)}, j_{2l-1},j_{2l}: l=1,2,3,4) \notag \\
=&~	\mathrm{E}\Big(\prod_{t=1}^4 x_{k,j_t}\Big) \mathrm{E}\Big(\prod_{t=5}^8 x_{k,j_t} \Big) \Big(\prod_{l=1}^4\prod_{t=1}^{a_l	-1} x_{i_t^{(l)}, \, j_{2l-1}}x_{i_t^{(l)},\, j_{2l}}\Big). 
\end{align*}
By $\mathrm{var}\{ \mathbb{T}_{k,a_{1},a_{2},(M)} \}=\mathrm{E}\{ \mathbb{T}_{k,a_{1},a_{2},(M)}^2 \} -\{\mathrm{E}( \mathbb{T}_{k,a_{1},a_{2},(M)})\}^2,$ 
\begin{eqnarray*}
&&\mathrm{var}\{ \mathbb{T}_{k,a_{1},a_{2},(M)} \} \notag \\
&=&	\sum_{ \substack{ \mathbf{i}^{(l)} \in \mathcal{P}(k-1,a_{l}-1),\, l=1,2,3,4;\\  (j_{1},j_2),(j_3,j_4),(j_{5},j_6),(j_7,j_8)\in J_A^c } }   	\mathbf{1}_{\Big\{ \substack{ |\{j_1,j_2\}\cup \{j_3,j_4\}|=M, \\ |\{j_5,j_6\}\cup \{j_7,j_8\}|=M} \Big\}} \Big\{\prod_{l=1}^2c(n,a_l)\Big\}\notag \\
&&\times \Big[ \mathrm{E}\Big\{\mathbb{X}( k,\mathbf{i}^{(l)}, j_{2l-1},j_{2l}: l=1,2,3,4)  \Big\} \notag \\
&& \quad -  \mathrm{E}\Big\{\mathbb{X}( k,\mathbf{i}^{(l)}, j_{2l-1},j_{2l}: l=1,2)\Big\}\times \mathrm{E}\Big\{\mathbb{X}( k,\mathbf{i}^{(l)}, j_{2l-1},j_{2l}: l=3,4)\Big\}  \Big],
\end{eqnarray*} where we similarly define
\begin{eqnarray*}
\mathbb{X}( k,\mathbf{i}^{(l)}, j_{2l-1},j_{2l}: l=3,4)=\mathrm{E}\Big( \prod_{t=5}^8 x_{k,j_t}  \Big)\prod_{l=3}^4 \prod_{t=1}^{a_{l}-1} (x_{i_t^{(l)},\, j_{2l-1}} x_{i_t^{(l)},\, j_{2l}} ). \notag 
\end{eqnarray*}

To prove $ \mathrm{var}(\mathbb{T}_{k,a_{1},a_{2},(M)})=o(n^{-2})$, we   examine the value of
\begin{eqnarray}\label{eq:expaltvarpi1} 
&&\quad \quad \mathrm{E}\Big\{\mathbb{X}( k,\mathbf{i}^{(l)}, j_{2l-1},j_{2l}: l=1,2,3,4)  \Big\}  \\
&&\quad \quad -  \mathrm{E}\Big\{\mathbb{X}( k,\mathbf{i}^{(l)}, j_{2l-1},j_{2l}: l=1,2)\Big\} \mathrm{E}\Big\{\mathbb{X}( k,\mathbf{i}^{(l)}, j_{2l-1},j_{2l}: l=3,4)\Big\}.	 \notag
\end{eqnarray} 
We next show that when $\eqref{eq:expaltvarpi1}\neq 0$, the following two claims hold:
\begin{align}
	& \textit{Claim 1: \ } (\{ \mathbf{i}^{(1)}\} \cup \{ \mathbf{i}^{(2)}\})\cap (\{ \mathbf{i}^{(3)}\} \cup \{ \mathbf{i}^{(4)}\}) \neq  \emptyset, \label{eq:twoclaims43sec} \\
	&\textit{Claim 2:\ } |\cup_{l=1}^4\{\mathbf{i}^{(l)}\}|\leq a_1+a_2-2. \notag
\end{align}
\textit{Claim 1} can be straightforwardly seen from the definition \eqref{eq:expaltvarpi1}. We then prove \textit{Claim 2}. 
Note that $\mathrm{E}\{\mathbb{X}( k,\mathbf{i}^{(l)}, j_{2l-1},j_{2l}: l=1,2,3,4) \}\neq 0$ only when $|\cup_{l=1}^4\{\mathbf{i}^{(l)}\}|\leq a_1+a_2-2$ following similar analysis to Section \ref{sec:pfvartaamigixing}.   In addition, as  $\sigma_{j_1,j_2}=\sigma_{j_3,j_4}=0$ 
when $(j_1,j_2), (j_3,j_4)\in J_A^c,$ 
we know that $\mathrm{E}\{ \mathbb{X}( k,\mathbf{i}^{(l)}, j_{2l-1},j_{2l}: l=1,2)\}\neq 0$ only when $\{\mathbf{i}^{(1)}\}=\{\mathbf{i}^{(2)}\}$; as $\sigma_{j_5,j_6}=\sigma_{j_7,j_8}=0$, we similarly know that $\mathrm{E}\{  \mathbb{X}( k,\mathbf{i}^{(l)}, j_{2l-1},j_{2l}: l=3,4)\}\neq 0$ only when $\{\mathbf{i}^{(3)}\}=\{\mathbf{i}^{(4)}\}$. It follows that if $|\cup_{l=1}^4\{\mathbf{i}^{(l)}\}|> a_1+a_2-2$, $\eqref{eq:expaltvarpi1} = 0.$ Thus   to evaluate $\mathrm{var}\{ \mathbb{T}_{k,a_{1},a_{2},(M)} \}$, it remains to consider \eqref{eq:expaltvarpi1} under the cases when $(\{ \mathbf{i}^{(1)}\} \cup \{ \mathbf{i}^{(2)}\})\cap (\{ \mathbf{i}^{(3)}\} \cup \{ \mathbf{i}^{(4)}\}) \neq \emptyset$ and  $|\cup_{l=1}^4\{\mathbf{i}^{(l)}\}|\leq a_1+a_2-2$.

Given the two claims above, we examine $\mathrm{var}\{ \mathbb{T}_{k,a_{1},a_{2},(M)} \}$ for  $M=2,3,4$ respectively. To facilitate the discussion, we decompose  
$\mathrm{var}\{ \mathbb{T}_{k,a_{1},a_{2},(M)} \} =\mathrm{var}\{\mathbb{T}_{k,a_{1},a_{2},(M)} \}_{(1)}+\mathrm{var}\{ \mathbb{T}_{k,a_{1},a_{2},(M)} \}_{(2)},$ where 
\begin{align*}
	&\mathrm{var}\{ \mathbb{T}_{k,a_{1},a_{2},(M)}\}  _{(1)} \notag \\
=& \sum_{ \substack{ \mathbf{i}^{(l)} \in \mathcal{P}(k-1,a_{l}-1),\, l=1,2,3,4;\\  (j_{1},j_2),(j_3,j_4),(j_{5},j_6),(j_7,j_8)\in J_A^c } } \mathbf{1}_{\left\{ \substack{ |\cup_{l=1}^4\{\mathbf{i}^{(l)}\}|= a_1+a_2-2;\\ |\{j_1,j_2\}\cup \{j_3,j_4\}|=M; \\ |\{j_5,j_6\}\cup \{j_7,j_8\}|=M } \right\}} 	\prod_{l=1}^2c(n,a_l)\times  \eqref{eq:expaltvarpi1},
\end{align*} and 
\begin{align*}
	&\mathrm{var}\{ \mathbb{T}_{k,a_{1},a_{2},(M)}\}  _{(2)} \notag \\
=& \sum_{ \substack{ \mathbf{i}^{(l)} \in \mathcal{P}(k-1,a_{l}-1),\, l=1,2,3,4;\\  (j_{1},j_2),(j_3,j_4),(j_{5},j_6),(j_7,j_8)\in J_A^c } } \mathbf{1}_{\left\{ \substack{|\cup_{l=1}^4\{\mathbf{i}^{(l)}\}|< a_1+a_2-2;\\ |\{j_1,j_2\}\cup \{j_3,j_4\}|=M; \\ |\{j_5,j_6\}\cup \{j_7,j_8\}|=M } \right\}} 	\prod_{l=1}^2c(n,a_l) \times \eqref{eq:expaltvarpi1}.
\end{align*}
We next consider  $M=2,3,4$ in the following \textbf{Cases (1)--(3)},  respectively. We assume  without loss of generality that $a_1\leq a_2$ in the following.  

\vspace{0.6em}

\textbf{Case (1):} When $M=2$, by the definition of $ \mathbb{T}_{k,a_{1},a_{2},(M)}$, we know $\{j_1,j_2\}=\{j_3,j_4\}$, $\{j_5,j_6\}=\{j_7,j_8\}$, and $|\{j_t: t=1,\ldots,8\}|\leq 4$. 
It follows that  $\mathrm{var}\{  \mathbb{T}_{k,a_{1},a_{2},(M)}\}_{(2)}=O\{\prod_{l=1}^2c(n,a_l)p^4n^{a_1+a_2-3}\}=o(n^{-2})$ by the boundedness of moments in Condition \ref{cond:finitemomt} and the definition of $\mathrm{var}\{ \mathbb{T}_{k,a_{1},a_{2},(M)}\}_{(2)}$. 

We next prove $\mathrm{var}\{ \mathbb{T}_{k,a_{1},a_{2},(M)}\}_{(1)}=o(n^{-2}).$ 
Recall that we consider $|\cup_{l=1}^4\{\mathbf{i}^{(l)}\}|= a_1+a_2-2$ here by the construction of $\mathrm{var}\{ \mathbb{T}_{k,a_{1},a_{2},(M)}\}_{(1)}$. 
Suppose $|\{\mathbf{i}^{(1)}\}\cap\{\mathbf{i}^{(2)}\}|= s$, where $s\leq a_1-1$. Then symmetrically $|\{\mathbf{i}^{(3)}\}\cap\{\mathbf{i}^{(4)}\}|= s$. Further assume $|\{\mathbf{i}^{(1)}\}\cap\{\mathbf{i}^{(3)}\}|= s_1$, then $|\{\mathbf{i}^{(2)}\}\cap\{\mathbf{i}^{(3)}\}|= a_1-1-s-s_1$, $|\{\mathbf{i}^{(1)}\}\cap\{\mathbf{i}^{(4)}\}|= a_1-1-s-s_1$ and $|\{\mathbf{i}^{(2)}\}\cap\{\mathbf{i}^{(4)}\}|= a_2-a_1+s_1$. It follows that $| (\{ \mathbf{i}^{(1)}\} \cup \{ \mathbf{i}^{(2)}\})\cap (\{ \mathbf{i}^{(3)}\} \cup \{ \mathbf{i}^{(4)}\}) |=a_1+a_2-2-2s$. Note that $\eqref{eq:expaltvarpi1}=0$ if $a_1+a_2-2-2s=0$, which can  only be achieved when $a_1=a_2$ and $s=a_1-1$.   It remains to consider $a_1+a_2-2-2s\geq 1$, that is, $0\leq s\leq A_0$, where $A_0=(a_1+a_2-3)/2.$ 
Given $s$ and $s_1$, we have
\begin{align}
\eqref{eq:expaltvarpi1} =~& \Big\{\mathrm{E}\Big(\prod_{t=1,2,5,6}x_{1,j_t}\Big)\Big\}^{s_1}\Big\{\mathrm{E}\Big(\prod_{t=3,4,7,8}x_{1,j_t}\Big)	\Big\}^{a_2-a_1+s_1} \label{eq:2am1orderexpalt} \\
~ &\times \Big\{\mathrm{E}\Big(\prod_{t=3,4,5,6}x_{1,j_t}\Big) \mathrm{E}\Big(\prod_{t=1,2,7,8}x_{1,j_t}\Big)\Big\}^{a_1-1-s-s_1}\notag \\
~&\times \Big\{\mathrm{E}\Big(\prod_{t=1,2,3,4}x_{1,j_t}\Big)\mathrm{E}\Big(\prod_{t=5,6,7,8}x_{1,j_t}\Big)	\Big\}^{s+1}. \notag 
\end{align}
Under the considered \textbf{Case (1)}, $\{j_1,j_2\}=\{j_3,j_4\}$ and $\{j_5,j_6\}=\{j_7,j_8\}$.
If $|\{j_t: t=1,\ldots,8\}|\leq 3,$  we  know by Condition \ref{cond:finitemomt}, 
\begin{align}
 \Biggr|\sum_{\substack{(j_1,j_2),(j_3,j_4),\\(j_5,j_6),(j_7,j_8)\in J_A^c}} \eqref{eq:expaltvarpi1} \times \mathbf{1}_{|\{j_t: t=1,\ldots,8\}|\leq 3}	\Biggr|=O(p^3). \label{eq:p3orderalt} 
\end{align}
If $|\{j_t: t=1,\ldots,8\}|=4,$ $\{j_1,j_2\}\cap \{j_5,j_6\}=\emptyset.$
By Conditions \ref{cond:finitemomt},  \ref{cond:altonecovellip} and \ref{cond:rhoijconst}, we know $\mathrm{E}(\prod_{t=1}^4x_{1,j_t})=\kappa_1(\sigma_{j_1,j_1}+\sigma_{j_2,j_2})=O(1)$ and similarly $\mathrm{E}(\prod_{t=5}^8 x_{1,j_t})=O(1).$ 
By \eqref{eq:2am1orderexpalt},  $\eqref{eq:expaltvarpi1}\neq 0$ only if $\mathrm{E}(\prod_{t=1,2,5,6}x_{1,j_t})\neq 0$.  This induces $(j_1,j_5)$, $(j_2,j_6) \in J_A$ or $(j_1,j_6)$, $(j_2,j_5) \in J_A$, and then $\eqref{eq:expaltvarpi1}= O(\rho^{2(a_1+a_2-2s)})$.   
By the symmetricity of $j$ indexes, we have
\begin{align}
	& \Big|\sum_{(j_1,j_2),(j_3,j_4),(j_5,j_6),(j_7,j_8)\in J_A^c} \eqref{eq:expaltvarpi1}\times \mathbf{1}_{\{|\{j_t: t=1,\ldots,8\}|=4\}}	\Big| \label{eq:p3orderalt2} \\
	\leq & ~ C\sum_{(j_1,j_5),(j_2,j_6)\in J_A} \rho^{2(a_1+a_2-2-2s)} \leq C|J_A|^2 \rho^{2(a_1+a_2-2-2s)}. \notag 
\end{align} 
By \eqref{eq:p3orderalt} and \eqref{eq:p3orderalt2},
\begin{align*}
\mathrm{var}\{  \mathbb{T}_{k,a_{1},a_{2},(M)}\}_{(1)}=\sum_{s=0}^{A_0} O\Big\{p^3+|J_A|^2 \rho^{2(a_1+a_2-2-2s)}\Big\}  n^{a_1+a_2-2}  \prod_{l=1}^2c(n,a_l). 
\end{align*}
Note that $O(p^3n^{a_1+a_2-2})\prod_{l=1}^2c(n,a_l)=o(n^{-2})$, and  
\begin{align}
&~|J_A|^2 \rho^{2(a_1+a_2-2-2s)} n^{a_1+a_2-2}c(n,a_1)c(n,a_2) \label{eq:case1orderaltclt}\\
=&~O(1)p^{-4}n^{-2} |J_A|^{2-\frac{2(a_1+a_2-2-2s)}{a_1+a_2}}  (|J_A|	\rho^{a_1}\times|J_A|\rho^{a_2})^{\frac{2(a_1+a_2-2-2s)}{a_1+a_2}} \notag \\
=&~ O(1)|J_A|^{2-\frac{2(a_1+a_2-2-2s)}{a_1+a_2}}p^{\frac{2(a_1+a_2-2-2s)}{a_1+a_2}-4}n^{-(a_1+a_2-2-2s)-2} \notag \\
=&~O(1)|J_A|^{-\frac{a_1+a_2-2-2s}{a_1+a_2}} (|J_A|/p^2)^{2-\frac{a_1+a_2-2-2s}{a_1+a_2}} n^{-(a_1+a_2-2-2s)-2}\notag \\
=&~o(n^{-2}). \notag
\end{align}
Therefore $\mathrm{var}\{ \mathbb{T}_{k,a_{1},a_{2},(M)}\}_{(1)}=o(n^{-2}).$


\medskip

\textbf{Case (2):} When $M=3$, we assume without loss of generality that $j_1=j_3$ and $j_5=j_7$, then
\begin{align}
	\{j_1,j_2,j_3,j_4\}=\{j_1,j_2,j_4\} \text{ \ and \ }\{j_5,j_6,j_7,j_8\}=\{j_5,j_6,j_8\}.  \label{eq:assumm3casej}
\end{align}
It follows that $ \mathrm{E}( \prod_{t=1}^4 x_{1,j_t})=\kappa_1\sigma_{j_1,j_1}\sigma_{j_2,j_4}$ and $ \mathrm{E}(\prod_{t=5}^8 x_{1,j_t})=\kappa_1\sigma_{j_5,j_5}\sigma_{j_6,j_8},$
which are 0 when $(j_2,j_4) \text{ and } (j_6,j_8) \in J_A^c$; and are $O(\rho)$ when $(j_2,j_4)$  and  $(j_6,j_8) \in J_A$.  
This suggests that if $\eqref{eq:expaltvarpi1}\neq 0$, $(j_2,j_4)$  and  $(j_6,j_8) \in J_A$.

We first examine  $\mathrm{var}\{ \mathbb{T}_{k,a_{1},a_{2},(3)}\}  _{(1)}$, which is  the part of summation in $\mathrm{var}\{ \mathbb{T}_{k,a_{1},a_{2},(3)}\}$ when
  $|\cup_{l=1}^4\{\mathbf{i}^{(l)}\}|=a_1+a_2-2.$ Recall that the two claims in \eqref{eq:twoclaims43sec} also hold here.  Similarly to \textbf{Case (1)} above,  
  we still assume  $|\{\mathbf{i}^{(1)}\}\cap\{\mathbf{i}^{(2)}\}|= s$, and $|\{\mathbf{i}^{(1)}\}\cap\{\mathbf{i}^{(3)}\}|= s_1$, then  \eqref{eq:2am1orderexpalt} holds. 
  We next discuss several sub-cases based on the size of the set  $\{j_t:t=1,\ldots,8\}$.

\smallskip

\textbf{Case (2.1):} When $|\{j_t:t=1,\ldots,8\}|=6,$ 
we know $\{j_1,j_2,j_3,j_4\}\cap \{j_5,j_6,j_7,j_8\}=\emptyset$ by \eqref{eq:assumm3casej}.   
Then by \eqref{eq:2am1orderexpalt}, we know if $\eqref{eq:expaltvarpi1} \neq 0$, then $(j_2,j_4), (j_6,j_8), (j_1,j_5), (j_2,j_6)\in J_A$ or $(j_2,j_4), (j_6,j_8), (j_1,j_6), (j_2,j_5)\in J_A$. Thus by the symmetricity of the $j$ indexes, we have
\begin{align*}
	 \sum_{\substack{(j_1,j_2),(j_3,j_4),\\(j_5,j_6),(j_7,j_8)\in J_A^c}} \mathbf{1}_{ \eqref{eq:expaltvarpi1} \neq 0}\times \mathbf{1}_{|\{j_t:t=1,\ldots,8\}|=6}	\leq C\sum_{\substack{(j_1,j_5), (j_2,j_6),\\ (j_2,j_4), (j_6,j_8)\in J_A}}  1 \leq C|J_A|^3. \notag 
\end{align*} 

By Conditions \ref{cond:altonecovellip} and \ref{cond:rhoijconst},  $\eqref{eq:expaltvarpi1} =O(\rho^{\tilde{A}_1})$, where $\tilde{A}_1=2(a_1+a_2-2-2s) +2(s+1)=2(a_1+a_2)-2(s+1)$. Thus
\begin{align*}
 \Big|\sum_{(j_1,j_2),(j_3,j_4),(j_5,j_6),(j_7,j_8)\in J_A^c}\eqref{eq:expaltvarpi1}  \mathbf{1}_{|\{j_t:t=1,\ldots,8\}|=6}	\Big|=  O(|J_A|^3 \rho^{\tilde{A}_1}).	
\end{align*}
\smallskip


\textbf{Case (2.2):} When $|\{j_t:t=1,\ldots,8\}|=5,$
recall that we assume \eqref{eq:assumm3casej}, where $j_1=j_3$ and $j_5=j_7$ without loss of generality. 
If we further assume $j_1=j_5$,   $\{j_t:t=1,\ldots,8\}=\{j_1,j_2,j_4,j_6,j_8\}$. Then for $\eqref{eq:expaltvarpi1}\neq 0$,
$\mathrm{E}(\prod_{t=1,2,3,4}x_{1,j_t})\times \mathrm{E}(\prod_{t=5,6,7,8}x_{1,j_t})\neq 0$,  then  $(j_2,j_4), (j_6,j_8) \in J_A$ holds. 
In addition, under this case, $\eqref{eq:expaltvarpi1}=O\{\rho^{(a_1+a_2-2-2s)+2(s+1)}\}=O(\rho^{a_1+a_2})$, and we have
\begin{align*}
& \Big|\sum_{(j_1,j_2),(j_3,j_4),(j_5,j_6),(j_7,j_8)\in J_A^c}\mathbf{1}_{\eqref{eq:expaltvarpi1}=O(\rho^{a_1+a_2}),\, |\{j_t:t=1,\ldots,8\}|=5}	\Big| = O(	p|J_A|^2).
\end{align*} 
If given  $j_1=j_3$ and $j_5=j_7$, instead,  assume $j_1\neq j_5$. We have $j_1\neq j_2$, $j_1\neq j_4$ and $j_1\neq j_5$. Then for $\eqref{eq:expaltvarpi1}\neq 0$,  by  discussing different cases of $j$ indexes, we know that \eqref{eq:expaltvarpi1} achieves the order between $O(\rho^{ \tilde{A}_1 })$ and  $O(\rho^{ \tilde{A}_2 })$  where $\tilde{A}_1$ is defined as above and $ \tilde{A}_2 = 2(s+1)+ (1+2)\times (a_1+a_2-2s-2)/2=3(a_1+a_2)/2-(s+1)$. Moreover, we have
\begin{align*}
\Biggr|\sum_{\substack{(j_1,j_2),(j_3,j_4),\\(j_5,j_6),(j_7,j_8)\in J_A^c}}\mathbf{1}_{\{\eqref{eq:expaltvarpi1}=O(\rho^{u}),\, \tilde{A}_2\leq u\leq \tilde{A}_1, \, |\{j_t:t=1,\ldots,8\}|=5\}}	\Biggr|=O(D_{\max}|J_A|^2).
\end{align*}
In summary,
\begin{align*}
 & ~\Big|\sum_{{(j_1,j_2),(j_3,j_4),(j_5,j_6),(j_7,j_8)\in J_A^c}} \eqref{eq:expaltvarpi1} \times \mathbf{1}_{|\{j_t:t=1,\ldots,8\}|=5}	\Big| \notag \\
%
=&~ O (D_{\max}|J_A|^2 \rho^{\tilde{A}_1} )+O (D_{\max}|J_A|^2 \rho^{\tilde{A}_2} ) +O(p|J_A|^2\rho^{a_1+a_2}). \notag 
\end{align*}

\vspace{0.5em}

\textbf{Case (2.3):} When $|\{j_t:t=1,\ldots,8\}|=4,$ similarly as case (2.3), we can discuss $j_1=j_5$ and $j_1\neq j_5$ respectively. 
When $j_1=j_5$, we note that \eqref{eq:expaltvarpi1}  can achieve the orders between $O(\rho^{a_1+a_2})$ and $O(\rho^{\tilde{A}_3})$ 
with $\tilde{A}_3=(a_1+a_2-2-2s)/2+2(s+1)=(a_1+a_2)/2+s+1$. Moreover, 
\begin{align*}
& \Biggr|\sum_{\substack{(j_1,j_2),(j_3,j_4),\\(j_5,j_6),(j_7,j_8)\in J_A^c}}  \mathbf{1}_{\eqref{eq:expaltvarpi1}=O(\rho^{u}), \tilde{A}_3\leq u\leq a_1+a_2, \, |\{j_t:t=1,\ldots,8\}|=4}	\Biggr|=O(pD_{\max}|J_A|).
\end{align*} In addition, when $j_1\neq j_5$, we note that  \eqref{eq:expaltvarpi1} can achieve the order between $O(\rho^{a_1+a_2})$ and $O(\rho^{\tilde{A}_1})$. Under this case, 
\begin{align*}
\Biggr|\sum_{\substack{(j_1,j_2),(j_3,j_4),\\(j_5,j_6),(j_7,j_8)\in J_A^c}}  \mathbf{1}_{\eqref{eq:expaltvarpi1}=O(\rho^{u}), \tilde{A}_4\leq u\leq a_1+a_2, \, |\{j_t:t=1,\ldots,8\}|=4}	\Biggr|=	O(|J_A|^2).
\end{align*}
In summary, by $|J_A|\leq pD_{\max}$, 
\begin{align*}
&~ \Big|\sum_{(j_1,j_2),(j_3,j_4),(j_5,j_6),(j_7,j_8)\in J_A^c} \eqref{eq:expaltvarpi1}\times  \mathbf{1}_{|\{j_t:t=1,\ldots,8\}|=4}	\Big| \notag \\
=&~O(pD_{\max}|J_A|\rho^{\tilde{A}_3})+ O(pD_{\max}|J_A|\rho^{a_1+a_2})+O(|J_A|^2 \rho^{\tilde{A}_1}). 	
\end{align*}

\smallskip

\textbf{Case (2.4):} When $|\{j_t:t=1,\ldots,8\}|\leq 3,$  we know by Condition \ref{cond:finitemomt}, 
\begin{align*}
& \Big|\sum_{(j_1,j_2),(j_3,j_4),(j_5,j_6),(j_7,j_8)\in J_A^c} \eqref{eq:expaltvarpi1}\times \mathbf{1}_{|\{j_t:t=1,\ldots,8\}|\leq 3}	\Big|=O(p^3). \notag 
\end{align*}
\smallskip

In summary, combining Cases (2.1)--(2.4) above, 
we know 
\begin{align}
&\mathrm{var}\{ \mathbb{T}_{k,a_{1},a_{2},(3)}\}_{(1)}\label{eq:vart3ordersumalt}\\
= &~\prod_{l=1}^2c(n,a_l) n^{a_1+a_2-2} \sum_{s=0}^{A_0}\Big\{ O(p^3) + O(|J_A|^3 \rho^{\tilde{A}_1} ) \notag \\
&~+O (D_{\max}|J_A|^2 \rho^{\tilde{A}_1} )+O (D_{\max}|J_A|^2 \rho^{\tilde{A}_2} ) +O(p|J_A|^2\rho^{a_1+a_2})\notag \\
&~+O(pD_{\max}|J_A|\rho^{\tilde{A}_3})+ O(pD_{\max}|J_A|\rho^{a_1+a_2})+O(|J_A|^2 \rho^{\tilde{A}_1})\Big\}, \notag 
\end{align}
where  $\tilde{A}_1=2(a_1+a_2)-2(s+1)$, $ \tilde{A}_2 = 3(a_1+a_2)/2-(s+1)$ and $\tilde{A}_3=(a_1+a_2)/2+s+1$. 

Note that  
\begin{align}
&~\prod_{l=1}^2c(n,a_l)\times  n^{a_1+a_2-2} 	|J_A|^3 \rho^{\tilde{A}_1} \notag \\
=& ~ p^{-4}n^{-2}|J_A|^3\rho^{2(a_1+a_2-s-1)} \notag \\
=& ~ p^{-4}n^{-2}(|J_A|\rho^{a_1}\times |J_A|\rho^{a_2})^{\frac{2(a_1+a_2-s-1)}{a_1+a_2}}|J_A|^{3-\frac{4(a_1+a_2-s-1)}{a_1+a_2}} \label{eq:pluglat1} \\
=&~ O(1)n^{-2}p^{\frac{4(a_1+a_2-s-1)}{a_1+a_2} -4}n^{-(a_1+a_2-s-1)}|J_A|^{-1+\frac{4(s+1)}{a_1+a_2}} \label{eq:pluglat2} \\
=&~  O(1)n^{-2}p^{-\frac{4(s+1)}{a_1+a_2}}n^{-(a_1+a_2-s-1)}|J_A|^{-1+\frac{4(s+1)}{a_1+a_2}} \notag \\
=&~O(1)n^{-2}(|J_A|/p^2)^{\frac{2(s+1)}{a_1+a_2}}|J_A|^{-1+\frac{2(s+1)}{a_1+a_2}} \notag \\
=&~o(n^2),\notag 
\end{align} where from \eqref{eq:pluglat1} to \eqref{eq:pluglat2}, we use $|J_A|\rho^a=O(pn^{-a/2})$, and in the last equation, we use $2(s+1)\leq a_1+a_2-1$. 
Following similar analysis, we know that all the terms in \eqref{eq:vart3ordersumalt} are $o(n^{-2})$ and $\mathrm{var}\{ \mathbb{T}_{k,a_{1},a_{2},(3)}\}_{(1)}=o(n^{-2}).$

We next examine   $\mathrm{var}\{ \mathbb{T}_{k,a_{1},a_{2},(3)}\}_{(2)}.$ Note that if $\eqref{eq:expaltvarpi1}\neq 0$, $(j_2,j_4)$ and  $(j_6,j_8) \in J_A$.  We can discuss different cases of $\{j_1,\ldots,j_8\}$ similarly as above. Then by Conditions \ref{cond:rhoijconst} and \ref{cond:altonecovellip}, as $\rho=O(|J_A|^{-1/a_t}p^{1/a_t}n^{-1/2})$ for $t=1,2$, we have 
$
	\sum_{ (j_{1},j_2),(j_3,j_4),(j_{5},j_6),(j_7,j_8)\in J_A^c }\eqref{eq:expaltvarpi1}=O(p^4). 
$
Given that $|\cup_{l=1}^4\{\mathbf{i}^{(l)}\}|<a_1+a_2-2$ in $\mathrm{var}\{ \mathbb{T}_{k,a_{1},a_{2},(3)}\}_{(2)},$ we obtain $\mathrm{var}\{\mathbb{T}_{k,a_{1},a_{2},(3)}\}_{(2)}=\prod_{l=1}^2 c(n,a_l)\times O\{p^4 n^{a_1+a_2-3}\}= o(n^{-2}).
$ 

In summary, we have $\mathrm{var}\{\mathbb{T}_{k,a_{1},a_{2},(3)}\}=o(n^{-2}).$

\medskip

\textbf{Case (3):} When $M=4$, we consider $j_1\neq j_2\neq j_3\neq j_4$ and $j_5\neq j_6\neq j_7\neq j_8$ under this case. Since $\sigma_{j_1,j_2}=\sigma_{j_3,j_4}=\sigma_{j_5,j_6}=\sigma_{j_7,j_8}=0,$ 
\begin{align*}
	& \mathrm{E}(x_{1,j_1}x_{1,j_2}x_{1,j_3}x_{1,j_4})=\kappa_1(\sigma_{j_1,j_3}\sigma_{j_2,j_4}+ \sigma_{j_1,j_4}\sigma_{j_2,j_3}),\notag \\ 
	&\mathrm{E}(x_{1,j_5}x_{1,j_6}x_{1,j_7}x_{1,j_8})=\kappa_1(\sigma_{j_5,j_7}\sigma_{j_6,j_8}+ \sigma_{j_5,j_8}\sigma_{j_6,j_7}),
\end{align*} which are $O(\rho^2)$.
 Following similar analysis to \textbf{Case (2)},  we can examine the different cases when $|\{j_t:t=1,\ldots,8\}|$ is between 4 and 8, and  obtain,  
\begin{align}
&~ \mathrm{var}\{ \mathbb{T}_{k,a_{1},a_{2},(4)}\}_{(1)} \label{eq:var4termalt} \\
= &~ O(1)\prod_{l=1}^2c(n,a_l)\times n^{a_1+a_2-2}\sum_{s=0}^{A_0}\Big[ |J_A|^2 \rho^{4(s+1)} \notag \\
&~+D_{\max}|J_A|^2\rho^{4(s+1)}\Big(\rho^{ a_1-1-s} +\rho^{a_2-1-s} \Big)\notag\\
&~+ \max\{|J_A|,D_{\max}^2\} \times |J_A|^2 \rho^{4(s+1)}\Big(\rho^{2(a_1-1-s)} +\rho^{2(a_2-1-s)} \Big)\notag\\
&~+ D_{\max}|J_A|^3 \Big(\rho^{2(a_1+a_2)-(a_1-1-s)}+ \rho^{ 2(a_1+a_2)-(a_2-1-s)}\Big) \notag\\
&~+ |J_A|^4 \rho^{2(a_1+a_2)}. \notag
\end{align}
Note that $\prod_{l=1}^2 c(n,a_l) n^{a_1+a_2-2}|J_A|^4\rho^{2(a_1+a_2)}=O(1)p^{-4}n^{-2}p^4n^{-(a_1+a_2)}=o(n^{-2}).$ Moreover,
\begin{eqnarray}
&&\prod_{l=1}^2 c(n,a_l) \times n^{a_1+a_2-2}D_{\max}|J_A|^2\rho^{4(s+1)}\Big(\rho^{ a_1-1-s} +\rho^{a_2-1-s} \Big)\label{eq:varaltod42} \\
&=& p^{-4}n^{-2} D_{\max}|J_A|^2\Big( \rho^{a_1+3(s+1)}+\rho^{a_2+3(s+1)}\Big).  \notag 
\end{eqnarray} 
To show $\eqref{eq:varaltod42}=o(n^{-2})$ by symmetricity, it suffices to show for any integer $a_1$, $p^{-4} D_{\max}|J_A|^2 \rho^{a_1+3(s+1)}=o(1)$. 
\begin{eqnarray}
&&p^{-4} D_{\max}|J_A|^2 \rho^{a_1+3(s+1)} \notag \\
&=&p^{-4} D_{\max} (|J_A|\rho^{a_1})^{\frac{a_1+3(s+1)}{a_1}}|J_A|^{2-\frac{a_1+3(s+1)}{a_1}} \label{eq:altvar4ord1} \\
&=&O(1)p^{-4} D_{\max} (pn^{-a_1/2})^{\frac{a_1+3(s+1)}{a_1}} |J_A|^{2-\frac{a_1+3(s+1)}{a_1}}\label{eq:altvar4ord2}\\
&=& O(1)n^{-\frac{a_1+3(s+1)}{2}}({|J_A|}/{p^2})^{1-\frac{(s+1)}{a_1}}  \notag \\
&&\times ({D_{\max}}/{p})^{1-\frac{s+1}{a_1}}({D_{\max}}/{|J_A|})^{\frac{s+1}{a_1}}|J_A|^{-\frac{s+1}{a_1}},\notag \\
&=&o(1) ,\notag
\end{eqnarray}
where from \eqref{eq:altvar4ord1} to \eqref{eq:altvar4ord2}, we use $|J_A|\rho^{a_1}=O(pn^{-a_1/2})$, and in the last equation we use $|J_A|=o(p^2)$, $D_{\max}\leq p$ and $D_{\max}\leq |J_A|$.  For other terms in \eqref{eq:var4termalt}, similar analysis can be applied and we have $\mathrm{var}\{ \mathbb{T}_{k,a_{1},a_{2},(4)}\}_{(1)}=o(n^{-2}).$

In addition, similarly to the analysis of $\mathrm{var}\{ \mathbb{T}_{k,a_{1},a_{2},(3)}\}_{(2)},$ by Conditions \ref{cond:rhoijconst} and \ref{cond:altonecovellip}, we still have  $
	\sum_{ (j_{1},j_2),(j_3,j_4),(j_{5},j_6),(j_7,j_8)\in J_A^c }\eqref{eq:expaltvarpi1}=O(p^4). 
$  Since $|\cup_{l=1}^4\{\mathbf{i}^{(l)}\}|<a_1+a_2-2$ in $\mathrm{var}\{ \mathbb{T}_{k,a_{1},a_{2},(4)}\}_{(2)}$ by construction, we obtain $\mathrm{var}\{\mathbb{T}_{k,a_{1},a_{2},(4)}\}_{(2)}=\prod_{l=1}^2 c(n,a_l)\times O\{p^4 n^{a_1+a_2-3}\}= o(n^{-2}).
$ In summary, $\mathrm{var}\{\mathbb{T}_{k,a_{1},a_{2},(4)}\}= o(n^{-2})$ is proved.

\subsubsection{Proof of Lemma \ref{lm:altcltmomentgoal2} (on Page \pageref{lm:altcltmomentgoal2}, Section \ref{sec:updatepoweronesampf})}\label{sec:pfaltcltmomentgoal2}

Similarly to Section \ref{sec:proofsecondgoaljointnormal}, 
\begin{align*}
 	\sum_{k=1}^n\mathrm{E}(D_{n,k}^4)=\sum_{k=1}^n \sum_{{1\leq r_1,r_2,r_3,r_4\leq m}}\prod_{l=1}^4 t_{r_l}\times \mathrm{E}\Big(\prod_{l=1}^4 A_{n,k,a_{r_l}}\Big), 
 \end{align*} where we use the redefined notation in Section \ref{sec:updatepoweronesampf}.  To prove  Lemma \ref{lm:secondgoaljointnormal}, it suffices to show that for given $1\leq k \leq n$ and $1\leq r_1,r_2,r_3,r_4\leq m$, we have $\mathrm{E}(\prod_{l=1}^4 A_{n,k,a_{r_l}})=o(n^{-1})$. Moreover by the Cauchy-Schwarz inequality, it suffices to show $\mathrm{E}(A_{n,k,a}^4)=o(n^{-1})$ for  $a\in \{a_1,\ldots,a_m\}$. 
Following  \eqref{eq:expank4indpill}, we have $A_{n,k,a}=0$ when $k<a$; and when $k\geq a$,
\begin{align*}
\mathrm{E}( A_{n,k,a}^4)=&~ c^2(n,a)	\sum_{ \substack{ \mathbf{i}^{(l)} \in \mathcal{P}(k-1,a-1),\, l=1,2,3,4; \\ (j_1,j_2),(j_3,j_4), (j_5,j_6),(j_7,j_8)\in J_A^c } }Q^*( \mathbf{i}^{(1)}, \mathbf{i}^{(2)},\mathbf{i}^{(3)},\mathbf{i}^{(4)},\mathbf{j}_8),	
\end{align*}
where $\mathbf{i}^{(l)}=(i_1^{(l)},\ldots, i_{a}^{(l)})$ represents tuples $1\leq i_1^{(l)}\neq \ldots \neq i_{a}^{(l)}\leq n$, and
\begin{align*}
Q^*( \mathbf{i}^{(1)}, \mathbf{i}^{(2)},\mathbf{i}^{(3)},\mathbf{i}^{(4)},\mathbf{j}_8)=&~\mathrm{E}\Big( \prod_{r=1}^8 x_{k,j_r}\Big)\,\mathrm{E}\Big( \prod_{l=1}^4\prod_{t=1}^{a-1} x_{i_t^{(l)},j_{2l-1}} x_{i_t^{(l)},j_{2l}}\Big).\notag 
\end{align*}  
As $c(n,a)=\Theta(p^{-1}n^{-a/2})$, to prove $\mathrm{E}(A_{n,k,a}^4)=o(n^{-1})$, it suffices to show
\begin{align*}
\sum_{ \substack{ \mathbf{i}^{(l)} \in \mathcal{P}(k-1,a-1),\, l=1,2,3,4; \\ (j_1,j_2),(j_3,j_4), (j_5,j_6),(j_7,j_8)\in J_A^c } }Q^*( \mathbf{i}^{(1)}, \mathbf{i}^{(2)},\mathbf{i}^{(3)},\mathbf{i}^{(4)},\mathbf{j}_8)=o(p^4n^{2a-1}).	
\end{align*}

Since $\sigma_{j_1,j_2}=0$ if $(j_1,j_2)\in J_A^c$, 
then similarly to Section \ref{sec:proofsecondgoaljointnormal}, we have
$Q^*( \mathbf{i}^{(1)}, \mathbf{i}^{(2)},\mathbf{i}^{(3)},\mathbf{i}^{(4)},\mathbf{j}_8)\neq 0$ 
only when 
$
	|\bigcup_{l=1}^4 \{\mathbf{i}^{(l)}\} |\leq 2(a-1), 
$ and similarly to \eqref{eq:n2amin1order},
\begin{align*}
&&\quad \quad \sum_{ \substack{ \mathbf{i}^{(l)} \in \mathcal{P}(k-1,a-1),\, l=1,\ldots,4 } } Q^*( \mathbf{i}^{(1)}, \mathbf{i}^{(2)},\mathbf{i}^{(3)},\mathbf{i}^{(4)},\mathbf{j}_8)=O(n^{2a-2}). 	
\end{align*}
 It then remains to show  
\begin{align}
\sum_{\substack{ (j_1,j_2),(j_3,j_4),  (j_5,j_6),(j_7,j_8)\in J_A^c } } Q^*( \mathbf{i}^{(1)}, \mathbf{i}^{(2)},\mathbf{i}^{(3)},\mathbf{i}^{(4)},\mathbf{j}_8)=O(p^4).\label{eq:secondgoalp4alt}
\end{align}

We next prove by discussing  $|\{j_t: t=1,\ldots,8\}|$ and the corresponding value of $Q^*( \mathbf{i}^{(1)}, \mathbf{i}^{(2)},\mathbf{i}^{(3)},\mathbf{i}^{(4)},\mathbf{j}_8)$. By Condition  \ref{cond:altonecovellip},  $Q^*( \mathbf{i}^{(1)}, \mathbf{i}^{(2)},\mathbf{i}^{(3)},\mathbf{i}^{(4)},\mathbf{j}_8)$ can be written as certain linear combination of $\prod_{t=1}^{4a} (\sigma_{j_{g_{2t-1}},j_{g_{2t}}})$, where $g_{2t-1}\neq g_{2t}$ and $(g_1,\ldots,g_{8a})$ contain $a$ number of $1,\ldots,8$ respectively. 
If $|\{j_t: t=1,\ldots,8\}|\leq 4$, by Condition  \ref{cond:finitemomt}, 
\begin{align*}
\sum_{\substack{ (j_1,j_2),(j_3,j_4),  (j_5,j_6),(j_7,j_8)\in J_A^c } } Q^*( \mathbf{i}^{(1)}, \mathbf{i}^{(2)},\mathbf{i}^{(3)},\mathbf{i}^{(4)},\mathbf{j}_8)\times \mathbf{1}_{\{|\{j_t: t=1,\ldots,8\}|\leq 4\}}=O(p^4). 	
\end{align*}
If $|\{j_t: t=1,\ldots,8\}|=5$, note that for $j_1\neq j_2$, $\sigma_{j_1,j_2}\neq 0$ only when $(j_1,j_2)\in J_A$, then
\begin{align*}
&~ \Big|\sum_{\substack{ (j_1,j_2),(j_3,j_4),  (j_5,j_6),(j_7,j_8)\in J_A^c } } Q^*( \mathbf{i}^{(1)}, \mathbf{i}^{(2)},\mathbf{i}^{(3)},\mathbf{i}^{(4)},\mathbf{j}_8)\times \mathbf{1}_{\{|\{j_t: t=1,\ldots,8\}|=5 \}}\Big|	\notag \\
\leq &~C \sum_{\substack{1\leq j_1,j_2,j_5\leq p,\\(j_6,j_8)\in J_A}} \sigma_{j_1,j_1}^a\sigma_{j_2,j_2}^a\sigma_{j_5,j_5}^a\sigma_{j_6,j_8}^a=O(p^3|J_A|\rho^a)=o(p^4),
\end{align*}where in the last equation, we use $|J_A|\rho^a=O(pn^{-a/2})$. 
In addition, similarly, if $|\{j_t: t=1,\ldots,8\}|=6$, 
\begin{align*}
&~ \Big|\sum_{\substack{ (j_1,j_2),(j_3,j_4),  (j_5,j_6),(j_7,j_8)\in J_A^c } } Q^*( \mathbf{i}^{(1)}, \mathbf{i}^{(2)},\mathbf{i}^{(3)},\mathbf{i}^{(4)},\mathbf{j}_8)\times \mathbf{1}_{\{|\{j_t: t=1,\ldots,8\}|=6 \}}\Big|	\notag \\
\leq &~ C \sum_{\substack{1\leq j_1,j_2\leq p,\\(j_5,j_7),(j_6,j_8)\in J_A}} \sigma_{j_1,j_1}^a\sigma_{j_2,j_2}^a\sigma_{j_5,j_7}^a\sigma_{j_6,j_8}^a=O(p^2|J_A|^2\rho^{2a})=o(p^4).
\end{align*}
If $|\{j_t: t=1,\ldots,8\}|=7$, 
\begin{align*}
&~ \Big|\sum_{\substack{ (j_1,j_2),(j_3,j_4),  (j_5,j_6),(j_7,j_8)\in J_A^c } } Q^*( \mathbf{i}^{(1)}, \mathbf{i}^{(2)},\mathbf{i}^{(3)},\mathbf{i}^{(4)},\mathbf{j}_8)\times \mathbf{1}_{\{|\{j_t: t=1,\ldots,8\}|=7 \}}\Big|	\notag \\
\leq &~ C \sum_{\substack{1\leq j_1\leq p,\\(j_2,j_4), (j_5,j_7),(j_6,j_8)\in J_A}} \sigma_{j_1,j_1}^a\sigma_{j_2,j_4}^a\sigma_{j_5,j_7}^a\sigma_{j_6,j_8}^a=O(p|J_A|^3\rho^{3a})=o(p^4).
\end{align*}
If $|\{j_t: t=1,\ldots,8\}|=8$, 
\begin{align*}
&~ \Big|\sum_{\substack{ (j_1,j_2),(j_3,j_4),  (j_5,j_6),(j_7,j_8)\in J_A^c } } Q^*( \mathbf{i}^{(1)}, \mathbf{i}^{(2)},\mathbf{i}^{(3)},\mathbf{i}^{(4)},\mathbf{j}_8)\times \mathbf{1}_{\{|\{j_t: t=1,\ldots,8\}|=8 \}}\Big|	\notag \\
\leq &~ C \sum_{(j_1,j_3),(j_2,j_4), (j_5,j_7),(j_6,j_8)\in J_A}\sigma_{j_1,j_3}^a\sigma_{j_2,j_4}^a\sigma_{j_5,j_7}^a\sigma_{j_6,j_8}^a=O(|J_A|^4\rho^{4a})=o(p^4).
\end{align*}
In summary, \eqref{eq:secondgoalp4alt} is obtained and Lemma \ref{lm:altcltmomentgoal2} is proved.
\end{proof}

\subsection{Lemmas for the proof of Theorem \ref{thm:onesamplemean}} \label{sec:onesammeanlm}
In this section, we prove Lemma \ref{lm:cltonesammeanlm} on Page \pageref{lm:cltonesammeanlm}, where we prove  $\mathrm{var}(\sum_{k=1}^n\pi^2_{n,k})\to 0$ and $\sum_{k=1}^n\mathrm{E}(D_{n,k}^4)\to 0$ in the following Sections \ref{sec:onesamclt1pf} and \ref{sec:onesamclt2pf}, 
respectively. 

\subsubsection{Proof of Lemma \ref{lm:cltonesammeanlm} (on Page \pageref{lm:cltonesammeanlm}, Section \ref{sec:proof3132})}
\paragraph{Proof of $\mathrm{var}(\sum_{k=1}^n\pi^2_{n,k})\to 0$}\label{sec:onesamclt1pf}

Similarly to Section \ref{sec:prooftargetorder}, 
 $D_{n,k}=\sum_{r=1}^m t_{r}A_{n,k,a_r}$, and then 
$\pi_{n,k}^2=  \sum_{1\leq r_1,r_2\leq m} t_{r_1}t_{r_2}\mathrm{E}_{k-1}( A_{n,k,a_{r_1}} A_{n,k,a_{r_2}}).$
Note that by the Cauchy-Schwarz inequality, for some constant $C$, $$\mathrm{var}\Big(\sum_{k=1}^n \pi_{n,k}^2\Big)\leq Cn^2 \max_{1\leq k\leq n;\, 1\leq r_1,r_2\leq m} \mathrm{var}(\mathbb{T}_{k,a_{r_1},a_{r_2}}), $$
where $c(n,a)=[a\times \{ \sigma(a)P^n_{a}\}^{-1}]^{2}$ and for two integers $a_1$ and $a_2$ we still define $\mathbb{T}_{k,a_{1},a_{2}}=\mathrm{E}_{k-1}( A_{n,k,a_{1}} A_{n,k,a_{2}}).$ In particular, when $k<\max\{a_1,a_2\},$ $\mathbb{T}_{k,a_{1},a_{2}}=0$; when  $k\geq \max\{a_1,a_2\},$
\begin{eqnarray*}
\mathbb{T}_{k,a_{1},a_{2}}&=&\mathrm{E}_{k-1}( A_{n,k,a_{1}} A_{n,k,a_{2}}) \notag \\
&=&\sum_{\substack{1\leq j_1,j_2\leq p ; \\ \mathbf{i}^{(l)}\in \mathcal{P}(k-1,a_l-1):\, l=1,2}}\{c(n,a_1)c(n,a_2)\}^{1/2}\sigma_{j_1,j_2}\prod_{l=1}^2\prod_{t=1}^{a_l-1}x_{ i_t^{(l)},j_l}. 
\end{eqnarray*}
To prove Lemma $\mathrm{var}(\sum_{k=1}^n\pi^2_{n,k})\to 0$, it suffices to prove $\mathrm{var}(\mathbb{T}_{k,a_{1},a_{2}})=o(n^{-2}),$ where $\mathrm{var}(\mathbb{T}_{k,a_{1},a_{2}})=\mathrm{E}(\mathbb{T}_{k,a_{1},a_{2}}^2)-\{\mathrm{E}(\mathbb{T}_{k,a_{1},a_{2}}) \}^2$. We consider without loss of generality that $k\geq \max\{a_1,a_2\}.$ 

When $\{\mathbf{i}^{(1)}\}\neq \{\mathbf{i}^{(2)}\}$, $\mathrm{E}(\prod_{l=1}^2\prod_{t=1}^{a_l}x_{ i_t^{(l)},j_t})=0$; and when $\{\mathbf{i}^{(1)}\}=\{\mathbf{i}^{(2)}\}$, it induces $a_1=a_2$ and $\mathrm{E}(\prod_{l=1}^2\prod_{t=1}^{a_l}x_{ i_t^{(l)},j_t})=\sigma_{j_1,j_2}^{a}$ where we write $a_1=a_2=a$. It follows that when $a_1\neq a_2$, $\mathrm{E}(\mathbb{T}_{k,a_{1},a_{2}})=0$; when $a_1=a_2=a$,
\begin{align*}
	\mathrm{E}(\mathbb{T}_{k,a_{1},a_{2}})=\sum_{\substack{1\leq j_1,j_2\leq p ; \\ \mathbf{i}^{(l)}\in \mathcal{P}(k-1,a_l-1):\, l=1,2}}\mathbf{1}_{ \{\{\mathbf{i}^{(1)}\}=\{\mathbf{i}^{(2)}\}\}} \times \{c(n,a_1)c(n,a_2)\}^{1/2}\sigma_{j_1,j_2}^{a}. 
\end{align*}
Then
\begin{align*}
	\{\mathrm{E}(\mathbb{T}_{k,a_{1},a_{2}})\}^2=\sum_{\substack{1\leq j_1,j_2,j_3,j_4\leq p ; \\ \mathbf{i}^{(l)}\in \mathcal{P}(k-1,a_l-1):\, l=1,2,3,4}}\mathbf{1}_{\Big\{ \substack{\{\mathbf{i}^{(1)}\}=\{\mathbf{i}^{(2)}\}\\ \{\mathbf{i}^{(3)}\} =\{\mathbf{i}^{(4)}\}}  \Big\} } \prod_{l=1}^2 c(n,a_l)\times(\sigma_{j_1,j_2}\sigma_{j_3,j_4})^{a}.
\end{align*}
In addition, we obtain
\begin{eqnarray*}
\mathrm{E}(\mathbb{T}_{k,a_{1},a_{2}}^2)=\sum_{\substack{1\leq j_1,j_2,j_3,j_4\leq p ; \\ \mathbf{i}^{(l)}\in \mathcal{P}(k-1,a_l-1):\, l=1,2,3,4}}\Big\{\prod_{l=1}^2c(n,a_l)\sigma_{j_{2l-1},j_{2l}}\Big\}\mathrm{E}\Big( \prod_{l=1}^4\prod_{t=1}^{a_l-1} x_{i_t^{(l)}, \, j_{l}}\Big),
\end{eqnarray*}where for simplicity of representation, we set   $a_3=a_1$ and $a_4=a_2$. 
Define
\begin{align*}
	G_{k,a_1,a_2,1}
	=&~\sum_{\substack{1\leq j_1,j_2,j_3,j_4\leq p ; \\ \mathbf{i}^{(l)}\in \mathcal{P}(k-1,a_l-1):\, l=1,2,3,4}}\mathbf{1}_{\Big\{ \substack{ \{\mathbf{i}^{(1)}\}=\{\mathbf{i}^{(2)}\},\\\{\mathbf{i}^{(3)}\}=\{\mathbf{i}^{(4)}\},\\\{\mathbf{i}^{(1)}\}\cap \{\mathbf{i}^{(3)}\}= \emptyset } \Big\}}\notag \\
	&~ \quad \times \Big\{\prod_{l=1}^2c(n,a_l)\sigma_{j_{2l-1},j_{2l}}\Big\}\mathrm{E}\Big( \prod_{l=1}^4\prod_{t=1}^{a_l-1} x_{i_t^{(l)}, \, j_{l}}\Big).
\end{align*}
Since $|\mathrm{E}(\mathbb{T}_{k,a_{1},a_{2}}^2)-\{\mathrm{E}(\mathbb{T}_{k,a_{1},a_{2}})\}^2|\leq |\mathrm{E}(\mathbb{T}_{k,a_{1},a_{2}}^2)-G_{k,a_1,a_2,1}|+|G_{k,a_1,a_2,1}- \{\mathrm{E}(\mathbb{T}_{k,a_{1},a_{2}})\}^2|$, we next prove $\mathrm{E}(\mathbb{T}_{k,a_{1},a_{2}}^2)-G_{k,a_1,a_2,1}=o(n^{-2})$ and $G_{k,a_1,a_2,1}- \{\mathrm{E}(\mathbb{T}_{k,a_{1},a_{2}})\}^2=o(n^{-2})$ respectively. 

\smallskip
\subparagraph*{Step I: $\mathrm{E}(\mathbb{T}_{k,a_{1},a_{2}}^2)-G_{k,a_1,a_2,1}=o(n^{-2})$} 
When $\{\mathbf{i}^{(1)}\}=\{\mathbf{i}^{(2)}\},\, \{\mathbf{i}^{(3)}\} =\{\mathbf{i}^{(4)}\}$ and $\{\mathbf{i}^{(1)}\}\cap \{\mathbf{i}^{(3)}\}=\emptyset$, it implies that $a_1=a_2=a$, $|\cup_{l=1}^4 \{\mathbf{i}^{(l)}\}|\leq a_1+a_2-3$, and  
\begin{align*}
\Big(\prod_{l=1}^2\sigma_{j_{2l-1},j_{2l}}\Big)\times \mathrm{E}\Big(\prod_{l=1}^4\prod_{t=1}^{a_l-1} x_{i_t^{(l)}, \, j_{l}}\Big)=(\sigma_{j_1,j_2}\sigma_{j_3,j_4})^a.	
\end{align*}
It follows that if $a_1\neq a_2$, $\{\mathrm{E}(\mathbb{T}_{k,a_{1},a_{2}})\}^2-G_{k,a_1,a_2,1}=0$;  if $a_1= a_2=a$,
\begin{align*}
&~\Big|\{\mathrm{E}(\mathbb{T}_{k,a_{1},a_{2}})\}^2-G_{k,a_1,a_2,1}\Big| \notag \\
=&~c(n,a_1)c(n,a_2) O(n^{a_1+a_2-3})\Big|\sum_{1\leq j_1,j_2,j_3,j_4\leq p}(\sigma_{j_1,j_2}\sigma_{j_3,j_4})^a\Big|=o(n^{-2})	
\end{align*} where we use $c(n,a)=\Theta(p^{-1}n^{-a})$ and by Condition \ref{cond:onesamplemean},
\begin{align}
	\sum_{1\leq j_1,j_2,j_3,j_4\leq p}(\sigma_{j_1,j_2}\sigma_{j_3,j_4})^a=O(p^2). \label{eq:onesampordersum}
\end{align} 

\subparagraph*{Step II: $G_{k,a_1,a_2,1}- \{\mathrm{E}(\mathbb{T}_{k,a_{1},a_{2}})\}^2=o(n^{-2})$} 
We write
$
	\mathrm{E}(\mathbb{T}_{k,a_{1},a_{2}}^2)-G_{k,a_1,a_2,1}=G_{k,a_1,a_2,2}+G_{k,a_1,a_2,3},
$ where 
\begin{align*}
G_{k,a_1,a_2,2}=&~	\sum_{\substack{1\leq j_1,j_2,j_3,j_4\leq p ; \\ \mathbf{i}^{(l)}\in \mathcal{P}(k-1,a_l-1):\, l=1,2,3,4}}\mathbf{1}_{\Big\{ \substack{ \{\mathbf{i}^{(1)}\}=\{\mathbf{i}^{(2)}\},\\\{\mathbf{i}^{(3)}\}=\{\mathbf{i}^{(4)}\},\\\{\mathbf{i}^{(1)}\}\cap \{\mathbf{i}^{(3)}\} \neq \emptyset } \Big\}}\notag \\
&~\times  \Big\{\prod_{l=1}^2c(n,a_l)\sigma_{j_{2l-1},j_{2l}}\Big\}\mathrm{E}\Big( \prod_{l=1}^4\prod_{t=1}^{a_l-1} x_{i_t^{(l)}, \, j_{l}}\Big),
\end{align*} and 
\begin{align*}
	G_{k,a_1,a_2,3}=&~	\sum_{\substack{1\leq j_1,j_2,j_3,j_4\leq p ; \\ \mathbf{i}^{(l)}\in \mathcal{P}(k-1,a_l-1):\, l=1,2,3,4}}\mathbf{1}_{\Big\{ \substack{ \{\mathbf{i}^{(1)}\}\neq \{\mathbf{i}^{(2)}\} \text{ or }\\\{\mathbf{i}^{(3)}\}\neq \{\mathbf{i}^{(4)}\}} \Big\}}\notag \\
&~\times  \Big\{\prod_{l=1}^2c(n,a_l)\sigma_{j_{2l-1},j_{2l}}\Big\}\mathrm{E}\Big( \prod_{l=1}^4\prod_{t=1}^{a_l-1} x_{i_t^{(l)}, \, j_{l}}\Big).
\end{align*}

For $G_{k,a_1,a_2,2}$, it is a summation over the indexes satisfying $ \{\mathbf{i}^{(1)}\}=\{\mathbf{i}^{(2)}\},\{\mathbf{i}^{(3)}\}=\{\mathbf{i}^{(4)}\}$ and $\{\mathbf{i}^{(1)}\}\cap \{\mathbf{i}^{(3)}\} \neq \emptyset$. Thus $|\cup_{l=1}^4 \{\mathbf{i}^{(l)}\}|\leq a_1+a_2-3,$ and by $c(n,a)=\Theta(p^{-1}n^{-a})$ and \eqref{eq:onesampordersum}, 
\begin{align*}
|G_{k,a_1,a_2,2}|\leq &~	Cp^{-2}n^{-(a_1+a_2)}n^{a_1+a_2-3}\sum_{1\leq j_1,j_2,j_3,j_4\leq p}\sigma_{j_1,j_2}\sigma_{j_3,j_4}=o(n^{-2}).
\end{align*}
For $G_{k,a_1,a_2,3}$, it is a summation over the indexes satisfying $ \{\mathbf{i}^{(1)}\}\neq \{\mathbf{i}^{(2)}\}$ or $\{\mathbf{i}^{(3)}\}\neq \{\mathbf{i}^{(4)}\}$. We assume without loss of generality that $\{\mathbf{i}^{(1)}\}\neq \{\mathbf{i}^{(2)}\}$ and there exists an index $m\in \{\mathbf{i}^{(1)}\}$ but $m\not \in \{\mathbf{i}^{(2)}\}$.  Similarly to Section \ref{sec:prooftargetorder}, we know 
\begin{align}
\Big(\prod_{l=1}^2\sigma_{j_{2l-1},j_{2l}}\Big)\times \mathrm{E}\Big(\prod_{l=1}^4\prod_{t=1}^{a_l-1} x_{i_t^{(l)}, \, j_{l}}\Big)\label{eq:nonzeroonesamm}
\end{align} is nonzero only when $|\cup_{l=1}^4\{\mathbf{i}^{(l)}\}|\leq a_1+a_2-2$, that is, each index appears at least twice among the four sets $\{\mathbf{i}^{(l)}\}, l=1,2,3,4$. Therefore, we know if $\eqref{eq:nonzeroonesamm}\neq 0$, $m\in \{\mathbf{i}^{(3)}\}\cup \{\mathbf{i}^{(4)}\}$. If $m\in \{\mathbf{i}^{(3)}\}$ but $m \not \in \{\mathbf{i}^{(4)}\}$, 
$
\eqref{eq:nonzeroonesamm}= \sigma_{j_1,j_2}\sigma_{j_3,j_4}\sigma_{j_1,j_3} \mathrm{E}(\text{other terms}). 	
$ Under this case, we define $
		\tilde{K}_0={-(2+\epsilon)(4+\gamma)\log p}/{(\epsilon \log \delta)} 
	$, where $\gamma$ and $\epsilon$ are some positive constants and $\delta$ is from Condition \ref{cond:onesamplemean}. Then we have 
\begin{align}
	\sum_{1\leq j_1,j_2,j_3,j_4\leq p} \eqref{eq:nonzeroonesamm}\leq &~ C \sum_{1\leq j_1,j_2,j_3,j_4\leq p} \sigma_{j_1,j_2}\sigma_{j_3,j_4}\sigma_{j_1,j_3} \label{eq:jodreonesammean1taa} \\
	\leq & ~  C\sum_{ \substack{ |j_1-j_2|\leq \tilde{K}_0, \\ |j_3-j_4|\leq \tilde{K}_0 ,\\|j_1-j_3|\leq \tilde{K}_0 }}1+ C\sum_{| j_1- j_2| \geq \tilde{K}_0} \delta^{|j_1-j_2|\epsilon/(2+\epsilon)} \notag \\
	=&~ O(p\tilde{K}_0^2) + O(p^4p^{-(4+\gamma)}),\notag
\end{align} where in the second inequality,  we use the symmetricity of $j$ indexes and also use Lemma \ref{lm:mixingineq} similarly as in Section \ref{sec:proof3132}. If $m\in \{\mathbf{i}^{(4)}\}$ but $m \not \in \{\mathbf{i}^{(s)}\}$, \eqref{eq:jodreonesammean1taa} also holds similarly.  If $m\in \{\mathbf{i}^{(3)}\}$ and $m\in \{\mathbf{i}^{(4)}\}$, $
\eqref{eq:nonzeroonesamm}= \sigma_{j_1,j_2}\sigma_{j_3,j_4}\mathrm{E}(x_{m,j_1}x_{m,j_3}x_{m,j_4})\mathrm{E}(\text{other terms}). 	
$ Similarly to \eqref{eq:jodreonesammean1taa}, as $\mathrm{E}(\mathbf{x})=\mathbf{0}$, if $|j_1-j_3|>\tilde{K}_0$ and $|j_1-j_4|>\tilde{K}_0$, $\eqref{eq:nonzeroonesamm}\leq C\delta^{|j_1-j_2|\epsilon/(2+\epsilon)}.$ Thus under this case, we also have $\sum_{1\leq j_1,j_2,j_3,j_4\leq p} \eqref{eq:nonzeroonesamm}=O(p\tilde{K}_0^2) + O(p^{-\gamma}).$  Recall that $\eqref{eq:nonzeroonesamm}\neq 0$ only when $|\cup_{l=1}^4\{\mathbf{i}^{(l)}\}|\leq a_1+a_2-2$. By $c(n,a)=\Theta(p^{-1}n^{-a})$ and $\tilde{K}_0=O(\log p)$,
\begin{align*}
	|G_{k,a_1,a_2,3}|\leq &~  Cp^{-2} n^{-(a_1+a_2)}n^{a_1+a_2-2}\sum_{1\leq j_1,j_2,j_3,j_4\leq p}\left|\eqref{eq:nonzeroonesamm}\right| \notag \\
	=&~ n^{-2}p^{-2}\Big\{ O(p\tilde{K}_0^2) + O(p^{-\gamma})\Big\} =o(n^{-2}). 
\end{align*}

In summary, 
\begin{align*}
	\mathrm{var}(\mathbb{T}_{k,a_{1},a_{2}})\leq |\mathrm{E}(\mathbb{T}_{k,a_{1},a_{2}}^2)-G_{k,a_1,a_2,1}|+|G_{k,a_1,a_2,2}|+|G_{k,a_1,a_2,3}|=o(n^{-2}),
\end{align*} and then $\mathrm{var}(\sum_{k=1}^n \pi_{n,k}^2)\to 0$ is proved.

\paragraph{Proof of $\sum_{k=1}^n\mathrm{E}(D_{n,k}^4)\to 0$}\label{sec:onesamclt2pf}

Similarly to Section \ref{sec:proofsecondgoaljointnormal}, 
\begin{align*}
\sum_{k=1}^n\mathrm{E}(D_{n,k}^4)=\sum_{k=1}^n \sum_{{1\leq r_1,r_2,r_3,r_4\leq m}}\prod_{l=1}^4 t_{r_l}\times \mathrm{E}\Big(\prod_{l=1}^4 A_{n,k,a_{r_l}}\Big). 	
\end{align*}To prove  $\sum_{k=1}^n\mathrm{E}(D_{n,k}^4)\to 0$, it suffices to show that for given $1\leq k \leq n$ and finite integers $(a_1,a_2,a_3, a_4)$, we have $\mathrm{E}(\prod_{l=1}^4 A_{n,k,a_{l}})=o(n^{-1})$. 

In particular,
\begin{eqnarray*}
\mathrm{E}\Big(\prod_{l=1}^4 A_{n,k,a_{l}}\Big)&=&	\Big\{\prod_{l=1}^4c(n,a_l)\Big\}^{1/2} \sum_{ \substack{ \mathbf{i}^{(l)} \in \mathcal{P}(k-1,a_l-1),\, l=1,\ldots,4; \\ 1\leq j_1,j_2,j_3,j_4\leq p} }\notag \\
&&\mathrm{E}\Big(\prod_{l=1}^4 x_{k,j_l}\Big)\mathrm{E}\Big( \prod_{l=1}^4 \prod_{t=1}^{a_l-1}x_{i_t,j_l}\Big).
\end{eqnarray*}
Similarly to Section \ref{sec:proofsecondgoaljointnormal}, we have $\mathrm{E}( \prod_{l=1}^4 \prod_{t=1}^{a_l-1}x_{i_t,j_l})\neq 0$ only when $|\cup_{l=1}^4\{\mathbf{i}^{(l)}\}|\leq \sum_{l=1}^4(a_l-1)/2$. 
We will prove that 
\begin{align}
\sum_{1\leq j_1,j_2,j_3,j_4\leq p}\mathrm{E}\Big(\prod_{l=1}^4 x_{k,j_l}\Big)=O(p^2). \label{eq:sumjexponsammeanord}	
\end{align}
Then as $c(n,a)=\Theta(p^{-1}n^{-a})$,
\begin{align*}
	\mathrm{E}\Big(\prod_{l=1}^4 A_{n,k,a_{l}}\Big)=O(1)p^{-2}n^{-\sum_{l=1}^4 a_l/2}n^{\sum_{l=1}^4(a_l-1)/2}p^2=o(n^{-1}).
\end{align*}

To finish the proof, it remains to show \eqref{eq:sumjexponsammeanord}. When $|\{j_1,j_2,j_3,j_4\}|\leq 2$,
\begin{align*}
\sum_{1\leq j_1,j_2,j_3,j_4\leq p}\mathrm{E}\Big(\prod_{l=1}^4 x_{k,j_l}\Big)\mathbf{1}_{ \{ |\{j_1,j_2,j_3,j_4\}|\leq 2 \}}	=O(p^2).
\end{align*}When $|\{j_1,j_2,j_3,j_4\}|\geq 3$, we assume without loss of generality that $j_1\leq j_2\leq j_3\leq j_4$. For $\tilde{K}_0$ defined in Section \ref{sec:onesamclt1pf}, if $|j_1-j_2|> \tilde{K}_0$ or  $|j_3-j_4|> \tilde{K}_0$, $|\mathrm{E}(\prod_{l=1}^4 x_{k,j_l})|\leq C\delta^{|j_1-j_2|\epsilon/(2+\epsilon)}=O(p^{-(4+\gamma)})$. If $|j_1-j_2|\leq \tilde{K}_0$ and  $|j_3-j_4|\leq \tilde{K}_0$, but $|j_2-j_3|>K_0$, by Lemma \ref{lm:mixingineq}, 
\begin{align*}
\Big|\mathrm{E}\Big(\prod_{l=1}^4 x_{k,j_l}\Big)\Big|	\leq  \sigma_{j_1,j_2}\sigma_{j_3,j_4}+C\delta^{|j_1-j_2|\epsilon/(2+\epsilon)} =\sigma_{j_1,j_2}\sigma_{j_3,j_4}+ O(p^{-(4+\gamma)}).
\end{align*} Therefore
\begin{align*}
&~\sum_{1\leq j_1,j_2,j_3,j_4\leq p}\mathrm{E}\Big(\prod_{l=1}^4 x_{k,j_l}\Big)\mathbf{1}_{ \{ |\{j_1,j_2,j_3,j_4\}|\geq 3 \}} \notag \\
=&~O(p\tilde{K}_0^3)+O(p^4p^{-(4+\gamma)})+	\sum_{1\leq j_1,j_2,j_3,j_4\leq p}\sigma_{j_1,j_2}\sigma_{j_3,j_4}=O(p^2),	
\end{align*}
where in the last equation, we use Condition \ref{cond:onesamplemean} (2). 
In summary,  \eqref{eq:sumjexponsammeanord} is proved and the proof is finished.

\subsection{Lemmas for the proof of Theorem \ref{thm:twosamplemean}} 
\subsubsection{Proof of Lemma \ref{lm:twosamvar} (on Page \pageref{lm:twosamvar}, Section \ref{sec:proofthm33})} \label{sec:lmtwosamvar}

 Under 	$H_0: \boldsymbol{\mu}=\boldsymbol{\nu}$, we assume $\boldsymbol{\mu}=\boldsymbol{\nu}=\mathbf{0}$ without loss of generality by Proposition \ref{prop:locinvatwosam}. To derive $\mathrm{var}\{\mathcal{U}(a)\}$, we write $	\mathcal{U}(a)=\sum_{j=1}^p \mathcal{U}^{(j)}(a)$, where we define $G(a,c)=(-1)^{a-c}  \binom{a}{c} (P^{n_x}_{c})^{-1} (P^{n_y}_{a-c})^{-1}$, and
 \begin{align}
 	\mathcal{U}^{(j)}(a)=&\sum_{c=0}^{a} G(a,c)\sum_{\substack{\mathbf{k}\in \mathcal{P}(n_x,c),\\ \mathbf{s}\in \mathcal{P}(n_y,a-c) }} \prod_{t=1}^c x_{k_{t},j} \prod_{m=1}^{a-c}  y_{s_{m},j}.\label{eq:equivaujdefitwosam} 
 \end{align} 
Since $\mathrm{E}\{\mathcal{U}(a)\}=0$ under $H_0$, 
\begin{eqnarray}
	&&\mathrm{var}\{\mathcal{U}(a)\}=\mathrm{E}\{\mathcal{U}^2(a)\} =\sum_{1\leq j_1,j_2\leq p}\mathrm{E}\{\mathcal{U}^{(j_1)}(a) \times \mathcal{U}^{(j_2)}(a)\}. \label{eq:varuatwosam}
\end{eqnarray}
Note that for given $1\leq j_1,j_2 \leq p$, 
\begin{align*}
\mathrm{E}\{\mathcal{U}^{(j_1)}(a)\mathcal{U}^{(j_2)}(a)\} =\sum_{\substack{0\leq c\leq a, \\ \mathbf{k}\in \mathcal{P}(n_x,c),\\ \mathbf{s}\in \mathcal{P}(n_y,a-c) }}\,  \sum_{\substack{0\leq \tilde{c}\leq a, \\ \tilde{\mathbf{k}}\in \mathcal{P}(n_x,\tilde{c}),\\ \tilde{\mathbf{s}}\in \mathcal{P}(n_y,a-\tilde{c}) }}G(a,c)G(a,\tilde{c}) Q(\mathbf{k},\mathbf{s},\tilde{\mathbf{k}},\tilde{\mathbf{s}},\mathbf{j}). \notag
\end{align*} where we define
\begin{align*}
	Q(\mathbf{k},\mathbf{s},\tilde{\mathbf{k}},\tilde{\mathbf{s}},\mathbf{j})=\mathrm{E}\Big( \prod_{t=1}^c x_{k_{t},j_1} \prod_{\tilde{t}=1}^{\tilde{c}} x_{\tilde{k}_{\tilde{t}},j_2}\Big) \mathrm{E}\Big(\prod_{m=1}^{a-c}  y_{s_{m},j_1}  \prod_{\tilde{m}=1}^{a-\tilde{c}}  y_{\tilde{s}_{\tilde{m}},j_2} \Big).
\end{align*}
Since we assume the $n=n_x+n_y$ copies are independent from each other and  $\boldsymbol{\mu}=\boldsymbol{\nu}=\mathbf{0}$, 
then $Q(\mathbf{k},\mathbf{s},\tilde{\mathbf{k}},\tilde{\mathbf{s}})= 0$ if $\{\mathbf{k}\}\neq \{\tilde{\mathbf{k}}\}$ or $\{\mathbf{s}\}\neq \{\tilde{\mathbf{s}}\}$. If $\{\mathbf{k}\}=\{\tilde{\mathbf{k}}\}$ and $\{\mathbf{s}\}= \{\tilde{\mathbf{s}}\}$, it induces $c=\tilde{c}$ and $ Q(\mathbf{k},\mathbf{s},\tilde{\mathbf{k}},\tilde{\mathbf{s}},\mathbf{j})=\sigma_{x,j_1,j_2}^c\sigma_{y,j_1,j_2}^{a-c}$. It follows that
\begin{align}
	\mathrm{E}\{\mathcal{U}^{(j_1)}(a)\mathcal{U}^{(j_2)}(a)\} =&~\sum_{c=0}^aG^2(c)P^{n_x}_{c}P^{n_y}_{a-c}c!(a-c)!\sigma_{x,j_1,j_2}^c\sigma_{y,j_1,j_2}^{a-c} \label{eq:covjtwosam}  \\
	=&~a!\sum_{c=0}^a \binom{a}{c} (P^{n_x}_{c})^{-1} (P^{n_y}_{a-c})^{-1}\sigma_{x,j_1,j_2}^c\sigma_{y,j_1,j_2}^{a-c} \notag \\
	\simeq &~ a!\Big(\frac{\sigma_{x,j_1,j_2}}{n_x} + \frac{\sigma_{y,j_1,j_2}}{n_y}\Big)^a. \notag
\end{align}

Combining \eqref{eq:varuatwosam} and \eqref{eq:covjtwosam}, we obtain $\mathrm{var}\{\mathcal{U}(a)\}$. By Condition \ref{cond:twosamplemeancond}, $\mathrm{var}\{\mathcal{U}(a)\}=\Theta(pn^{-a})$.

\subsubsection{Proof of Lemma \ref{lm:twosamcov} (on Page \pageref{lm:twosamcov}, Section \ref{sec:proofthm33})} \label{sec:prooftwosamcov}
Since under $H_0$, $\mathrm{E}\{\mathcal{U}(a)\}=\mathrm{E}\{\mathcal{U}(b)\}=0$, we have $\mathrm{cov}\{\mathcal{U}(a), \mathcal{U}(b)\}=\mathrm{E}\{\mathcal{U}(a)\times \mathcal{U}(b)\}$. Following \eqref{eq:equivaujdefitwosam}, 
\begin{eqnarray}
&&\mathrm{E}\{ \mathcal{U}(a) \times \mathcal{U}(b)\}=
	\sum_{1\leq j_1,j_2\leq p}\mathrm{E}\{\mathcal{U}^{(j_1)}(a) \times \mathcal{U}^{(j_2)}(b)\}, \label{eq:covuatwosam}
\end{eqnarray} where
\begin{align*}
\mathrm{E}\{\mathcal{U}^{(j_1)}(a)\times \mathcal{U}^{(j_2)}(b)\} 
=&~\sum_{\substack{0\leq c\leq a, \\ \mathbf{k}\in \mathcal{P}(n_x,c),\\ \mathbf{s}\in \mathcal{P}(n_y,a-c) }}\, \sum_{\substack{0\leq \tilde{c}\leq b, \\ \tilde{\mathbf{k}}\in \mathcal{P}(n_x,c),\\ \tilde{\mathbf{s}}\in \mathcal{P}(n_y,b-c) }}G(a,c)G(b,\tilde{c}) \notag \\
&~\times \mathrm{E}\Big( \prod_{t=1}^c x_{k_{t},j_1}\prod_{\tilde{t}=1}^{\tilde{c}} x_{\tilde{k}_{\tilde{t}},j_2} \Big)  \mathrm{E}\Big( \prod_{m=1}^{a-c}  y_{s_{m},j_1} \prod_{\tilde{m}=1}^{b-\tilde{c}}  y_{\tilde{s}_{\tilde{m}},j_2} \Big).
\end{align*} 
As $a\neq b$, $\{\mathbf{k}\}\neq \{\tilde{\mathbf{k}}\}$ and $\{\mathbf{s}\}\neq \{\tilde{\mathbf{s}}\}$ always hold. Then as $\boldsymbol{\mu}=\boldsymbol{\nu}=\mathbf{0}$, $\mathrm{E}( \prod_{t=1}^c x_{k_{t},j_1}\prod_{\tilde{t}=1}^{\tilde{c}} x_{\tilde{k}_{\tilde{t}},j_2} )=0$ and  $\mathrm{E}( \prod_{m=1}^{a-c}  y_{s_{m},j_1} \prod_{\tilde{m}=1}^{b-\tilde{c}}  y_{\tilde{s}_{\tilde{m}},j_2} )=0$, similarly to Section \ref{sec:proofcovariancezro}. It follows that $\eqref{eq:covuatwosam}=0$ and the lemma is proved.

\subsubsection{Proof of Lemma \ref{lm:twosamjointnormal} (on Page \pageref{lm:twosamjointnormal}, Section \ref{sec:proofthm33})} \label{sec:prooftwosamjointnormal}
By the  Cram\'er-Wold Theorem, to prove the asymptotic joint normality of the U-statistics, it suffices to prove that any of their fixed converges to normal. For illustration, we first prove the asymptotic normality for each $\mathcal{U}(a)$ of finite $a$. The similar arguments can be applied to the linear combination of finite U-statistics and then the joint normality is obtained. 


Recall $\mathcal{U}(a)=\sum_{j=1}^p \mathcal{U}^{(j)}(a)$ from  \eqref{eq:equivaujdefitwosam}. 
To derive the limiting distribution of $\mathcal{U}(a)$, we use Bernstein's block method  in \cite[][page 338]{ibragimov1971independent}; see also \cite{chen2014two,xu2016adaptive}. Specifically,  we partition the sequence,
	$
		\sigma^{-1}(a) \times \mathcal{U}^{(j)}(a), 
	$ $j=1,\ldots,p$, 
into $r$ blocks, where each block contains $b$ variables such that $rb\leq p <(r+1)b$. For each $1\leq k \leq r$, we partition the $k$th block into two sub-blocks with a larger one $A_{k,1}$ and a smaller one $A_{k,2}$. Suppose each $A_{k,1}$ has $b_1$ variables and each $A_{k,2}$ has $b_2=b-b_1$ variables. We require $r\rightarrow \infty$, $b_1\rightarrow \infty$, $b_2\rightarrow \infty$, $rb_1/p\rightarrow 1$ and $rb_2/p\rightarrow 0$ as $p\rightarrow \infty$. We write
\begin{eqnarray*}
	A_{k,1}(a)=\sum_{i=1}^{b_1}  \mathcal{U}^{(k-1)b+i}(a), \quad
	A_{k,2}(a)=\sum_{i=1}^{b_2}  \mathcal{U}^{(k-1)b+b_1+i}(a),
\end{eqnarray*}
and further define $\mathcal{U}_1=\sigma^{-1}(a) \sum_{k=1}^r A_{k,1}(a)$, $\mathcal{U}_2=\sigma^{-1}(a) \sum_{k=1}^r A_{k,2}(a)$, and $\mathcal{U}_3=\sigma^{-1}(a) \sum_{j=rb+1}^p \mathcal{U}^{(j)}(a)$. 
Thus we have the decomposition:
$
	\sigma^{-1}(a) \times \mathcal{U}(a)=\mathcal{U}_1+\mathcal{U}_2+\mathcal{U}_3.
$

The Bernstein's block method makes $A_{k,1}$ ``almost" independent, thus the study of $\mathcal{U}_1$ may be related to the cases of sums of independent random variables. In addition, since $b_2$ is small compared with $b_1$, we will show that the sums $\mathcal{U}_2$ and $\mathcal{U}_3$ will be small compared with the total sum of variables in the sequence, i.e., $\sigma^{-1}(a) \times \mathcal{U}(a)$. In particular, we first show
\begin{eqnarray*}
	\sigma^{-1}(a)\times \mathcal{U}(a)=\mathcal{U}_1+o_p(1),
\end{eqnarray*} where $o_p(1)$ represents that the remaining term  converges to 0 in probability. Since $\mathrm{E}(\mathcal{U}_2)=\mathrm{E}(\mathcal{U}_3)=0$, it suffices to prove that $\mathrm{var}(\mathcal{U}_2)=\mathrm{var}(\mathcal{U}_3)=o(1)$. 

For $\mathcal{U}_2$, note that $\mathcal{U}_2=\sigma^{-1}(a) \sum_{k=1}^r A_{k,2}(a)$. Then
\begin{align}
	& ~\mathrm{var}({\mathcal{U}_2}) \label{eq:varu2form1} \\
\leq &~ \sigma^{-2}(a) \sum_{\substack{ 1\leq k_1,k_2\leq r; \\  1\leq i_1,i_2\leq b_2} } \left| \mathrm{cov} \Big\{ \mathcal{U}^{((k_1-1)b+b_1+i_1)}(a),\ \mathcal{U}^{((k_2-1)b+b_1+i_2)}(a) \Big\}\right|. \notag 
\end{align}
Recall $\alpha_x(s)$ and $\alpha_y(s)$ in Condition \ref{cond:twosamplemeancond}. 
Define $\alpha(s)=\alpha_x(s)+\alpha_y(s)$, then $\alpha(s)\leq C\delta^s$, where $\delta=\max\{\delta_x, \delta_y\}\in (0,1)$. 
By the  $\alpha$-mixing inequality in Lemma \ref{lm:mixingineq}, 
\begin{align*}
\Big|\mathrm{cov}\left \{n^{a/2}\mathcal{U}^{(i)}(a) , n^{a/2} \mathcal{U}^{(j)}(a)  \right\} \Big|\leq  8\{\alpha(|i-j|)\}^{\frac{\epsilon}{2+\epsilon}} \max_{1\leq j \leq p} \left[\mathrm{E}\left| n^{a/2} \mathcal{U}^{(j)}(a)   \right| ^{2+\epsilon}\right]^{\frac{2}{2+\epsilon}}.
\end{align*}
We take $\epsilon=2$, and by Lemma \ref{lemma:orderfinaltwo} (on Page \pageref{lemma:orderfinaltwo}, Section \ref{sec:prooflmorderfinaltwo}),  we have $
	\max_{1\leq j \leq p} \mathrm{E}\{ n^{a/2} \mathcal{U}^{(j)}(a)   \} ^{2+\epsilon} < \infty.
$
It follows that
\begin{eqnarray}
	&&\left |\mathrm{cov} \left\{ \mathcal{U}^{((k_1-1)b+b_1+i_1)}(a),\mathcal{U}^{((k_2-1)b+b_1+i_2)}(a) \right\}  \right |\label{eq:varu2form2}\\
	&=&n^{-a}\left |\mathrm{cov} \left\{ n^{a/2}  \mathcal{U}^{((k_1-1)b+b_1+i_1)}(a), n^{a/2} \mathcal{U}^{((k_2-1)b+b_1+i_2)}(a) \right\} \right |\notag \\
	&\leq & Cn^{-a} \alpha \left\{ | ((k_1-1)b+b_1+i_1)-((k_2-1)b+b_1+i_2)  | \right\}^{\frac{2}{4}} \notag  \\
	&\leq & Cn^{-a}  \delta^{|k_1b+i_1-k_2b-i_2|/2}. \notag 
\end{eqnarray}
By \eqref{eq:varu2form1}, \eqref{eq:varu2form2} and $\sigma^2(a) =\Theta(pn^{-a})$ from Lemma \ref{lm:twosamvar},
\begin{eqnarray*}
	&& \mathrm{var}( \mathcal{U}_2)\notag \\
	 &\leq & \sigma^{-2}(a)  \sum_{\substack{ 1\leq k_1,k_2\leq r; \\  1\leq i_1,i_2\leq b_2} } \left| \mathrm{cov} \Big\{ \mathcal{U}^{((k_1-1)b+b_1+i_1)}(a),\mathcal{U}^{((k_2-1)b+b_1+i_2)}(a) \Big\}\right| \\
	&\leq& \sigma^{-2} (a)  \sum_{\substack{ 1\leq k_1,k_2\leq r; \\  1\leq i_1,i_2\leq b_2} }  n^{-a}C\delta^{|k_1b+i_1-k_2b-i_2|/2} \\
	&=& O(1) p^{-1}n^a rb_2n^{-a} =O(1) {rb_2}p^{-1},
\end{eqnarray*} which converges to 0 by our construction, i.e., $rb_2/p \rightarrow 0$. This shows that $\mathrm{var}( \mathcal{U}_2)=o(1)$. Next we exmaine 
$
	\mathcal{U}_3=\sigma^{-1}(a) \sum_{j=rb+1}^p \mathcal{U}^{(j)}(a).
$
Similarly,  by Lemmas  \ref{lm:mixingineq}
and \ref{lemma:orderfinaltwo},  and $\epsilon=2$,
\begin{eqnarray*}
	\mathrm{var}( \mathcal{U}_3 ) &=& \sigma^{-2}(a)n^{-a} \sum_{i=rb+1}^{p} \sum_{j=rb+1}^{p} \mathrm{cov} \left \{ n^{a/2}\mathcal{U}^{(i)}(a)  , n^{a/2}\mathcal{U}^{(j)}(a) \right\} \\
	&\leq & O(1)p^{-1} {n^{a}} n^{-a} \sum_{i=rb+1}^{p} \sum_{j=rb+1}^{p} C \alpha \left( \left | i-j \right|  \right)^{\frac{\epsilon}{2+\epsilon}} \\
	&\leq & O(1) p^{-1} \sum_{i=rb+1}^{p} \sum_{j=rb+1}^{p} \delta^{\left|i-j \right|/2}\\
	&\leq & O(1) p^{-1} (p-rb-1) \\
	&\leq & O(1) p^{-1} b.
\end{eqnarray*}
Since $b/p \rightarrow 0$, $\mathrm{var}(\mathcal{U}_3)=o(1)$.

Given $\mathrm{var}( \mathcal{U}_2 ) =o(1)$ and $\mathrm{var}( \mathcal{U}_3 ) =o(1)$ above, next we focus on $\mathcal{U}_1$. By the $\alpha$-mixing assumption in Condition \ref{cond:twosamplemeancond}, and following the similar arguments   in \cite[][page 338]{ibragimov1971independent},  we have for properly chosen $r$ and $b_2$,
\begin{eqnarray*}
	\Big | \mathrm{E} \left \{ \mathrm{exp}(it\mathcal{U}_1 ) \right \} - \prod_{k=1}^r \mathrm{E} \left [ \mathrm{exp} \left \{ it \sigma^{-1}(a) A_{k,1}(a) \right \} \right]  \Big | \leq 16r\alpha(b_2) \rightarrow 0.
\end{eqnarray*}
This suggests there exist  independent random variables $\left \{ \xi_k: k=1,\cdots,r \right \}$ such that $\xi_k$ and $A_{k,1}(a) $ are identically distributed and $\mathcal{U}_1$ has the same asymptotic distribution as $\sigma^{-1}(a) \sum_{k=1}^r \xi_k$. To prove the asymptotic normality of $\sigma^{-1}(a)\mathcal{U}_1$, now it remains to show that central limit theorem holds for $\sigma^{-1}(a)\sum_{k=1}^r \xi_k$. Then we check the Lyapunov condition, i.e.,  check that the moments of $\xi_{k}$ satisfy
\begin{eqnarray} \label{eq:lynaproof}
	s_r^{-4}{\sum_{k=1}^r \mathrm{E} \left \{ \sigma^{-1}(a) |\xi_k| \right\} ^4  } \rightarrow 0,
\end{eqnarray} where we define $ s_r^2=\sum_{k=1}^r \mathrm{var} \{ \sigma^{-1}(a) \xi_k \} $.  By Lemma \ref{lemma:orderfinaltwo},  for even $\epsilon>0$,
\begin{eqnarray}\label{proof:bound}
	\mathrm{M}_{4+\epsilon}:=\mathrm{max}_{1 \leq j \leq p} \left \{ \left \| n^{a/2} \left \{ \mathcal{U}^{(j)} (a)  \right \} \right \|_{4+\epsilon}\right \} < \infty.
\end{eqnarray}
Then by the moment bounds in  \cite[][Theorem 1]{kim1994momentbounds}, 
and the $\alpha$-mixing assumption in Condition \ref{cond:twosamplemeancond}, for $g(2,\epsilon)=\epsilon/(4+\epsilon) $,
\begin{eqnarray*}
    \mathrm{E} \left(\left [ \sum_{j=1}^{b_1} n^{a/2} \left \{\mathcal{U}^{(j)} (a)  \right \}  \right]^4 \right)\leq Cb_1^{2} \Big \{ C +M_{4+\epsilon}^4 \sum_{j=1}^{b_1} j^{2-1} \alpha(j)^{g(2,\epsilon)}      \Big \} 
\end{eqnarray*}
As $\delta \in (0,1)$ and $0<g(2,\epsilon)< 1$, 
\begin{eqnarray*}
	\sum_{j=1}^{\infty} j \alpha (j)^{g(2,\epsilon)} \leqslant  C \sum_{j=1}^{\infty} j \times (\delta^{g(2,\epsilon)})^j <\infty.  
	\end{eqnarray*}
It follows that
\begin{align*}
	\mathrm{E}\left \{ \sigma^{-1}(a)A_{1,1}(a) \right \}^4 =&~ \sigma^{-4}(a) n^{-2a} \mathrm{E} \left [ \sum_{j=1}^{b_1} n^{a/2} \left \{ \mathcal{U}^{(j)}(a) \right \} \right ]^4 \\
	\leq &~ O(1) p^{-2} n ^{2a} n^{-2a} \times b_1^{2} \left \{C +M_{4+\epsilon}^4 \sum_{j=1}^{b_1} j^{2-1} \alpha(j)^{g(2,\epsilon)}  \right \} \\
	=&~ O(1) p^{-2} \times b_1^2.
\end{align*}
Similarly, for other $k>1$, $\mathrm{E}\left \{ \sigma^{-1}(a)A_{k,1}(a) \right \}^4$ have the same bound. Thus,
\begin{eqnarray} \label{eq:numeratoroder}
	\sum_{k=1}^r \sigma^{-4}(a) E|\xi_k|^4 = O(1)rp^{-2}b_1^2.
\end{eqnarray}
In addition, \begin{align*}
	\mathrm{var} \{ \sigma^{-1}(a)\xi_{k}   \} =&~ \sigma^{-2}(a) \mathrm{var} \left\{  \sum_{i=1}^{b_1} \mathcal{U}^{\left( (k-1)b+i \right)} (a)  \right \} \\
	=&~ \sigma^{-2}(a)\sum_{1\leq i_1,i_2\leq b_1}\mathrm{cov} \left\{ \mathcal{U}^{\left( (k-1)b+i_1   \right) } (a),\mathcal{U}^{\left( (k-1)b+i_2   \right) } (a)   \right \} \\
	=&~ \sigma^{-2}(a)\sum_{1\leq i_1,i_2\leq b_1}\eqref{eq:covjtwosam}.
\end{align*}
By Condition \ref{cond:twosamplemeancond} and $rb_1/p\to 1$, we have
\begin{eqnarray} \label{eq:denominatorder}
	s_r^4&=&\Big [\sum_{j=1}^r \mathrm{var} \left \{{\xi_j}/{\sigma(a)} \right\}    \Big ]^2  \\
	&=&  \Theta(1) p^{-2} n^{2a} ( r\times b_1 n^{-a}  )^2 = \Theta(1) p^{-2}r^2 b_1^2 .\notag
\end{eqnarray}
Combine \eqref{eq:numeratoroder} and \eqref{eq:denominatorder},   \eqref{eq:lynaproof} is proved as $r\rightarrow \infty$. 

In summary, for any finite integer $a$, we prove the  asymptotic normality of $\mathcal{U}(a)/\sigma(a)$.   For any linear combination of U-statistics  $Z_n:=\sum_{r=1}^m t_r \mathcal{U}(a_r)/\sigma(a_r)$, we can similarly decompose  $Z_n$ into three parts and apply the analysis above. The similar conclusion holds for finite $m$ and the asymptotic joint normality is obtained by the Cram\'er-Wold Theorem. 


\subsubsection{Proof of Lemma \ref{lemma:orderfinaltwo} (on Page \pageref{lemma:orderfinaltwo}, Section \ref{sec:prooftwosamjointnormal})} \label{sec:prooflmorderfinaltwo}
\begin{lemma} \label{lemma:orderfinaltwo}
For $\forall$ finite even $\omega >0$ any $\forall$ finite integer $a>0$,  
	\begin{eqnarray*}
	\max_{1\leq j \leq p} \mathrm{E}\left\{ n^{a/2} \mathcal{U}^{(j)}(a)   \right\} ^{\omega} < \infty.
\end{eqnarray*}
\end{lemma}

\begin{proof}
Recall the definition of $\mathcal{U}^{(j)}(a)$ in \eqref{eq:equivaujdefitwosam}.
For positive even $\omega$, 
\begin{eqnarray}
	&&\mathrm{E}[\{\mathcal{U}^{(j)}(a)\}^{\omega}] \label{eq:uaomegaexp} \\
&=& \sum_{l=1}^{\omega} \sum_{\substack{ 0\leq c_l \leq a,\\ \mathbf{k}^{(l)}\in \mathcal{P}(n_x,c_l),\\ \mathbf{s}^{(l)}\in \mathcal{P}(n_y,a-c_l) } } G(c_l) \mathrm{E}\Biggr( \prod_{l=1}^{\omega} \prod_{t_l=1}^{c_l}x_{k_{t_l}^{(l)},j} \Biggr)\mathrm{E}\Biggr( \prod_{l=1}^{\omega} \prod_{m_l=1}^{a-c_l}y_{s_{m_l}^{(l)},j} \Biggr).  \notag 
\end{eqnarray}
Define the index tuple $(\mathbf{k}^{(1)},\ldots,\mathbf{k}^{(\omega)})=(k_1^{(1)},\ldots, k_{c_1}^{(1)},\ldots, k_1^{(\omega)},\ldots, k_{c_{\omega}}^{(\omega)})$. When $|\{(\mathbf{k}^{(1)},\ldots,\mathbf{k}^{(\omega)}) \}|>\sum_{l=1}^{\omega}c_l/2$, it means that one of the index appears only once. Suppose index $i\in \{(\mathbf{k}^{(1)},\ldots,\mathbf{k}^{(\omega)}) \}$ only appears once, then under $H_0$, 
\begin{eqnarray}
	&& \mathrm{E}\Biggr( \prod_{l=1}^{\omega} \prod_{t_l=1}^{c_l}x_{k_{t_l}^{(l)},j} \Biggr)=\mathrm{E}(x_{i,j})\times\mathrm{E}(\mathrm{other\ terms})=0.\label{eq:expxtwosammean}
\end{eqnarray}
Thus $\eqref{eq:expxtwosammean}\neq 0$ only when  $|\{(\mathbf{k}^{(1)},\ldots,\mathbf{k}^{(\omega)}) \}|\leq \sum_{l=1}^{\omega}c_l/2$. By the  boundedness of moments in Condition \ref{cond:twosamplemeancond},
\begin{align*}
	\max_{1\leq j \leq p}\ \sum_{\substack{ 0\leq c_l \leq a,\, \mathbf{k}^{(l)}\in \mathcal{P}(n_x,c_l)} }  \mathrm{E}\Biggr( \prod_{l=1}^{\omega} \prod_{t_l=1}^{c_l}x_{k_{t_l}^{(l)},j} \Biggr)=O\left(n_x^{\sum_{l=1}^{\omega}c_l/2}\right). 
\end{align*}
Similarly, we have
\begin{align*}
	\max_{1\leq j \leq p}\ \sum_{\substack{ 0\leq c_l \leq a,\, \mathbf{s}^{(l)}\in \mathcal{P}(n_y,a-c_l)} } \mathrm{E}\Biggr( \prod_{l=1}^{\omega} \prod_{m_l=1}^{a-c_l}y_{s_{m_l}^{(l)},j} \Biggr)=O\left(n_y^{\sum_{l=1}^{\omega}(a-c_l)/2}\right). 
\end{align*}
As $G(a,c)=\Theta(n_x^{-c}n_y^{-(a-c)})$, by \eqref{eq:uaomegaexp}, 
$
\max_{1\leq j\leq p}\mathrm{E}[\{n^{a/2}\mathcal{U}^{(j)}(a)\}^{\omega}] <\infty.
$
\end{proof}

\subsection{Lemmas for the proof of Theorem \ref{THM:TWOSAMPLEMANINF}} \label{sec:condindpmeas}

\subsubsection{Proof of Lemma \ref{lm:condindpmeas}  (on Page \pageref{lm:condindpmeas}, Section \ref{sec:proofasymindptwosam})}

Recall $\mathcal{U}^{(j)}(a)$ defined in \eqref{eq:equivaujdefitwosam}. 
Similarly to $\tilde{\mathcal{U}}_c(a)$, we define $\tilde{\mathcal{U}}_c^{(j)}(a)$ as the sequence of random variables on the conditional probability measure $\tilde{P}$, given the event $n_xn_y\mathcal{U}(\infty)/(n_x+n_y)-\tau_p\leq u$ such that
\begin{align*}
&~\tilde{P}\Big\{\tilde{\mathcal{U}}_c^{(j)}(a)\leq u_j: 1\leq j\leq p\Big\} \notag \\
=&~{P}\Big\{{\mathcal{U}}^{(j)}(a)\leq u_j: 1\leq j\leq p ~\Big| ~ \frac{n_xn_y}{n_x+n_y}\mathcal{U}(\infty)\leq \tau_p+u \Big\}.
\end{align*}
Then $\sigma^{-1}(a)\tilde{\mathcal{U}}_c(a)=\sigma^{-1}(a)\sum_{j=1}^p \tilde{\mathcal{U}}_c^{(j)}(a)$, and
 we prove the asymptotic normality of $\sigma^{-1}(a)\tilde{\mathcal{U}}_c(a)$
similarly to  Section  \ref{sec:prooftwosamjointnormal}. In particular, we partition the sequence
$
	\{\sigma^{-1}(a) \times \tilde{\mathcal{U}}_c^{(j)}(a):  1\leq j\leq p\}
$ into $r$ blocks, where each block contains $b$ variables such that $rb\leq p<(r+1)b$. For each $1\leq k\leq r$, we further partition the $k$th block into  two sub-blocks such that a larger one $\tilde{A}_{k,1}$ contains the first $b_1$ variables and a smaller one $\tilde{A}_{k,2}$ contains the last $b_2=b-b_1$ variables. Similarly, for $1\leq k\leq r$, we write
\begin{eqnarray*}
	\tilde{A}_{k,1}(a)=\sum_{i=1}^{b_1}  \tilde{\mathcal{U}}_c^{(k-1)b+i}(a), \quad
	\tilde{A}_{k,2}(a)=\sum_{i=1}^{b_2}  \tilde{\mathcal{U}}_c^{(k-1)b+b_1+i}(a).
\end{eqnarray*}Correspondingly,  define $\tilde{\mathcal{U}}_1=\sigma^{-1}(a) \sum_{k=1}^r \tilde{A}_{k,1}(a)$, $\tilde{\mathcal{U}}_2=\sigma^{-1}(a) \sum_{k=1}^r \tilde{A}_{k,2}(a)$ and $\tilde{\mathcal{U}}_3=\sigma^{-1}(a) \sum_{j=rb+1}^p \tilde{\mathcal{U}}_c^{(j)}(a)$.
Then we have the decomposition:
$
	\sigma^{-1}(a) \times \tilde{\mathcal{U}}_c(a)=\tilde{\mathcal{U}}_1+\tilde{\mathcal{U}}_2+\tilde{\mathcal{U}}_3.
$ To show that $\sigma^{-1}(a) \times \tilde{\mathcal{U}}_c(a)$ satisfies the central limit theorem, we first show that $\tilde{\mathrm{E}}(\tilde{\mathcal{U}}_2^2)=o(1)$ and $\tilde{\mathrm{E}}(\tilde{\mathcal{U}}_3^2)=o(1)$.
\begin{align*}
	\tilde{\mathrm{E}}(\tilde{\mathcal{U}}_2^2)=&~\sigma^{-2}(a)\tilde{\mathrm{E}}\Big\{ \Big(\sum_{k=1}^r \tilde{A}_{k,2}(a)\Big)^2  \Big\}\notag \\
	\leq &~ \sigma^{-2}(a)\Biggr( \sum_{1\leq k_1,k_2\leq r}\Big[ \tilde{\mathrm{E}}\Big\{ \tilde{A}_{k_1,2}^2(a) \Big\}    \Big]^{1/2} \Big[ \tilde{\mathrm{E}}\Big\{ \tilde{A}_{k_2,2}^2(a) \Big\}    \Big]^{1/2} \Biggr) \notag \\
	\leq &~\sigma^{-2}(a) \Big[P\Big\{ \frac{n_xn_y}{n_x+n_y}\mathcal{U}(\infty)<\tau_p    \Big\}  \Big]^{-1} \notag \\
	&~\times \Biggr( \sum_{1\leq k_1,k_2\leq r}\Big[ {\mathrm{E}}\Big\{ {A}_{k_1,2}^2(a) \Big\}    \Big]^{1/2} \Big[ {\mathrm{E}}\Big\{ {A}_{k_2,2}^2(a) \Big\}    \Big]^{1/2} \Biggr),
\end{align*} where in the last inequality we use the fact that
\begin{align*}
	\tilde{\mathrm{E}}\Big\{ \tilde{A}_{k,2}^2(a) \Big\} =&~ \frac{\mathrm{E}\{  A_{k,2}^2(a)  \mathbf{1}_{\{ n_xn_y\mathcal{U}(\infty)/(n_x+n_y)<\tau_p+u \}} \}}{P\{ n_xn_y\mathcal{U}(\infty)/(n_x+n_y)<\tau_p+u   \}  }\notag \\
	\leq &~\frac{\mathrm{E}\{A_{k,2}^2(a) \}}{ P\{ n_xn_y\mathcal{U}(\infty)/(n_x+n_y)<\tau_p+u \} }.
\end{align*} The upper bound above converges to 0 under the $\alpha$-mixing condition by choosing proper convergence rate $b_2$; see Eq. (18.4.8) of \cite{ibragimov1971independent}. Similarly, we can also show $\tilde{\mathrm{E}}(\tilde{\mathcal{U}}_3^2)=o(1)$. It remains to examine the $\tilde{\mathcal{U}}_1$. 
Define $\alpha(s)$ as the mixing coefficient of $\{(x_{1,j},\ldots,x_{n_x,j},y_{1,j},\ldots, y_{n_y,j}: j=1,\ldots,p)\}$ and define $\tilde{\alpha}(s)$ as the corresponding  mixing coefficient on the conditional probability measure. 
Following a similar argument to that in \cite[][Lemma 2.2]{hsing1995}, we have
\begin{align*}
	\tilde{\alpha}(d)\leq 4\frac{ \max_{1\leq h\leq p-d}P\{ U^0_{h,d}(\infty)>\tau_p+u  \}+\alpha(d) }{ [P\{ n_xn_y\mathcal{U}(\infty)/(n_x+n_y)<\tau_p+u \}]^3  },
\end{align*}  
where    $U_{h,d}^0(\infty)=\max_{ h\leq j\leq h+d} U^{(j)}(\infty)$, $U^{(j)}(\infty)=\sigma_{j,j}^{-1}\times(\bar{x}_j-\bar{y}_j)^2\times n_xn_y/(n_x+n_y)$, and recall  $\tau_p=2\log p-\log \log p$. Since $x_{i,j}$ and $y_{i,j}$ are sub-gaussian random variables by Condition \ref{cond:twosamplemeancond}  \cite[Proposition 2.5.2]{vershynin2018high}, we know 
$
	\sigma_{j,j}^{-1/2}\times (\bar{x}_j-\bar{y}_j)\times \sqrt{n_xn_y}/\sqrt{n_x+n_y}
$ is a sub-gaussian variable with variance 1.  Therefore, $\max_{1\leq h\leq p-d}P\{U_{h,d}^0(\infty)>\tau_p+u\}\leq d\max_{1\leq j\leq p} P\{U^{(j)}(\infty)>\tau_p+u \}\leq Cd\exp\{-(\tau_p+u)/2\}\leq Cdp^{-1}\sqrt{\log p}$.
Then similarly to \cite[page 338]{ibragimov1971independent},  we have
\begin{eqnarray*}
	&&\Big | \tilde{\mathrm{E}} \left \{ \mathrm{exp}(it\tilde{\mathcal{U}}_1 ) \right \} - \prod_{k=1}^r \tilde{\mathrm{E}} \left [ \mathrm{exp} \left \{ it \sigma^{-1}(a) \tilde{A}_{k,1}(a) \right \} \right]  \Big | \notag\\
	 &\leq & 16r\tilde{\alpha}(b_2) \notag \\
	  &\leq & 64r\frac{ \max_{1\leq h\leq p-b_2}P\{ U^0_{h,b_2}(\infty)>\tau_p+u  \}+\alpha(b_2) }{ [P\{ n_xn_y\mathcal{U}(\infty)/(n_x+n_y)<\tau_p+u \}]^3  },
\end{eqnarray*} which converges to 0 for properly chosen $r$ and $b_2$ such that $rb_2\sqrt{\log p}/p\to 0$.
Thus there exist independent $\{\tilde{\xi}_k: k=1,\ldots, r \}$ such that $\tilde{\xi}_k$ and $\tilde{A}_{k1}(a)$ are identically distributed on probability measure $\tilde{P}$. Similarly to \cite[Lemma 2.4, Lemma 2.5]{hsing1995}, we have $\tilde{\mathrm{E}}\{\sigma^{-1}(a) \sum_{k=1}^r \tilde{\xi}_k \}\to 0$ and $\tilde{\mathrm{E}}[\{\sigma^{-1}(a) \sum_{k=1}^r \tilde{\xi}_k \}^2]\to 1$. To show the asymptotic normality on the conditional probability measure, it remains to check the Lyapunov condition that
\begin{eqnarray*} 
	{\sum_{k=1}^r \tilde{\mathrm{E}} \left \{ \sigma^{-1}(a) |\tilde{\xi}_k| \right\} ^4  } \leq \sigma^{-4}(a) \frac{\sum_{k=1}^r\mathrm{E}(\xi_k^4)}{P\{ n_xn_y\mathcal{U}(\infty)/(n_x+n_y)<\tau_p+u \}}  \rightarrow 0,
\end{eqnarray*}  where $\xi_k$ are define same as in Appendix Section \ref{sec:prooftwosamjointnormal}, and the convergence result follows from \eqref{eq:lynaproof}. This implies the asymptotic normality of conditional distribution given $\{ n_xn_y\mathcal{U}(\infty)/(n_x+n_y)<\tau_p+u \}$. Thus we obtain the asymptotic independence between $\mathcal{U}(a)/\sigma(a)$ and $\mathcal{U}(\infty)$.

\subsection{Lemmas for the proof of Theorem \ref{thm:altcltmeantest}}


\subsubsection{Proof of Lemma \ref{lm:twomeanaltvar}}\label{sec:pftwomeanaltvar}
Recall the definitions in \eqref{eq:altmeandeft}.  $T_{a,2}$ is the  summation over $j$ indexes in the set $\{k_0,\ldots,p\}$ such that $\mu_j=\nu_j=0$. Then $\mathrm{E}(T_{a,2})=0$.  Following the argument in Section \ref{sec:lmtwosamvar}, we obtain 
\begin{align*}
	\mathrm{var}(T_{a,2}) \simeq  \sum_{k_0+1 \leq j_1,j_2\leq p} a!\Big(\frac{\sigma_{x,j_1,j_2}}{n_x} + \frac{\sigma_{y,j_1,j_2}}{n_y}\Big)^a.
\end{align*}
Let $\mathcal{V}_{a,j_1,j_2}=\{{\sigma_{x,j_1,j_2}}/\gamma + {\sigma_{y,j_1,j_2}}/(1-\gamma)\}^a$.   By the mixing assumption in Condition \ref{cond:twosamplemeancond} and Lemma \ref{lm:mixingineq}, we know there exist some constants $C$ and $\tilde{\delta}$ such that $|\mathcal{V}_{a,j_1,j_2}|\leq C\tilde{\delta}^{|j_1-j_2|}$. Note that 
\begin{align*}
	&~\Big|\sum_{1\leq j_1,j_2\leq p}\mathcal{V}_{a,j_1,j_2}-  \sum_{k_0+1 \leq j_1,j_2\leq p} \mathcal{V}_{a,j_1,j_2} \Big|\notag \\
= &~\Big|\Big(\sum_{1\leq j_1,j_2\leq k_0}+ \sum_{1\leq j_1\leq k_0, \, k_0+1\leq j_2\leq p}+ \sum_{1\leq j_2\leq k_0, \, k_0+1\leq j_1\leq p} \Big)\mathcal{V}_{a,j_1,j_2}\Big| \notag \\
\leq  &~C \Big(\sum_{1\leq j_1,j_2\leq k_0}+ \sum_{1\leq j_1\leq k_0, \, k_0+1\leq j_2\leq p}+ \sum_{1\leq j_2\leq k_0, \, k_0+1\leq j_1\leq p} \Big) \tilde{\delta}^{|j_1-j_2|}=O(k_0). 
\end{align*}
Since $k_0=o(p)$ and Condition \ref{cond:twosamplemeancond} assumes that $\sum_{1\leq j_1,j_2\leq p}\mathcal{V}_{a,j_1,j_2}=\Theta(p)$, then $\sum_{k_0+1 \leq j_1,j_2\leq p} \mathcal{V}_{a,j_1,j_2} =\Theta(p)$. It follows that $ \mathrm{var}(T_{a,2})=\Theta(p^2n^{-a})$.  


It remains to prove $\mathrm{var}(T_{a,1})=o(pn^{-a})$. Note that $\mathrm{var}(T_{a,1})=\mathrm{E}(T_{a,1}^2)-\{\mathrm{E}(T_{a,1})\}^2$, and $\mathrm{E}(T_{a,1})=k_0\rho^a$. Following the definition in  \eqref{eq:altmeandeft}, 
\begin{align*}
\mathrm{E}(T_{a,1}^2)=& ~\sum_{\substack{1\leq j_1,j_2\leq k_0 }}	\sum_{\substack{0\leq c\leq a, \\ \mathbf{k}\in \mathcal{P}(n_x,c),\\ \mathbf{s}\in \mathcal{P}(n_y,a-c) }}\,  \sum_{\substack{0\leq \tilde{c}\leq a, \\ \tilde{\mathbf{k}}\in \mathcal{P}(n_x,\tilde{c}),\\ \tilde{\mathbf{s}}\in \mathcal{P}(n_y,a-\tilde{c}) }}G(a,c)G(a,\tilde{c}) Q(\mathbf{k},\mathbf{s},\tilde{\mathbf{k}},\tilde{\mathbf{s}},\mathbf{j}),
\end{align*} where similarly to Section \ref{sec:lmtwosamvar}, 
\begin{align*}
	Q(\mathbf{k},\mathbf{s},\tilde{\mathbf{k}},\tilde{\mathbf{s}},\mathbf{j})=\mathrm{E}\Big( \prod_{t=1}^c x_{k_{t},j_1} \prod_{\tilde{t}=1}^{\tilde{c}} x_{\tilde{k}_{\tilde{t}},j_2}\Big) \mathrm{E}\Big(\prod_{m=1}^{a-c}  y_{s_{m},j_1}  \prod_{\tilde{m}=1}^{a-\tilde{c}}  y_{\tilde{s}_{\tilde{m}},j_2} \Big). 
\end{align*}

Since $\mathrm{E}(\mathbf{y})=\boldsymbol{\nu}=\mathbf{0}$, if $\{\mathbf{s}\}\neq \{\tilde{\mathbf{s}}\}$, $Q(\mathbf{k},\mathbf{s},\tilde{\mathbf{k}},\tilde{\mathbf{s}},\mathbf{j}) = 0$. If $\{\mathbf{s}\}= \{\tilde{\mathbf{s}}\}$, it induces  $c=\tilde{c}$. When $\{\mathbf{s}\}= \{\tilde{\mathbf{s}}\}$, let  $b=|\{\mathbf{k}\}\cap\{\tilde{\mathbf{k}}\}|$, then $0\leq b\leq c$,  
\begin{align*}
	\mathrm{E}\{Q(\mathbf{k},\mathbf{s},\tilde{\mathbf{k}},\tilde{\mathbf{s}},\mathbf{j}) \}
=\mu_{j_1}^{c-b} \mu_{j_2}^{c-b} \varphi_{j_1,j_2}^{b}\sigma_{j_1,j_2}^{a-c}=\rho^{2(c-b)}\varphi_{j_1,j_2}^{b}\sigma_{j_1,j_2}^{a-c}, 
\end{align*}and 
\begin{align*}
\mathrm{E}(T_{a,1}^2)=& ~\sum_{\substack{1\leq j_1,j_2\leq k_0 }}	\sum_{\substack{0\leq c\leq a, \\ \mathbf{k}, \tilde{\mathbf{k}}\in \mathcal{P}(n_x,c);\\ \mathbf{s},\tilde{\mathbf{s}}\in \mathcal{P}(n_y,a-c) }}G^2(a,c)\times \rho^{2(c-b)}\varphi_{j_1,j_2}^{b}\sigma_{j_1,j_2}^{a-c}\times \mathbf{1}_{\{ \{\mathbf{s}\}=\{\tilde{\mathbf{s}}\} \}}	. 
\end{align*}
We next decompose $\mathrm{E}(T_{1,a}^2)=G_{t,1,a,1}+G_{t,1,a,2}+G_{t,1,a,3}$, where
\begin{align*}
G_{t,1,a,1}=\sum_{\substack{1\leq j_1,j_2\leq k_0 }}	\sum_{\substack{0\leq c\leq a, \\ \mathbf{k}, \tilde{\mathbf{k}}\in \mathcal{P}(n_x,c);\\ \mathbf{s},\tilde{\mathbf{s}}\in \mathcal{P}(n_y,a-c) }}G^2(a,c) \rho^{2(c-b)}\varphi_{j_1,j_2}^{b}\sigma_{j_1,j_2}^{a-c} \mathbf{1}_{\{ \{\mathbf{s}\}=\{\tilde{\mathbf{s}}\},c=a, b=0 \}},		
\end{align*}
\begin{align*}
	G_{t,1,a,2}=\sum_{\substack{1\leq j_1,j_2\leq k_0 }}	\sum_{\substack{0\leq c\leq a, \\ \mathbf{k}, \tilde{\mathbf{k}}\in \mathcal{P}(n_x,c);\\ \mathbf{s},\tilde{\mathbf{s}}\in \mathcal{P}(n_y,a-c) }}G^2(a,c) \rho^{2(c-b)}\varphi_{j_1,j_2}^{b}\sigma_{j_1,j_2}^{a-c} \mathbf{1}_{\{ \{\mathbf{s}\}=\{\tilde{\mathbf{s}}\},c\leq a-1, b=0 \}},
\end{align*} and 
\begin{align*}
	G_{t,1,a,3}=\sum_{\substack{1\leq j_1,j_2\leq k_0 }}	\sum_{\substack{0\leq c\leq a, \\ \mathbf{k}, \tilde{\mathbf{k}}\in \mathcal{P}(n_x,c);\\ \mathbf{s},\tilde{\mathbf{s}}\in \mathcal{P}(n_y,a-c) }}G^2(a,c) \rho^{2(c-b)}\varphi_{j_1,j_2}^{b}\sigma_{j_1,j_2}^{a-c} \mathbf{1}_{\{ \{\mathbf{s}\}=\{\tilde{\mathbf{s}}\}, 1\leq b\leq c \}}.
\end{align*}
Note that
$
	|\mathrm{var}(T_{a,1})|	\leq |G_{t,1,a,1}-\{\mathrm{E}(T_{a,1})\}^2|+ |G_{t,1,a,2}|+|G_{t,1,a,3}|. 
$ To prove $\mathrm{var}(T_{a,1})=o(pn^{-a})$, we will next show $|G_{t,1,a,1}-\{\mathrm{E}(T_{a,1})\}^2|$, $|G_{t,1,a,2}|$ and $|G_{t,1,a,3}|$ are $o(pn^{-a})$ respectively. 

First, as $\sum_{\mathbf{k}, \tilde{\mathbf{k}}\in \mathcal{P}(n_x,a);\, \mathbf{s},\tilde{\mathbf{s}}\in \mathcal{P}(n_y,a-c) } \mathbf{1}_{\{ \{\mathbf{s}\}=\{\tilde{\mathbf{s}}\},c=a, b=0 \}}=P^{n_x}_{2a}$ and $G(a,a)=(P^{n_x}_a)^{-1}$,
\begin{align*}
	G_{t,1,a,1}=\sum_{\substack{1\leq j_1,j_2\leq k_0 }}	\sum_{\substack{0\leq c\leq a, \\ \mathbf{k}, \tilde{\mathbf{k}}\in \mathcal{P}(n_x,c);\\ \mathbf{s},\tilde{\mathbf{s}}\in \mathcal{P}(n_y,a-c) }}G^2(a,c)\rho^{2a}\mathbf{1}_{\{ \{\mathbf{s}\}=\{\tilde{\mathbf{s}}\},c=a, b=0 \}}=\frac{P^{n_x}_{2a}}{(P^{n_x}_a)^{2}}k_0^2\rho^{2a}.
\end{align*} Then $|G_{t,1,a,1}-\{\mathrm{E}(T_{a,1})\}^2|= o(1)k_0^2n^{-2a}n^{2a}\rho^{2a}=o(pn^{-a})$, where we use  $\mathrm{E}(T_{a,1})=k_0\rho^{a}$. In addition, as  $\sum_{\substack{\mathbf{k}, \tilde{\mathbf{k}}\in \mathcal{P}(n_x,c); \mathbf{s},\tilde{\mathbf{s}}\in \mathcal{P}(n_y,a-c) }} \mathbf{1}_{\{ \{\mathbf{s}\}=\{\tilde{\mathbf{s}}\},c\leq a-1, b=0 \}}=O(n^{2c+a-c})$ and $G(a,c)=\Theta(n^{-a})$, we have
\begin{align*}
	|G_{t,1,a,2}|\leq C\sum_{\substack{1\leq j_1,j_2\leq k_0 }}\sum_{c=0}^{a-1}n^{-(a-c)}\rho^{2c}\sigma_{j_1,j_2}^{a-c}. 
\end{align*} 
Since $\sum_{1\leq j_1,j_2\leq k_0 } \sigma_{j_1,j_2}= O(k_0)$ by Condition \ref{cond:twosamplemeancond} and  Lemma \ref{lm:mixingineq}, we further know $|G_{t,1,a,2}| =\sum_{c=0}^{a-1}O(k_0\rho^{2c} n^{-(a-c)}).$ As $\rho=O(k_0^{-1/a}p^{1/(2a)}n^{-1/2})$ and $k_0=o(p)$, we obtain $|G_{t,1,a,2}|=o(pn^{-a}).$ Moreover, as $G(a,c)=\Theta(n^{-a})$, $\varphi_{j_1,j_2}=\rho^2+\sigma_{j_1,j_2}$, and $\sum_{\substack{\mathbf{k}, \tilde{\mathbf{k}}\in \mathcal{P}(n_x,c); \mathbf{s},\tilde{\mathbf{s}}\in \mathcal{P}(n_y,a-c) }} \mathbf{1}_{\{ \{\mathbf{s}\}=\{\tilde{\mathbf{s}}\}, b\geq 1 \}}=O(n^{2c-b+a-c})$, 
\begin{align*}
	|G_{t,1,a,3}|\leq C\sum_{ \substack{ 0\leq c\leq a,\\ 1\leq b \leq c} }\sum_{1\leq j_1,j_2\leq k_0}n^{-(b+a-c)} \rho^{2(c-b)}(\sigma_{j_1,j_2}+\rho^2)^{b}_{j_1,j_2} \sigma_{j_1,j_2}^{a-c}. 
\end{align*}
For given $c$ and $b$, the maximum order of $\sum_{1\leq j_1,j_2\leq k_0}n^{-(b+a-c)} \rho^{2(c-b)}(\sigma_{j_1,j_2}+\rho^2)^{b}_{j_1,j_2} \sigma_{j_1,j_2}^{a-c}$ is bounded by the following two quantities:
\begin{eqnarray}
	&&\sum_{1\leq j_1,j_2\leq k_0}Cn^{-(b+a-c)}  \rho^{2c}\sigma_{j_1,j_2}^{a-c}, \label{eq:uppertwosamalt1}\\
	&&\sum_{1\leq j_1,j_2\leq k_0} Cn^{-(b+a-c)} \sigma_{j_1,j_2}^{b+a-c}\rho^{2(c-b)}. \label{eq:uppertwosamalt2}
\end{eqnarray}
For \eqref{eq:uppertwosamalt1}, when $c=a$,  $\eqref{eq:uppertwosamalt1}=O(k_0^2n^{-b}\rho^{2a})=o(pn^{-a})$.  When $c\leq a-1$, since $\sum_{1\leq j_1,j_2\leq k_0 } \sigma_{j_1,j_2}= O(k_0)$ by Condition \ref{cond:twosamplemeancond} and Lemma \ref{lm:mixingineq}, then $\eqref{eq:uppertwosamalt1}=O(k_0n^{-(b+a-c)} \rho^{2c})=o(pn^{-a})$. For \eqref{eq:uppertwosamalt2}, as $b\geq 1$, $b+a-c\geq 1$.  Then similarly by Condition \ref{cond:twosamplemeancond} and Lemma \ref{lm:mixingineq}, $\eqref{eq:uppertwosamalt2}=O(k_0n^{-(b+a-c)}\rho^{2(c-b)})=o(pn^{-a})$.

In summary, we obtain $\mathrm{var}(T_{a,1})=o(pn^{-a})=o(1)\mathrm{var}(T_{a,2})$. Then 
\begin{align*}
	\mathrm{var}\{\mathcal{U}(a) \}\simeq \mathrm{var}(T_{a,2}) \simeq  \sum_{k_0+1 \leq j_1,j_2\leq p} a!\Big(\frac{\sigma_{x,j_1,j_2}}{n_x} + \frac{\sigma_{y,j_1,j_2}}{n_y}\Big)^a. 
\end{align*} By the Markov's inequality, $\{T_{a,1}-\mathrm{E}(T_{a,1}) \}/\sigma(a) \xrightarrow{P} 0.$


\subsubsection{Proof of Lemma \ref{lm:twomeanaltcov}}\label{sec:pftwomeanaltcov}

Note that 
\begin{align*}
		\{\sigma(a)\sigma(b)\}^{-1}\mathrm{cov}\{\mathcal{U}(a),\mathcal{U}(b)  \} =\{\sigma(a)\sigma(b)\}^{-1}\times \sum_{1\leq l_1,l_2\leq 2} \mathrm{cov}(T_{a,l_1},T_{b,l_2}). 
\end{align*}
Lemma \ref{lm:twomeanaltvar} suggests that $\mathrm{var}(T_{a,1})=o(1)\sigma^2(a)$.  By the Cauchy-Schwarz inequality, $\{\sigma(a)\sigma(b)\}^{-1}\mathrm{cov}\{\mathcal{U}(a),\mathcal{U}(b)  \}=\{\sigma(a)\sigma(b)\}^{-1}\mathrm{cov}(T_{a,2},T_{b,2})+o(1)$. To finish the proof, it suffices to show $\mathrm{cov}(T_{a,2},T_{b,2})=0$. Note that $T_{a,2}$ and $T_{b,2}$ are summation over $j$ indexes in the set $\{k_0,\ldots,p\}$ such that $\mu_j=\nu_j=0$. Then the proof in Section \ref{sec:prooftwosamcov} applies similarly and we have $\mathrm{cov}(T_{a,2},T_{b,2})=0$.

\subsection{Lemmas for the proof of Theorem \ref{thm:twosamnull}}

\subsubsection{Proof of Lemma \ref{lm:twosampvartest} (on Page \pageref{lm:twosampvartest}, Section \ref{sec:firstpfthmtwoclt})}\label{sec:pftwosampvartest}
In the following, we will first derive the form of $\mathrm{var}\{\tilde{\mathcal{U}}(a)\} $ and then prove that $\mathrm{var}\{\tilde{\mathcal{U}}(a)\} =o(1)\mathrm{var}\{\tilde{\mathcal{U}}^*(a)\}$.

As we assume $\mathrm{E}(\mathbf{x})=\mathrm{E}(\mathbf{y})=\mathbf{0}$, then $\mathrm{cov}(x_{1,j_1},x_{1,j_2})=\mathrm{E}(x_{1,j_1}x_{1,j_2})$ and $\mathrm{cov}(y_{1,j_1},y_{1,j_2})=\mathrm{E}(y_{1,j_1}y_{1,j_2})$. It follows that $\mathrm{E}\{\tilde{\mathcal{U}}(a)\}=0$ and  $\mathrm{var}\{\tilde{\mathcal{U}}(a)\} = \mathrm{E}\{\tilde{\mathcal{U}}^2(a)\}$. By definition, 
\begin{eqnarray*}
	\tilde{\mathcal{U}}(a)=(P^{n_x}_{a}P^{n_y}_{a})^{-1}\sum_{1\leq j_1,j_2\leq p} \sum_{\substack{\mathbf{i}\in \mathcal{P}(n_x,a);\\ \mathbf{w}\in \mathcal{P}(n_y,a)  } }\ \mathbb{D}_{\mathbf{x},\mathbf{y}}(\mathbf{i},\mathbf{w},j_1,j_2), 
\end{eqnarray*}
where we define  $\mathbb{D}_{\mathbf{x},\mathbf{y}}(\mathbf{i},\mathbf{w},j_1,j_2)=\prod_{t=1}^a  (x_{i_{t},j_1}x_{i_{t},j_2} -y_{w_{t},j_1}y_{w_{t},j_2})$.  Then
\begin{align*}
\mathrm{var}\{\tilde{\mathcal{U}}(a)\}=\frac{1}{(P^{n_x}_{a}P^{n_y}_{a})^2} \sum_{\substack{1\leq j_1,j_2,j_3,j_4 \leq p;\\ \mathbf{i},\, \tilde{\mathbf{i}}\in \mathcal{P}(n_x,a);\\ \mathbf{w},\, \tilde{\mathbf{w}}\in \mathcal{P}(n_y,a)  } } \mathrm{E}\Big\{\mathbb{D}_{\mathbf{x},\mathbf{y}}(\mathbf{i},\mathbf{w},j_1,j_2)\mathbb{D}_{\mathbf{x},\mathbf{y}}(\tilde{\mathbf{i}},\tilde{\mathbf{w}},j_3,j_4)    \Big\}.
\end{align*}
Under $H_0$, $\Sigma_{x}=\Sigma_{y}=\Sigma =(\sigma_{j_1,j_2})_{p\times p}$, then $\mathrm{E}(x_{1,j_1}x_{1,j_2}-\sigma_{j_1,j_2})=0$ and $\mathrm{E}(y_{1,j_1}y_{1,j_2}-\sigma_{j_1,j_2})=0$. If  $|\{\mathbf{i}\} \cap \{\tilde{\mathbf{i}}\}|+|\{\mathbf{w}\} \cap \{\tilde{\mathbf{w}}\}|<a,$ it means that the common indexes between $(\mathbf{i},\mathbf{w})$ and  $(\tilde{\mathbf{i}}, \tilde{\mathbf{w}})$ is smaller than $a$, then we know $\mathrm{E}\{\mathbb{D}_{\mathbf{x},\mathbf{y}}(\mathbf{i},\mathbf{w},j_1,j_2)\mathbb{D}_{\mathbf{x},\mathbf{y}}(\tilde{\mathbf{i}},\tilde{\mathbf{w}},j_3,j_4) \}=0.$ If  $|\{\mathbf{i}\} \cap \{\tilde{\mathbf{i}}\}|+|\{\mathbf{w}\} \cap \{\tilde{\mathbf{w}}\}|\geq a,$ we know $\mathrm{E}\{\mathbb{D}_{\mathbf{x},\mathbf{y}}(\mathbf{i},\mathbf{w},j_1,j_2)\mathbb{D}_{\mathbf{x},\mathbf{y}}(\tilde{\mathbf{i}},\tilde{\mathbf{w}},j_3,j_4) \}$ is a linear combination of  $(\mathbf{X}_{j_1,j_2,j_3,j_4})^{m}(\mathbf{Y}_{j_1,j_2,j_3,j_4})^{a-m}$, where $a- |\{\mathbf{w}\} \cap \{\tilde{\mathbf{w}}\}| \leq  m \leq |\{\mathbf{i}\} \cap \{\tilde{\mathbf{i}}\}| $ and 
\begin{eqnarray*}
	\mathbf{X}_{j_1,j_2,j_3,j_4}&=&\mathrm{E}\{(x_{1,j_1}x_{1,j_2}-\sigma_{j_1,j_2})(x_{1,j_3}x_{1,j_4}-\sigma_{j_3,j_4})\}, \notag \\
	\mathbf{Y}_{j_1,j_2,j_3,j_4}&=&\mathrm{E}\{(y_{1,j_1}y_{1,j_2}-\sigma_{j_1,j_2})(y_{1,j_3}y_{1,j_4}-\sigma_{j_3,j_4})\}.	
\end{eqnarray*} 
And if $|\{\mathbf{i}\} \cap \{\tilde{\mathbf{i}}\}|+|\{\mathbf{w}\} \cap \{\tilde{\mathbf{w}}\}|=t_0,$ 
\begin{align*}
\sum_{\substack{ \mathbf{i},\, \tilde{\mathbf{i}}\in \mathcal{P}(n_x,a);\\ \mathbf{w},\, \tilde{\mathbf{w}}\in \mathcal{P}(n_y,a)  } }	\mathbf{1}_{\{|\{\mathbf{i}\} \cap \{\tilde{\mathbf{i}}\}|+|\{\mathbf{w}\} \cap \{\tilde{\mathbf{w}}\}|=t_0 \} }= O(n^{4a-t_0}),
\end{align*} which achieves the largest order at $t_0=a$ when $t_0\geq a$. Therefore,
\begin{eqnarray*}
\mathrm{var}\{\tilde{\mathcal{U}}(a)\} &\simeq &\frac{1}{(P^{n_x}_{a}P^{n_y}_{a})^2} \sum_{\substack{1\leq j_1,j_2,j_3,j_4 \leq p;\\ \mathbf{i},\, \tilde{\mathbf{i}}\in \mathcal{P}(n_x,a);\\ \mathbf{w},\, \tilde{\mathbf{w}}\in \mathcal{P}(n_y,a)  } }\mathbf{1}_{\{|\{\mathbf{i}\} \cap \{\tilde{\mathbf{i}}\}|+|\{\mathbf{w}\} \cap \{\tilde{\mathbf{w}}\}|=a \} } \notag \\
	&&\times \mathrm{E}\Big\{\mathbb{D}_{\mathbf{x},\mathbf{y}}(\mathbf{i},\mathbf{w},j_1,j_2)\mathbb{D}_{\mathbf{x},\mathbf{y}}(\tilde{\mathbf{i}},\tilde{\mathbf{w}},j_3,j_4)    \Big\}.
\end{eqnarray*}
It follows that
\begin{eqnarray}
	&&\mathrm{var}\{\tilde{\mathcal{U}}(a)\}\label{eq:varleadtwosam} \\
	&\simeq &\sum_{1\leq j_1,j_2,j_3,j_4\leq p} \sum_{m=0}^a \frac{P^{n_x}_{2a-m} P^{n_y}_{a+m}}{(P^{n_x}_aP^{n_y}_a)^2}\binom{a}{m}^2\binom{a-m}{a-m}^2 \notag \\ 
	&& \quad \quad \times m!(a-m)! (\mathbf{X}_{j_1,j_2,j_3,j_4})^{m}(\mathbf{Y}_{j_1,j_2,j_3,j_4})^{a-m},  \notag 
\end{eqnarray} and then $\eqref{eq:varleadtwosam}\simeq \sum_{1\leq j_1,j_2,j_3,j_4\leq p} a!(\mathbf{X}_{j_1,j_2,j_3,j_4}/n_x+ \mathbf{Y}_{j_1,j_2,j_3,j_4}/n_y )^a.$	 


We next prove $\mathrm{var}\{\tilde{\mathcal{U}}(a)\} =o(1)\mathrm{var}\{\tilde{\mathcal{U}}^*(a)\}$ under Conditions \ref{cond:twosamplecov1} and  \ref{cond:twosamplecov2} in the following Sections \ref{sec:twosamcovsmallord1}  and \ref{sec:twosamcovsmallord2}  respectively.  

\paragraph{Under Condition \ref{cond:twosamplecov1}} \label{sec:twosamcovsmallord1}

To prove $\mathrm{var}\{\tilde{\mathcal{U}}(a)\} =o(1)\mathrm{var}\{\tilde{\mathcal{U}}^*(a)\}$ under Condition  \ref{cond:twosamplecov1}, we will first show $\mathrm{var}\{\tilde{\mathcal{U}}(a)\}=\Theta(p^2n^{-a})$. Note that ${P^{n_x}_{2a-m} P^{n_y}_{a+m}}/(P^{n_x}_aP^{n_y}_a)^2 \simeq Cn^{a}$. By \eqref{eq:varleadtwosam}, it remains to show that for any $m\in \{0,1,\ldots,a\}$, 
\begin{align}
	\sum_{ {1\leq j_1,j_2,j_3,j_4 \leq p} }  (\mathbf{X}_{j_1,j_2,j_3,j_4})^{m} (\mathbf{Y}_{j_1,j_2,j_3,j_4})^{a-m}=\Theta(p^2). \label{eq:twopsqorderxy}
\end{align}
We next prove \eqref{eq:twopsqorderxy} by discussing different cases of $\{j_1,j_2,j_3,j_4\}$, and using $K_0 = {-(2+\epsilon)(8+2\mu) (\log p)}/{(\epsilon \log \delta)}$ similarly to  \eqref{eq:thresholdd}, where $\epsilon$ and $\mu$ are positive constants and $\delta=\max\{\delta_x, \delta_y\}$ from Condition \ref{cond:twosamplecov1}.   

\medskip
\textbf{Case 1:} If $|j_1-j_2|\leq K_0$ and $|j_3-j_4|\leq K_0$,  we define  a distance $\kappa_d=\min\{ |j_1-j_3|, |j_1-j_4|, |j_2-j_3|, |j_2-j_4| \} $,  and discuss when $\kappa_d > K_0$ and $\kappa_d \leq K_0$ respectively.  For the simplicity of notation, define two indicator functions $I_1=\mathbf{1}_{\{|j_1-j_2|\leq K_0, |j_3-j_4|\leq K_0 , \kappa_d> K_0 \} }$ and $I_2=\mathbf{1}_{\{|j_1-j_2|\leq K_0, |j_3-j_4|\leq K_0 , \kappa_d\leq K_0 \} }$.
By definition, we have  $ \mathbf{X}_{j_1,j_2,j_3,j_4}= \mathrm{cov}(x_{1,j_1}x_{1,j_2},x_{1,j_3}x_{1,j_4})$ and $ \mathbf{Y}_{j_1,j_2,j_3,j_4}= \mathrm{cov}(y_{1,j_1}y_{1,j_2},y_{1,j_3}y_{1,j_4})$. When $\kappa_d > K_0$, we know $\mathbf{X}_{j_1,j_2,j_3,j_4} \leq  C\delta^{\frac{K_0\epsilon}{2+\epsilon}}$ by Condition \ref{cond:twosamplecov1} (2) and  (3) and Lemma \ref{lm:mixingineq}. It follows that
\begin{align}
	& \Big|\sum_{ {1\leq j_1,j_2,j_3,j_4 \leq p} }  (\mathbf{X}_{j_1,j_2,j_3,j_4})^{m} (\mathbf{Y}_{j_1,j_2,j_3,j_4})^{a-m}\times  I_1 \Big|	\label{eq:case1discp4twosam} \\
	 \leq  &Cp^4\delta^{\frac{K_0\epsilon}{2+\epsilon}}=O(1)p^4\times p^{-(8+2\mu)}=o(1). \notag
\end{align} In addition, note that $\sum_{1\leq j_1,j_2,j_3,j_4 \leq p}I_2= O(pK_0^3)=O(p\log^3p)$. By Condition \ref{cond:twosamplecov1} (2), we know
\begin{align*}
	 \Big|\sum_{ {1\leq j_1,j_2,j_3,j_4 \leq p} }  (\mathbf{X}_{j_1,j_2,j_3,j_4})^{m} (\mathbf{Y}_{j_1,j_2,j_3,j_4})^{a-m}\times  I_2 \Big|	=O(p\log^3p).
\end{align*}
\medskip

\textbf{Case 2:} If $|j_1-j_2|>K_0$ or $|j_3-j_4|> K_0$, by Lemma \ref{lm:mixingineq}, we know that $|\sigma_{j_1,j_2}\sigma_{j_3,j_4}|\leq C\delta^{\frac{K_0\epsilon}{2+\epsilon}}$. We consider $|j_1-j_2|>K_0$ without loss of generality and discuss the following cases (i)--(iv). 

\smallskip

\textbf{(i)} When $|j_2-j_3|>K_0/2$ and $|j_2-j_4|>K_0/2$, 
 \begin{align*}
 	|\mathbf{X}_{j_1,j_2,j_3,j_4}|=&| \mathrm{cov}(x_{1,j_1}x_{1,j_3}x_{1,j_4} \, , \, x_{1,j_2})-\sigma_{j_1,j_2}\sigma_{j_3,j_4}|
 	\leq  C\delta^{\frac{K_0\epsilon}{2(2+\epsilon)}}.
 \end{align*} 
 
\textbf{(ii)} When  $|j_2-j_3|\leq K_0/2$ and $|j_2-j_4|\leq K_0/2$, we know that $|j_1-j_3|\geq |j_1-j_2|-|j_2-j_3|>K_0/2$ and $|j_1-j_4|\geq |j_1-j_2|-|j_2-j_4|>K_0/2$. Then
\begin{align}
	|\mathbf{X}_{j_1,j_2,j_3,j_4}|=&| \mathrm{cov}(x_{1,j_1}\, ,\, x_{1,j_2}x_{1,j_3}x_{1,j_4})-\sigma_{j_1,j_2}\sigma_{j_3,j_4}|\leq C\delta^{\frac{K_0\epsilon}{2(2+\epsilon)}}.  \label{eq:xj4upbdtwosam2}
\end{align}

\textbf{(iii)} When  $|j_2-j_3|\leq K_0/2$ and $|j_2-j_4|> K_0/2$, as we know $|j_1-j_2|>K_0$, then $|j_1-j_3|>K_0/2$.  We next discuss three sub-cases.

\smallskip

\textbf{(iiia)} If $|j_1-j_4|> K_0/2$, we know \eqref{eq:xj4upbdtwosam2} also holds.

\smallskip

For easy presentation, let $I_3$ be an indicator function when $\{j_1,j_2,j_3,j_4\}$ satisfies the sub-cases (i), (ii) and (iiia) above. Then similarly to \eqref{eq:case1discp4twosam}, 
\begin{align*}
 \Big|\sum_{ {1\leq j_1,j_2,j_3,j_4 \leq p} }  (\mathbf{X}_{j_1,j_2,j_3,j_4})^{m} (\mathbf{Y}_{j_1,j_2,j_3,j_4})^{a-m}\times  I_3 \Big|	=o(1).\notag	
\end{align*}

\textbf{(iiib)} If $|j_1-j_4|\leq K_0/2$, and $|j_3-j_4|\leq K_0/2$, we  know under this case $|j_2-j_3|, |j_1-j_4|, |j_3-j_4|\leq K_0$. Let $I_4=\mathbf{1}_{ \{ |j_2-j_3|, |j_1-j_4|, |j_3-j_4|\leq K_0 \} }$. We have $\sum_{1\leq j_1,j_2,j_3,j_4\leq p}I_4=O(pK_0^3).$ By Condition \ref{cond:twosamplecov1} (2), we know
\begin{align*}
	 \Big|\sum_{ {1\leq j_1,j_2,j_3,j_4 \leq p} }  (\mathbf{X}_{j_1,j_2,j_3,j_4})^{m} (\mathbf{Y}_{j_1,j_2,j_3,j_4})^{a-m}\times  I_4 \Big|	=O(p\log^3p).
\end{align*}

\smallskip

\textbf{(iiic)} If $|j_1-j_4|\leq K_0/2$, and $|j_3-j_4|> K_0/2$, we know
\begin{align*}
	\mathbf{X}_{j_1,j_2,j_3,j_4}\geq  \mathrm{E}(x_{1,j_1}x_{1,j_4}) \mathrm{E}(x_{1,j_2}x_{1,j_3})-C\delta^{\frac{K_0\epsilon}{2(2+\epsilon)}}. \notag
\end{align*}Let $I_5$ be an indicator function of the sub-case (iiic) above. Then
\begin{align*}
	& ~\Big|\sum_{ {1\leq j_1,j_2,j_3,j_4 \leq p} }  (\mathbf{X}_{j_1,j_2,j_3,j_4})^{m} (\mathbf{Y}_{j_1,j_2,j_3,j_4})^{a-m}\times  I_5 \Big| \notag \\
=&~ \Big| \sum_{ {1\leq j_1,j_2,j_3,j_4 \leq p} }(\sigma_{j_1,j_4}\sigma_{j_2,j_3})^a\times I_5\Big|  + O(p^4 p^{-(4+\mu)}) \notag \\
 =&~ \Big|\sum_{|j_1-j_4|\leq K_0/2,\,|j_2-j_3|\leq K_0/2} (\sigma_{j_1,j_4}\sigma_{j_2,j_3})^a\Big|+o(1) \notag \\
= &~ \Big| \sum_{ {1\leq j_1,j_2,j_3,j_4 \leq p} }(\sigma_{j_1,j_4}\sigma_{j_2,j_3})^a - \sum_{ |j_1-j_4|> K_0\, \mathrm{ or }\, |j_2-j_3|> K_0 }(\sigma_{j_1,j_4}\sigma_{j_2,j_3})^a \Big| +o(1) \notag \\
=&~ \Theta(p^2).
\end{align*} where the last equation uses  Conditions \ref{cond:twosamplecov1} (3) and (4) and Lemma \ref{lm:mixingineq}.

\medskip
\textbf{(iv)} When  $|j_2-j_3|> K_0/2$ and $|j_2-j_4|\leq K_0/2$, this is symmetric to the sub-case (iii) discussed above.  Define an  indicator function $I_6=\mathbf{1}_{ \{|j_2-j_3|> K_0/2,  |j_2-j_4|\leq K_0/2\} }$. We then have
\begin{align*}
	& \Big|\sum_{ {1\leq j_1,j_2,j_3,j_4 \leq p} }  (\mathbf{X}_{j_1,j_2,j_3,j_4})^{m} (\mathbf{Y}_{j_1,j_2,j_3,j_4})^{a-m}\times  I_6 \Big|=\Theta(p^2). \notag 
\end{align*}

In summary,  \eqref{eq:twopsqorderxy} is proved and thus $\mathrm{var}\{\tilde{\mathcal{U}}(a)\}=\Theta(p^2 n^{-a})$ is obtained. To prove $\mathrm{var}\{\tilde{\mathcal{U}}(a)\} =o(1)\mathrm{var}\{\tilde{\mathcal{U}}^*(a)\}$, it remains to show that $\mathrm{var}\{\tilde{\mathcal{U}}^*(a)\}=o(p^2n^{-a})$.

We write $
	\mathcal{U}(a)=\sum_{c=0}^a \sum_{b_1=0}^c \sum_{b_2=0}^{a-c} C_{a,c,b_1,b_2}T_{b_1,b_2,c},
$ where we define $C_{a,c,b_1,b_2}= (-1)^{c-b_1+b_2} a!/\{b_1!b_2!(c-b_1)!(a-c-b_2)!\}$, and 
\begin{align}
	T_{b_1,b_2,c}=& \sum_{\substack{1\leq j_1, j_2 \leq p}} \ \sum_{\substack{\mathbf{i}\in \mathcal{P}(n_x,2c-b_1);\\ \mathbf{w}\in \mathcal{P}(n_y,2(a-c)-b_2) }}  (P^{n_x}_{2c-b_1}P^{n_y}_{2(a-c)-b_2})^{-1}\label{eq:tb1b2c} \\
	& \times \prod_{k=1}^{b_1}  (x_{i_k, j_1}x_{i_k,  j_2}-\sigma_{j_1,j_2} ) \prod_{k=b_1+1}^{c}  x_{i_{k}, j_1} \prod_{k=c+1}^{2c-b_1} x_{i_{k}, j_2}  \notag \\
	& \times \prod_{m=1}^{b_2}  (y_{{w}_m, j_1}y_{{w}_m,  j_2}-\sigma_{j_1,j_2}) \prod_{l=b_2+1}^{a-c} y_{{w}_{l}, j_1} \prod_{q=a-c+1}^{2(a-c)-b_2}y_{{w}_{q}, j_2}.  \notag 
\end{align}
Then  $\tilde{\mathcal{U}}(a)=\sum_{c=0}^a(-1)^{a-c}T_{c,a-c,c}$ and $\tilde{\mathcal{U}}^*(a)=\sum_{c=0}^a \sum_{b_1=0}^c \sum_{b_2=0}^{a-c} C_{a,c,b_1,b_2}\times T_{b_1,b_2,c}  \mathbf{1}_{b_1+b_2\leq a-1}$. Note that $\mathrm{var}\{ \tilde{\mathcal{U}}^*(a)\}\leq C\max_{b_1,b_2,c; b_1+b_2\leq a-1} \{\mathrm{var}( T_{b_1,b_2,c})\},$ where $C$ is some constant. When $a$ is finite, to prove $\mathrm{var}\{\tilde{\mathcal{U}}^*(a)\}=o(p^2n^{-a})$, it suffices to show that $\mathrm{var}( T_{b_1,b_2,c})=o(p^2n^{-a})$ for each $(b_1,b_2,c)$ satisfying $b_1+b_2\leq a-1$.  Note that $\mathrm{E}( T_{b_1,b_2,c})=0$ under $H_0$, then $\mathrm{var}( T_{b_1,b_2,c})=\mathrm{E}( T_{b_1,b_2,c}^2) $  and 
\begin{align}
	\mathrm{var}( T_{b_1,b_2,c})	=&~(P^{n_x}_{2c-b_1}P^{n_y}_{2(a-c)-b_2})^{-2}\sum_{\substack{1\leq j_1, j_2 \leq p; \\1\leq \tilde{j}_1, \tilde{j}_2 \leq p}} \ \sum_{\substack{\mathbf{i},\, \tilde{\mathbf{i}}\in \mathcal{P}(n_x,2c-b_1);\\ \mathbf{w}\, \tilde{\mathbf{w}}\in \mathcal{P}(n_y,2(a-c)-b_2) }}   \label{eq:vartb1b2c}  \\
	&~\quad \mathbb{T}(\mathbf{i},\tilde{\mathbf{i}},\mathbf{w},\tilde{\mathbf{w}},j_1,j_2,\tilde{j}_1,\tilde{j}_2),  \notag 
\end{align}
where we let
\begin{align*}
&~\mathbb{T}(\mathbf{i},\tilde{\mathbf{i}},\mathbf{w},\tilde{\mathbf{w}},j_1,j_2,\tilde{j}_1,\tilde{j}_2) \notag \\
=&~\mathrm{E}\Big\{ \prod_{k=1}^{b_1}  (x_{i_k, j_1}x_{i_k,  j_2} -\sigma_{j_1,j_2})( x_{\tilde{i}_k, \tilde{j}_1}x_{\tilde{i}_k,  \tilde{j}_2} -\sigma_{\tilde{j}_1,\tilde{j}_2 } )\prod_{k=b_1+1}^{c}  (x_{i_{k}, j_1} x_{\tilde{i}_{k}, \tilde{j}_1} ) \notag \\
&~ \times \prod_{k=c+1}^{2c-b_1} (x_{i_{k}, j_2}x_{\tilde{i}_{k}, \tilde{j}_2})\Big\}  \mathrm{E}\Big\{ \prod_{m=1}^{b_2}  (y_{{w}_m, j_1}y_{{w}_m,  j_2} -\sigma_{j_1,j_2})(y_{\tilde{w}_m, \tilde{j}_1}y_{\tilde{w}_m,  \tilde{j}_2} -\sigma_{j_1,j_2}) \notag \\
&~\times \prod_{m=b_2+1}^{a-c} (y_{{w}_{m}, j_1}y_{\tilde{w}_{m}, \tilde{j}_1} )\prod_{m=a-c+1}^{2(a-c)-b_2}(y_{{w}_{m}, j_2}y_{\tilde{w}_{m}, \tilde{j}_2})\Big\}.
\end{align*}
Since we assume without loss of generality that $\mathrm{E}(\mathbf{x})=\mathrm{E}(\mathbf{y})=\mathbf{0}$, then $\mathrm{E}(x_{1,j_1}x_{1,j_2}- \sigma_{j_1,j_2})=\mathrm{E}(y_{1,j_1}x_{1,j_2}- \sigma_{j_1,j_2})=0$. It follows that when $\{\mathbf{i}\}\neq \{\tilde{\mathbf{i}}\}$ or $\{\mathbf{w}\}\neq \{\tilde{\mathbf{w}}\}$, $\mathbb{T}(\mathbf{i},\tilde{\mathbf{i}},\mathbf{w},\tilde{\mathbf{w}},j_1,j_2,\tilde{j}_1,\tilde{j}_2)=0$. When  $\{\mathbf{i}\}= \{\tilde{\mathbf{i}}\}$ and $\{\mathbf{w}\}= \{\tilde{\mathbf{w}}\}$, we have $|\{\mathbf{i}\}\cup \{\tilde{\mathbf{i}}\} |+|\{\mathbf{w}\}\cup \{\tilde{\mathbf{w}}\}|= 2c-b_1+2(a-c)-b_2$. By Condition \ref{cond:twosamplecov1} (1) and (2), for any given $\{j_1,j_2,\tilde{j}_1,\tilde{j}_2\},$
\begin{align}
&~(P^{n_x}_{2c-b_1}P^{n_y}_{2(a-c)-b_2})^{-2} \ \sum_{\substack{\mathbf{i},\, \tilde{\mathbf{i}}\in \mathcal{P}(n_x,2c-b_1);\\ \mathbf{w}\, \tilde{\mathbf{w}}\in \mathcal{P}(n_y,2(a-c)-b_2) }}\mathbb{T}(\mathbf{i},\tilde{\mathbf{i}},\mathbf{w},\tilde{\mathbf{w}},j_1,j_2,\tilde{j}_1,\tilde{j}_2) \label{eq:varsmallordernsumtwosam1} \\
=&~O(n^{-2(2a+b_1+b_2)}\times n^{2a-b_1-b_2})=O(n^{-2a+b_1+b_2})=o(n^{-a-1}) \notag 
\end{align} where in the last equation, we use $b_1+b_2\leq a-1$. 
In addition, similarly to \eqref{eq:twopsqorderxy},  we have that for any given $(\mathbf{i},\tilde{\mathbf{i}},\mathbf{w},\tilde{\mathbf{w}}),$
\begin{align} \label{eq:varsmallorderpsumtwosam2}
 \sum_{1\leq j_1,j_2,\tilde{j}_1,\tilde{j}_2\leq p}	\mathbb{T}(\mathbf{i},\tilde{\mathbf{i}},\mathbf{w},\tilde{\mathbf{w}},j_1,j_2,\tilde{j}_1,\tilde{j}_2) = O(p^2).
\end{align}

In summary, by  \eqref{eq:varsmallordernsumtwosam1} and \eqref{eq:varsmallorderpsumtwosam2}, we know
$
	\mathrm{var}\{ \tilde{\mathcal{U}}^*(a) \} =O(p^2 n^{-a-1})=o(p^2n^{-a}). 
$

\paragraph{Under Condition \ref{cond:twosamplecov2}} \label{sec:twosamcovsmallord2}

In this section, we prove that $\mathrm{var}\{\tilde{\mathcal{U}}(a)\} =o(1)\mathrm{var}\{\tilde{\mathcal{U}}^*(a)\}$  under Condition \ref{cond:twosamplecov2}. Recall that we have already obtained $\mathrm{var}\{\tilde{\mathcal{U}}(a)\} $ in \eqref{eq:varleadtwosam}.   
By Condition \ref{cond:twosamplecov2} (3), we have
\begin{eqnarray}
	&&\quad \mathbf{X}_{j_1,j_2,j_3,j_4}\label{eq:twosamcovexpterm4th} 
	= \kappa_x( \sigma_{j_1,j_3}\sigma_{j_2,j_4}+\sigma_{j_1,j_4}\sigma_{j_2,j_3})+(\kappa_x-1)\sigma_{j_1,j_2}\sigma_{j_3,j_4},  \\
&&\quad \mathbf{Y}_{j_1,j_2,j_3,j_4}=\kappa_y( \sigma_{j_1,j_3}\sigma_{j_2,j_4}+\sigma_{j_1,j_4}\sigma_{j_2,j_3})+(\kappa_y-1)\sigma_{j_1,j_2}\sigma_{j_3,j_4}. \notag
\end{eqnarray}
Then by Condition \ref{cond:twosamplecov2} (1) and (4),  we know $(\mathbf{X}_{j_1,j_2,j_3,j_4})^{m}(\mathbf{Y}_{j_1,j_2,j_3,j_4})^{a-m}$ is a linear combination of 
\begin{align}
  \prod_{t=1}^a \Big\{\sigma_{j_{g_{1}^{(t)}},\, j_{g_2^{(t)}}}\times \sigma_{ j_{g_3^{(t)}},\,  j_{g_4^{(t)}}}\Big\},	 \label{eq:twosamvarsamllorder}
\end{align} where $\{(g_{1}^{(t)},g_2^{(t)}), ( g_3^{(t)},g_4^{(t)}): t=1,\ldots, a\}$ are $a$ allocations of the set $\{1,2,3,4\}$ into 2 (unordered) pairs. When the $a$ allocations are the same, by the symmetricity of $j$ indexes,
\begin{align*}
\sum_{1\leq j_1,j_2,j_3,j_4\leq p} \prod_{t=1}^a \sigma_{j_{g_{1}^{(t)}},j_{g_2^{(t)}}}\sigma_{ j_{g_3^{(t)}}, j_{g_4^{(t)}}}= \sum_{1\leq j_1,j_2,j_3,j_4\leq p}( \sigma_{j_1,j_3}\sigma_{j_2,j_4})^a. 
\end{align*}
When the $a$ allocations are different, by Condition  \ref{cond:twosamplecov2} (4), 
\begin{eqnarray}
&&\quad \quad \sum_{1\leq j_1,j_2,j_3,j_4\leq p} \prod_{t=1}^a \sigma_{j_{g_{1}^{(t)}},j_{g_2^{(t)}}}\sigma_{ j_{g_3^{(t)}}, j_{g_4^{(t)}}}=o(1) \sum_{1\leq j_1,j_2,j_3,j_4\leq p}( \sigma_{j_1,j_3}\sigma_{j_2,j_4})^a, \label{eq:smallodtwosamcond2}
\end{eqnarray} 
which can be obtained by taking  square of both sides of \eqref{eq:smallodtwosamcond2} and using Condition \ref{cond:twosamplecov2} (4). 
It follows that by \eqref{eq:varleadtwosam}, Condition \ref{cond:twosamplecov2} (1) and (4) and the symmetricity of $j$ indexes, 
\begin{align}
	\mathrm{var}\{\tilde{\mathcal{U}}(a)\} =\Theta(n^{-a})\sum_{1\leq j_1,j_2,j_3,j_4\leq p} ( \sigma_{j_1,j_3}\sigma_{j_2,j_4})^a.\label{eq:twosmavarordsecond}
\end{align}



We next show $\mathrm{var}\{\tilde{\mathcal{U}}^*(a)\}= o(1)\mathrm{var}\{\tilde{\mathcal{U}}(a)\}$.
Similarly to Section \ref{sec:twosamcovsmallord1}, we know it suffices to prove $\mathrm{var}(T_{b_1,b_2,c})=o(1)\mathrm{var}\{\tilde{\mathcal{U}}(a)\}$ for $0\leq c\leq a$, $0\leq b_1\leq c$, $0\leq b_2\leq a-c$ and $b_1+b_2\leq a-1$. 
Note that \eqref{eq:vartb1b2c} still holds here,   
and when $\{\mathbf{i}\}\neq \{\tilde{\mathbf{i}}\}$ or $\{\mathbf{w}\}\neq \{\tilde{\mathbf{w}}\}$, $\mathbb{T}(\mathbf{i},\tilde{\mathbf{i}},\mathbf{w},\tilde{\mathbf{w}},j_1,j_2,\tilde{j}_1,\tilde{j}_2)=0$. Therefore, \eqref{eq:varsmallordernsumtwosam1} also holds. 
By Condition \ref{cond:twosamplecov2} (3) and (4), similarly to the analysis of \eqref{eq:twosmavarordsecond}, we have for any given $(\mathbf{i},\tilde{\mathbf{i}},\mathbf{w},\tilde{\mathbf{w}}),$
\begin{eqnarray}
	&& \sum_{1\leq j_1,j_2,j_3,j_4\leq p}\mathbb{T}(\mathbf{i},\tilde{\mathbf{i}},\mathbf{w},\tilde{\mathbf{w}},j_1,j_2,\tilde{j}_1,\tilde{j}_2) \label{eq:varsmallorderpsumtwosam2second} \\
	&=& O(1)\sum_{1\leq j_1,j_2,j_3,j_4\leq p} ( \sigma_{j_1,j_3}\sigma_{j_2,j_4})^a. \notag
\end{eqnarray}
Combining \eqref{eq:varsmallordernsumtwosam1} and \eqref{eq:varsmallorderpsumtwosam2second}, 
\begin{align*}
	\mathrm{var}(T_{b_1,b_2,c}) = O(n^{-a-1})\sum_{1\leq j_1,j_2, j_3,j_4 \leq p} (\sigma_{j_1,j_3}\sigma_{j_2,j_4})^a \notag =  o(1) \mathrm{var}\{\tilde{\mathcal{U}}(a) \}. 
\end{align*}

\subsubsection{Proof of Lemma \ref{lm:twosampcovtest} (on Page \pageref{lm:twosampcovtest}, Section \ref{sec:firstpfthmtwoclt})} \label{lm:pftwosampcovtest}

Since $\mathrm{E}\{\mathcal{U}(a)\}=\mathrm{E}\{\mathcal{U}(b)\}=0$ under $H_0$, $\mathrm{cov}\{ \mathcal{U}(a)/\sigma(a), \mathcal{U}(b)/\sigma(b)\}=\mathrm{E}\{\mathcal{U}(a) \mathcal{U}(b)\}/\{\sigma(a)\sigma(b)\}.$ Recall that $\mathcal{U}(a) = \tilde{\mathcal{U}}(a)+\tilde{\mathcal{U}}^*(a) $ and $\mathcal{U}(b) = \tilde{\mathcal{U}}(b)+\tilde{\mathcal{U}}^*(b).$ Then
\begin{align}
\mathrm{E}\Big\{ \frac{\mathcal{U}(a)}{\sigma(a)}\times \frac{\mathcal{U}(b)}{\sigma(b)}\Big\}=&~\mathrm{E}\Big\{ \frac{\tilde{\mathcal{U}}(a)+\tilde{\mathcal{U}}^*(a) }{\sigma(a)}\times \frac{\tilde{\mathcal{U}}(b)+\tilde{\mathcal{U}}^*(b)}{\sigma(b)}\Big\} \label{eq:covprodexptwosam} \\
= &~ \mathrm{E} \Big\{ \frac{\tilde{\mathcal{U}}(a) \tilde{\mathcal{U}}(b) } {\sigma(a)\sigma(b)}\Big\}+ o(1), \notag
\end{align} where the last equation follows by Lemma \ref{lm:twosampvartest}. By the definition and notation in Section \ref{sec:pftwosampvartest},
\begin{eqnarray*}
	\tilde{\mathcal{U}}(a)=\tilde{C}_a\sum_{\substack{1\leq j_1,j_2\leq p;\\ \mathbf{i}\in \mathcal{P}(n_x,a);\\ \mathbf{w}\in \mathcal{P}(n_y,a)  } }\ \mathbb{D}_{\mathbf{x},\mathbf{y}}(\mathbf{i},\mathbf{w},j_1,j_2), \quad 	\tilde{\mathcal{U}}(b) = \tilde{C}_b\sum_{\substack{ 1\leq \tilde{j}_1,\tilde{j}_2\leq p;\\ \tilde{\mathbf{i}}\in \mathcal{P}(n_x,b);\\  \tilde{\mathbf{w}}\in \mathcal{P}(n_y,b)  } }    \mathbb{D}_{\mathbf{x},\mathbf{y}}(\tilde{\mathbf{i}},\tilde{\mathbf{w}},\tilde{j}_1,\tilde{j}_2),  \notag 
\end{eqnarray*}
where we let $\tilde{C}_a=(P^{n_x}_{a}P^{n_y}_{a})^{-1}$, $\tilde{C}_b=(P^{n_x}_{b}P^{n_y}_{b})^{-1}$, $\mathbb{D}_{\mathbf{x},\mathbf{y}}(\mathbf{i},\mathbf{w},j_1,j_2)=\prod_{t=1}^a  (x_{i_{t},j_1}x_{i_{t},j_2} -y_{w_{t},j_1}y_{w_{t},j_2})$ and $\mathbb{D}_{\mathbf{x},\mathbf{y}}(\tilde{\mathbf{i}},\tilde{\mathbf{w}},\tilde{j}_1,\tilde{j}_2) = \prod_{t=1}^b  (x_{\tilde{i}_{t},\tilde{j}_1}x_{\tilde{i}_{t},\tilde{j}_2} -y_{\tilde{w}_{t},\tilde{j}_1}y_{\tilde{w}_{t},\tilde{j}_2}).$
It follows that 
\begin{align*}
	 \mathrm{E}\{ \tilde{\mathcal{U}}(a)  \tilde{\mathcal{U}}(b)\} =\tilde{C}_a\tilde{C}_b 	\sum_{\substack{1\leq j_1,j_2, \tilde{j}_1,\tilde{j}_2\leq p;\\ \mathbf{i}\in \mathcal{P}(n_x,a);\, \tilde{\mathbf{i}}\in \mathcal{P}(n_x,b) \\ \mathbf{w}\in \mathcal{P}(n_y,a) ; \, \tilde{\mathbf{w}}\in \mathcal{P}(n_y,b) } } \mathrm{E}\Big\{  \mathbb{D}_{\mathbf{x},\mathbf{y}}(\mathbf{i},\mathbf{w},j_1,j_2) \mathbb{D}_{\mathbf{x},\mathbf{y}}(\tilde{\mathbf{i}},\tilde{\mathbf{w}},\tilde{j}_1,\tilde{j}_2) \Big\}.
\end{align*}
As $a\neq b$, we know $\{\mathbf{i}\}\neq \{\tilde{\mathbf{i}}\}$ and $ \{\mathbf{w}\} \neq \{\tilde{\mathbf{w}}\} $. It follows that similarly to Section \ref{sec:proofcovariancezro},  $\mathrm{E}\{  \mathbb{D}_{\mathbf{x},\mathbf{y}}(\mathbf{i},\mathbf{w},j_1,j_2) \mathbb{D}_{\mathbf{x},\mathbf{y}}(\tilde{\mathbf{i}},\tilde{\mathbf{w}},\tilde{j}_1,\tilde{j}_2) \}=0$. Therefore  $ \mathrm{E}\{\tilde{\mathcal{U}}(a) \tilde{\mathcal{U}}(b)\} =0$ and $\mathrm{cov}\{ \mathcal{U}(a)/\sigma(a), \mathcal{U}(b)/\sigma(b)\}=o(1).$

\subsubsection{Derivation of $D_{n,k}$ and $\pi_{n,k}^2$} \label{lm:pftwosampiform}

To prove Lemmas  \ref{lm:twosamvario1} and  \ref{lm:twosamvario2}, we  derive the forms of $D_{n,k}$ and $\pi_{n,k}^2$ in this section.     By construction, $D_{n,k}=\sum_{r=1}^m t_rA_{n,k,a_r},$ where 	$A_{n,k,a_r}=(\mathrm{E}_k-\mathrm{E}_{k-1})[\tilde{\mathcal{U}}(a_r)/\sigma(a_r)]$.
In addition, $\pi_{n,k}^2=  \sum_{1\leq r_1,r_2\leq m} t_{r_1}t_{r_2}\mathrm{E}_{k-1}( A_{n,k,a_{r_1}} A_{n,k,a_{r_2}}).$ 
It then suffices to derive the form of  $A_{n,k,a}$ for a given integer $a$, and also derive  $\mathrm{E}_{k-1}( A_{n,k,a_{1}} A_{n,k,a_{2}})$ for two given integers $a_1$ and $a_2.$ 

For easy presentation, we define $\mathcal{X}_{i,j_1,j_2}=x_{i,j_1}x_{i_{t},j_2}-\sigma_{j_1,j_2}$ and $\mathcal{Y}_{i,j_1,j_2}=y_{i,j_1}y_{i_{t},j_2}-\sigma_{j_1,j_2}$ in the following. Then under $H_0$,
\begin{align*}
\tilde{\mathcal{U}}(a)=(P^{n_x}_{a}P^{n_y}_{a})^{-1}\sum_{\substack{1\leq j_1,j_2\leq p;\\  \mathbf{i}\in \mathcal{P}(n_x,a);\, \mathbf{w}\in \mathcal{P}(n_y,a)  } }\ \prod_{t=1}^a  (\mathcal{X}_{w_t,j_1,j_2}-\mathcal{Y}_{i_t,j_1,j_2}).	
\end{align*}
 


\paragraph{Part I: $1\leq k\leq n_x$}\label{sec:part1sumankform}  When $1\leq k\leq n_x$, similarly to Section \ref{sec:proofcltabnkform}, as $\mathrm{E}(\mathcal{X}_{1,j_1,j_2})=0$ under $H_0$, we have
\begin{align*}
(\mathrm{E}_{k}-\mathrm{E}_{k-1})\Big\{\prod_{t=1}^a  (\mathcal{X}_{i_t,j_1,j_2}-\mathcal{Y}_{w_t,j_1,j_2})\Big\}=(\mathrm{E}_{k}-\mathrm{E}_{k-1})\Big(\prod_{t=1}^a \mathcal{X}_{i_t,j_1,j_2} \Big),
\end{align*} 
which is nonzero only when $i_1,\ldots, i_a \leq k$ and $k \in \{ i_1,\ldots, i_a \}$. Then we know when $k<a$, $A_{n,k,a}=0$  and when $k\geq a$,
\begin{align}
	A_{n,k,a}=&~c_1(n,a)\sum_{\substack{1\leq j_1,j_2 \leq p;\\ \mathbf{i}\in \mathcal{P}(k-1,a-1)}}\Big( \prod_{t=1}^{a-1} \mathcal{X}_{i_t,j_1,j_2} \Big)\mathcal{X}_{k,j_1,j_2}, \label{eq:twosamdnkform1} 
\end{align} where $c_1(n,a)={a!}/\{P^{n_x}_a \sigma(a)\}$.
For two integers $a_1$ and $a_2$, 
\begin{align*}
&~\mathrm{E}_{k-1}(A_{n,k,a_1}A_{n,k,a_2}) \notag \\
=&~\prod_{l=1}^2 c(n,a_l)\sum_{\substack{1\leq j_1,j_2,j_3,j_4 \leq p;\\ \mathbf{i}^{(l)}\in \mathcal{P}(k-1,a_l-1),\, l=1,2}}\mathbb{M}_{\mathbf{x},\mathbf{y},1}(k,\mathbf{i}^{(l)},j_{2l-1},j_{2l}: l=1,2),\notag 
\end{align*}where 
\begin{align*}
	&~\mathbb{M}_{\mathbf{x},\mathbf{y},1}(k,\mathbf{i}^{(l)},j_{2l-1},j_{2l}: l=1,2) \notag \\
	=&~ \prod_{l=1}^2\Big( \prod_{t=1}^{a_l-1}\mathcal{X}_{i^{(l)}_t,j_{2l-1},j_{2l}} \Big) \mathrm{E}(\mathcal{X}_{k,j_{1},j_{2}}\mathcal{X}_{k,j_{3},j_{4}}). \notag 
\end{align*}

\paragraph{Part II: $n_x+1\leq k\leq n_x+n_y$} \label{sec:part2sumankform}
When $n_x+1\leq k\leq n_x+n_y$,   we have
\begin{eqnarray*}
 \prod_{t=1}^a  (\mathcal{X}_{i_t,j_1,j_2}-\mathcal{Y}_{i_t,j_1,j_2})=\sum_{s=0}^a  \sum_{ \substack{\mathbf{i}^* \in \mathcal{S}(\mathbf{i},s);\\ \mathbf{w}^* \in \mathcal{S}(\mathbf{w},a-s)}}  \Big(\prod_{t=1}^s \mathcal{X}_{i_t^*,j_1,j_2} \Big) \Big( \prod_{\tilde{t}=1}^{a-s} \mathcal{Y}_{w_{\tilde{t}}^*,j_1,j_2} \Big), 
\end{eqnarray*}
where $\mathcal{S}(\mathbf{i},s)$ represents the collection of sub-tuples of $\mathbf{i}$ with length $s$ and  $\mathcal{S}(\mathbf{w},a-s)$ represents the collection of sub-tuples of $\mathbf{w}$ with length $a-s$, which is similarly used in Section \ref{sec:lma41}.  
When $n_x+1\leq k\leq n_x+n_y$,  similarly to Section \ref{sec:proofcltabnkform}, $(\mathrm{E}_{k}-\mathrm{E}_{k-1}) \{\prod_{t=1}^s (x_{i_t^*,j_1}x_{i_t^*,j_2}-\sigma_{j_1,j_2})  \prod_{\tilde{t}=1}^{a-s} (y_{w_{\tilde{t}}^*,j_1}y_{w_{\tilde{t}}^*,j_2} -\sigma_{j_1,j_2} )\}\neq 0$ only when $ w_1^*,\ldots, w_{a-s}^*\leq k-n_x$ and $k-n_x \in \{ w_1^*,\ldots, w_{a-s}^*\}$, and then
\begin{align*}
(\mathrm{E}_{k}-\mathrm{E}_{k-1})\Big(\prod_{t=1}^s \mathcal{X}_{i_t^*,j_1,j_2} \prod_{\tilde{t}=1}^{a-s} \mathcal{Y}_{w_{\tilde{t}}^*,j_1,j_2} \Big)=\mathcal{Y}_{k-n_x,j_1,j_2}  \prod_{t=1}^s \mathcal{X}_{i_t^*,j_1,j_2}  \prod_{\tilde{t}=1}^{a-s-1} \mathcal{Y}_{w_{\tilde{t}}^*,j_1,j_2}. 
\end{align*}
It follows that  
\begin{align*}
	A_{n,k,a}=&\sum_{s=L_k}^{a-1}  \sum_{\substack{1\leq j_1,j_2\leq p;\\ \mathbf{i}\in \mathcal{P}(n_x,s); \\ \mathbf{w}\in \mathcal{P}(k-n_x-1, a-s-1)} } c_2(n,a,s)\mathcal{Y}_{k-n_x,j_1,j_2}  \prod_{t=1}^s \mathcal{X}_{i_t,j_1,j_2}  \prod_{\tilde{t}=1}^{a-s-1} \mathcal{Y}_{w_{\tilde{t}},j_1,j_2},  
\end{align*}
where $L_k=\max\{n_x-k+a,0\}$ and $c_2(n,a,s)={P^{n_x-s}_{a-s}P^{n_y-a+s}_{s}}\{P^{n_x}_a P^{n_y}_a\sigma(a)\}^{-1}$. 
Thus for two constants $a_1$ and $a_2$, 
\begin{align*}
	&~\mathrm{E}_{k-1}(A_{n,k,a_1}A_{n,k,a_2}) \notag \\
 =&~\sum_{\substack{1\leq j_1,j_2,j_3,j_4\leq p;\\ L_k\leq  s_l \leq a_l:\, l=1,2; \\ \mathbf{i}^{(l)}\in \mathcal{P}(n_x,s_l):\, l=1,2;\\ \mathbf{w}^{(l)}\in \mathcal{P}(k-n_x-1,a_l-s_l-1):\, l=1,2}} \prod_{l=1}^2c_2(n,a_l,s_l) \mathbb{M}_{\mathbf{x},\mathbf{y},2}(k-n_x,\mathbf{i}^{(l)},j_{2l-1},j_{2l}: l=1,2),
\end{align*}
where
\begin{align*}
&~ \mathbb{M}_{\mathbf{x},\mathbf{y},2}(k-n_x,\mathbf{i}^{(l)},j_{2l-1},j_{2l}: l=1,2) \notag \\
=&~ \prod_{l=1}^2\Big( \prod_{t=1}^{s_l} \mathcal{X}_{i_t^{(l)},j_{2l-1},j_{2l}}  \prod_{\tilde{t}=1}^{a_l-s_l-1} \mathcal{Y}_{w_{\tilde{t}}^{(l)},j_{2l-1},j_{2l}} \Big) \mathrm{E}(\mathcal{Y}_{k-n_x,j_{1},j_{2}}\mathcal{Y}_{k-n_x,j_{3},j_{4}} ).
\end{align*}

\subsubsection{Proof of Lemma \ref{lm:twosamvario1} (on Page \pageref{lm:twosamvario1}, Section \ref{sec:firstpfthmtwoclt})} \label{sec:pftwosamvario1} 

Note that by the Cauchy-Schwarz inequality, for some constant $C$, $$\mathrm{var}\Big(\sum_{k=1}^n \pi_{n,k}^2\Big)\leq Cn^2 \max_{1\leq k\leq n;\, 1\leq r_1,r_2\leq m} \mathrm{var}(\mathbb{T}_{k,a_{r_1},a_{r_2}}), $$
where for two integers $a_1$ and $a_2$, $\mathbb{T}_{k,a_{1},a_{2}}=\mathrm{E}_{k-1}( A_{n,k,a_{1}} A_{n,k,a_{2}}) 
$ is given in Section \ref{lm:pftwosampiform}.
Therefore to prove Lemma \ref{lm:twosamvario1}, it suffices to prove $\mathrm{var}(\mathbb{T}_{k,a_{r_1},a_{r_2}})=o(n^{-2})$ for every  $1\leq k\leq n$ and $1\leq r_1,r_2\leq m$.  We next prove $ \mathrm{var}(\mathbb{T}_{k,a_1,a_2})=o(n^{-2})$ when $a\leq k\leq n_x$ and $n_x+1\leq k\leq n_x+n_y$ in the following Parts I and II respectively.  


\paragraph{Part I: $a\leq k\leq n_x$}\label{sec:pdfcltpart1} 
We first derive the form of $\mathrm{var}(\mathbb{T}_{k,a_1,a_2} )$ when $a\leq k\leq n_x$. As $\mathrm{var}(\mathbb{T}_{k,a_1,a_2} )=\mathrm{E}(\mathbb{T}_{k,a_1,a_2}^2)- \{\mathrm{E}(\mathbb{T}_{k,a_1,a_2})\}^2$, we next derive $\mathrm{E}(\mathbb{T}_{k,a_1,a_2})$ and $\mathrm{E}(\mathbb{T}_{k,a_1,a_2}^2)$.  In particular,
\begin{align*}
	\mathrm{E}(\mathbb{T}_{k,a_1,a_2})=\prod_{l=1}^2c(n,a_l)\sum_{\substack{1\leq j_1,j_2,j_3,j_4 \leq p;\\ \mathbf{i}^{(l)}\in \mathcal{P}(k-1,a_l-1),\, l=1,2}}\mathrm{E}\Big\{\mathbb{M}_{\mathbf{x},\mathbf{y},1}(k,\mathbf{i}^{(l)},j_{2l-1},j_{2l}: l=1,2)\Big\}.
\end{align*}
For easy presentation, we let $a_3=a_1$ and $a_4=a_2$, and have
\begin{align*}
&~\Big\{\mathrm{E}(\mathbb{T}_{k,a_1,a_2})\Big\}^2 \notag \\
=&~ \prod_{l=1}^4c(n,a_l)	\sum_{\substack{1\leq j_1,j_2,j_3,j_4,j_5,j_6,j_7,j_8 \leq p;\\ \mathbf{i}^{(l)}\in \mathcal{P}(k-1,a_l-1),\, l=1,2,3,4}}\mathrm{E}\Big\{\mathbb{M}_{\mathbf{x},\mathbf{y},1}(k,\mathbf{i}^{(l)},j_{2l-1},j_{2l}: l=1,2)\Big\}\notag \\
&~ \times \mathrm{E}\Big\{\mathbb{M}_{\mathbf{x},\mathbf{y},1}(k,\mathbf{i}^{(l)},j_{2l-1},j_{2l}: l=3,4)\Big\}.
\end{align*}
In addition, we have
\begin{align*}
	\mathrm{E}(\mathbb{T}_{k,a_1,a_2}^2)=&~\prod_{l=1}^4c(n,a_l) \sum_{\substack{1\leq j_1,j_2,j_3,j_4,j_5,j_6,j_7,j_8 \leq p;\\ \mathbf{i}^{(l)}\in \mathcal{P}(k-1,a_l-1),\, l=1,2,3,4}} \notag \\
	&~ \mathrm{E}\Big\{\mathbb{M}_{\mathbf{x},\mathbf{y},1}(k,\mathbf{i}^{(l)},j_{2l-1},j_{2l}: l=1,2,3,4)   \Big\},
\end{align*} where we define
\begin{align*}
&~\mathbb{M}_{\mathbf{x},\mathbf{y},1}(k,\mathbf{i}^{(l)},j_{2l-1},j_{2l}: l=1,2,3,4)	\notag \\
=&~ \prod_{l=1}^4\Big(  \prod_{t=1}^{a_l-1}\mathcal{X}_{i^{(l)}_t,j_{2l-1},j_{2l}} \Big)\mathrm{E}(\mathcal{X}_{k,j_{1},j_{2}}\mathcal{X}_{k,j_{3},j_{4}})\mathrm{E}(\mathcal{X}_{k,j_{5},j_{6}}\mathcal{X}_{k,j_{7},j_{8}}).
\end{align*}

Let $\mathbf{1}_{E}$ be an indicator function of the event that $(\{\mathbf{i}^{(1)}\}\cup \{\mathbf{i}^{(2)}\})\cap (\{\mathbf{i}^{(3)}\}\cup \{\mathbf{i}^{(4)}\})=\emptyset$. 
Then define
\begin{align*}
	G_{a_1,a_2,1}=&~\prod_{l=1}^4c(n,a_l)	\sum_{\substack{1\leq j_1,j_2,j_3,j_4,j_5,j_6,j_7,j_8 \leq p;\\ \mathbf{i}^{(l)}\in \mathcal{P}(k-1,a_l-1),\, l=1,2,3,4}}\times \mathbf{1}_{E} \notag \\
	&~\times \mathrm{E}\Big\{\mathbb{M}_{\mathbf{x},\mathbf{y},1}(k,\mathbf{i}^{(l)},j_{2l-1},j_{2l}: l=1,2,3,4)\Big\}.
\end{align*} We also note that
\begin{align}
&~\mathrm{E}\Big\{\mathbb{M}_{\mathbf{x},\mathbf{y},1}(k,\mathbf{i}^{(l)},j_{2l-1},j_{2l}: l=1,2,3,4)\Big\}\times \mathbf{1}_{E} \label{eq:equiedef} \\
=&~	\mathrm{E}\Big\{\mathbb{M}_{\mathbf{x},\mathbf{y},1}(k,\mathbf{i}^{(l)},j_{2l-1},j_{2l}: l=1,2)\Big\} \notag \\
	&\quad \times \mathrm{E}\Big\{\mathbb{M}_{\mathbf{x},\mathbf{y},1}(k,\mathbf{i}^{(l)},j_{2l-1},j_{2l}: l=3,4)\Big\} \times \mathbf{1}_{E}. \notag 
\end{align}
Since $|\mathrm{var}(\mathbb{T}_{k,a_1,a_2} )|\leq |\mathrm{E}(\mathbb{T}^2_{k,a_1,a_2} )-G_{a_1,a_2,1}|+|\{\mathrm{E}(\mathbb{T}_{k,a_1,a_2} )\}^2-G_{a_1,a_2,1}|$, to prove $\mathrm{var}(\mathbb{T}_{k,a_1,a_2} )=o(n^{-2})$, we will next show that   $|\{\mathrm{E}(\mathbb{T}_{k,a_1,a_2} )\}^2-G_{a_1,a_2,1}|=o(n^{-2})$ and $|\mathrm{E}(\mathbb{T}^2_{k,a_1,a_2} )-G_{a_1,a_2,1}|=o(n^{-2})$. In particular, we present the proof under Conditions \ref{cond:twosamplecov1}  and \ref{cond:twosamplecov2} in the following Sections \ref{sec:pflem133cond1} and \ref{sec:pflem133cond2}, respectively.  

\medskip

\subparagraph{Proof under Condition \ref{cond:twosamplecov1}} \label{sec:pflem133cond1}
\quad

\smallskip

\noindent \textit{Step I: $|\{\mathrm{E}(\mathbb{T}_{k,a_1,a_2} )\}^2-G_{a_1,a_2,1}|=o(n^{-2})$.}
If $a_1\neq a_2$, we have $\mathrm{E}(\mathbb{T}_{k,a_1,a_2} )=G_{a_1,a_2,1}=0$. It remains to consider $a_1=a_2$ below. 
Note that
\begin{eqnarray}
&&  \mathrm{E}\{\mathbb{M}_{\mathbf{x},\mathbf{y},1}(k,\mathbf{i}^{(l)},j_{2l-1},j_{2l}: l=1,2)\}\label{eq:prodnonzero}	 \\
&&\times \mathrm{E}\{\mathbb{M}_{\mathbf{x},\mathbf{y},1}(k,\mathbf{i}^{(l)},j_{2l-1},j_{2l}: l=3,4)\} \notag
\end{eqnarray}satisfies that $\eqref{eq:prodnonzero}\neq 0$ only if $ \{\mathbf{i}^{(1)}\}= \{\mathbf{i}^{(2)}\}$ and $ \{\mathbf{i}^{(3)}\}= \{\mathbf{i}^{(4)}\}$. Thus, 
\begin{align*}
	\{\mathrm{E}(\mathbb{T}_{k,a_1,a_2} )\}^2=&~\prod_{l=1}^4c(n,a_l)	\sum_{\substack{1\leq j_1,j_2,j_3,j_4,j_5,j_6,j_7,j_8 \leq p;\\ \mathbf{i}^{(l)}\in \mathcal{P}(k-1,a_l-1),\, l=1,2,3,4}}\mathbf{1}_{\Big\{ \substack{\{\mathbf{i}^{(1)}\}= \{\mathbf{i}^{(2)}\},\\ \{\mathbf{i}^{(3)}\}= \{\mathbf{i}^{(4)}\} }\Big\}}\times \eqref{eq:prodnonzero}. \notag 
\end{align*} 
Similarly, $\mathrm{E}\{\mathbb{M}_{\mathbf{x},\mathbf{y},1}(k,\mathbf{i}^{(l)},j_{2l-1},j_{2l}: l=1,2,3,4)\}\times \mathbf{1}_{E}\neq 0$ only when $ \{\mathbf{i}^{(1)}\}= \{\mathbf{i}^{(2)}\}$ and $ \{\mathbf{i}^{(3)}\}= \{\mathbf{i}^{(4)}\}$. Therefore, by \eqref{eq:equiedef},
\begin{align*}
G_{a_1,a_2,1}=	\prod_{l=1}^4c(n,a_l)	\sum_{\substack{1\leq j_1,j_2,j_3,j_4,j_5,j_6,j_7,j_8 \leq p;\\ \mathbf{i}^{(l)}\in \mathcal{P}(k-1,a_l-1),\, l=1,2,3,4}} \mathbf{1}_{\Big\{ \substack{\{\mathbf{i}^{(1)}\}= \{\mathbf{i}^{(2)}\},\\ \{\mathbf{i}^{(3)}\}= \{\mathbf{i}^{(4)}\},\\ \{\mathbf{i}^{(1)}\}\cap\{\mathbf{i}^{(3)}\}= \emptyset \} } \Big\} }\times \eqref{eq:prodnonzero},
\end{align*}and then
\begin{eqnarray}
\quad &&|\{\mathrm{E}(\mathbb{T}_{k,a_1,a_2} )\}^2-G_{a_1,a_2,1}| \label{eq:diffpart1clt1twocov} \\
& \leq &\prod_{l=1}^4c(n,a_l)	\sum_{\substack{1\leq j_1,j_2,j_3,j_4,j_5,j_6,j_7,j_8 \leq p;\\ \mathbf{i}^{(l)}\in \mathcal{P}(k-1,a_l-1),\, l=1,2,3,4}} \mathbf{1}_{\Big\{ \substack{\{\mathbf{i}^{(1)}\}= \{\mathbf{i}^{(2)}\},\\ \{\mathbf{i}^{(3)}\}= \{\mathbf{i}^{(4)}\},\\ \{\mathbf{i}^{(1)}\}\cap \{\mathbf{i}^{(3)}\}\neq \emptyset \} } \Big\} }\times |\eqref{eq:prodnonzero}|. \notag 
\end{eqnarray}
Note that
\begin{align}
	\sum_{\substack{ \mathbf{i}^{(l)}\in \mathcal{P}(k-1,a_l-1),\, l=1,2,3,4}} \mathbf{1}_{\Big\{ \substack{\{\mathbf{i}^{(1)}\}= \{\mathbf{i}^{(2)}\},\\ \{\mathbf{i}^{(3)}\}= \{\mathbf{i}^{(4)}\},\\ \{\mathbf{i}^{(1)}\}\cap \{\mathbf{i}^{(3)}\}\neq \emptyset \} } \Big\} }=O(n^{a_1+a_2-3}). \label{eq:taaorderncond2}
\end{align}
In addition, by Condition \ref{cond:twosamplecov1} (2), 
\begin{align}
	&~\sum_{1\leq j_1,j_2,j_3,j_4,j_5,j_6,j_7,j_8 \leq p} |\eqref{eq:prodnonzero} |\label{eq:twocovpordond11} \\
\leq & C\sum_{1\leq j_1,j_2,j_3,j_4,j_5,j_6,j_7,j_8 \leq p}\Big|\mathrm{E}(\mathcal{X}_{k,j_{1},j_{2}}\mathcal{X}_{k,j_{3},j_{4}})\mathrm{E}(\mathcal{X}_{k,j_{5},j_{6}}\mathcal{X}_{k,j_{7},j_{8}})\Big|. \notag
\end{align}
Recall that $\mathrm{E}(\mathcal{X}_{k,j_{1},j_{2}}\mathcal{X}_{k,j_{3},j_{4}})=\mathbf{X}_{j_1,j_2,j_3,j_4}$ and $\mathrm{E}(\mathcal{X}_{k,j_{5},j_{6}}\mathcal{X}_{k,j_{7},j_{8}})=\mathbf{X}_{j_5,j_6,j_7,j_8}$ following the notation in Section \ref{sec:pftwosampvartest}. Following the similar analysis for the proof of \eqref{eq:twopsqorderxy}, we obtain 
$
	\sum_{1\leq j_1,j_2,j_3,j_4\leq p} |\mathbf{X}_{j_1,j_2,j_3,j_4}|=O(p^2)$ and $  \sum_{1\leq j_5,j_6,j_7,j_8\leq p} |\mathbf{X}_{j_5,j_6,j_7,j_8}|=O(p^2).
$
It follows that $ \eqref{eq:twocovpordond11}=O(p^4)$. Note that $c(n,a)=\Theta(p^{-1}n^{-a/2})$ by Lemma \ref{lm:twosampvartest}. 
Combining \eqref{eq:taaorderncond2} and \eqref{eq:twocovpordond11}, we obtain $\{\mathrm{E}(\mathbb{T}_{k,a_1,a_2} )\}^2-G_{a_1,a_2,1}=o(n^{-2}).$

\smallskip

\noindent \textit{Step II: $|\mathrm{E}(\mathbb{T}^2_{k,a_1,a_2} )-G_{a_1,a_2,1}|=o(n^{-2})$.}
By construction, we have
\begin{eqnarray}
\quad \quad \quad && \mathrm{E}(\mathbb{T}^2_{k,a_1,a_2} )-G_{a_1,a_2,1}=\prod_{l=1}^4c(n,a_l)	\sum_{\substack{1\leq j_1,j_2,j_3,j_4,j_5,j_6,j_7,j_8 \leq p;\\ \mathbf{i}^{(l)}\in \mathcal{P}(k-1,a_l-1),\, l=1,2,3,4}} (1-\mathbf{1}_E) \label{eq:twosamdiffp1g} \\
	&&\quad \quad \quad\quad \quad \quad\quad \quad \quad \times \mathrm{E}\Big\{\mathbb{M}_{\mathbf{x},\mathbf{y},1}(k,\mathbf{i}^{(l)},j_{2l-1},j_{2l}: l=1,2,3,4)\Big\}.\notag
\end{eqnarray} 
When  $|\cup_{l=1}^4 \{\mathbf{i}^{(l)}\}|> a_1+a_2-2$, which means that there exists one index that  only appears once among the four sets $\{\mathbf{i}^{(l)}\}$, $l=1,2,3,4$, then similarly to Section \ref{sec:prooftargetorder},  
\begin{align}
	&~ \mathrm{E}\Big\{\mathbb{M}_{\mathbf{x},\mathbf{y},1}(k,\mathbf{i}^{(l)},j_{2l-1},j_{2l}: l=1,2)\Big\}\times (1-\mathbf{1}_E)\label{eq:diffmtwosamm} 
\end{align} satisfies that $\eqref{eq:diffmtwosamm}=0$. 
 When $|\cup_{l=1}^4 \{\mathbf{i}^{(l)}\}|< a_1+a_2-2$, 
 \begin{align}
 	\sum_{\substack{ \mathbf{i}^{(l)}\in \mathcal{P}(k-1,a_l-1),\, l=1,2,3,4}}\mathbf{1}_{\{|\cup_{l=1}^4 \{\mathbf{i}^{(l)}\}|< a_1+a_2-2 \} }=O(n^{a_1+a_2-3}).\label{eq:twosamnodr3mincond1}
 \end{align}  
Similarly to the analysis of \eqref{eq:twocovpordond11} above, by  Condition \ref{cond:twosamplecov1}, we have
\begin{align}
\sum_{1\leq j_1,j_2,j_3,j_4,j_5,j_6,j_7,j_8 \leq p}\eqref{eq:diffmtwosamm}=O(p^4).	\label{eq:twosamp4ordc1}
\end{align}
Therefore, by \eqref{eq:twosamnodr3mincond1}, \eqref{eq:twosamp4ordc1} and $c(n,a)=\Theta(p^{-1}n^{-a/2})$,
\begin{eqnarray*}
	&&\prod_{l=1}^4c(n,a_l)	\sum_{\substack{1\leq j_1,j_2,j_3,j_4,j_5,j_6,j_7,j_8 \leq p;\\ \mathbf{i}^{(l)}\in \mathcal{P}(k-1,a_l-1),\, l=1,2,3,4}}(1-\mathbf{1}_E)  \mathbf{1}_{ \{|\cup_{l=1}^4 \{\mathbf{i}^{(l)}\}|< a_1+a_2-2 \}} \notag \\
	&& \quad \quad \quad \quad\quad \quad \times \mathrm{E}\Big\{\mathbb{M}_{\mathbf{x},\mathbf{y},1}(k,\mathbf{i}^{(l)},j_{2l-1},j_{2l}: l=1,2)\Big\}\notag \\
	&= & O(1)n^{-a_1-a_2}p^{-4} n^{a_1+a_2-3}p^4 = o(n^{-2}). 
\end{eqnarray*}
Last,  we consider $|\cup_{l=1}^4 \{\mathbf{i}^{(l)}\}|= a_1+a_2-2$. Note that $1-\mathbf{1}_E \neq 0$ indicates that $(\{\mathbf{i}^{(1)}\}\cup \{\mathbf{i}^{(2)}\})\cap (\{\mathbf{i}^{(3)}\}\cup \{\mathbf{i}^{(4)}\})\neq \emptyset$ under this case.  By the symmetricity of the $j$ indexes, we have
\begin{align}
	&~ \Big| \sum_{1\leq j_1,j_2,j_3,j_4,j_5,j_6,j_7,j_8 \leq p}\eqref{eq:diffmtwosamm}\Big| \label{eq:sumjmixtwocov} \\
	\leq &~C \sum_{1\leq j_1,j_2,j_3,j_4,j_5,j_6,j_7,j_8 \leq p}\Big|\mathrm{E}(\mathcal{X}_{k,j_{1},j_{2}}\mathcal{X}_{k,j_{3},j_{4}})\mathrm{E}(\mathcal{X}_{k,j_{5},j_{6}}\mathcal{X}_{k,j_{7},j_{8}}) \notag \\
	& \quad \quad \times \mathrm{E}(\mathcal{X}_{k,j_{1},j_{2}}\mathcal{X}_{k,j_{5},j_{6}})\mathrm{E}(\mathcal{X}_{k,j_{3},j_{4}}\mathcal{X}_{k,j_{7},j_{8}})\Big|. \notag 
\end{align}Following similar arguments to that in  Sections \ref{sec:pfvartaamigixing} and \ref{sec:twosamcovsmallord1}, by discussing different cases of $j$ indexes, we have $\eqref{eq:sumjmixtwocov}=o(p^4)$. Thus,
\begin{eqnarray*}
	&&\prod_{l=1}^4c(n,a_l)	\sum_{\substack{1\leq j_1,j_2,j_3,j_4,j_5,j_6,j_7,j_8 \leq p;\\ \mathbf{i}^{(l)}\in \mathcal{P}(k-1,a_l-1),\, l=1,2,3,4}}(1-\mathbf{1}_E)  \mathbf{1}_{ \{|\cup_{l=1}^4 \{\mathbf{i}^{(l)}\}|= a_1+a_2-2 \}} \notag \\
	&& \quad \quad \quad \quad\quad \quad \times \mathrm{E}\Big\{\mathbb{M}_{\mathbf{x},\mathbf{y},1}(k,\mathbf{i}^{(l)},j_{2l-1},j_{2l}: l=1,2)\Big\}\notag \\
	&= & o(1)n^{-a_1-a_2}p^{-4} n^{a_1+a_2-2}p^4 = o(n^{-2}). 
\end{eqnarray*}In summary, we obtain $ \mathrm{E}(\mathbb{T}^2_{k,a_1,a_2} )-G_{a_1,a_2,1}=o(n^{-2}).$

%

\medskip

\subparagraph{Proof under Condition \ref{cond:twosamplecov2}} \label{sec:pflem133cond2}


Similarly to  Section \ref{sec:pflem133cond1}, we next prove $|\{\mathrm{E}(\mathbb{T}_{k,a_1,a_2} )\}^2-G_{a_1,a_2,1}|=o(n^{-2})$ and  $|\mathrm{E}(\mathbb{T}^2_{k,a_1,a_2} )-G_{a_1,a_2,1}|=o(n^{-2})$. 

\smallskip

\noindent \textit{Step I: $|\{\mathrm{E}(\mathbb{T}_{k,a_1,a_2} )\}^2-G_{a_1,a_2,1}|=o(n^{-2})$.}
Following the same analysis in Section \ref{sec:pflem133cond1}, we obtain \eqref{eq:diffpart1clt1twocov} and \eqref{eq:taaorderncond2}. 
By Condition \ref{cond:twosamplecov2} (2) and (4), we have
\begin{eqnarray}
\quad	&&\sum_{\substack{1\leq j_1,j_2,j_3,j_4  j_5,j_6,j_7,j_8 \leq p}}\eqref{eq:prodnonzero}\label{eq:taaorderpcond2} \\
\quad	&=& O(1)\Big\{ \sum_{1\leq j_1,j_2,j_3,j_4\leq p}(\sigma_{j_1,j_2}\sigma_{j_3,j_4})^{a_1}\Big\}  \Big\{ \sum_{1\leq j_5,j_6,j_7,j_8\leq p}(\sigma_{j_5,j_6}\sigma_{j_7,j_8})^{a_2}\Big\}. \notag 
\end{eqnarray}Note that $\sigma^2(a)=\Theta(n^{-a})\times \sum_{1\leq j_1,j_2,j_3,j_4\leq p}( \sigma_{j_1,j_3}\sigma_{j_2,j_4})^a$ by Lemma \ref{lm:twosampvartest}, and $c(n,a)=\Theta(1)\{n^a\sigma(a)\}^{-1}$. Combining \eqref{eq:taaorderncond2} and \eqref{eq:taaorderpcond2}, we have $|\{\mathrm{E}(\mathbb{T}_{k,a_1,a_2} )\}^2-G_{a_1,a_2}|=o(n^{-2})$.

\smallskip

\noindent \textit{Step II: $|\mathrm{E}(\mathbb{T}^2_{k,a_1,a_2} )-G_{a_1,a_2,1}|=o(n^{-2})$.}
Similarly to Section \ref{sec:pflem133cond1}, we have \eqref{eq:twosamdiffp1g} and $\mathrm{E}\{\mathbb{M}_{\mathbf{x},\mathbf{y},1}(k,\mathbf{i}^{(l)},j_{2l-1},j_{2l}: l=1,2,3,4)\}\neq 0$ only when $|\cup_{l=1}^4 \{\mathbf{i}^{(l)}\}|\leq a_1+a_2-2$. 

When $|\cup_{l=1}^4 \{\mathbf{i}^{(l)}\}|< a_1+a_2-2$, \eqref{eq:twosamnodr3mincond1} still holds. 
By  Condition \ref{cond:twosamplecov2} (2) and (4), similarly to \eqref{eq:taaorderpcond2}, we have
\begin{eqnarray*}
	&&\sum_{\substack{1\leq j_1,j_2,j_3,j_4  j_5,j_6,j_7,j_8 \leq p}}  \mathrm{E}\Big\{\mathbb{M}_{\mathbf{x},\mathbf{y},1}(k,\mathbf{i}^{(l)},j_{2l-1},j_{2l}: l=1,2,3,4)\Big\} \notag \\
&=& O(1)\Big\{ \sum_{1\leq j_1,j_2,j_3,j_4\leq p}(\sigma_{j_1,j_2}\sigma_{j_3,j_4})^{a_1}\Big\}  \Big\{ \sum_{1\leq j_5,j_6,j_7,j_8\leq p}(\sigma_{j_5,j_6}\sigma_{j_7,j_8})^{a_2}\Big\}. \notag 
\end{eqnarray*}
Note that $\sigma^2(a)=\Theta(n^{-a})\times \sum_{1\leq j_1,j_2,j_3,j_4\leq p}( \sigma_{j_1,j_3}\sigma_{j_2,j_4})^a$ by Lemma \ref{lm:twosampvartest}, and $c(n,a)=\Theta(1)\{n^a\sigma(a)\}^{-1}$. Then we have
\begin{align*}
&~\prod_{l=1}^4c(n,a_l)	\sum_{\substack{1\leq j_1,j_2,j_3,j_4,j_5,j_6,j_7,j_8 \leq p;\\ \mathbf{i}^{(l)}\in \mathcal{P}(k-1,a_l-1),\, l=1,2,3,4}} \mathbf{1}_{ \{|\cup_{l=1}^4 \{\mathbf{i}^{(l)}\}|< a_1+a_2-2 \}} \notag \\
	&~\times \mathrm{E}\Big\{\mathbb{M}_{\mathbf{x},\mathbf{y},1}(k,\mathbf{i}^{(l)},j_{2l-1},j_{2l}: l=1,2,3,4)\Big\} = o(n^{-2}).\notag	
\end{align*}

When $|\cup_{l=1}^4 \{\mathbf{i}^{(l)}\}|= a_1+a_2-2 $, by the construction of $\mathbf{1}_E$, we know 
\begin{align}
	&~\mathrm{E}\Big\{\mathbb{M}_{\mathbf{x},\mathbf{y},1}(k,\mathbf{i}^{(l)},j_{2l-1},j_{2l}: l=1,2,3,4)\Big\} \times (1-\mathbf{1}_E)\label{eq:twosam2pfxpi4form} 
\end{align} 
satisfies that $\eqref{eq:twosam2pfxpi4form}\neq 0$ if $(\{\mathbf{i}^{(1)}\}\cup \{\mathbf{i}^{(2)}\})\cap (\{\mathbf{i}^{(3)}\}\cup \{\mathbf{i}^{(4)}\})\neq \emptyset$.
Then by Condition \ref{cond:twosamplecov2} (3) and (4), we know \eqref{eq:twosam2pfxpi4form}  is a linear combination of $\sum_{1\leq j_1,\ldots,j_8\leq p}\prod_{t=1}^{a+b} \sigma_{j_{g_{2t-1}},\,  j_{g_{2t}} }$ with $S_{\mathcal{G}}>4$,  where we recall that $S_{\mathcal{G}}$ is the number of distinct sets among $\{g_{2t-1},g_{2t}\}, t=1,\ldots, a+b$, induced by  $\mathcal{G}=(g_1,\ldots,g_{2(a+b)})$. 
Therefore,
\begin{align*}
	&~\prod_{l=1}^4c(n,a_l)	\sum_{\substack{1\leq j_1,j_2,j_3,j_4,j_5,j_6,j_7,j_8 \leq p;\\ \mathbf{i}^{(l)}\in \mathcal{P}(k-1,a_l-1),\, l=1,2,3,4}} (1-\mathbf{1}_E) \times \mathbf{1}_{ \{|\cup_{l=1}^4 \{\mathbf{i}^{(l)}\}|= a_1+a_2-2 \}} \notag \\
	&~\times \mathrm{E}\Big\{\mathbb{M}_{\mathbf{x},\mathbf{y},1}(k,\mathbf{i}^{(l)},j_{2l-1},j_{2l}: l=1,2,3,4)\Big\} \notag \\
\leq & ~ C \Big\{\prod_{l=1}^4c(n,a_l)\Big\}\times n^{a_1+a_2-2} \sum_{\mathcal{G}: S_{\mathcal{G}}>4} \Big|\sum_{ 1\leq j_1,\ldots,j_8\leq p }\prod_{t=1}^{a+b}\sigma_{j_{g_{2t-1}},\, j_{g_{2t}}}\Big|\notag \\
=&~ o(n^{-2}).
\end{align*} where the last equation follows by Condition \ref{cond:twosamplecov2} (4), $\sigma^2(a)=\Theta(n^{-a})\times \sum_{1\leq j_1,j_2,j_3,j_4\leq p}( \sigma_{j_1,j_3}\sigma_{j_2,j_4})^a$, and $c(n,a)=\Theta(1)\{n^a\sigma(a)\}^{-1}$. In summary, we obtain $ \mathrm{E}(\mathbb{T}^2_{k,a_1,a_2} )-G_{a_1,a_2,1}=o(n^{-2}).$

\smallskip

\paragraph{Part II: $n_x\leq k\leq n_x+n_y$}\label{sec:pdfcltpart2} 
In this section, we prove that when $n_x\leq k\leq n_x+n_y$, $\mathrm{var}(\mathbb{T}_{k,a_1,a_2})=o(n^{-2})$. Recall the form derived in Section \ref{sec:part2sumankform}.  We have $\mathbb{T}_{k,a_1,a_2}=\sum_{ L_1\leq s_1\leq a_1, L_2\leq s_2\leq a_2}\mathbb{T}_{k,a_1,a_2,s_1,s_2}$, where
\begin{align*}
	&~ \mathbb{T}_{k,a_1,a_2,s_1,s_2}=\sum_{\substack{1\leq j_1,j_2,j_3,j_4\leq p; \\ \mathbf{i}^{(l)}\in \mathcal{P}(n_x,s_l):\, l=1,2;\\ \mathbf{w}^{(l)}\in \mathcal{P}(k-n_x-1,a_l-s_l-1):\, l=1,2}} \prod_{l=1}^2c_2(n,a_l,s_l) \notag \\
&~\quad \quad \quad \quad \quad \quad\quad  \times  \mathbb{M}_{\mathbf{x},\mathbf{y},2}(k-n_x,\mathbf{i}^{(l)},j_{2l-1},j_{2l}: l=1,2). 
\end{align*}
To prove $\mathrm{var}(\mathbb{T}_{k,a_1,a_2})=o(n^{-2})$, it suffices to prove $\mathrm{var}(\mathbb{T}_{k,a_1,a_2,s_1,s_2})=o(n^{-2})$. In particular, for easy presentation, we set $a_3=a_1$, $a_4=a_2$  $s_3=s_1$ and $s_4=s_2$,  and then have 
\begin{align*}
	\{\mathrm{E}(\mathbb{T}_{k,a_1,a_2,s_1,s_2})\}^2 = &~ \sum_{\substack{1\leq j_1,j_2,j_3,j_4,j_5,j_6,j_7,j_8\leq p; \\ \mathbf{i}^{(l)}\in \mathcal{P}(n_x,s_l):\, l=1,2,3,4;\\ \mathbf{w}^{(l)}\in \mathcal{P}(k-n_x-1,a_l-s_l-1):\, l=1,2,3,4}} \prod_{l=1}^4c_2(n,a_l,s_l) \notag \\
	&~ \mathrm{E}\Big\{ \mathbb{M}_{\mathbf{x},\mathbf{y},2}(k-n_x,\mathbf{i}^{(l)},j_{2l-1},j_{2l}: l=1,2)\Big\} \notag \\
	&~ \times \mathrm{E}\Big\{ \mathbb{M}_{\mathbf{x},\mathbf{y},2}(k-n_x,\mathbf{i}^{(l)},j_{2l-1},j_{2l}: l=3,4)\Big\}.
\end{align*}In addition, we have 
\begin{align*}
\mathrm{E}(\mathbb{T}_{k,a_1,a_2,s_1,s_2}^2)= &~ \sum_{\substack{1\leq j_1,j_2,j_3,j_4,j_5,j_6,j_7,j_8\leq p; \\ \mathbf{i}^{(l)}\in \mathcal{P}(n_x,s_l):\, l=1,2,3,4;\\ \mathbf{w}^{(l)}\in \mathcal{P}(k-n_x-1,a_l-s_l-1):\, l=1,2,3,4}} \prod_{l=1}^4c_2(n,a_l,s_l) \notag \\
	& ~ \quad \times  \mathrm{E}\Big\{ \mathbb{M}_{\mathbf{x},\mathbf{y},2}(k-n_x,\mathbf{i}^{(l)},j_{2l-1},j_{2l}: l=1,2,3,4)\Big\},
\end{align*} where we define
\begin{align*}
&~\mathbb{M}_{\mathbf{x},\mathbf{y},2}(k-n_x,\mathbf{i}^{(l)},j_{2l-1},j_{2l}: l=1,2,3,4) \notag \\
=&~ \prod_{l=1}^4\Big( \prod_{t=1}^{s_l} \mathcal{X}_{i_t^{(l)},j_{2l-1},j_{2l}}  \prod_{\tilde{t}=1}^{a_l-s_l-1} \mathcal{Y}_{w_{\tilde{t}}^{(l)},j_{2l-1},j_{2l}} \Big) \notag \\
&~ \mathrm{E}(\mathcal{Y}_{k-n_x,j_{1},j_{2}}\mathcal{Y}_{k-n_x,j_{3},j_{4}} )\times \mathrm{E}(\mathcal{Y}_{k-n_x,j_{5},j_{6}}\mathcal{Y}_{k-n_x,j_{7},j_{8}} ).
\end{align*}
Therefore $\mathrm{var}(\mathbb{T}_{k,a_1,a_2,s_1,s_2})=\mathrm{E}(\mathbb{T}_{k,a_1,a_2,s_1,s_2}^2)-\{\mathrm{E}(\mathbb{T}_{k,a_1,a_2,s_1,s_2})\}^2$ is derived. We note that the form of $\mathrm{var}(\mathbb{T}_{k,a_1,a_2,s_1,s_2})$ is very similar to the $\mathrm{var}(\mathbb{T}_{k,a_1,a_2})$ in Section \ref{sec:pdfcltpart1}.
In particular, we can write $\mathcal{Z}_{i,j_1,j_2}=\mathcal{X}_{i,j_1,j_2}$ if $i\leq n_x$ and  $\mathcal{Z}_{i,j_1,j_2}=\mathcal{Y}_{i-n_x,j_1,j_2}$ if $i>n_x$. 
Then we let $\mathbf{q}^{(l)}=( \mathbf{i}^{(l)}, \tilde{\mathbf{w}}^{(l)})$ to be a joint index tuple of $  \mathbf{i}^{(l)}$ and ${\mathbf{w}}^{(l)}$, where $ \tilde{\mathbf{w}}^{(l)}$ is transformed from ${\mathbf{w}}^{(l)}$  by adding each index  with $n_x.$ 
Also let $\mathbf{1}_{\tilde{E}}$ be an indicator function of the event that $(\{\mathbf{q}^{(1)}\}\cup \{\mathbf{q}^{(2)}\})\cap (\{\mathbf{q}^{(3)}\}\cup \{\mathbf{q}^{(4)}\})=\emptyset$. 
Then define
\begin{align*}
	G_{a_1,a_2,2}=&~\prod_{l=1}^4c(n,a_l)	\sum_{\substack{1\leq j_1,j_2,j_3,j_4,j_5,j_6,j_7,j_8\leq p; \\ \mathbf{i}^{(l)}\in \mathcal{P}(n_x,s_l):\, l=1,2,3,4;\\ \mathbf{w}^{(l)}\in \mathcal{P}(k-n_x-1,a_l-s_l-1):\, l=1,2,3,4}}\times \mathbf{1}_{\tilde{E}} \notag \\
	&~\times \mathrm{E}\Big\{\mathbb{M}_{\mathbf{x},\mathbf{y},2}(k,\mathbf{i}^{(l)},j_{2l-1},j_{2l}: l=1,2,3,4)\Big\}.
\end{align*} 
Similarly to Section \ref{sec:pdfcltpart1}, we also note that
\begin{align}
&~\mathrm{E}\Big\{\mathbb{M}_{\mathbf{x},\mathbf{y},2}(k,\mathbf{i}^{(l)},j_{2l-1},j_{2l}: l=1,2,3,4)\Big\}\times \mathbf{1}_{\tilde{E}} \notag \\
=&~	\mathrm{E}\Big\{\mathbb{M}_{\mathbf{x},\mathbf{y},2}(k,\mathbf{i}^{(l)},j_{2l-1},j_{2l}: l=1,2)\Big\} \notag \\
	&\quad \times \mathrm{E}\Big\{\mathbb{M}_{\mathbf{x},\mathbf{y},2}(k,\mathbf{i}^{(l)},j_{2l-1},j_{2l}: l=3,4)\Big\} \times \mathbf{1}_{\tilde{E}}. \notag 
\end{align}
Given Conditions \ref{cond:twosamplecov1} and \ref{cond:twosamplecov2}, we know that similarly   to Section  \ref{sec:pdfcltpart1}, we can show $|\{\mathrm{E}(\mathbb{T}_{k,a_1,a_2,s_1,s_2})\}^2-G_{a_1,a_2,2}|=o(n^{-2})$ and $|\mathrm{E}(\mathbb{T}_{k,a_1,a_2,s_1,s_2}^2 )-G_{a_1,a_2,2}|=o(n^{-2})$ respectively.  Finally we obtain $\mathrm{var}(\mathbb{T}_{k,a_1,a_2,s_1,s_2} )=o(n^{-2})$. The proof is very similar and the details is thus skipped. 


%
%
%
%
%
%

\smallskip

\subsubsection{Proof of Lemma \ref{lm:twosamvario2} (on Page \pageref{lm:twosamvario2}, Section \ref{sec:firstpfthmtwoclt})} \label{sec:pftwosamvario2}
Recall the form of $D_{n,k}$ derived in Section \ref{lm:pftwosampiform}:  
\begin{align*}
 	\sum_{k=1}^n\mathrm{E}(D_{n,k}^4)=\sum_{k=1}^n \sum_{{1\leq r_1,r_2,r_3,r_4\leq m}}\prod_{l=1}^4 t_{r_l}\times \mathrm{E}\Big(\prod_{l=1}^4 A_{n,k,a_{r_l}}\Big). 
 \end{align*} 
To prove  Lemma \ref{lm:twosamvario2}, it suffices to show that for given $1\leq k \leq n$ and $1\leq r_1,r_2,r_3,r_4\leq m$, we have $\mathrm{E}(\prod_{l=1}^4 A_{n,k,a_{r_l}})=o(n^{-1})$. In addition, by the Cauchy-Schwarz inequality, it suffices to show $\mathrm{E}(A_{n,k,a}^4)=o(n^{-1})$ for each given finite $a$. 

\paragraph{Part I: $1\leq k \leq n_x$}\label{sec:pfpartidnk4}
We consider without loss of generality that $k\geq a$ and 
\begin{align*}
	\mathrm{E}\Big(\prod_{l=1}^4 A_{n,k,a}^4\Big)=&~c^4(n,a) \sum_{\substack{1\leq j_1,\ldots, j_8 \leq p;\\ \mathbf{i}^{(l)}\in \mathcal{P}(k-1,a-1),\, l=1,\ldots, 4}} 	\mathrm{E}\Big(\prod_{l=1}^4 \prod_{t_l=1}^{a-1} \mathcal{X}_{i_{t_l}^{(l)} ,j_{2l-1},j_{2l} }  \Big)\notag \\
	&~\times \mathrm{E}\Big(\prod_{l=1}^4\mathcal{X}_{j_{2l-1},j_{2l}}\Big). 
\end{align*}
As $\mathrm{E}(\mathcal{X}_{j_1,j_2})=0$ under $H_0$, we know  
\begin{align*}
	\mathrm{E}\Big(\prod_{l=1}^4 \prod_{t_l=1}^{a-1} \mathcal{X}_{i_{t_l}^{(l)},j_{2l-1},j_{2l}} \Big)\neq 0
\end{align*} only when $|\cup_{l=1}^4\{\mathbf{i}^{(l)} \}|\leq 2(a-1)$. 
Note that $c(n,a)=\Theta(1)\{ n^a\sigma(a)\}^{-1}$. To finish the proof, it suffices to show that for given $(\mathbf{i}^{(1)},\mathbf{i}^{(2)},\mathbf{i}^{(3)},\mathbf{i}^{(4)})$, we have
\begin{eqnarray}
&&\quad \quad \ \sum_{1\leq j_1,\ldots,j_8\leq p} \mathrm{E}\Big(\prod_{l=1}^4 \prod_{t_l=1}^{a-1} \mathcal{X}_{i_{t_l}^{(l)},j_{2l-1},j_{2l}} \Big)\mathrm{E}\Big(\prod_{l=1}^4\mathcal{X}_{j_{2l-1},j_{2l}}\Big) = O(n^{2a})\sigma^4(a).\label{eq:pfgoaldnk4twocov}
\end{eqnarray}
We next prove \eqref{eq:pfgoaldnk4twocov} under Conditions \ref{cond:twosamplecov1} and \ref{cond:twosamplecov2} in the following Sections \ref{sec:dnk4condmix} and \ref{sec:pfconda6clt2k1}, respectively.
\smallskip

\subparagraph{Under Condition \ref{cond:twosamplecov1}} \label{sec:dnk4condmix}
Recall that $\mathcal{X}_{i,j_1,j_2}=x_{i,j_1}x_{i,j_2}-\sigma_{j_1,j_2}$.  By the symmetricity of the $j$ indexes, we have
\begin{align*}
&\sum_{1\leq j_1,\ldots,j_8\leq p} \Big| \mathrm{E}\Big(\prod_{l=1}^4\mathcal{X}_{j_{2l-1},j_{2l}}\Big)\Big|\leq C\sum_{1\leq j_1,\ldots,j_8\leq p}\Biggr\{\ \Big| \mathrm{E}\Big(\prod_{l=1}^8 x_{1,j_l}\Big)  \Big| \notag \\
&~\quad +\Big| \mathrm{E}\Big(\prod_{l=1}^6 x_{1,j_l} \Big)\sigma_{j_7,j_8} \Big|+ \Big|\mathrm{E}\Big(\prod_{l=1}^4 x_{1,j_l}\Big)\sigma_{j_5,j_6}\sigma_{j_7,j_8} \Big|+ \Big|\prod_{l=1}^4\sigma_{j_{2l-1},\, j_{2l}} \Big|\ \Biggr\}.
\end{align*}
Under Condition \ref{cond:twosamplecov1} with the mixing-type assumption, 
following similar analysis in Sections \ref{sec:pfvartaamigixing} and \ref{pf:dnk4cond21pfcompl}, we know  $\sum_{1\leq j_1,\ldots,j_8\leq p} |\mathrm{E}(\prod_{l=1}^8 x_{1,j_l})|,$ $\sum_{1\leq j_1,\ldots,j_8\leq p}| \mathrm{E}(\prod_{l=1}^6 x_{1,j_l})\sigma_{j_7,j_8}|$, $\sum_{1\leq j_1,\ldots,j_8\leq p}  |\mathrm{E}(\prod_{l=1}^4 x_{1,j_l})\sigma_{j_5,j_6}\sigma_{j_7,j_8} |$ and $\sum_{1\leq j_1,\ldots,j_8\leq p} |\prod_{l=1}^4\sigma_{j_{2l-1},\, j_{2l}} |$ are all $O(p^4)$. It follows that 
\begin{align}
	\sum_{1\leq j_1,\ldots,j_8\leq p} 
	\Big|\mathrm{E}\Big(\prod_{l=1}^4\mathcal{X}_{j_{2l-1},j_{2l}}\Big) \Big|=O(p^4), \label{eq:p4ordtwocvpartial}
\end{align} 
Recall that Lemma \ref{lm:twosampvartest} shows that $\sigma^2(a)=\Theta(p^2n^{-a})$. By \eqref{eq:p4ordtwocvpartial} and Condition \ref{cond:twosamplecov1} (2), we have \eqref{eq:pfgoaldnk4twocov} holds and $\mathrm{E}( A_{n,k,a}^4)=o(n^{-1})$.   

\smallskip
\subparagraph{Under Condition \ref{cond:twosamplecov2}} \label{sec:pfconda6clt2k1}

By Condition \ref{cond:twosamplecov2} (3), we know that  $\mathrm{E}(\prod_{l=1}^4 \prod_{t_l=1}^{a-1} \mathcal{X}_{i_{t_l}^{(l)},j_{2l-1},j_{2l}} )\times  \mathrm{E}(\prod_{l=1}^4\mathcal{X}_{j_{2l-1},j_{2l}}) $ is a linear combination of
$\mathrm{E}(\prod_{t=1}^{4a}\sigma_{j_{g_{2t-1}},\,  j_{g_{2t}} } )$, where $\mathcal{G}=(g_1,\ldots, g_{8a})\in\{1,\ldots,8\}^{8a}$ satisfies  that $g_{2t-1}\neq g_{2t}$ for $t=1,\ldots, 4a$ and the number of $g$'s equal to $m$ is $a$ for each $m\in \{1,\ldots,8\}$. By Condition \ref{cond:twosamplecov2} (4), for given $\mathcal{G}$ satisfying the constraints, $\sum_{1\leq j_1,\ldots,j_8\leq p}\sigma_{j_{g_{2t-1}},\,  j_{g_{2t}}}=O(1)\sum_{1\leq j_1,\ldots,j_8\leq p}(\sigma_{j_1,j_2}\sigma_{j_3,j_4}\sigma_{j_5,j_6}\sigma_{j_7,j_8})^a$. Then we have
\begin{align*}
&~\sum_{1\leq j_1,\ldots,j_8\leq p}\mathrm{E}\Big(\prod_{l=1}^4 \prod_{t_l=1}^{a-1} \mathcal{X}_{i_{t_l}^{(l)},j_{2l-1},j_{2l}} \Big)\times  \mathrm{E}(\prod_{l=1}^4\mathcal{X}_{j_{2l-1},j_{2l}})\notag \\
=&~O(1)\sum_{1\leq j_1,\ldots,j_8\leq p}	(\sigma_{j_1,j_2}\sigma_{j_3,j_4}\sigma_{j_5,j_6}\sigma_{j_7,j_8})^a= O(1)\Big(\sum_{1\leq j_1,j_2\leq p}\sigma_{j_1,j_2}^a \Big)^4. 
\end{align*}Recall that Lemma \ref{lm:twosampvartest} shows that $\sigma^2(a)=\Theta(n^{-a})(\sum_{1\leq j_1,j_2\leq p}\sigma_{j_1,j_2}^a )^2$. Therefore, \eqref{eq:pfgoaldnk4twocov} is obtained and Lemma \ref{lm:twosamvario2} is proved.  


\smallskip

\paragraph{Part II: $n_x+1\leq k \leq n_x+n_y$}
Section \ref{sec:part2sumankform} derives that  $A_{n,k,a}=\sum_{s=L_k}^{a-1}A_{n,k,a,s}$, where
\begin{align*}
	A_{n,k,a,s} =&~ \sum_{\substack{1\leq j_1,j_2\leq p;\\ \mathbf{i}\in \mathcal{P}(n_x,s); \\ \mathbf{w}\in \mathcal{P}(k-n_x-1, a-s-1)} } c_2(n,a,s)\mathcal{Y}_{k-n_x,j_1,j_2}  \prod_{t=1}^s \mathcal{X}_{i_t,j_1,j_2}  \prod_{\tilde{t}=1}^{a-s-1} \mathcal{Y}_{w_{\tilde{t}},j_1,j_2}. 
\end{align*}
Similarly to Section \ref{sec:pfpartidnk4}, it suffices to show that for given finite integers $a$ and $s$, $\mathrm{E}(  A_{n,k,a,s}^4)=o(n^{-1}).$ 
Following the arguments in  Section \ref{sec:pdfcltpart2}, we know $A_{n,k,a,s}$ takes a similar form to $A_{n,k,a}$ in Section \ref{sec:pfpartidnk4}. Therefore the proof in Section \ref{sec:dnk4condmix} can be applied  similarly to show $	\mathrm{E}( A_{n,k,a,s}^4)=o(n^{-1})$ in this section. The proof will be very similar and the details are thus skipped.  

\subsection{Lemmas for the proof of Theorem \ref{thm:twosamaltclt}} \label{sec:twocovalt}

\subsubsection{Proof of Lemma \ref{lm:twosamaltsmall1} (on Page \pageref{lm:twosamaltsmall1}, Section \ref{sec:pfthm49alttwo})}  \label{sec:pftwocovaltlm}
In this section, to prove Lemma \ref{lm:twosamaltsmall1}, we study $\mathrm{var}(T_{D,a,1})$, $\mathrm{var}(T_{D,a,2})$ and $\mathrm{var}\{\tilde{\mathcal{U}}^*(a) \}$ respectively. 

\vspace{0.3em}
\subparagraph*{Part I: $\mathrm{var}(T_{D,a,1})$} 
We first derive $\mathrm{var}(T_{D,a,1})$. Note that $T_{D,a,1}$ is a summation over $j$ indexes in $\mathbb{J}_0$, and $\sigma_{x,j_1,j_2}=\sigma_{y,j_1,j_2}$ for $j_1,j_2\in \mathbb{J}_0$. Following the arguments in Section \ref{sec:pftwosampvartest}, similarly to \eqref{eq:varleadtwosam},  we have 
\begin{align*}
	\mathrm{var}(T_{D,a,1})\simeq &\sum_{1\leq j_1,j_2,j_3,j_4\in \mathbb{J}_0} a!(\mathbf{X}_{j_1,j_2,j_3,j_4}/n_x+ \mathbf{Y}_{j_1,j_2,j_3,j_4}/n_y )^a. \notag
\end{align*} 
By Condition \ref{cond:twosamaltdist} (3), \eqref{eq:twosamcovexpterm4th} still  holds. Then by Condition \ref{cond:twosamaltmoment} and the symmetricity of $j$ indexes,
\begin{eqnarray}
\quad \mathrm{var}(T_{D,a,1})\simeq 	 C_{\kappa,a} \label{eq:twocovvarord}  \sum_{1\leq j_1,j_2,j_3,j_4\in \mathbb{J}_0} a! \sigma_{j_1,j_2}^a\sigma_{j_3,j_4}^a	, 
\end{eqnarray} where $C_{\kappa,a} =\{(\kappa_x-1)/n_x+(\kappa_y-1)/n_y\}^a+ 2(\kappa_x/n_x+\kappa_y/n_y)^a$, and $\mathrm{var}(T_{D,a,1})$ is of order $\Theta(n^{-a}\mathbb{V}_{a,a,0,0}^{1/2})$ with $\mathbb{V}_{a,a,0,0}^{1/2}=\sum_{j_1,\ldots,j_4\in \mathbb{J}_0}(\sigma_{x,j_1,j_2}\sigma_{x,j_3,j_4})^a$ defined on Page \pageref{eq:condtwosamdef}. 

\vspace{0.3em}
\subparagraph*{Part II: $\mathrm{var}(T_{D,a,2})$}
We show $\mathrm{var}(T_{D,a,2})=o(1)\mathrm{var}(T_{D,a,1})$. Particularly,
\begin{align*}
T_{D,a,2}=\sum_{(j_1,j_2)\in J_{0,D}} \frac{1}{P^{n_x}_a P^{n_y}_a} \sum_{ \substack{ \mathbf{i}\in \mathcal{P}(n_x,a),\\ \mathbf{w}\in \mathcal{P}(n_y,a) }} \prod_{t=1}^a  (\mathcal{X}_{i_{t},j_1,j_2} -\mathcal{Y}_{w_{t},j_1,j_2}),
\end{align*} 
where we redefine $\mathcal{X}_{i,j_1,j_2}=x_{i,j_1}x_{i,j_2}-\sigma_{y,j_1,j_2}$ and $\mathcal{Y}_{i,j_1,j_2}=y_{i,j_1}y_{i,j_2}-\sigma_{y,j_1,j_2}$. 
Moreover, we define
\begin{align*}
	G_{D,a}=\sum_{(j_1,j_2), (j_3,j_4)\in J_{0,D}}(P^{n_x}_a P^{n_y}_a)^{-2}\sum_{ \substack{ \mathbf{i},\, \tilde{\mathbf{i}} \in \mathcal{P}(n_x,a),\\ \mathbf{w},\, \tilde{\mathbf{w}}\in \mathcal{P}(n_y,a) }} \mathbf{1}_{\{\{\mathbf{i}\}\cap \{\tilde{\mathbf{i}} \}=\emptyset\} }(D_{j_1,j_2}D_{j_3,j_4})^a. 
\end{align*}
To prove $\mathrm{var}(T_{D,a,2})=\mathrm{E}(T_{D,a,2}^2)-\{\mathrm{E}(T_{D,a,2})\}^2$ is $o(1)\mathrm{var}(T_{D,a,1})$, we next show $|\mathrm{E}(T_{D,a,2}^2)-G_{D,a}|$ and $|\{\mathrm{E}(T_{D,a,2})\}^2-G_{D,a}|$ are both $o(1)\mathrm{var}(T_{D,a,1})$. 

Note that $\mathrm{E}(\mathcal{X}_{i,j_1,j_2})=D_{j_1,j_2}$ and $\mathrm{E}(\mathcal{Y}_{i,j_1,j_2})=0$. We have
\begin{align*}
	\{\mathrm{E}(T_{D,a,2})\}^2=\sum_{(j_1,j_2), (j_3,j_4)\in J_{0,D}}(P^{n_x}_a P^{n_y}_a)^{-2}\sum_{ \substack{ \mathbf{i},\, \tilde{\mathbf{i}} \in \mathcal{P}(n_x,a),\\ \mathbf{w},\, \tilde{\mathbf{w}}\in \mathcal{P}(n_y,a) }}(D_{j_1,j_2}D_{j_3,j_4})^a.
\end{align*} Then
\begin{align*}
	&~ |\{\mathrm{E}(T_{D,a,2})\}^2-G_{D,a}| \notag \\
	\leq &~\Big|\sum_{(j_1,j_2), (j_3,j_4)\in J_{0,D}}(P^{n_x}_a P^{n_y}_a)^{-2}\sum_{ \substack{ \mathbf{i},\, \tilde{\mathbf{i}} \in \mathcal{P}(n_x,a),\\ \mathbf{w},\, \tilde{\mathbf{w}}\in \mathcal{P}(n_y,a) }} \mathbf{1}_{\{\{\mathbf{i}\}\cap \{\tilde{\mathbf{i}} \}\neq \emptyset\} }(D_{j_1,j_2}D_{j_3,j_4})^a\Big| \notag \\
	\leq & Cn^{-1}\sum_{(j_1,j_2), (j_3,j_4)\in J_{0,D}}|D_{j_1,j_2}D_{j_3,j_4}|^a,
\end{align*} where we use $\sum_{ \substack{ \mathbf{i},\, \tilde{\mathbf{i}} \in \mathcal{P}(n_x,a), \mathbf{w},\, \tilde{\mathbf{w}}\in \mathcal{P}(n_y,a) }} \mathbf{1}_{\{\{\mathbf{i}\}\cap \{\tilde{\mathbf{i}} \}\neq \emptyset\} }=O(n^{4a-1})$. 
In addition, 
\begin{align*}
	&~ |\mathrm{E}(T_{D,a,2}^2)-G_{D,a}| \notag \\
\leq &~	C\sum_{(j_1,j_2), (j_3,j_4)\in J_{0,D}}(P^{n_x}_a P^{n_y}_a)^{-2}\sum_{ \substack{ \mathbf{i},\, \tilde{\mathbf{i}} \in \mathcal{P}(n_x,a),\\ \mathbf{w},\, \tilde{\mathbf{w}}\in \mathcal{P}(n_y,a) }}  \notag \\
&\Biggr(\mathbf{1}_{\{\{\mathbf{i}\}\cap \{\tilde{\mathbf{i}} \}= \emptyset\} }\Big| \mathrm{E}\Big\{  \prod_{t=1}^a  (\mathcal{X}_{i_{t},j_1,j_2} -\mathcal{Y}_{w_{t},j_1,j_2})(\mathcal{X}_{\tilde{i}_{t},j_3,j_4} -\mathcal{Y}_{\tilde{w}_{t},j_3,j_4})\Big\}-(D_{j_1,j_2}D_{j_3,j_4})^a\Big| \notag \\
&~\ +\mathbf{1}_{\{\{\mathbf{i}\}\cap \{\tilde{\mathbf{i}} \}\neq \emptyset\} }\Big|\mathrm{E}\Big\{  \prod_{t=1}^a  (\mathcal{X}_{i_{t},j_1,j_2} -\mathcal{Y}_{w_{t},j_1,j_2})(\mathcal{X}_{\tilde{i}_{t},j_3,j_4} -\mathcal{Y}_{\tilde{w}_{t},j_3,j_4})  \Big\}\Big| \Biggr). \notag 
\end{align*}
We redefine $\mathbf{X}_{j_1,j_2,j_3,j_4}=\mathrm{E}(\mathcal{X}_{i,j_1,j_2}\mathcal{X}_{i,j_3,j_4})$ and $\mathbf{Y}_{j_1,j_2,j_3,j_4}=\mathrm{E}(\mathcal{Y}_{i,j_1,j_2}\mathcal{Y}_{i,j_3,j_4})$. Then
\begin{align*}
	&~|\mathrm{E}(T_{D,a,2}^2)-G_{D,a}| \notag \\
	\leq &~ C \sum_{1\leq  m_1+m_2 \leq a} n^{-m_1-m_2} \notag \\
	&~\quad \times  \sum_{(j_1,j_2), (j_3,j_4)\in J_{0,D}} \Big| \mathbf{X}_{j_1,j_2,j_3,j_4}^{m_1} \mathbf{Y}^{m_2}_{j_1,j_2,j_3,j_4} (D_{j_1,j_2}D_{j_3,j_4})^{a-m_1-m_2}\Big|. 
\end{align*}
\noindent 
Note that $\mathbf{Y}_{j_1,j_2,j_3,j_4}=\sigma_{y,j_1,j_3}\sigma_{y,j_2,j_4}+\sigma_{y,j_1,j_4}\sigma_{y,j_2,j_3}$ and $\sigma_{y,j_1,j_2}=\sigma_{x,j_1,j_2}-D_{j_1,j_2}$. 
By Conditions \ref{cond:twosamaltdist} and \ref{cond:twosamaltmoment}, the H\"older's inequality and definitions in \eqref{eq:condtwosamdef}, we have
\begin{align*}
\mathrm{var}(T_{D,a,2})\leq &~ C \max_{\substack{\mathcal{H}\in \mathbb{H},\\ t=1,2 }}\Big\{ \sum_{m=1}^a(n^{-a}\mathbb{V}_{a,\mathcal{H},x ,t})^{m/a}(\mathbb{V}_{a, \mathcal{H}, D,3})^{1-m/a},\, n^{-1}\mathbb{V}_{a, \mathcal{H}, D,3} \Big\}.
\end{align*}
Therefore by Condition \ref{cond:twosamaltmoment} and \eqref{eq:twocovvarord}, $\mathrm{var}(T_{D,a,2})=o(1)n^{-a} \mathbb{V}_{a, a,0,0 }^{1/2}=o(1)\mathrm{var}(T_{D,a,1})$.

\vspace{0.3em}

\subparagraph*{Part III: $\mathrm{var}\{\tilde{\mathcal{U}}^*(a) \}$}
Last, we prove $\mathrm{var}\{\tilde{\mathcal{U}}^*(a) \}=o(1)\mathrm{var}(T_{D,a,1})$. Similarly to Section \ref{sec:pftwosampvartest}, we write $\tilde{\mathcal{U}}^*(a)=\sum_{c=0}^a \sum_{b_1=0}^c \sum_{b_2=0}^{a-c} C_{a,c,b_1,b_2}\times T_{b_1,b_2,c}  \mathbf{1}_{b_1+b_2\leq a-1}$, where $T_{b_1,b_2,c} $ is defined in \eqref{eq:tb1b2c}. For finite $a$, to prove $\mathrm{var}\{\tilde{\mathcal{U}}^*(a) \}=o(1)\mathrm{var}(T_{D,a,1})$, it suffices to prove $\mathrm{var}( T_{b_1,b_2,c})=o(1)\mathrm{var}(T_{D,a,1})$ for $0\leq c\leq a$ and $b_1+b_2\leq a-1$. As $\mathrm{E}(\mathcal{Y}_{i,j_1,j_2})=0$ and $\mathrm{E}(\mathbf{x})=\mathrm{E}(\mathbf{y})=0$, we know that if $b_1+b_2\leq a-1$, $\mathrm{E}(T_{b_1,b_2,c})=0$. Then $\mathrm{var}(T_{b_1,b_2,c})=\mathrm{E}(T_{b_1,b_2,c}^2)$, which takes a similar form to \eqref{eq:vartb1b2c}. Specifically, we can write $\mathrm{var}(T_{b_1,b_2,c})=\mathrm{var}(T_{b_1,b_2,c})_{(1)}+\mathrm{var}(T_{b_1,b_2,c})_{(2)}$, where
\begin{align*}
	\mathrm{var}(T_{b_1,b_2,c})_{(1)}=&~\sum_{\substack{j_1, j_2, \tilde{j}_1, \tilde{j}_2 \in \mathbb{J}_0}}(P^{n_x}_{2c-b_1}P^{n_y}_{2(a-c)-b_2})^{-2} \ \sum_{\substack{\mathbf{i},\, \tilde{\mathbf{i}}\in \mathcal{P}(n_x,2c-b_1);\\ \mathbf{w}\, \tilde{\mathbf{w}}\in \mathcal{P}(n_y,2(a-c)-b_2) }}   \notag \\
	&~\quad \mathbb{T}(\mathbf{i},\tilde{\mathbf{i}},\mathbf{w},\tilde{\mathbf{w}},j_1,j_2,\tilde{j}_1,\tilde{j}_2),  \notag 
\end{align*}  and 
\begin{align*}
	\mathrm{var}(T_{b_1,b_2,c})_{(2)}=&~\sum_{\substack{(j_1, j_2),\\ (\tilde{j}_1, \tilde{j}_2) \in J_{0,D}}}(P^{n_x}_{2c-b_1}P^{n_y}_{2(a-c)-b_2})^{-2} \ \sum_{\substack{\mathbf{i},\, \tilde{\mathbf{i}}\in \mathcal{P}(n_x,2c-b_1);\\ \mathbf{w}\, \tilde{\mathbf{w}}\in \mathcal{P}(n_y,2(a-c)-b_2) }}   \notag \\
	&~\quad \mathbb{T}(\mathbf{i},\tilde{\mathbf{i}},\mathbf{w},\tilde{\mathbf{w}},j_1,j_2,\tilde{j}_1,\tilde{j}_2),  \notag 
\end{align*} and $\mathbb{T}(\mathbf{i},\tilde{\mathbf{i}},\mathbf{w},\tilde{\mathbf{w}},j_1,j_2,\tilde{j}_1,\tilde{j}_2)$ is defined same as in \eqref{eq:vartb1b2c}. 

Note that $\mathrm{var}(T_{b_1,b_2,c})_{(1)}$ is a summation over $j$ indexes in $\mathbb{J}_0$, and $\sigma_{x,j_1,j_2}=\sigma_{y,j_1,j_2}$ for $j_1,j_2\in \mathbb{J}_0$. Therefore the arguments under $H_0$ in Section \ref{sec:pftwosampvartest} can be applied similarly to $\mathrm{var}(T_{b_1,b_2,c})_{(1)}$. Then we have $\mathrm{var}(T_{b_1,b_2,c})_{(1)}=o(n^{-a})(\sum_{j_1,j_2\in \mathbb{J}_0}\sigma_{j_1,j_2}^a)^2$ which is $o(1)\mathrm{var}(T_{D,a,1})$. We next consider $\mathrm{var}(T_{b_1,b_2,c})_{(2)}$. As $\mathrm{E}(\mathcal{Y}_{i,j_1,j_2})=0$ and $\mathrm{E}(\mathbf{x})=\mathrm{E}(\mathbf{y})=0$, by the definition in \eqref{eq:vartb1b2c}, we know  $\mathrm{E}\{\mathbb{T}(\mathbf{i},\tilde{\mathbf{i}},\mathbf{w},\tilde{\mathbf{w}},j_1,j_2,\tilde{j}_1,\tilde{j}_2) \}\neq 0$ only when $\{i_{b_1+1},\ldots,i_{2c-b_1}\}=\{\tilde{i}_{b_1+1},\ldots,\tilde{i}_{2c-b_1}\}$ and $\{\mathbf{w}\}=\{\tilde{\mathbf{w}}\}.$  
 Let $m_0=b_1-| \{i_1,\ldots,i_{b_1}\}\cap  \{\tilde{i}_1,\ldots,\tilde{i}_{b_1}\}|$.  By Condition \ref{cond:twosamaltdist} (3) and the  H\"older's inequality, 
\begin{align*}
&~ \mathrm{var}(T_{b_1,b_2,c})_{(2)}\notag \\
\leq &~ Cn_x^{-(c-b_1)}n_y^{-(a-c-b_2)}\max_{\substack{ \mathcal{H}\in \mathbb{H}_0,\\ 0\leq m_0\leq b_1}} \Big\{ \Big(n_y^{-a}\sum_{ (j_1,j_2),(j_3,j_4) \in J_{0,D} }| \sigma_{y,j_{h_1},j_{h_2}}\sigma_{y,j_{h_3},j_{h_4}}|^{a} \Big)^{\frac{a-c}{a}}\notag \\
&~\quad \times \Big(n_x^{-a}\sum_{ (j_1,j_2),(j_3,j_4) \in J_{0,D} }| \sigma_{x,j_{h_1},j_{h_2}}\sigma_{x,j_{h_3},j_{h_4}}|^{a} \Big)^{\frac{c-m_0}{a}} \notag \\
&~\quad  \times \Big( \sum_{ (j_1,j_2),(j_3,j_4) \in J_{0,D} } |D_{j_{h_1},j_{h_2}}D_{j_{h_3},j_{h_4}}|^a\Big)^{\frac{m_0}{a}}\Big\}    \notag \\
\leq &~Cn^{-(a-b_1-b_2)}\max_{\substack{ \mathcal{H}\in \mathbb{H}_0, t=1,2}} \{ n^{-a}\mathbb{V}_{a,\mathcal{H},x,t}, \mathbb{V}_{a,\mathcal{H},D,3} \},
\end{align*} where the last inequality uses $\sigma_{y,j_{1},j_{2}}=\sigma_{x,j_{1},j_{2}}-D_{j_1,j_2}$. As   $b_1+b_2\leq a-1$, $
\mathrm{var}(T_{b_1,b_2,c})_{(2)}\leq Cn^{-1}	\max_{\substack{ \mathcal{H}\in \mathbb{H}_0;  t=1,2}} \{ n^{-a}\mathbb{V}_{a,\mathcal{H},x,t}, \mathbb{V}_{a,\mathcal{H},D,3} \}.
$
By Condition \ref{cond:twosamaltmoment} and \eqref{eq:twocovvarord}, we  know $\mathrm{var}( T_{b_1,b_2,c})_{(2)}=o(1)\mathrm{var}(T_{D,a,1})$.

\subsection{Proof of Remark \ref{rm:standizedmax}} \label{sec:pfstandmax}
In this section, we prove the conclusion in Remark \ref{rm:standizedmax}. To be specific, we prove in the following that under the conditions of Theorem \ref{thm:asymindpt}, 
\begin{align}
	& \Big| P\Big( n  (M_n^{\dag})^2 > y_p, \frac{{\mathcal{U}}(a_1)  }{\sigma(a_1)} \leq 2z_1, \ldots,  \frac{{\mathcal{U}}(a_m)  }{\sigma(a_m)} \leq 2z_m\Big) \label{eq:goalaltstandmaxindp} \\ 
	&\ -  P\Big( n  (M_n^{\dag})^2 > y_p \Big) \prod_{r=1}^m P \Big( \frac{{\mathcal{U}}(a_r)  }{\sigma(a_r)} \leq 2z_r \Big)\Big| \to 0. \notag
\end{align}


Note that we already know $ M_n/n$ and $\mathcal{U}(a_r)/\sigma(a_r)$'s for $r=1,\ldots,m$ are asymptotically independent by the proof of Lemmas \ref{lm:boundedsmalldiff} and \ref{lm:addprop0termtrue}. In this section, the proof idea  is that we  show the difference between  $n(M_n^{\dagger})^2$ and $M_n/n$ is $o_p(1)$ and then obtain \eqref{eq:goalaltstandmaxindp}. To prove that   $n(M_n^{\dagger})^2-M_n/n$ is $o_p(1)$, we introduce an intermediate variable $\tilde{M}_n/n$ defined below, and show that $\tilde{M}_n/n-n(M_n^{\dagger})^2=o_p(1)$ and $\tilde{M}_n/n-M_n/n=o_p(1)$ respectively.


Specifically, we define $$\tilde{M}_n/n=\max_{1\leq j_1\neq j_2\leq p} | n  \hat{\sigma}_{j_1,j_2}^2/{\theta}_{j_1,j_2} |,$$ where  $\hat{\sigma}_{j_1,j_2}=\sum_{i=1}^n\{(x_{i,j_1}-\bar{x}_{j_1})(x_{i,j_2}-\bar{x}_{j_2})\}/n$ and $\theta_{j_1,j_2}=\mathrm{var}\{(x_{i,j_1}-\mu_{j_1})(x_{i,j_2}-\mu_{j_2})\}$. Moreover, by  \eqref{eq:indpalldefinewvar}, we have $$M_n/n=\max_{1\leq j_1\neq j_2\leq p} | n  \tilde{\sigma}_{j_1,j_2}^2/{\theta}_{j_1,j_2} |,$$ where we use the fact that $\theta_{j_1,j_2}=\sigma_{j_1,j_1} \sigma_{j_2,j_2}$ by Condition \ref{cond:maxiidcolumn} and define $\tilde{\sigma}_{j_1,j_2}=\sum_{i=1}^n\{(x_{i,j_1}-\mu_{j_1})(x_{i,j_2}-\mu_{j_2})\}/n$. In addition, we have $$M_n^{\dagger}=\max_{1\leq j_1\neq j_2\leq p} |   \hat{\sigma}_{j_1,j_2}| /(\hat{\theta}_{j_1,j_2})^{1/2},$$ where we let 
$
	\hat{\theta}_{j_1,j_2}=\widehat{\mathrm{var}}(\hat{\sigma}_{j_1,j_2})=n^{-1}\sum_{i=1}^n\{(x_{i,j_1}-\bar{x}_{j_1})(x_{i,j_2}-\bar{x}_{j_2})-\hat{\sigma}_{j_1,j_2}\}^2.
$ 
In the following, we will first compare $\tilde{M}_n/n$ and $n(M_n^{\dagger})^2$, and then compare  $\tilde{M}_n/n$ and $M_n/n$. Also for simplicity, we assume without loss of generality that $\mu_j=0$ and $\sigma_{j,j}=1$.

Note that   $n(M_n^{\dagger})^2=\max_{1\leq j_1\neq j_2\leq p} | n  \hat{\sigma}_{j_1,j_2}^2/\hat{\theta}_{j_1,j_2} |$, which differs from $\tilde{M}_n/n$ only by replacing $\theta_{j_1,j_2}$ with $\hat{\theta}_{j_1,j_2}$. 
By the proof of Lemma 3 in \cite{cai2013two}, we know that for any $C_2>0$, there exists some constant $C_1$ such that
\begin{align*}
	P\Big( \max_{1\leq j_1\neq j_2\leq p} |\hat{\theta}_{j_1,j_2}-\theta_{j_1,j_2}|/\theta_{j_1,j_2} \geq C_1\sqrt{\log p/n} \Big)=O(p^{-C_2}).
\end{align*}
%
%
Under the event $|\hat{\theta}_{j_1,j_2}/\theta_{j_1,j_2}-1|\leq C_1\sqrt{\log p /n}$, we have  
\begin{align*}
&~ |  \tilde{M}_n/n - n(M_n^{\dag})^2 |\notag \\
= &~ \Big|\max_{1\leq j_1\neq j_2\leq p} n\hat{\sigma}_{j_1,j_2}^2/\theta_{j_1,j_2}-\max_{1\leq j_1\neq j_2\leq p} n \hat{\sigma}_{j_1,j_2}^2/\hat{\theta}_{j_1,j_2} \Big| \notag \\
\leq &~ \max_{1\leq j_1\neq j_2\leq p} | n  \hat{\sigma}_{j_1,j_2}^2/{\theta}_{j_1,j_2} | \times \max_{1\leq j_1\neq j_2\leq p} |1 - \theta_{j_1,j_2}/\hat{\theta}_{j_1,j_2} | \notag \\
\leq &~ \max_{1\leq j_1\neq j_2\leq p} | n  \hat{\sigma}_{j_1,j_2}^2/{\theta}_{j_1,j_2} | C_1\sqrt{\log p/n}.
\end{align*} 
It follows that $n(M_n^{\dag})^2=\tilde{M}_n/n\{1+O(\sqrt{\log p/n})\}$.  Since $\log p /n \to 0$ and $\tilde{M}_n/n$ has a limit by   Theorem 3 in \citet{cai2011}, then $|\tilde{M}_n/n-n(M_n^{\dag})^2|=o_p(1)$. 


We next  compare  $\tilde{M}_n/n$ and $M_n/n$.  by Lemma \ref{lm:maxdiffbound}, 
\begin{align*} 
&~ |\tilde{M}_n/n - M_n/n| \notag \\
\leq &~ C \max_{1\leq j_1\neq j_2\leq p} \Big|     \sum_{i=1}^n(x_{i,j_1}-\bar{x}_{j_1} )(x_{i,j_2}-\bar{x}_{j_2})-\sum_{i=1}^nx_{i,j_1}x_{i,j_2} \Big|^2\Big/n \notag \\
& ~+ C\sqrt{  M_n/n} \max_{1\leq j_1\neq j_2\leq p} \Big|     \sum_{i=1}^n(x_{i,j_1}-\bar{x}_{j_1} )(x_{i,j_2}-\bar{x}_{j_2})-\sum_{i=1}^nx_{i,j_1}x_{i,j_2} \Big|\Big/\sqrt{n} \notag \\
\leq &~C\max_{1\leq j\leq p} n\bar{x}_j^4+Cn^{1/2}\sqrt{M_n/n}\max_{1\leq j\leq p} \bar{x}_{j}^2,
\end{align*} where in the last inequality we use $ \max_{1\leq j_1\neq j_2\leq p} \bar{x}_{j_1}\bar{x}_{j_2} \leq \max_{1\leq j_1\neq j_2\leq p} (\bar{x}_{j_1}^2+\bar{x}_{j_2}^2)/2 \leq \max_{1\leq j\leq p}\bar{x}_j^2$. 
By  Eq. (27) in Lemma 2 of \citet{cai2011adaptive}, we know that $\max_{1\leq j\leq p}|\bar{x}_j|=O_p(\sqrt{\log p/n}).$ Since we assume $\log p=o(n^{1/7})$, and  Proposition 6.3 in \cite{cai2011} shows that $M_n/n$ has a limit, we know  $|\tilde{M}_n/n - M_n/n| =o_p(1)$. 



In summary, $| M_n/n-n(M_n^{\dag})^2| \leq |M_n/n- \tilde{M}_n/n |+|\tilde{M}_n/n-n(M_n^{\dag})^2|=o_p(1)$. Since $| M_n/n-n(M_n^{\dag})^2| =o_p(1)$ and  $ M_n/n$ and $\mathcal{U}(a_r)/\sigma(a_r)$'s for $r=1,\ldots,m$ are asymptotically independent,   similarly to the proof of Lemma \ref{lm:addprop0termtrue}, we know \eqref{eq:goalaltstandmaxindp} is proved.

\subsection{Proof of Corollary  \ref{prop:generalrestwomean}}\label{sec:pfcor41}
Since the proofs in Sections \ref{sec:proofthm33} and  \ref{sec:proofthem34} do not rely on $\boldsymbol{\Sigma}_x=\boldsymbol{\Sigma}_y$, 
the proof of Corollary  \ref{prop:generalrestwomean} follows from Sections \ref{sec:proofthm33} and  \ref{sec:proofthem34} directly. We also obtain $\mathrm{var}\{\mathcal{U}(a)\}$ under the null and alternative hypotheses by Lemma \ref{lm:twosamvar} (on Page \pageref{lm:twosamvar}) and Lemma \ref{lm:twomeanaltvar}  (on Page \ref{lm:twomeanaltvar}), respectively.


\section{Computation \& Supplementary Simulations}\label{sec:c}
 
\subsection{Computation}

\subsubsection{Formulae for (2.15)} \label{sec:formulaevu}

Note that $\mathcal{U}_l(a)=U_l^{\mathbf{1}_a}$ by the  definitions in \eqref{eq:uvorderstat}, and for different $l$'s, the computation  methods of $U_l^{\mathbf{1}_a}$'s    are the same. Therefore  
 in the following, for simplicity, we give the formulae of $U_l^{\mathbf{1}_a}$ without the subscript $l$:
\begin{align*}
	\ U^{\mathbf{1}_1} =& V^{(1)}, \\
	\ U^{\mathbf{1}_{2}}=& V^{(1,1)}-V^{(2)}, \\
	\ U^{\mathbf{1}_{3}}=& V^{\mathbf{1}_3}-3V^{(2,1)}+2V^{(3)}, \\
	\ U^{\mathbf{1}_{4}}=&V^{\mathbf{1}_4} - 6V^{(2,1,1)}+8V^{(3,1)}+3V^{(2,2)}-6V^{(4)},\\
\ U^{\mathbf{1}_{5}} =& V^{\mathbf{1}_5} -10 V^{(2,\mathbf{1}_3)}+20V^{(3,\mathbf{1}_2)}+15V^{(2,2,1)} -30 V^{(4,1)} \notag \\
	& -20  V^{(2,3)} +24V^{(5)},  \\
\ U^{\mathbf{1}_{6}}  =& V^{\mathbf{1}_6} -15 V^{(\mathbf{1}_4,2)} +40 V^{(3,\mathbf{1}_3)} + 45V^{(1,1,2,2)}, \notag \\ 
	&-90V^{(1,1,4)} -120 V^{(1,2,3)}+144V^{(1,5)}  -15 V^{(2,2,2)}\notag \\
	&+ 90V^{(2,4)} +40 V^{(3,3)}  -120 V^{(6)},
\end{align*}	where $U^{\mathbf{1}_a}$ and $V^{(t_1,\ldots, t_k)}$ are defined as in \eqref{eq:uvorderstat}.
 
\subsubsection{Computation with unknown mean}  \label{sec:u2compute}
In this section, we provide the details of the computation of $\mathcal{U}(a)$ when $\mathrm{E}(x_{i,j})$ is unknown. 
We note that $\mathcal{U}(a)$ is some linear combination of 
\begin{align}
\sum_{1\leq i_1\neq \ldots \neq i_{k} \leq n} \prod_{t=1}^{k} x_{i_t,j_1}^{r_{t,1}}x_{i_t,j_2}^{r_{t,2}}, \label{eq:generalucomp1}	
\end{align} where $a\leq k\leq 2a$, $r_{t,1},r_{t,2}\geq 0$ and $r_{t,1}+r_{t,2}\geq 1$. 
A direct calculation of \eqref{eq:generalucomp1}	 has computational cost $O(n^k)$, which is large when $k$ is large. But following the discussion in Section \ref{sec:computtest}, we  can similarly reduce the computational cost of \eqref{eq:generalucomp1}	 to order $O(n)$ with an iterative method. In particular, we note that 
\begin{align}
& \sum_{1\leq i_1\neq \ldots \neq i_{k} \leq n} \prod_{t=1}^{k} x_{i_t,j_1}^{r_{t,1}}x_{i_t,j_2}^{r_{t,2}} \label{eq:generalucomputerelation} \\
=&\Big( \sum_{1\leq i_1\neq \ldots \neq i_{k-1} \leq n} \prod_{t=1}^{k-1} x_{i_t,j_1}^{r_{t,1}}x_{i_t,j_2}^{r_{t,2}}\Big)\Big( \sum_{i=1}^n x_{i,j_1}^{r_{k,1}} x_{i,j_2}^{r_{k,2}}\Big) \notag \\
&-\sum_{m=1}^{k-1}\sum_{1\leq i_1\neq \ldots \neq i_{k-1} \leq n}\Big(\prod_{t=1}^{k-1} x_{i_t,j_1}^{r_{t,1}}x_{i_t,j_2}^{r_{t,2}}\Big)x_{i_{m},j_1}^{r_{k},1} x_{i_{m},j_2}^{r_{k},2}. \notag	
\end{align}Suppose we can  compute $\sum_{1\leq i_1\neq \ldots \neq i_{k-1} \leq n} \prod_{t=1}^{k-1} x_{i_t,j_1}^{r_{t,1}}x_{i_t,j_2}^{r_{t,2}}$ with cost $O(n)$ for any $(r_{t,1},r_{t,2}), t=1,\ldots,k-1$. Then by the relationship in \eqref{eq:generalucomputerelation}, we can obtain \eqref{eq:generalucomp1} with cost $O(n)$ iteratively.

We then illustrate the iterative method with some examples. When $k=1$, for any given $(r_{1,1},r_{1,2})$, we know $\sum_{i=1}^n x_{i,j_1}^{r_{1,1}}x_{i,j_2}^{r_{1,2}}$ can be computed with cost $O(n)$. When $k=2$, by \eqref{eq:generalucomputerelation}, we have $\sum_{1\leq i_1\neq i_2\leq n}\prod_{t=1}^2 x_{i_t,j_1}^{r_{t,1}}x_{i_t,j_2}^{r_{t,2}}=(\sum_{i=1}^n x_{i,j_1}^{r_{1,1}} x_{i,j_2}^{r_{1,2}})(\sum_{i=1}^n x_{i,j_1}^{r_{2,1}} x_{i,j_2}^{r_{2,2}})-\sum_{i=1}^n x_{i,j_1}^{r_{1,1}+r_{2,1}} x_{i,j_2}^{r_{1,2}+r_{2,2}}$, which can be computed with cost $O(n)$. For a general $k$, suppose for any given $(r_{t,1},r_{t,2}), t=1,\ldots, k-1$, we can compute $\sum_{1\leq i_1\neq \ldots \neq i_{k-1} \leq n} \prod_{t=1}^{k-1} x_{i_t,j_1}^{r_{t,1}}x_{i_t,j_2}^{r_{t,2}}$ with cost $O(n)$. Then by \eqref{eq:generalucomputerelation}, we can obtain  \eqref{eq:generalucomp1} with computational cost $O(n)$.

Given the iterative method discussed above, we can compute $\mathcal{U}(a)$ with cost $O(p^2n)$. 
For example,  we can write $\mathcal{U}(1)$ as
\begin{align*}
\sum_{1\leq j_1\neq j_2\leq p}\Big\{n^{-1}\sum_{i=1}^nx_{i,j_1}x_{i,j_2}-(P^n_2)^{-1}\Big( \sum_{i_1=1}^nx_{i_1,j_1} \sum_{i_2=1}^n x_{i_2,j_2}-\sum_{i=1}^n x_{i,j_1}x_{i,j_2}\Big)\Big\}.
\end{align*}
For $a=2$, similar analysis holds. Note that 
\begin{align*}
\mathcal{U}(2)=\sum_{1\leq j_1\neq j_2\leq p}\Big\{ {(P^n_2)^{-1}}\mathcal{U}_1(2)-2{(P^n_3)^{-1}}\mathcal{U}_2(2)+{(P^n_4)^{-1}} \mathcal{U}_3(2)\Big\},
\end{align*} where 
\begin{align*}
	\mathcal{U}_1(2)=&\sum_{1\leq i_1\neq i_2\leq n}\prod_{t=1}^2 x_{i_t,j_1} x_{i_t,j_2}, \notag \\
	\mathcal{U}_2(2)=& \sum_{1\leq i_1\neq i_2 \neq i_3 \leq n}(x_{i_1,j_1} x_{i_1,j_2})(x_{i_2,j_1})(x_{i_3,j_2}),\notag \\
	\mathcal{U}_3(2)=& \sum_{1\leq i_1\neq i_2 \neq i_3 \neq i_4 \leq n} \prod_{t=1}^2 x_{i_t,j_1} \prod_{t=3}^4 x_{i_t,j_2}. 
\end{align*}
We then find that $\mathcal{U}_1(2), \mathcal{U}_2(2)$ and $\mathcal{U}_3(2)$ can be computed with  cost $O(n)$ using the following formulae. 
\begin{align*}
	\mathcal{U}_1(2)=\Big(\sum_{i=1}^n x_{i,j_1}x_{i,j_2}\Big)^2-\sum_{i=1}^n (x_{i,j_1}x_{i,j_2})^2.
\end{align*}
\begin{align*}
\mathcal{U}_2(2)=&	\Big(\sum_{i=1}^n x_{i,j_1}x_{i,j_2}\Big)\Big(\sum_{1\leq i_1\neq i_2\leq n} x_{i,j_1}x_{i,j_2}\Big) \notag \\
&- \sum_{1\leq i_1\neq i_2 \leq n} (x_{i_1,j_1}^2 x_{i_1,j_2}) x_{i_2,j_2}-\sum_{1\leq i_1\neq i_2 \leq n} (x_{i_1,j_1} x_{i_1,j_2}^2) x_{i_2,j_1}, \notag 
\end{align*} where we use $\sum_{1\leq i_1\neq i_2\leq n} x_{i,j_1}x_{i,j_2}=(\sum_{i=1}^n x_{i,j_1} )(\sum_{i=1}^n x_{i,j_2})-\sum_{i=1}^n x_{i,j_1} x_{i,j_2}$, and $\sum_{1\leq i_1\neq i_2 \leq n} (x_{i_1,j_1}^2 x_{i_1,j_2}) x_{i_2,j_2}=(\sum_{i=1}^n x_{i,j_1}^2x_{i,j_2})(\sum_{i=1}^n x_{i,j_2})-\sum_{i=1}^n x_{i,j_1}^2 x_{i,j_2}^2.$
\begin{align*}
	\mathcal{U}_3(2)=& \Big( \sum_{1\leq i_1\neq i_2\leq n} x_{i_1,j_1}x_{i_2,j_1} \Big)\Big( \sum_{1\leq i_3\neq i_4\leq n} x_{i_3,j_2}x_{i_4,j_2} \Big) - 2\mathcal{U}_1(2)-4\mathcal{U}_3(2), \notag
\end{align*}where we use $ \sum_{1\leq i_1\neq i_2\leq n} x_{i_1,k}x_{i_2,k}=(\sum_{i=1}^n x_{i,k})^2- \sum_{i=1}^n x_{i,k}^2$ for $k=j_1,j_2$.

When $a\geq 3$, the similar iterative method can be applied. But the closed form for computation might be hard to derive directly. Alternatively, we 
  introduce a simplified form of U-statistics: 
$
	\mathcal{U}_c(a)=(P^n_a)^{-1}\sum_{1\leq i_1\neq \ldots \neq i_a\leq n}\allowbreak  \sum_{1\leq j_1\neq j_2\leq p} \prod_{t=1}^a(x_{i_t,j_1}-\bar{x}_{j_1})(x_{i_t,j_2}-\bar{x}_{j_2}). 
$ We note that $\mathcal{U}_c(a)$ takes a similar form to  $\tilde{\mathcal{U}}(a)$ in  \eqref{eq:originleadingterm}, but replacing each observation $x_{i,j}$ with the centered correspondence $x_{i,j}-\bar{x}_j$. Therefore, $\mathcal{U}_c(a)$ can be computed with cost $O(n)$ using Algorithm \ref{Algo2}, if we set $s_{i,l}=(x_{i,j_1}-\bar{x}_{j_1})(x_{i,j_2}-\bar{x}_{j_2})$ in Algorithm \ref{Algo2} for $l\in \{(j_1,j_2): 1\leq j_1\neq j_2\leq p\}$.  We then show that we can substitute $\mathcal{U}(a)$  with  $\mathcal{U}_c(a)$ when $a\geq 3$ in computation under certain conditions. 



\begin{proposition}\label{prop:centeredclt}
Under the Conditions of Theorem \ref{thm:computation}, consider $a\geq 3$. If $a$ is odd, $p=o(n^{1+a/2})$; if $a$ is even, $p=o(n^{a/2})$. Then $\{\mathcal{U}(a)-\mathcal{U}_c(a)\}/\sigma(a)\xrightarrow{P} 0$. 
\end{proposition}
Proposition \ref{prop:centeredclt} is proved in the following Section \ref{sec:proofpropsimplecomputation}. It implies that the results in Theorem \ref{thm:computation} sill hold by 
  replacing $\mathcal{U}(a)$ with $\mathcal{U}_c(a)$.  As discussed above, we recommend including U-statistics of orders $\{1,2,3,\ldots,6,\infty\}$   in the adaptive testing procedure. Then Proposition  \ref{prop:centeredclt} requires that $p=o(n^2)$, which suits a  wide range of applications.  Combining Theorem \ref{thm:computation} and   Proposition  \ref{prop:centeredclt}, we can conduct the test with quick computation of cost $O(p^2n)$.

  On the other hand, we  can conduct the test more generally   without  Condition \ref{cond:higherordermomentvarest} and the requirement $p=o(n^2)$. Specifically, we compute $\tilde{\mathcal{U}}(a)$ in \eqref{eq:originleadingterm} with cost $O(p^2n)$. Then $[\tilde{\mathcal{U}}(a)-\mathrm{E}\{ \tilde{\mathcal{U}}(a)\}  ] /\sqrt{\mathrm{var}\{\tilde{\mathcal{U}}(a)\}}\xrightarrow{D}\mathcal{N}(0,1)$ by Lemma \ref{lm:varianceorder} in \ref{suppA} and Theorem \ref{thm:computation}. To test $H_0$ in \eqref{eq:nullhypindepen}, it  suffices to estimate $\mathrm{E}\{\tilde{\mathcal{U}}(a)\}$ and $\mathrm{var}\{\tilde{\mathcal{U}}(a)\}$ with permutation.  This may have higher computational cost than the method above due to permutation, but is computationally more efficient than  estimating $p$-values directly via permutation or bootstrap, especially when evaluating small $p$-values.

\subsubsection{Proof of Proposition \ref{prop:centeredclt} (on Page \pageref{prop:centeredclt})} \label{sec:proofpropsimplecomputation}

In this section, we prove Proposition \ref{prop:centeredclt}. 
As both $\mathcal{U}_c(a)$ and $\mathcal{U}(a)$  are  location invariant in the sense of Proposition \ref{prop:locinvariance},   similarly to the proof of Theorem \ref{thm:computation}, we  assume $\mathrm{E}({\mathbf x})={\mathbf 0}$ in the proofs in this section.

Let $\mathcal{U}_{c,1}=\tilde{\mathcal{U}}(a)$ in \eqref{eq:originleadingterm},  
and $\mathcal{U}_{c,2}(a)=\mathcal{U}_{c}(a)-\mathcal{U}_{c,1}(a)$. By the proof of Theorem \ref{thm:jointnormal}, we know $\{\mathcal{U}(a)-\mathcal{U}_{c,1}(a)\}/\sqrt{\mathrm{var}\{\mathcal{U}(a) \}} \xrightarrow{P} 0$. 
 To finish the proof of Proposition \ref{prop:centeredclt}, it suffices to prove $ {\mathcal{U}_{c,2}(a) }/{   \sqrt{ \mathrm{var} \{\mathcal{U}_{}(a) \}    }  }\xrightarrow{P} 0$. 
By Lemma \ref{lm:varianceorder}, $ \mathrm{var}\{\mathcal{U}_{}(a)\}=\Theta(p^2n^{-a})$.    Then it  suffices to prove $\mathrm{E}\{ \mathcal{U}_{c,2}^2(a) \} =o(p^2 n^{-a} ) $ by the Markov's inequality. 
To derive $\mathcal{U}_{c,2}(a),$
we similarly use the notation in Section \ref{lm:pfthm24lm}. 
Specifically, 
given tuple $\mathbf{i}\in \mathcal{P}(n,a)$,  let $\mathbf{i}_{(s_1+s_2+s_3)}$ represent a sub-tuple of $\mathbf{i}$ with length $s_1+s_2+s_3$, and define  $\mathcal{S}(\mathbf{i},s_1+s_2+s_3)$ to be the collection of sub-tuples of $\mathbf{i}$ with length $s_1+s_2+s_3$.  Then we write
\begin{align*}
	&~\mathcal{U}_{c,2}(a) \notag \\
	 =&~ \sum_{\substack{\mathbf{i}\in \mathcal{P}(n,a);\\ 1\leq j_1\neq j_2\leq p}} \sum_{ \substack{ 0\leq s_1,s_2\leq a;\\  0\leq s_3 < a  } } \ \sum_{ \substack{ \mathbf{i}_{(s_1+s_2+s_3)} \in \mathcal{S}(\mathbf{i},s_1+s_2+s_3)} }  (\bar{x}_{j_1}\bar{x}_{j_2})^{a-s_1-s_2-s_3}\notag \\
	&~\times  \Big\{ (-\bar{x}_{j_2})^{s_1} \prod_{t=1}^{s_1} x_{i_t,j_1} \Big\} \Big\{ (-\bar{x}_{j_1})^{s_2}\prod_{t=s_1+1}^{s_1+s_2} x_{i_t,j_2}\Big\}  \Big\{ \prod_{t=s_1+s_2+1}^{s_1+s_2+s_3} x_{i_{t},j_1}x_{i_{t},j_2} \Big\} \notag \\
   =&  \sum_{ \substack{ 0\leq s_1,s_2\leq a;\,  0\leq s_3 < a  } }  C_{s_1,s_2,s_3}T_{s_1,s_2,s_3},\notag 
\end{align*} where $C_{s_1,s_2,s_3}$ are some constants that only depend on $s_1,s_2,s_3$ and $a$, and
\begin{align*}
	T_{s_1,s_2,s_3}=& \sum_{\substack{1\leq j_1\neq j_2\leq p;\, \mathbf{i}\in \mathcal{P}(n,s_1+s_2+s_3)}} \frac{1}{P^n_{s_1+s_2+s_3}}\times (\bar{x}_{j_1}\bar{x}_{j_2})^{a-s_1-s_2-s_3} \notag \\
   & \times  \Big\{ (-\bar{x}_{j_2})^{s_1} \prod_{t=1}^{s_1} x_{i_t,j_1} \Big\} \Big\{ (-\bar{x}_{j_1})^{s_2}\prod_{t=s_1+1}^{s_1+s_2} x_{i_t,j_2}\Big\}  \Big\{ \prod_{t=s_1+s_2+1}^{s_1+s_2+s_3} x_{i_{t},j_1}x_{i_{t},j_2} \Big\}. \notag 
\end{align*}  When $a$ is finite, it suffices to prove $\mathrm{E}(T_{s_1,s_2,s_3}^2)=o(p^2 n^{-a})$. 

Particularly,
\begin{eqnarray}
\quad &&\mathrm{E}(T_{s_1,s_2,s_3}^2)\label{eq:ets1s2s3sq} \\
	&=& \sum_{ \substack{ 1\leq j_1\neq j_2\leq p \\  1\leq \tilde{j}_1\neq \tilde{j}_2\leq p } } \ \sum_{\mathbf{i},\tilde{\mathbf{i}}\in \mathcal{P}(n,s_1+s_2+s_3)} \Big(\frac{1}{P^n_{s_1+s_2+s_3}}  \Big)^2 \notag \\
	&&\times \mathrm{E}\Biggr[ (\bar{x}_{j_1}\bar{x}_{j_2})^{a-s_1-s_2-s_3}  \Big\{ (-\bar{x}_{j_2})^{s_1} \prod_{t=1}^{s_1} x_{i_t,j_1} \Big\} \Big\{ (-\bar{x}_{j_1})^{s_2}\prod_{t=s_1+1}^{s_1+s_2} x_{i_t,j_2}\Big\} \notag \\
	&& \quad \times \Big\{ \prod_{t=s_1+s_2+1}^{s_1+s_2+s_3} x_{i_{t},j_1}x_{i_{t},j_2} \Big\} (\bar{x}_{\tilde{j}_1}\bar{x}_{\tilde{j}_2})^{a-s_1-s_2-s_3}\Big\{ (-\bar{x}_{\tilde{j}_2})^{s_1} \prod_{t=1}^{s_1} x_{\tilde{i}_t,\tilde{j}_1} \Big\} \notag \\
	&& \quad \times  \Big\{ (-\bar{x}_{\tilde{j}_1})^{s_2}\prod_{t=s_1+1}^{s_1+s_2} x_{\tilde{i}_t,\tilde{j}_2}\Big\} \Big\{ \prod_{t=s_1+s_2+1}^{s_1+s_2+s_3} x_{\tilde{i}_{t},\tilde{j}_1}x_{\tilde{i}_{t},\tilde{j}_2} \Big\}\Biggr] \notag \\
	&=& \sum_{ \substack{ 1\leq j_1\neq j_2\leq p \\  1\leq \tilde{j}_1\neq \tilde{j}_2\leq p } } \ \sum_{\substack{\mathbf{i},\tilde{\mathbf{i}}\in \mathcal{P}(n,s_1+s_2+s_3);\\ \mathbf{w},\tilde{\mathbf{w}}\in \mathcal{C}(n,2a-s_1-s_2-2s_3) } }C_{n,s_1,s_2,s_3}M(\mathbf{i},\tilde{\mathbf{i}}, \mathbf{w},\tilde{\mathbf{w}},\mathbf{j}),\notag
\end{eqnarray}
where we define on Page \pageref{par:notationindpcond} that $\mathbf{w}\in \mathcal{C}(n,s)$ represents tuples $i_1,\ldots,i_s$ satisfying $1\leq i_1,\ldots, i_s\leq n$, and $C_{n,s_1,s_2,s_3}=({P^n_{s_1+s_2+s_3} n^{2a-s_1-s_2-s_3} }  )^{-2} $ and
\begin{align}
&~M(\mathbf{i},\tilde{\mathbf{i}}, \mathbf{w},\tilde{\mathbf{w}},\mathbf{j})\label{eq:defmiiww} \\
    = & ~ \prod_{t=1}^{s_1}x_{i_t,j_1}x_{\tilde{i}_t,\tilde{j}_1}\prod_{t=s_1+1}^{s_1+s_2} x_{i_t,j_2}  x_{\tilde{i}_t,\tilde{j}_2} \prod_{t=s_1+s_2+1}^{s_1+s_2+s_3} (x_{i_t,j_1}x_{i_t,j_2})(x_{\tilde{i}_t,\tilde{j}_1}x_{\tilde{i}_t,\tilde{j}_2}) \notag \\
	& ~\quad \times \prod_{k=1}^{a-s_1-s_3} x_{w_k,j_1}x_{\tilde{w}_k,\tilde{j}_1} \prod_{k=a-s_1-s_3+1}^{2a-s_1-s_2-2s_3} x_{w_k,j_2}x_{\tilde{w}_k,\tilde{j}_2} . \notag	
\end{align}
We write  $M(\mathbf{i},\tilde{\mathbf{i}}, \mathbf{w},\tilde{\mathbf{w}},\mathbf{j}) =M_{j_1} M_{j_2}M_{\tilde{j}_1}M_{\tilde{j}_2}$, where
\begin{align*}
	M_{j_1}=\prod_{t=1}^{s_1} x_{i_t,j_1} \prod_{t=s_1+s_2+1}^{s_1+s_2+s_3}  x_{i_t,j_1} \prod_{k=1}^{a-s_1-s_3} {x_{w_k, j_1}}, \ \ & M_{j_2}=  \prod_{t=s_1+1}^{s_1+s_2+s_3} x_{i_t,j_2}\prod_{k=a-s_1-s_3+1}^{2a-s_1-s_2-2s_3} {x_{w_k,j_2}}, \notag \\
M_{\tilde{j}_1}=\prod_{t=1}^{s_1}  x_{\tilde{i}_t,\tilde{j}_1} \prod_{t=s_1+s_2+1}^{s_1+s_2+s_3} x_{\tilde{i}_t,\tilde{j}_1} \prod_{k=1}^{a-s_1-s_3} {x_{\tilde{w}_k, \tilde{j}_1}}, \ \ & M_{\tilde{j}_2}=  \prod_{t=s_1+1}^{s_1+s_2+s_3} x_{\tilde{i}_t,\tilde{j}_2} \prod_{k=a-s_1-s_3+1}^{2a-s_1-s_2-2s_3} {x_{\tilde{w}_k,\tilde{j}_2}}.
\end{align*}

As $\mathrm{E}(\mathbf{x})=\mathbf{0}$, when $a=1$, $\mathrm{E}(M_{j_1})=\mathrm{E}(M_{j_2})=\mathrm{E}(M_{\tilde{j}_1})=\mathrm{E}(M_{\tilde{j}_2})=0$.  We then consider $a\geq 2$. As $\mathrm{E}(\mathbf{x})=\mathbf{0}$, $i_1\neq \ldots \neq i_{s_1+s_2+s_3}$ and $\tilde{i}_1\neq \ldots \neq \tilde{i}_{s_1+s_2+s_3}$, we know that $\mathrm{E}(M_{j_1})\neq 0$ only when 
$
	\{i_1,\ldots, i_{s_1}, i_{s_1+s_2+1},\ldots,\break i_{s_1+s_2+s_3}\}\subseteq \{w_1,\ldots, w_{a-s_1-s_3} \}
$ and
\begin{align}
	&|S_{j_1}| \leq  s_1+s_3 + \lfloor (a-2s_1-2s_3)/2 \rfloor =  \lfloor a/2\rfloor,\label{eq:centersetcont1} 
\end{align} where
$
	S_{j_1}=\{i_1,\ldots, i_{s_1}, i_{s_1+s_2+1},\ldots, i_{s_1+s_2+s_3}, w_1,\ldots, w_{a-s_1-s_3} \}.
$
Similarly, when  $\mathrm{E}(M_{j_2})\neq 0$, we know $
	\{i_{s_1+1},\ldots, i_{s_1+s_2+s_3}\}\subseteq \{ w_{a-s_1-s_3+1},\break \ldots, w_{2a-s_1-s_2-2s_3} \},
$ and 
\begin{align}
	&|S_{j_2}|\leq  s_2+s_3 + \lfloor (a-2s_2-2s_3)/2 \rfloor=  \lfloor a/2\rfloor, \label{eq:centersetcont2} 
\end{align} where
$S_{j_2}=\{i_{s_1+1},\ldots, i_{s_1+s_2+s_3},  w_{a-s_1-s_3+1},\ldots, w_{2a-s_1-s_2-2s_3} \}$.
As $|S_{j_1}\cap S_{j_2}|=s_3$, 
 combining \eqref{eq:centersetcont1} and \eqref{eq:centersetcont2}, we know that if $\mathrm{E}(M_{j_1})\neq 0$ and $\mathrm{E}(M_{j_2})\neq 0$,
\begin{align}
	& |S_{j_1}\cup S_{j_2}|\leq 2 \lfloor a/2\rfloor-s_3\label{eq:centersetcont3} 
\end{align}
Similarly,  if $\mathrm{E}(M_{\tilde{j}_1})\neq 0$, we know 
\begin{align}
	&|S_{\tilde{j}_1}|\leq  \lfloor a/2\rfloor, \label{eq:centersetconttilde1} 
\end{align} where $S_{\tilde{j}_1}=\{\tilde{i}_1,\ldots, \tilde{i}_{s_1}, \tilde{i}_{s_1+s_2+1},\ldots, \tilde{i}_{s_1+s_2+s_3}, \tilde{w}_1,\ldots, \tilde{w}_{a-s_1-s_3} \}. $
If $\mathrm{E}(M_{\tilde{j}_2})\neq 0$, we know
\begin{align}
	&|S_{\tilde{j}_2}|\leq \lfloor a/2\rfloor \label{eq:centersetconttilde2} ,
\end{align} where
$S_{\tilde{j}_2}=\{\tilde{i}_{s_1+1},\ldots, \tilde{i}_{s_1+s_2+s_3},  \tilde{w}_{a-s_1-s_3+1},\ldots, \tilde{w}_{2a-s_1-s_2-2s_3} \}.$
If  $\mathrm{E}(M_{\tilde{j}_1})\neq 0$ and $\mathrm{E}(M_{\tilde{j}_2})\neq 0$, we know
 \begin{align}
 |S_{\tilde{j}_1}\cup S_{\tilde{j}_2}|\leq 2 \lfloor a/2\rfloor -s_3. \label{eq:centersetcont4}
\end{align} 

%
%

To evaluate $\mathrm{E}(T_{s_1,s_2,s_3}^2)$ in \eqref{eq:ets1s2s3sq}, 
for the simplicity of representation, in the following we write
\begin{align*}
	\sum_{\mathrm{ALL\ SUM}} =\sum_{ \substack{ 1\leq j_1\neq j_2\leq p; \,  1\leq \tilde{j}_1\neq \tilde{j}_2\leq p } } \ \sum_{\substack{\mathbf{i},\tilde{\mathbf{i}}\in \mathcal{P}(n,s_1+s_2+s_3);\, \mathbf{w},\tilde{\mathbf{w}}\in \mathcal{C}(n,2a-s_1-s_2-2s_3) } }.
\end{align*}  We next evaluate $\mathrm{E}(T_{s_1,s_2,s_3}^2)$ by discussing the indexes $\{j_1, j_2, \tilde{j}_1, \tilde{j}_2\}$. 
We first consider $|\{j_1, j_2, \tilde{j}_1, \tilde{j}_2\}|=4$, and the summation  
$$
\sum_{\mathrm{ALL\ SUM}} \mathbf{1}_{ \{|\{j_1, j_2, \tilde{j}_1, \tilde{j}_2\}|=4\} }\times  C_{n,s_1,s_2,s_3}\times \mathrm{E}\{M(\mathbf{i},\tilde{\mathbf{i}}, \mathbf{w},\tilde{\mathbf{w}},\mathbf{j})\}.
$$ Note that $|\{j_1, j_2, \tilde{j}_1, \tilde{j}_2\}|=4$ implies that $j_1\neq j_2 \neq \tilde{j}_1 \neq \tilde{j}_2$.  Without loss of generality, we assume $j_1<j_2<\tilde{j}_1<\tilde{j}_2$ , while the other cases can follow similar analysis. Define $\kappa_1=j_2-j_1$, $\kappa_2=\tilde{j}_1-j_2$ and $\kappa_3=\tilde{j}_2-\tilde{j}_1$. In addition, for some small positive constants $\mu$ and $\epsilon$ and $\delta$ in Condition \ref{cond:alphamixing}, define $K_0=-(2+\epsilon)(4+\mu)(\log p)/(\epsilon \log \delta)$. 
If $\kappa_m=\max\{\kappa_1, \kappa_2, \kappa_3 \}\geq K_0$, we can write
\begin{align*}
	|\mathrm{E}\{M(\mathbf{i},\tilde{\mathbf{i}}, \mathbf{w},\tilde{\mathbf{w}},\mathbf{j})\}|\leq C\delta^{K_0\epsilon/(2+\epsilon)} + \Delta_{j,\tilde{j}}.
\end{align*} We next evaluate  $\Delta_{j,\tilde{j}}$ by discussing the following cases (a)--(c).
 

\medskip

\noindent \textbf{Case (a)}  If all three $\kappa_1,\kappa_2,\kappa_3>K_0$, we have
\begin{align*}
	\Delta_{j,\tilde{j}}=|\mathrm{E}(M_{j_1})\mathrm{E}(M_{j_2})\mathrm{E}(M_{\tilde{j}_1})\mathrm{E}(M_{\tilde{j}_2} )|.
\end{align*}
Then if $\Delta_{j,\tilde{j}}\neq 0$, we know $\mathrm{E}(M_{j_1}), \mathrm{E}(M_{j_2}), \mathrm{E}(M_{\tilde{j}_1})$ and $ \mathrm{E}(M_{\tilde{j}_2} )\neq 0$, which implies that \eqref{eq:centersetcont3} and \eqref{eq:centersetcont4} hold. By Condition \ref{cond:higherordermomentvarest}, we  know that
\begin{align*}
\sum_{\mathrm{ALL\ SUM}}   \Delta_{j,\tilde{j}}   \mathbf{1}_{\{|  \{j_1, j_2, \tilde{j}_1, \tilde{j}_2\}|=4 ,\kappa_1,\kappa_2,\kappa_3>K_0\}  }=O(1)p^4 n^{4 \lfloor a/2 \rfloor -2s_3 }.
\end{align*}
In addition, $\mathrm{E}\{M(\mathbf{i},\tilde{\mathbf{i}}, \mathbf{w},\tilde{\mathbf{w}},\mathbf{j})\}\neq 0$ only if $|\{\mathbf{i}\}\cup \{\tilde{\mathbf{i}\}}\cup \{\mathbf{w}\}\cup \{\tilde{\mathbf{w}\}}|\leq 2a-s_3$. 
It follows that
\begin{align}
	&~ \Biggr|\sum_{\mathrm{ALL\ SUM}} C_{n,s_1,s_2,s_3}  \mathrm{E}\{M(\mathbf{i},\tilde{\mathbf{i}}, \mathbf{w},\tilde{\mathbf{w}},\mathbf{j})\}  \mathbf{1}_{\Big\{\substack{|\{j_1, j_2, \tilde{j}_1, \tilde{j}_2\}|=4;\\\kappa_1,\kappa_2,\kappa_3>K_0}\Big\}}\Biggr| \label{eq:varests1partallK0} \\
\leq &~C \sum_{s_3=0}^{a-1}n^{-2(2a-s_3)}n^{2a-s_3}p^4  C\delta^{K_0\epsilon/(2+\epsilon)} \notag \\
&~+   \sum_{\mathrm{ALL\ SUM}} Cn^{-2(2a-s_3)}\Delta_{j,\tilde{j}}  \mathbf{1}_{ \{|\{j_1, j_2, \tilde{j}_1, \tilde{j}_2\}|=4, \kappa_1, \kappa_2,   \kappa_3>K_0\} } ,  \notag  \\
=&~ o(n^{-(a+1)})+ O(1)p^4 n^{4\lfloor a/2 \rfloor -4a }, \notag
\end{align} where we use $\sum_{\text{ALL SUM}}\mathbf{1}_{\{\mathrm{E}\{M(\mathbf{i},\tilde{\mathbf{i}}, \mathbf{w},\tilde{\mathbf{w}},\mathbf{j})\}\neq 0\} }=\sum_{s_3=0}^{a-1} n^{2a-s_3}p^4$,  $\delta^{K_0\epsilon/(2+\epsilon)} =O(1)p^{-(4+\mu)}$, and $C_{n,a,s_1,s_2,s_3}=\Theta(1)n^{-2(2a-s_3)}$.
If $a$ is even, $\eqref{eq:varests1partallK0}=O(1)p^4 n^{-2a} =o(1)p^2n^{-a}$. If $a$ is odd, $\eqref{eq:varests1partallK0}=O(1)p^4 n^{-2a-2} =o(1)p^2n^{-a}$.

\medskip

\noindent \textbf{Case (b.1)} If $\kappa_1\leq K_0$, $\kappa_2> K_0$ and $\kappa_3>K_0$,
\begin{align*}
	\Delta_{j,\tilde{j}}=|\mathrm{E}(M_{j_1}M_{j_2})\mathrm{E}(M_{\tilde{j}_1})\mathrm{E}(M_{\tilde{j}_2} )|
\end{align*}
If $\mathrm{E}(M_{\tilde{j}_1})$ and $\mathrm{E}(M_{\tilde{j}_2} )\neq 0$, we know \eqref{eq:centersetcont4} holds. We then consider $\mathrm{E}(M_{j_1}M_{j_2})$ with $j_1\neq j_2$.  Note that
\begin{align*}
	&M_{j_1}M_{j_2}\notag \\
	=&\prod_{t=1}^{s_1} x_{i_t,j_1}\prod_{t=s_1+1}^{s_1+s_2} x_{i_t,j_2} \prod_{t=s_1+s_2+1}^{s_1+s_2+s_3}  (x_{i_t,j_1} x_{i_t,j_2} ) \prod_{k=1}^{a-s_1-s_3} {x_{w_k, j_1}} \prod_{k=a-s_1-s_3+1}^{2a-s_1-s_2-2s_3} {x_{w_k,j_2}}.
\end{align*}As $\mathrm{E}(\mathbf{x})=\mathbf{0}$ and  $\mathrm{E}(x_{1,j_1}x_{1,j_2})=0$ under $H_0$ when $j_1\neq j_2$, we know $\mathrm{E}(M_{j_1}M_{j_2})\neq 0$ only when $\{i_1,\ldots, i_{s_1+s_2+s_3}\}\subseteq \{w_1,\ldots, w_{2a-s_1-s_2-2s_3}\}$ and
\begin{align}
	|S_{j_1}\cup S_{j_2}|\leq \lfloor (2a-s_3)/2 \rfloor \label{eq:centersetcont5}
\end{align}
 We then know $\Delta_{j,\tilde{j}}\neq 0$ only when \eqref{eq:centersetcont4} and \eqref{eq:centersetcont5} hold, and thus
\begin{align*}
&~\sum_{\mathrm{ALL\ SUM}}   \Delta_{j,\tilde{j}}\times    \mathbf{1}_{\{|  \{j_1, j_2, \tilde{j}_1, \tilde{j}_2\}|=4 ,\kappa_1 \leq K_0,\kappa_2,\kappa_3>K_0\}  }\notag \\
=&~ \sum_{s_3=0}^{a-1} O(1)p^3K_0 n^{2 \lfloor a/2 \rfloor-s_3 + \lfloor (2a-s_3)/2 \rfloor }.
\end{align*}Then similarly to \eqref{eq:varests1partallK0}, we have 
\begin{align}
	&~ \Big|\sum_{\mathrm{ALL\ SUM}} C_{n,a,s_1,s_2,s_3}\mathrm{E}\{M(\mathbf{i},\tilde{\mathbf{i}}, \mathbf{w},\tilde{\mathbf{w}},\mathbf{j})\}  \mathbf{1}_{\Big\{\substack{|\{j_1, j_2, \tilde{j}_1, \tilde{j}_2\}|=4 ;\\\kappa_1\leq K_0; \kappa_2,\kappa_3>K_0}\Big\}}\Big| \label{eq:varests1part3K0} \\
\leq &~ o(n^{-(a+1)})  + \sum_{\mathrm{ALL\ SUM}} C_{n,a,s_1,s_2,s_3}\Delta_{j,\tilde{j}}  \mathbf{1}_{ \left\{ \substack{ |\{j_1, j_2, \tilde{j}_1, \tilde{j}_2\}|=4; \\  \,  \kappa_1\leq K_0;\, \kappa_2,   \kappa_3>K_0}  \right\}  }   \notag  \\
=&~ o(n^{-(a+1)})+ \sum_{s_3=0}^{a-1}  O(1) p^3K_0 n^{2 \lfloor a/2 \rfloor-s_3 + \lfloor (2a-s_3)/2 \rfloor -4a+2s_3 }. \notag  
\end{align}
If $a$ is even, we use  $2 \lfloor a/2 \rfloor-s_3 + \lfloor (2a-s_3)/2 \rfloor -4a+2s_3 \leq -2a + s_3/2 \leq -a -(a+1)/2 $ as $s_3\leq a-1$. Then $\eqref{eq:varests1part3K0}=O(1)p^3K_0 n^{-a-(a+1)/2}=o(1) p^2 n^{-a}$.
If $a$ is odd, we use  $2 \lfloor a/2 \rfloor-s_3 + \lfloor (2a-s_3)/2 \rfloor -4a+2s_3 \leq -2a + s_3/2 \leq -a -(a+3)/2 $ as $2 \lfloor a/2 \rfloor=a-1$ and $s_3\leq a-1$. Then $\eqref{eq:varests1part3K0}=O(1)p^3K_0 n^{-a-(a+3)/2}=o(1) p^2 n^{-a}$.

\medskip
\noindent \textbf{Case (b.2)} If $\kappa_1>K_0$, $\kappa_2> K_0$ and $\kappa_3\leq K_0$, similarly to Case (b.1), by symmetricity, we know  
\begin{align}
	&~ \Big|\sum_{\mathrm{ALL\ SUM}} C_{n,a,s_1,s_2,s_3}\mathrm{E}\{M(\mathbf{i},\tilde{\mathbf{i}}, \mathbf{w},\tilde{\mathbf{w}},\mathbf{j})\}  \mathbf{1}_{\Big\{\substack{|\{j_1, j_2, \tilde{j}_1, \tilde{j}_2\}|=4 ;\\\kappa_1,\kappa_2>K_0;\, \kappa_3\leq K_0}\Big\}}\Big| \label{eq:varests1part3K03} \\
=&~ o(n^{-(a+1)})+ \sum_{s_3=0}^{a-1} O(1) p^3K_0 n^{2 \lfloor a/2 \rfloor-s_3 + \lfloor (2a-s_3)/2 \rfloor -4a+2s_3 }. \notag  
\end{align} Then $\eqref{eq:varests1part3K03} =o(1)p^2n^{-a}$.

\medskip

\noindent \textbf{Case (b.3)} If $\kappa_1>K_0$, $\kappa_2\leq K_0$ and $\kappa_3>K_0$, 
\begin{align*}
	\Delta_{j,\tilde{j}}=|\mathrm{E}(M_{j_1})\mathrm{E}(M_{j_2}M_{\tilde{j}_1})\mathrm{E}(M_{\tilde{j}_2} )|.
\end{align*}
If $\mathrm{E}(M_{{j}_1}), \mathrm{E}(M_{\tilde{j}_2} )\neq 0$, we know \eqref{eq:centersetcont1} and \eqref{eq:centersetconttilde1} hold. We then consider $\mathrm{E}(M_{j_2}M_{\tilde{j}_1})$. Note that
\begin{align*}
	& M_{j_2}M_{\tilde{j}_1} \notag \\
	=&\prod_{t=s_1+1}^{s_1+s_2+s_3} x_{i_t,j_2}  \prod_{t=a-s_1-s_3+1}^{2a-s_1-s_2-2s_3} {x_{w_t,j_2}} \prod_{t=1}^{s_1} x_{\tilde{i}_t,\tilde{j}_1} \prod_{t=s_1+s_2+1}^{s_1+s_2+s_3}  x_{\tilde{i}_t,\tilde{j}_1}\prod_{t=1}^{a-s_1-s_3} {x_{\tilde{w}_t, \tilde{j}_1}} .
\end{align*} If $\mathrm{E}(M_{j_2}M_{\tilde{j}_1})\neq 0$, we know that  $|S_{j_2}\cup S_{\tilde{j}_1}|\leq a$. As  
$
	|(S_{j_2}\cup S_{\tilde{j}_1})\cap (S_{j_1}\cup S_{\tilde{j}_2} )|=2s_3,
$ we have $|S_{j_1}\cup S_{j_2}\cup S_{\tilde{j}_1}\cup S_{\tilde{j}_2}|\leq a+2\lfloor a/2 \rfloor -2s_3. $ We then know
\begin{align*}
\sum_{\mathrm{ALL\ SUM}}   \Delta_{j,\tilde{j}}   \mathbf{1}_{\{|  \{j_1, j_2, \tilde{j}_1, \tilde{j}_2\}|=4 ,\kappa_1,\kappa_3 > K_0,\kappa_2\leq K_0\}  }= \sum_{s_3=0}^{a-1}  O(1)p^3K_0 n^{a+2 \lfloor a/2 \rfloor-2s_3  }.
\end{align*}
Then similarly to \eqref{eq:varests1part3K0}, we have
\begin{align}
	& ~\Big|\sum_{\mathrm{ALL\ SUM}} C_{n,a,s_1,s_2,s_3}\mathrm{E}\{M(\mathbf{i},\tilde{\mathbf{i}}, \mathbf{w},\tilde{\mathbf{w}},\mathbf{j})\}  \mathbf{1}_{\Big\{\substack{|\{j_1, j_2, \tilde{j}_1, \tilde{j}_2\}|=4 ;\\\kappa_1,\kappa_3 > K_0;\kappa_2\leq K_0}\Big\}}\Big| \label{eq:varests1part3K02} \\
	=&~o(n^{-(a+1)})+O(1) p^3K_0 n^{2 \lfloor a/2 \rfloor -3a }. \notag
\end{align}
If $a$ is even, we know $\eqref{eq:varests1part3K02} =p^3K_0n^{-2a} =o(1)p^2n^{-a} $. If $a$ is odd, we know $\eqref{eq:varests1part3K02} =p^3K_0n^{-2a-1} =o(1)p^2n^{-a} $.

\medskip

%
%
%

\noindent \textbf{Case (c)} If two of $\kappa_1,\kappa_2,\kappa_3\leq K_0$, we know $$\sum_{j_1,j_2,\tilde{j}_1,\tilde{j}_2} \mathbf{1}_{\{\text{two of } \kappa_1,\kappa_2,\kappa_3\leq K_0\} }=O(p^2K_0^2).$$ Following definition in \eqref{eq:defmiiww}, we know $\mathrm{E}\{M(\mathbf{i},\tilde{\mathbf{i}}, \mathbf{w},\tilde{\mathbf{w}},\mathbf{j})\}\neq 0$ only when $|S_{j_1}\cup S_{j_2}\cup S_{\tilde{j}_1}\cup S_{\tilde{j}_2}|\leq 2a-s_3$. It implies that 
\begin{align*}
\sum_{\mathrm{ALL\ SUM}}   \Delta_{j,\tilde{j}}   \mathbf{1}_{\{|  \{j_1, j_2, \tilde{j}_1, \tilde{j}_2\}|=4 ,\text{ two of }\kappa_1,\kappa_2,\kappa_3\leq K_0\}  }=O(1)p^2K_0^2 n^{2a-s_3 }.
\end{align*} 
Similarly to \eqref{eq:varests1part3K02},  we have
\begin{eqnarray}
	\quad && \Big|\sum_{\mathrm{ALL\ SUM}} C_{n,a,s_1,s_2,s_3}\mathrm{E}\{M(\mathbf{i},\tilde{\mathbf{i}}, \mathbf{w},\tilde{\mathbf{w}},\mathbf{j})\}  \mathbf{1}_{\Big\{\substack{|\{j_1, j_2, \tilde{j}_1, \tilde{j}_2\}|=4 ;\\\text{two of }\kappa_1,\kappa_2 ,\kappa_3\leq K_0} \Big\}}\Big| \label{eq:varests1part3K03} \\
	&=&o(n^{-(a+1)})+ \sum_{s_3=0}^{a-1}  O(1) p^2K_0^2 n^{-2a +s_3}. \notag
\end{eqnarray}
As $s_3\leq a-1$ and $K_0=O(\log p)$, we know $\eqref{eq:varests1part3K03}=O(1)p^2 K_0^2 n^{-a-1}=o(1)p^2n^{-a}$.

\medskip

\noindent \textbf{Case (d)} If $|\{j_1,j_2,j_3,j_4\}|=3$ or $2$, similar analysis can be applied, and we know that
\begin{align}
	&~ \Big|\sum_{\mathrm{ALL\ SUM}} C_{n,a,s_1,s_2,s_3}\mathrm{E}\{M(\mathbf{i},\tilde{\mathbf{i}}, \mathbf{w},\tilde{\mathbf{w}},\mathbf{j})\}  \mathbf{1}_{\{|\{j_1, j_2, \tilde{j}_1, \tilde{j}_2\}|=2 \mbox{ or }3\}}\Big| \label{eq:varests1part3K031} \\
	=&~o(n^{-(a+1)})+o(1)p^2n^{-2a}. \notag
\end{align}

Summarizing Cases (a)--(d) above, we obtain $\mathrm{E}(T_{s_1,s_2,s_3}^2)=o(p^2n^{-a})$.
 
\subsection{Simulations on One-Sample Covariance Testing} \label{sec:suppsimu}
In this section, we provide extensive simulation studies for the one-sample covariance testing discussed in Section \ref{sec:mainexamsec}. We present the results of the five simulation settings introduced in Section  \ref{sec:simulation} in the following Sections \ref{sec:npcombsimulsize}--\ref{sec:simulationstudyii}. 

\subsubsection{Study 1: Empirical Size}\label{sec:npcombsimulsize}
In this study, we verify the theoretical results under $H_0$ in Section \ref{sec:mainexamsec} and the show validity of the adaptive testing procedure across different $n$ and $p$ values.   In particular, we fix $n=100$ and take $p\in \{50,100,200,400,600,800,1000\}$. Then we
 generate $n$ i.i.d. $p$-dimensional $\mathbf{x}_i$ for $i=1,\ldots, n$, and each $\mathbf{x}_i$ has i.i.d. entries of $\mathcal{N}(0,1)$ and $\mathrm{Gamma}(2,0.5)$  respectively. 
   The results are summarized in the following Tables \ref{table:nullgaussian} and     \ref{table:nullgamma} respectively.

\begin{table}[!ht]
\centering
\caption{Empirical Type \RNum{1} errors  under Guassian distribution; $n=100$.}
\label{table:nullgaussian}
\begin{tabular}{rrrrrrrr}
  \hline
$p$ & 50 & 100 & 200 & 400 & 600 & 800 & 1000 \\ 
  \hline  
$\mathcal{U}(1)$ & 0.054 & 0.055 & 0.045 & 0.053 & 0.048 & 0.052 & 0.036 \\ 
$\mathcal{U}(2)$ & 0.058 & 0.058 & 0.066 & 0.050 & 0.071 & 0.048 & 0.063 \\ 
$\mathcal{U}(3)$ & 0.057 & 0.066 & 0.061 & 0.055 & 0.051 & 0.063 & 0.052 \\ 
$\mathcal{U}(4)$ & 0.054 & 0.067 & 0.052 & 0.080 & 0.053 & 0.041 & 0.056 \\ 
$\mathcal{U}(5)$ & 0.049 & 0.054 & 0.059 & 0.070 & 0.045 & 0.049 & 0.053 \\ 
$\mathcal{U}(6)$ & 0.039 & 0.057 & 0.063 & 0.061 & 0.056 & 0.057 & 0.074 \\ 
$\mathcal{U}(\infty)$ 1 & 0.046 & 0.055 & 0.049 & 0.067 & 0.064 & 0.042 & 0.044 \\ 
$\mathcal{U}(\infty)$ 2 & 0.040 & 0.047 & 0.045 & 0.056 & 0.048 & 0.050 & 0.048 \\ 
  adpUmin 1 & 0.056 & 0.066 & 0.067 & 0.064 & 0.067 & 0.056 & 0.051 \\ 
  adpUf 1 & 0.065 & 0.083 & 0.069 & 0.079 & 0.063 & 0.058 & 0.060 \\ 
  adpUmin 2 & 0.054 & 0.069 & 0.065 & 0.060 & 0.062 & 0.055 & 0.057 \\ 
  adpUf 2 & 0.069 & 0.082 & 0.065 & 0.065 & 0.058 & 0.057 & 0.062 \\ 
  Identity & 0.055 & 0.053 & 0.058 & 0.053 & 0.061 & 0.049 & 0.053 \\ 
  Sphericity & 0.053 & 0.050 & 0.058 & 0.053 & 0.062 & 0.049 & 0.054 \\ 
  LW & 0.058 & 0.051 & 0.053 & 0.045 & 0.067 & 0.048 & 0.058 \\ 
  Schott & 0.052 & 0.055 & 0.050 & 0.052 & 0.050 & 0.044 & 0.051 \\ 
   \hline
\end{tabular}
\end{table}

\begin{table}[!ht]
\centering
\caption{Empirical Type \RNum{1} errors  under Gamma distribution; $n=100$. }
\label{table:nullgamma}
\begin{tabular}{rrrrrrrr}
  \hline
$p$ & 50 & 100 & 200 & 400 & 600 & 800 & 1000 \\ 
  \hline
$\mathcal{U}(1)$ & 0.043 & 0.049 & 0.054 & 0.048 & 0.050 & 0.049 & 0.043 \\ 
$\mathcal{U}(2)$& 0.057 & 0.075 & 0.062 & 0.054 & 0.057 & 0.055 & 0.061 \\ 
$\mathcal{U}(3)$ & 0.054 & 0.064 & 0.050 & 0.041 & 0.057 & 0.051 & 0.056 \\ 
$\mathcal{U}(4)$ & 0.047 & 0.056 & 0.061 & 0.056 & 0.052 & 0.053 & 0.045 \\ 
$\mathcal{U}(5)$ & 0.043 & 0.043 & 0.054 & 0.052 & 0.050 & 0.053 & 0.049 \\ 
$\mathcal{U}(6)$ & 0.032 & 0.035 & 0.059 & 0.045 & 0.046 & 0.053 & 0.044 \\ 
  $\mathcal{U}(\infty)$ 1 & 0.052 & 0.045 & 0.048 & 0.053 & 0.045 & 0.049 & 0.055 \\ 
  $\mathcal{U}(\infty)$ 2 & 0.044 & 0.052 & 0.052 & 0.053 & 0.044 & 0.051 & 0.045 \\ 
  adpUmin 1 & 0.051 & 0.054 & 0.069 & 0.062 & 0.049 & 0.058 & 0.065 \\ 
  adpUf 1 & 0.055 & 0.060 & 0.075 & 0.067 & 0.054 & 0.058 & 0.067 \\ 
  adpUmin 2 & 0.049 & 0.055 & 0.068 & 0.063 & 0.049 & 0.059 & 0.066 \\ 
  adpUf 2 & 0.063 & 0.067 & 0.070 & 0.058 & 0.047 & 0.057 & 0.061 \\ 
  Identity & 1.000 & 1.000 & 1.000 & 1.000 & 1.000 & 1.000 & 1.000 \\ 
  Sphericity & 0.088 & 0.065 & 0.071 & 0.056 & 0.060 & 0.059 & 0.050 \\ 
  LW & 1.000 & 1.000 & 1.000 & 1.000 & 1.000 & 1.000 & 1.000 \\ 
  Schott & 0.051 & 0.063 & 0.053 & 0.053 & 0.055 & 0.046 & 0.060 \\ 
   \hline
\end{tabular}
\end{table}

In Tables \ref{table:nullgaussian} and     \ref{table:nullgamma}, we provide the simulation results of all the single U-statistics with orders in $\{1,\ldots,6\}$. For $\mathcal{U}(\infty)$, we first use the test statistic \eqref{eq:inftyteststat} same as in \citet{jiang2004}, which is denoted as ``$\mathcal{U}(\infty) $ 1" below. Since the convergence in  \cite{jiang2004} is slow, we use permutation to approximate the distribution in the simulations.  We also use the standardized version $M_n^{\dag}$ given in  Remark \ref{rm:standizedmax}, which is denoted as ``$\mathcal{U}(\infty) $ 2" below. Given  ``$\mathcal{U}(\infty) $ 1" and ``$\mathcal{U}(\infty) $ 2", we apply the adaptive testing with minimum combination and Fisher's method respectively. The results are denoted as ``adpUmin1", ``adpUf1", ``adpUmin2" and  ``adpUf2" respectively below. In addition, we also compare several methods in the literature. The identity and sphericity tests in \citet{chen2011} are denoted as ``Equal" and ``Spher" below; the methods in  \citet{ledoit2002} and \citet{schott2007test}, which are referred to as ``LW" and ``Schott" respectively.  

\newpage

\subsubsection{Study 2} \label{sec:study2onecovsim}
In this section, we provide the simulation results for the second setting in Section \ref{sec:twosimdata}. In particular, we  generate $n$ i.i.d. $p$-dimensional $\mathbf{x}_i$ for $i=1,\ldots, n$, and $\mathbf{x}_i$ follows multivariate Gaussian distribution with mean zero and covariance $\boldsymbol{\Sigma}_A=(1-\rho)I_p+\rho \mathbf{1}_{p,k_0}\mathbf{1}_{p,k_0}^{\intercal}$.  

Similarly to Figure \ref{fig:alternsparsfigure}, we conduct simulations on the adaptive procedure with U-statistics of orders in $\{1,\ldots, 6, \infty\}$. We provide the simulation results of all the single U-statistics and the adaptive procedure, and also compare with some other methods in the literature. We take  $(n,p)\in \{(100,300), (100,600), (100, 1000)\}$, and provide the results in the following Figures  \ref{fig:addtionalfigureplot}--\ref{fig:study2p1000} respectively. 
 
In Figure \ref{fig:addtionalfigureplot}, the first 7 plots  are simulated with $k_0\in \{2, 5, 7, 10, 13, 20, 50\}$. Particularly, we include results of $\mathcal{U}(a)$ for $a \in \{1,\ldots, 6, \infty\}$;  the adaptive procedure ``adpU" by minimum combination of these single U-statistics;  identity and sphericity tests in \cite{chen2011}, which are denoted as `Equal" and ``Shper",  respectively.  We can see that when $k_0\in\{7, 10, 13\}$, the results of ``adpU" are better than all the  other test statistics. For other cases, the results of ``adpU" are close to  the best results of single U-statistics.  In addition, we also examine the case when the nonzero off-diagonal elements of $\boldsymbol{\Sigma}_A$,  i.e., $\sigma_{j_1,j_2}$  with $1\leq j_1\neq j_2\leq k_0$, have same absolute value $|\rho|$, but can be  positive or negative with equal probability. The results of powers versus different $|\rho|$ values  are given by 8th plot in Figure \ref{fig:addtionalfigureplot}, which is consistent with  Remark \ref{rm:notallpositive} in Section \ref{sec:powerana}.


\begin{figure}[!htbp]
    \centering
    \includegraphics[width=0.4\textwidth,height=0.23\textheight]{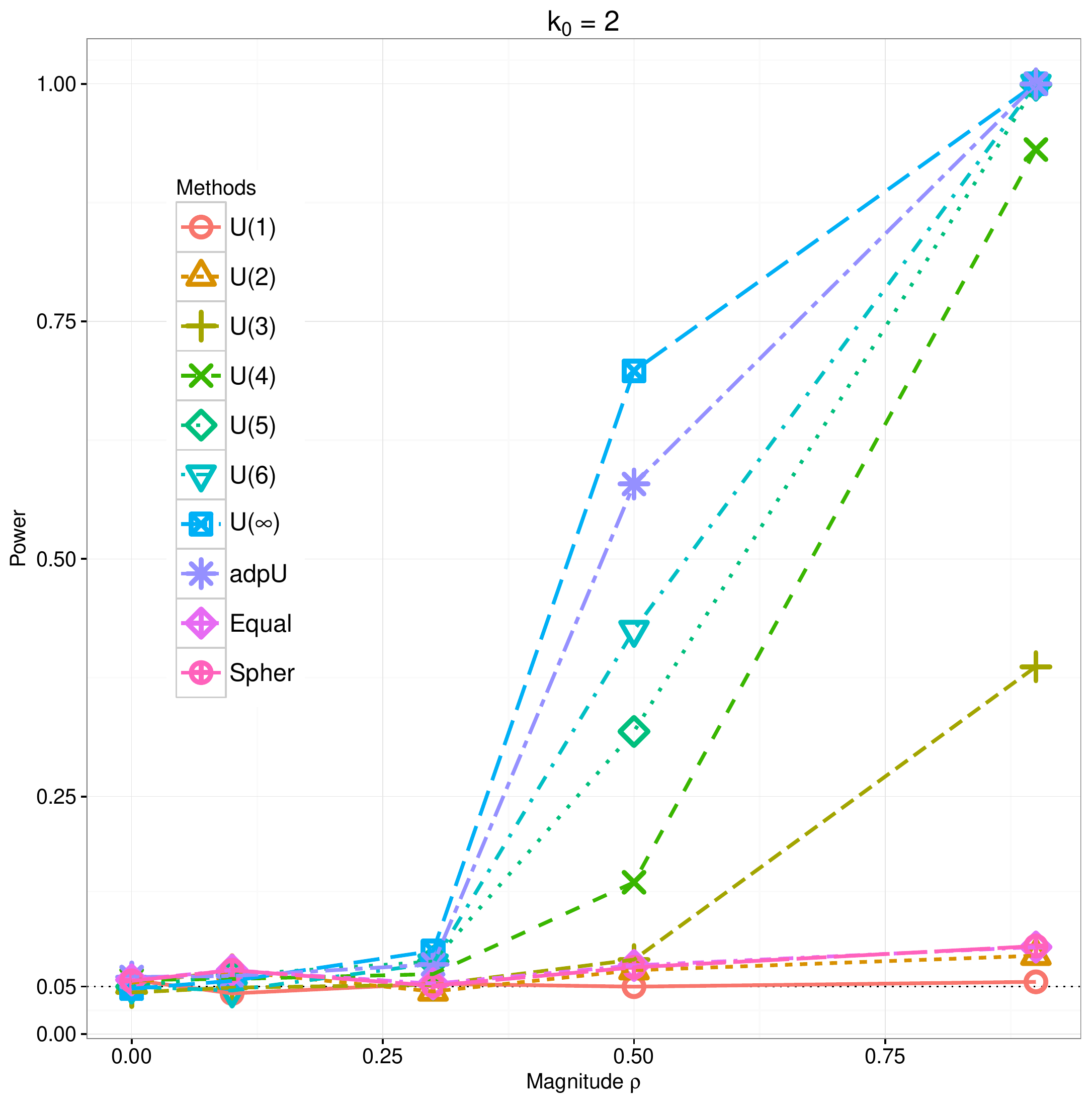} \quad
       \includegraphics[width=0.4\textwidth,height=0.23\textheight]{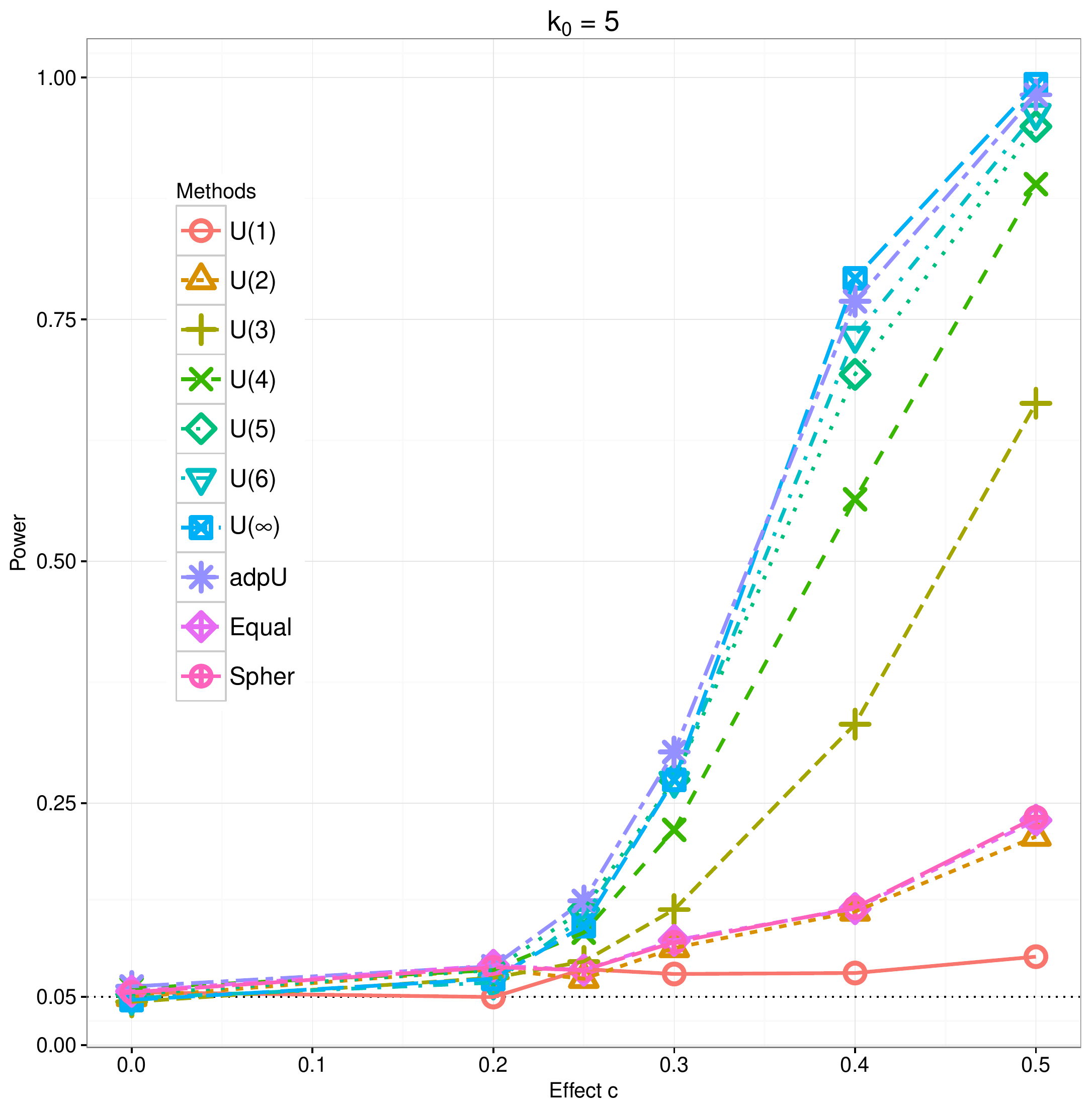}\\
           \includegraphics[width=0.41\textwidth,height=0.23\textheight]{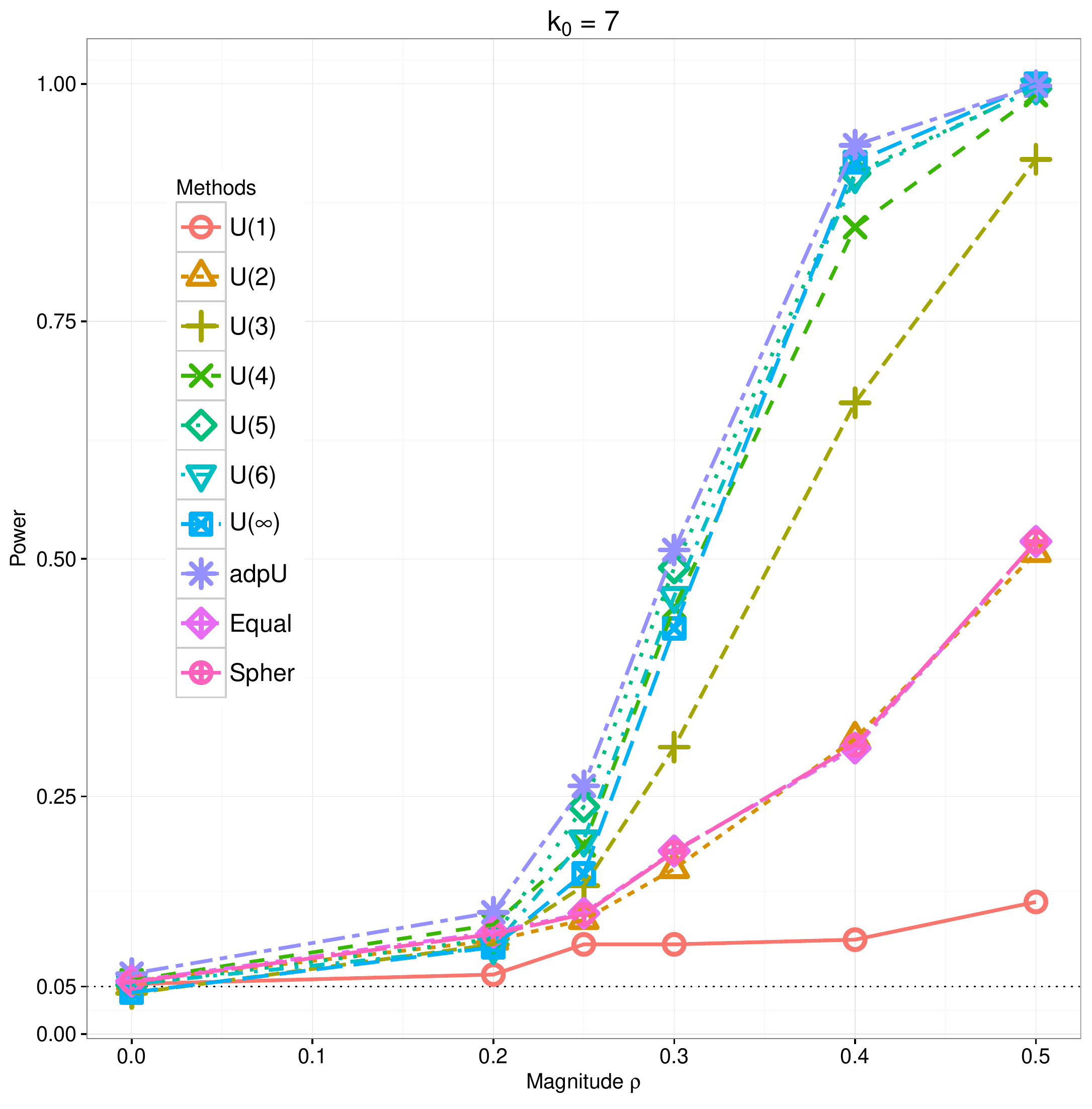} \quad
       \includegraphics[width=0.4\textwidth,height=0.23\textheight]{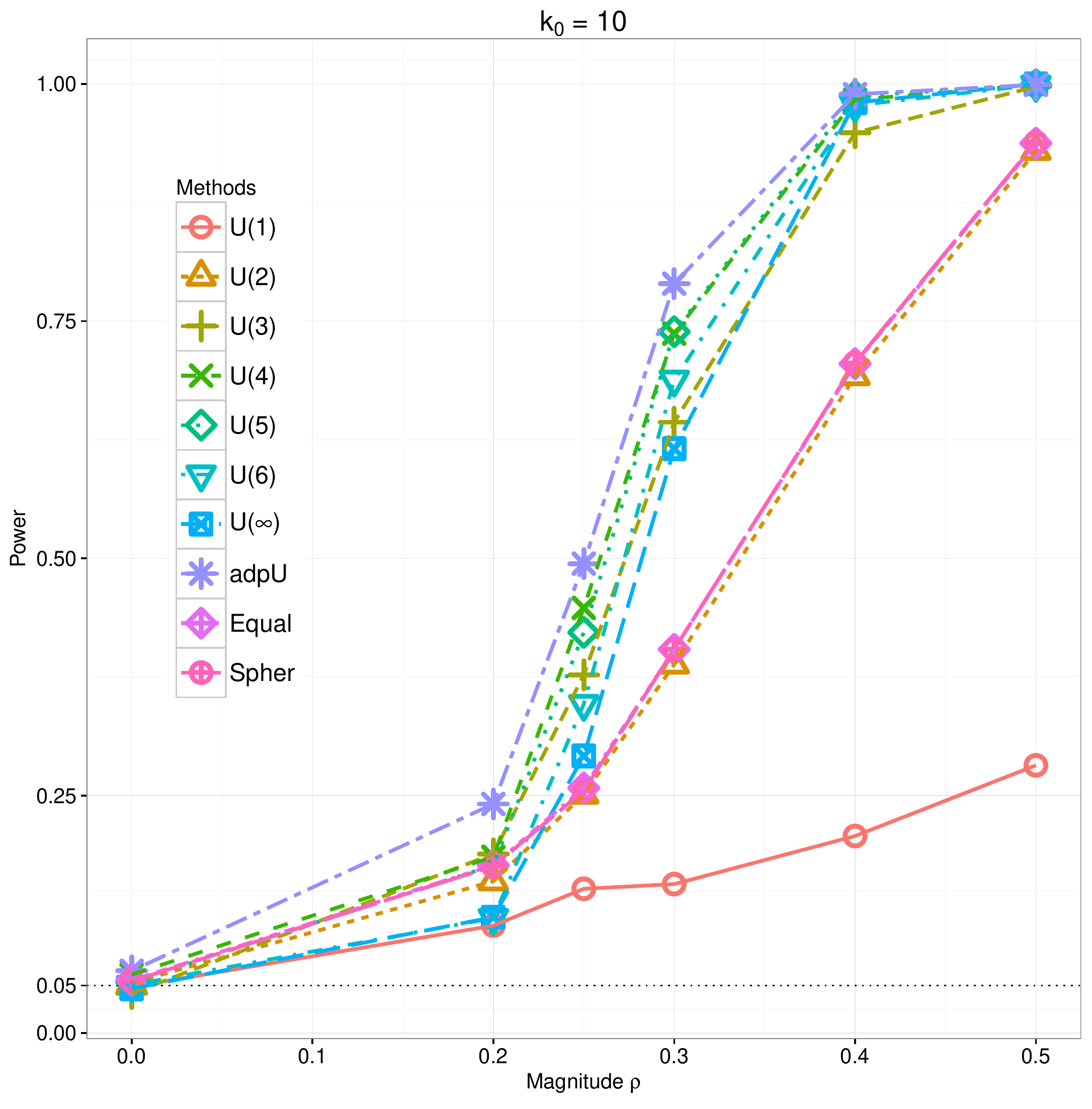}\\
           \includegraphics[width=0.4\textwidth,height=0.23\textheight]{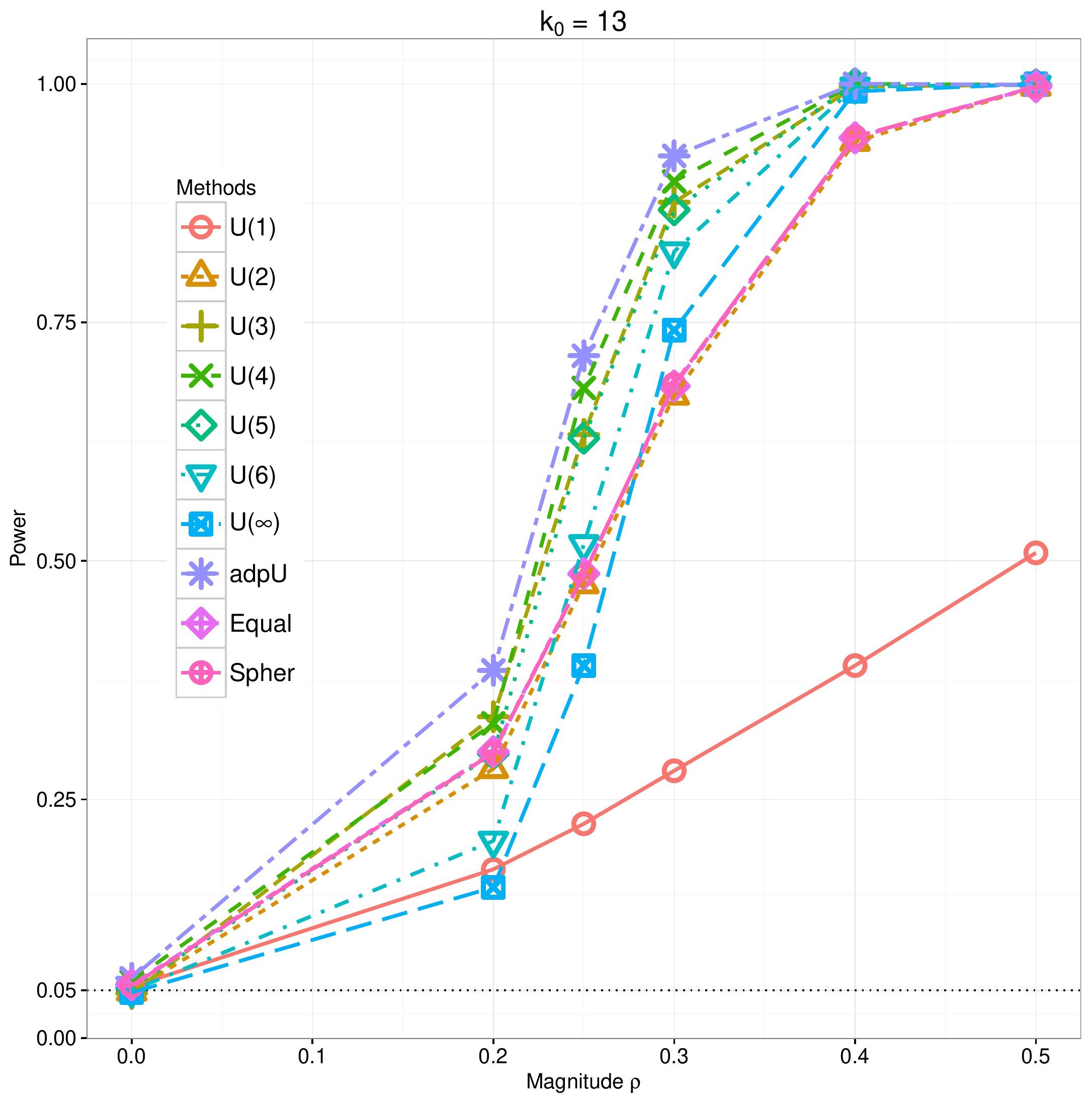}  \quad
       \includegraphics[width=0.4\textwidth,height=0.23\textheight]{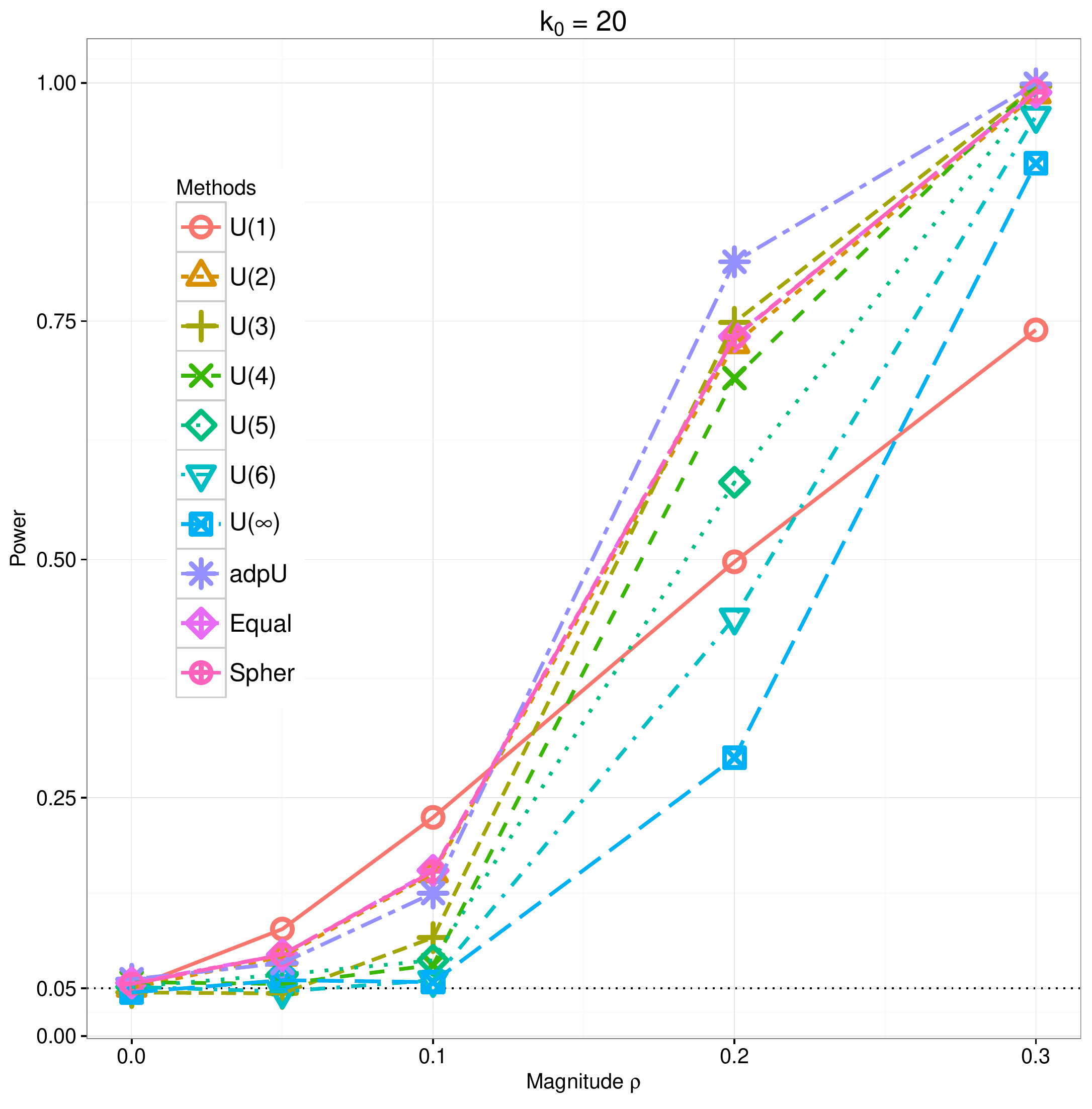} \\
       \includegraphics[width=0.4\textwidth,height=0.23\textheight]{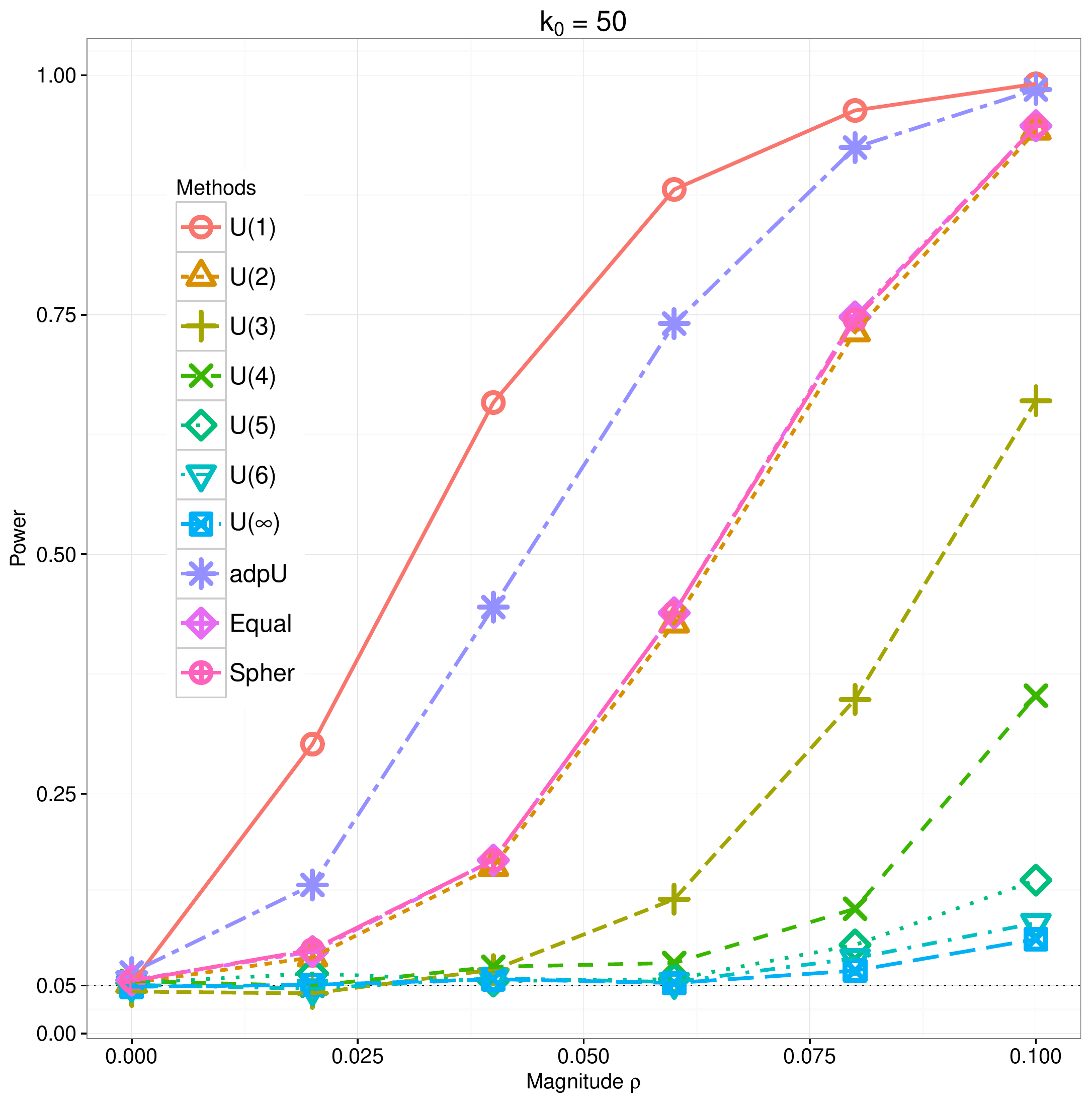}  \quad
       \includegraphics[width=0.4\textwidth,height=0.23\textheight]{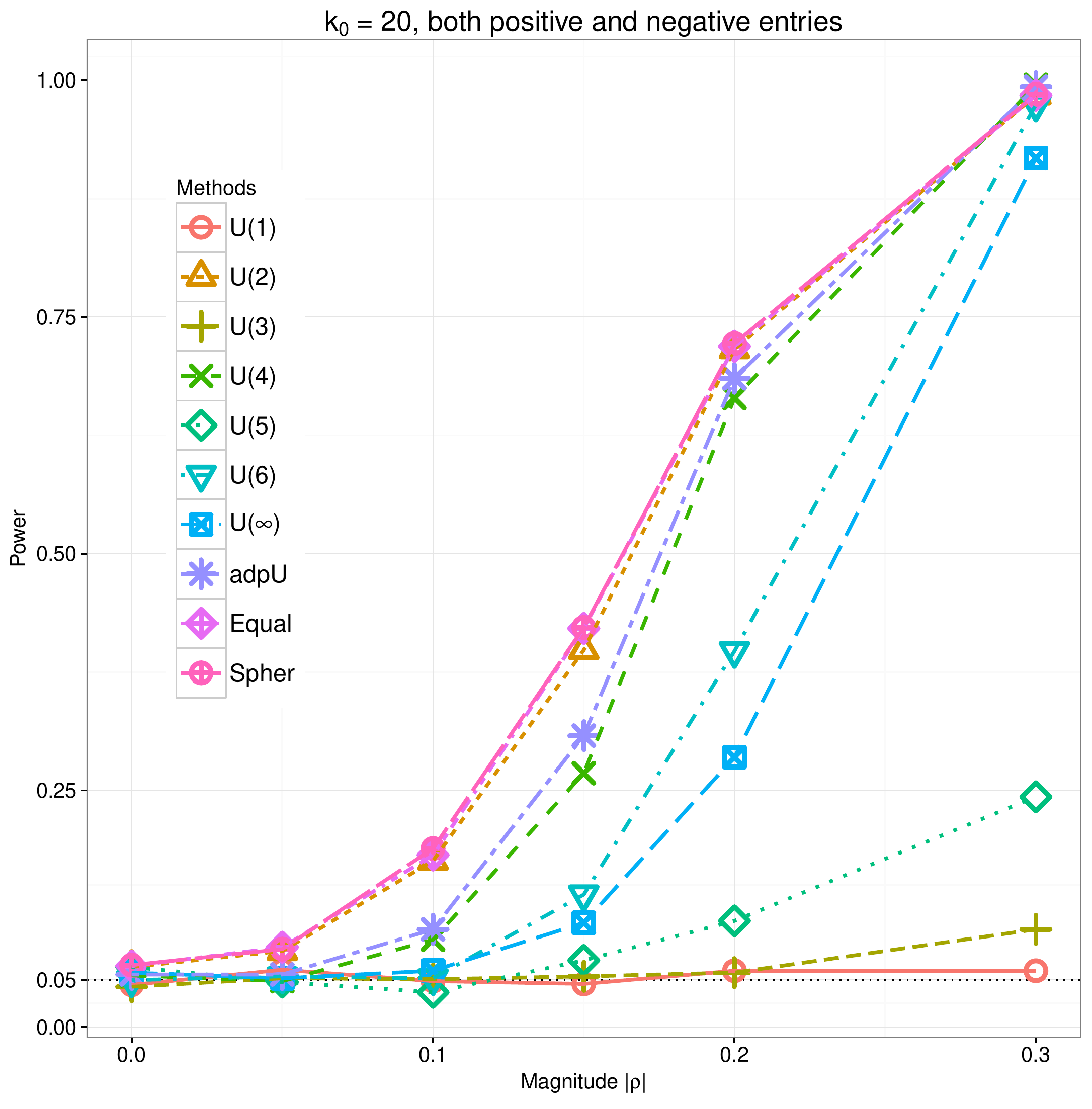}
    \caption{Study 2: $n=100, p=300$. 
    }
    \label{fig:addtionalfigureplot}
\end{figure}

In Figures \ref{fig:study2p600} and \ref{fig:study2p1000},  the meanings of the legends are the same as in Tables \ref{table:nullgaussian} and     \ref{table:nullgamma}, and are already explained in Section  \ref{sec:npcombsimulsize}.  We can find similar patterns to that in Figure \ref{fig:addtionalfigureplot}.

\begin{figure}[!htbp]
    \centering
    \includegraphics[width=0.48\textwidth,height=0.3\textheight]{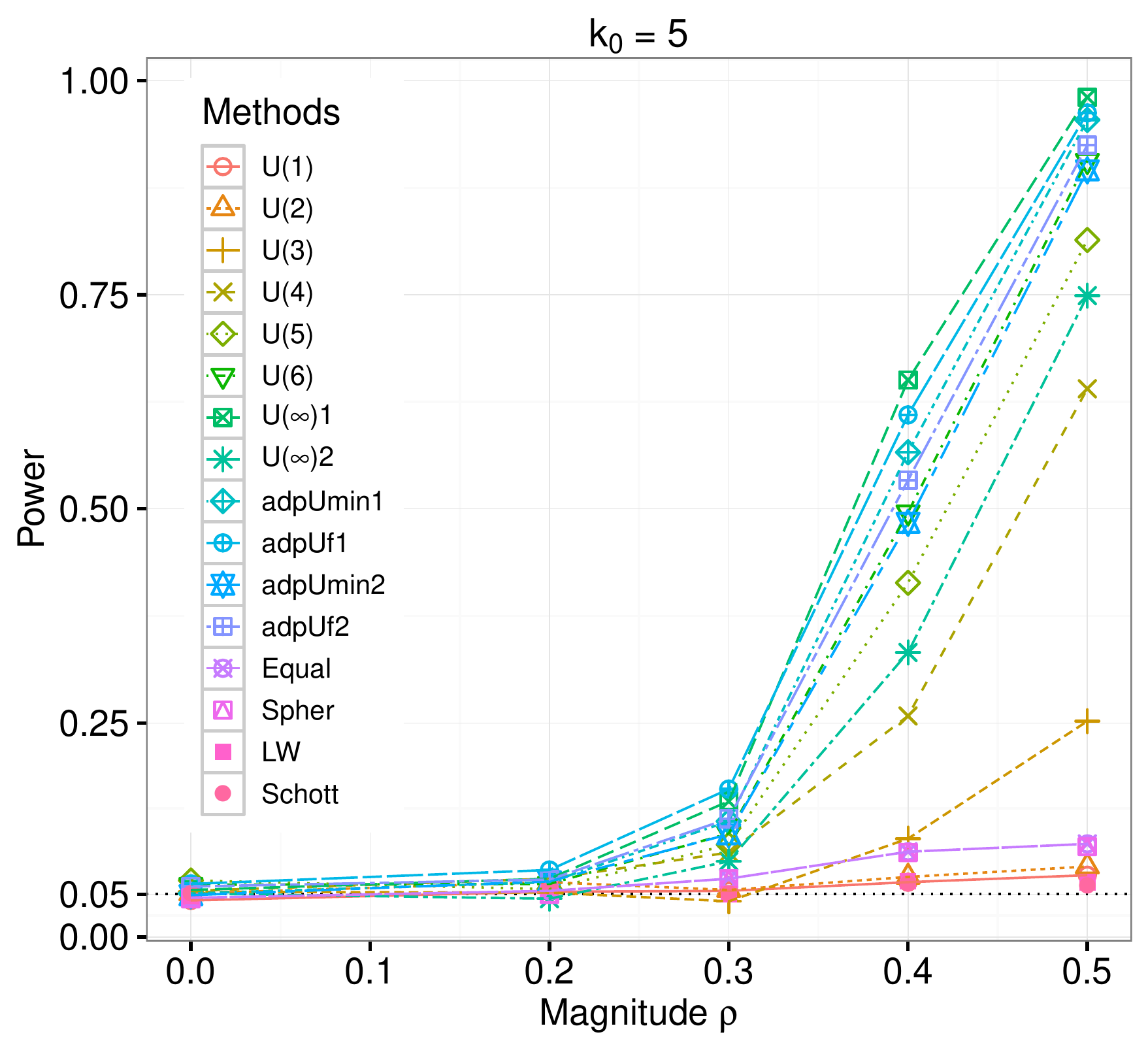} \quad
       \includegraphics[width=0.48\textwidth,height=0.3\textheight]{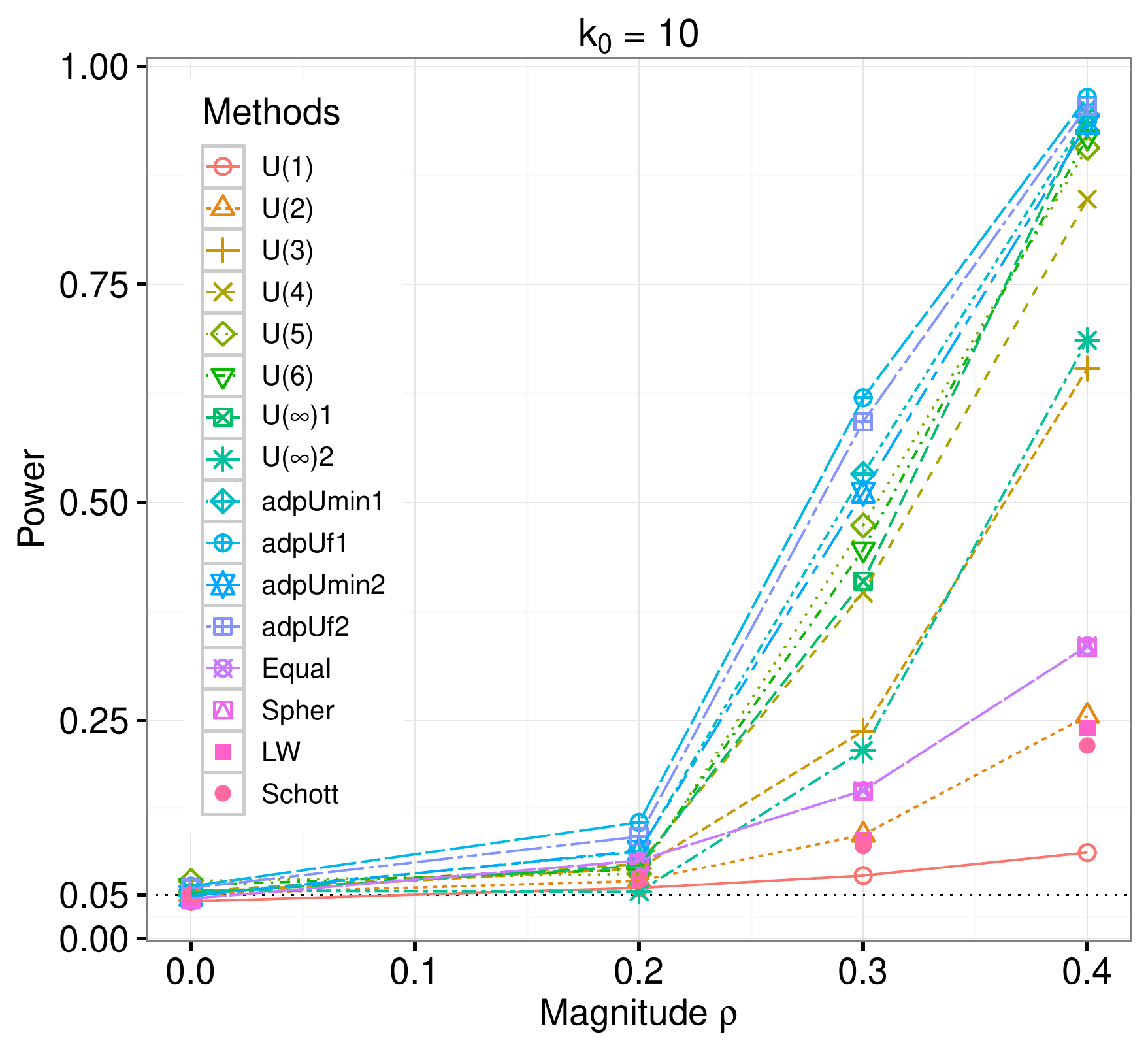}\\
           \includegraphics[width=0.48\textwidth,height=0.3\textheight]{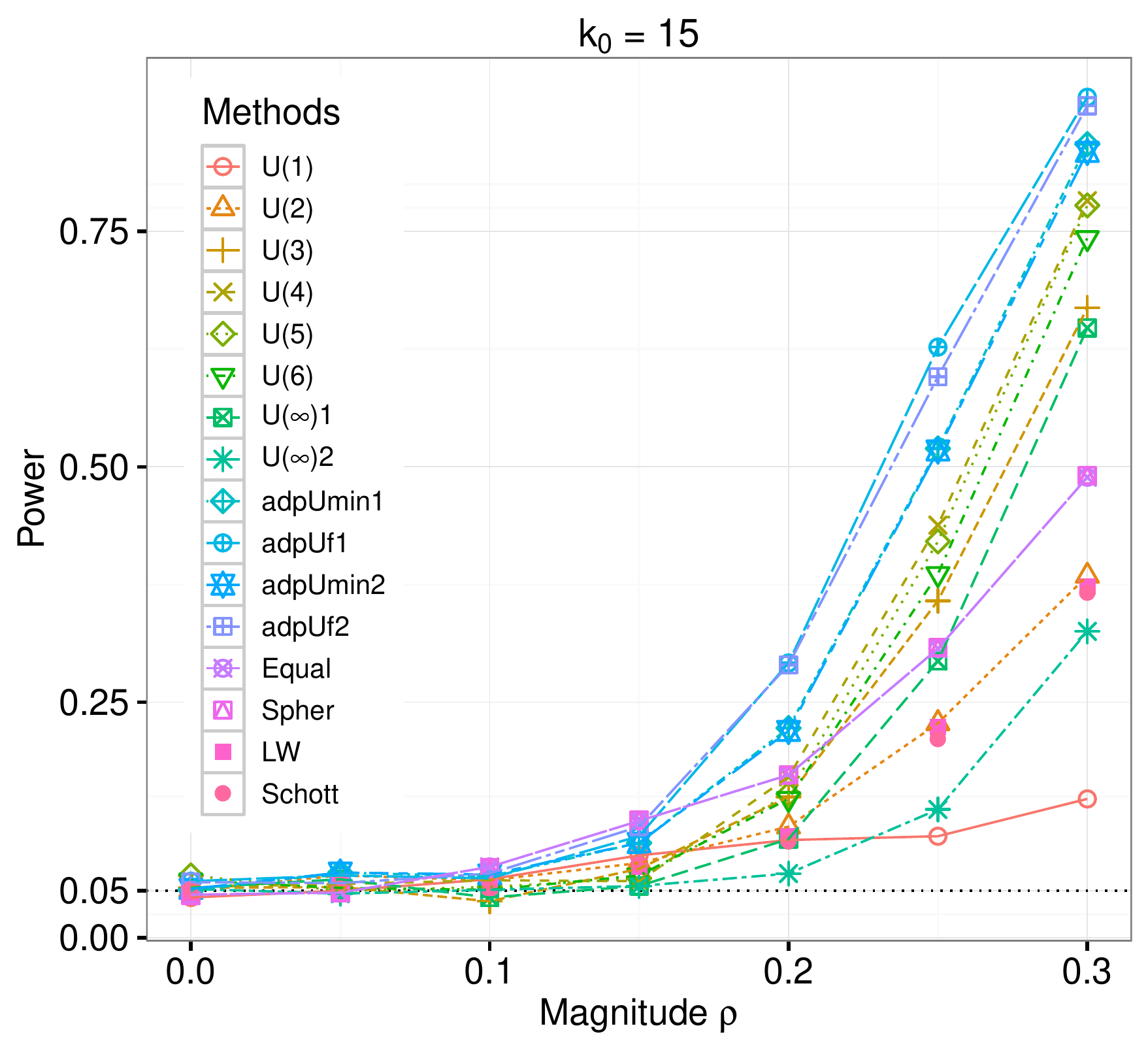} \quad
       \includegraphics[width=0.48\textwidth,height=0.3\textheight]{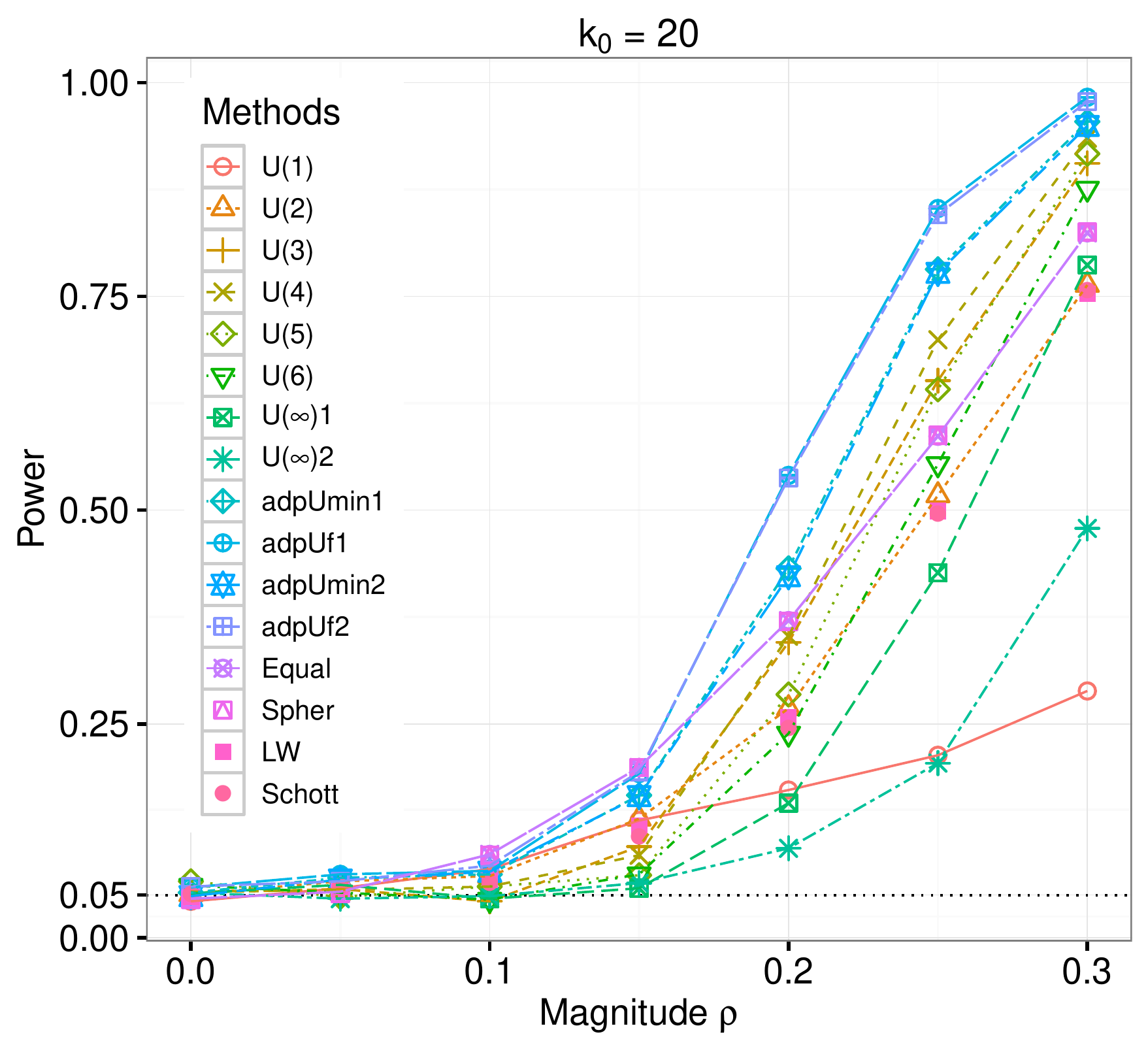} 
       \includegraphics[width=0.48\textwidth,height=0.3\textheight]{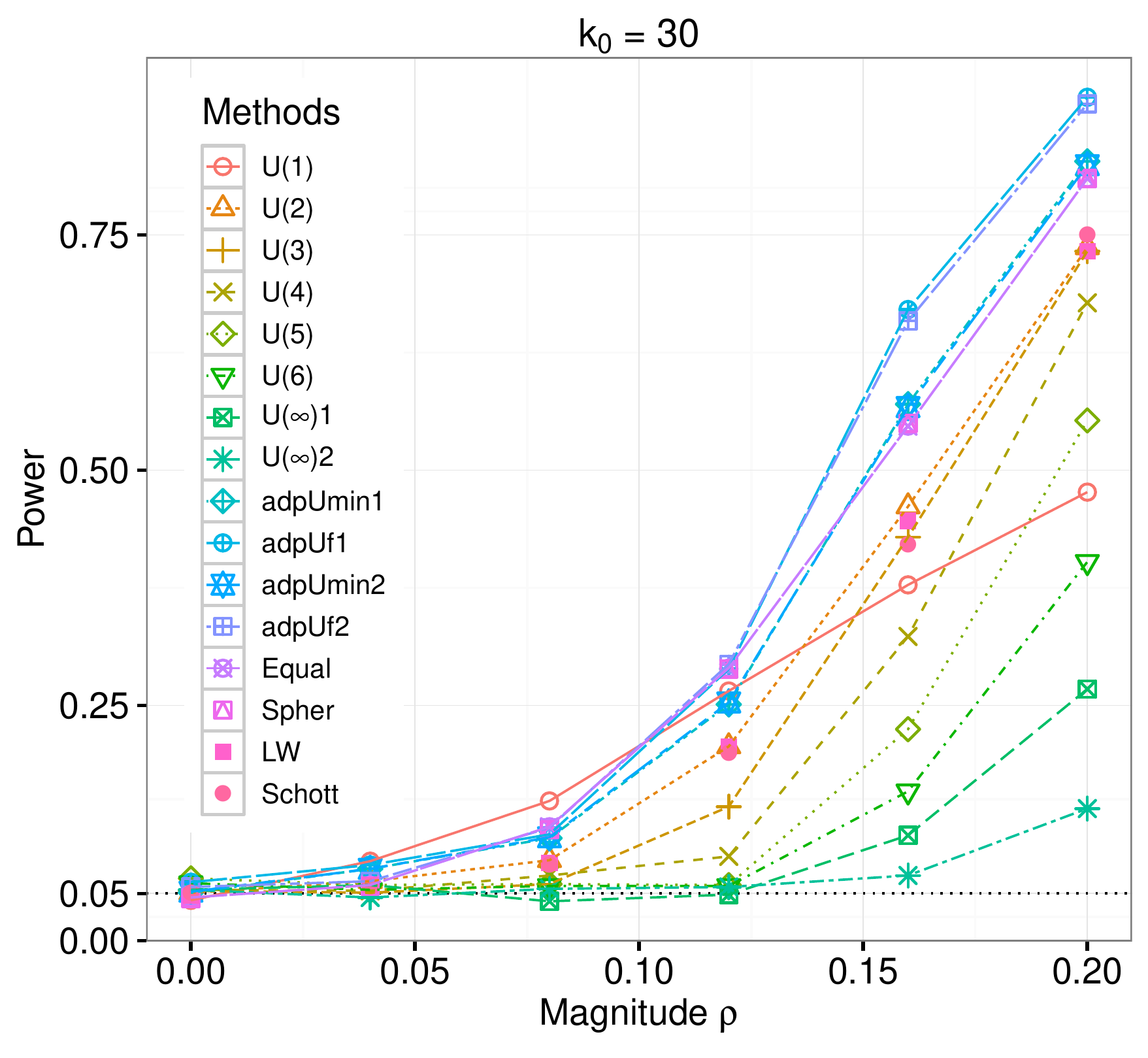} \quad
       \includegraphics[width=0.48\textwidth,height=0.3\textheight]{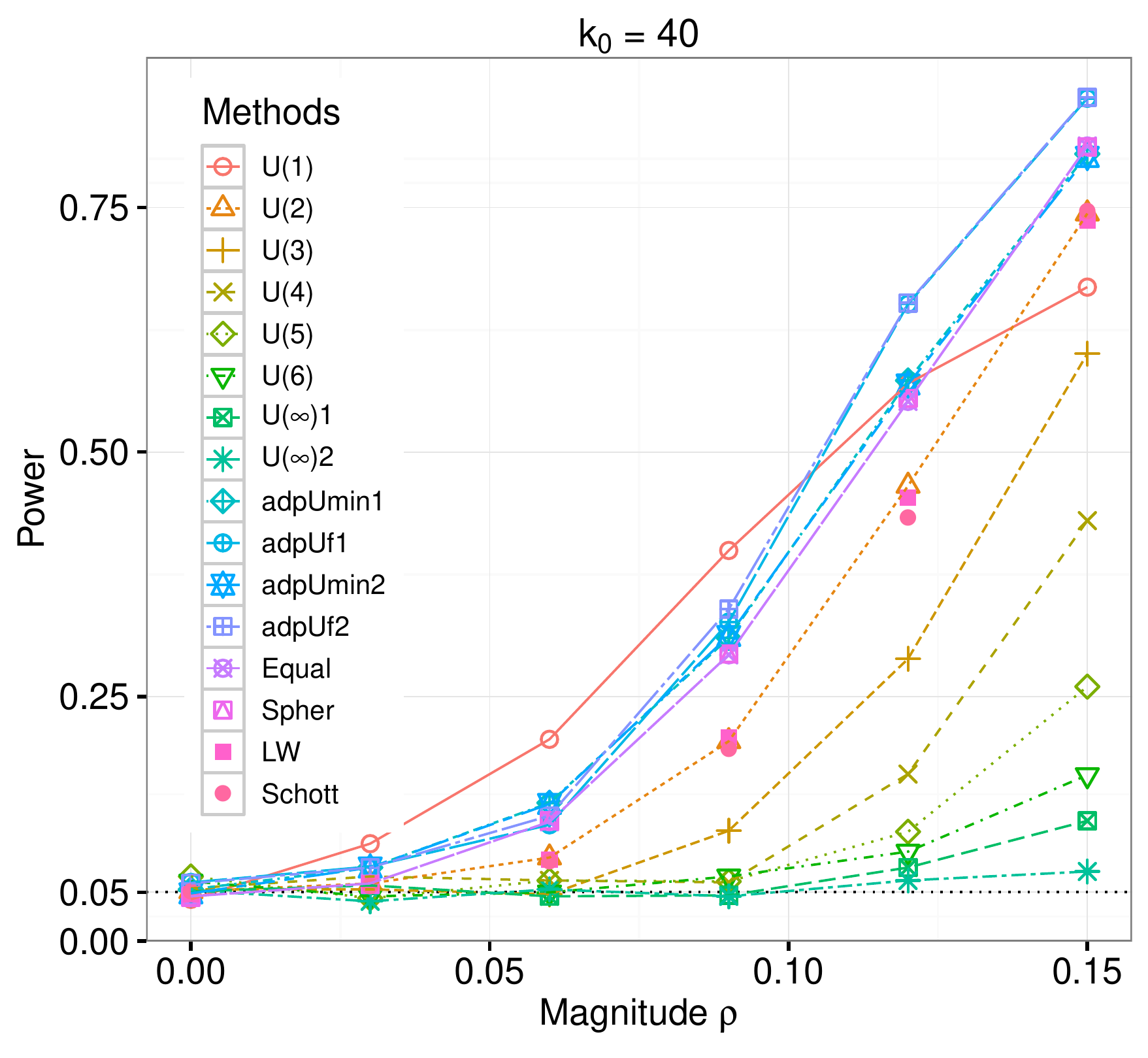}  
         \caption{Study 2: $n=100, p=600$.  
    } 
    \label{fig:study2p600}
\end{figure}

\begin{figure}[!htbp]
    \centering
    \includegraphics[width=0.48\textwidth,height=0.3\textheight]{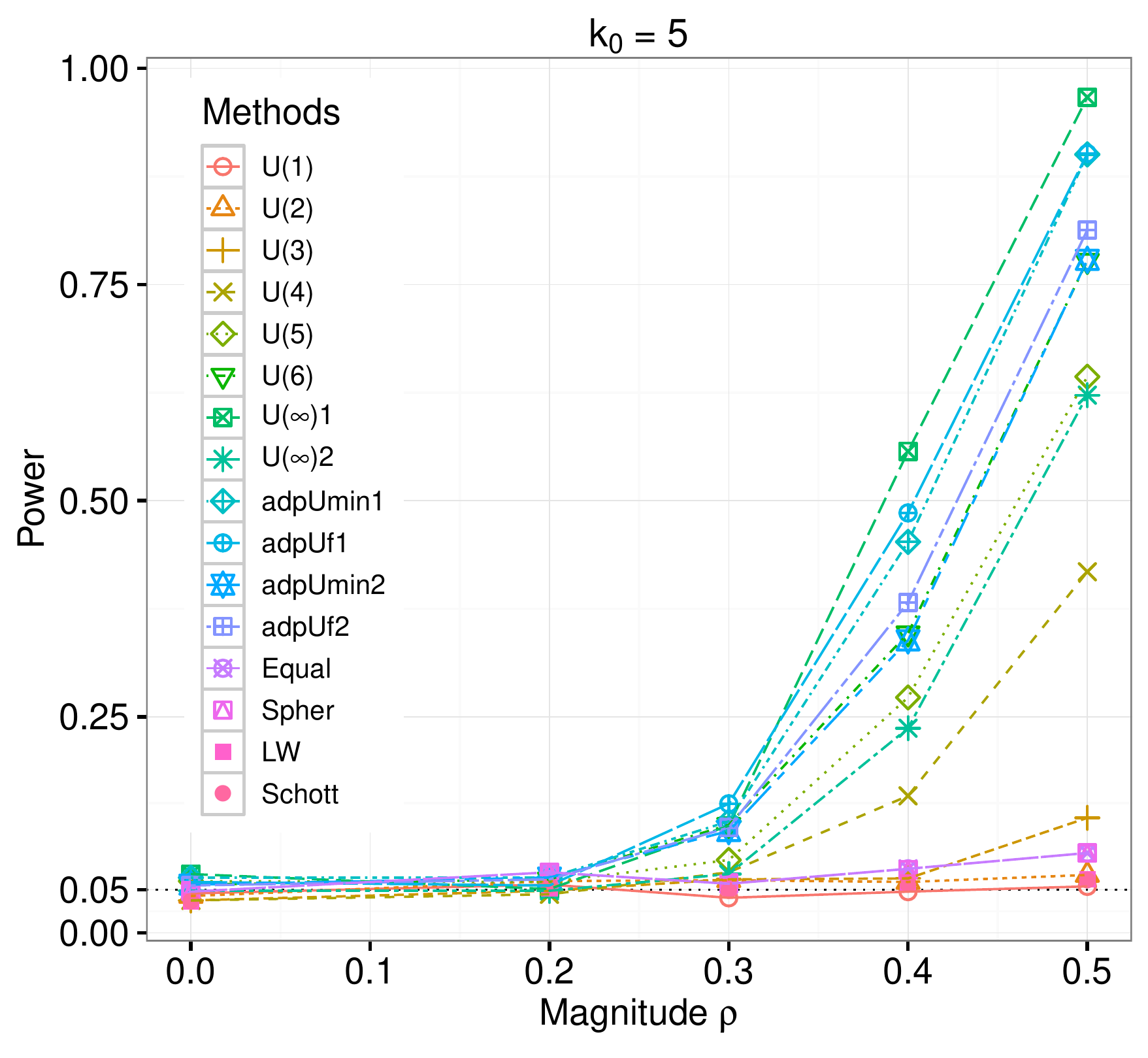} \quad
       \includegraphics[width=0.48\textwidth,height=0.3\textheight]{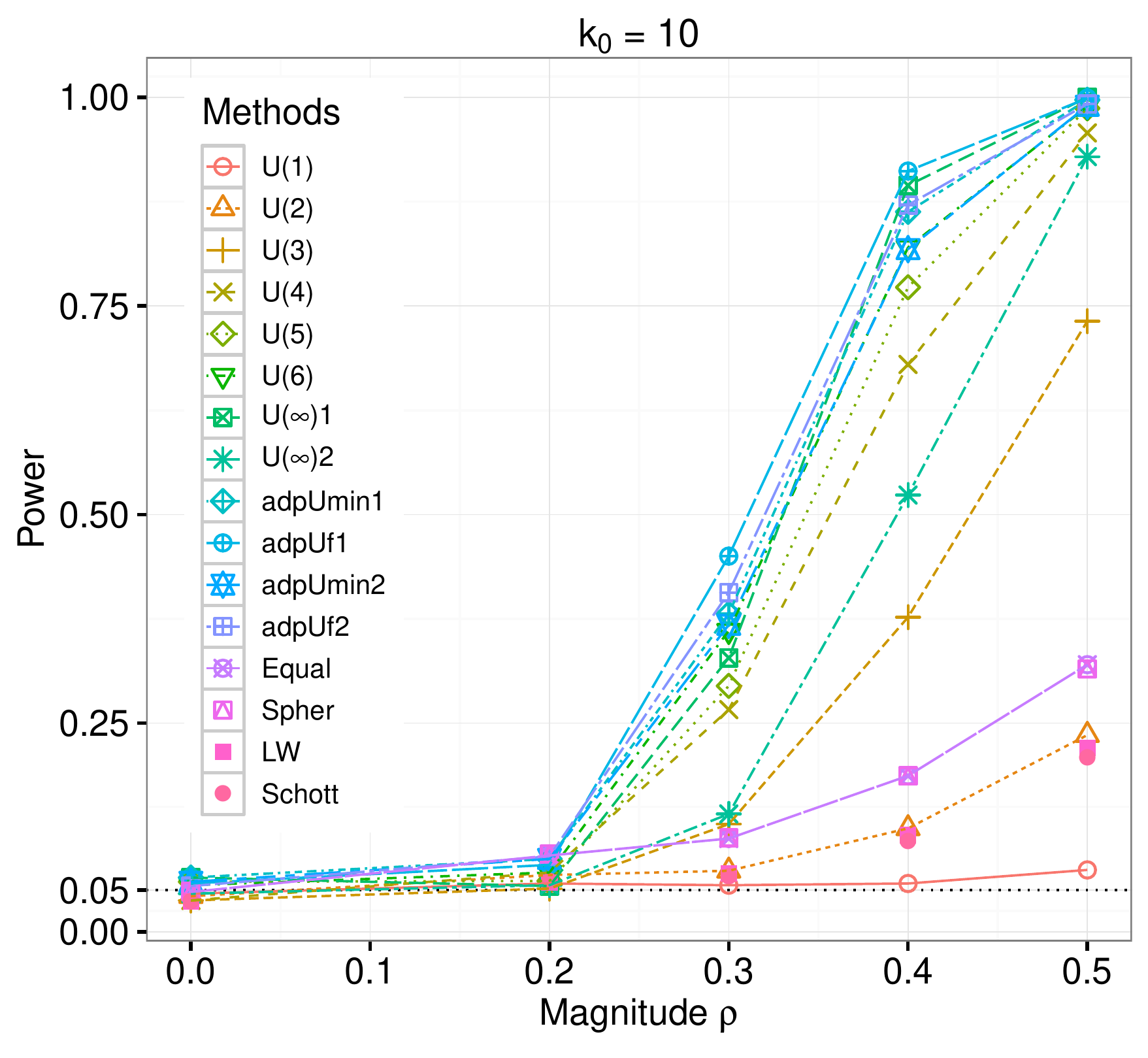}\\
           \includegraphics[width=0.48\textwidth,height=0.3\textheight]{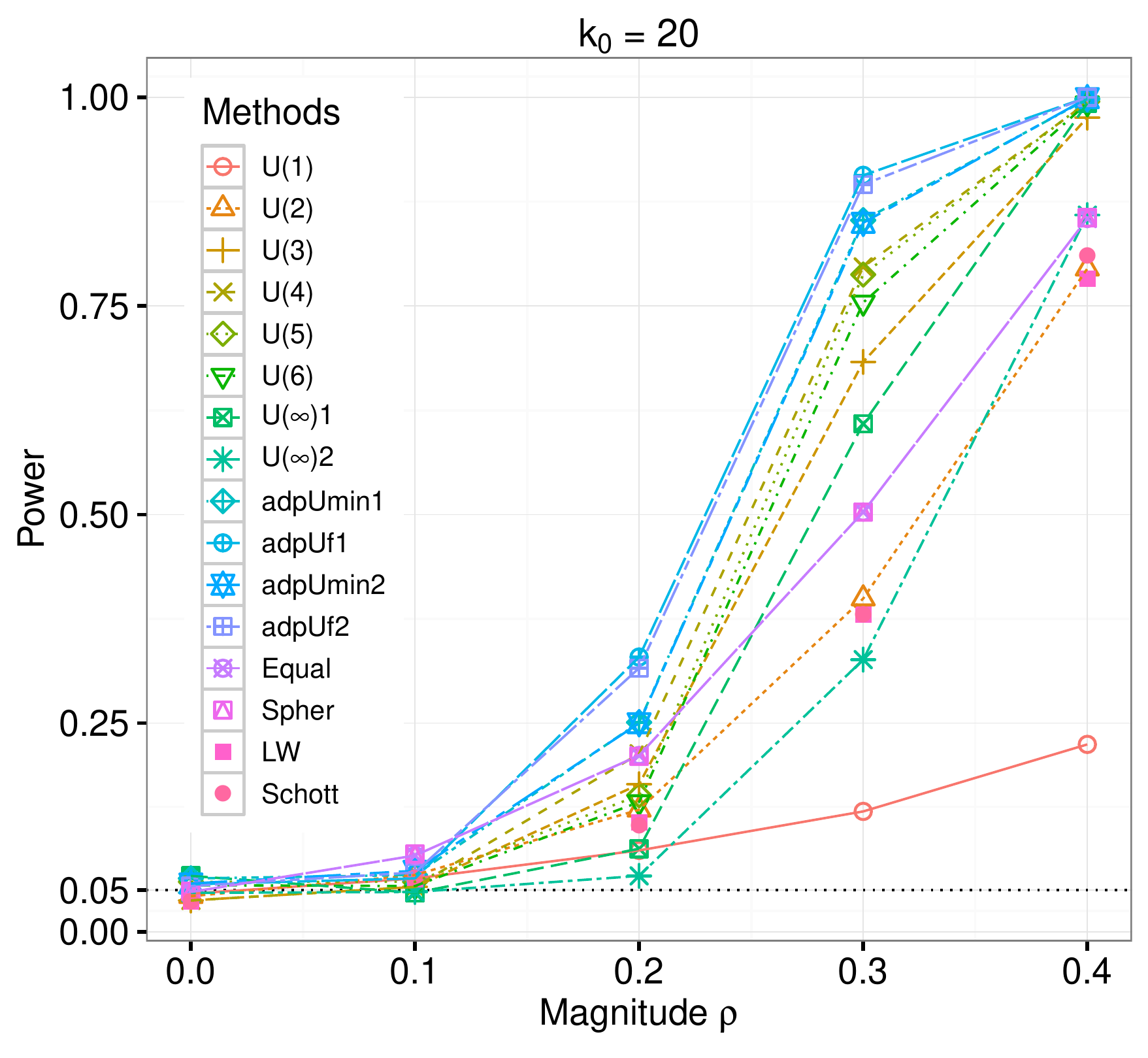} \quad
       \includegraphics[width=0.48\textwidth,height=0.3\textheight]{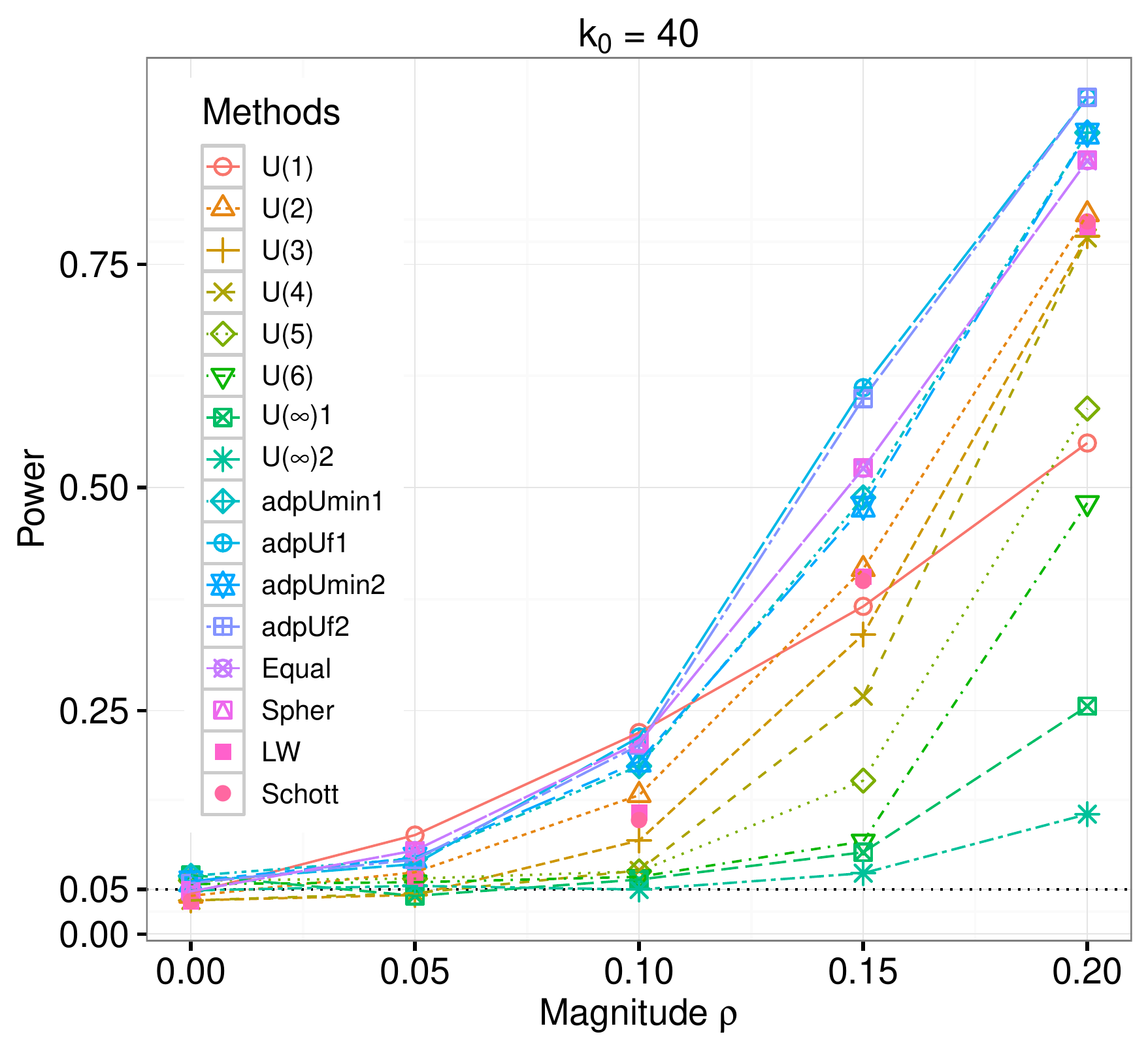}   \caption{Study 2: $n=100, p=1000$.  
    }
    \label{fig:study2p1000}
\end{figure}



\newpage

\subsubsection{Study 3}
We provide supplementary simulations for the third setting in Section \ref{sec:simulation}. In particular, we  generate $n$ i.i.d. $p$-dimensional $\mathbf{x}_i$ for $i=1,\ldots, n$, and $\mathbf{x}_i$ follows multivariate Gaussian distribution with mean zero and covariance $\boldsymbol{\Sigma}_A$. In this case, $\boldsymbol{\Sigma}_A$ is symmetric and positive definite, and has the diagonal being all one and only $|J_A|$ random positions being nonzero with value $\rho$. Note that here $\rho$ represents  the magnitude of the  alternative signal; and $|J_A|$ represents its sparsity level with a larger value indicating a denser alternative, and vice versa. We let $|J_A|$ and $\rho$ vary to examine how the power changes correspondingly. We take $(n,p)\in \{(100,600), (100, 1000)\}$, and provide the results in the following Figures \ref{fig:study3p600}--\ref{fig:study3p1000} respectively. The meanings of the legends are the same as in Tables \ref{table:nullgaussian} and     \ref{table:nullgamma}, and are already explained in Section \ref{sec:npcombsimulsize}. We observe similar patterns to that in the figures in Section \ref{sec:study2onecovsim}.

\begin{figure}[!htbp]
    \centering
    \includegraphics[width=0.48\textwidth,height=0.3\textheight]{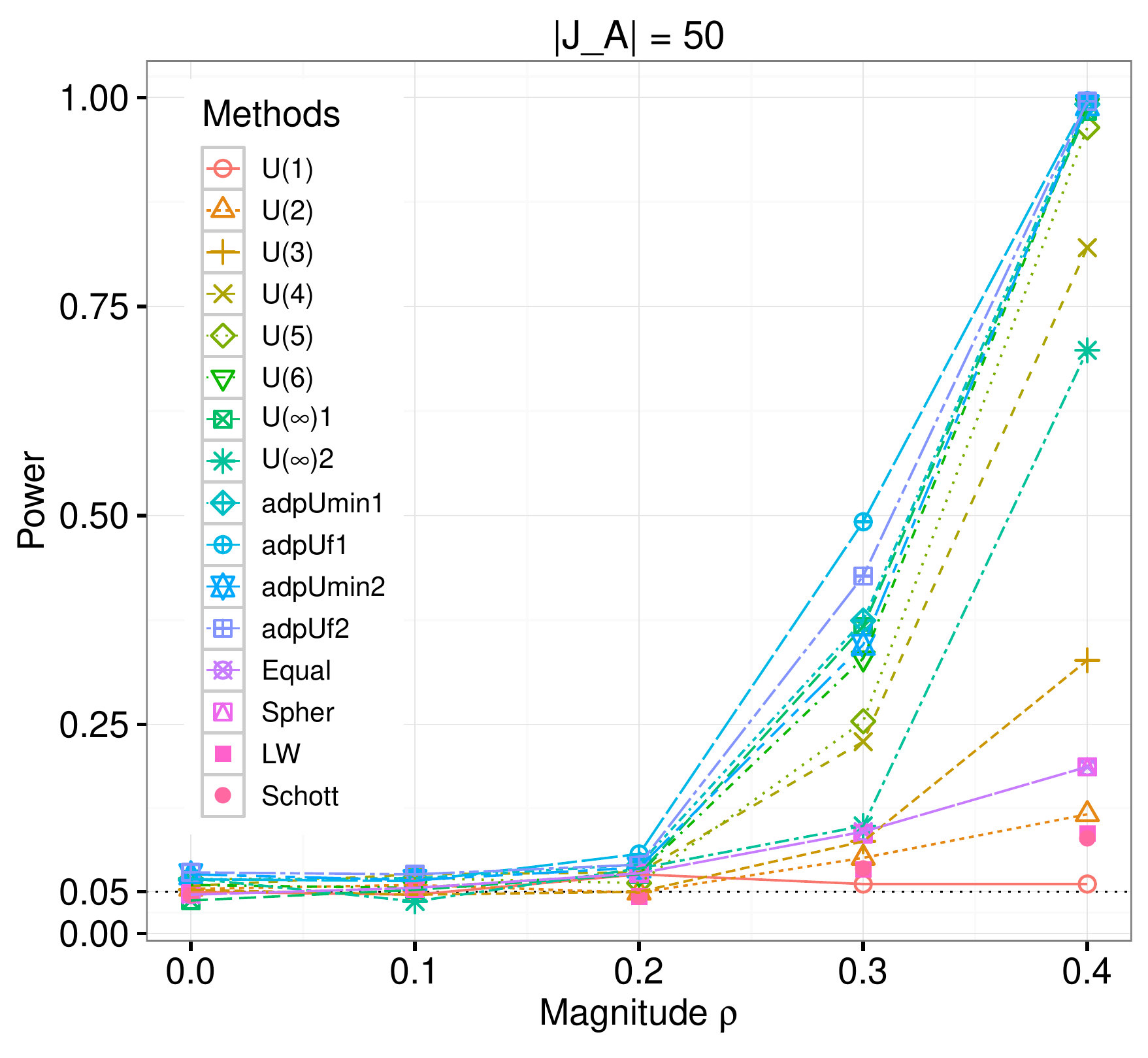} \quad
       \includegraphics[width=0.48\textwidth,height=0.3\textheight]{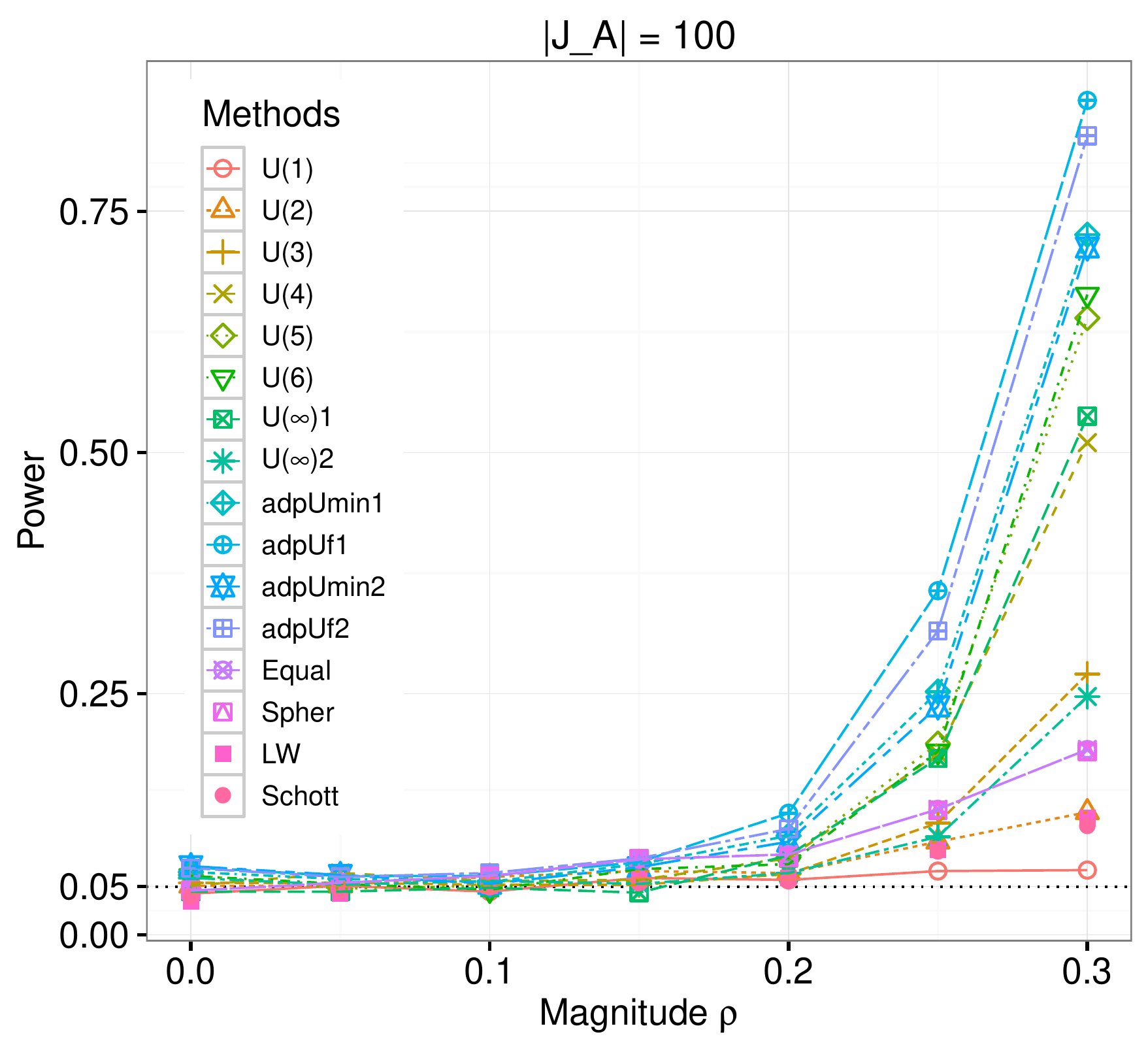}\\
           \includegraphics[width=0.48\textwidth,height=0.3\textheight]{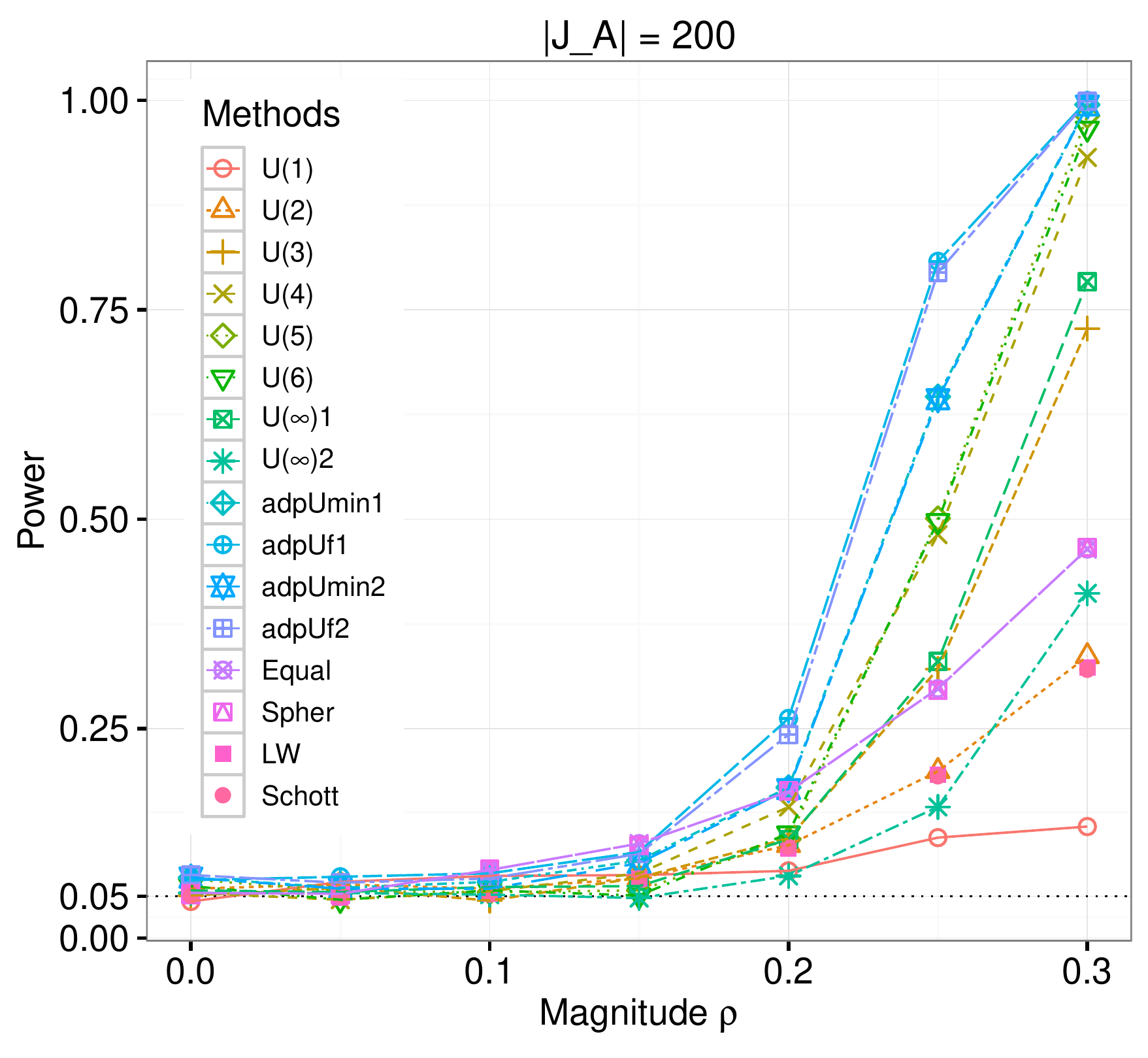} \quad
       \includegraphics[width=0.48\textwidth,height=0.3\textheight]{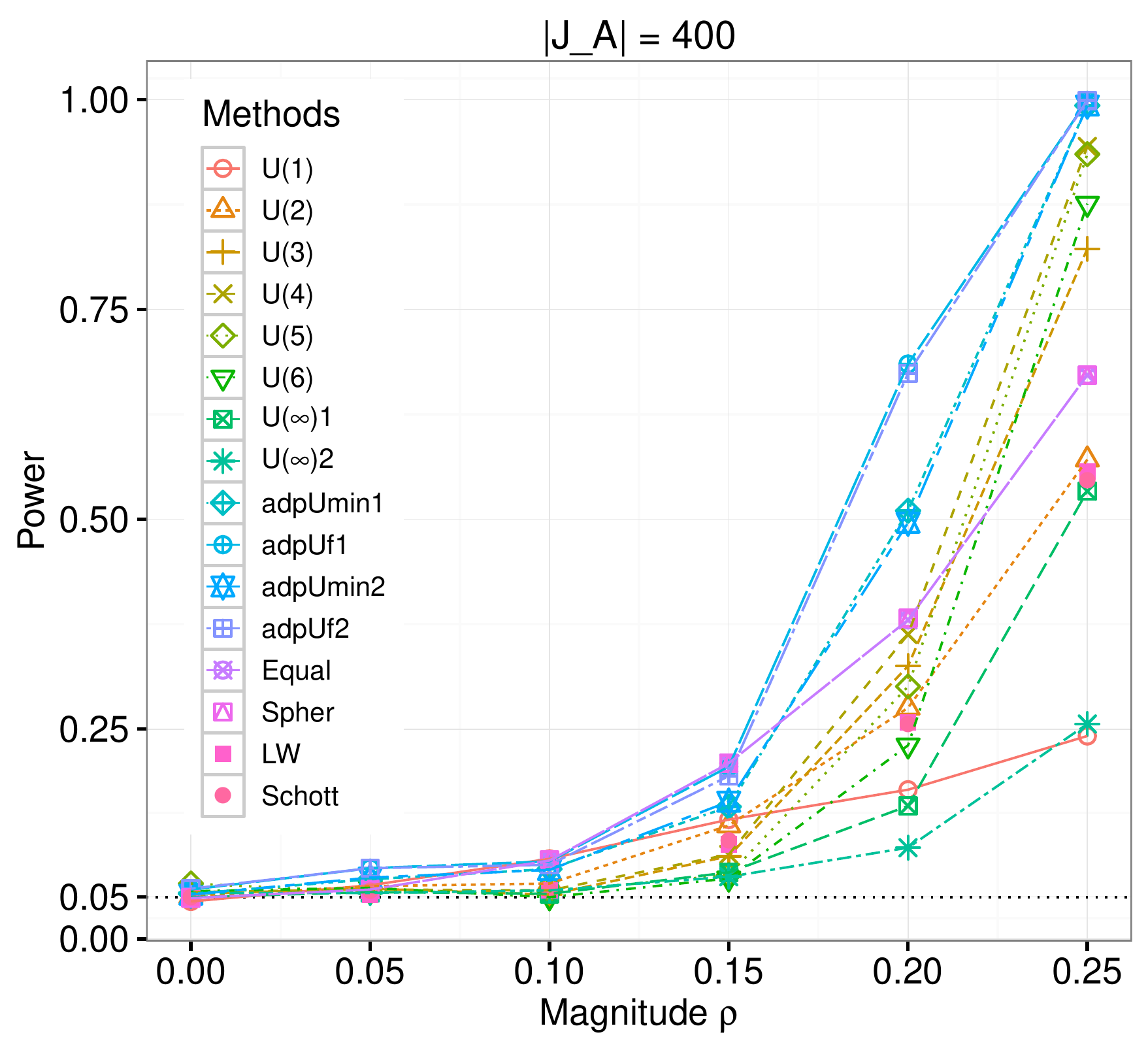}\\
           \includegraphics[width=0.48\textwidth,height=0.3\textheight]{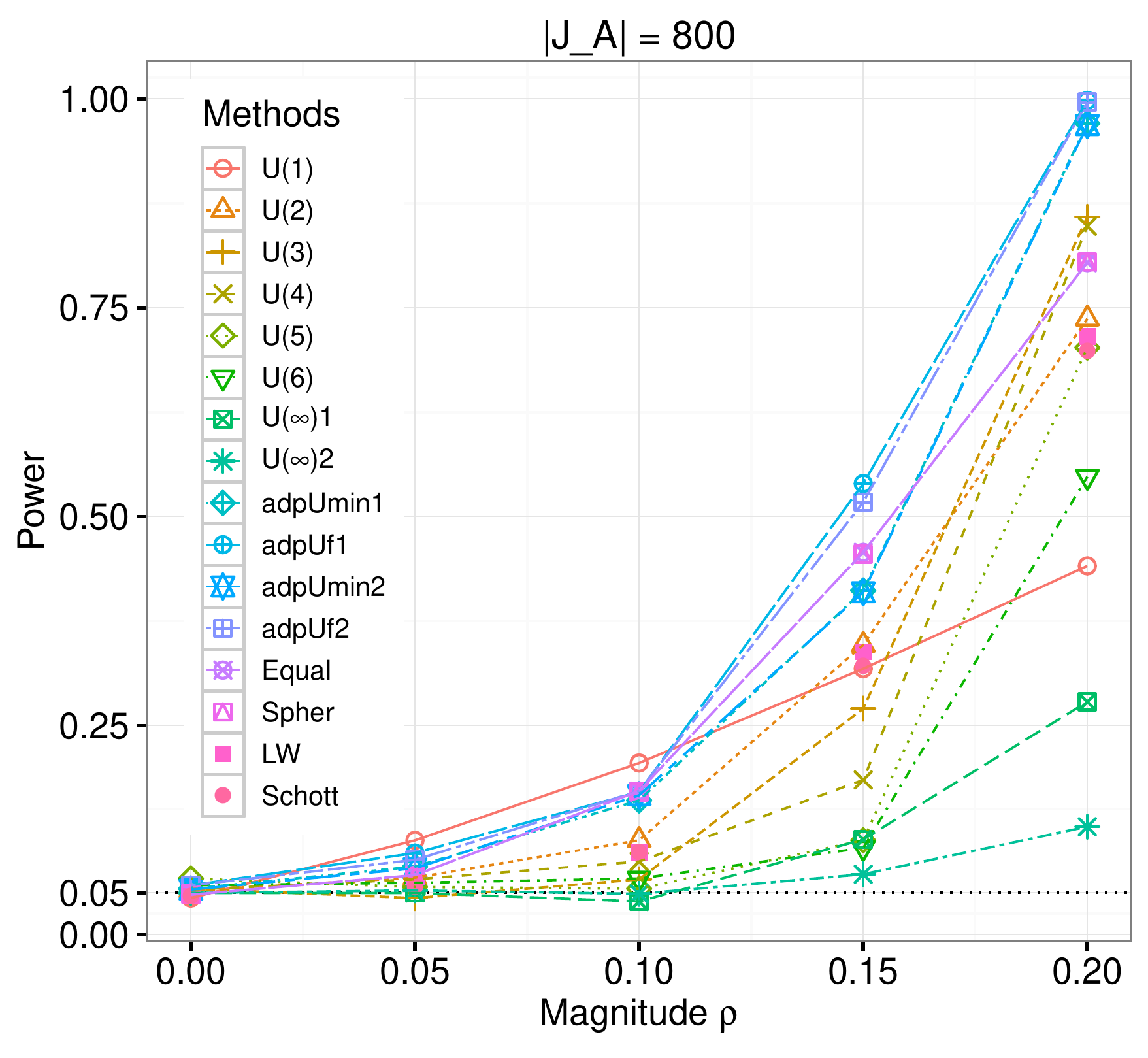}  \quad
       \includegraphics[width=0.48\textwidth,height=0.3\textheight]{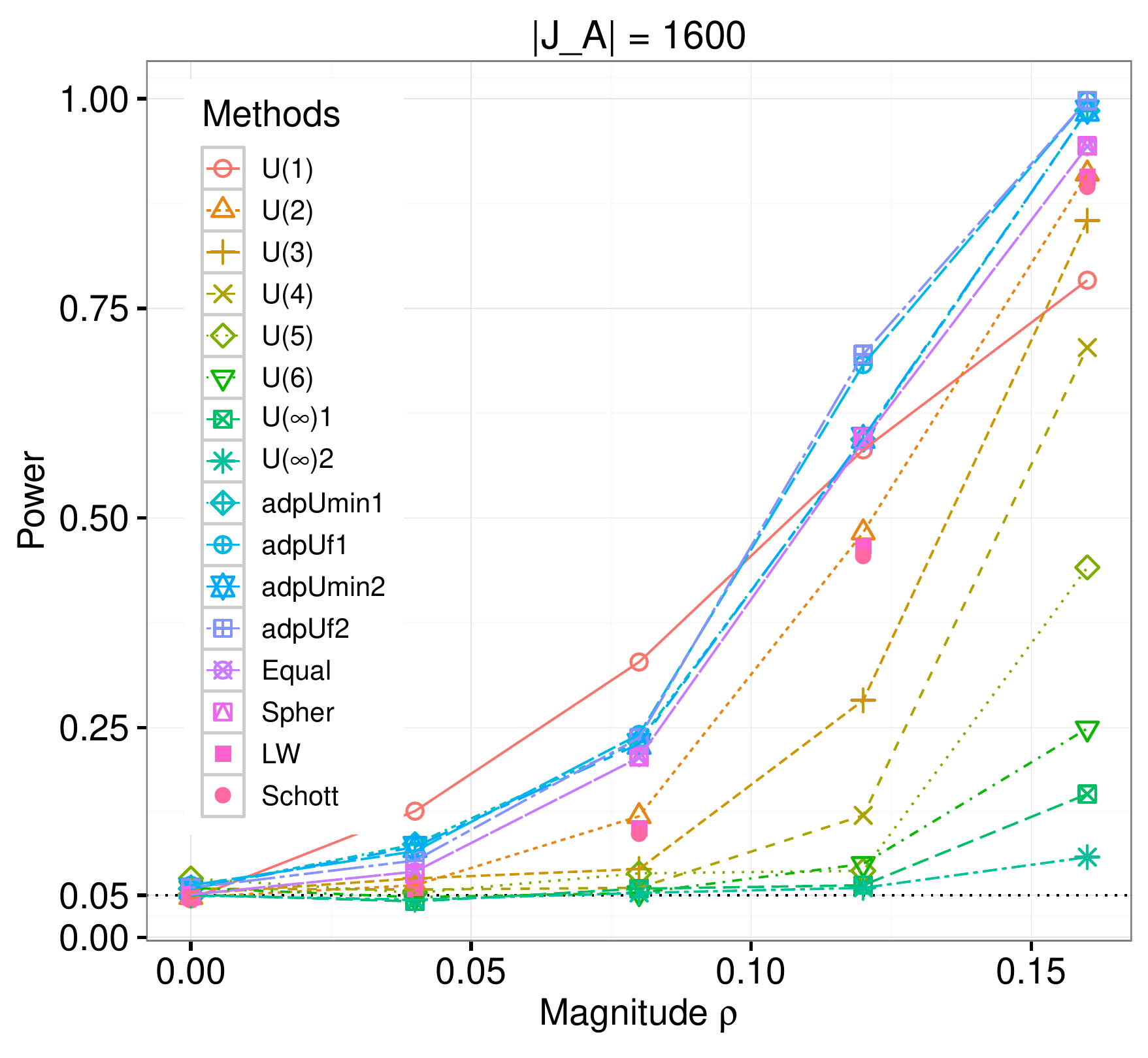}     \caption{Study 3: $n=100, p=600$.  
    }
    \label{fig:study3p600}
\end{figure}

\begin{figure}[!htbp]
    \centering
    \includegraphics[width=0.48\textwidth,height=0.3\textheight]{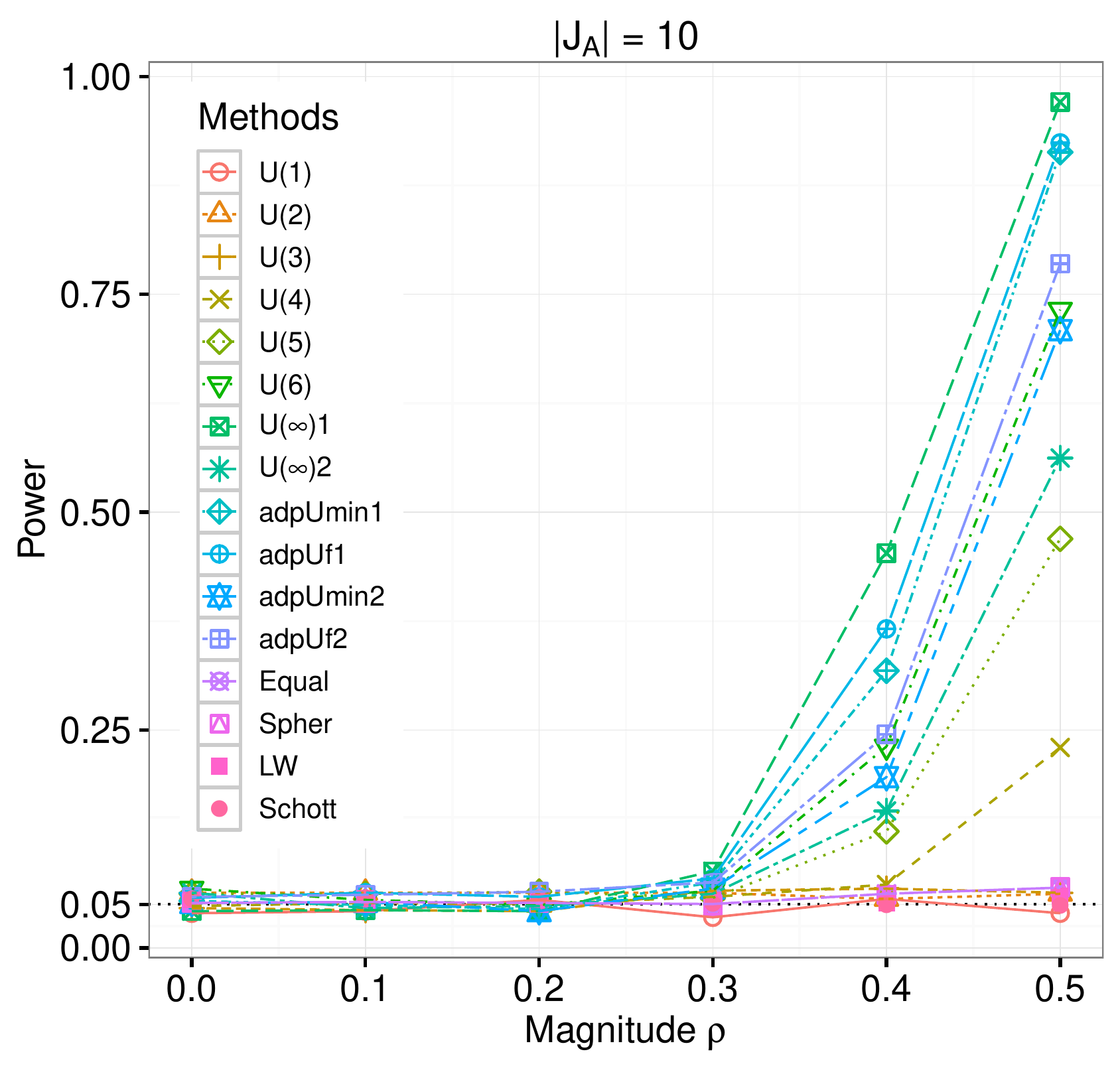} \quad
       \includegraphics[width=0.48\textwidth,height=0.3\textheight]{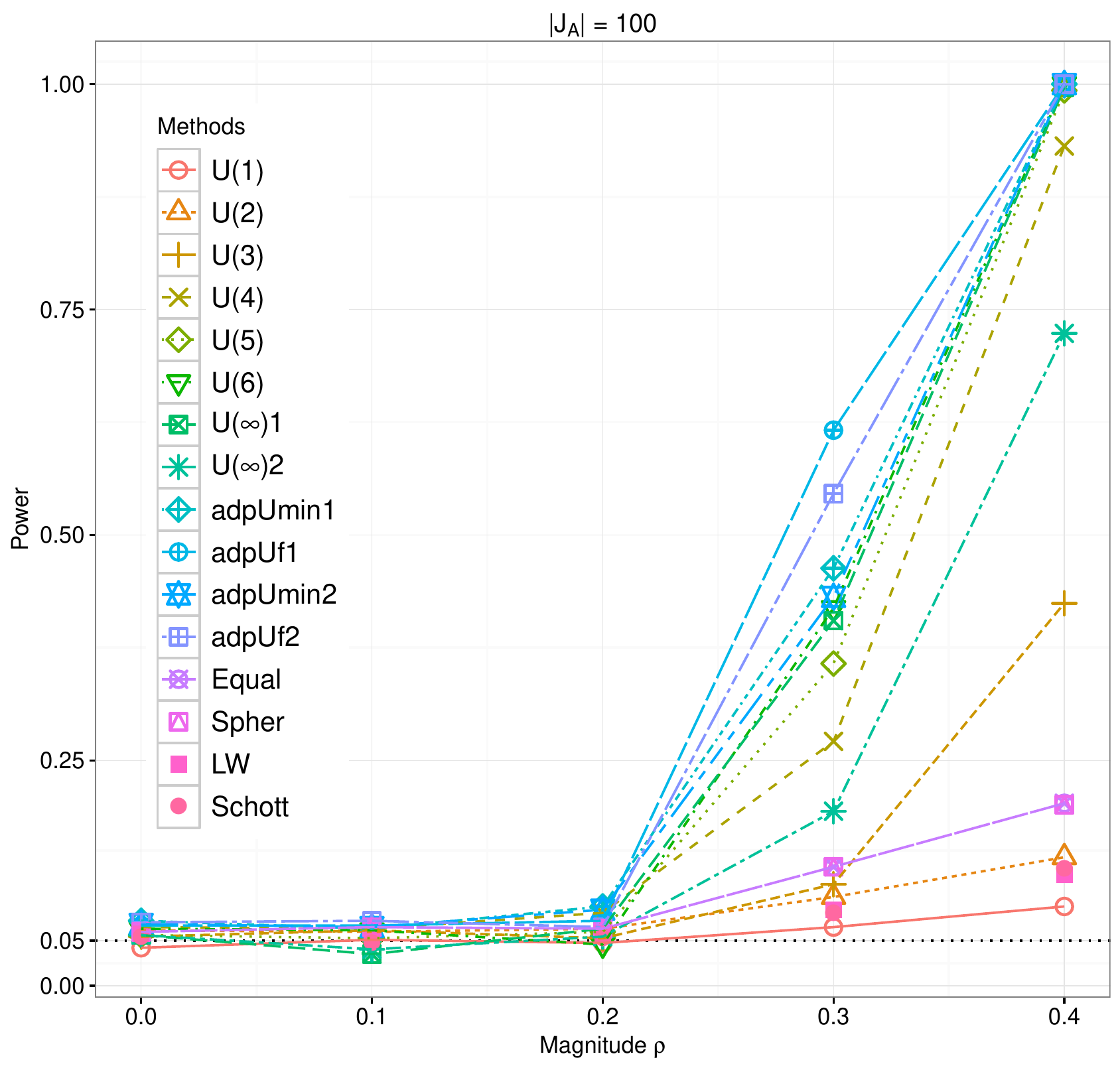}\\
           \includegraphics[width=0.48\textwidth,height=0.3\textheight]{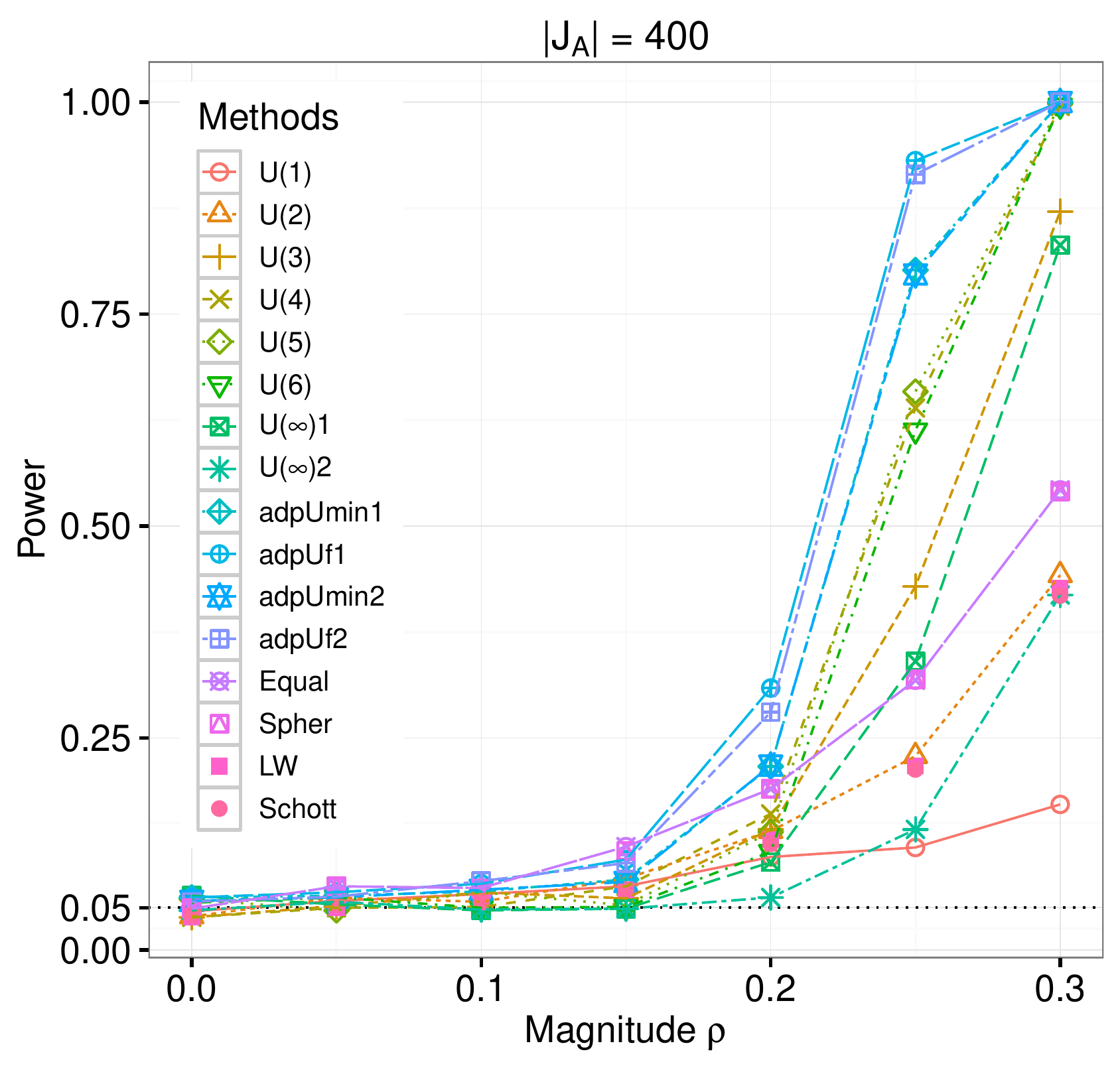} \quad
       \includegraphics[width=0.48\textwidth,height=0.3\textheight]{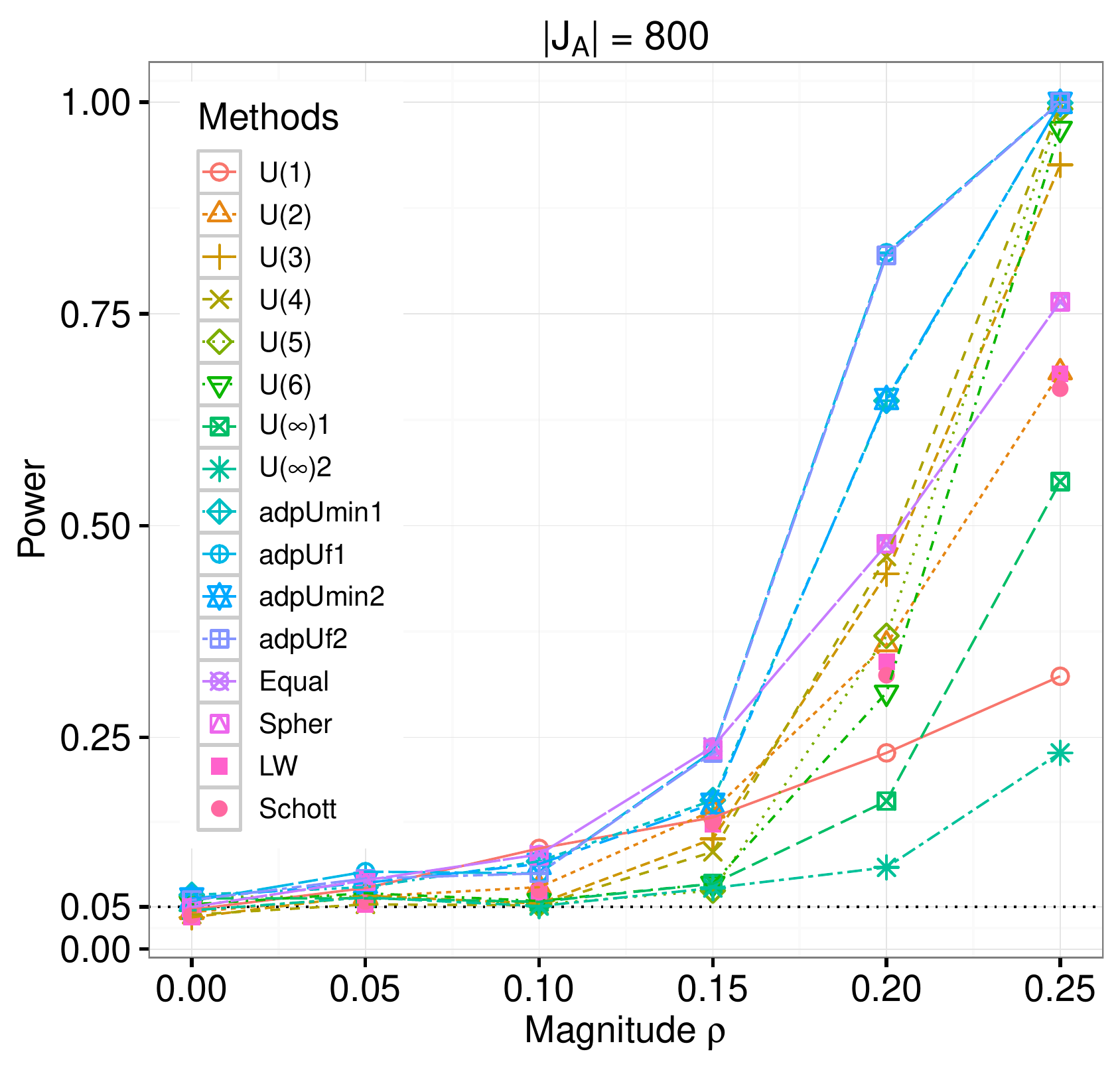}\\
           \includegraphics[width=0.48\textwidth,height=0.3\textheight]{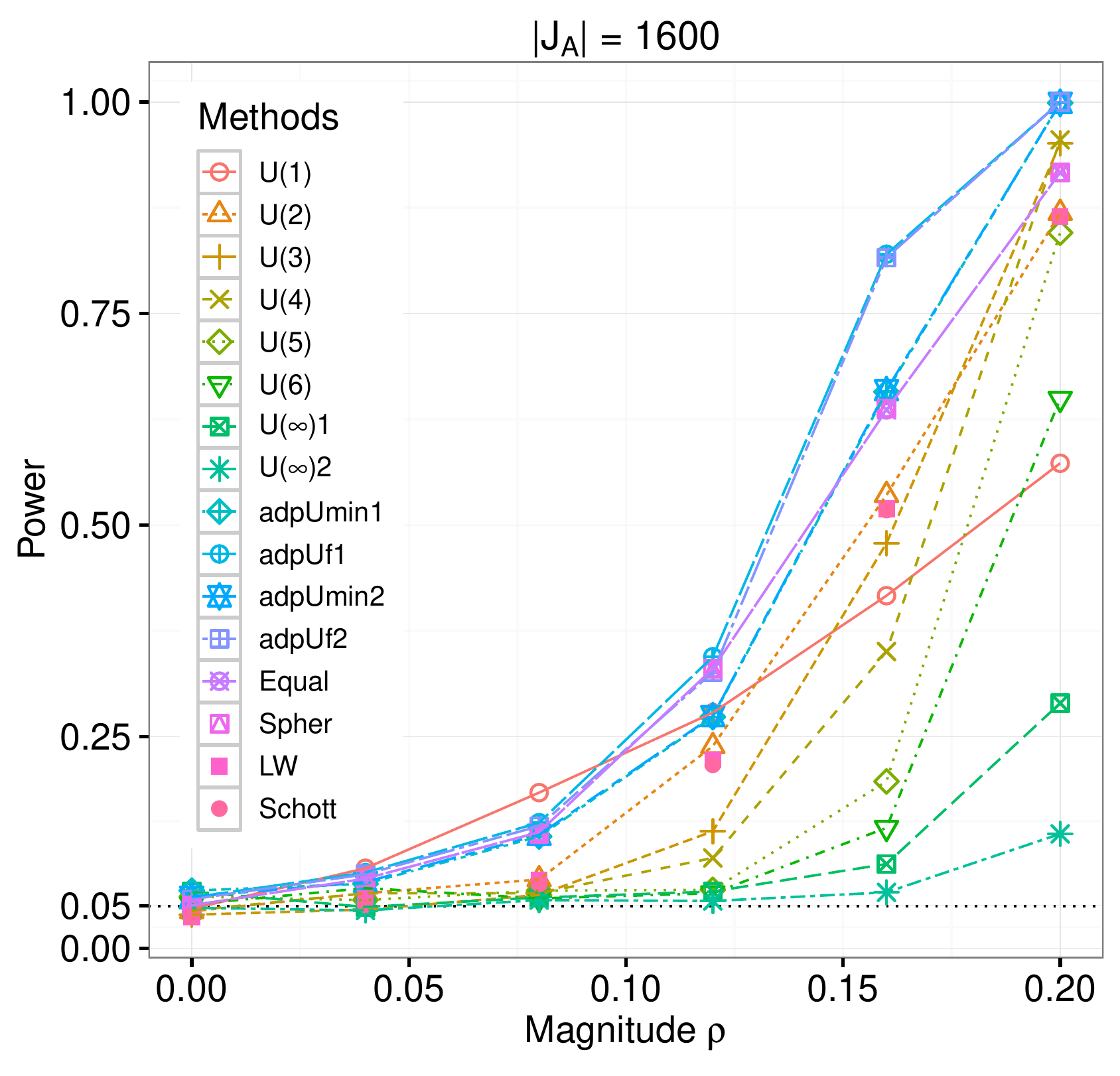}  \quad
       \includegraphics[width=0.48\textwidth,height=0.3\textheight]{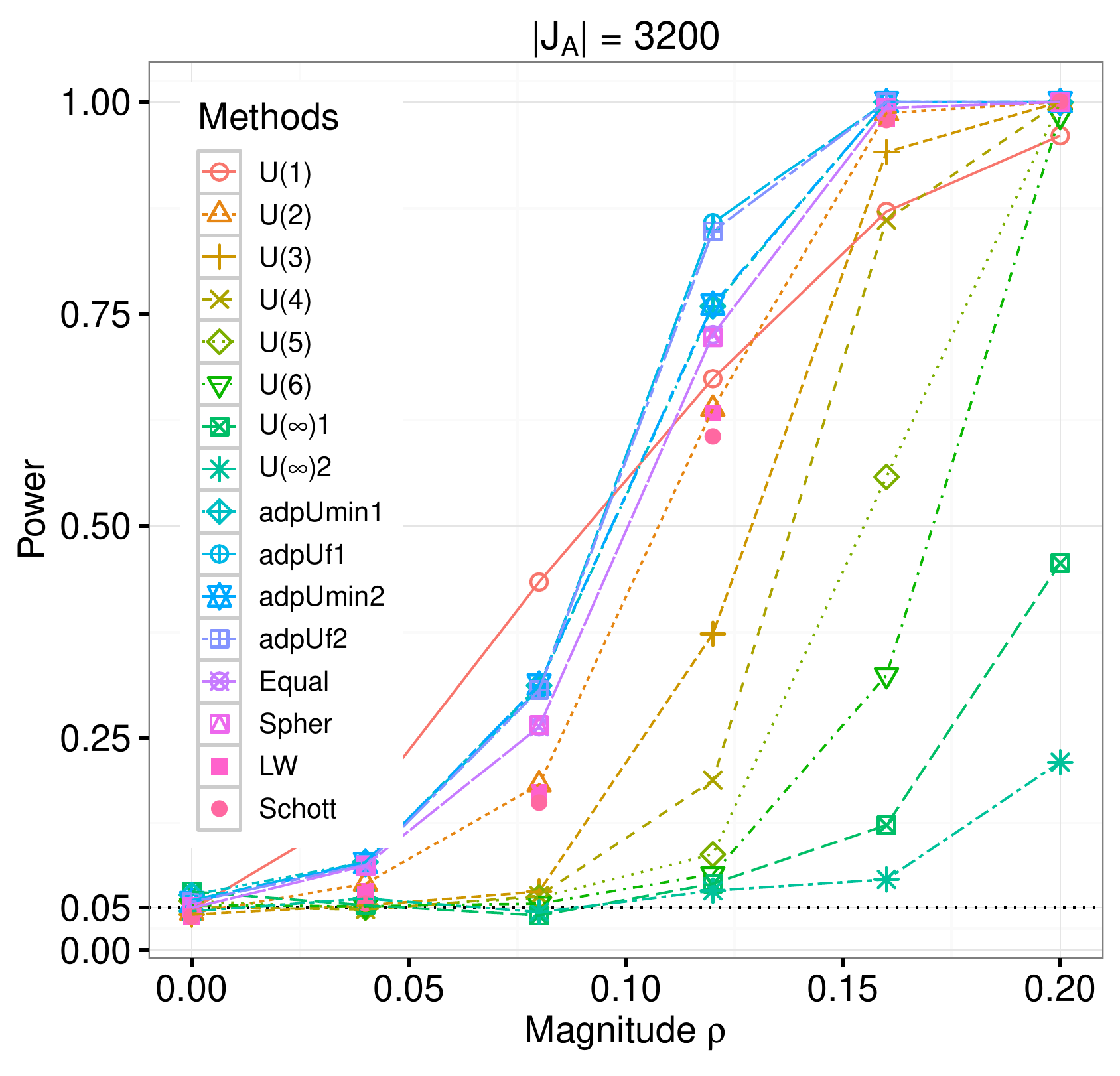}     \caption{Study 3: $n=100, p=1000$.  
    }
    \label{fig:study3p1000}
\end{figure}



\subsubsection{Study 4}

In this section, we provide the simulation results of the fourth setting in Section  \ref{sec:simulation}. In particular, we  generate $n$ i.i.d. $p$-dimensional $\mathbf{x}_i$ for $i=1,\ldots, n$, and $\mathbf{x}_i$ follows multivariate Gaussian distribution with mean zero and covariance $\boldsymbol{\Sigma}_A$. Under this setting, $\boldsymbol{\Sigma}_A$ is symmetric and positive definite and has the diagonal being all one and  $|J_A|$ random  positions taking values uniformly in the range $(0,2\rho)$. Therefore, the nonzero off-diagonal elements in $\boldsymbol{\Sigma}_A$ are different. Figure \ref{fig:study4p1000} below presents the power versus $\rho$ when $n=100$ and $p=1000$. The meanings of the legends are the same as in Tables \ref{table:nullgaussian} and     \ref{table:nullgamma}, and are already explained in Section \ref{sec:npcombsimulsize}. We observe similar patterns to that in the figures in Section \ref{sec:study2onecovsim}.

\begin{figure}[!htbp]
    \centering
    \includegraphics[width=0.48\textwidth,height=0.3\textheight]{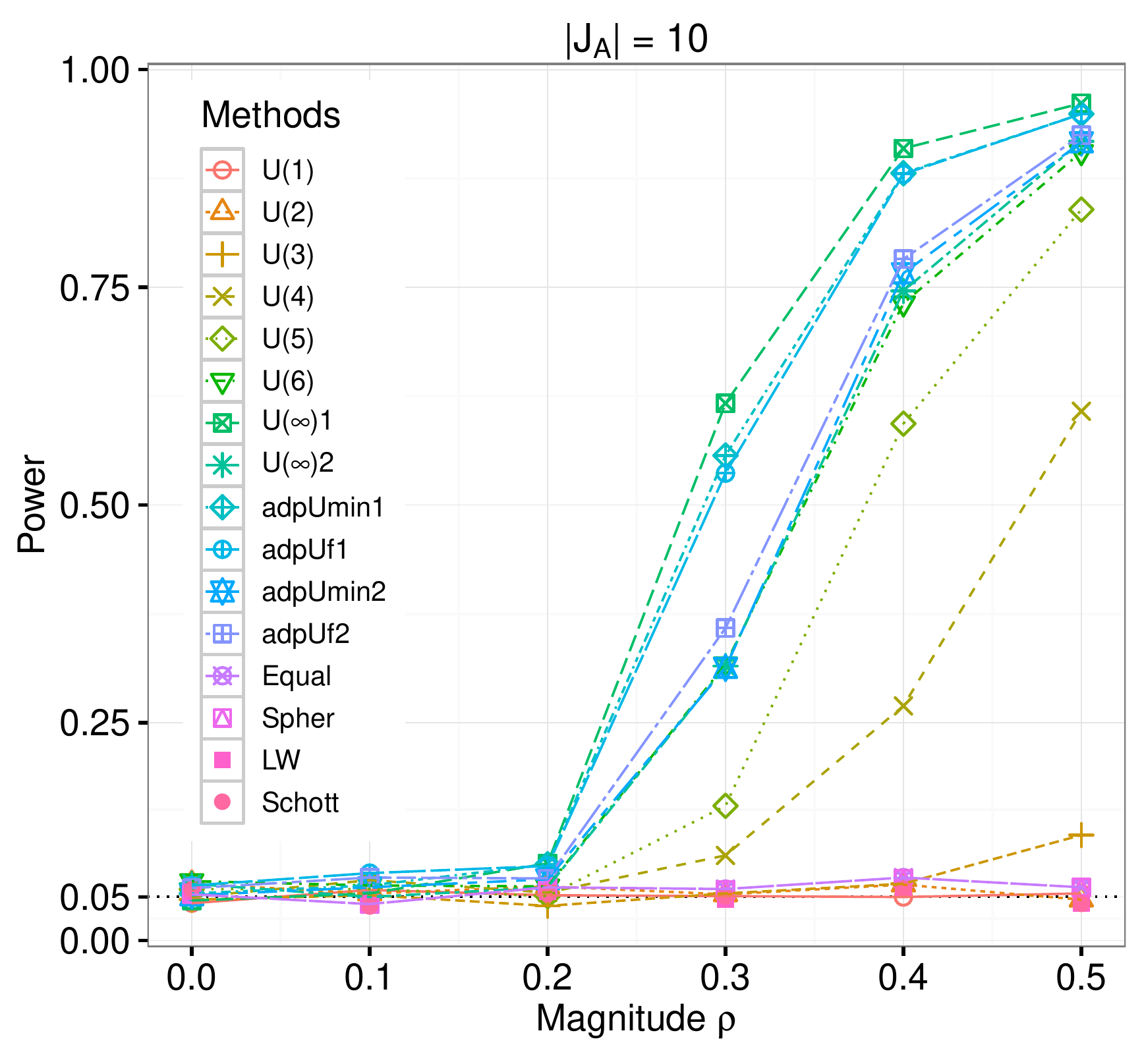} \quad
       \includegraphics[width=0.48\textwidth,height=0.3\textheight]{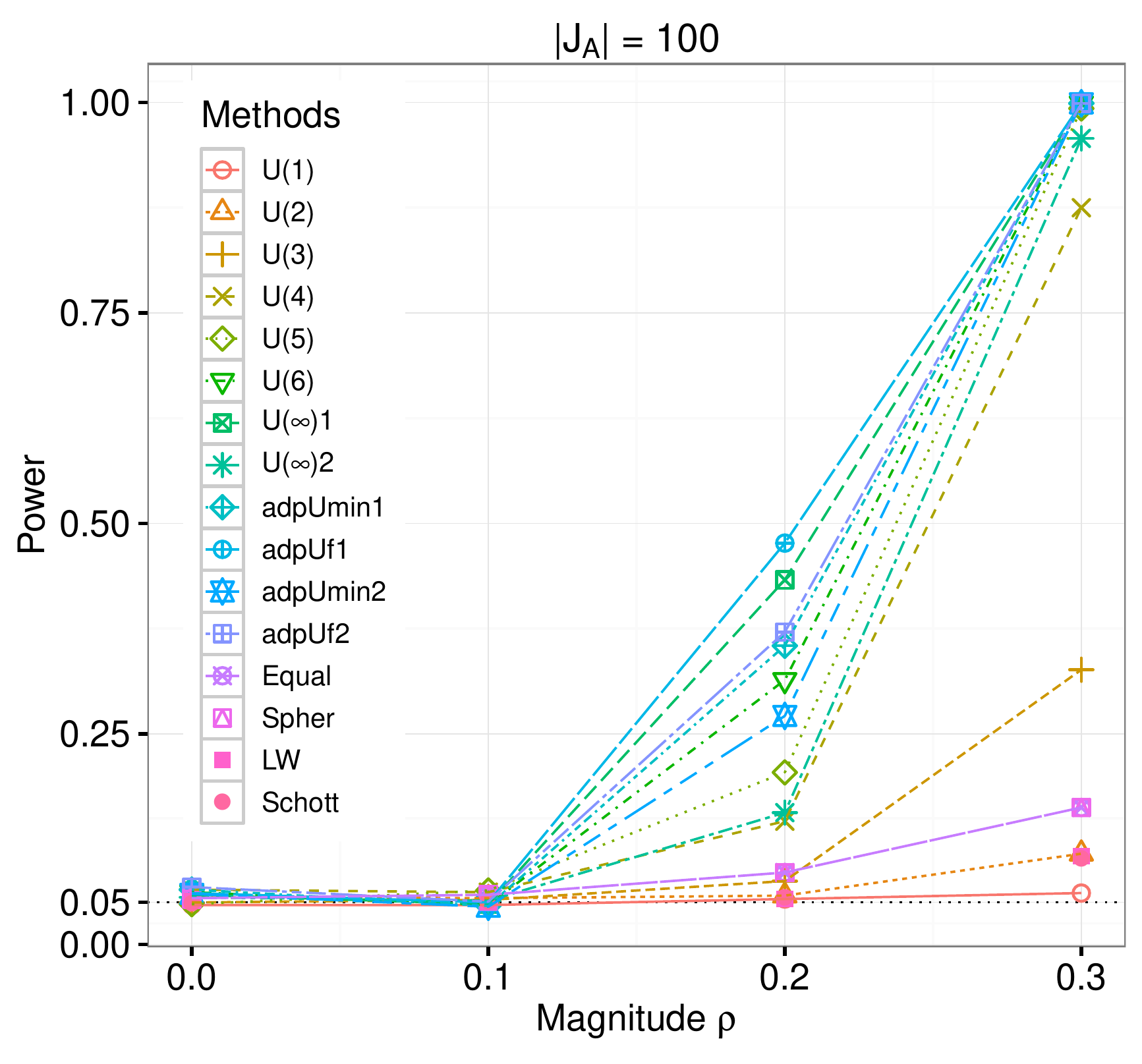}\\
           \includegraphics[width=0.48\textwidth,height=0.3\textheight]{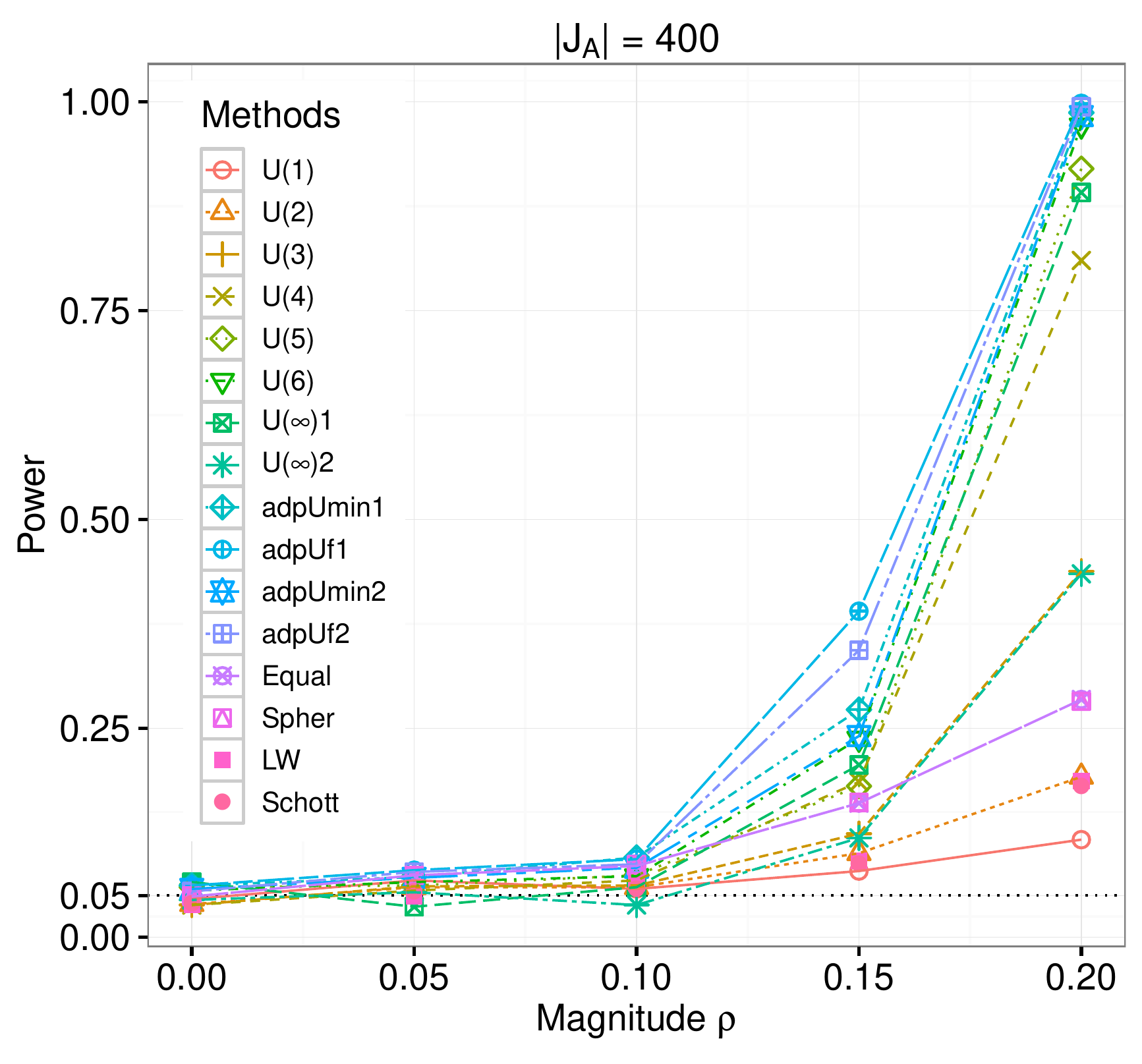} \quad
       \includegraphics[width=0.48\textwidth,height=0.3\textheight]{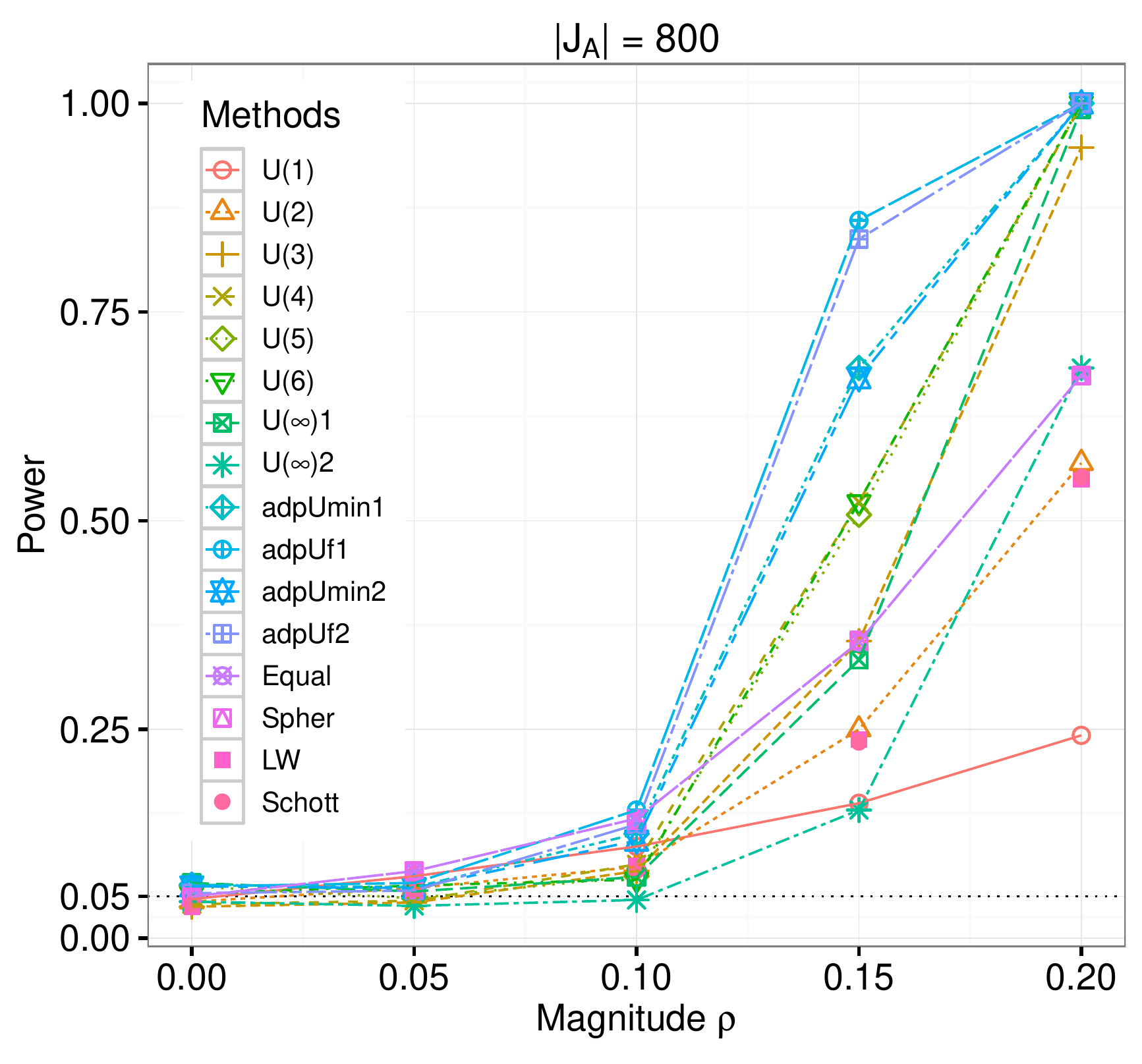}\\
           \includegraphics[width=0.48\textwidth,height=0.3\textheight]{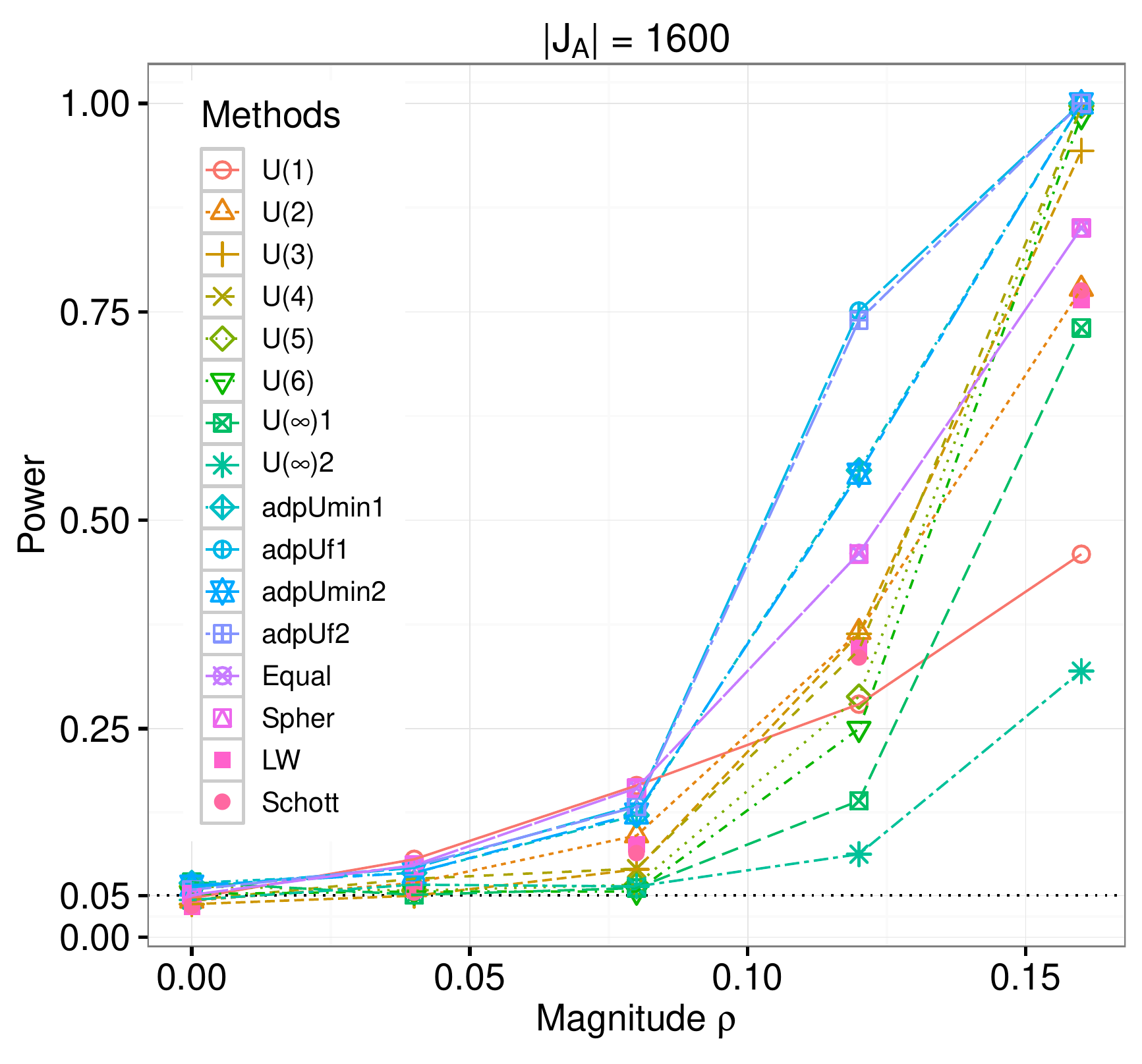}  \quad
       \includegraphics[width=0.48\textwidth,height=0.3\textheight]{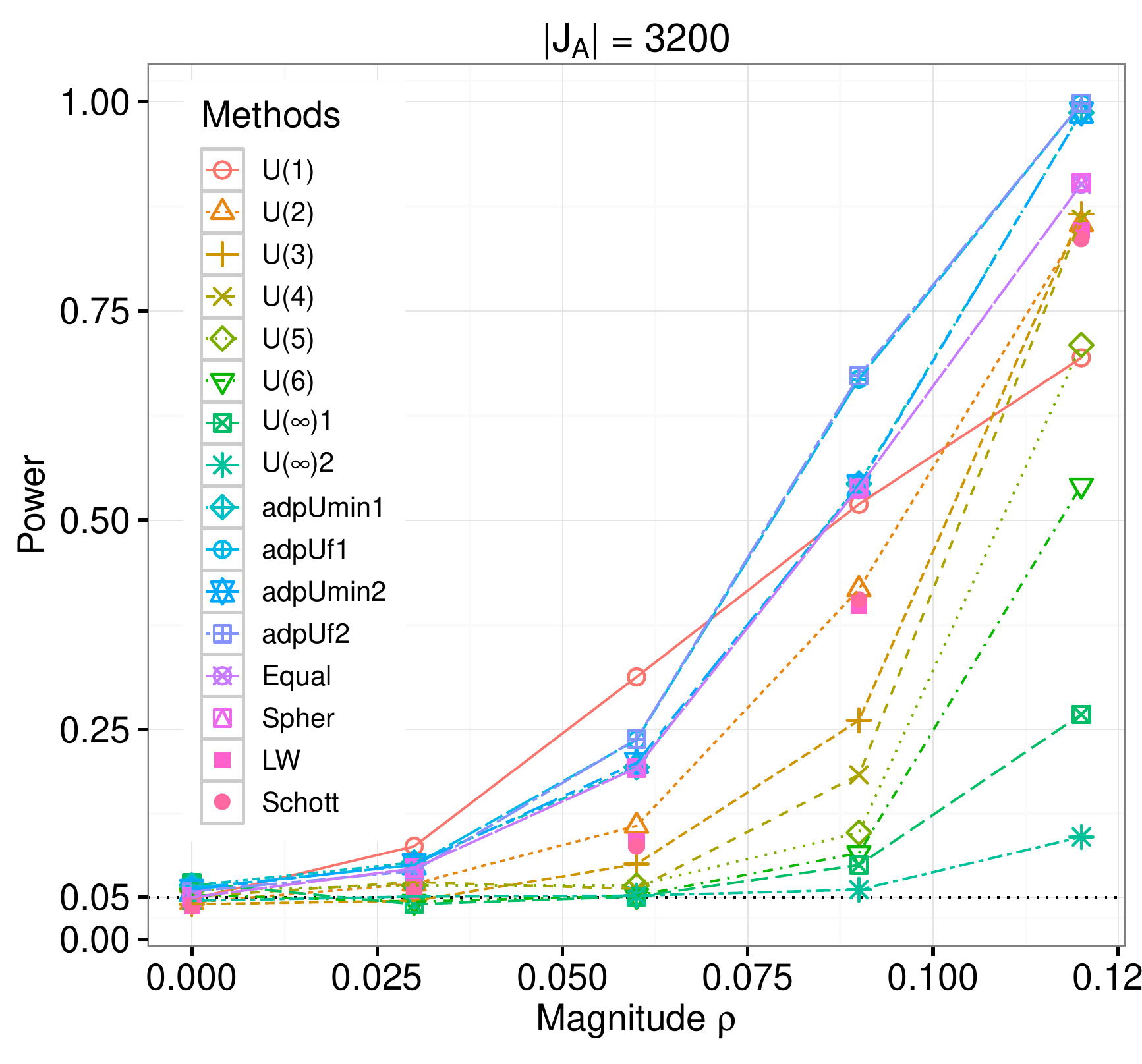}     
       \caption{Study 4: $n=100, p=1000$.  }
    \label{fig:study4p1000}
\end{figure}

\subsubsection{Study 5} \label{sec:simulationstudyii}
%
In this section, we compare our methods with the methods in \citet{chen2011} following their multivariate models. Specifically, for each $i=1,\ldots, n$, $\mathbf{x}_i= \Xi \mathbf{z}_i + \boldsymbol{\mu}$,  
where $\Xi$ is a matrix of dimension $p\times m$ with $m\geq p$. Under null hypothesis, $m=p$, $\Xi=I_p$ $\boldsymbol{\mu}=\mu_0 \mathbf{1}_p$ with $\mu_0=2$; under alternative hypothesis, $m=p+1$, $\boldsymbol{\mu}=2(\sqrt{1-\rho}+\sqrt{2\rho})\mathbf{1}_p$, $\Xi=(\sqrt{1-\rho} I_p, \sqrt{2\rho} \mathbf{1}_p)$, thus $\boldsymbol{\Sigma}=(1-\rho)I_p+2\rho \mathbf{1}_p \mathbf{1}_p^{\intercal}$.
Two settings  are examined: first, $\mathbf{z}_i$'s are i.i.d. multivariate Gaussian random vectors with mean $\mathbf{0}$ and covariance $I_p$; second, $\mathbf{z}_i=(z_{i,1},\ldots, z_{i,m})^{\intercal}$ consists of i.i.d. random variables $z_{i,j}$ which are standardized Gamma$(4,0.5)$ random variables so that $\mathbf{z}_i$ has mean $\mathbf{0}$ and covariance $I_p$. 


To mimic ``large $p$, small $n$" situation, \cite{chen2011} sets  dimension $p=c_1\exp(n^{\eta})+c_2$, where $\eta=0.4$, for $(c_1,c_2)=(1,10)$ and $(c_1,c_2)=(2,0)$ respectively. In particular, we consider $(n,p)\in \{(40,159), (40,331), (80,159),(80,331), \allowbreak (80,642)\}$.  The results are based on 1000 simulations and the nominal significance level of the tests is $5\%$.


In the tables \ref{tab:normalset1n80p331}--\ref{tab:gammaset1n80p642},  results outside and inside parentheses are calculated from parametric-permutation- and asymptotics-based methods, respectively. To be specific, psarametric-permutation-based method means estimating $p$-values or powers by permutation; and asymptotic-based method uses the asymptotic theoretical results and is described in Section \ref{sec:computtest}.   For each $a\in \{1,\ldots, 6, \infty\}$,  the row of ``$\mathcal{U}(a)$" has results using the single test statistic $\mathcal{U}(a)$; and the row of ``adpU" is obtained by the adaptive testing procedure which combines all single candidate U-statistics in the tables using the minimum combination.  In addition,  ``Ident" and ``Spher" rows  denote the identity and  sphericity tests in \cite{chen2011}  separately.

 


In the tables \ref{tab:normalset1n80p331}--\ref{tab:normalset1n80p642}, we find that the empirical sizes of most tests are close to the nominal level, except $\mathcal{U}(\infty)$ due to the slow convergence to extreme value distribution as pointed out in \cite{hall1979rate}.   ``Ident" and ``Spher" tests perform similarly to $\mathcal{U}(2)$ in both settings. This is reasonable because they are all sum-of-squares-type statistics. Moreover, for the $\rho$'s examined, $\mathcal{U}(1)$ has higher power than $\mathcal{U}(2)$, as the constructed alternative  is very dense and only has positive entries. In addition,  ``adpU" achieves  high power for different cases, and its power converges to 1, as one of the test statistics has power converging to 1.   In Tables \ref{tab:gammaset1n80p159} and  \ref{tab:gammaset1n80p642},  data are standardized with sample mean and variance. It can be seen that methods in \cite{chen2011} perform poorly in this case.  Other than this,  the results follow similar patterns to results in other tables.

\begin{table}[!htp]
\centering
\caption{Empirical Type \RNum{1} errors and power (\%) under simulation setting 1. $n=80, p = 331$. }
\label{tab:normalset1n80p331}
\begin{tabular}{rlllll}
  \hline\hline
 $\rho$ & 0 &  0.001& 0.002 & 0.003 & 0.004 \\ 
  \hline
$\mathcal{U}$(1) & 4.4 (4) & 93.4 (90.6) & 100 (99.9) & 100 (100) & 100 (100) \\ 
$\mathcal{U}$(2) & 5 (5.6) & 5.5 (6) & 7.2 (5.9) & 13.1 (10.2) & 19.7 (14.4) \\ 
$\mathcal{U}$(3) & 5.4 (6.1) &  4.5 (4) & 6.3 (5.4) & 6.9 (4.5) & 9 (5.4) \\ 
$\mathcal{U}$(4) & 4.7 (5.1) &  6 (5.4) & 3.7 (4.6) & 4.2 (5.3) & 6 (4.8) \\ 
$\mathcal{U}$(5) & 5.4 (6.3) &  4.9 (4.7) & 5.3 (5.6) & 6 (5.7) & 6.1 (5.1) \\ 
$\mathcal{U}$(6) & 4.6 (4.9) &  5.8 (5.4) & 4.9 (4.5) & 5.2 (4.8) & 4.8 (5) \\ 
$\mathcal{U}(\infty)$ & 4.7 (0.3) &  5 (0.6) & 5.5 (0.7) & 5.1 (0.4) & 5.9 (0.8) \\ 
  aSPU & 5 (5.4) &  81 (81.8) & 99.4 (99.4) & 100 (100) & 100 (100) \\ 
  Ident & 5.5 &  5.7 & 8.2 & 14.4 & 21.8 \\ 
  Spher & 5.6 &  5.7 & 8.1 & 14.2 & 21.4 \\ 
   \hline\hline
\end{tabular}
\end{table}


\begin{table}[!htbp]
\centering
\caption{Empirical Type \RNum{1} errors and power (\%) under simulation setting 2; $n=80, p = 331$.}
\label{tab:gammaset1n80p331}
\begin{tabular}{rlllll}
  \hline\hline
 $\rho$ & 0 & 0.001& 0.002 & 0.003 & 0.004 \\ 
  \hline
$\mathcal{U}$(1) & 5.3 (4.6)  & 56.7 (50.3) & 92.5 (89.3) & 99.3 (99.1) & 100 (99.8) \\ 
  $\mathcal{U}$(2) & 5.4 (5)& 5.5 (5.7) & 6.9 (5.4) & 7.7 (5.8) & 11.4 (7.3) \\ 
  $\mathcal{U}$(3) & 5.6 (5.4)  & 4.5 (3.5) & 5.7 (4) & 5.8 (4.8) & 7.2 (5.1) \\ 
  $\mathcal{U}$(4) & 4.8 (3.9)  & 4.9 (4.1) & 4.9 (5) & 6.5 (6.8) & 4.9 (5.1) \\ 
  $\mathcal{U}$(5) & 6.1 (5.1)  & 5.6 (6.1) & 5.1 (5.2) & 5.5 (5.7) & 5.2 (5.5) \\ 
  $\mathcal{U}$(6) & 6.4 (5.6)  & 5.4 (4.1) & 5.1 (5.3) & 5.1 (5.4) & 5.8 (5.3) \\ 
  $\mathcal{U}(\infty)$ & 5.5 (3)  & 5.3 (2.5) & 6 (2.8) & 5.5 (2.8) & 6.8 (3.1) \\ 
  adpU & 6.4 (6.5) & 35 (36.3) & 78.7 (79.2) & 96.1 (96.1) & 99.5 (99.6) \\ 
  Ident & 6.7  & 6.5 & 7.4 & 9.2 & 13.5 \\ 
  Spher & 6.2 & 6.2 & 7 & 9.1 & 12.9 \\ 
   \hline\hline
\end{tabular}
\end{table}

\begin{table}[!htp]
\centering
\caption{Empirical Type \RNum{1} errors and power (\%) under simulation setting 1; $n=40, p = 159$.}
\label{tab:normalset1n40p159}
\begin{tabular}{rllllll}
  \hline\hline
 $\rho$ & 0 & 0.0005 & 0.001 & 0.0015 & 0.002 & 0.0025 \\ 
  \hline
$\mathcal{U}$(1) & 5.8 (4.6) & 16.6 (13.6) & 36.5 (32.3) & 57.4 (51.3) & 69.2 (65.1) & 83.3 (80) \\ 
  $\mathcal{U}$(2) & 5.2 (4.9) & 4.6 (3.1) & 4.6 (5.6) & 5.3 (4.5) & 5.5 (4.8) & 5.9 (4.8) \\ 
  $\mathcal{U}$(3) & 4.9 (4.8) & 5.8 (5.4) & 5.6 (5.6) & 5.6 (4.9) & 4.6 (4.7) & 5.6 (5) \\ 
  $\mathcal{U}$(4) & 4.6 (5.7) & 4.2 (4.1) & 5.6 (4.6) & 4.7 (4.6) & 4.5 (5.1) & 5.3 (4.9) \\ 
  $\mathcal{U}$(5) & 5.5 (5.6) & 5.3 (6.2) & 5.7 (4.9) & 3.1 (3.1) & 4.7 (4.4) & 5.5 (5.4) \\ 
  $\mathcal{U}$(6) & 4.4 (4.3) & 4.8 (4.6) & 4.4 (4.7) & 4.3 (4.3) & 4.8 (4.6) & 5 (4.2) \\ 
  $\mathcal{U}(\infty)$ & 5.1 (0.1) & 5.1 (0.1) & 4.2 (0) & 4.6 (0.1) & 4.6 (0) & 5.5 (0.1) \\ 
  adpU & 5.7 (5.8) & 8.9 (10.6) & 18.5 (21.1) & 31.5 (34.2) & 47.4 (50.8) & 63.2 (66.2) \\ 
  Ident & 5.8 & 5.3 & 5.9 & 6.8 & 6.8 & 7.1 \\ 
  Spher & 5.8 & 5.1 & 5.7 & 6.5 & 6.5 & 7.2 \\ 
   \hline\hline
\end{tabular}
\end{table}

\begin{table}[!htp]
\centering
\caption{Empirical Type \RNum{1} errors and power (\%) under simulation setting 1; $n=40, p = 331$. }
\label{tab:normalset1n40p331}
\begin{tabular}{rllllll}
  \hline\hline
 $\rho$ & 0 & 0.0025 & 0.005 & 0.01 & 0.015 & 0.02 \\ 
  \hline
$\mathcal{U}$(1) & 5.9 (5.4) & 99.4 (99.3) & 100 (100) & 100 (100) & 100 (100) & 100 (100) \\ 
  $\mathcal{U}$(2) & 5.1 (4.4) & 7 (6.3) & 15.5 (10.7) & 65.8 (60) & 95.1 (93.1) & 99.3 (98.7) \\ 
  $\mathcal{U}$(3) & 5.4 (5.5) & 7.6 (4.6) & 13 (7.5) & 26.3 (19.7) & 53.9 (44.1) & 76.9 (68.9) \\ 
  $\mathcal{U}$(4) & 4.8 (5.1) & 4.9 (5.4) & 6.8 (5.6) & 6.3 (6.6) & 11.4 (7.7) & 14.4 (11.7) \\ 
  $\mathcal{U}$(5) & 5.9 (4.8) & 5.5 (4.9) & 7 (6.6) & 5.6 (4.9) & 8.6 (7.3) & 8.5 (8.2) \\ 
  $\mathcal{U}$(6) & 4.1 (4.9) & 3.4 (4.5) & 6.8 (4.6) & 4.8 (6.5) & 5.5 (6.6) & 8 (8.6) \\ 
  $\mathcal{U}(\infty)$ & 4.2 (0) & 4.1 (0) & 6.1 (0) & 4.9 (0) & 6.6 (0) & 7.3 (0.1) \\ 
  adpU & 5.2 (5.8) & 97.5 (98.5) & 100 (100) & 100 (100) & 100 (100) & 100 (100) \\ 
  Ident & 6.2 & 8.3 & 19.2 & 68 & 95.5 & 99.3 \\ 
  Spher & 6.3 & 8.2 & 18.6 & 67.6 & 95.4 & 99.3 \\ 
   \hline\hline
\end{tabular}
\end{table}

\begin{table}[!htp]
\centering
\caption{Empirical Type \RNum{1} errors and power (\%) under simulation setting 1; $n=80, p = 159$.}
\label{tab:normalset1n80p159}
\begin{tabular}{rllllll}
  \hline\hline
 $\rho$ & 0 & 0.0025 & 0.005 & 0.01 & 0.015 & 0.02 \\ 
  \hline
$\mathcal{U}$(1) & 5.7 (4.7) & 98.1 (97) & 100 (100) & 100 (100) & 100 (100) & 100 (100) \\ 
  $\mathcal{U}$(2) & 6.2 (5.1) & 6.8 (5.5) & 16.5 (11.4) & 68.4 (60.6) & 96.7 (94.7) & 100 (99.9) \\ 
  $\mathcal{U}$(3) & 6 (4.7) & 6.2 (5.5) & 7.4 (5.9) & 15.2 (9.2) & 34.8 (26.2) & 69.2 (61.4) \\ 
  $\mathcal{U}$(4) & 5.4 (5.6) & 4 (3.8) & 4.7 (4.2) & 7.6 (7.1) & 10.6 (9) & 18.2 (15.7) \\ 
  $\mathcal{U}$(5) & 4.5 (4.9) & 4.6 (4.2) & 4.8 (4.5) & 5.3 (5.3) & 9.6 (7.6) & 13.1 (13) \\ 
  $\mathcal{U}$(6) & 5.6 (5.3) & 3.9 (4.7) & 4 (3.3) & 5.3 (4.9) & 8.7 (8) & 12 (12.4) \\ 
  $\mathcal{U}(\infty)$ & 4.5 (0.8) & 6.1 (1.1) & 4.9 (1.4) & 5.4 (1.7) & 8 (1.5) & 10.7 (3.3) \\ 
  adpU & 5.7 (7) & 91.8 (92.6) & 99.8 (99.8) & 100 (100) & 100 (100) & 100 (100) \\ 
  Ident & 6.7 & 7.8 & 18.5 & 71.1 & 97.3 & 100 \\ 
  Spher & 6.7 & 7.2 & 18 & 69.6 & 97 & 100 \\
   \hline\hline
\end{tabular}
\end{table}

\begin{table}[!htp]
\centering
\caption{Empirical Type \RNum{1} errors and power (\%) under simulation setting 1; $n=80, p = 642$.}
\label{tab:normalset1n80p642}
\begin{tabular}{rllllll}
  \hline\hline
 $\rho$ & 0 & 0.0025 & 0.005 & 0.01 & 0.015 & 0.02 \\ 
  \hline
$\mathcal{U}$(1) & 5.8 (4.8) & 100 (100) & 100 (100) & 100 (100) & 100 (100) & 100 (100) \\ 
  $\mathcal{U}$(2) & 6.4 (6.2) & 17.9 (12.7) & 71.2 (63.4) & 99.8 (99.8) & 100 (100) & 100 (100) \\ 
  $\mathcal{U}$(3) & 5.2 (5.6) & 6.2 (3.6) & 19.3 (13.3) & 68.4 (57.3) & 96.4 (94) & 99.8 (99.6) \\ 
  $\mathcal{U}$(4) & 5.2 (5.2) & 6.2 (6.4) & 5.2 (5.2) & 8.5 (6.4) & 25 (18.3) & 57.9 (51.7) \\ 
  $\mathcal{U}$(5) & 6.4 (4.6) & 5 (5.2) & 6.4 (5.4) & 7.8 (7.2) & 11.7 (9.9) & 21.1 (16.9) \\ 
  $\mathcal{U}$(6) & 4 (4.2) & 5.8 (6.4) & 6 (6) & 4.2 (5.2) & 9.3 (10.3) & 13.1 (15.3) \\ 
  $\mathcal{U}(\infty)$ & 4.4 (0.6) & 5 (0.2) & 5.6 (0.4) & 7 (0.8) & 9.3 (0.8) & 15.3 (0.6) \\ 
  adpU & 6 (4.2) & 100 (100) & 100 (100) & 100 (100) & 100 (100) & 100 (100) \\ 
  Ident & 6.8 & 18.9 & 72.6 & 100 & 100 & 100 \\ 
  Spher & 6.6 & 18.7 & 72.6 & 100 & 100 & 100 \\ 
   \hline\hline
\end{tabular}
\end{table}

\begin{table}[!htp]
\centering
\caption{Empirical Type \RNum{1} errors and power (\%) under simulation setting 2; $n=80, p = 159$.}
\label{tab:gammaset1n80p159}
\begin{tabular}{rllllll}
  \hline\hline
 $\rho$ & 0 & 0.0005 & 0.001& 0.002 & 0.003 & 0.004 \\ 
  \hline
$\mathcal{U}$(1) & 4.9 (4.2) & 26.1 (20.4) & 57.1 (49.7) & 95.2 (93.1) & 99.9 (99.8) & 100 (99.9) \\ 
  $\mathcal{U}$(2) & 4.9 (4.4) & 3.9 (5.3) & 5.9 (5.2) & 6.7 (4.8) & 8.3 (5.6) & 12.2 (7.7) \\ 
  $\mathcal{U}$(3) & 5.4 (5.2) & 4.7 (5.3) & 4.3 (4.1) & 6 (4) & 5.9 (5.1) & 7 (5) \\ 
  $\mathcal{U}$(4) & 5.4 (4.9) & 5.5 (5.2) & 4.8 (4.8) & 5.9 (6.3) & 6.7 (7.2) & 4.6 (4.6) \\ 
  $\mathcal{U}$(5) & 7.3 (6.2) & 5.4 (5.6) & 5.8 (6.5) & 5.3 (6.3) & 5.8 (5.5) & 5.6 (5.6) \\ 
  $\mathcal{U}$(6) & 6.5 (5.6) & 4.9 (5) & 5.5 (5.3) & 4.9 (5.2) & 5.5 (5.4) & 4.2 (4.7) \\ 
  $\mathcal{U}(\infty)$ & 5.9 (3) & 5.7 (2.1) & 5.8 (2.5) & 5.7 (2.6) & 5.5 (2.9) & 6.7 (3.3) \\ 
  adpU & 5.7 (5) & 12.1 (13.1) & 34.8 (34.6) & 81.9 (82.6) & 98.1 (98.1) & 99.9 (99.8) \\ 
  Ident & 0.2 & 0.1 & 0.1 & 0.2 & 0.1 & 0.1 \\ 
  Spher & 0.2 & 0.1 & 0.1 & 0.2 & 0 & 0.1 \\ 
   \hline\hline
\end{tabular}
\end{table}

\begin{table}[!htp]
\centering
\caption{Empirical Type \RNum{1} errors and power (\%) under simulation setting 2; $n=80, p = 642$.}
\label{tab:gammaset1n80p642}
\begin{tabular}{rllllll}
  \hline\hline
 $\rho$ & 0 & 0.0005 & 0.001& 0.002 & 0.003 & 0.004 \\ 
  \hline
$\mathcal{U}$(1) & 2.8 (2.2) & 94.2 (93) & 100 (100) & 100 (100) & 100 (100) & 100 (100) \\ 
  $\mathcal{U}$(2) & 5.8 (4.2) & 4.2 (4.8) & 6 (5.6) & 11.9 (7.2) & 22.3 (14.5) & 45.9 (36.2) \\ 
  $\mathcal{U}$(3) & 3.6 (3.8) & 5.4 (5.2) & 7.2 (5) & 6 (3.6) & 11.9 (7.6) & 15.1 (9.3) \\ 
  $\mathcal{U}$(4) & 4.4 (4.4) & 4.6 (4.4) & 6.4 (6.2) & 4.8 (3.8) & 5.4 (5.2) & 7 (6.2) \\ 
  $\mathcal{U}$(5) & 7 (5.6) & 6 (5) & 6.2 (5.4) & 7 (6.2) & 6.6 (5.4) & 7.4 (5.6) \\ 
  $\mathcal{U}$(6) & 7 (5.4) & 5 (4.6) & 4.6 (5.6) & 6.8 (7.2) & 5.4 (4.6) & 5.6 (5.8) \\ 
  $\mathcal{U}(\infty)$ & 4.8 (2.2) & 6.2 (2.4) & 4.8 (0.8) & 6.2 (3) & 6.4 (2.6) & 5.2 (1.6) \\ 
  adpU & 5 (4) & 84.5 (85.9) & 100 (100) & 100 (100) & 100 (100) & 100 (100) \\ 
  Ident & 0 & 0.4 & 0.2 & 0.4 & 2.4 & 8.3 \\ 
  Spher & 0 & 0.4 & 0.2 & 0.4 & 2.4 & 7.8 \\ 
   \hline\hline
\end{tabular}
\end{table}

\newpage

\subsection{Simulations on Other Testing Examples} \label{sec:simulothertesting}

In this section, we provide the simulation results on other testing examples discussed in Section \ref{sec:extension}. We present simulations on generalized linear model in Section \ref{sec:glmsim}. In addition, we provide  simulations on two-sample covariance testing to examine the empirical type I error and power in Sections \ref{sec:twosamplecovsim1} and \ref{sec:twosamplecovsim2},  respectively.

\subsubsection{Study 6: GLM} \label{sec:glmsim}
In this study, we conduct simulations for generalized linear model considering the following model 
\begin{align}
	y_i = \mathbf{z}_i^{\intercal}\boldsymbol{\alpha} +\mathbf{x}_i^{\intercal}\boldsymbol{\beta} + \epsilon_i, \label{eq:simuglmmodel}
\end{align} for $i=1,\ldots, n$.
We generate i.i.d. $\mathbf{x}_i$  from the multivariate normal distribution $\mathcal{N}(0,\Sigma)$. We show the results with an equal variance and a first-order autoregressive correlation matrix case, that is,  $\Sigma = (0.4^{|i-j|})$.  We further generate $\mathbf{z}_i$  of two covariates with entries i.i.d. from standard normal distribution $\mathcal{N}(0,1)$, and  $\epsilon_i$ are the random errors following i.i.d. normal distribution $\mathcal{N}(0, 0.5)$. In \eqref{eq:simuglmmodel}, we take $\boldsymbol{\alpha}=(0.3,0.3)^{\intercal}$,  $\boldsymbol{\beta} = \mathbf{0}$ or $\neq \mathbf{0}$ corresponded to the null hypothesis $H_0$ and  the alternative hypothesis $H_A$, respectively. Under $H_A$, $\lfloor{p s} \rfloor$ elements in $\boldsymbol{\beta}$ are set to be non-zero, where $s \in [0,1]$ controls signal sparsity. We vary $s$ to mimic varying sparsity situations, from sparse to dense signals with $s \in \{0.001, 0.1, 0.3, 0.7, 0.9 \}$. 
The positions of non-zero elements in $\boldsymbol{\beta}$ are assumed to be uniformly distributed in $\{1,2,\dots,p\}$, and their values are constant $c$, where $c$ is the effect of signals that vary in the simulations.  The results are based on  1000 simulations with $5\%$ nominal significance level, $n=500$ and  $p=1000$.  We summarized the results in Figure \ref{fig:glmsimulation}. It shows similar patterns as in Study I.
\begin{figure}[!htbp]
    \centering
    \includegraphics[width=0.46\textwidth,height=0.3\textheight]{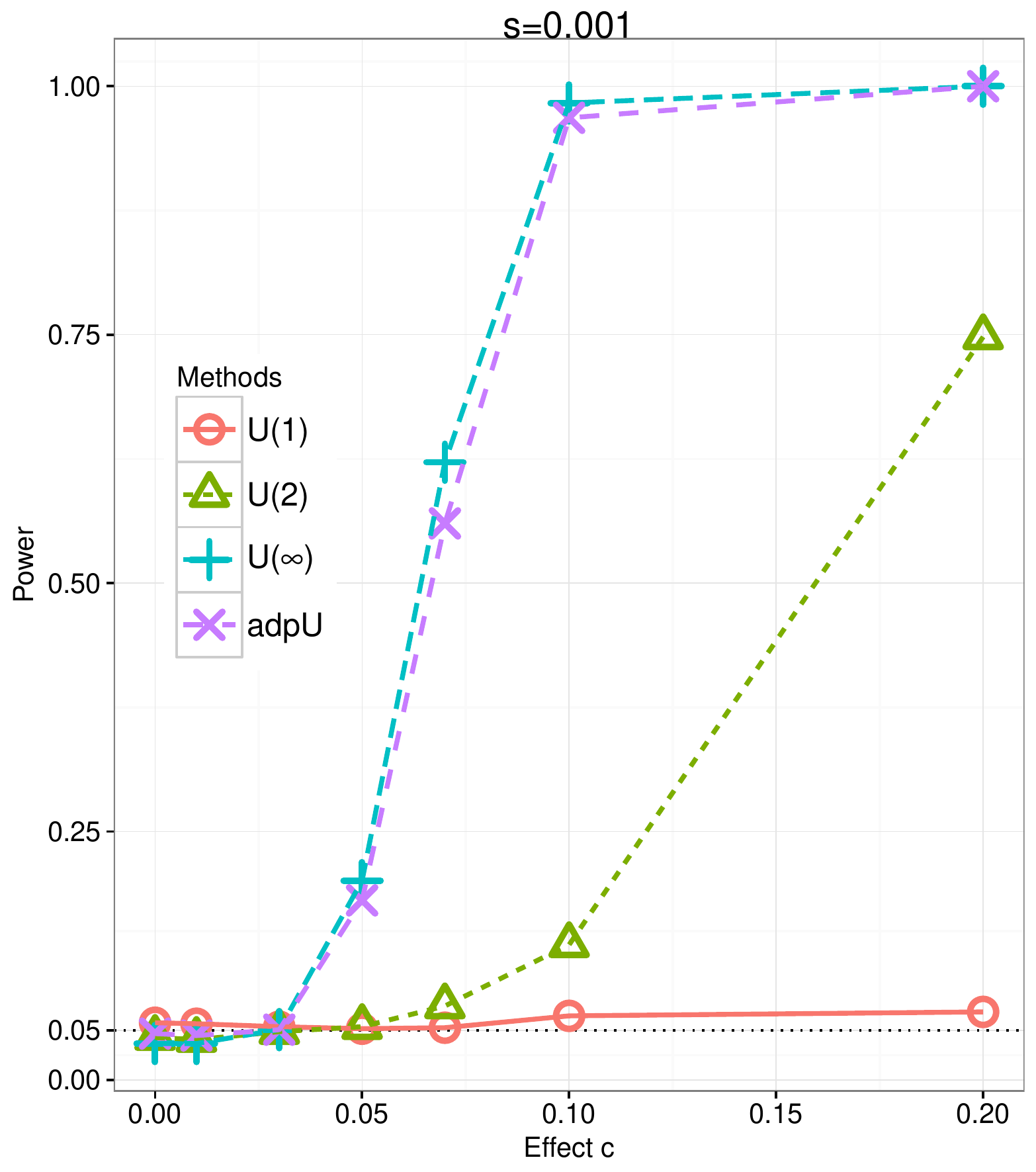} \quad \quad
       \includegraphics[width=0.46\textwidth,height=0.3\textheight]{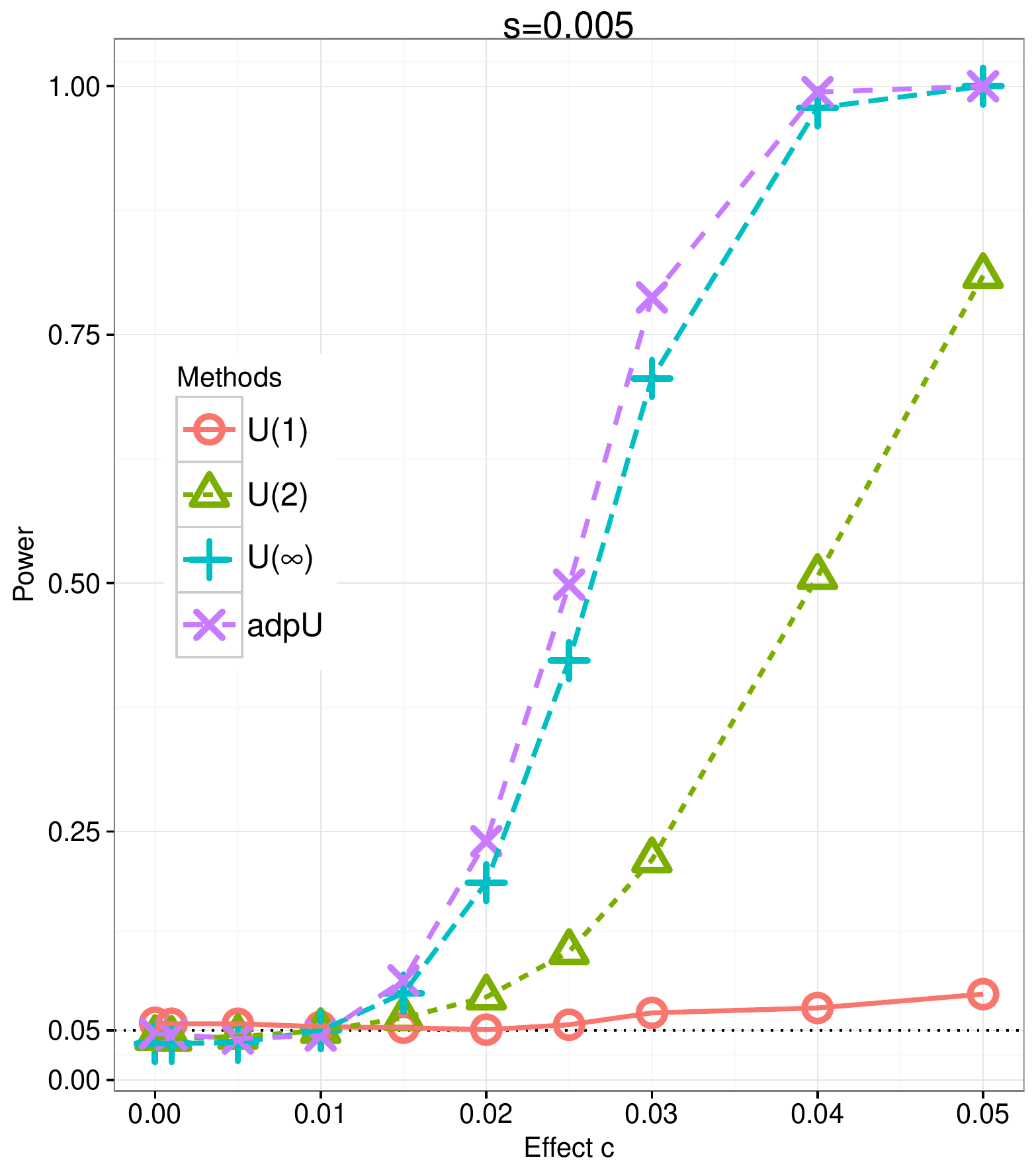}\\
           \includegraphics[width=0.46\textwidth,height=0.3\textheight]{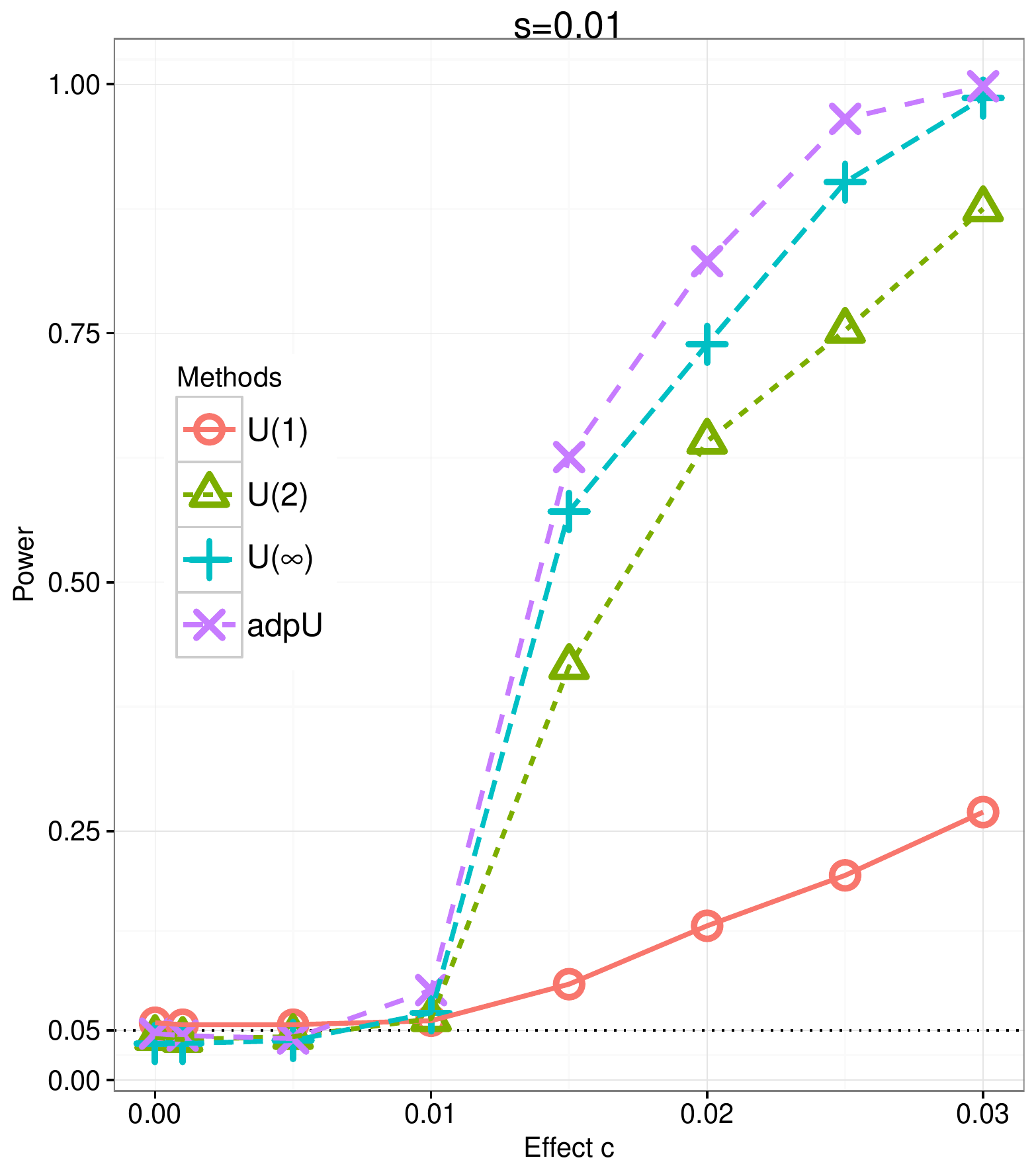} \quad \quad
       \includegraphics[width=0.46\textwidth,height=0.3\textheight]{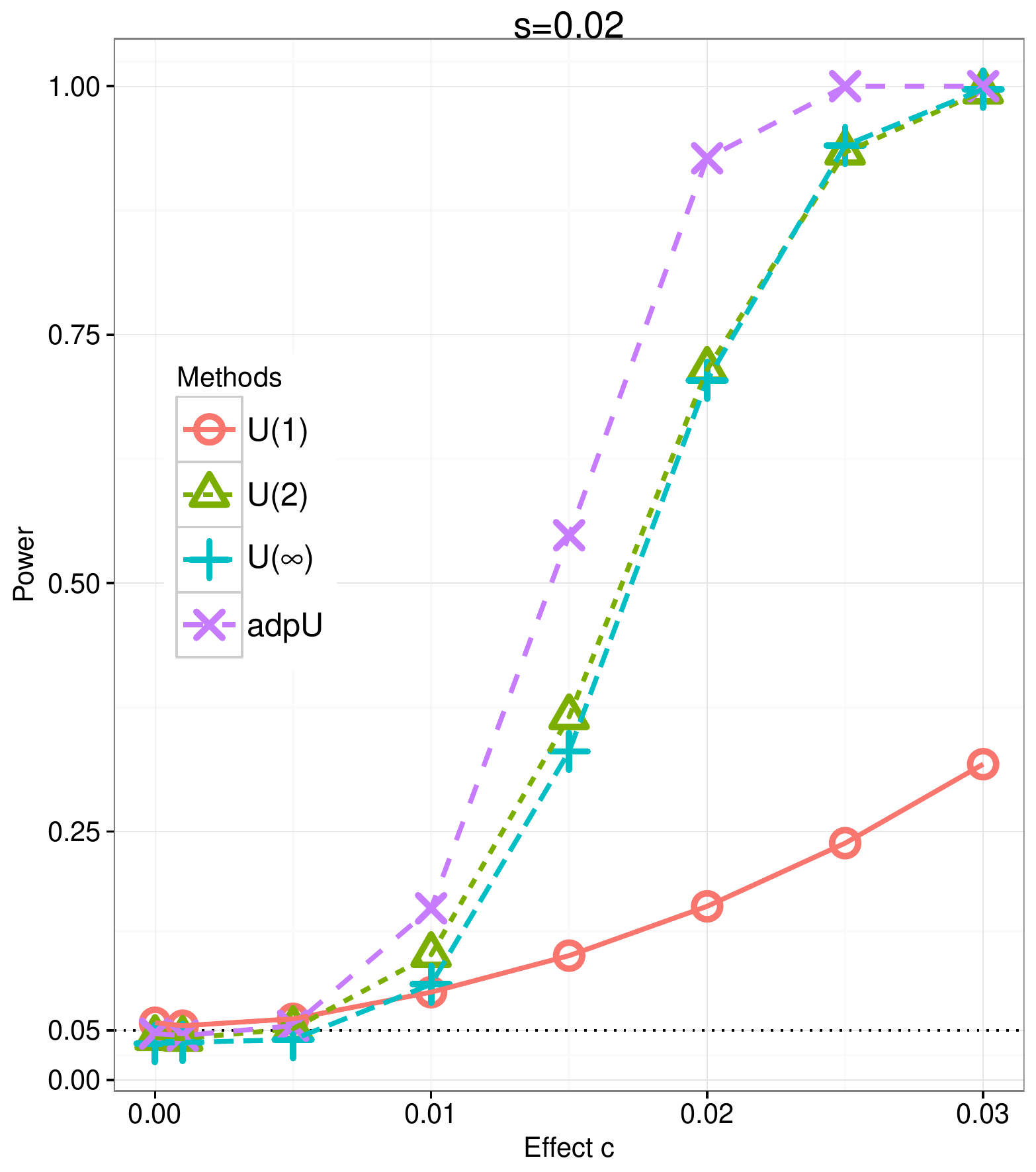}\\
           \includegraphics[width=0.46\textwidth,height=0.3\textheight]{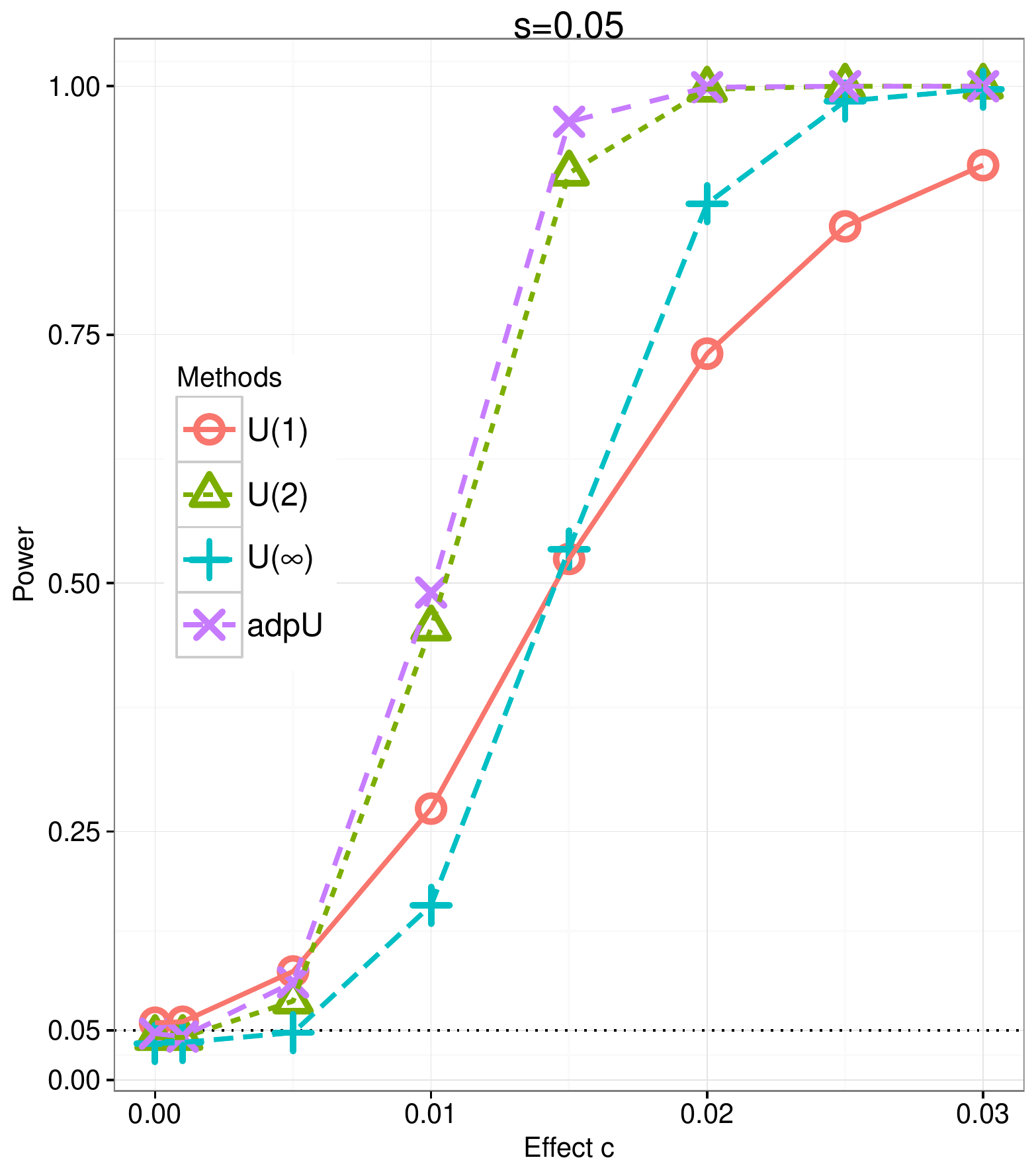}  \quad \quad
       \includegraphics[width=0.46\textwidth,height=0.3\textheight]{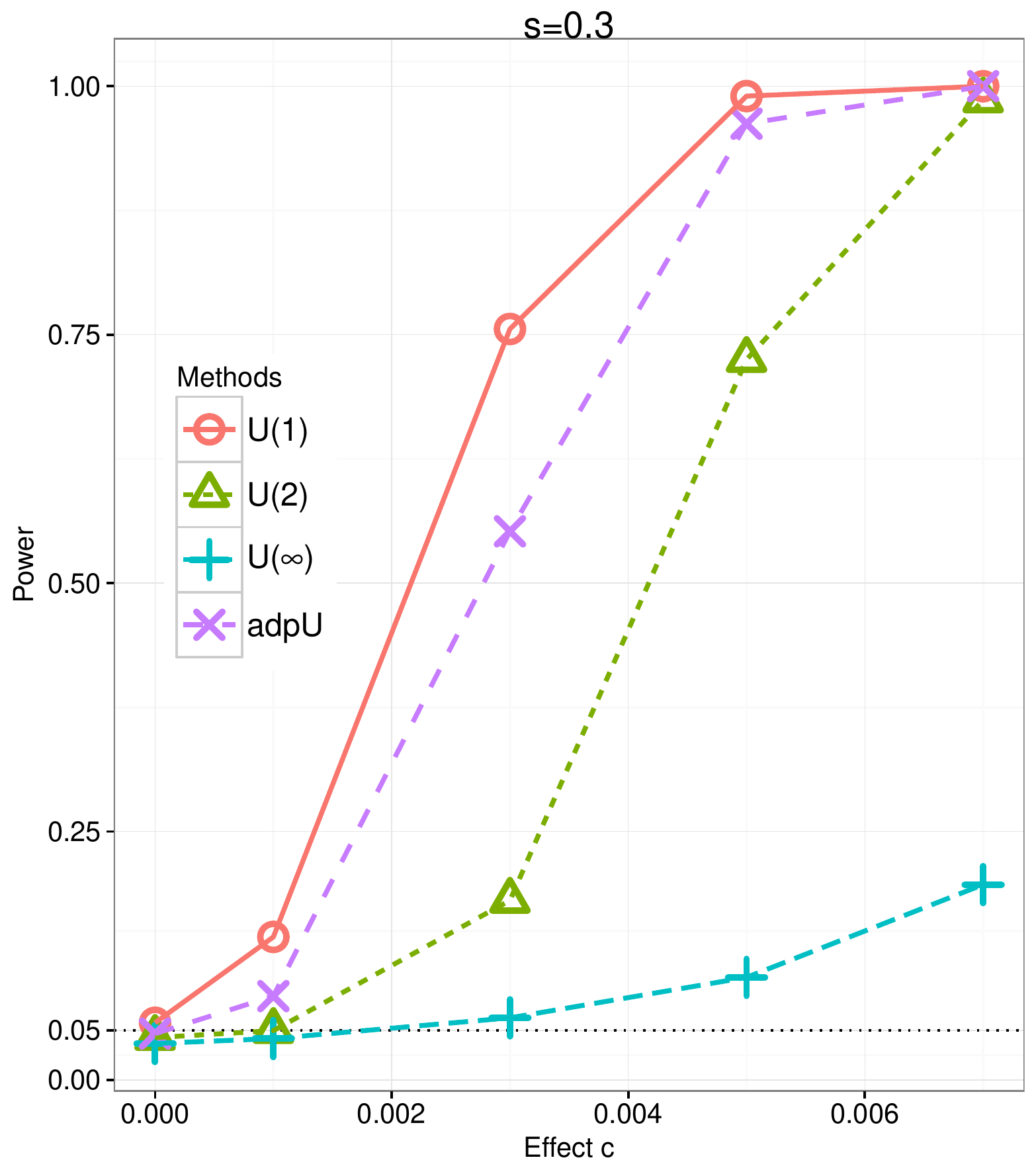}
    \caption{Power comparison under generalized linear model simulation setting. 
    }
    \label{fig:glmsimulation}
\end{figure}

\newpage

\subsubsection{Study 7: Two-sample covariance testing under $H_0$}\label{sec:twosamplecovsim1}

In this section, we examine the empirical  Type \RNum{1} errors  of the proposed the adaptive testing procedure and compare it with the other methods. 

We follow the simulation settings in \citet{weightstatsyang2017}. In particular, let $A(s)$ be the $s\times s$ covariance matrix of MA(1) model with the parameter $\theta_1=0.4$. 
In addition, $B=0.7 I_{p-s}$ is a $(p-s)\times (p-s)$ scaled identity matrix. We then define the matrix $Q(s)=\mathrm{BlkDiag}(A(s),B),$
where ``BlkDiag" indicates a block diagonal matrix. We take $s=p^{1/2}$ and $n=100$, and consider $\boldsymbol{\Sigma}_x=\boldsymbol{\Sigma}_y=Q(s)$. The results are presented in Table \ref{tb:twosamsize1}.

In Table \ref{tb:twosamsize1}, we provide the simulation results of the single U-statistics $\mathcal{U}(a)$ with $a\in \{1,\ldots,6\}$. In addition, we provide the simulation results of $\mathcal{U}(\infty)$ using permutation  and the asymptotic distribution in \citet{cai2013two}, which are denoted as ``$\mathcal{U}(\infty)$ permutation" and   ``$\mathcal{U}(\infty)$ Tony" respectively. Given the results of $\mathcal{U}(1),\ldots, \mathcal{U}(6)$ and ``$\mathcal{U}(\infty)$ (permutation)",  ``adpUmin 1" and  ``adpUf 1" represent the results of the adaptive testing procedure   using  minimum combination and Fisher's method respectively. Similarly, given the results of $\mathcal{U}(1),\ldots, \mathcal{U}(6)$ and ``$\mathcal{U}(\infty)$ (Tony)", ``adpUmin 2" and  ``adpUf 2" represent the results of the adaptive testing procedure   using  minimum combination and Fisher's method respectively.   Moreover, ``Schott", ``Sriva"  and ``Chen"   represent the methods in \citet{schott2007test,srivastava2010testing} and \citet{lijun2012},  respectively.  In addition, we denote the tests without and with Micro term in \citet{weightstatsyang2017} as ``Pan1" and ``Pan2" respectively. The tests in \cite{weightstatsyang2017} are time-consuming. Therefore we only provide the simulation results at $p=50$, which takes about 100 times the time of the proposed adaptive testing procedure.


Based on our simulation results, we find that the empirical  Type \RNum{1} errors  of the single U-statistics  are close the nominal levels, which verifies the theoretical results of Theorem \ref{thm:twosamnull}. Moreover, comparing ``$\mathcal{U}(\infty)$ (permutation)"  and ``$\mathcal{U}(\infty)$ (Tony)",  we find that using the asymptotic distribution in \citet{cai2013two} gives conservative  Type \RNum{1} errors  that are smaller than the nominal levels. In addition, by examining the results of minimum combination and Fisher's method, we find that both of the two methods give empirical  Type \RNum{1} errors  that are close to the nominal level,  while the Fisher's method may have slight size inflation compared to the minimum combination.

\begin{table}[ht]
\centering
\caption{Empirical Type-I errors under $\boldsymbol{\Sigma}_x=\boldsymbol{\Sigma}_y=Q(s)$; $n=100$, $s=p^{1/2}$}
\label{tb:twosamsize1}
\begin{tabular}{rrrrr}
  \hline \hline
$p$ & 50 & 100 & 200 & 300 \\ 
  \hline
$\mathcal{U}(1)$ & 0.052 & 0.055 & 0.040 & 0.039 \\ 
$\mathcal{U}(2)$ & 0.051 & 0.060 & 0.053 & 0.047 \\ 
$\mathcal{U}(3)$ & 0.048 & 0.061 & 0.054 & 0.054 \\ 
$\mathcal{U}(4)$ & 0.039 & 0.059 & 0.067 & 0.053 \\ 
$\mathcal{U}(5)$ & 0.056 & 0.046 & 0.041 & 0.066 \\ 
$\mathcal{U}(6)$ & 0.045 & 0.044 & 0.041 & 0.044 \\ 
$\mathcal{U}(\infty)$ (permutation) & 0.047 & 0.042 & 0.049 & 0.052 \\
 adpUmin 1 & 0.043 & 0.057 & 0.059 & 0.053 \\ 
  adpUf 1 & 0.076 & 0.081 & 0.060 & 0.076 \\ 
$\mathcal{U}(\infty)$ (Tony) & 0.018 & 0.024 & 0.016 & 0.013 \\  
  adpUmin 2 & 0.044 & 0.056 & 0.059 & 0.051 \\ 
  adpUf 2 & 0.051 & 0.056 & 0.040 & 0.050 \\ 
   Chen & 0.050 & 0.049 & 0.049 & 0.050 \\ 
  Sriva & 0.166 & 0.002 & 0.000 & 0.000 \\ 
  Schott & 0.074 & 0.119 & 0.236 & 0.418 \\ 
  Pan1 & 0.055 & NA & NA & NA  \\ 
  Pan2 & 0.058 & NA & NA  &  NA \\ 
   \hline \hline
\end{tabular}
\end{table}


\newpage

\subsubsection{Study 8: Two-sample covariance testing power}\label{sec:twosamplecovsim2}
In this section, we examine the power of the two-sample covariance testing. 

We follow the covariance matrix models in \citet{weightstatsyang2017}. In particular, let $H(\tau_0,\tau_1,r)=(h_{i,j})_{p\times p}$, where $h_{i,j}=0$ except $h_{i,i}=\tau_0$, $i=1,\ldots,r$ and $h_{i,i+1}=h_{i,i-1}=\tau_1$, $i=1,\ldots, r-1$. Here $\tau_0$ and $\tau_1$ are used to measure the level of faint alternatives and $r$ is used to measure the sparsity level  of  alternative.  We fix $\boldsymbol{\Sigma}_x=I_{p}$, the $p\times p$ identity matrix, and  examine the following three representative covariance matrix models of $\boldsymbol{\Sigma}_y$. 

\textit{Model 1:} (Extreme faint, $\tau_0=0.04, \tau_1=0.2, r=p$).  $\boldsymbol{\Sigma}_y=I_p+H(0.04, 0.2, p)$. This matrix can also be considered as the covariance matrix of MA(1) model with the parameter $\theta_1=0.2$, which is also used in  \citet{lijun2012}. 

\textit{Model 2:} (Extreme sparse, $\tau_0=1, \tau_1=1.5, r=2$). $\boldsymbol{\Sigma}_y=I_p+H(1, 1.5, 2)$. This model only has four large disturbances compared with $\boldsymbol{\Sigma}_x$, which is regarded as the extreme sparse (ES) alternative.

\textit{Model 3:} (Reasonable faint and sparse, $\tau_0=0.3, \tau_1=0.3, r=p/10$) $\boldsymbol{\Sigma}_y=I_p+H(0.3, 0.3, p/10)$. The value of $r$ here is between $2$ (in Model 2) and $p$ (in Model 1), which is regarded as a moderately sparse setting. 

Under each model above, we take $n=100$, $p\in \{50, 100, 200, 300\}$, and provide the simulation results of the Models 1--3  in the Tables \ref{tb:model1}--\ref{tb:model3} respectively. The explanation of each row are the same as in Table \ref{tb:twosamsize1}, which is given in Section \ref{sec:twosamplecovsim1}.  Similarly, we note that the tests in \citet{weightstatsyang2017} are very time-consuming. Therefore for ``Pan 1" and ``Pan 2",  we only provide the simulation results at $p=50$, which takes about 100 times the time of the proposed adaptive testing procedure.


\begin{table}[!ht]
\centering
\caption{Empirical Power under Model 1 (Extreme faint); $n=100$.}
\label{tb:model1}
\begin{tabular}{rrrrr}
  \hline \hline
$p$ & 50 & 100 & 200 & 300 \\ 
  \hline
$\mathcal{U}(1)$ & 0.397 & 0.389 & 0.408 & 0.416 \\ 
$\mathcal{U}(2)$ & 0.445 & 0.458 & 0.456 & 0.484 \\ 
$\mathcal{U}(3)$ & 0.290 & 0.309 & 0.354 & 0.371 \\ 
$\mathcal{U}(4)$ & 0.197 & 0.211 & 0.199 & 0.205 \\ 
$\mathcal{U}(5)$ & 0.244 & 0.397 & 0.752 & 0.855 \\ 
$\mathcal{U}(6)$ & 0.054 & 0.052 & 0.054 & 0.091 \\ 
$\mathcal{U}(\infty)$ (permutation) & 0.066 & 0.062 & 0.044 & 0.029 \\ 
  adpUmin 1 & 0.478 & 0.511 & 0.692 & 0.783 \\ 
  adpUf 1 & 0.600 & 0.648 & 0.843 & 0.886 \\ 
$\mathcal{U}(\infty)$ (Tony) & 0.091 & 0.072 & 0.087 & 0.072 \\ 
  adpUmin 2 & 0.480 & 0.513 & 0.691 & 0.781 \\ 
  adpUf 2 & 0.619 & 0.669 & 0.855 & 0.903 \\ 
  Chen & 0.573 & 0.574 & 0.569 & 0.623 \\ 
  Sriva & 0.513 & 0.586 & 0.598 & 0.569 \\ 
  Schott & 0.667 & 0.731 & 0.888 & 0.956 \\ 
  Pan1 & 0.640 & NA  & NA & NA \\ 
  Pan2 & 0.669 & NA & NA & NA \\ 
   \hline \hline
\end{tabular}
\end{table}

\begin{table}[!ht]
\centering
\caption{Empirical Power under Model 2 (Extreme sparse); $n=100$.}
\label{tb:model2}
\begin{tabular}{rrrrr}
  \hline \hline
$p$ & 50 & 100 & 200 & 300 \\ 
  \hline
$\mathcal{U}(1)$ & 0.068 & 0.056 & 0.048 & 0.049 \\ 
$\mathcal{U}(2)$ & 0.725 & 0.364 & 0.122 & 0.086 \\ 
$\mathcal{U}(3)$ & 0.993 & 0.960 & 0.850 & 0.660 \\ 
$\mathcal{U}(4)$ & 1.000 & 0.997 & 0.988 & 0.956 \\ 
$\mathcal{U}(5)$ & 0.934 & 0.874 & 0.803 & 0.682 \\ 
$\mathcal{U}(6)$ & 0.972 & 0.960 & 0.935 & 0.914 \\ 
$\mathcal{U}(\infty)$ (permutation) & 0.966 & 0.919 & 0.852 & 0.772 \\ 
  adpUmin 1 & 1.000 & 0.992 & 0.984 & 0.959 \\ 
  adpUf 1 & 1.000 & 0.996 & 0.989 & 0.970 \\
$\mathcal{U}(\infty)$ (Tony) & 0.999 & 1.000 & 0.997 & 1.000 \\  
  adpUmin 2 & 1.000 & 0.997 & 0.993 & 0.995 \\ 
  adpUf 2 & 1.000 & 0.999 & 0.992 & 0.992 \\ 
  Chen & 0.800 & 0.457 & 0.196 & 0.127 \\ 
  Sriva & 0.787 & 0.433 & 0.166 & 0.101 \\ 
  Schott & 0.864 & 0.640 & 0.550 & 0.654 \\ 
  Pan1 & 0.673 & NA & NA & NA \\ 
  Pan2 & 0.694 & NA & NA & NA \\ 
   \hline \hline
\end{tabular}
\end{table}

\begin{table}[!ht]
\centering
\caption{Empirical Power under Model 3 (Reasonable faint and sparse); $n=100$.}
\label{tb:model3}
\begin{tabular}{rrrrr}
  \hline \hline
$p$ & 50 & 100 & 200 & 300 \\ 
  \hline
$\mathcal{U}(1)$ & 0.072 & 0.067 & 0.069 & 0.070 \\ 
$\mathcal{U}(2)$ & 0.090 & 0.096 & 0.096 & 0.083 \\ 
$\mathcal{U}(3)$ & 0.155 & 0.151 & 0.152 & 0.145 \\ 
$\mathcal{U}(4)$ & 0.175 & 0.162 & 0.162 & 0.154 \\ 
$\mathcal{U}(5)$ & 0.347 & 0.582 & 0.868 & 0.946 \\ 
$\mathcal{U}(6)$ & 0.308 & 0.494 & 0.732 & 0.854 \\ 
$\mathcal{U}(\infty)$ (permutation) & 0.028 & 0.034 & 0.027 & 0.018 \\ 
  adpUmin 1 & 0.337 & 0.496 & 0.797 & 0.901 \\ 
  adpUf 1 & 0.355 & 0.535 & 0.802 & 0.910 \\
$\mathcal{U}(\infty)$ (asymptotic) & 0.254 & 0.319 & 0.409 & 0.403 \\    
  adpUmin 2 & 0.348 & 0.508 & 0.798 & 0.901 \\ 
  adpUf 2 & 0.426 & 0.620 & 0.862 & 0.940 \\ 
  Chen & 0.138 & 0.149 & 0.153 & 0.144 \\ 
  Sriva & 0.092 & 0.096 & 0.097 & 0.100 \\ 
  Schott & 0.189 & 0.283 & 0.486 & 0.712 \\ 
  Pan1 & 0.167 & NA & NA & NA \\ 
  Pan2 & 0.186 & NA & NA & NA \\ 
   \hline \hline
\end{tabular}
\end{table}

We then analyze the simulation results. Model 1 is the extreme faint case and $\boldsymbol{\Sigma}_y-\boldsymbol{\Sigma}_x$ is dense. We find that under this case, the U-statistics of small orders, e.g., $\mathcal{U}(1)$ and $\mathcal{U}(2)$ are powerful. The tests based on the sum-of-squares type statistics including ``Chen", ``Sriva" and ``Schott"  are also powerful under this case. Our proposed adaptive testing procedure using Fisher's method has comparable power performance to ``Pan 1" and ``Pan 2", and is computationally more efficient.  Model 2 is the extreme sparse case. Under this case, we find that generally U-statistics of higher orders, e.g., $\mathcal{U}(4)$ and $\mathcal{U}(\infty)$, are more powerful than the U-statistics of smaller orders, e.g., $\mathcal{U}(1)$ and $\mathcal{U}(2)$. Model 3 is the moderately faint and sparse case. Under this case, we can see that a finite-order U-statistic $\mathcal{U}(5)$ is the most powerful one. Neither the maximum-type test statistic $\mathcal{U}(\infty)$ and  the sum-of-squares type test statistic $\mathcal{U}(2)$, ``Chen", ``Sriva" and ``Schott" are very powerful.  Tests in \cite{weightstatsyang2017} considering only faint or sparse alternatives are not very powerful under this case. On the other hand, the proposed adaptive testing procedure maintains high power under this case.


\bibliographystyle{chicago}
\bibliography{Citation.bib}

\end{document}